\documentclass[12pt,a4paper]{amsart}
\usepackage[a4paper]{geometry}
\usepackage{amsmath}
\usepackage{amsfonts}
\usepackage{amssymb}
\theoremstyle{plain}
\newtheorem{thm}{Theorem}
\newtheorem{prp}[thm]{Proposition}
\newtheorem{lem}[thm]{Lemma}
\newtheorem{cor}[thm]{Corollary}
\theoremstyle{definition}
\newtheorem*{rems}{Remarks}
\newtheorem{remN}[thm]{Remark}
\newtheorem*{rem}{Remark}
\theoremstyle{definition}
\newtheorem*{dfn}{Definition}
\numberwithin{equation}{section}
\numberwithin{thm}{section}
\newenvironment{smal}%
{\par\medskip
\smaller\everypar{\setlength{\leftskip}{0.5cm}%
\setlength{\rightskip}{0.5cm}}}
{\par\medskip\goodbreak }
\def\SmallDisplay#1\[#2\]{%
\par
\vskip 2pt plus 1pt minus 1pt
\hbox to \linewidth{%
\kern0.5cm #1 \hfill $\displaystyle #2$ \hfill
\hphantom{#1}\kern0.5cm}%
\vskip 4pt plus 2pt minus 2pt
}
\font\ninerm=cmr9
\font\ninett=cmtt10 at 9pt
\font\sevenit=cmti7
\def\Arg{\mathop{\rm Arg}\nolimits}
\def\bar{\mathop{\rm bar}\nolimits}
\def\barre{\raise -1.5pt\hbox{\big|}}
\def\C{\mathbf{C}}
\def\ca#1{\mathcal{#1}}
\def\Di#1#2{{#1}_{(#2)}}
\def\DiE#1#2#3{{#1}_{(#2)}^{#3}}
\def\d#1{{\rm d}\hbox{$\mskip 0.5mu$}#1} 
\def\di#1#2{#1_{[#2]}}
\def\dint{\int \!\!\! \int }
\def\E{\mathop{\rm E{}}\nolimits}
\def\e{\mathop{\rm e{}}\nolimits}
\def\et{ \ms4 \& \ms4 }
\def\fg{\rq\rq}  
\def\fge{\fg\ }  
\def\ge{\geqslant}
\def\gr#1{{\mathbf{#1}}}  
\def\I{\mskip 1.5mu{\rm I}\mskip 1.0mu}
\def\Id{\mskip 1.5mu{\rm Id}\mskip 1.5mu}
\def\Im{\mathop{\rm Im}\nolimits}
\def\ii{\mskip 1.5mu{\rm i}\mskip 1.5mu}  
\def\Klc{K_{\textit{lc}}}
\def\Kg{K_{\textit{g}}}
\def\le{\leqslant}
\def\lra{\mathop{\longrightarrow}\limits}
\def\M{\mathrm{M}}
\def\Mg{\mathit{M}}
\def\mg{m_{\textit{g}}}
\def\mlc{m_{\textit{lc}}}
\def\mr#1{\mathrm{#1}}
\def\ms#1{\mskip#1mu}   
\def\ns#1{\mskip-#1mu}  
\def\N{\mathbf{N}}
\def\og{\lq\lq}  
\def\PF{\widehat{\hphantom{h}\vphantom{n}}}
\def\ps{\mathbin{\cdot}}  
\def\R{\mathbf{R}}
\def\Re{\mathop{\rm Re}\nolimits}
\def\sca#1#2{\langle #1, #2 \rangle}  
\def\sign{\mathop{\rm sign}\nolimits}
\def\T{\mathbf{T}}
\def\xG{x_{\scriptscriptstyle \Gamma}}
\def\up#1{\raise 1.5pt\hbox{#1}}  
\def\Z{\mathbf{Z}}
\font\ninesy=cmsy9   
\newbox\sileq
\setbox\sileq=\vtop{\ialign{#\crcr 
\raise1.33pt\hbox{$\hfil\displaystyle < \hfil$}\crcr
\noalign{\kern -0.05pt\nointerlineskip}
\raise1.05pt\hbox{\kern0.31pt$\ninesy\hfil \sim \hfil$}\crcr
 \noalign{\kern1pt}}}
\def\bsimleq{\mathop{\vtop{\unvcopy\sileq}}\limits}
\def\simleq{\mathop{\vtop{\ialign{##\crcr 
\raise1.33pt\hbox{$\hfil\displaystyle < \hfil$}\crcr
\noalign{\kern -0.05pt\nointerlineskip}
\raise1.05pt\hbox{\kern0.31pt$\ninesy\hfil \sim \hfil$}\crcr
 \noalign{\kern1pt}}}}\limits}
\def\Dsimleq{\mathop{\vtop{\ialign{##\crcr 
\raise1.33pt\hbox{$\hfil\displaystyle < \hfil$}\crcr
\noalign{\kern -0.05pt\nointerlineskip}
\raise1.05pt\hbox{\kern0.15pt$\ninesy\hfil \sim \hfil$}\crcr
 \noalign{\kern1pt}}}}\limits}
\def\bsle#1{\ms{#1}\bsimleq\ms{#1}}
\def\Dsle#1{\ms{#1}\Dsimleq\ms{#1}}
\def\dumou{\vskip 2pt plus 2pt minus 4pt}
\newbox\LDOG
\setbox\LDOG=\hbox
 {\ninerm L.~D., O.~G.~: Laboratoire d'Analyse et de Math\'ematiques
  Appliqu\'ees, }
\newbox\LDOGB
\setbox\LDOGB=\hbox
 {\ninerm Universit\'e Paris-Est--Marne la Vall\'ee,
  77454 Marne la Vall\'ee CEDEX 2, France}
\newbox\LDOGC
\setbox\LDOGC=\hbox
 {\ninett luc.deleaval@u-pem.fr, olivier.guedon@u-pem.fr}
\ht\LDOGC=10pt
\newbox\VV
\setbox\VV=\hbox{}
\ht\VV=1pt
\newbox\BM
\setbox\BM=\hbox
 {\ninerm B.~M.~: Institut de Math\'ematiques de Jussieu--PRG, 
  UPMC, 75005 Paris, France }
\newbox\BMB
\setbox\BMB=\hbox
 {\ninett bernard.maurey{@}imj-prg.fr}
\begin{document}

\title[Dimension free bounds]{Dimension free bounds for the
 Hardy--Littlewood maximal operator associated to convex sets}
\date{September, 2016}
\author[L.~Deleaval]{Luc Deleaval}
\author[O.~Gu\'edon]{Olivier Gu\'edon}
\author[B.~Maurey]{Bernard Maurey \\
        \\
 \box\LDOG \\
 \box\LDOGB \\
 \box\LDOGC \\
 \box\VV \\
 \box\BM \\
 \box\BMB }

\vskip -0.35cm

\begin{abstract}
This survey is based on a series of lectures given by the authors at the
working seminar \og Convexit\'e et Probabilit\'es\fge at UPMC Jussieu,
Paris, during the spring 2013. It is devoted to maximal functions
associated to symmetric convex sets in high dimensional linear spaces, a
topic mainly developed between 1982 and 1990 but recently renewed by
further advances. 

 The series focused on proving these maximal function inequalities in
$L^p(\R^n)$, with bounds independent of the dimension~$n$ and for 
all~$p \in (1, +\infty]$ in the best cases. This program was initiated in
1982 by Elias Stein, who obtained the first theorem of this kind for the
family of Euclidean balls in arbitrary dimension. We present several
results along this line, proved by Bourgain, Carbery and M\"uller during
the period 1986--1990, and a new one due to Bourgain~(2014) for the family
of cubes in arbitrary dimension. We complete the cube case with a negative
answer to the possible dimensionless behavior of the weak type (1,~1)
constant, due to Aldaz, Aubrun and Iakovlev--Str\"omberg between 2009 and
2013.
\end{abstract}

\maketitle

\begin{smal}
\noindent{\sc R\'esum\'e.} 
Ces Notes reprennent et compl\`etent une s\'erie d'expos\'es donn\'es par
les auteurs au groupe de travail \og Convexit\'e et Probabilit\'es\fge \`a
l'UPMC Jussieu, Paris, au cours du printemps 2013. Elles sont consacr\'ees
\`a l'\'etude des fonctions maximales de type Hardy--Littlewood associ\'ees
aux corps convexes sym\'etriques dans~$\R^n$. On s'int\'eresse tout
particuli\`erement au comportement des constantes intervenant dans les
estimations lorsque la dimension $n$ tend vers l'infini. Ce sujet a \'et\'e
d\'evelopp\'e principalement entre 1982 et 1990, mais a \'et\'e relanc\'e
par des avanc\'ees r\'ecentes.

 Le but de la s\'erie d'expos\'es \'etait de prouver des in\'egalit\'es
maximales dans $L^p(\R^n)$ avec des bornes ind\'ependantes de la
dimension~$n$, pour certaines familles de corps convexes. Dans les
meilleurs cas, on a pu obtenir de tels r\'esultats pour toutes les valeurs
de $p$ dans $(1, +\infty]$. Ce th\`eme de recherche a \'et\'e initi\'e en
1982 par Elias Stein~\cite{SteinSF}, qui a d\'emontr\'e le premier
th\'eor\`eme de ce genre pour la famille des boules euclidiennes en
dimension arbitraire, obtenant pour tout $p \in (1, +\infty]$ une borne
dans $L^p(\R^n)$ ind\'ependante de~$n$. Nous pr\'esentons ce th\'eor\`eme de
Stein ainsi que plusieurs autres r\'esultats dans cette direction,
d\'emontr\'es par Bourgain, par Carbery et par M\"uller dans la p\'eriode
1986--1990. En 1986, Bourgain~\cite{BourgainL2} obtient une borne
ind\'ependante de~$n$ valable dans $L^2(\R^n)$ pour tout corps convexe
sym\'etrique dans~$\R^n$, puis Bourgain~\cite{BourgainLp} et
Carbery~\cite{CarberyLp} \'etendent le r\'esultat $L^p(\R^n)$ de Stein aux
corps convexes sym\'etriques quelconques, mais sous la condition que 
$p > 3/2$. M\"uller~\cite{MullerQC} obtient un r\'esultat valable pour tout
$p > 1$ quand un certain param\`etre g\'eom\'etrique, li\'e aux volumes des
projections du corps convexe sur les hyperplans, reste born\'e. Ce
param\`etre ne reste pas born\'e pour tous les convexes, en particulier, il
tend vers l'infini pour les cubes de grande dimension. Nous donnons un
th\'eor\`eme r\'ecent (2014) d\^u \`a Bourgain~\cite{BourgainCube} qui
obtient pour tout $p > 1$ une borne dans $L^p(\R^n)$ ind\'ependante de $n$
pour la famille des fonctions maximales associ\'ees aux cubes en dimension
arbitraire. Nous compl\'etons l'\'etude du cas du cube par des r\'esultats
pour la constante de type faible $(1, 1)$, dus \`a Aldaz~\cite{AldazWT},
\`a Aubrun~\cite{Aubrun} et \`a Iakovlev--Str\"omberg~\cite{IS} entre 2009
et 2013. \`A l'inverse du cas $L^p(\R^n)$, $1 < p \le +\infty$, cette
constante de type faible ne reste pas born\'ee quand la dimension tend vers
l'infini.

 Nous tenons \`a remercier ceux qui nous ont encourag\'es dans notre projet
de r\'edaction de ces Notes, et tout particuli\`erement Franck Barthe qui a
su nous mettre en mouvement. Nous sommes reconnaissants \`a P.~Auscher et
E.~Stein, qui nous ont t\'emoign\'e leur int\'er\^et et indiqu\'e des
r\'ef\'erences importantes qui nous avaient \'echapp\'e. Nous remercions
les rapporteurs pour leurs suggestions constructives.
\end{smal}

\bigskip

\def\Dsection#1#2{\noindent \ref{#2}. #1 
 \leaders \hbox to 0.8em{\hss.\hss}\hfill  \pageref{#2} \par}
\def\DBsection#1#2{\noindent #1 
 \leaders \hbox to 0.8em{\hss.\hss}\hfill  \pageref{#2} \par}
\def\Dsubsection#1#2{\noindent\kern 0.35cm  \ref{#2}. #1 
 \leaders \hbox to 0.8em{\hss.\hss}\hfill  \pageref{#2} \par}
\def\Dsubsubsection#1#2{\noindent\kern 0.70cm  \ref{#2}. #1
 \leaders \hbox to 0.8em{\hss.\hss}\hfill  \pageref{#2} \par}
\def\Dthebibliography{\noindent References}

\noindent{\bf Contents}

{\smaller

\DBsection{Introduction}{Intro}
\dumou

\Dsection{General dimension free inequalities, first part}{Gdfi}
\Dsubsection{Doob's maximal inequality}{MaxiDoob}
\Dsubsection{The Hopf maximal inequality}{MaxiHopf}
\Dsubsection{From martingales to semi-groups, via an argument of Rota}
 {RotaRem}
\Dsubsection{Brownian motion, and more on martingales}{BrowniMarti}
\Dsubsubsection{Gaussian distributions and Brownian motion}{GaussiDist}
\Dsubsubsection{The Burkholder--Gundy inequalities}{BGIneqs}
\Dsubsubsection{A consequence of the \og reflection principle\fg}
 {PrincipeReflexion}
\Dsubsection{The Poisson semi-group}{PoissonSG}

\dumou
\Dsection{General dimension free inequalities, second part}{GdfiTwo}
\Dsubsection{Littlewood--Paley functions}{LittlePal}
\Dsubsubsection{Littlewood--Paley and maximal functions}{MaxiLittlePal}
\Dsubsection{Fourier multipliers}{FourierMult}
\Dsubsubsection{Multipliers \og of Laplace type\fg}{LaplaceMulTyp}
\Dsubsection{Riesz transforms}{TransfoRiesz}

\dumou
\Dsection{Analytic tools}{AnaTool}
\Dsubsection{Some known facts about the Gamma function}{FoGamm}
\Dsubsection{The interpolation scheme}{InterpoSchem}
\Dsubsubsection{The three lines lemma}{TheTLL}
\Dsubsubsection{Interpolation of holomorphic families of linear operators}
 {Ihf}
\Dsubsection{On the definition of maximal functions}{DefiMaxiFunc}

\dumou
\Dsection{The results of Stein for Euclidean balls}{SteinsResults}
\Dsubsection{Proof of Theorem \ref{indepdim}}{SphericOp}
\Dsubsection{Boundedness of the spherical maximal operator}{Bsmo}
\Dsubsubsection{Maximal operator and square function}{Mmoasf}
\Dsubsubsection{Strong and weak type results for $\Mg_{K_\ell}$, 
  $\ell \ge 1$}
 {Sawtr}
\Dsubsubsection{Conclusion}{PoTsph}

\dumou
\Dsection{The $L^2$ result of Bourgain}{ArtiBour}
\Dsubsection{The general setting}{TheSetting}
\Dsubsection{On the volume of sections}{VoluSections}
\Dsubsection{Fourier analysis in $L^2(\R^n)$}{AnaFouri}
\Dsubsubsection{Conclusion of Bourgain's argument}{ConcluBour}

\dumou
\Dsection{The $L^p$ results of Bourgain and Carbery}{ArtiCarbe}
\Dsubsection{{\it A priori\/} estimate and interpolation}{EstiPrio}
\Dsubsection{Fractional derivatives}{FractiDeri}
\Dsubsubsection{Multipliers associated to fractional derivatives}{MultipAssoc}
\Dsubsection{Fourier criteria for bounding the maximal function}{CritFou}
\Dsubsection{Proofs of Theorems~\ref{TheoDyad} and~\ref{TheoMaxi},
 and Proposition~\ref{MoreGeneral}}{PreuveTheos}
\Dsubsubsection{Where is the gap?}{GapQuestion}
\Dsubsection{A proof for the property $(\gr S_2)$}{AutrePreuve}
\Dsubsubsection{A solution to the gap question}{SGQ}
\Dsubsection{Annex: proof of Bourgain's $L^2$ theorem by Carbery's criterion}
 {PreuvBour}

\dumou
\Dsection{The Detlef M\"uller article}{AMuller}
\Dsubsection{The M\"uller strategy}{StrateMull}
\Dsubsection{Model of proof: the Poisson case}{Model}
\Dsubsection{The interpolation part of Carbery's proof for 
 Theorem~\ref{TheoMaxi}}{InterpoCarbe}
\Dsubsection{Upper bounds for the functions  
 $\xi \mapsto m^\varepsilon_z (\xi)$}{LesMajos}
\Dsubsection{Lemma 4 of M\"uller's article}{LemmeQuatre}
\Dsubsubsection{Conclusion}{ConclusionMull}

\dumou
\Dsection{Bourgain's article on cubes}{LeCube}
\Dsubsection{Holding on M\"uller and Carbery}{Raccord}
\Dsubsubsection{\textit{A priori} estimate}{APriori}
\Dsubsection{First reduction}{FirstReduc}
\Dsubsubsection{Decoupling}{Decoupling}
\Dsubsection{Second reduction}{SecondReduc}
\Dsubsection{Conclusion}{ConcluBourCube}

\dumou
\Dsection{The Aldaz weak type result for cubes,
          and improvements}{AlAu}

\dumou

\DBsection{References}{References}

\DBsection{Index}{Index}

\DBsection{Notation}{Notation}

}

\vfill\eject

\section*{Introduction\label{Intro}}
\dumou

\noindent
First defined by Hardy and Littlewood~\cite{HaLi} in the one-dimensional
setting, the Hardy--Littlewood maximal operator was generalized in
arbitrary dimension by Wiener~\cite{Wiener}. It turned out to be a
powerful tool, for instance in harmonic or Fourier analysis, in
differentiation theory or in singular integrals theory. It was extended to
various situations, including not only homogeneous settings, as in the book
of Coifman and Weiss~\cite{CW}, but also non-homogeneous, like noncompact
symmetric spaces in works by Clerc and Stein~\cite{ClSt} or
Str\"omberg~\cite{StrA}. Also studied in vector-valued settings with the
Fefferman--Stein type inequalities~\cite{FS}, it gave rise to several kinds
of maximal operators which are now important in real analysis.
\dumou

 We shall denote by $\M$ the classical centered Hardy--Littlewood maximal
operator, defined on the class of locally integrable functions $f$
on~$\R^n$ by
\begin{equation}
   (\M f)(x)
 = \sup_{r > 0} \frac 1 {|B_r|} \int_{B_r} |f(x-y)| \, \d y,
 \quad x \in \R^n,
 \label{ClassicalM}
\end{equation}
\dumou
\noindent
where $B_r$ is the Euclidean ball of radius~$r$ and center~$0$ in $\R^n$,
and $|S|$ denotes here the $n$-dimensional Lebesgue volume of a Borel
subset $S$ of $\R^n$. It is well known that this nonlinear operator $\M$ is
of \emph{strong type} $(p, p)$\label{StrongType} 
when $1 < p \le +\infty$ and of \emph{weak type} $(1, 1)$, as stated in the
following famous theorem. We write $L^p(\R^n)$ for the $L^p$-space
corresponding to the Lebesgue measure on~$\R^n$.
\dumou

\begin{thm}[Hardy--Littlewood maximal theorem]%
\label{TheoHL}
Let $n$ be an integer~$\ge 1$.

\begin{enumerate}
\item 
For every function $f \in L^1(\R^n)$ and $\lambda > 0$, the
\emph{weak type inequality}
\begin{equation} 
      \bigl| \bigl\{ x \in \R^n : (\M f) (x) > \lambda \bigr\} \bigr|
 \le \frac {C(n)} \lambda \ms2 \|f\|_{L^1(\R^n)}
 \label{weak-type} \tag{WT}
\end{equation}
holds true, with a constant $C(n)$ depending only on the dimension~$n$.

\item Let\/ $1 < p \le +\infty$. There exists a constant $C(n, p)$ such that
for every function $f$ in~$L^p(\R^n)$, one has
\begin{equation}
     \|\M f\|_{L^p(\R^n)}
 \le C(n, p) \ms2 \|f\|_{L^p(\R^n)}.
 \label{strong-type} \tag{ST}
\end{equation}
\end{enumerate}
\end{thm}
\dumou

 The weak type inequality is optimal in the sense that $\M f$ is never
in~$L^1(\R^n)$, unless $f = 0$ almost everywhere. Zygmund introduced the
so-called \og$L\log L$ class\fge to give a sufficient condition for the
local integrability of the Hardy--Littlewood maximal function, a condition
that is actually necessary, as proved by Stein~\cite{SteinLlogL}. The proof
of Theorem~\ref{TheoHL} by Hardy and Littlewood was combinatorial and used
decreasing rearrangements. The authors said: \og The problem is most easily
grasped when stated in the language of cricket, or any other game in which
a player compiles a series of scores of which an average is recorded\fg.
Passing through the Vitali covering lemma, which is recalled below,
has become later a standard approach.
\dumou

 A natural question that can be raised is the following. Could we compute
the best constant in both inequalities~\eqref{weak-type}
and~\eqref{strong-type}? This question seems to be out of reach in full
generality. There is a very remarkable exception to this statement, the
one-dimensional case where Melas has shown in~\cite{Melas} by a mixture
of combinatorial, geometric and analytic arguments, that the best constant
in~\eqref{weak-type} is $(11 + \sqrt{61}) / 12$. The case $p > 1$ is still
open, even in the one-dimensional case, despite of substantial progress by
Grafakos, Montgomery-Smith and Motrunich~\cite{GMSM}, who obtained by
variational methods the best constant in~\eqref{strong-type} for the class
of positive functions on the line that are convex except at one point. 
The \emph{uncentered}\label{Uncent} 
maximal operator $f \mapsto f^*$ is better understood~\cite{GMS}, the
uncentered maximal function $f^*$ being defined for every $x \in \R^n$ by
\begin{equation}
   f^*(x)
 = \sup_{B \in \ca B(x)} \ms2 \frac 1 {|B|} \int_{B} |f(u)| \, \d u,
 \label{UncentOp}
\end{equation}
where $\ca B(x)$ denotes the family of Euclidean balls $B$ containing~$x$,
with arbitrary center $y$ and radius $> d(y, x)$. It is clear that 
$f^* \ge \M f$, and the maximal theorem also holds for $f^*$ since any 
\og uncentered\fge ball $B \in \ca B(x)$ of radius $r$ is contained
in~$B(x, 2 r)$, yielding the far from sharp pointwise inequality 
$f^* \le 2^n \ms1 \M f$.
\dumou

 Lacking for exact values, one may address the question of the asymptotic
behavior of the constants when the dimension $n$ tends to infinity. This
program was initiated at the beginning of the 80s by Stein. In the usual
proof of the Hardy--Littlewood maximal theorem based on the Vitali covering
lemma, the dependence on the dimension~$n$ in the weak type result is
exponential, of the form $C(n) = C^n$ for some $C > 1$. Then, by
interpolation of Marcinkiewicz-type between the weak-$L^1$ case and the
trivial $L^\infty$ case, one can get for the strong type in $L^p(\R^n)$ a
constant of the form $C(n, p) = p \ms2 C^{n/p} / (p - 1)$,
when $1 < p \le +\infty$ (see~\cite[Exercises, 1.3.3~(a)]{GrafakosCFA}). 
In~\cite{SteinSF}, Stein has improved this asymptotic behavior in a
spectacular fashion. Indeed, by using a spherical maximal operator together
with a lifting method, he showed that for every $p > 1$, one can replace
the bound $C(n, p)$ in~\eqref{strong-type} by a bound $C(p)$ independent
of~$n$. The detailed proof appeared in the paper~\cite{SteStr} by Stein and
Str\"omberg. 
\dumou

 The use of an appropriate spherical maximal operator is now a decisive
approach for bounding the $L^p$ norm of Hardy--Littlewood-type maximal
operators independently of the dimension~$n$, when $p > 1$. This is the
case, for instance, for the Heisenberg group~\cite{Zien} or for hyperbolic
spaces~\cite{Li2}. Moreover, Stein and Str\"omberg proved that the
weak type $(1, 1)$ constant grows at most like $\mr O(n)$, and it is still
unknown whether or not this constant may be bounded independently of the
dimension. The proof in~\cite{SteStr} draws on the
Hopf--Dunford--Schwartz ergodic theorem, about which Stein says
in~\cite{SteinTHA} that it is \og one of the most powerful results in
abstract analysis\fg. The strategy, which exploits the relationship between
averages on balls and either the heat semi-group or the Poisson
semi-group, is well explained in~\cite{CGGM}, and has been applied in
several different settings~\cite{Del, Li1, Li3, LiLo}. 
\dumou

 In a large part of these Notes, we shall replace Euclidean balls in the
definition~\eqref{ClassicalM} of the maximal operator by other centrally
symmetric convex bodies in~$\R^n$ (in what follows, we shall omit
\og centrally\fge and abbreviate it as \emph{symmetric convex body}). For
example, replacing averages over Euclidean balls $B_r$ of radius~$r$ by
averages over $n$-dimensional cubes $Q_r$ with side $2 r$ gives an operator
$\M_Q$ which satisfies both the weak type and strong type maximal
inequalities. Indeed, since $B_r \subset Q_r \subset \sqrt n B_r$, it is
obvious that $\M_Q$ is bounded in~$L^p(\R^n)$ with $C(n, p)$ replaced by 
$n^{n/2} C(n, p)$, but this painless route badly spoils the constants.
Several results specific to the cube case have been obtained, as we shall
indicate below.
\dumou
\dumou

 More generally, as in Stein and Str\"omberg~\cite{SteStr}, one can give a
symmetric convex body $C$ in $\R^n$ and introduce the \emph{maximal
operator $\M_C$ associated to the convex set $C$} as follows: for every 
$f \in L^1_{\rm loc}(\R^n)$ one defines the function $\M_C f$ on $\R^n$ by 
\begin{subequations}\label{OpMaxiG}%
\begin{equation}
   (\M_{C} f)(x) 
 = \sup_{t > 0} \frac 1 { |t \ms1 C| } \int_{x + t C} |f(y)| \, \d y
 = \sup_{t > 0} \frac 1 {|C|} \int_{C} |f(x + t \ms1 v)| \, \d v,
 \quad
 x \in \R^n,
 \label{OpMaxi} \tag{\ref{OpMaxiG}.{\bf M}}
\end{equation}
\end{subequations}
\dumou
\noindent
where $x + t C := \{ x + t c : c \in C \}$. 
One may also consider $\M_C$ when $C$ is not symmetric but has its centroid
at $0$, see Fradelizi~\cite[Section~1.5]{FradelHSCB}. The maximal operator
$\M_C$ satisfies, again, a maximal theorem of Hardy--Littlewood type. 

 Let $C$ be a symmetric convex body in~$\R^n$. The weak type $(1, 1)$
property for $\M_C$ can be deduced from the 
\emph{Vitali covering lemma}:\label{ViCovLem} 
given a finite family of translated-dilated sets $x_i + r_i C$, $i \in I$,
$x_i \in \R^n$, $r_i > 0$, one can extract a \emph{disjoint} subfamily
$(x_j + r_j C)_{j \in J}$, $J \subset I$, such that each set $x_i + r_i C$,
$i \in I$, of the original family is contained in the \emph{dilate} 
$x_j + 3 \ms1 r_j C$ of some member $x_j + r_j C$, $j \in J$, of the
\emph{extracted} disjoint family. One may explain the constant~3 by the use
of the triangle inequality for the norm on $\R^n$ whose unit ball is~$C$.
Passing to the Lebesgue measure in~$\R^n$, this statement naturally
introduces a factor $3^n$ corresponding to the dilation factor~3. If
$f^*_C$ denotes the corresponding uncentered maximal function of $f$
associated to~$C$, then for every $\lambda > 0$, one has that 
\begin{equation}
      \bigl| \bigl\{ x \in \R^n : f^*_C(x) > \lambda \bigr\} \bigr|
 \le \frac {3^n \ns4} \lambda \ms4
      \int_{ \{f^*_C > \lambda\} } |f(x)| \, \d x.
 \label{VitaliEstim}
\end{equation}
We briefly sketch a proof, very similar to that of Doob's maximal inequality
presented in Section~\ref{MaxiDoob}. It is convenient here to consider that
$C$ is an \emph{open} subset of $\R^n$. Given an arbitrary compact subset
$K$ of the open set $U_\lambda = \{ f^*_C > \lambda \}$, one applies the
Vitali lemma to a finite covering of~$K$ by open sets $S_i = x_i + r_i C$
such that $\int_{S_i} |f| > \lambda |S_i|$. A simple feature of $f^*_C$ is
that each such $S_i$ is actually \emph{contained} in~$U_\lambda$. If 
$J \subset I$ corresponds to the disjoint family given by Vitali, then
\[
     |K|
 \le \sum_{j \in J} |x_j + 3 r_j C|
  =  3^n \sum_{j \in J} |x_j + r_j C|
 \le \frac {3^n \ns4} \lambda \ms4 \int_{U_\lambda} |f(x)| \, \d x,
\]
implying~\eqref{VitaliEstim}. Next, a direct argument involving only Fubini
and H\"older can give an $L^p$ bound, exactly as in the proof of Doob's
Theorem~\ref{TheoDoob} below, but giving a factor $3^n$ instead of $3^{n/p}$
obtained by interpolation. This Vitali method does not depend upon the
symmetric body $C$, does not distinguish the centered and uncentered
operators, and introduces a quite unsatisfactory exponential constant.
\dumou
\dumou

 Stein and Str\"omberg have greatly improved this exponential dependence
in~\cite{SteStr}. By a clever covering argument with less overlap
than in Vitali's lemma, they proved that the weak type constant admits a
bound of the form $\mr O(n\log n)$, and by using the Calder\'on--Zygmund
method of rotations, they obtained for the strong type property a constant
which behaves as~$n \ms1 p / (p-1)$. Concerning the weak type constant, Naor
and Tao~\cite{NaorTao} have established the same $n\log n$
behavior for the large class of \emph{$n$-strong micro-doubling} metric
measure spaces (see also~\cite{CrSo}). Several powerful results about the
strong type constant for maximal functions associated to convex sets,
beyond the one of Stein--Str\"omberg, have been established between 1986 and
1990. First of all, Bourgain proved a dimensionless theorem for general
symmetric convex bodies in the $L^2$ case~\cite{BourgainL2}, applying
geometrical arguments and methods from Fourier analysis. This result has
been generalized to $L^p(\R^n)$, for all $p > 3 / 2$, by
Bourgain~\cite{BourgainLp} and Carbery~\cite{CarberyLp} in two
independent papers. They both bring into play an auxiliary dyadic maximal
operator, but Bourgain uses it together with square function techniques
while Carbery uses multipliers associated to fractional derivatives. Detlef
M\"uller extended in~\cite{MullerQC} the $L^p$ bound to every $p > 1$, but
under an additional geometrical condition on the family of convex sets~$C$
under study. M\"uller also proved that for every fixed 
$q \in [1, +\infty)$, his condition is fulfilled by the family $\ca F_q$ of
$\ell^q_n$ balls, $n \in \N^*$.
\dumou
\dumou

 After M\"uller's article, activity in this area slowed down. Nevertheless,
Bourgain recently proved in~\cite{BourgainCube} that for all $p > 1$,
the strong type constant can be bounded independently of the dimension when
we average over cubes. In order to attack this problem, Bourgain applies an
arsenal of techniques, including a holomorphic semi-group theorem due to
Pisier~\cite{PisierHSG} and ideas inspired by martingale theory. The cube case
is rather well understood since Aldaz~\cite{AldazWT} has proved that the
weak type~$(1, 1)$ constant $\kappa_{Q, n}$\label{KappaQ} for cubes must
tend to infinity with the dimension~$n$. The best lower bound known at the
time of our writing is due to Iakovlev--Str\"omberg~\cite{IS} who obtained 
$\kappa_{Q, n} \ge \kappa \ms1 n^{1/4}$, improving a previous estimate
$    \kappa_{Q, n}
 \ge \kappa_\varepsilon \ms1 (\log n)^{1 - \varepsilon}$ for every 
$\varepsilon > 0$, which was obtained by Aubrun~\cite{Aubrun} following the
Aldaz result.
\dumou
\dumou

 In the present survey, except for Section~\ref{AlAu} on the Aldaz 
\og negative\fge result, we shall restrict ourselves to $p > 1$ and examine
the strong type $(p, p)$ behavior of maximal functions associated to
symmetric convex bodies in~$\R^n$. We shall present the dimensionless
result of Stein for Euclidean balls, the works of Bourgain, Carbery and
M\"uller during the 80s and the recent dimensionless theorem of Bourgain
for cubes. As we shall see, the proofs require a lot of methods and tools,
including multipliers, square functions, Littlewood--Paley theory, complex
interpolation, holomorphic semi-groups and geometrical arguments involving
convexity. The study of weak type inequalities for Hardy--Littlewood-type
operators needs powerful methods as well: not only the aforementioned
Hopf--Dunford--Schwartz ergodic theorem, but also sharp estimates for heat
or Poisson semi-group, Iwasawa decomposition, $K$-bi-invariant
convolution-type operators, expander-type estimates\dots
\dumou
\dumou

 The first two sections contain general dimension free inequalities
obtained respectively by probabilistic methods or by Fourier transform
methods. The Poisson semi-group plays an important r\^ole in Stein's
book~\cite{SteinTHA}, and appears also in Bourgain's
articles~\cite{BourgainL2, BourgainLp} and in Carbery~\cite{CarberyLp}.
We give a presentation of this semi-group, both on the
probabilistic and Fourier analytic viewpoints. The third section is about
some analytic tools that are employed later on, namely, estimates for the
Gamma function in the complex plane, and the complex interpolation scheme
for linear operators, as developed in Stein~\cite{SteinILO}. The Stein
result for Euclidean balls in arbitrary dimension is our
Theorem~\ref{indepdim}. Section~\ref{ArtiBour} is about Bourgain's
$L^2$-theorem in arbitrary dimension~$n$, stating that there exists a
constant $\kappa_2$ independent of~$n$ such that for any symmetric convex
body $C$ in~$\R^n$, one has
\[
 \| \M_C f \|_{L^2(\R^n)} \le \kappa_2 \ms2 \|f\|_{L^2(\R^n)}
\]
for every $f \in L^2(\R^n)$. The next section presents Carbery's proof of
the generalization to $L^p$ of the latter bound, obtained by
Bourgain~\cite{BourgainLp} and Carbery~\cite{CarberyLp}. In both
papers, the $L^p$ result for general symmetric convex bodies is proved for
$p > 3 / 2$ only. A theorem due to Detlef M\"uller~\cite{MullerQC} is given
in Section~\ref{AMuller}; for families of symmetric convex sets~$C$ for
which a certain parameter $q(C)$ remains bounded, it extends the
dimensionless $L^p$ bound to every $p > 1$. This parameter is related to the
$(n-1)$-dimensional measure of hyperplane projections of a specific volume
one linear image of~$C$, the so-called \emph{isotropic position}.
Section~\ref{LeCube} presents the result of Bourgain about cubes in
arbitrary dimension. In this special case, an $L^p$ bound independent of
the dimension is valid for all $p > 1$, although the M\"uller condition is
not satisfied. Bourgain's proof is highly dependent on the product
structure of the cube. In Section~\ref{AlAu}, we prove the Aldaz result
that the weak type $(1, 1)$ constant for cubes is not bounded when the
dimension~$n$ tends to infinity. We mention the quantitative improvement
by Aubrun~\cite{Aubrun}, and give a proof for the lower bound 
$\kappa \ms1 n^{1/4}$ due to Iakovlev--Str\"omberg~\cite{IS}.
\dumou

 We have put a notable emphasis on the notion of log-concavity. We shall see
that with not much more effort, most maximal theorems for convex sets
generalize to symmetric log-concave probability densities. This kind of
extension from convex sets to log-concave functions has attracted a lot of
attention in convex geometry in recent years,
see~\cite{Ball, Klar, Klar2, GuMi} among many others. In fact, Bourgain's
estimate~\eqref{EstimaBour}, which is crucial to all results in
Section~\ref{ArtiBour} and after, is only based on properties of
log-concave distributions.
\dumou

 We have chosen a very elementary expository style. We shall give fully
detailed proofs, except in the first two introductory sections. Most
readers will know the contents of these sections and may start by reading
Section~\ref{SteinsResults}. Some may be happy though to see a gentle
introduction to a few points they are less familiar with. Our choice of
topics in these two first sections owes a lot to Stein's monograph
\emph{Topics in harmonic analysis}~\cite{SteinTHA}. In the next sections,
we have chosen to recall and usually follow the methods from the original
papers. This leads sometimes to unnecessary complications, but we shall try
to give hints to other possibilities.
\dumou

 We believe that most of our notation is standard. We write
$\lfloor x \rfloor$, $\lceil x \rceil$ for the \emph{floor} and
\emph{ceiling} of a real number $x$, integers verifying
$x - 1 < \lfloor x \rfloor \le x \le \lceil x \rceil < x + 1$. We pay 
a special attention to constants independent of the dimension, for instance
those appearing in results about martingale inequalities, Riesz transforms,
and try to keep specific letters for these constants throughout the paper,
such as $c_p$, $\rho_p,\ldots$\ We use the letter $\kappa$ to denote a 
\og universal\fge constant that does not deserve to be remembered. Most
often in our Notes, \og we\fge is a two-letter abbreviation for \og the
author\fg, namely, Stein, Bourgain, Carbery, M\"uller and several
others\dots\ We include an index and a notation index.
\dumou

\section{General dimension free inequalities, first part%
\label{Gdfi}}%

\noindent
This first section is devoted to general facts obtained by probabilistic
methods, or merely employing the probabilistic language. We begin by
reviewing the basic definitions. The functions here are real or complex
valued, or they take values in a finite dimensional real or complex linear
space $F$ equipped with a norm denoted by $|x|$, for every vector 
$x \in F$. If $\Omega$ is a set, a \emph{$\sigma$-field\label{SigmaFie} 
$\ca G$} of subsets of $\Omega$ is a family of subsets that is closed under
countable unions $\bigcup_{n \in \N} A_n$, closed under taking complement
$A \mapsto A^c$, and such that $\emptyset \in \ca G$. If $\Omega$ is a set
and $\ca G$ a $\sigma$-field of subsets of~$\Omega$, one says that a
function $g$ on~$\Omega$ is \emph{$\ca G$-measurable} when for every Borel
subset~$B$ of the range space, the inverse image $g^{-1}(B)$, also denoted
by
\[
 \{ g \in B \} := \{\omega \in \Omega : g(\omega) \in B\},
\]
belongs to the collection $\ca G$. 
\dumou

 A \emph{probability space}\label{ProbaSpa} 
$(\Omega, \ca F, P)$ consists of a set $\Omega$, a $\sigma$-field $\ca F$
of subsets of $\Omega$ and a \emph{probability measure} $P$ on 
$(\Omega, \ca F)$, {\it i.e.}, a nonnegative $\sigma$-additive measure on
$(\Omega, \ca F)$ such that $P(\Omega) = 1$. If a function $f$ is 
$\ca F$-measurable (we say then that $f$ is a 
\emph{random variable}\label{RandoVar}) 
and if $f$ is $P$-integrable, the
\emph{expectation of $f$}\label{Expecte} 
is the integral of $f$ with respect to $P$, denoted by
\[
 \E f := \int_\Omega f(\omega) \, \d P(\omega).
\]
Random variables $(f_i)_{i \in I}$ on $(\Omega, \ca F, P)$ are
\emph{independent}\label{IndRaVar}
if for any finite subset $J \subset I$, one has 
$\E \bigl( \prod_{j \in J} h_j \circ f_j \bigr)
 = \prod_{j \in J} \E (h_j \circ f_j)$ for all nonnegative Borel functions
$(h_j)_{j \in J}$ on the range space. The 
\emph{distribution}\label{Distribu} 
of the random variable $f$ with values in $Y = \R$, $\C$ or $F$ is the
image probability measure $\mu = f_\# P$,\label{PushForward} 
defined on the Borel $\sigma$-field $\ca B_Y$ of $Y$ by letting 
$\mu(B) = P \bigl( \{ f \in B \} \bigr)$ for every $B \in \ca B_Y$. If
$\mu$ is a distribution on the Euclidean space $F$, the
\emph{marginals}\label{Margi} 
of $\mu$ on the linear subspaces $F_0$ of $F$ are the distributions
$\mu_{F_0}$ obtained from $\mu$ as images by orthogonal projection,
\textit{i.e.}, one sets $\mu_{F_0} = (\pi_0)_\# \mu$ where $\pi_0$ is the
orthogonal projection from $F$ onto~$F_0$. If $f$ is $F$-valued and if
$\mu$ is the distribution of~$f$, then $\mu_{F_0}$ is that 
of~$\pi_0 \circ f$.
\dumou

 If $\ca G$ is a sub-$\sigma$-field of $\ca F$, the \emph{conditional
expectation}\label{CondExpe} 
on $\ca G$ of an integrable function $f$ is the unique element
$\E (f \ms1 | \ms1 \ca G)$ of $L^1(\Omega, \ca F, P)$ possessing a 
$\ca G$-measurable representative $g$ such that 
\[
   \E (\gr 1_A f) 
 = \E (\gr 1_A g) 
 = \E \bigl( \gr 1_A \E (f \ms1 | \ms1 \ca G) \bigr)
\]
for every set $A \in \ca G$, where $\gr 1_A$ denotes the \emph{indicator
function}\label{IndiFunc} 
of $A$, equal to $1$ on~$A$ and $0$ outside. It follows that
\[
 \E (h f) = \E \bigl( h \E (f \ms1 | \ms1 \ca G) \bigr),
 \ms{10} \hbox{and actually} \ms{10}
 \E (h f \ms1 | \ms1 \ca G) = h \E (f \ms1 | \ms1 \ca G)
\]
for every bounded $\ca G$-measurable scalar function $h$ on $\Omega$. When
$f$ is scalar and belongs to $L^2(\Omega, \ca F, P)$, the conditional
expectation of $f$ on $\ca G$ is the orthogonal projection of $f$ onto the
closed linear subspace $L^2(\Omega, \ca G, P)$ of $L^2(\Omega, \ca F, P)$
formed by $\ca G$-measurable and square integrable functions. When $A$ is an
\emph{atom}\label{Atom} 
of $\ca G$, \textit{i.e.}, a minimal non-empty element of $\ca G$, and if
$P(A) > 0$, the value of $\E (f \ms1 | \ms1 \ca G)$ on the atom $A$ is the
average of $f$ on~$A$, hence
\[
   \E (f \ms1 | \ms1 \ca G)(\omega)
 = \frac 1 {P(A)} \int_A f(\omega') \, \d P(\omega'),
 \quad \omega \in A.
\]

 The conditional expectation operator $\E( \cdot \ms1 | \ms 1 \ca G)$ is
linear, and \emph{positive},\label{Positi} 
\textit{i.e.}, it sends nonnegative functions to nonnegative functions. It
follows that we have the inequality
$\varphi( \E (f \ms1 | \ms1 \ca G) )
 \le \E ( \varphi(f) \ms1 | \ms1 \ca G)$ when the real-valued function
$\varphi$ is convex on the range space of $f$. In particular, one has that
$\bigl| \E (f \ms1 | \ms1 \ca G) \bigr|
 \le \E \bigl( |f| \ms1 \big| \ms1 \ca G \bigr)$, and
\[
     \bigl\| \E (f \ms1 | \ms1 \ca G) \bigr\|_{L^p(\Omega, \ca F, P)}
 \le \bigl\| f \bigr\|_{L^p(\Omega, \ca F, P)},
 \quad
 1 \le p < +\infty.
\]
The inequality is true also when $p = +\infty$, it is easy and treated
separately.

\subsection{Doob's maximal inequality%
\label{MaxiDoob}}

\noindent
A (discrete time) \emph{martingale} on a probability space 
$(\Omega, \ca F, P)$ consists of a \emph{filtration},\label{Filtra} 
\textit{i.e.}, an increasing sequence $(\ca F_k)_{k \in I}$ of
sub-$\sigma$-fields of~$\ca F$ indexed by a subset $I$ of $\Z$, and of a
sequence $(M_k)_{k \in I}$ of integrable functions on $\Omega$ such that
for all $k, \ell \in I$ with $k \le \ell$, one has
\[
 M_k = \E \bigl( M_\ell \ms1 \big| \ms1 \ca F_k \bigr).
\]
Notice that each $M_k$, $k \in I$, is $\ca F_k$-measurable. If $I$ has a
maximal element~$N$, the martingale is completely determined by its last
element $M_N$, since we have then 
that~$M_k = \E \bigl( M_N \ms1 \big| \ms1 \ca F_k \bigr)$ for every 
$k \in I$. In the case of a finite field $\ca F_k$, the martingale condition
means that the value of $M_k$ on each atom of $\ca F_k$ is the average of
the values of $M_\ell$ on that atom, for every $\ell \in I$ with $\ell > k$.
Clearly, any subsequence $(M_k)_{k \in J}$, $J \subset I$, is a martingale
with respect to the filtration $(\ca F_k)_{k \in J}$. 
\dumou

 Let us consider a finite martingale $(M_k)_{k=0}^N$ on 
$(\Omega, \ca F, P)$, with respect to a filtration $(\ca F_k)_{k=0}^N$.
This martingale can be real or complex valued, or may take values in a
finite dimensional normed space $F$. We introduce the \emph{maximal process}
$(M^*_k)_{k=0}^N$,\label{MaxiProc} 
which is defined by $M^*_k = \max_{\ms1 0 \le j \le k} |M_j|$ for 
$k = 0, \ldots, N$. In the vector-valued case, $|M_j|$ is the function
assigning to each $\omega \in \Omega$ the norm of the vector 
$M_j(\omega) \in F$. We also employ the lighter notation $\|M\|_p$ for the
norm $\|M\|_{L^p}$ of a function $M$ in $L^p(\Omega, \ca F, P)$, when 
$1 \le p \le +\infty$.
\dumou

\begin{thm}[Doob's inequality]%
\label{TheoDoob}
Let\/ $(M_k)_{k=0}^N$ be a martingale (real, complex or vector-valued). For
every real number $c > 0$, one has that
\[
     c \ms2 P \bigl( \{ M^*_N > c \} \bigr) 
 \le \int_{ \{M^*_N > c\} } |M_N| \, \d P.
\]
Furthermore, for every $p \in (1, +\infty]$, one has when 
$M_N \in L^p(\Omega, \ca F, P)$ that
\begin{equation}
 \|M^*_N\|_p \le \frac p {p - 1} \ms2 \|M_N\|_p.
 \label{DoobLp}
\end{equation}
\end{thm}
\dumou

\begin{proof}
We cut the set $\{M^*_N > c\}$ into disjoint events $A_0, \ldots, A_N$,
corresponding to the first time $k$ when $|M_k| > c$. Let 
$A_0 = \{ \ms1 |M_0| > c \}$ and for each integer~$k$ between $1$ and~$N$,
let $A_k$ denote the set of $\omega \in \Omega$ such that 
$|M_k(\omega)| > c$ and $M^*_{k-1}(\omega) \le c$. On the set~$A_k$, we
have $|M_k| > c$, and $A_k$ belongs to the $\sigma$-field~$\ca F_k$ since
$|M_k|$ and $M^*_{k-1}$ are $\ca F_k$-measurable, hence
\begin{align*}
     c \ms2 P(A_k)
 &\le \int_{A_k} |M_k| \, \d P 
  =  \int_{A_k} 
      \bigl| 
       \E \bigl( M_N \ms1 \big| \ms1 \ca F_k \bigr)
      \bigr| \, \d P 
 \\
 &\le \int_{A_k} 
       \E \bigl( |M_N| \ms1 \big| \ms1 \ca F_k \bigr) \, \d P 
  =  \int_{A_k} |M_N| \, \d P.
\end{align*}
\dumou\noindent
On the other hand, we see that
$
 \{M^*_N > c\} = \bigcup_{k=0}^N A_k
$,
union of pairwise disjoint sets, therefore
\begin{equation}
     c \ms2 P \bigl( \{ M^*_N > c \} \bigr) 
  =  \sum_{k=0}^N c \ms2 P(A_k)
 \le \sum_{k=0}^N \int_{A_k} |M_N| \, \d P
  =  \int_{ \{ M^*_N > c  \} } |M_N| \, \d P.
 \label{DoobProof}
\end{equation}
\dumou

 The result for $L^p$ when $1 < p < +\infty$ follows. For each value 
$t > 0$, we apply~\eqref{DoobProof} with $c = t$, we use Fubini's theorem
and H\"older's inequality, obtaining thus
\begin{align*}
     \E \bigl( (M^*_N)^p \bigr)
  &= \E \Bigl( \int_0^{M^*_N} p \ms1 t^{p-1} \, \d t \Bigr)
   = \int_0^{+\infty} p \ms1 t^{p-1} 
                       P \bigl( \{ M^*_N > t \} \bigr) \, \d t
 \\
 &\le \int_0^{+\infty} p \ms1 t^{p-2} \ms1
       \E \bigl( \gr 1_{ \{M^*_N > t\} } \ms2 |M_N| \bigr) \, \d t
   = \E \Bigl( \frac p {p - 1}  (M^*_N)^{p-1} \ms2 |M_N| \Bigr)
 \\
 &\le \frac p {p - 1} 
       \bigl( \E \bigl( (M^*_N)^p \bigr) \bigr)^{1 - 1/p} 
        \bigl(\E \bigl( |M_N|^p \bigr) \bigr)^{1 / p},
\end{align*}
hence
$
     \|M^*_N\|_p
 \le p \ms1 (p - 1)^{-1} \ms1 \|M_N\|_p
$.
The case $p = +\infty$ is straightforward. 
\end{proof}

\begin{remN}%
\label{InfiniProba} 
In some contexts, it is useful to observe that the notion of conditional
expectation on a sub-$\sigma$-field $\ca F_0$ of $\ca F$ remains well
defined if we have a possibly infinite measure $\mu$ on $(\Omega, \ca F)$,
but which is \emph{$\sigma$-finite} on~$\ca F_0$, in other words,
if~$\Omega$ can be split in countably many sets $A_i$ in $\ca F_0$ such that
$\mu(A_i) < +\infty$ for each~$i$. If this condition is fulfilled by $\mu$
and by the smallest sub-$\sigma$-field $\ca F_0$ of a filtration 
$(\ca F_k)_{k=0}^N$, we can also speak about martingales with respect to
the infinite measure~$\mu$, and Theorem~\ref{TheoDoob} remains true with the
same proof, simply replacing the words \og probability of an event\fge by 
\og measure of a set\fg. 

 We can always consider the orthogonal projection $\pi_0$ from 
$L^2(\Omega, \ca F, \mu)$ onto $L^2(\Omega, \ca F_0, \mu)$, but
$L^2(\Omega, \ca F_0, \mu) = \{0\}$ when $\ca F_0$ does not contain any set
with finite positive measure. On the other hand, when $A \in \ca F_0$ has
finite measure, the formula $\pi_0(\gr 1_A f) = \gr 1_A \pi_0(f)$
allows one to work on $A$ as in the case of a probability measure.
\end{remN}

\subsection{The Hopf maximal inequality%
\label{MaxiHopf}}

\noindent
We are given a measure space $(X, \Sigma, \mu)$ and a linear operator $T$
from $L^1(X, \Sigma, \mu)$ to itself. We shall only consider
$\sigma$-finite measures throughout these Notes, and we work in this
section with the space $L^1(X, \Sigma, \mu)$ of real-valued functions. We
assume that $T$ is positive and nonexpansive, which means that for every
nonnegative function $g \in L^1(X, \Sigma, \mu)$, $T g$ is nonnegative, and
that the norm of~$T$ is~$\le 1$. We can sum up these two properties by
saying that when $g \ge 0$, then~$T g \ge 0$ and 
$\int_X T g  \, \d \mu \le \int_X g \, \d \mu$. 
\dumou

 Let us consider a function $f$ in $L^1(X, \Sigma, \mu)$, and for every
integer $k \ge 0$ let
\[
 S_k(f) = f + T f + T^{2} f + \cdots + T^{k} f.
\]
If $N$ is a nonnegative integer, we set
$
 S_N^*(f) = \max \ms1 \{S_j(f) : 0 \le j \le N\}
$.

\begin{lem}[Hopf]%
\label{GarsiLem}
With the preceding notation, we have for every function
$f \in L^1(X, \Sigma, \mu)$ and $N \ge 0$ that
\[
 \int_{ \{S_N^*(f) > 0 \} } f \, \d \mu \ge 0.
\]
\end{lem}

\begin{proof}[Proof, after Garsia~\cite{Garsia}]
Let us simply write $S_k$ for $S_k(f)$ and $S^*$ for $S^*_N(f)$. By
definition, we have $S_k \le S^*$ for each integer $k \le N$; since $T$ is
positive and linear, we see that
\[
 T \ms1 S_k \le T \ms1 S^*,
 \ms{16} \hbox{and} \ms{16}
 S_{k+1} = f + T \ms1 S_k \le f + T \ms1 S^*.
\]
In order to get for $S_0 = f$ an inequality similar to
$S_{k+1} \le f + T \ms1 S^*$, we replace~$S^*$ by its nonnegative part
$S^{*\scriptscriptstyle +} = \max(S^*, 0) \ge S^*$. Using positivity, we
can write
\[
 S_0 = f \le f + T (S^{*\scriptscriptstyle +}),
 \ms{16}
 S_{k+1} \le f + T \ms1 S^* \le f + T (S^{*\scriptscriptstyle +}).
\]
Taking the supremum of $S_k\ms1$s for $0 \le k \le N$, we obtain the
crucial inequality
\begin{equation}
 S^* \le f + T (S^{*\scriptscriptstyle +}),
 \ms{10} \hbox{or} \ms{10}
 f \ge S^* - T (S^{*\scriptscriptstyle +}).
 \label{Cruci}
\end{equation}
Since $T$ is positive and nonexpansive on $L^1(X, \Sigma, \mu)$, we have
\[
   \int_{ \{S^* > 0\} } S^* \, \d \mu
 = \int_X S^{*\scriptscriptstyle +} \, \d \mu
 \ge \int_X T (S^{*\scriptscriptstyle +}) \, \d \mu
 \ge \int_{ \{S^* > 0\} } T (S^{*\scriptscriptstyle +}) \, \d \mu,
\]
and the result follows by~\eqref{Cruci}, because
\[
     \int_{ \{S^* > 0\} } f \, \d \mu 
 \ge \int_{ \{S^* > 0\} } 
      \bigl(  S^* - T (S^{*\scriptscriptstyle +}) \bigr) \, \d \mu
 \ge 0.
\qedhere
\]
\end{proof}

 We go on with the same linear operator $T$. For each integer $k \ge 0$,
let us define the $k\ms1$th average operator $a_k = a_{k, T}$ associated 
to~$T$ by writing
\[
   a_k(f) 
 = \frac { f + T f + \cdots + T^k f } {k+1}
 = \frac {S_k(f)} {k+1} \up,
 \quad f \in L^1(X, \Sigma, \mu).
\]
For each integer $N \ge 0$, let
$  a^*_N(f)
 = \max \ms1 \{a_j(f) : 0 \le j \le N\}
$.
It is clear that the set $\{a^*_N(f) > 0 \}$ coincides with the set
$\{S^*_N(f) > 0 \}$ which appears in Lemma~\ref{GarsiLem}.
\dumou

 We continue in a simplified setting where we also assume that $\mu$ is
finite and that $T \ms1 \gr 1 = \gr 1$. It follows that 
$a_k (\gr 1) = \gr 1$ for each $k \ge 0$ and $a_k(f - c) = a_k(f) - c$ for
every $c \in \R$, thus $a^*_N(f - c) = a^*_N(f) - c$. Lemma~\ref{GarsiLem}
yields
\[
     \int_{ \{ a^*_N(f - c) > 0 \} } (f - c) \, \d \mu 
  =  \int_{ \{ S^*_N(f - c) > 0 \} } (f - c) \, \d \mu 
 \ge 0.
\]
Equivalently, for every $f \in L^1(X, \Sigma, \mu)$, we have
\begin{equation}
     c \ms1 \mu \bigl( \{a^*_N(f) > c\} \bigr)
 \le \int_{ \{a^*_N(f) > c\} } f \, \d \mu,
 \qquad N \ge 0, \ms8 c \in \R.
 \label{HopfN}
\end{equation}

\begin{smal}
\noindent
This inequality makes sense also when $\mu$ is infinite. Note that if 
$c < 0$ and if $\mu$ is infinite, then 
$\mu \bigl( \{f \le c\} \bigr)
 \le \mu \bigl( \{ |f| \ge |c| \} \bigr) < +\infty$, the measure of
$\{a^*_N (f) > c\}$ is thus infinite and~\eqref{HopfN} is trivial. 
We can extend~\eqref{HopfN} to an infinite $\mu$ if there exists an
increasing sequence $(C_\ell)_{\ell \ge 0}$ of subsets of $X$ with finite
measure such that 
\begin{subequations}\label{QuasiMarko}
\SmallDisplay{\eqref{QuasiMarko}}%
\[
 T^j \gr 1_{C_\ell} \le \gr 1
  \ms{ 8} \hbox{for all} \ms6 j, \ell \ge 0,
  \ms{12}
 T^j \gr 1_{C_\ell} \lra_{\ell \to +\infty} \gr 1 
  \ms{ 8} \hbox{pointwise for each} \ms6 j \ge 0.
\]
\end{subequations}
\noindent
Let $c, \varepsilon > 0$ and abbreviate $\{ a^*_N(f) > t\}$ as $D(t)$, 
for~$t > 0$. Choose $c' > c$ such that
$\int_{D(c) \setminus D(c')} \bigl( 1 + |f| \bigr) \, \d \mu
 < \varepsilon$. Let $E(c', \ell)
 = \bigl\{
    \min_{\ms2 0 \le j \le N} T^j \gr 1_{C_\ell} \le c / c' 
   \bigr\}$,
choose a large $\ell$ such that $\mu(D(c') \setminus C_\ell) < \varepsilon$
and $\int_{E(c', \ell)} |f| \, \d \mu < \varepsilon$, then observe that
\[
         D(c')
 \subset
  \bigl\{ a^*_N \bigl( f - c' \ms1 \gr 1_{C_\ell} \bigr) > 0 \bigr\}
 \subset D(c) \cup E(c', \ell)
\]
and apply Lemma~\ref{GarsiLem} to $f - c' \ms1 \gr 1_{C_\ell}$.
The assumption~\eqref{QuasiMarko} is fulfilled when $T$ is an operator of
convolution with a probability measure on~$\R^n$, acting on~$L^1(\R^n)$.

\end{smal}

\noindent
For each function $f \in L^1(X, \Sigma, \mu)$, let us define
\[
   a^*(f) 
 = \sup_{k \ge 0} \ms1 a_k(f) 
 = \sup_{k \ge 0} \ms2 \frac {f + T f + \cdots + T^k f} {k+1}
 = \lim_{N \to +\infty} \ms2 a^*_N(f). 
\]
The set $\{ a^*(f) > c \}$ is the increasing union of the sets
$\{a^*_N(f) > c\}$, $N \ge 0$, so, passing to the limit by dominated
convergence, we deduce from~\eqref{HopfN} that
\begin{equation}
     c \ms2 \mu \bigl( \{ a^*(f) > c \} \bigr)
 \le \int_{ \{ a^*(f) > c \} } f \, \d \mu,
 \qquad c \in \R.
 \label{Hopf0}
\end{equation}
Following~\cite[Lemma~VIII.6.7]{DunfordSchwartz}, we now get a
variant of~\eqref{Hopf0}. Assume $c > 0$ in what follows. We define $f_c$ by
$f_c(x) = f(x)$ when $f(x) > c$ and $f_c(x) = 0$ otherwise, for $x \in X$.
Note that $f \le f_c + c$. If $a^*(f_c)(x) \le c$, then 
$f_c(x) = a_0(f_c)(x) \le c$ thus $f_c(x) = 0$ by construction. Hence $f_c$
vanishes outside $\{ a^*(f_c) > c \}$ and
\[
   \int_{ \{ f > c \} } f \, \d \mu
 = \int_X f_c \, \d \mu
 = \int_{ \{ a^*(f_c) > c \} } f_c \, \d \mu.
\] 
Using the positivity of $T$ and of $a_k$ for each $k \ge 0$, we infer
from $f_c \ge f - c$ that $a^*(f_c) \ge a^*(f - c) = a^*(f) - c$. Then,
by~\eqref{Hopf0} for $f_c$ and since $c > 0$, we get
\[
     \int_{ \{ f > c \} } f \, \d \mu
 \ge c \ms1 \mu \bigl( \{ a^*(f_c) > c \} \bigr)
 \ge c \ms1 \mu \bigl( \{ a^*(f) - c > c \} \bigr).
\]
Finally, we have obtained
\begin{equation}
     c \ms1 \mu \bigl( \{ a^*(f) > 2 \ms1 c \} \bigr)
 \le \int_{ \{ f > c \} } f \, \d \mu,
 \quad c > 0.
 \label{Hopf00}
\end{equation}
\dumou

 Let us define
$
   A^*(f) 
 = \sup_{k \ge 0} |a_k(f)|
 = \max (a^*(f), a^*(-f))
$.
Still assuming $c > 0$, we decompose the set 
$\{ A^*(f) > c \} = \{ a^*(f) > c \} \cup \{ a^*(-f) > c \}$ 
into three disjoint pieces,
$E_0 = \{ a^*(f) > c  \ms1\et\ms1  a^*(-f) \le c \}$,
$E_1 = \{ a^*(f) > c  \ms1\et\ms1  a^*(-f) > c \}$
and $E_2 = \{ a^*(f) \le c  \ms1\et\ms1  a^*(-f) > c \}$.
According to~\eqref{Hopf0} we have
\begin{align}
      c \ms2 \mu \bigl( \{ A^*(f) > c \} \bigr)
 &\le c \ms1 \mu \bigl( \{ a^*(f) > c \} \bigr)
       + c \ms1 \mu \bigl( \{ a^*(-f) > c \} \bigr)
 \notag
 \\
 &\le \int_{ \{ a^*(f) > c \} } f \, \d \mu
       + \int_{ \{ a^*(-f) > c \} } (-f) \, \d \mu
 \notag
 \\
 & =  \int_{ E_0 } f \, \d \mu
       + \int_{ E_2 } (-f) \, \d \mu
  \le \int_{ \{ A^*(f) > c \} } |f| \, \d \mu,
 \label{Hopf01}
\end{align}
noting that the integrals of $f$ and $-f$ on $E_1$ cancel each other. In
the same way, we can get from~\eqref{Hopf00} the variant form
$
     c \ms2 \mu \bigl( \{ A^* f > 2 \ms1 c \} \bigr)
 \le \int_{ \{ |f| > c \} } |f| \, \d \mu
$.
Notice that the latter \og variant form\fge will be inherited by any
linear operator $S$ satisfying that $|S^k f| \le T^k |f|$ for every 
$k \ge 0$, and see Remark~\ref{DSRemark}.
\dumou

 When $1 < p < +\infty$, we deduce from~\eqref{Hopf01} the $L^p$ inequality
\begin{equation}
     \Bigl\| \sup_{k \ge 0} 
      \frac {|f + T f + \cdots + T^k f|} {k + 1}
     \Bigr\|_p
 \le \frac p {p - 1} \ms2 \|f\|_p
 \label{GarsiaLp}
\end{equation}
as we have seen with Doob's inequality~\eqref{DoobLp}, while the variant
form leads to a constant $2 \ms2 \bigl(p / (p - 1) \bigr)^{1/p}$ which is
larger than $p / (p-1)$ for every $p > 1$.
\dumou
\dumou

 Let now $(T_t)_{t \ge 0}$ be a semi-group of linear operators 
on~$L^1(X, \Sigma, \mu)$,\label{SemiGrou} 
\textit{i.e.}, operators satisfying $T_{s+t} = T_s \circ T_t$ for all 
$s, t \ge 0$. We assume in addition that each~$T_t$ is positive and
nonexpansive on $L^1$. We also assume that $T_t$ is actually defined
on $L^1(X, \Sigma, \mu) + L^\infty(X, \Sigma, \mu)$ and that
$T_t \ms1 \gr 1 = \gr 1$ for every $t \ge 0$. This implies that $T_t$ is
continuous from $L^\infty$ to $L^\infty$, with norm~$1$. By interpolation,
we get that the norm $\|T_t\|_{p \rightarrow p}$ on $L^p$, for 
$p \in [1, +\infty]$, is~$\le 1$. Suppose that the semi-group is
\emph{strongly continuous} on~$L^1$, which means that 
$\|f - T_t f\|_1 \rightarrow 0$ as $t \rightarrow 0$, for each $f \in L^1$.
Together with our other assumptions, it follows that $t \mapsto T_t f$ is
continuous from $[0, +\infty)$ to~$L^p$ for every function $f \in L^p$ and
$1 \le p < +\infty$. For $f \in L^p(X, \Sigma, \mu)$ let
\[
   a^* f 
 = \sup_{t > 0} \ms2
    \frac 1 t \int_0^t T_s \ms1 f \, \d s,
 \ms{18}
   A^* f 
 = \sup_{t > 0} \ms2
    \Bigl| \frac 1 t \int_0^t T_s \ms1 f \, \d s \Bigr|,
\]
where the supremum can be defined as an \emph{essential supremum}, see the
discussion in Section~\ref{DefiMaxiFunc}. Yet, for the main examples of
semi-groups of interest in these Notes, namely, the Gaussian
semi-group or the Poisson semi-group on $\R^n$, the function
$t \mapsto (T_t f)(x)$ is continuous on $(0, +\infty)$ for each fixed 
$x \in \R^n$ and $f \in L^1(\R^n)$, so $a^* f$ and $A^* f$ have then a
well defined pointwise value, possibly~$+\infty$.
\dumou

 Suppose now that the measure $\mu$ is finite (or that a continuous analog
of~\eqref{QuasiMarko} is satisfied). When $1 < p < +\infty$, we obtain
from~\eqref{GarsiaLp} the $L^p$ inequality
\begin{equation}
     \bigl\| A^* f \bigr\|_p
 \le \frac p {p - 1} \ms2 \|f\|_p.
 \label{HopfLp}
\end{equation}
If $T_t$ is positive and $T_t \gr 1 = \gr 1$, the case $p = +\infty$
in~\eqref{HopfLp} is clear.

\begin{smal}
\noindent 
Since $t \mapsto a(t, f) := t^{-1} \int_0^t T_s f \, \d s$ is continuous
from $(0, +\infty)$ to $L^p$, we can reach any $a(t, f)$, $t > 0$, as an
almost everywhere limit of a sequence $(a(t_j, f))_{j \ge 0}$, where each
$t_j$ is rational and~$> 0$. It follows that $A^* f$ can be defined as the
supremum of $|a(t, f)|$ for $t > 0$ rational. For all integers
$k \ge 0$ and $n \ge 1$, observe that
\[
   a \Bigl( \frac {k+1} n \up, \ms1 f \Bigr)
 = \frac n {k+1} \sum_{i=0}^k \int_{i/n}^{(i+1)/n} T_s f \, \d s
 = \biggl( \frac {\sum_{i=0}^k T_{i/n} } {k+1} \biggl)
    \Bigl( n \int_0^{1/n} T_s f \, \d s \Bigr).
\]
Letting $f_n = n \int_0^{1/n} T_s f \, \d s = a(1/n, f)$ and 
$T = T_{1/n}$ we see that
\[
   a \Bigl( \frac {k+1} n \up, \ms1 f \Bigr)
 = \frac {f_n + T f_n + \cdots + T^k f_n} {k+1}
 = a_{k, T} (f_n).
\]
Let $Q_n$ be the set of positive multiples of~$1/n$. By~\eqref{GarsiaLp}
applied to $T_{1/n}$ and $f_n$, and because $a(1/n, \cdot)$ is an average
of operators with norm $\le 1$ on $L^p$, we get
\[
     \Bigl\| 
      \sup_{t \in Q_n} \bigl| a(t, f) \bigr|
     \Bigr\|_p
  =  \Bigl\| 
      \sup_{j \ge 1} \ms1 \bigl| a(j / n, f) \bigr|
     \Bigr\|_p
 \le \frac p {p-1} \ms2 \bigl\| a(1/n, f) \bigr\|_p
 \le \frac p {p-1} \ms2 \|f\|_p.
\]
We see that $Q_m \subset Q_{m \ms1 n}$ for all $m, n \ge 1$. The sets $Q_n$
corresponding to $n = \ell\ms{0.8}!$ for $\ell \ge 1$ are increasing with
$\ell$, and they cover the set of positive rationals. We can conclude by
noticing that $A^* f$ is the increasing limit of 
$\sup_{t \in Q_{\ell\ms{0.4}!}} \ms1 |a(t, f)|$.
\end{smal}

\noindent
Applying~\eqref{Hopf0} we can obtain a version
of Hopf's maximal inequality as
\[
     c \ms{2.5} \mu \bigl( \{ a^* f > c \} \bigr)
 \le \int_{ \{ a^* f > c \} } f \, \d \mu,
 \quad
 c \in \R, \ms8
 f \in L^1(X, \Sigma, \mu),
\]
and from~\eqref{Hopf01}, we have
$
     c \ms{2.5} \mu \bigl( \{  A^* f > c \} \bigr) 
 \le \int_{ \{ A^* f > c \} } |f| \, \d \mu
$
when $c > 0$.

\begin{smal}
\noindent
By the preceding remark about the sets $Q_{\ell\ms{0.9}!}$ it is enough to
prove the inequality with
$a^*_n := \sup_{t \in Q_n} a(t, f) = \sup_{k \ge 0} a_{k, T_{1/n}}(f_n)$ 
replacing $a^* f$, with $n \ge 1$ arbitrary and with a vanishing error term.
By~\eqref{Hopf0} we have
$
     c \ms1 \mu \bigl( \{ a^*_n > c \} \bigr)
 \le \int_{ \{ a^*_n > c \} } f_n
$.
Since the semi-group $(T_t)_{t \ge 0}$ is strongly continuous, we know
that $\|f_n - f\|_1 \to 0$ and we can conclude because
$\int_{ \{ a_n^* f > c \} } f_n \, \d \mu
 - \int_{ \{ a_n^* f > c \} } f \, \d \mu$ tends to zero.
\end{smal}

\noindent
We have made here assumptions more restrictive than those of the
Hopf--Dunford--Schwartz statement~\cite[Chapter~VIII]{DunfordSchwartz}
praised by Stein~\cite{SteinTHA}, which does not assume $T_t$ positive, nor
$\mu$ finite and $T_t \gr 1 = \gr 1$. Theorem~\ref{HDS} below contains
Lemma~VIII.7.6 and Theorem~VIII.7.7 from~\cite{DunfordSchwartz} in a
slightly simplified form (the set $U$ there has only one element here).
The semi-group $(T_t)_{t \ge 0}$ on $L^1(X, \Sigma, \mu)$ is said to be
\emph{strongly measurable} if, for each $f$ in $L^1(X, \Sigma, \mu)$, the
mapping $t \mapsto T_t f \in L^1(X, \Sigma, \mu)$ is measurable with
respect to the Lebesgue measure on $[0, +\infty)$.

\begin{thm}[\cite{DunfordSchwartz}]\label{HDS}
Let\/ $(T_t)_{t \ge 0}$ be a strongly measurable semi-group on
the space $L^1(X, \Sigma, \mu)$, with\/ $\|T_t\|_{1 \to 1} \le 1$ and\/ 
$\|T_t\|_{\infty \to \infty} \le 1$ for all $t \ge 0$. For 
every function $f \in L^1(X, \Sigma, \mu)$ and every $c > 0$ one has
\[
     c \ms2 \mu \bigl( \{ A^* f > 2 \ms1 c \} \bigr)
 \le \int_{ \{ |f| > c \} } |f| \, \d \mu.
\]  
If\/ $1 < p < +\infty$ and $f \in L^p(X, \Sigma, \mu)$, the function 
$A^* f$ is in $L^p(X, \Sigma, \mu)$ and 
\[
       \|A^* f\|_p 
 \le 2 \ms1 \Bigl( \frac p {p-1} \Bigr)^{1/p} \|f\|_p.
\]
\end{thm}

\begin{remN}\label{DSRemark}
In Dunford--Schwartz, Chapter~VIII, Section~VIII.6, the authors consider
first a linear operator $T$ acting from $L^1$ to $L^1$ with norm~$\le 1$
and also acting from $L^\infty$ to $L^\infty$ with norm $\le 1$; in this
discrete parameter case, they study
\[
   A_T^* f
 = \sup_{n \ge 1}
    \ms 2 \frac 1 n \ms1 \Bigl| \sum_{k=0}^{n-1} T^k f \Bigr|,
\]
before going to the continuous setting of a semi-group $(T_t)_{t \ge 0}$.
One of the steps in their proof consists in introducing a \emph{positive}
operator $P$ which acts from $L^1$ to $L^1$ and from $L^\infty$ to
$L^\infty$, with norm $\le 1$ in both cases, and such that
\[
 \forall n \ge 0,
 \ms{12}
 |T^n f| \le P^n ( |f| ),
 \quad 
 f \in L^\infty \ns1 \cap \ms1 L^1.
\]
This step is easy when the measure is the uniform measure on a finite set.
The assumptions imply that~$T$ is given by a matrix $(t_{i, j})$ such that
the sum of absolute values in each row and in each column is $\le 1$. It is
then enough to take $P$ to be the matrix with entries $p_{i, j}$ equal to
the absolute values $|t_{i,j}|$ of the entries of $T$.
\end{remN}

\subsection{From martingales to semi-groups, via an argument of Rota%
\label{RotaRem}}

\noindent
The arguments in this section, due to Rota~\cite{Rota}, are presented in a
more sophisticated manner in Stein's book~\cite[Chap.~4, \S~4]{SteinTHA}.
We consider a Markov chain $X_0, \ldots, X_N$ with transition matrix $P$,
assumed to be \emph{symmetric}. We suppose for simplicity that the state
space $\ca E$ is finite, with cardinality~$Z$. For every $e_0 \in \ca E$,
we have
\[
 \sum_{e \in \ca E} P(e_0, e) = 1.
\] 
For each integer $k$ such that $0 \le k < N$ and for all 
$e_0, e_1 \in \ca E$, the probability that $X_{k+1} = e_1$ knowing that
$X_k = e_0$ is given by the entry $P(e_0, e_1)$ of the matrix $P$. This
statement introduces implicitly the \emph{Markov property}, which loosely
speaking, prescribes that what happens after time $k$ depends only on what
we know at the instant~$k$, regardless of the past positions at times 
$j < k$. For each integer $j \ge 2$, the power $P^j$ of the matrix $P$
controls the moves in $j$ successive steps, the entry $P^j(e_0, e)$ giving
the probability of moving from $e_0$ to $e$ in exactly~$j$ steps. If $Q$ is
a transition matrix and $f$ a scalar function on $\ca E$, we introduce the
notation
\[
 (Q f)(x) = \sum_{y \in \ca E} Q(x, y) f(y),
 \quad x \in \ca E.
\]
When applied to a power $P^j$, the notation $P^j f$ corresponds to the
semi-group notation $P_t f$, with $j \in \N$ replacing $t \ge 0$. If the
transition matrix $Q$ is symmetric, hence bistochastic, and 
if~$1 \le p \le +\infty$, convexity implies that $\|Q f\|_p \le \|f\|_p$
with respect to the uniform measure on~$\ca E$. Let $f$ be a function on
$\ca E$ and let $j, k$ be two nonnegative integers with $j + k \le N$. If
we fix $x_0 \in \ca E$, the mean of the values $f(y)$, when the chain makes
$j$ steps from the position $x_0$ at time $k$ to the position $y$ at time 
$k + j$, is equal to~$(P^j f)(x_0)$.
\dumou

 A simple but important symmetric example is that of the \emph{Bernoulli
random walk}\label{BernouWal} 
on~$\Z$, where for all $x, y \in \Z$ we have $P(x, y) = 1/2$ when 
$|x - y| = 1$, and $P(x, y) = 0$ otherwise. This is not a finite example,
but it can be \og approximated\fge by considering on the finite set 
$\ca E_N = \{-N, \ldots, N\}$, for $N$ large, the modified matrix $P_N$
which still has $P_N(x, y) = 1/2$ when $|x - y| = 1$, for 
$x, y \in \ca E_N$, but where $P_N(N, N) = P_N(-N, -N) = 1/2$. One can also
consider the Bernoulli random walk on $\Z^n$, for which $P(x, y) = 2^{-n}$
when $|x_i - y_i| = 1$ for all coordinates $x_i, y_i$, $i = 1, \ldots, n$,
of the points $x = (x_1, \ldots, x_n)$ and $y = (y_1, \ldots, y_n)$
in~$\Z^n$.
\dumou
 
 Assume that the distribution of the initial position $X_0$ is uniform,
that is to say, that $P(X_0 = e_0) = 1 / Z$ for every $e_0 \in \ca E$. Then
for each $e_1 \in \ca E$, we have
\[
   P(X_1 = e_1) 
 = \sum_{e \in \ca E} P(X_0 = e \et X_1 = e_1)
 = \sum_{e \in \ca E} \frac 1 Z \ms2 P(e, e_1)
 = \frac 1 Z \ms1 \sum_{e \in \ca E} P(e_1, e)
 = \frac 1 Z \ms1 \up,
\]
\dumou
\noindent
since the matrix $P$ is symmetric. The distribution of the position $X_1$
of the chain at time $i = 1$ remains the uniform distribution, as well as
that of $X_2, \ldots, X_N$. The uniform distribution is \emph{invariant\/}
under\label{InvariMe} 
the action of $P$. Recalling the meaning of the transition matrix in terms
of conditional probability, using Markov's property and letting
$A_{N-1} = \{ X_0 = e_0  \et  X_1 = e_1  \et \ldots
       \et X_{N-1} = e_{N-1} \}$, we have that
\begin{align*}
 E:=  & \ms4 P( X_0 = e_0  \et  X_1 = e_1  \et \ldots \et X_N = e_N ) 
 \\
   =  & \ms4 P(A_{N-1} \et X_N = e_N) 
 =   \ms1 P(A_{N-1}) \ms3 P(X_N = e_N \,|\, A_{N-1} )
 \\
 = & \ms4 P(A_{N-1})
   \ms2 P(X_N = e_N \,|\, X_{N-1} = e_{N-1} )
 = \ms1 P(A_{N-1}) \ms2 P(e_{N-1}, e_N).
\end{align*}
We may go on, and by the symmetry property of the matrix we get
\begin{align*}
 E = & \ms4 \cdots
 = \frac 1 Z \ms2 P(e_0, e_1) P(e_1, e_2) \ldots
    P(e_{N-2}, e_{N-1}) P(e_{N-1}, e_N)
 \\
 = & \ms4 \frac 1 Z \ms2 P(e_N, e_{N-1}) P(e_{N-1}, e_{N-2}) \ldots
    P(e_2, e_1) P(e_1, e_0)
 \\
 = & \ms4 P(X_N = e_0  \et  X_{N-1} = e_1  \et \ldots
       \et X_1 = e_{N-1}  \et X_0 = e_N).
\end{align*}
We see that the \og reversed\fge chain has the same behavior as that of
the original chain. Since the matrix is symmetric, we certainly have,
whatever the distribution of $X_0$ can be, that the probability to arrive
at a fixed~$y_0$ at time $N$, starting from an arbitrary point $x$ at time
$k = N - j$, is given by $P^j(x, y_0) = P^j(y_0, x)$, the probability of
moving from $y_0$ at time $0$ to $x$ at time $j$. But under the invariant
distribution, we can say more: if $g$ is a function on~$\ca E$, the
\emph{mean} of the values $g(x)$ on all trajectories starting from $x$ at
time $k$ and arriving at $y_0$ at time~$N$ is equal to $(P^j g)(y_0)$.
Clearly, this statement is not true in general, since this mean value
depends on the distribution of $X_k$, hence on that of $X_0$. Under the
uniform distribution, we see by reversing the chain that the preceding mean
is equal to the mean of $g(x)$, when starting from $y_0$ at time $0$ and
arriving at $x$ at time $j$, namely, this mean is equal to $(P^j g)(y_0)$.

 Let us describe the situation more formally. Let $\Omega = \ca E^{N+1}$
denote the space of all possible trajectories 
$(e_0, e_1, \ldots, e_N) \in \ca E^{N+1}$ for the chain. On this model
space~$\Omega$ and for $k = 0, \ldots, N$, we set
\[
 X_k(\omega) = \omega_k \in \ca E,
 \qquad
 \omega = (\omega_0, \ldots, \omega_N) \in \ca E^{N+1}.
\]
It is easy to determine the probability measure $\gr P$ on $\Omega$ that
corresponds to the behavior of our Markov chain under the invariant
distribution. For each singleton 
$\{\omega\} =  \{(\omega_0, \ldots, \omega_N)\}$ in $\ca P(\Omega)$, we must
have that
\[
   \gr P \bigl( \{(\omega_0, \ldots, \omega_N)\} \bigr)
 = \frac 1 Z \ms2 P(\omega_0, \omega_1)
               \ms1 P(\omega_1, \omega_2) \ldots
               \ms1 P(\omega_{N-1}, \omega_N).
\]
For $k = 0, \ldots, N$, let $\ca F_k$ denote the finite field of subsets
of $\Omega$ whose atoms $A$ are of the following form: to any 
$e_0, \ldots, e_k$ fixed in $\ca E$ we associate $A_{\gr e} \in \ca F_k$
defined by
\[
   A 
 = A_{\gr e}
 = \{ \omega = (\omega_0, \ldots, \omega_N) : 
       \omega_j = e_j, \ms3 0 \le j \le k \} \in \ms{0.3} \ca F_k,
 \quad
 \gr e = (e_0, \ldots, e_k).
\]
This $\ca F_k$ is the \og field of past events\fge at time $k$, it increases
with~$k$. Let $\ca G_k$ denote the field of events occurring precisely at
time $k$, whose atoms $B$ are of the form
\[
 B = \{ \omega = (\omega_0, \ldots, \omega_N): 
          \omega_k = e_k \} \in \ms{0.3} \ca G_k.
\]
Clearly, we have $\ca G_k \subset \ca F_k$. A function on $\Omega$ which is
$\ca G_k$-measurable depends only on the coordinate $\omega_k$, and is thus
of the form $g(X_k)$ with $g$ a function on~$\ca E$. If $f$ is a function
on $\ca E$, the Markov property yields
\[
   \E (f(X_N) \ms1|\ms1 \ca F_k) 
 = \E (f(X_N) \ms1|\ms1 \ca G_k)
 = g(X_k)
\]
where $g(x) = (P^{N - k} f)(x)$ for every $x \in \ca E$. The preliminary
discussion shows that
\begin{equation}
 (P^{N - k} f)(X_k) = \E (f(X_N) \ms1|\ms1 \ca F_k), 
 \ms{16}
 (P^{N - k} g)(X_N) = \E (g(X_k) \ms1|\ms1 \ca G_N). 
 \label{allerretour}
\end{equation}
We introduce the \og canonical\fge martingale associated to a function
$f$ on $\ca E$, by letting
\begin{equation}
 M_i = (P^{N - i} f)(X_i) = \E (f(X_N) \ms1|\ms1 \ca F_i), 
 \ms{16} 0 \le i \le N.
 \label{MartiCano}
\end{equation}
We see that in~\eqref{allerretour}, one occurrence of $P^{N - k}$ relates
to the expectation at time $k < N$ of future positions $f(X_N)$, while the
other is about expectation at time~$N$ of past positions $g(X_k)$.
Combining the two equalities in~\eqref{allerretour} in a \og back and
forth\fge move, by taking $g = P^j f$ and $j = N - k$, we conclude that
\begin{equation}
 (P^{2j} f)(X_N) = \E \bigl( M_{N - j} \ms1 \big| \ms1 \ca G_N \bigr).
 \label{MartinSG}
\end{equation}
Since the conditional expectation operator on $\ca G_N$ is positive, we see
that for every $j = N - k = 0, \ldots, N$, we have the inequality
\[
     \max_{0 \le j \le N} |(P^{2 j} f)(X_N)|
  =  \max_{0 \le j \le N} 
      \bigl| \E \bigl( M_{N - j} \ms1 \big| \ms1 \ca G_N \bigr) \bigr|
 \le \E \bigl( \max_{0 \le i \le N} |M_i| \ms2 \big| \ms2 \ca G_N \bigr).
\]
It implies when $1 < p \le +\infty$, according to Doob's
inequality~\eqref{DoobLp} and to the non-expansivity on $L^p$ of
conditional expectations, the chain of inequalities
\begin{align}
    &\ms4 \Bigl\| \max_{0 \le j \le N} |(P^{2 j} f)(X_N)| \ms1 \Bigr\|_p
 \le   \Bigl\|
        \E \bigl(
            \max_{0 \le i \le N} |M_i| \ms2 \big| \ms2 \ca G_N  
           \bigr)
       \Bigr\|_p
 \label{supMC}
 \\
 \le &\ms4 \Bigl\| \max_{0 \le i \le N} |M_i| \ms1 \Bigr\|_p
 \le \frac p {p - 1} \ms2 \| M_N \|_p
  =  \frac p {p - 1} \ms2 \| f(X_N) \|_p.
 \notag
\end{align}
We could recover the odd indices $2 j + 1$ by applying the latter inequality
to $P f$ instead of~$f$ and using $\|P f\|_p \le \|f\|_p$, to the cost of
an extra factor~2.
\dumou

\begin{smal}
\noindent
Estimating the maximal function of semi-groups is a central theme
in~\cite{SteinTHA}. The discrete case of~\eqref{supMC} was obtained by
Stein in the short article~\cite{SteinMET}, independently of
Rota~\cite{Rota}, by methods preluding those of~\cite{SteinTHA}. Theorem~1
in~\cite{SteinMET} applies to self-adjoint operators $P$ on 
$L^2(X, \Sigma, \mu)$ satisfying also $\|P\|_{1 \to 1} \le 1$ and
$\|P\|_{\infty \to \infty} \le 1$.
\end{smal}

\noindent
 One can play the same game with convex functions other than the supremum
function on $\R^{N+1}$. For example, let us begin with the convexity
inequality
\[
     \Bigl( \sum_{0 \le i \le N} 
      | \E (f_i \ms1|\ms1 \ca G) |^2 
     \Bigr)^{1/2}
 \le \E \Bigl( \bigl( \sum_{0 \le i \le N} |f_i|^2 \bigr)^{1/2}
               \ms1 \big| \ms1 \ca G \Bigr),
\]
and make use of the Burkholder--Gundy inequalities of Theorem~\ref{BurGun},
in order to obtain, when $0 \le j_0 < j_1 < \ldots < j_r \le N$, 
$1 < p < +\infty$, and with respect to the invariant measure~$\mu$, the
inequality
\begin{equation}
      \Bigl\|
       \Bigl(
        \sum_{k = 1}^r |(P^{2 j_k} f - P^{2 j_{k-1}} f)|^2 
       \Bigr)^{1/2}
      \Bigr\|_{L^p(\mu)}
 \le c_p \ms1 \|f\|_{L^p(\mu)}.
 \label{RBG}
\end{equation}
Indeed, we have seen in~\eqref{MartinSG} that $(P^{2 j_k} f)(X_N)$ is the
projection on $\ca G_N$ of the member 
$M_{N - j_k} = \E (f(X_N) \ms1|\ms1 \ca F_{N - j_k})$ of the martingale
$(M_j)_{j=0}^N$ in~\eqref{MartiCano}. Then $L_i = M_{N - j_{r-i}}$, 
$i = 0, \ldots, r$ is another martingale, and
\[
   (P^{2 j_{k-1}} f) (X_N) - (P^{2 j_k} f)(X_N)
 = \E (M_{N - j_{k-1}} - M_{N - j_k} \ms1|\ms1 \ca G_N)
\] 
appears as projection on $\ca G_N$ of the \emph{martingale difference}
$d_{r - k + 1} = L_{r - k + 1} - L_{r - k}$ (see Section~\ref{BGIneqs})
when $1 \le k \le r$. This principle can be applied for bounding diverse
convex functions of a semi-group, by considering them as projections of
corresponding functions of a martingale, for which we may have an 
\og $L^p$~inequality\fg.
\dumou

 Let us come back to~\eqref{supMC}. Since the distribution of $X_N$ is
uniform, we can restate~\eqref{supMC} when $1 < p \le +\infty$ as
\[
     \Bigl( \frac 1 Z  
      \sum_{x \in \ca E} \ms2 \max_{0 \le j \le N} |(P^{2 j} f)(x)|^p 
     \Bigr)^{1/p}
 \le \frac p {p - 1} 
      \ms 2 \Bigl(
             \frac 1 Z \sum_{x \in \ca E} |f(x)|^p 
            \Bigr)^{1 / p} \ns4,
\]
or else, changing the normalization and letting $N$ tend to infinity, we
obtain
\begin{equation}
     \Bigl( 
      \sum_{x \in \ca E} \ms2
       \sup_{j \ge 0} |(P^{2 j} f)(x)|^p 
     \Bigr)^{1/p}
 \le \frac p {p - 1} 
      \ms 2 \Bigl( \sum_{x \in \ca E} |f(x)|^p \Bigr)^{1 / p}.
 \label{maxiSemiGr}
\end{equation}
We can also write
\[
     \Bigl( 
      \sum_{x \in \ca E} \ms2
       \sup_{j \ge 0} |(P^j f)(x)|^p 
     \Bigr)^{1/p}
 \le \frac {2 \ms1 p} {p - 1} 
      \ms 2 \Bigl( \sum_{x \in \ca E} |f(x)|^p \Bigr)^{1 / p}.
\]
\dumou

 If we want to deal with a countably infinite state space $\ca E$ such as 
$\ca E = \Z^n$, we may accept (as Stein~\cite{SteinTHA} does) to work with
an \emph{infinite invariant measure}, uniform on~$\ca E$, that gives 
measure~$1$ to each singleton $\{e\}$, $e \in \ca E$. We then obtain the
same maximal inequality~\eqref{maxiSemiGr}, applying
Remark~\ref{InfiniProba}. If we do not accept an \og infinite
probability\fg, we may, for example with the Bernoulli random walk, work
with \og boxes\fge finite but large enough: if $f$ is finitely supported in
$\Z^n$ and if $N$ is fixed, we can find a finite box $B$ in $\ca E$, so big
that $P^j f$ vanishes outside $B$ for every $j \le 2 N$. Changing the
Bernoulli transition matrix $P(x, y)$ at the boundary of $B$, in order to
force the Markov chain to remain inside, we are back to the finite case.

\subsection{Brownian motion, and more on martingales%
\label{BrowniMarti}}

\subsubsection{Gaussian distributions and Brownian motion%
\label{GaussiDist}}

\noindent
Let $|x|$ denote here the Euclidean norm of a vector~$x$ in~$\R^n$. For
every probability measure $\mu$ on~$\R^n$ having a finite first order moment
$\int_{\R^n} |x| \, \d \mu(x)$, one defines the 
\emph{barycenter}\label{Baryce}
of $\mu$ as
\[
 \bar \mu = \int_{\R^n} x \, \d \mu(x) \in \R^n.
\]
To a probability measure $\mu$ on $\R^n$ with finite second order moment
$\int_{\R^n} |x|^2 \, \d \mu(x)$, one associates the quadratic form
\[
 Q_\mu : \xi \mapsto 
          \int_{\R^n} \bigl( (x - \bar \mu) \ps \xi \bigr)^2 \, \d \mu(x),
 \quad
 \xi \in \R^n .
\]
The matrix $\mathrm{Q}$ of $Q_\mu$ with respect to the canonical basis
of~$\R^n$ is the \emph{covariance matrix}\label{CovariMa} 
of $\mu$. The quadratic form $Q_\mu$ is positive definite when $\mu$ is not
supported on any affine hyperplane, for example when~$\mu$ is the uniform
probability measure on a bounded convex set $C$ with non empty
interior,\label{ConveBod}
\textit{i.e.}, a \emph{convex body}~$C$. We say that $\mu$ is 
\emph{centered}\label{CenteDist} 
when $\bar \mu = 0$, and in this case the expression of $Q_\mu$ simplifies
to $Q_\mu(\xi) = \int_{\R^n} ( x \ps \xi )^2 \, \d \mu(x)$ for every 
$\xi \in \R^n$.
\dumou

 When $f$ is a probability density on $\R$ with finite second order moment,
the \emph{variance}\label{VarianceDef} 
$\sigma^2$ of $f(x) \, \d x$ is defined by
\[
   \sigma^2
 = \int_\R \Bigl( x - \int_\R y f(y) \, \d y \Bigr)^2 f(x) \, \d x.
\]
When $f$ is centered, one has that $\sigma^2 = \int_\R x^2 f(x) \, \d x$.
\dumou

 A Gaussian random variable with distribution 
$N(0, \I_n)$\label{NZero} 
takes values in~$\R^n$, its distribution $\gamma_n$ is symmetric, thus
centered, defined on $\R^n$ 
by\label{Gaussiennes} 
\begin{equation}
   \d \gamma_n (x)
 = (2 \pi)^{-n/2} \e^{ - |x|^2 / 2} \, \d x
 \label{LoiNZeroId}
\end{equation}
and $\gamma_n$ admits the identity matrix $\I_n$ as covariance matrix. If
$F$ is a $n$-dimensional Euclidean space, we denote by $\gamma_F$ the image
of $\gamma_n$ under an (any) isometry from~$\R^n$ onto~$F$. If $X$ is a 
$N(0, \I_n)$ Gaussian random variable and $\sigma > 0$, then the multiple
$\sigma X$ admits the distribution
$
   \d \gamma_{n, \sigma}(x)
 = (2 \pi)^{-n/2} \e^{ - |x / \sigma|^2 / 2} \, \d (x / \sigma)
$, called the $N(0, \sigma^2 \I_n)$ distribution, with $\sigma^2 \I_n$ as
covariance matrix. One can consider that the \emph{Dirac probability
measure}~$\delta_0$\label{Dirac} 
at the origin of $\R^n$ corresponds to $N(0, 0_n)$.
\dumou

 The (absolute) moments of the one-dimensional distribution $\gamma_1$ can
be computed in terms of values of the Gamma function. For every $p > -1$,
one has that
\[
   \int_\R |x|^p \, \d \gamma_1(x) 
 = (2 \pi)^{-1/2} \int_\R |x|^p \e^{- x^2 / 2} \, \d x 
 = 2^{p/2} \pi^{-1/2} \Gamma \bigl( (p+1)/2 \bigr).
\]
As $p$ tends to $+\infty$, it follows from Stirling's formula that
\begin{equation}
    g_p
 := \Bigl( \int_\R |x|^p \, \d \gamma_1(x) \Bigr)^{1/p}
 \simeq \sqrt {p / \e}.
 \label{AsymptoGauss}
\end{equation}
\dumou

 A $n$-dimensional \emph{Brownian motion}\label{BrowniMot} 
$(B_t)_{t \ge 0}$ starting at $x_0 \in \R^n$ is a $\R^n$-valued random
process defined on some probability space $(\Omega, \ca F, P)$, such that
$B_0 = x_0$, such that $B_t - B_s$ is a Gaussian random variable with
distribution $N(0, (t - s) \I_n)$ whenever $0 \le s \le t$, and with
\emph{independent increments}: for every integer $k \ge 1$, when 
$0 \le t_0 < \ldots < t_k$ are given, then
\[
 B_{t_0}, \, B_{t_1} - B_{t_0}, \, B_{t_2} - B_{t_1}, 
 \ldots, \, B_{t_k} - B_{t_{k-1}}
\]
are independent. The coordinates $(B_{t, i})_{t \ge 0}$, $i = 1, \ldots, n$,
are independent one-dimens\-ional Brownian motions. It is possible to choose
\emph{everywhere defined} measurable functions $(B_t)_{t \ge 0}$ satisfying
the above properties in such a way that the trajectories 
$0 \le t \mapsto B_t(\omega)$, or \emph{random paths}, are continuous for
(almost) every $\omega \in \Omega$. The Brownian motion is a martingale
with continuous time parameter $t \ge 0$, with respect to a continuous time
filtration $(\ca F_t)_{t \ge 0}$ where $\ca F_t$ is generated by the
variables $B_s$, $0 \le s \le t$. See for example Durrett~\cite{Durrett}
for a detailed account.

 It is well known that the Brownian motion on $\R^n$ is the limit of
Markov chains with symmetric transition matrix, namely, a limit of suitably
scaled Bernoulli random walks. Indeed, if $\delta > 0$ is given and if we
consider a Bernoulli walk on the real line moving at each time $k \delta$,
$k \in \N^*$, by a step $\pm \sqrt \delta$, so that
$$
   X^{(\delta)}_t 
 = \sqrt \delta \ms2 
    \sum_{k=1}^{ \lfloor t / \delta \rfloor } \varepsilon_k,
 \quad
 t \ge 0,
 \ms{11}
 \varepsilon_k = \pm 1,
$$
then the distribution of $(X^{(\delta)}_t)_{t \ge 0}$ tends when
$\delta \rightarrow 0$ to that of a one-dimensional Brownian motion. Here,
$(\varepsilon_k)_{k=1}^\infty$\label{BernouVaris} 
is a sequence of independent Bernoulli random variables, taking 
values~$\pm 1$ with probability $1/2$. If $(B_t)_{t \ge 0}$ is the Brownian
motion in~$\R^n$, starting at $0$, and if we consider the associated
Gaussian semi-group $(G_{s})_{s \ge 0}$ defined for $f \in L^1(\R^n)$ and
$s > 0$ by
\begin{equation}
   (G_s f)(x) 
 = \E f(x + B_s)
 = (2 \pi s)^{-n/2} 
    \int_{\R^n} f(x + y) \e^{ - |y|^2 / (2 s)} \, \d y,
 \quad
 x \in \R^n,
 \label{GSG}
\end{equation}
we can show an inequality analogous to~\eqref{maxiSemiGr}. For every~$p$ in
$(1, +\infty]$ and for every function $f \in L^p(\R^n)$, we have a
\emph{maximal inequality for the Gaussian semi-group} with a bound
independent of the dimension~$n$, stating that
\begin{subequations}\label{maxiSemiGrGaG}
\begin{equation}
     \Bigl(
      \int_{\R^n} \sup_{s > 0} \ms1 \bigl| (G_{s} f)(x) \bigr|^p \, 
       \d x \Bigr)^{1/p}
 \le \frac p {p - 1} \Bigl( \int_{\R^n} |f(x)|^p \, \d x \Bigr)^{1/p} \ns6.
 \label{maxiSemiGrGa} \tag{\ref{maxiSemiGrGaG}.$G^*$}
\end{equation}
\end{subequations}
If we just need a maximal inequality possibly dimension dependent, there is
an easy proof relating the Gaussian maximal function to the classical
maximal function~$\M f$, because the Gaussian kernel is radial and radially
decreasing, see~\eqref{MajOmega}. Once Stein's Theorem~\ref{indepdim}
giving dimensionless estimates for $\M f$ is established, this easy bound
of $G_s f$ by $\M f$ implies a dimensionless estimate for the Gaussian
semi-group, or for the Poisson semi-group as well. With Bourgain, Carbery
and M\"uller, we shall follow the opposite route, from the semi-group
estimates to $\M f$ or~$\M_C f$. We sketch an argument for
obtaining~\eqref{maxiSemiGrGa} from the Bernoulli case.
\dumou

\begin{smal}

\noindent
Let us give some more details in dimension~$n = 1$. Let
$(\varepsilon_k)_{k=1}^\infty$ be a sequence of independent Bernoulli
random variables, taking values~$\pm 1$ with probability $1/2$. The
associated semi-group $(P_j)$, indexed by $j \in \N$, is defined by
\[
   (P_j \ms1 g)(i)
 = \E g \Bigl( i + \ns4 \sum_{1 \le k \le j} \varepsilon_k \Bigr),
 \quad
 j \ge 0, \ms8 i \in \Z,
\]
and it satisfies~\eqref{maxiSemiGr}. As a consequence of the de
Moivre--Laplace theorem and by classical tail estimates, we know that
$(1 + x^2) \ms2 P \bigl( \bigl\{ N^{-1/2} \sum_{1 \le k \le N} \varepsilon_k
 < x \bigr\} \bigr)$ tends to
$(1 + x^2) \ms2 \gamma_1 \bigl( (-\infty, x) \bigr)$
when~$N \rightarrow \infty$, uniformly in $x$ real. It follows that
$\E f \bigl( N^{-1/2} \sum_{1 \le k \le N} \varepsilon_k \bigr)$ tends to
$\int_\R f(y) \, d \gamma_1(y)$, uniformly on Lipschitz functions
having a Lipschitz constant bounded by some fixed~$C$. If
$f$ is Lipschitz on~$\R$, then
\[
 \E f \Bigl( x + N^{-1/2} \sum_{1 \le k < s N} \varepsilon_k \Bigr)
 \ms4 \lra_N \ms6
 (2 \pi s)^{-1/2} \int_\R f(x + y) \e^{- y^2 / (2 s)} \, \d y,
\]
uniformly in $x \in \R$ and $s \in [t_0, t_1]$, with $0 < t_0 \le t_1$
fixed. This implies that for any given $\varepsilon > 0$ and $N$ large
enough, letting $g_N(i) = f(i / \sqrt N)$ for $i \in \Z$ and assuming 
$s \ms1 N - 1 \le j_N < s \ms1 N$, we have that
\[
   \Bigl| P_{j_N} g_N(i) - (G_s f)(i / \sqrt N) \Bigr| 
 < \varepsilon,
 \quad
 i \in \Z,
\]
for every $s \in [t_0, t_1]$. Applying~\eqref{maxiSemiGr} to $g_N$, we
obtain when $s_0, s_1, \ldots, s_k$ and $a > 0$ are given that
\[
     \int_{-a}^a \ms2 \max_{0 \le j \le k} \ms1 
      \bigl| (G_{s_j} f)(x) \bigr|^p \, \d x
 \le \eta^p(\varepsilon) + \Bigl( \frac p {p - 1} \Bigr)^p
      \ms2 \frac 1 {\sqrt N} \ms2 \sum_{i \in \Z} 
       \Bigl| f \Bigl( \frac i {\sqrt N} \Bigr) \Bigr|^p,
\]
where $\eta(\varepsilon)$ tends to $0$ with $\varepsilon$,
implying~\eqref{maxiSemiGrGa} when $\varepsilon \rightarrow 0$, 
$N \rightarrow +\infty$, $a \rightarrow +\infty$ and if the sequence
$\{s_j\}_{j \ge 0}$ is dense in $(0, +\infty)$. The same argument works
in~$\R^n$, thanks to the product structure of the Bernoulli and Gaussian
measures and to the fact that the linear space generated by products 
$f(x_1, \ldots, x_n) = \prod_{j=1}^n f_j(x_j)$ is uniformly dense in the
space of compactly supported Lipschitz functions on~$\R^n$.

\end{smal}

\noindent
These considerations generalize to semi-groups of convolution with
\emph{symmetric} probability measures $(\mu_t)_{t \ge 0}$ on $\R^n$, that
is to say, when $\mu_s * \mu_t = \mu_{s + t}$, $s, t \ge 0$, and
$\mu_t(A) = \mu_t(-A)$ for every Borel subset $A \subset \R^n$. Given 
$k > 1$, one can find a finitely supported symmetric probability measure
$\nu_{1/k}$ on $\R^n$ which is an approximation of~$\mu_{1 / k}$, in the
sense that the integrals of a given finite family of functions~$f$
on~$\R^n$ are nearly the same for $\mu_{j / k}$ and for $\nu_{1/k}^{*j}$
whenever $j \le k^2$. We may assume that~$\nu_{1/k}$ is supported in 
$\varepsilon \Z$, $\varepsilon > 0$. The symmetric Markov chain 
$(X_j)_{j \le k^2}$ on $\ca E = \varepsilon \Z$ with transition governed by
$\nu_{1/k}$ permits us to approximate the maximal function 
$\sup_t |\mu_t * f|$ of the semi-group, replacing it with
$
 \max_{j \le k^2} |\nu_{1/k}^{*j} * f|
$. 
\dumou

 It follows that some convex functions of the convolution semi-group can be
estimated in $L^p$ by projecting functions of a martingale. For example,
the sum of squares of differences, already mentioned in the Gaussian case,
can be studied also in the Poisson case by relating it to
the square function of a martingale and applying the Burkholder--Gundy
inequalities presented in the next section.

\subsubsection{The Burkholder--Gundy inequalities%
\label{BGIneqs}}

\noindent
When $(M_k)_{k=0}^N$ is a martingale with respect to a filtration
$(\ca F_k)_{k=0}^N$, one introduces the \emph{difference sequence\/}
$(d_k)_{k=0}^N$, which is defined by $d_0 = M_0$ and $d_k = M_k - M_{k-1}$
if $0 < k \le N$. Observe that $d_k$ is $\ca F_k$-measurable for 
$0 \le k \le N$ and that $\E ( d_k \ms1|\ms1 \ca F_{k-1} ) = 0$ for 
$k > 0$. Conversely, given a sequence $(d_k)_{k=0}^N$ with these two
properties, we obtain a martingale by setting $M_k = \sum_{j=0}^k d_j$, for
$0 \le k \le N$. For a scalar martingale $(M_k)_{k=0}^N$, we define the
\emph{square function process\/}\label{SquaFun} 
$(S_k)_{k=0}^N$ of the martingale by
\[
 S_k = \Bigl( \sum_{j=0}^k |d_j|^2 \Bigr)^{1/2}, 
 \quad k = 0, \ldots, N.
\]
\def\dkbar{\overline{d_k\ns4}\ms{4}}%
For a real or complex martingale in $L^2$, the differences $d_k$ and
$d_\ell$ are orthogonal when $k \ne \ell$. If $k < \ell$ for example, then
$d_k$ and its complex conjugate $\dkbar$ are $\ca F_{\ell-1}$-measurable,
thus
$\E ( \dkbar \ms1 d_\ell)
 = \E \bigl( \dkbar \E ( d_\ell \ms1|\ms1 \ca F_{\ell - 1} ) 
      \bigr) = 0$. It follows that
\begin{equation}
 \E |M_N|^2 = \sum_{k=0}^N \E |d_k|^2 = \E |S_N|^2.
 \label{OrthogMD}
\end{equation}
This equality $\|M_N\|_2 = \|S_N\|_2$ appears as an evident case of the
following result.

\begin{thm}[Burkholder--Gundy~\cite{BurkholderGundy}]%
\label{BurGun}
For every $p$ in\/ $(1, +\infty)$, there exists a constant $c_p \ge 1$ such
that for every integer $N \ge 1$, for every real or complex martingale\/
$(M_k)_{k=0}^N$, one has that
\[
     c_p^{-1} \|M_N\|_p
 \le \|S_N\|_p
 \le c_p \|M_N\|_p.
\]
\end{thm}

 The \emph{Khinchin inequalities} (see for example 
Zygmund~\cite[vol.~I, V.8, Th.~8.4]{ZygmundTS}) are a very particular
instance of the preceding theorem. Let $(\varepsilon_k)_{k=1}^N$ be a
sequence of independent Bernoulli random variables defined on a probability
space $(\Omega, \ca F, P)$, taking the values $\pm 1$ with probability
$1/2$. For every $p$ in $(0, +\infty)$, there exist constants 
$A_p, B_p > 0$ such that for every $N \ge 1$ and all scalars
$(a_k)_{k=0}^N$, one has
\begin{subequations}\label{HincinG}
\begin{align}
     A_p \Bigl( \sum_{k=1}^N |a_k|^2 \Bigr)^{1/2}
 &\ns1\le\ns1 \Bigl(
      \E \Bigl| \sum_{k=1}^N a_k \varepsilon_k \Bigr|^p 
     \Bigr)^{1/p}
 \ns1\le\ns1 \Bigl( \sum_{k=1}^N |a_k|^2 \Bigr)^{1/2},
 \ms9 0 < p \le 2,
 \label{Hincin} \tag{\ref{HincinG}.$\gr K$}
 \\
     \Bigl( \sum_{k=1}^N |a_k|^2 \Bigr)^{1/2}
 & \ns1\le\ns1 \Bigl(
      \E \Bigl| \sum_{k=1}^N a_k \varepsilon_k \Bigr|^p 
     \Bigr)^{1/p}
 \ns1\le\ns1 B_p \Bigl( \sum_{k=1}^N |a_k|^2 \Bigr)^{1/2},
 \ms9 2 \le p.
 \notag
\end{align}
\end{subequations}
The exact values of the constants $A_p, B_p$ are known (\cite{Sza, Haa}).
In order to relate these inequalities to Theorem~\ref{BurGun} when
$1 < p < +\infty$, we consider a special filtration on 
$(\Omega, \ca F, P)$,\label{DyadiFiltr} 
generated by the sequence $(\varepsilon_k)_{k=1}^N$. Let~$\ca F_0$ be the
trivial field consisting of~$\Omega$ and~$\emptyset$, and for $k > 0$, let
$\ca F_k$ be the finite field generated by 
$\varepsilon_1, \ldots, \varepsilon_k$. This field $\ca F_k$ has $2^k$
atoms of the form
\begin{equation}
   A 
 = A_{\gr u}
 = \{ \omega \in \Omega : \varepsilon_j(\omega) = u_j, 
       \ms2 j = 1, \ldots, k\},
 \quad
 \gr u = (u_1, \ldots, u_k),
 \label{atomA}
\end{equation}
where $u_j = \pm 1$. We shall call this particular sequence 
$(\ca F_k)_{k=0}^N$ of finite fields a \emph{dyadic filtration}. In this
framework, for $1 \le k \le N$, any scalar multiple $a_k \varepsilon_k$ of
$\varepsilon_k$ is a martingale difference $d_k$. For the associated
martingale with $M_N = \sum_{k=1}^N a_k \varepsilon_k$, the square function
$S_N$ is the constant function equal to $(\sum_{k=1}^N |a_k|^2 )^{1/2}$ and
the Khinchin inequalities appear indeed as a simple example of application
of Theorem~\ref{BurGun}. Of course, the latter sentence is historically
totally inaccurate.
\dumou

 We shall prove only special cases of Theorem~\ref{BurGun}. We say that
a sequence of random variables $(m_k)_{k=0}^N$ is \emph{predictable} when
\begin{equation} 
 m_0 \ms7 \hbox{is} \ms7
 \ca F_0 \hbox{-measurable, and} \ms{ 9}
 m_k \ms7 \hbox{is} \ms7
 \ca F_{k-1} \hbox{-measurable for} \ms7 0 < k \le N.
 \label{Predict}
\end{equation}
If $(m_k)_{k=0}^N$ is scalar valued and predictable, and if $(d_k)_{k=0}^N$
is a martingale difference sequence, then $(m_k d_k)_{k=0}^N$ is again a
martingale difference sequence since one has that
$\E (m_k d_k \ms1\big|\ms1 \ca F_{k-1})
 = m_k \E (d_k \ms1\big|\ms1 \ca F_{k-1}) = 0$. The new martingale
$(L_k)_{k=0}^N$ defined by $L_k = \sum_{j=0}^k m_j d_j$ is said to be
obtained as \emph{martingale transform},\label{MartiTrans} 
see~\cite{BurkhoMT, Burkho}.
\dumou
 
 Consider a dyadic filtration $(\ca F_k)_{k=0}^N$ as defined above. Notice
that each atom~$A$ of $\ca F_k$ as in~\eqref{atomA} has probability
$2^{-k}$, and is split into two atoms $A_{\pm}$ of $\ca F_{k+1}$,
$A_{\pm} := A \cap \{ \varepsilon_{k+1} = \pm 1\}$, according to the value
of $\varepsilon_{k+1}$. Let $d_{k+1}$ be a martingale difference with
respect to these dyadic fields. The function $d_{k+1}$ should have mean~$0$
on the atom~$A$ of $\ca F_k$, and be constant on each of the two
atoms~$A_{\pm}$ of $\ca F_{k+1}$ contained in $A$, which have equal measure
$P(A) / 2$. It follows that $d_{k+1}$ must take on $A$ two opposite values
$\pm v$. Consequently, the modulus (or the norm) of $d_{k+1}$ is constant
on $A$, thus $|d_{k+1}|$ is $\ca F_k$-measurable, so that $(|d_k|)_{k=0}^N$
is predictable, as defined in~\eqref{Predict}. We shall call \emph{Bernoulli
martingale}\label{BernouMart} 
any martingale $(M_k)_{k=0}^N$ with respect to this dyadic filtration 
$(\ca F_k)_{k=0}^N$. A Bernoulli martingale with values in a vector space
can be pictured as a \emph{tree} 
$(v_{\varepsilon_1, \ldots, \varepsilon_k})$ of vectors, $0 \le k \le N$
and $\varepsilon_j = \pm 1$, such that each vector 
$v_{\varepsilon_1, \ldots, \varepsilon_k}$ in the tree is the midpoint of
his two successors $v_{\varepsilon_1, \ldots, \varepsilon_k, 1}$ and
$v_{\varepsilon_1, \ldots, \varepsilon_k, -1}$. The vectors
$v_{\varepsilon_1, \ldots, \varepsilon_k}$ are the values of the $k\ms1$th
random variable $M_k$ of the martingale, which can be defined by
$M_k(\varepsilon_1, \ldots, \varepsilon_k)
 = v_{\varepsilon_1, \ldots, \varepsilon_k}$.
\dumou

 The next Lemma contains an easier case of a result due to Burgess
Davis~\cite{BurgessDavis}, namely, the left-hand inequality when 
$p = 1$. The rest of the statement presents a mixture of Doob's and
Burkholder--Gundy's inequalities.
\dumou

\begin{lem}%
\label{BernouMarti}
For every $p$ with\/ $1 \le p \le 2$ and for every real or complex
Bernoulli martingale\/ $(M_k)_{k=0}^N$, one has that
\[
      6^{-1} \|M^*_N\|_p
 \le \|S_N\|_p
 \le 6 \ms1 \|M^*_N\|_p.
\]
\end{lem}
 
\begin{proof}[Partial proof, after~\cite{Mau}]
We consider the case $p = 1$. The general strategy is to bring the problem
to $L^2$, where $\|S_N\|_2 = \|M_N\|_2$ by~\eqref{OrthogMD}, and this is
essentially done by dividing $f = M_N \in L^1$ by a \og parent\fge of
$\sqrt{|f|}$, in order to get an element in~$L^2$ \og similar\fge to
$\sqrt{ |f| }$. One then applies known facts in~$L^2$, and finally come
back to $L^1$ by multiplication with a suitable $L^2$ function. We begin
with the proof of the left-hand inequality in Lemma~\ref{BernouMarti}.
\dumou

 Let $(M_k)_{k=0}^N$ be a Bernoulli martingale. We know that
$(|d_k|)_{k=0}^N$ is predictable, as well as $(S_k)_{k=0}^N$. Consider the
martingale transform $L_k = \sum_{j=0}^k S_j^{-1/2} d_j$. In~$L^2$ we know
that
$
 \E |L_N|^2 = \sum_{j = 0}^N \E (S_j^{-1} |d_j|^2)
$.
We see that $S_0^{-1} |d_0|^2 = S_0$, and
$S_j^{-1} |d_j|^2 \le 2 (S_j - S_{j-1})$ for $j \ge 1$ because, letting 
$t = S_{j-1}^2$ and $h = |d_j|^2$, we have
\[
     2 \ms1 ( \sqrt {t + h} - \sqrt t )
  =  \int_t^{t+h} u^{-1/2} \, \d u
 \ge h \ms1 (t + h)^{-1/2}.
\]
It follows that 
\begin{equation}
 \E |L_N|^2 \le 2 \E S_N.
 \label{Facile}
\end{equation}
Notice that
$
     \Bigl| \sum_{j = 0}^s  S_j^{-1/2} d_j \Bigr|
  =  |L_s|
 \le L^*_N
$
and
$     \Bigl| \sum_{j = r+1}^s  S_j^{-1/2} d_j \Bigr|
  =  |L_s - L_r|
 \le 2 \ms1 L^*_N
$ when $0 \le r < s \le N$.
Multiplying term\-wise the sequence $(S_k^{-1/2} d_k)_{k=0}^N$ with the
non-decreasing sequence $(S_k^{1/2})_{k=0}^N$, we obtain for every 
$s \le N$ by Abel's summation method that
\[
     |M_s|
  =  \Bigl| \sum_{j = 0}^s  d_j \Bigr|
 \le S_s^{1/2} \sup_{0 \le r \le s}
                \Bigl| \sum_{j = r}^s  S_j^{-1/2} d_j \Bigr| 
 \le 2 \ms1 S_N^{1/2} L^*_N,
\]
thus $M_N^* \le 2 S_N^{1/2} L^*_N$. By Cauchy--Schwarz, Doob's
inequality~\eqref{DoobLp} with $p = 2$ and by~\eqref{Facile}, we get the
conclusion
\[
     \E M_N^* 
 \le 2 (\E S_N)^{1/2} \|L^*_N\|_2
 \le 2^2 (\E S_N)^{1/2} \|L_N\|_2
 \le 2^{5 / 2} \E S_N
 \le 6 \E S_N.
\]
We leave the rewriting of this proof when $1 < p < 2$ as an easy exercise
for the reader, and we pass to the right-hand side inequality using the same
method, with the help of the non-decreasing predictable sequence
$(A_k)_{k=0}^N$ defined by
\[
 A_0 = |d_0| = |M_0|,
 \quad
 A_k = \max \bigl( A_{k-1}, M^*_{k-1} + |d_k| \bigr) \ge |M_k|,
 \ms8
 k = 1, \ldots, N,
\]
and of the martingale transform $L_k = \sum_{j=0}^k A_j^{-1/2} d_j$,
$k = 0, \ldots, N$. Observe that
$|d_k| \le |M_k| + |M_{k-1}| \le 2 M^*_N$, thus $A_N \le 3 M^*_N$. By
Abel, writing $d_k = M_k - M_{k-1}$ for $k \ge 1$, we see that
\begin{align*}
     |L_N|
 & =  \Bigl| 
      A_N^{-1/2} M_N + \sum_{k=0}^{N - 1} M_k (A_k^{-1/2} - A_{k+1}^{-1/2}) 
     \Bigr|
 \\
 &\le A_N^{1/2} + \sum_{k=0}^{N - 1} A_k (A_k^{-1/2} - A_{k+1}^{-1/2})
 \le A_N^{1/2} + \sum_{k=0}^{N - 1} (\sqrt{A_{k+1}} - \sqrt {A_k} )
 \le 2 A_N^{1/2},
\end{align*}
where we make use of $u^2(u^{-1} - v^{-1}) \le v - u$ when $0 < u \le v$.
In $L^2$ we know that $\E \bigl( \sum_{k=0}^N A_k^{-1} |d_k|^2 \bigr)
 = \E |L_N|^2 \le 4 \, \E A_N$, and we go back to $L^1$ with
Cauchy--Schwarz and the obvious inequality 
$\sum_{k=0}^N |d_k|^2 \le A_N \sum_{k=0}^N A_k^{-1} |d_k|^2$. We obtain
\[
     \E S_N
  =  \E \Bigl( \sum_{k=0}^N |d_k|^2 \Bigr)^{1 / 2}
 \le (\E A_N)^{1/2} \, \|L_N\|_2
 \le 2 \ms1 \E A_N
 \le 6 \, \|M^*_N\|_1.
\]
\end{proof}

\begin{rem}
\label{BrowniMarting}
The Brownian martingales can be approximated by Bernoulli martingales, and
we can obtain the analogous result for them. Actually, the preceding proof
is even simpler to write in this case. Brownian martingales are defined by
means of (It\={o}'s) \emph{stochastic integrals}\label{StochInteg}
\[
 M_t(\omega) = \int_0^t m_s(\omega) \, \d B_s(\omega),
 \quad
 t \ge 0,
\]
where $(m_s)_{s \ge 0}$ is an \emph{adapted process}, meaning essentially
that each $m_s$, $s \ge 0$, is $\ca F_s$-measurable. The square function is
then defined by
$
 S_t^2(\omega) = \int_0^t |m_s(\omega)|^2 \, \d s
$ for every $t \ge 0$, and one can replace in the proof of
Lemma~\ref{BernouMarti} the Abel summation method by the more pleasant
integration by parts.
\end{rem}
\dumou

\begin{remN}%
\label{Incondi}
Together with Doob's inequality, Lemma~\ref{BernouMarti} implies
Theorem~\ref{BurGun} for Bernoulli martingales when $1 < p \le 2$. The
Burkholder--Gundy inequalities are equivalent to saying that martingale
difference sequences are \emph{unconditional} in~$L^p$ when 
$1 < p < +\infty$, that is to say, that there exists a constant 
$\kappa_{u, p}$ such that for each integer $N \ge 0$, all scalars
$(a_k)_{k=0}^N$ with $|a_k| \le 1$ and all martingale differences
$(d_k)_{k=0}^N$, we have
\begin{equation}
     \bigl\| \sum_{k=0}^N a_k d_k \bigr\|_p
 \le \kappa_{u, p} \ms1 \bigl\| \sum_{k=0}^N d_k \bigr\|_p.
 \label{UMD}
\end{equation}
Going from Theorem~\ref{BurGun} to unconditionality is simple, since
the square function of the martingale at the left-hand side of~\eqref{UMD}
is less than that on the right-hand side, and we can take 
$\kappa_{u, p} = c_p^2$. The other direction follows from Khinchin, by
averaging over signs $a_k = \pm 1$. Indeed, one obtains from~\eqref{Hincin}
for $(f_k)_{k=1}^N$ in $L^p(X, \Sigma, \mu)$, $1 \le p < +\infty$, that
\begin{equation}
     A_p^p \ms2 
      \Bigl\|
       \Bigl( \sum_{k=1}^N |f_k|^2 \Bigr)^{1/2} 
      \Bigr\|_{L^p(\mu)}^p
 \le \E \int_X \bigl| \sum_{k=1}^N \varepsilon_k f_k \bigr|^p \, \d \mu
 \le B_p^p \ms2 
      \Bigl\|
       \Bigl( \sum_{k=1}^N |f_k|^2 \Bigr)^{1/2} 
      \Bigr\|_{L^p(\mu)}^p.
 \label{KinLp}
\end{equation}
It is possible (see Pisier~\cite[5.8]{PisierMart}) to obtain the general
case of unconditionality of martingale differences by approximating general
martingale difference sequences by \emph{blocks} of Bernoulli martingale
differences. Also, one can see that~\eqref{UMD} is self-dual and obtain by
duality the Burkholder--Gundy inequalities for $2 \le p < +\infty$.
\dumou

 The proof of Lemma~\ref{BernouMarti} is valid with almost no change when
the martingale takes values in a Hilbert space $H$, because
$L^2(\Omega, \ca F, P, H)$ is a Hilbert space where the $H$-valued
martingale differences are orthogonal. For values in a Banach space, two
difficulties arise. First, the relevant \og square function\fge has to be
defined, and second, the Banach space-valued martingale differences are
not unconditional in general. The Banach spaces where martingale
differences are unconditional form a nice class of spaces, see
Pisier~\cite[Chap.~5, The UMD property for Banach spaces]{PisierMart}.
\end{remN}

\begin{remN}\label{BurkhoConst}
Let $f = \sum_{k=0}^N d_k$ be the sum of a Bernoulli martingale and let
$g = \sum_{k=0}^N a_k d_k$ be obtained from $f$ by a martingale transform
operation, with $|a_k| \le 1$ for $k = 0, \ldots, N$. By 
Lemma~\ref{BernouMarti} and Doob's inequality~\eqref{DoobLp}, we have
\[
     \|g\|_p
 \le \|g^*\|_p
 \le 6 \ms2 \| S(g) \|_p
 \le 6 \ms2 \| S(f) \|_p
 \le 36 \ms2 \| f^* \|_p
 \le \frac { 36 \ms3 p } { p - 1} \ms3 \| f \|_p,
 \ms8
 1 < p \le 2,
\]
which shows that the constant $\kappa_{u, p}$ in~\eqref{UMD} is of order
$1 / (p-1)$ in this case. Actually, Burkholder has found the exact value of
the unconditional constant for general martingale transforms and for every 
$p \in (1, +\infty)$. It is given by
\[
 \label{PStar}
 \kappa_{u, p} = p^* - 1,
 \ms{10} \hbox{where} \ms{10}
 p^* := \max(p, p / (p-1)).
\]
One can consult~\cite{Burkho} and the references given there to several
other articles by Burkholder. One can also find
in~\cite[Section~5.4]{Burkho} a bound $c_p \le p^* - 1$ for the constant 
$c_p$ in Theorem~\ref{BurGun}.
\end{remN}

\subsubsection{A consequence of the \og reflection principle\fg%
\label{PrincipeReflexion}}

\noindent
Consider a Brownian motion $(B_s)_{s \ge 0}$ on $\R$, defined on a
probability space $(\Omega, \ca F, P)$ and with respect to a filtration
$(\ca F_s)_{s \ge 0}$. We assume that $B_0 = 0$, we fix a real 
number~$v > 0$ and we let~$S_v(\omega)$ denote the first time when the
trajectory $s \mapsto B_s(\omega)$, $s \ge 0$, which is continuous for almost
every $\omega \in \Omega$, reaches the point $v$. It is clear that if 
$s_0 > 0$ is given, one has $\{B_{s_0} \ge v\} \subset \{S_v \le s_0\}$,
thus
\[
     P \bigl( \{S_v \le s_0\} \bigr) 
 \ge P \bigl( \{B_{s_0} \ge v\} \bigr)
  =  P \bigl( \{B_1 \ge v / \sqrt{s_0} \ms1 \} \bigr)
  =  \int_{v / \sqrt {s_0}}^{+\infty} \e^{- y^2 / 2} \, 
      \frac {\d y} {\sqrt{2 \pi}} \ms1 \up.
\]
From now on, we write $P ( S_v \le s_0 )$ for 
$P \bigl( \{ S_v \le s_0 \} \bigr)$. We will show that actually
\[
     P ( S_v \le s_0 ) 
  =  2 \ms2 P(B_{s_0} \ge v)
  =  2 \ms2 \int_{v / \sqrt {s_0}}^{+\infty} \e^{- y^2 / 2} \, 
      \frac {\d y} {\sqrt{2 \pi}} \ms1 \up,
\]
which proves in passing that $S_v$ is finite almost surely, since we have
then
\[
     P ( S_v < +\infty ) 
  =  2 \ms2 \int_0^{+\infty} \e^{- y^2 / 2} \, 
      \frac {\d y} {\sqrt{2 \pi}}
  = 1.
\]
The reasoning makes use of the reflection of the Brownian motion after a
stopping time~$\tau$. A 
\emph{stopping time\/}\label{StoppingT} 
is a random variable $\tau$ with values in $[0, +\infty]$, such that for
every $t \ge 0$, the event $\{\tau \le t\}$ belongs to the $\sigma$-field
$\ca F_t$ of the past of time~$t$. Intuitively, a stopping time corresponds
to a decision to quit at time $\tau(\omega)$ that an observer, embarked on
a path $t \mapsto X_t(\omega)$ of the random process $(X_t)_{t \ge 0}$
since the time $t = 0$, can take from his only knowledge of what happened
on his way between $0$ and the present time. The random time $S_v$ is an
excellent example of stopping time, with a quite simple rule: I stop when I
reach the point $v > 0$.
\dumou

 The Brownian reflected after the random time $\tau$ changes its direction,
its trajectory becomes the symmetric of the original trajectory with
respect to the point\label{BTau} 
$(B_\tau)(\omega) := B_{\tau(\omega)}(\omega)$ that was reached at time
$\tau(\omega)$. Let us denote by $(B^\tau_s)_{s \ge 0}$ the reflected
Brownian, given by
\[
 B^\tau_s(\omega) = B_s(\omega)
 \ms{16} \hbox{if} \ms{16}
 0 \le s \le \tau(\omega),
 \ms{22}
   \frac { B^\tau_s(\omega) + B_s(\omega)} 2
 = B_{\tau(\omega)}(\omega)
 \ms{16} \hbox{if} \ms{16}
 s \ge \tau(\omega).
\]
\emph{The reflected Brownian $B^\tau$ is still a Brownian motion}. Consider
first the simplest stopping time and reflection. Choosing a set $A_1$ in
the $\sigma$-field $\ca F_{s_1}$ at time $s_1 > 0$, we define a stopping
time~$\tau_1$ equal to $s_1$ on~$A_1$ and to $+\infty$ outside. The
corresponding reflection $(B^{\tau_1}_s)_{s \ge 0}$ is given by
\begin{align*}
    B^{\tau_1}_s(\omega) 
 &= B_s(\omega)
 \ms{16} \hbox{if} \ms{16}
 0 \le s \le s_1
 \ms{16} \hbox{or} \ms{16}
 \omega \notin A_1,
 \\
    \frac { B^{\tau_1}_s(\omega) + B_s(\omega)} 2
 &= B_{s_1}(\omega)
 \ms{16} \hbox{if} \ms{16}
 s \ge s_1
 \ms{16} \hbox{and} \ms{16}
 \omega \in A_1.
\end{align*}
One shows easily that $(B^{\tau_1}_s)_{s \ge 0}$ is a Brownian motion.
Iterating this operation, one can reach discrete stopping times, and pass
to the limit for dealing with general stopping times. Indeed, a stopping
time $\tau$ can be approximated by the first time $\tau_k > \tau$ such that 
$2^k \tau_k$ is an integer, \textit{i.e.},
$\tau_k = 2^{-k} (\lfloor{2^k \tau}\rfloor + 1)$, for every $k \in \N$. 
\dumou

 Another important property that can be checked following the same route is
the following: if $\tau$ is an almost surely finite stopping time, the
process \og starting afresh at time $\tau$\fg, defined by
$X_s = B_{\tau + s} - B_\tau$, \textit{i.e.}, 
$X_s(\omega) = B_{\tau(\omega) + s}(\omega) - B_{\tau(\omega)}(\omega)$, is
also a Brownian motion.
\dumou

 Consider the Brownian reflected after the stopping time $S_v$, with 
$v > 0$. Since the Brownian paths are continuous and $B_0 = 0$, we have 
$B_{S_v(\omega)}(\omega) = v$ and for every $s_0 > 0$, the event 
$\{ B_{s_0} > v \}$ is contained in $\{ S_v < s_0 \}$. Clearly, the event 
$\{ B^{S_v}_{s_0} > v \}$ is also contained in $\{ S_v < s_0 \}$ and
disjoint from $\{ B_{s_0} > v \}$. Actually, since on the set 
$\{ S_v < s_0 \}$ one has $B^{S_v}_{s_0} + B_{s_0} = 2 \ms1 v$, one sees
that
\[
   \{ S_v < s_0 \} \setminus \{ B_{s_0} \ge v \}
 = \{ B^{S_v}_{s_0} > v \}.
\]
The event $\{ B^{S_v}_{s_0} > v \}$ has the same probability as
$\{ B_{s_0} > v \}$, since $(B^{S_v}_s)_{s \ge 0}$ is another Brownian, and
$P(S_v = s_0) \le P(B_{s_0} = v) = 0$. We have therefore that
\[
   P(S_v \le s_0)
 = P(S_v < s_0)
 = 2 \ms2 P(B_{s_0} > v)
 = 2 \int_v^{+\infty} \e^{ - u^2 / (2 s_0)} \, 
    \frac {\d u} {\sqrt{2 \pi s_0}} \ms1 \up.
\]
Consequently, for every $s > 0$, we obtain
\[
   P ( S_v \le s)
 = P \bigl( \sup_{0 \le u \le s} B_u \ge v \bigr)
 = 2 \int_{v / \sqrt s}^{+\infty} \e^{ - y^2 / 2} \, 
    \frac {\d y} {\sqrt{2 \pi}} \ms2 \up.
\]
This allows us to find the density $h_v$ of the distribution of~$S_v$, which
is given by
\begin{equation}
   h_v(s)
 = \gr 1_{s \ge 0} \ms2
    \frac {v s^{-3/2}} {\sqrt{2 \pi}} \e^{ - v^2 / (2 s) },
 \quad s \in \R.
 \label{LaLoi}
\end{equation}
\dumou

\begin{rem}
A variant of the preceding reasoning applies to the 
exit time~$S$\label{ExitTime} 
from an open convex subset $D$ of $\R^n$ containing the starting
point $x_0$ of an $n$-dimensional Brownian motion. Suppose that this
Brownian motion touches the boundary of~$D$, for the first time, at the 
point~$x = x(\omega)$ and at time~$S(\omega)$. Let $E_x$ be an affine
half-space tangent to $D$ at~$x$, and exterior to~$D$ (this $E_x$ is not
unique in general). Starting again from~$x$ at time $S(\omega)$, there is a
probability $1 / 2$ to end in $E_x$ at time $s_0 > S(\omega)$, so there is
at least one chance out of two to end up outside $D$ at time~$s_0$. The set
$\{ B_{s_0} \notin D \}$ is a subset of $\{ S < s_0 \}$ that occupies thus
at least one half of it. We have therefore
\[
 P(S < s_0) \le 2 \ms1 P(B_{s_0} \notin D).
\]
This inequality says that the probability to be outside $D$ at a time
between $0$ and~$s_0$ is bounded by twice the probability to be outside $D$
at time $s_0$. This can be readily interpreted in terms of maximal
function. If $\|\cdot\|_C$\label{CNorm} 
denotes the norm on~$\R^n$ associated to a symmetric convex body~$C$ in
$\R^n$, we deduce maximal inequalities in $L^p(\R^n)$ for the
$\|\cdot\|_C$~norm of the martingale $(B_s)_{s \ge 0}$ that are better than
Doob's inequality. Namely, for every $p > 0$ we have
\begin{align*}
     \E \max_{0 \le s \le s_0} \|B_s\|_C^p
  &=  p \int_0^{+\infty} t^{p-1} 
        P \bigl( \ms2 \max_{0 \le s \le s_0} \|B_s\|_C > t \bigr) \, \d t
 \\
 &\le 2 \ms1 p \int_0^{+\infty} t^{p-1} 
        P \bigl( \|B_{s_0}\|_C > t \bigr) \, \d t
  =  2 \ms 2 \E \|B_{s_0}\|_C^p.
\end{align*}
\end{rem}
\dumou

\begin{smal}
\noindent 
For $p \le 1$, there is no Doob's inequality in $L^p$, and when $p > 1$,
one has always that
$
 2^{1 / p} < p / (p - 1)
$,
because $(1 - x) 2^x < (1 - x) \e^x \le 1$ for $0 < x < 1$.

\end{smal}

\dumou
\noindent
One could get a similar estimate when the set $D$ is no longer convex, but
has the property that for every boundary point $x$ of $D$, there is a cone
$E_x$ based at~$x$, disjoint from $D$ and with a \emph{solid angle} bounded
below by $\delta > 0$ independent of $x$. If we measure the angle as the
proportion of the unit sphere $S^{n-1}$ of $\R^n$ intersected by the cone
$E_x - x$ based at~$0$, then the constant $2$ above has to be replaced
by~$\delta^{-1}$.

\subsection{The Poisson semi-group%
\label{PoissonSG}}

\noindent
Let us recall that the \emph{Schwartz class} 
$\ca S(\R^n)$\label{SchwaSpa} 
consists of all $C^\infty$ functions $\varphi$ such that 
$(1 + |x|^k) \varphi^{(\ell)}(x)$ is bounded on $\R^n$ for all integers 
$k, \ell \ge 0$. We shall denote by 
$(P_t)_{t \ge 0}$\label{PoissoSG} 
the Poisson semi-group on~$\R^n$, which can be defined, for $f$ in the
Schwartz class $\ca S(\R^n)$, by
\begin{equation}
 (P_t f)(x) = u(x, t),
 \ms{12}
 x \in \R^n, \ms3 t \ge 0,
 \label{PSG}
\end{equation}
where $u(x, t)$\label{UdeXT} 
is the (bounded) \emph{harmonic extension}\label{HarmonExt} 
of $f$ to the upper half-space $H^+$ of~$\R^{n+1}$ formed by all $(x, t)$
with $x \in \R^n$ and $t \ge 0$. For $x \in \R^n$ one has $u(x, 0) = f(x)$,
$\Delta u(x, t) = 0$ when $t > 0$, and $u$ is continuous on~$H^+$. The
semi-group property $P_{t+s} = P_t P_s$ amounts to saying that the
harmonic extension of the function $f_s$ defined on~$\R^n$ by 
$f_s(x) = u(x, s)$ is given by $v(x, t) = u(x, t + s)$.
\dumou

 The Poisson semi-group is intimately related to the Brownian motion
$(B_s)_{s \ge 0}$ in~$\R^{n+1}$. If the Brownian $(B_s)_{s \ge 0}$
starts at time $s = 0$ from the point $(x_0, t_0)$, where $x_0 \in \R^n$
and $t_0 > 0$, we know that almost every path $s \mapsto B_s(\omega)$ will
hit the hyperplane $H_0 = \{ t = 0 \}$ at some 
time~$\tau_{t_0}(\omega) < +\infty$. If we decompose $B_s$ into 
$(x_0 + X_s, t_0 + T_s)$, then~$T_s$ is a one-dimensional Brownian motion,
starting from~$0$ at time $0$, and $X_s$ is a $n$-dimensional Brownian
motion, starting from the point~$0$ in~$\R^n$ and independent of~$T_s$. The
stopping time $\tau_{t_0}$ is the first time $s > 0$ when $T_s = -t_0$. If
$f$ is reasonable, for example continuous and bounded on $\R^n$, one sees
that the (bounded) harmonic extension~$u$ of~$f$ to the upper half-space is
given by
\[
   u(x_0, t_0) 
 = \E F(B_{\tau_{t_0}})
 = \E f(x_0 + X_{\tau_{t_0}})
 = \int_\Omega 
    f \bigl( x_0 + X_{\tau_{t_0}(\omega)} (\omega) \bigr) \, \d P(\omega),
\]
where $F$ is defined on the hyperplane $H_0$ of~$\R^{n+1}$ by
$
 F(x, 0) = f(x)
$ 
for every $x \in \R^n$. The Poisson probability measure 
$P_{t_0}(x) \, \d x$ on $\R^n$ is the distribution of~$X_{\tau_{t_0}}$,
distribution of the Brownian motion $(X_s)$ starting from $0 \in \R^n$ and
stopped at time $\tau_{t_0}$, when $B_s$ reaches $H_0$. We shall employ the
same notation $P_t$ for the semi-group, for the Poisson distribution on
$\R^n$ and for its density $P_t(x)$. The operator $P_t$ is the convolution
with the corresponding probability measure, it acts thus on $L^p(\R^n)$ for
$1 \le p \le +\infty$. We shall say that $t$ is the \emph{parameter} of
$P_t$.\label{ParaPois} 
\dumou

 The distribution of the stopping time $\tau_{t_0}$ is clearly the same as
the distribution of the first time $S_{t_0}$ when the one-dimensional
Brownian motion starting from~$0$ reaches $t_0 > 0$, and we know
by~\eqref{LaLoi} the density~$h_t$ of the distribution of $S_t$. The
Poisson distribution $P_t$ on $\R^n$ is obtained by mixing Gaussian
distributions on~$\R^n$, distributions of $X_s$ at various times $s$, the
mixing being done according to the distribution of $S_t$. In the portion of
the space~$\Omega$ where $s_0 \le \tau_t \le s_0 + \delta \ns1 s$, the
coordinate~$x$ of the Brownian point $B_s = (X_s, t + T_s)$ at time
$\tau_t$ is approximately~$X_{s_0}$, with probability of order 
$h_t(s_0) \, \delta \ns1 s$, and $(X_s)_{s \ge 0}$ is independent of
$\tau_t$. The point $(x, 0) = (X_{s_0}, 0)$ is the point where the Brownian
$B_s$ touches the hyperplane~$H_0$, knowing that $\tau_t = s_0$. This is
the reason behind the \emph{subordination 
principle\/}\label{SuboPrin} 
of the Poisson semi-group to the Gaussian semi-group, which implies in
particular that the maximal function of the Poisson semi-group is bounded
by that of the Gaussian semi-group $(G_s)_{s \ge 0}$ on~$\R^n$. Indeed, we
have by~\eqref{LaLoi} that $P_t$ is \og in the (closed) convex hull\fge of
the Gaussian semi-group, since
\begin{equation}
 P_t = \int_0^{+\infty} G_s
        \frac {t s^{-3/2}} {\sqrt{2 \pi}} \e^{ - t^2 / (2 s) } \, \d s.
 \label{Subor}
\end{equation}
It follows that
\[
     |P_t * f|
 \le \int_0^{+\infty} |G_s * f| \ms2
        \frac {t s^{-3/2}} {\sqrt{2 \pi}} \e^{ - t^2 / (2 s) } \, \d s
 \le \sup_{u \ge 0} |G_u *f|.
\]
We get a dimensionless estimate for the maximal function of the
Poisson semi-group, consequence of the one in~\eqref{maxiSemiGrGa} for the
Gaussian case. We have
\begin{subequations}\label{MaxiPoissG}
\begin{equation}
     \Bigl\| \ms1 \sup_{t>0} |P_t f| \ms1 \Bigr\|_{L^p(\R^n)}
 \le \frac p {p - 1} \Bigl( \int_{\R^n}  |f(x)|^p \, \d x \Bigr)^{1/p}.
 \label{MaxiPoiss} \tag{\ref{MaxiPoissG}.$P^*$}
\end{equation}
\end{subequations}
The remarks about comparing to $\M f$ are still in order here. 
Stein~\cite[Lemma~1, p.~48]{SteinTHA} proves~\eqref{MaxiPoiss} with
different constants and in a different way, capable of easier
generalizations to non Euclidean settings. He does not deal with the
Gaussian maximal function, but applies the Hopf maximal
inequality~\eqref{HopfLp} to the Gaussian semi-group together with the
subordination principle. Using subordination, Stein shows that the Poisson
maximal function $P^* f = \sup_{t > 0} |P_t f|$ is bounded by an average of
expressions $t^{-1} \int_0^t (G_s f) \, \d s$ that are controlled by Hopf.
\dumou

 The formula~\eqref{Subor} proves that the marginals of $P_t$ are other
Poisson distributions: indeed, the mixing distribution, which has density
$h_t$, does not depend on the dimension~$n$, and the projections on
$\R^\ell$, $1 \le \ell < n$, of Gaussian distributions 
$N(0, \sigma^2 \I_n)$ on $\R^n$ are $N(0, \sigma^2 \I_\ell)$ Gaussian
distributions. We can also deduce the density of the distribution~$P_t$ for
each $t > 0$, writing
\begin{align*}
    P_t(x) 
 &= \int_0^{+\infty} \e^{ - |x|^2 / (2 s)} (2 \pi s)^{-n / 2}
           \frac {t s^{-3/2}} {\sqrt{2 \pi}} \e^{ - t^2 / (2 s) } \, \d s
 \\
 &= t \int_0^{+\infty} (2 \pi s)^{-n / 2 - 1 / 2} 
       \e^{ - (t^2 + |x|^2) / (2 s)} \,
           \frac {\d s} s \up,
 \quad x \in \R^n.
\end{align*}
Setting $u = s / (t^2 + |x|^2)$, then $v = 1 / (2u)$, we get
\[
   P_t(x) 
 = t \bigl( \pi (t^2 + |x|^2) \bigr)^{-(n+1)/2}
           \int_0^{+\infty} \e^{ - v} v^{(n+1) / 2}
            \, \frac {\d v} v \up.
\]
The \emph{Poisson kernel} $P_t$ on $\R^n$ is thus given by the formula
\begin{equation}
   P_t(x) 
 = P_{\ms2 t}^{(\ns{0.7} n \ns1)} \ns1 (x) 
 = \frac {\Gamma[(n+1)/2]} {\pi^{(n+1)/2}} \ms4
            \frac t {(t^2 + |x|^2)^{(n+1) / 2} \ns{50}} \ms{50} \up,
 \ms{16} x \in \R^n, \ms2 t > 0.
 \label{PoissonDensi}
\end{equation}
In dimension $n = 1$, the Poisson kernel is the \emph{Cauchy kernel}, equal
to
\begin{subequations}\label{CauchyKa}%
\begin{equation}
   P_t(x) 
 = P_{\ms2 t}^{(\ns{0.7} 1 \ns1)}(x) 
 = \frac t {\pi (t^2 + x^2)} \up,
 \ms{20} x \in \R, \ms6 t > 0.
 \label{CauchyK} \tag{\ref{CauchyKa}.{\bf C}}
\end{equation}
\end{subequations}
The coefficient that comes into the $n$-dimensional
formula~\eqref{PoissonDensi} verifies the asymptotic estimate
\[
      \frac {\Gamma[(n + 1) / 2]} {\pi^{ (n + 1) / 2 }} 
 \simeq \sqrt \frac 2 {\pi n} \ms2 \frac 1 {\omega_n}
  =   {\sqrt \frac {2 n} \pi } \ms2 \frac 1 {s_{n - 1}} \up,
\]
where $\omega_n$ is the volume of the unit ball in $\R^n$ and $s_{n-1}$ the
$(n-1)$-dimensional measure of the unit sphere $S^{n-1}$ in $\R^n$, given by
\begin{equation}
   \omega_n 
 = \frac {\pi^{n/2}} {(n/2) \ms{0.5} !} 
 := \frac {\pi^{n/2}} {\Gamma \bigl( (n/2) + 1 \bigr)} \ms1 \up,
 \ms{20}
 s_{n-1} = n \ms2 \omega_n. 
 \label{OmegaN}
\end{equation}

\begin{smal}%
\noindent
From this, we obtain estimates on the measure of Euclidean balls for the
probability measure $P_1(x) \, \d x$ on~$\R^n$. Writing $P_1(x) = F(|x|)$,
we get an exact asymptotic estimate when the dimension $n$ tends to
infinity: for $\nu > 0$ fixed, we have
\begin{align*}
    & \ms5 \int_{ \{ |x| > \sqrt n / \nu \} } P_1(x) \, \d x
  \\
  = & \ms5
     s_{n-1} \int_{ \sqrt n / \nu }^{+ \infty} r^{n-1} F(r) \, \d r
  \simeq
    \sqrt{ \frac 2 \pi } \int_{ \sqrt n / \nu }^{+ \infty} 
            \frac {\sqrt n} {\ms4 r} \ms2
             \Bigl( \frac {r^2} {1 + r^2} \Bigr)^{(n+1)/2}  
              \, \frac {\d r} {r \ns2}
 \\
  = & \ms5 \sqrt{ \frac 2 \pi } \int_{1 / \nu}^{+ \infty}
             \Bigl( 1 + \frac 1 {n u^2} \Bigr)^{ - (n+1)/2}  
              \, \frac {\d u} {u^2 \ns5}
  = \sqrt{ \frac 2 \pi} \int_0^\nu
             \Bigl( 1 + \frac {y^2} n \Bigr)^{ - (n+1)/2}  
              \, {\d y}.
\end{align*}
Therefore, when $n$ tends to infinity, we see that
\begin{subequations}\label{IntegraPois}
\SmallDisplay{\eqref{IntegraPois}}%
\[
 \int_{ \{ \nu \ms1 |x| > \sqrt n \} } P_1(x) \, \, \d x
 \ms4 \longrightarrow \ms4
 2 \int_0^{\nu} \e^{- y^2 / 2} 
  \, \frac {\d y} {\sqrt{2 \pi}} \ms1 \up.
\]
\end{subequations}
\end{smal}

\section{General dimension free inequalities, second part%
\label{GdfiTwo}}

\noindent
In this section, we gather results that depend on the Fourier transform. In
order that the Fourier transform be isometric on $L^2(\R^n)$, we set
\[
\label{FouTran}
 \forall \xi \in \R^n,
 \ms{16}
   \widehat f(\xi) 
 = \int_{\R^n} f(x) \e^{- 2 \ii \pi x \ps \xi} \, \d x,
 \ms{16}
   \widehat \mu(\xi) 
 = \int_{\R^n} \e^{- 2 \ii \pi x \ps \xi} \, \d \mu(x),
\]
when $f$ is in $L^1(\R^n) \cap L^2(\R^n)$ or when
$\mu$\label{FouriMu}
is a bounded measure on~$\R^n$. By the Plancherel theorem (some say
Parseval's theorem), we know that this defines a mapping from 
$L^1(\R^n) \cap L^2(\R^n)$ to $L^2(\R^n)$ that extends to a unitary
transformation $\ca F$ of~$L^2(\R^n)$.\label{PlanPars} 
The inverse mapping $\ca F^{-1}$ of $\ca F$ sends every square integrable
function $\xi \mapsto g(\xi)$ to $\ca F \bigl( \xi \mapsto g(-\xi) \bigr)$,
also expressible by $x \mapsto (\ca F g)(-x)$. We shall employ the notation 
$g^\vee = \ca F^{-1} g$\label{InveFour} 
for the inverse Fourier transform. 
\dumou

 The Plancherel--Parseval theorem extends to functions $f$ with values in a
Euclidean space~$F$, giving then an isometry from $L^2(\R^n, F)$ to itself. 
This is clear for instance by looking at coordinates in an orthonormal
basis of~$F$.
\dumou

 With this normalization of the Fourier transform, we have that
\[
 \widehat \gamma_n(\xi) = \e^{ - 2 \pi^2 |\xi|^2},
 \quad
 \xi \in \R^n,
\]
and the Fourier transform of the Poisson kernel $P_t$ on~$\R^n$ is equal to
$\e^{ - 2 \pi t |\xi|}$, for every $\xi \in \R^n$. Indeed, as the marginals 
on~$\R$ of $P_t$ are Cauchy distributions with the same parameter $t$, we
find by the residue theorem that
\[
   \widehat {P_t}(\xi) 
 = \int_\R \frac { t \e^{- 2 \ii \pi s \ms1 |\xi|} }
                 { \pi(t^2 + s^2) } \, \d s
 = \e^{- 2 \pi t |\xi|}.
\]
This information on the Fourier transform gives another way of checking
the semi-group property $P_s * P_t = P_{s+t}$ of Poisson distributions.
Using the Fourier inversion formula, we notice for future use that the
harmonic extension $u(x, t) = (P_t f)(x)$ of $f \in \ca S(\R^n)$ considered
in~\eqref{PSG} can be written as
\begin{equation}
   u(x, t)
 = \int_{\R^n} \e^{- 2 \pi t |\xi|} \widehat f(\xi) 
     \e^{ 2 \ii \pi x \ps \xi} \, \d \xi,
 \ms{16} x \in \R^n, \ms2 t > 0.
 \label{FutureUse}
\end{equation}

\subsection{Littlewood--Paley functions%
\label{LittlePal}}

\noindent
The Littlewood--Paley function $g(f)$ associated to a function $f$
on~$\R^n$ is defined by 
\[
 \label{GdeF}
 \forall x \in \R^n,
 \ms{12}
   g(f)(x)
 = \Bigl( \int_0^{+\infty} 
    \bigl| t \ms1 \nabla u(x, t) \bigr|^2 \, \frac {\d t} t 
   \Bigr)^{1/2},
\]
where $u$ is the harmonic extension of $f$ to the upper half-space 
in~$\R^{n+1}$, and where $\nabla u$ is the gradient of $u$ in~$\R^{n+1}$.
The classical theory, see for example Zygmund~\cite[vol.~2]{ZygmundTS} for the
circle case in Chap.~14, \S3 and Chap.~15, \S2, indicates that the norm of
$f$ in $L^p(\R^n)$, $1 < p < +\infty$, is equivalent to that of $g(f)$. One
has that
\begin{equation} 
     \kappa_p^{-1} \ms1 \|f\|_p
 \le \|g(f)\|_p 
 \le \kappa_p \ms1 \|f\|_p,
 \label{GFunction}
\end{equation} 
with a constant $\kappa_p$ depending on $p$, but independent of the
dimension~$n$. A variant of this Littlewood--Paley function is defined by
\begin{equation}
   g_1(f)^2
 = \int_0^{+\infty} 
    \Bigl|
     t \ms1 \frac \partial {\partial t} \ms2 P_t f 
    \Bigr|^2 \, \frac {\d t} t \up.
 \label{L-PfunctionG1}
\end{equation}
It is clear that $g_1(f) \le g(f)$, since $(\partial / \partial t) (P_t f)$
is a coordinate of the vector~$\nabla u$. The function $g_1$ is one of the
variants studied by Stein~\cite{SteinTHA}. More generally, for every
integer~$k \ge 1$, Stein sets
\[
   g_k(f)^2
 = \int_0^{+\infty} 
    \Bigl| t^k \frac {\partial^k} {\partial t^k} \ms2 P_t f 
    \Bigr|^2 \, \frac {\d t} t \up.
\]
\dumou

 Let us define $Q_j = P_{2^j} - P_{2^{j+1}}$, for every $j \in \Z$.
Since
\[
   \sum_{j \in \Z} |Q_j f|^2 
 = \sum_{j \in \Z} \ms2
    \Bigl|
     \int_{2^j}^{2^{j+1}} 
      \Bigl( \frac \partial {\partial t} \ms2 P_t f \Bigr) \, \d t
    \Bigr|^2,
\]
we obtain by Cauchy--Schwarz that
\[
     \sum_{j \in \Z} |Q_j f|^2 
 \le \sum_{j \in \Z} 2^j
      \int_{2^j}^{2^{j+1}}
       \Bigl|\frac \partial {\partial t} \ms2 P_t f
       \Bigr|^2 \, \d t
 \le \sum_{j \in \Z}
      \int_{2^j}^{2^{j+1}}
       \Bigl|\frac \partial {\partial t} \ms2 P_t f
       \Bigr|^2 \, t \, \d t
  =  g_1(f)^2.
\]
The classical result~\eqref{GFunction} on $g(f)$ implies that for 
$1 < p < +\infty$, there exists a constant $\textrm{q}_{\ms1 p}$
independent of the dimension $n$ such that
\begin{equation}
    \Bigl\| 
     \bigl( \sum_{j \in \Z} |Q_j f|^2 \bigr)^{1/2} 
    \Bigr\|_{L^p(\R^n)}
 \le \textrm{q}_p \ms2 \|f\|_{L^p(\R^n)},
 \quad f \in L^p(\R^n).
 \label{LiPaIneq}
\end{equation}
Observe that the same proof implies that a similar inequality, with a
different constant depending on $c > 1$, will hold for differences of the
form $\widetilde Q_j = P_{t_j} - P_{t_{j-1}}$, where $(t_j)_{j \in \Z}$ is
an increasing sequence of positive real numbers, provided that we have 
$t_{j+1} \le c \ms1 t_j$ for all $j\ms1$s. On the other hand, by Rota's
argument~\eqref{RBG}, one can obtain~\eqref{LiPaIneq} from the
Burkholder--Gundy inequalities of Theorem~\ref{BurGun}. Inequalities
similar to~\eqref{LiPaIneq} would hold for the Gaussian
semi-group~$(G_t)_{t \ge 0}$ defined in~\eqref{GSG}. Let us fix $T > 0$. We
have seen that $G_{2 t} f$, $0 \le t \le T$, is the projection on the
$\sigma$-field $\ca G_T$ generated by~$B_T$ of the member $M_{T - t}$ of the
Brownian martingale $M_s = (P_{T-s} f)(B_s)$, $0 \le s \le T$, running under
the infinite invariant measure given by the Lebesgue measure on $\R^n$. We
then apply~\eqref{RBG}. Using Gaussian $Q_j\ms1$s would allow us to avoid a
few minor technical difficulties later, and this is essentially what
Bourgain~\cite{BourgainCube} does for the cube problem, see
Section~\ref{LeCube}.
\dumou

 Relying on~\eqref{RBG} and Remark~\ref{BurkhoConst} gives for the constant
$\textrm{q}_{\ms1 p}$ in~\eqref{LiPaIneq} an upper bound of order 
$p / (p - 1)$ when $p \to 1$. This can also be obtained if one follows
Stein~\cite{SteinTHA}, p.~48--51. When $1 < p \le 2$, the proof given there
yields
$\|g(f)\|_p 
 \le (p-1)^{-1/2} \textrm{p}_p^{1 - p/2} \ms2 \|f\|_p$ for the right-hand
side inequality in~\eqref{GFunction}, where $\textrm{p}_p$ is the constant
in the maximal $L^p$-inequality for the Poisson semi-group. Since we have
$\textrm{p}_p \le p / (p-1)$ by~\eqref{MaxiPoiss}, we get that
\begin{equation}
 \textrm{q}_p \le p / (p-1)
 \ms{14} \hbox{when} \ms{10}
 1 < p \le 2.
 \label{EstiQp}
\end{equation}
\dumou

 Looking at the Fourier side, we see that
$
 \sum_{j \in \Z} \widehat Q_j(\xi) = 1
$
for every $\xi \ne 0$, since $\widehat {P_{2^j}} (\xi)
 = \e^{ - 2^{j+1} \pi |\xi|}$ tends to $1$ when $j \rightarrow - \infty$
and to $0$ when $j \rightarrow + \infty$. It implies for the convolution
operators, still denoted by~$Q_j$, that
\begin{equation}
 \sum_{j \in \Z} Q_j = \Id.
 \label{LesQjs}
\end{equation}

\subsubsection{Littlewood--Paley and maximal functions%
\label{MaxiLittlePal}}

\noindent
Stein~\cite[Chap.~III, \S~3, p.~75]{SteinTHA} explains how to get $L^p$
estimates for several maximal functions related to semi-groups, by using
the Littlewood--Paley functions. Consider a continuous function $\varphi$
on the half-line $[0, +\infty)$, differentiable on $(0, +\infty)$, and
denote by $\Phi$ its antiderivative vanishing at $0$. For every $t > 0$,
one has
\[
   t \ms1 \varphi(t)
 = \int_0^t \bigl( s \varphi(s) \bigr)' \, \d s
 = \int_0^t \varphi(s) \, \d s + \int_0^t s \ms1 \varphi'(s) \, \d s
 = \Phi(t) + \int_0^t s \ms1 \varphi'(s) \, \d s.
\]
Comparing $L^1$ and $L^2$ norms, one sees that
\[
     \int_0^t | s \ms1 \varphi'(s) | \, \frac {\d s} t
 \le \Bigl( 
      \int_0^t | s \ms1 \varphi'(s) |^2 \, \frac {\d s} t 
     \Bigr)^{1/2}
 \le \Bigl( 
      \int_0^t | s \ms1 \varphi'(s) |^2 \, \frac {\d s} s 
     \Bigr)^{1/2}.
\]
Therefore, one has
\[
     |\varphi(t)|
 \le \frac {|\Phi(t)|} t 
      + \Bigl( \int_0^{+\infty} 
         | s \ms1 \varphi'(s) |^2 \, \frac {\d s} s 
        \Bigr)^{1/2},
 \quad
 t > 0.
\]
One gets that
\[
     \sup_{t > 0} \ms1 |\varphi(t)|
 \le \sup_{t > 0} \ms1 \frac {|\Phi(t)|} t 
      + \Bigl( \int_0^{+\infty} 
         | s \ms1 \varphi'(s) |^2 \, \frac {\d s} s 
        \Bigr)^{1/2}.
\]
If $\varphi(s) = (P_s f)(x)$ for a given $x \in \R^n$, the upper bound
becomes
\[
     \sup_{t > 0} \ms1 |(P_t f)(x)|
 \le \sup_{t > 0} \ms1 \frac 1 t \ms1
      \Bigl| \int_0^t (P_s f)(x) \, \d s \Bigr|
       + g_1(f)(x).
\]
One can (again) control the norm in $L^p$, $1 < p < +\infty$, of the
maximal function of the Poisson semi-group, by the Hopf maximal
inequality and the estimate for the Littlewood--Paley function. This
control is easy in~$L^2$, especially when $L^2$ admits an orthonormal basis
$(f_j)$ such that $P_t f_j = \e^{ - t \lambda_j} f_j$ for every $j$, 
$\lambda_j \ge 0$, for example in the case of the Laplacian on a bounded
domain~$\Omega \subset \R^n$. If $f = \sum_j a_j f_j$ in $L^2(\Omega)$, one
has $P_t f = \sum_j a_j \e^{- \lambda_j t} f_j$, and 
\begin{align*}
   & \ms5 \int_\Omega g_1(f)(x)^2 \, \d x
 = \int_0^{+\infty} \int_\Omega \ms1
    \Bigl| 
     \sum_j a_j t \lambda_j \e^{- t \lambda_j} f_j(x) 
    \Bigr|^2
     \, \d x \ms2 \frac {\d t} t
 \\
 = & \ms5 \int_0^{+\infty} 
    \Bigl( \sum_j |a_j|^2 t^2 \lambda_j^2 \e^{- 2 t \lambda_j} \Bigr)
     \, \frac {\d t} t
 = \sum_j |a_j|^2 \int_0^{+\infty} 
          t^2 \lambda_j^2 \e^{- 2 t \lambda_j} \, \frac {\d t} t
 \\
  = & \ms5 
     \Bigl( \int_0^{+\infty} u^2 \e^{ - 2 u} \, \frac {\d u} u \Bigr)
      \sum_{\lambda_j > 0} |a_j|^2
 \le \frac {\Gamma(2)} 4 \ms1 \|f\|_2^2
  =  \frac 1 4 \ms1 \|f\|_2^2.
\end{align*}

\begin{smal}
\noindent
For the other Littlewood--Paley functions $g_k(f)$, one has in the same way
\begin{align*}
   \int_\Omega g_k(f)(x)^2 \, \d x
 &= \int_0^{+\infty} \int_\Omega \ms2
    \Bigl| 
     \sum_j a_j t^k \lambda_j^k \e^{- t \lambda_j} f_j(x) 
    \Bigr|^2
     \, \d x \ms2 \frac {\d t} t
 \\
  &=  \sum_j |a_j|^2 \int_0^{+\infty} 
          t^{2 k} \lambda_j^{2 k} \e^{- 2 t \lambda_j} \, \frac {\d t} t
 \le \frac {\Gamma(2 k)} {4^k} \ms1 \|f\|_2^2.
\end{align*}
One can also work on $\R^n$ by Fourier transform with Parseval. One gets
\[
   \int_{\R^n} g_k(f)(x)^2 \, \d x
 = (2 \pi)^{2 k} \int_0^{+\infty} \int_{\R^n}
    \bigl| \widehat f(\xi) t^k |\xi|^k \e^{- 2 \pi t |\xi|} \bigr|^2
     \, \d \xi \ms2 \frac {\d t} t
 = \frac {\Gamma(2 k)} {4^k} \ms1 \|f\|_2^2.
\]
We have also other relations like
\[
   t^2 \varphi'(t) 
 = \int_0^t \bigl( s^2 \varphi'(s) \bigr)' \, \d s
 = 2 \int_0^t s \varphi'(s) \, \d s + \int_0^t s^2 \varphi''(s) \, \d s
\]
implying that
\[
     \sup_{t > 0} |t \varphi'(t)|
 \le 2 \int_0^{+\infty} |s \varphi'(s)|^2 \frac {\d s} s
     + \int_0^{+\infty} |s^2 \varphi''(s)|^2 \frac {\d s} s \up.
\]
This brings back the successive maximal functions associated with each of
the expressions $t^k \partial^k / \partial t^k (P_t f)$, $k \ge 1$, to
quantities that can be estimated or are already estimated, as in
\[
     \sup_{t > 0} \ms1 
      \Bigl| t \ms2 \frac \partial {\partial t} (P_t f)(x) \Bigr|
 \le 2 \ms1 g_1(f)(x) + g_2(f)(x),
 \quad
 x \in \R^n.
\]
\end{smal}

\subsection{Fourier multipliers%
\label{FourierMult}}

\noindent
We introduce two dilation operators\label{DilatOper} 
that appear in duality, for instance when dealing with the Fourier
transform. Given a function~$g$ on~$\R^n$ and $\lambda > 0$, we use for
these operations the notation
\begin{equation}
 \ms{16}
 {\Di g \lambda} (x) = \lambda^{-n} g(\lambda^{-1} x),
 \quad
 {\di g \lambda} (x) = g(\lambda x),
 \quad x \in \R^n.
 \label{Dilata}
\end{equation}
If $g$ already has a subscript, as in $g = g_1$, we shall use the heavier
notation ${\Di {(g_1)} \lambda}$ or~${\di {(g_1)} \lambda}$. One sees, for
example when $g$ is integrable and $h$ bounded, that
\[
   \int_{\R^n} {\Di g \lambda}(x) \ms1 h(x) \, \d x
 = \int_{\R^n} g(y) \ms1 {\di h \lambda}(y) \, \d y,
 \ms{16} \hbox{and} \ms{16}
 \widehat{ ({\Di g \lambda}) } (\xi) = \widehat g(\lambda \xi),
 \quad \xi \in \R^n,
\]
that is to say, we have $\widehat{ ({\Di g \lambda}) }
 = \di {(\widehat {g \ns{0.14}} \ms1 )} \lambda$. Clearly,
$\Di g { \lambda \mu } = \Di {(\Di g \lambda)} \mu$.
The $\Di g \lambda$ dilation preserves the integral of~$g$; it is extended
to measures $\mu$ on $\R^n$ by setting 
${\Di \mu \lambda}(f) = \mu({\di f \lambda})$, namely
\begin{equation}
   \int_{\R^n} f(x) \, \d {\Di \mu \lambda}(x)
 = \int_{\R^n} f(\lambda x) \, \d \mu(x)
 \label{MuLambda}
\end{equation}
for every $f$ in the space $\ca K(\R^n)$ of continuous and compactly
supported functions. The measure $\Di \mu \lambda$ is the image of $\mu$
under the mapping $\R^n \ni x \mapsto \lambda \ms1 x$.
If $\d \mu(x) = g(x) \, \d x$, then ${\Di g \lambda}$
is the density of ${\Di \mu \lambda}$.
\dumou

 Let $\xi \mapsto m(\xi)$ belong to $L^\infty(\R^n)$.
For $f \in L^2(\R^n)$, we have $\widehat f \in L^2(\R^n)$ by Plancherel,
$\xi \mapsto m(\xi) \widehat f(\xi)$ is also in $L^2(\R^n)$ and is
therefore the Fourier transform of some function $T_m f \in L^2(\R^n)$. We
thus get a linear operator $T_m$ on $L^2(\R^n)$ if we define~$T_m f$, for
every $f \in L^2(\R^n)$, by means of its Fourier transform, letting
\[
 \label{TsubM}
 (T_m f) \PF (\xi) = m(\xi) \widehat f(\xi),
 \quad \xi \in \R^n.
\]
Let $P_m$ be the operator of multiplication by $m$, defined by
$P_m \varphi = m \varphi$. The operator $T_m = \ca F^{-1} P_m \ca F$ is
bounded on $L^2(\R^n)$ since by Parseval, one has that
\begin{equation}
     \int_{\R^n} |(T_m f)(x)|^2 \, \d x
  =  \int_{\R^n} |m(\xi)|^2 |\widehat f(\xi)|^2 \, \d \xi
 \le \|m\|_\infty^2 \ms2 \|f\|_2^2.
 \label{EasyBoundA}
\end{equation}
We shall say that $T_m$ is the operator \emph{associated to the multiplier
$m$}.
\dumou

 One can ask whether $T_m$ also operates as a bounded mapping on certain
$L^p$ spaces. In this survey, \og bounded on $L^p$\fge will always mean
\emph{bounded from $L^p$ to~$L^p$}. Let $q$ be the \emph{conjugate
exponent} of $p$, defined by $1 / q + 1 / p = 1$. Assuming that 
$1 < p < +\infty$, we see that $T_m$ is bounded on $L^p(\R^n)$ if and only
if $\int_{\R^n} m(\xi) \widehat \varphi(\xi) \widehat \psi(\xi) \, \d \xi$
is uniformly bounded when $\varphi, \psi \in \ca S(\R^n)$ belong to the unit
balls of $L^p(\R^n)$ and $L^q(\R^n)$ respectively, hence $T_m$ is then also
bounded on $L^q(\R^n)$ (and on $L^2(\R^n)$ by interpolation, so $m$ has to
be a bounded function, see the line after~\eqref{EasyBound}).
\dumou

 We now observe that the multiplier $m$ and its dilates 
${\di m \lambda} : \xi \mapsto m(\lambda \xi)$, $\lambda > 0$, define
operators having equal norms on $L^p(\R^n)$. We see that
\[
   (T_{\di m \lambda} {\Di f \lambda}) \PF (\xi)
 = m(\lambda \xi) \widehat f(\lambda \xi)
\]
hence
$
   T_{\di m \lambda} {\Di f \lambda}
 = \Di {(T_m f)} \lambda
$.
Consider the operator $S_\lambda : f \mapsto {\Di f \lambda}$. For every 
$p \in [1, +\infty]$ and $1 / q + 1 / p = 1$, the multiple 
$S_{\lambda, p} := \lambda^{n / q} S_\lambda$ of $S_\lambda$ is an
isometric bijection of $L^p(\R^n)$ onto itself. The relation
$
 S_\lambda \circ T_m = T_{ \di m \lambda } \circ S_\lambda
$
becomes
\begin{equation}
   T_{ \di m \lambda } 
 = S_{\lambda, p}
    \ms2 T_m \ms2
   S_{\lambda, p}^{-1}
 \label{InvariMul}
\end{equation}
and this implies that $T_m$ and $T_{ \di m \lambda }$ have the same norm on
$L^p(\R^n)$. More generally, let~$\gr m = (m^{(j)})_{j \in J}$ be a family
of multipliers and define $T_{\gr m} f = \sup_{j \in J} |T_{m^{(j)}} f|$.
If we set ${\di {\gr m} \lambda}
 = \bigl( {\di {m^{(j)}} \lambda} \bigr)_{j \in J}$,
then we have again that
\begin{equation}
   T_{\di {\gr m} \lambda}
 = S_{\lambda, p}
    \ms2 T_{\gr m} \ms2
   S_{\lambda, p}^{-1}
 \label{InvariMulB}
\end{equation}
because $S_\lambda$ commutes with $f \mapsto |f|$ and
$S_\lambda(\sup_{j \in J} f_j) = \sup_{j \in J} S_\lambda f_j$.
Consequently, $T_{\di {\gr m} \lambda}$ and $T_{\gr m}$ also have the same
norm on $L^p(\R^n)$.
 
 We shall speak of the \emph{action on $L^p$ of the multiplier $m$}
and set
\[
 \label{MPdansP}
 \|m\|_{p \rightarrow p} := \|T_m\|_{p \rightarrow p}.
\]
If $T_m$ is bounded on $L^p$, one says that $m$ is a \emph{multiplier on
$L^p$}, or a \emph{$L^p$-multiplier}. The next lemma will be useful,
it is nothing but a direct consequence of the equality
$\| {\di m \lambda} \|_{p \rightarrow p} = \| m \|_{p \rightarrow p}$ for
every $\lambda > 0$, and of the triangle inequality in $L^p$.

\begin{lem}\label{EasyIntegral}
Suppose that\/ $1 \le p \le +\infty$ and that $m(\xi)$ is a
$L^p(\R^n)$-multiplier. If the function $\psi$ is integrable on\/ 
$(0, +\infty)$, the multiplier $N$ defined by
\[
   N(\xi)
 = \int_0^{+\infty} \psi(\lambda) \ms1 m(\lambda \xi) \, \d \lambda,
 \quad
 \xi \in \R^n,
\]
is a $L^p(\R^n)$-multiplier and\/
$
     \|N\|_{p \rightarrow p} 
 \le \|\psi\|_{L^1(0, +\infty)} \ms2 \|m\|_{p \rightarrow p}
$.
\end{lem}

 Note that clearly, multiplier operators commute to each other, and commute
to translations and differentiations. We will apply many times the easy
fact~\eqref{EasyBoundA}, which can be written as
\begin{subequations}\label{EasyB}
\begin{equation}
      \|m\|_{2 \rightarrow 2} 
   =  \|T_m\|_{2 \rightarrow 2} 
 \le \|m\|_{ L^\infty(\R^n) }.
 \label{EasyBound} \tag{\ref{EasyB}.{\bf P}}
\end{equation}
\end{subequations}
The inequality is actually an equality, since by Parseval, the norm of
$T_m$ on $L^2(\R^n)$ is equal to that of $P_m$, the multiplication operator
by~$m$.
\dumou

  If $K$ is a function integrable on $\R^n$, it acts by convolution
on~$L^p(\R^n)$ for all values $1 \le p \le +\infty$, and one gets easily
by convexity of the $L^p$ norm that
\begin{equation}
 \|K * f\|_{ L^p(\R^n) } \le \|K\|_{ L^1(\R^n) } \ms2 \|f\|_{ L^p(\R^n) }.
 \label{LOneMultiplier}
\end{equation}
This is an easy example of operator associated to a multiplier, since
convolution of~$f$ with $K$ corresponds to multiplication of $\widehat f$ by
$\widehat K$. The Fourier transform $m = \widehat K$ of $K$ is thus a
multiplier on all spaces~$L^p(\R^n)$, $1 \le p \le +\infty$. Consider the
Fourier transform $m$ of the \emph{convolution kernel} $K \in L^1(\R^n)$,
equal to
\[
   m(\xi)
 = \int_{\R^n} K(x) \e^{- 2 \ii \pi x \ps \xi} \, \d x,
 \quad \xi \in \R^n.
\]
For $\xi \ne 0$, let $\xi = |\xi| \ms1 \theta$ and $x = y + s \ms1 \theta$,
where $y$ is in the hyperplane $\theta^\perp$ orthogonal 
to~$\theta \in S^{n-1}$, and $s \in \R$. By Fubini, we have for every real
number~$u$ that
\[
   m(u \xi)
 = \int_\R 
    \Bigl( \int_{\theta^\perp} K(y + s \ms1 \theta) \, \d^{n-1} y \Bigr)
     \e^{- 2 \ii \pi s u \ms1 |\xi|} \, \d s,
\]
where $\d^{n-1} y$ denotes the normalized Lebesgue measure on the Euclidean
space $\theta^\perp \subset \R^n$. In what follows we associate to $K$ and
to $\theta$ in the unit sphere $S^{n-1}$ the function $\varphi_{\theta, K}$
defined on $\R$ by
\begin{equation}
 \forall s \in \R,
 \ms{12}
    \varphi_{\theta, K} (s)
 := \int_{\theta^\perp} K(y + s \ms1 \theta) \, \d^{n-1} y,
 \label{VarphiTheta}
\end{equation}
so that for $\xi \ne 0$ and $\theta = |\xi|^{-1} \ms1 \xi$, letting
$\varphi_\theta = \varphi_{\theta, K}$ we have
\begin{equation}
   m(u \xi)
 = \int_\R \varphi_\theta(s) \e^{ - 2 \ii \pi s u \ms1 |\xi|} \, \d s
 = \int_\R \frac 1 {|\xi|} \varphi_\theta \Bigl( \frac v {|\xi|} \Bigr) 
    \e^{ - 2 \ii \pi v u} \, \d v.
 \label{FourierPhi}
\end{equation}
The function $\R \ni u \mapsto m(u \theta)$ is the Fourier transform (in
dimension~1) of $\varphi_\theta$.

\subsubsection{Multipliers \og of Laplace type\fg%
\label{LaplaceMulTyp}}

\noindent
We consider a scalar function~$F$ on $(0, +\infty)$ that admits an
expression of the form
\begin{equation}
 \forall \lambda > 0,
 \ms{12}
 F(\lambda) = \lambda \int_0^{+\infty} \e^{- \lambda t} a(t) \, \d t,
 \label{FunctionF}
\end{equation}
where $a$ is a measurable function bounded on $(0, +\infty)$. The
multiplier $m(\xi)$ \og of Laplace type\fge associated to $F$ is defined by 
$m(\xi) = F(|\xi|)$, for $\xi \in \R^n$. We note that 
$\|F\|_\infty \le \|a\|_\infty$, thus by~\eqref{EasyBound}, this multiplier
$m$ is bounded on $L^2(\R^n)$ with operator norm $\le \|a\|_\infty$. Stein
proves the following result.

\def\CiteSteinTtrois{\cite[Theorem 3', p.~58]{SteinTHA}}
\begin{prp}[Stein~\CiteSteinTtrois]%
\label{LaplaceMultip}
Let $F$ be defined on\/ $(0, +\infty)$ by\/~\eqref{FunctionF}, for some
function $a \in L^\infty(0, +\infty)$. The operator $T_m$ associated to
the multiplier $m(\xi) = F(|\xi|)$ is bounded on $L^p(\R^n)$ for\/ 
$1 < p < +\infty$ and
\[
     \|T_m\|_{p \rightarrow p} 
 \le \lambda_p \ms1 \|a\|_\infty,
\]
where $\lambda_p$ is a constant independent of the dimension $n$.
\end{prp}

 The identity operator belongs to this class (when $a(t) \equiv 1$), we
thus see that $\lambda_p \ge 1$ for every $p$. It follows from the
proposition that the imaginary powers of $(- \Delta)^{1/2}$ act on the
spaces~$L^p(\R^n)$ when $1 < p < +\infty$, with norms bounded independently
of the dimension~$n$. Indeed, we have the formula of Laplace type
\begin{equation}
   \lambda^{\ii b} 
 = \frac 1 {\Gamma(1 - \ii b)} \ms3
    \lambda \int_0^{+\infty} \e^{ - \lambda t} t^{- \ii b} \,  \d t,
 \ms{18}
 \lambda > 0,
 \ms{10}
   a(t) 
 = \frac {t^{- \ii b}} {\Gamma(1 - \ii b)} \ms1 \up,
 \label{LaplaPow}
\end{equation}
hence
$
     \|a\|_\infty 
  =  | \Gamma(1 - \ii b) |^{-1}
$, for every $b \in \R$.
According to the estimate~\eqref{IneGamma} for the Gamma function,
we get from Proposition~\ref{LaplaceMultip} that
\begin{equation}
 \forall b \in \R,
 \ms{12}
     \bigl\| \ms1 |\xi|^{\ii b} \ms1 \bigr\|_{p \rightarrow p}
 \le \lambda_p \ms1 (1 + b^2)^{-1/2} \e^{ \pi \ms1 |b| / 2 },
 \quad
 1 < p < +\infty.
 \label{ImagiPow}
\end{equation}

\begin{smal}
\noindent
Stein's proof of Proposition~\ref{LaplaceMultip} draws on $L^p$
inequalities for the Littlewood--Paley functions $g_1(f)$ and~$g_2(f)$, and
a comparison $g_1(T_m f) \le \kappa \ms2 g_2(f)$. We now sketch another
possibility, which invokes martingale inequalities. If $F$ is as in
Proposition~\ref{LaplaceMultip} and $m(\xi) = F(|\xi|)$, then
$T_m f$, for $f \in \ca S(\R^n)$, can be expressed by
\begin{subequations}\label{LaplaceMultiplier}
\SmallDisplay{\eqref{LaplaceMultiplier}}
\[
   - \bigl( T_{ \di m {2\pi} } f \bigr)
 = \int_0^{+\infty} a(t) 
     \Bigl( \frac \partial {\partial t} \ms2 P_t f \Bigr) \, \d t.
\] 
\end{subequations}
\noindent
Indeed, we know by~\eqref{FutureUse} that
\[
   (P_t f)(x)
 = u(x, t) 
 = \int_{\R^n} \e^{ - 2 \pi t |\xi|} \widehat {f \ms1} \ns{0.7} (\xi) 
    \e^{2 \ii \pi x \ps \xi} \, \d \xi
\]
and
\begin{align*}
   \Bigl( \int_0^{+\infty} a(t) 
     \Bigl( \frac \partial {\partial t} \ms2 P_t f \Bigr)  \, \d t
   \Bigr) (x)
 &= \int_0^{+\infty} a(t) 
    \Bigl( \int_{\R^n} (- 2 \pi |\xi| ) \e^{ - 2 \pi t |\xi|}
     \widehat {f \ms1} \ns{0.7} (\xi) \e^{2 \ii \pi x \ps \xi} \, \d \xi \Bigr)
      \, \d t
 \\
 &= - \int_{\R^n} F(2 \pi |\xi|) \ms2 \widehat {f \ms1} \ns{0.7} (\xi) 
      \e^{2 \ii \pi x \ps \xi} \, \d \xi
 = - \bigl( T_{ \di m { 2\pi } } f \bigr) (x).
\end{align*}

 Suppose that $a$ is a step function supported in 
$[t_0, t_N] \subset [0, +\infty)$. Then
\[
 a(t) = \sum_{j=1}^N a_j \ms1 \gr 1_{ [t_{j-1}, t_j) } (t), 
\]
with $0 = t_0 < t_1 < \ldots < t_N$. By~\eqref{LaplaceMultiplier}, we
obtain that
\[
   - T_{ \di m { 2\pi } } f
 = \sum_{j=1}^{N} a_j (P_{t_{j}} - P_{t_{j-1}})(f).
\]
It follows that $T_m f$ can be considered as projection of a martingale
transform by a conditional expectation $\E_{\ca G}$. Let $u_j = t_j/2$,
$j = 0, \ldots, N$, and $T := u_N$. We have seen in~\eqref{MartinSG} that 
$P_{t_j} f = P_{2u_j} f$ is the image under the projection $\E_{\ca G}$ of
the martingale member $M_{T - u_j} = (P_{u_j} f)(X_{T - u_j})$, so letting
$L_i = M_{T - u_{N - i}}$, $i = 0, \ldots, N$, we see that 
$T_{ \di m { 2\pi } } f$ is equal to
\[
   \E_{\ca G} \Bigl( \sum_{j=1}^N a_j (M_{T - u_{j-1}} - M_{T - u_j}) \Bigr)
 = \E_{\ca G} \Bigl( \sum_{i=1}^N a_{N - i + 1} (L_i - L_{i-1}) \Bigr),
\]
which is the transform of the martingale $(L_i)_{i=0}^N$ by the bounded
non-random multipliers $(a_{N - i + 1})_{i=1}^N$. Also, $L_N$ is equal to
$M_T = f(X_T)$ that has the distribution of $f$ with respect to the
(infinite) invariant measure, the Lebesgue measure on~$\R^n$ (see
Remark~\ref{InfiniProba}), hence $\|f\|_p = \|M_T\|_p$. In this simple case,
one deduces Proposition~\ref{LaplaceMultip} from Remark~\ref{Incondi} about
the Burkholder--Gundy inequalities, and it can be easily generalized, first
to compactly supported continuous functions $a$. Using
Remark~\ref{BurkhoConst}, we find in this way that
\begin{subequations}\label{ConstaLambda}
\SmallDisplay{\eqref{ConstaLambda}}%
\[
 \lambda_p \le \kappa \ms2 p^*,
 \qquad p^* := \max(p, p / (p-1)),
 \ms{10}
 1 < p < +\infty.
\]
\end{subequations}

\end{smal}

\subsection{Riesz transforms%
\label{TransfoRiesz}}

\noindent
In dimension $1$, there is only one Riesz transform $R$, which is called
\emph{the Hilbert transform\/}~$H$.\label{HilbTrans}
It is defined for $f \in L^2(\R)$ by
\[
 \forall \xi \in \R,
 \ms{14}
   (R f) \PF (\xi) 
 = (H f) \PF (\xi) 
 = - \frac { \ii \xi} {|\xi|} \ms2 \widehat f(\xi).
\]
This is given by a multiplier of constant modulus~$1$ (almost everywhere),
thus the transformation is isometric and invertible on $L^2(\R)$ by
Parseval, and~$H$ is a unitary operator on $L^2(\R)$ with 
inverse~$H^{-1} = - H$. If $\widetilde u(x, t)$ denotes the harmonic
extension of $H f$ to the upper half-plane, then 
$u(x, t) + \ii \widetilde u(x, t)$ is a holomorphic function of the complex
variable $z = x + \ii t$, because its Fourier transform vanishes for 
$\xi < 0$, implying by inverse Fourier transform that $u(x, t)$ is an
integral in $\xi > 0$ of the holomorphic functions
$
    \e^{ - 2 \pi |\xi| t } \e^{2 \ii \pi \xi x}
  = \e^{ 2 \ii \pi \xi (x + \ii t) }
$.
A classical theorem going back to Marcel Riesz~\cite{RieszHT} states that
the Hilbert transform is bounded on $L^p(\R)$ when $1 < p < +\infty$. This
is also a consequence of the results on the Littlewood--Paley function
$g(f)$, or of martingale inequalities as we shall see below. Some of the
first deep connections between Brownian motion and classical Harmonic
Analysis can be found in Burkholder--Gundy--Silverstein~\cite{BGS}.

\begin{smal}
\noindent
The Brownian argument is easier for the Hilbert transform $H_\T$ on the
unit circle~$\T \subset \R^2$. Let $(B_t)_{t \ge 0}$ be a plane Brownian
motion defined on some $(\Omega, \ca F, P)$, starting from~$0$ in $\R^2$,
and let $\tau$ be the first time $t$ when $B_t$ hits the circle~$\T$. By
rotational invariance, the distribution of~$B_\tau$ is the uniform
probability measure on the circle. Let~$f$ be a function in $L^p(\T)$ and
let $u$ be its harmonic extension to the unit disk. Assume that 
$2 \pi \ms1 u(0)
 = \int_0^{2 \pi} f(\cos \theta, \sin \theta) \, \d \theta = 0$, and
denote by $a \wedge b$ the minimum of $a$ and $b$ real. The random process
$(M_t)_{t \ge 0} = (u(B_{t \wedge \tau}))_{t \ge 0}$ is a
Brownian martingale, which can be expressed by the It\={o} integral
\[
   u(B_{t \wedge \tau})
 = \int_0^{t \wedge \tau} \nabla u(B_s) \ps \d B_s.
\]
Suppose that $1 < p < +\infty$. By the continuous version of the
Burkholder--Gundy inequalities, the norm
$\|f\|_{L^p(\T)} = \|u(B_\tau)\|_{L^p(\Omega, \ca F, P)}$ is equivalent to
the norm in $L^p(\Omega, \ca F, P)$ of the square function of the martingale
$(M_t)_{t \ge 0}$, given by
\[
    S(f)
 := \Bigl( \int_0^\tau |\nabla u(B_s)|^2 \, \d s \Bigr)^{1/2}.
\]
If $\widetilde f = H_\T f$ denotes the function on $\T$ conjugate to $f$ and
$\widetilde u$ its harmonic extension to the unit disk, then 
$|\nabla \widetilde u(x)| = |\nabla u(x)|$ for $x$ in the unit disk,
according to the Cauchy--Riemann equations for the function 
$u + \ii \widetilde u$ holomorphic in the disk. It follows that 
$S(\widetilde f) = S(f)$ and the $L^p$-boundedness of the Hilbert transform
for the circle is established via the Burkholder--Gundy inequalities of
Theorem~\ref{BurGun}. The bound for the norm of $H_\T$ obtained in this
manner is related to the constants in Burkholder--Gundy. The exact value of
the $L^p$ norm of $H$ is known, this is due to Pichorides~\cite{Pich}, see
Remark~\ref{RemIwaMar} below.

\end{smal}

\noindent
In dimension $n$, there are $n$ Riesz transforms $R_j$, defined on
$L^2(\R^n)$ by
\[
 \label{RieszTra}
   (R_j f) \PF (\xi) 
 = - \frac { \ii \xi_j} {|\xi|} \ms2 \widehat f(\xi),
 \ms{16} j = 1, \ldots, n.
\]
Since
$\bigl\| \bigl(\sum_{j=1}^n |R_j f|^2 \bigr)^{1/2} \bigr\|_2^2
 = \sum_{j=1}^n \|R_j f\|_2^2$,
one has by Parseval that
\begin{equation}
     \Bigl\| \Bigl(\sum_{j=1}^n |R_j f|^2 \Bigr)^{1/2} \Bigr\|_2
  =  \|f\|_2.
 \label{Collective2}
\end{equation}
The Riesz transforms are \og collectively bounded\fge in $L^p(\R^n)$, by a
constant~$\rho_p$ independent of the dimension~$n$
(Stein~\cite{SteinTV}), meaning that
\begin{equation}
     \Bigl\| \Bigl(\sum_{j=1}^n |R_j f|^2 \Bigr)^{1/2} \Bigr\|_p
 \le \rho_p \ms1 \|f\|_p,
 \quad 1 < p < +\infty.
 \label{TransfosRiesz}
\end{equation}
Duoandikoetxea and Rubio de Francia~\cite{DuoRubio} have connected in a few
lines this inequality to the properties of the Hilbert transform (see also
Pisier~\cite{PisierRT}).

\begin{proof}
For each nonzero vector $u$ in~$\R^n$, let us introduce on $L^2(\R^n)$ the
Hilbert transform $H_u$ \emph{in the direction $u$} by setting
\[
 \forall \xi \in \R^n,
 \ms{16}
   (H_u f) \PF (\xi) 
 = - \frac { \ii u \ps \xi} {|u \ps \xi|} \ms2 \widehat f(\xi)
 = - \ii \sign (u \ps \xi) \ms2 \widehat f(\xi).
\]
We deduce easily from the one-dimensional case that $H_u$ acts on
$L^p(\R^n)$, with the same norm as that of $H$ on $L^p(\R)$. It is enough
to check the case when $u$ is the first basis vector $\gr e_1$; if one
writes the points $x$ in $\R^n$ as $x = (t, y)$, $t \in \R$, 
$y \in \R^{n-1}$, and if for $f$ belonging to the Schwartz class 
$\ca S(\R^n)$ we set
$
 f_y(t) = f(t, y)
$,
we can see that $(H_{\gr e_1} f)(t, y) = (H f_y)(t)$. Then, applying
Fubini's theorem, we obtain
\begin{align*}
     \dint |(H_{\gr e_1} f)(t, y)|^p \, \d t \ms1 \d y
  &=  \int_{\R^{n-1}} \Bigl( \int_\R |(H f_y)(t)|^p \, \d t \Bigr) \, \d y
 \\
 &\le \|H\|_{p \rightarrow p}^p 
      \int_{\R^{n-1}} \Bigl( \int_\R |f_y(t)|^p \, \d t \Bigr) \, \d y
  =  \|H\|_{p \rightarrow p}^p \ms1 \|f\|_p^p.
\end{align*}

 We can consider that $\ca R f = (R_1 f, \ldots, R_n f)$\label{VRieszTr} 
is the operator associated to the vector-valued multiplier
\[
 \gr m(\xi) = - \ii |\xi|^{-1} \xi \in \R^n,
 \quad
 \xi \in \R^n, 
\]
that is to say, the operator sending $f \in \ca S(\R^n)$ to the function
$T_{\gr m} f$ from $\R^n$ to~$\R^n$ whose $\R^n$-valued Fourier transform
is equal to $\widehat f(\xi) \ms1 \gr m(\xi)$. For $f \in \ca S(\R^n)$, let
us look at the vector-valued integral
\[
   (\ca H f)(x)
 = \int_{\R^n} (H_u f)(x) \ms1 u \, \d \gamma_n(u) \in \R^n,
 \quad
 x \in \R^n,
\]
where $\gamma_n$ is the Gaussian probability measure
from~\eqref{LoiNZeroId}. The operator $\ca H$ corresponds to the
vector-valued multiplier defined when $\xi \ne 0$ by
\[
   - \ii \int_{\R^n} \sign(u \ps \xi) \ms1 u \, \d \gamma_n(u)
 = - \ii \Bigl( \int_{\R} |v| \, \d \gamma_1(v) \Bigr) |\xi|^{-1} \xi
 = - \ii \sqrt{ \frac 2 \pi } \ms2 |\xi|^{-1} \xi.
\]
This can be seen by integrating on affine hyperplanes orthogonal to $\xi$.
The \og normalized\fge partial integral on the hyper\-plane 
$\xi^\perp + v |\xi|^{-1} \xi$, $v \in \R$, is equal to 
\[
    \int_{\xi^\perp} 
     \sign(v) \ms1 (w + v |\xi|^{-1} \xi) \, \d \gamma_{\xi^\perp} (w)
 = |v| \ms2 |\xi|^{-1} \xi.
\] 
It follows that $\ca R f = \sqrt{ \pi / 2 } \ms3 \ca H f$. For~$x$ fixed,
the norm of $(\ca H f)(x)$ is the supremum of scalar products with 
vectors $\theta \in S^{n-1}$, and letting $1 / q + 1 / p = 1$, one has that
\begin{align*}
     |(\ca H f)(x)|
 &\le \sup_{\theta \in S^{n-1}} \ms2
      \int_{\R^n} |\theta \ps u| \ms1 |(H_u f)(x)| \, \d \gamma_n(u)
 \\
 &\le \Bigl( \int_{\R} |v|^q \, \d \gamma_1(v) \Bigr)^{1/q}
     \ms1 \Bigl( \int_{\R^n}  |(H_u f)(x)|^p \, \d \gamma_n(u) 
          \Bigr)^{1/p}.
\end{align*}
Using the notation $g_q$ of~\eqref{AsymptoGauss} for the Gaussian moments,
we get
\begin{align}
     &\ms4 \Bigl( \int_{\R^n} |(\ca R f)(x)|^p \, \d x \Bigr)^{1/p}
  =  \sqrt{ \frac \pi 2 } \ms2
      \Bigl( \int_{\R^n} |(\ca H f)(x)|^p \, \d x \Bigr)^{1/p}
 \label{RieszBd}
 \\
 \le &\ms4 \sqrt{ \frac \pi 2 } \ms3 g_q
     \ms2 \Bigl( \int_{\R^n} \int_{\R^n}  
      |(H_u f)(x)|^p \, \d \gamma_n(u) \ms2 \d x \Bigr)^{1/p}
 \le \ms4 \sqrt{ \frac \pi 2 } \ms3 g_q \ms2 
       \|H\|_{p \rightarrow p} \ms2 \|f\|_p.
 \notag
\end{align}
This argument yields
$\rho_p \le \sqrt{\pi/2} \ms4 g_q \ms3 \|H\|_{p \rightarrow p}$ for the
constant $\rho_p$ in~\eqref{TransfosRiesz}. When $p = 2$, this gives
$\rho_2 \le \sqrt{\pi/2}$ instead of the correct value $\rho_2 = 1$
of~\eqref{Collective2}. When $p$ tends to $1$, we obtain
by~\eqref{AsymptoGauss} that
\[
     \Bigl( \int_{\R^n} |(\ca R f)(x)|^p \, \d x \Bigr)^{1/p}
 \le \kappa \ms2 \sqrt q \ms2 \|H\|_{p \rightarrow p} \ms2 \|f\|_p.
 \qedhere
\] 
\end{proof}

\begin{remN}\label{RemIwaMar}
The value $g_q = \bigl( \int_{\R} |v|^q \, \d \gamma_1(v) \bigr)^{1/q}$
tends to $\sqrt { 2 / \pi}$ when $p$ tends to~$+\infty$, and the asymptotic
result $\rho_p \simeq \|H\|_{p\to p}$ obtained from~\eqref{RieszBd} in this
case is essentially best possible. Indeed, Iwaniec and Martin~\cite{IwaMar}
have shown that the operator norm on $L^p(\R^n)$ of each individual Riesz
transform $R_j$, $j = 1, \ldots, n$, is equal to the one of the Hilbert
transform $H$ on $L^p(\R)$, hence 
$1 \le \|H\|_{p \rightarrow p} \le \rho_p$. According to
Pichorides~\cite{Pich}, the norm of the Hilbert transform is given by
\[
   \|H\|_{p\to p}
 = \cot 
    \Bigl(
     \frac {\ms1 \pi} { 2 p^* \ns5 } \ms3
    \Bigr),
  \ms{10} \hbox{with} \ms{10}
  p^* = \max \bigl( p, p / (p-1) \bigr).
\]
Iwaniec and Martin~\cite{IwaMar} also bound the \og collective\fge norm
in~\eqref{TransfosRiesz} by $\sqrt 2 \ms2 H_p(1)$, where $H_p(1)$ is the
norm on $L^p(\C) \simeq L^p(\R^2)$ of the \og complex Hilbert transform\fg,
which corresponds to the multiplier 
$\C \ni \xi \mapsto \ii \ms1 |\xi|^{-1} \xi$, in other words, the operator
$R_1 + \ii R_2$ on $L^p(\R^2)$.
Iwaniec and Sbordone~\cite[Appendix]{IwaSbo} add a few lines
and give $H_p(1) \le \frac \pi 2 \|H\|_{p \to p}$ so that finally
\begin{equation}
     \rho_p 
 \le \sqrt 2 \ms2 H_p(1)
 \le \frac \pi {\sqrt 2} \ms2 \|H\|_{p \to p}
 \le \kappa \ms2 p^*.
 \label{EstiRho}
\end{equation}
\end{remN}

\begin{rem}
The proof from~\cite{DuoRubio} is in the spirit of the \emph{method of
rotations}, which uses integration in polar coordinates to get directional
operators in its radial part, see also Section~\ref{SphericOp}. 
With this method, one can relate to the Hilbert transform not only the
Riesz transforms, but also more general singular integrals with odd
kernel, see~\cite[Section 5.2]{GrafakosCFA} for example.
\end{rem}

\section{Analytic tools%
\label{AnaTool}}

\subsection{Some known facts about the Gamma function%
\label{FoGamm}}

\noindent
From Euler's formula\label{EulerForm}
\[
 \forall z \in \C \setminus (-\N),
 \ms{24}
   \Gamma(z)
 = \lim_{n \rightarrow \infty} \ms2
    \frac {n! \ms5 n^z} {z \ms2 (z+1) \ldots (z+n)} \ms1 \up,
\]
one passes to the Weierstrass infinite product for $1 / \Gamma$, stating
that
\[
   \frac 1 {\Gamma(z+1)}
 = \frac 1 {z \Gamma(z)}
 = \e^{\gamma z}
    \prod_{n=1}^\infty \ms1 \Bigl( (1 + z / n ) \e^{ - z/n} \Bigr),
\]
where $\gamma$ is the Euler--Mascheroni constant. It follows that 
$1 / \Gamma$ is an entire function, with simple zeroes 
$z = 0, -1, -2, \ldots\ms5$. For the interpolation arguments to come, we
need upper estimates on the modulus of 
$1 / \Gamma(\sigma + \ii \tau)$ for $\sigma, \tau$ real. From the preceding
formula and from $\Gamma(\overline z) = \overline{\Gamma(z)}$, we infer that
\begin{equation}
   \Bigl| \frac 1 {\Gamma(1 + \ii \tau)} \Bigr|^2
 = \prod_{n=1}^\infty \ms1 \Bigl( 1 + \frac {\tau^2} {n^2} \Bigr)
 = \frac {\sinh(\pi \tau)} {\pi \tau} \up,
 \quad
 \tau \in \R,
 \label{RefFor}
\end{equation}
according to another result due to Euler, the famous formula
\begin{subequations}\label{EulerFormulaG}
\begin{equation}
   \frac {\sin (\pi z)} {\pi z}
 = \prod_{n = 1}^\infty \Bigl( 1 - \frac {z^2} {n^2} \Bigr).
 \label{EulerFormula} \tag{\ref{EulerFormulaG}.{\bf E}}
\end{equation}
\end{subequations}
The connoisseur has seen that we just came upon a special case of the
\og Euler reflection formula\fg, stating that
$\Gamma(z)^{-1} \Gamma(1 - z)^{-1} = \sin(\pi z) / \pi$ for every 
$z \in \C$, or equivalently
$\Gamma(1 + z)^{-1} \Gamma(1 - z)^{-1} = \sin(\pi z) / (\pi z)$. For every
$x$ real, one has
\begin{equation}
     \frac {\sinh(\pi x)} {\pi x}
 \le \frac { \e^{ \pi |x| } } { 1 + \pi |x| }
 \le \frac { \e^{ \pi |x| } } { (1 + x^2)^{1/2} } \up.
 \label{Calculus}
\end{equation}
The right-hand inequality is evident, the left-hand one is
equivalent to saying that for every $y \ge 0$, we have 
$(1 + y) \sinh(y) \le y \e^y$ or 
$h(y) := ( y - 1 ) \e^{2 y} + y + 1 \ge 0$, which is true because 
$h(0) = h'(0) = 0$ and $h''(y) = 4 y \e^{2 y} \ge 0$ when $y \ge 0$.
Using~\eqref{RefFor} and~\eqref{Calculus}, we get in particular that 
\begin{equation}
 \forall \tau \in \R,
 \ms{16}
      \Bigl| \frac 1 {\Gamma(1 + \ii \tau)} \Bigr|
  \le \bigl( \sqrt{1 + \tau^2} \bigr)^{-1/2} 
       \e^{\pi |\tau| / 2}.
 \label{IneGamma}
\end{equation}
More generally than in~\eqref{RefFor}, for $\sigma \in [0, 1]$, let us write
\begin{align*}
   & \ms4 \Bigl| \frac 1 {\Gamma(1 + \sigma + \ii \tau) } \Bigr|^2
 = \e^{2 \gamma \sigma} \prod_{n \ge 1}
    \Bigl( 
     \bigl[ (1 + \sigma/n)^2 + (\tau / n)^2 \ms1 \bigl] \ms2 
      \e^{-2 \sigma / n}
    \Bigr)
 \\
 = & \ms1 \e^{2 \gamma \sigma} 
    \prod_{n \ge 1} \bigl( (1 + \sigma/n)^2  \e^{- 2 \sigma / n} \bigr)
    \ns4
    \prod_{n \ge 1} 
     \Bigl( 1 + \Bigl( \frac \tau {n + \sigma} \Bigr)^2
     \Bigr)
  \ns4=\ns2 \Gamma(1 \ns1+\ns1 \sigma)^{-2}
   \ns3
      \prod_{n \ge 1} 
       \Bigl( 1 \ns1+\ns1 \Bigl( \frac \tau {n + \sigma} \Bigr)^2 \Bigr).
\end{align*}
We have by convexity of $\ln(1 + x^{-2})$ for $x > 0$ that
\begin{align*}
      \prod_{n \ge 1} 
       \Bigl( 1 + \Bigl( \frac \tau {n + \sigma} \Bigr)^2 \Bigr)
 &\le \Bigl( \prod_{n \ge 1} 
       \Bigl( 1 + \Bigl( \frac \tau n \Bigr)^2 \Bigr) \Bigr)^{1 - \sigma}
      \Bigl( \prod_{n \ge 1} 
       \Bigl( 1 + \Bigl( \frac \tau {n + 1} \Bigr)^2 \Bigr) \Bigr)^{\sigma}
 \\
 &= (1 + \tau^2)^{-\sigma} \ms2 
    \prod_{n \ge 1}
     \Bigl( 1 + \Bigl( \frac \tau n \Bigr)^2 \Bigr)   
 = (1 + \tau^2)^{-\sigma} \ms2 \frac {\sinh(\pi \tau)} {\pi \tau} \up.
\end{align*}
It follows that
\[
     \Bigl| \frac 1 {\Gamma(1 + \sigma + \ii \tau) } \Bigr|^2
 \le \Gamma(1 + \sigma)^{-2}
      (1 + \tau^2)^{-\sigma} \ms2 \frac {\sinh(\pi \tau)} {\pi \tau}
\]
and applying~\eqref{Calculus} we obtain
\begin{equation}
     \Bigl| \frac 1 {\Gamma(1 + \sigma + \ii \tau) } \Bigr|
 \le \Gamma(1 + \sigma)^{-1}
      \bigl( \sqrt { 1 + \tau^2 } \bigr)^{ 1/2 - 1 - \sigma}
       \ms2 \e^{ \pi |\tau| / 2 }.
 \label{Base}
\end{equation}
We extend this bound by using the functional equation
$z \Gamma(z) = \Gamma(z + 1)$. When $z = k + 1 + \sigma + \ii \tau$, with
$\sigma \in (0, 1)$ and $k \ge 1$ an integer, we have
\begin{align}
      \Bigl| \frac 1 { \Gamma( k + 1 + \sigma + \ii \tau ) } \Bigr|
  &=  \Bigl(
       \prod_{j = 1}^k \bigl( (j + \sigma)^2 + \tau^2 \bigr)^{-1/2}
      \Bigr)
      \Bigl| \frac 1 {\Gamma(1 + \sigma + \ii \tau)} \Bigr| 
  \notag
  \\
 &\le \bigl( \sqrt{1 + \tau^2} \bigr)^{-k}
       \Bigl| \frac 1 {\Gamma(1 + \sigma + \ii \tau)} \Bigr|.
 \label{BBB}
\end{align}
Letting $a \wedge b = \min(a, b)$ for $a, b \in \R$, we see that
\[
     \Gamma(1 + \sigma) 
  =  \int_0^{+\infty} u^\sigma \e^{- u} \, \d u
 \ge \int_0^{+\infty} (u \wedge 1) \e^{- u} \, \d u
  =  \int_0^1 \e^{- u} \, \d u
  =  1 - \e^{-1}
  >  \frac 1 2 \up.
\]

\begin{smal}
\noindent
Let us mention that the actual minimal value of $\Gamma$ on $(0, +\infty)$
is reached at
\begin{subequations}\label{MiniGamm}
\SmallDisplay{\eqref{MiniGamm}}%
\[
 \xG = 1.46163\ldots 
 \quad \hbox{and that} \quad
 \Gamma(\xG) > 0.88. 
\]
\end{subequations}
\noindent
Note that on $(0, +\infty)$, the function $x \mapsto \ln \Gamma(x)$ is
convex and $\ln \Gamma(1) = \ln \Gamma(2) = 0$, hence
$\Gamma(x) \le 1$ when $x \in [1, 2]$ and $\Gamma(x) \ge 1$ on
$(0, 1]$ and $[2, +\infty)$.

\end{smal}

\noindent
We get consequently by~\eqref{Base} and~\eqref{BBB} that
\begin{equation}
     \Bigl| \frac 1 {\Gamma( k + 1 + \sigma + \ii \tau )} \Bigr|
 \le 2 \ms2
      \bigl( \sqrt{ 1 + \tau^2 } \bigr)^{1/2 - k - 1 - \sigma }
       \ms1 \e^{ \pi |\tau| / 2 }.
 \label{Posi}
\end{equation}
When $z = - k + \sigma + \ii \tau$, with $k \ge 0$ an integer, we obtain by
the functional equation
\[
   \Bigl| \frac 1 {\Gamma(z)} \Bigr|
 = \Bigl| \frac 1 {\Gamma(1 + \sigma + \ii \tau)} \Bigr|
   \ms3
   \prod_{j = -k}^0 ( (j + \sigma)^2 + \tau^2)^{1/2}.
\]
For $j = 0, -1$, the factors in the product are $\le (1 + \tau^2)^{1/2}$,
thus
\begin{equation}
   \Bigl| \frac 1 {\Gamma(- k + \sigma + \ii \tau)} \Bigr|
 \le \bigl( 1 + \tau^2 \bigr)^{ (k+1) / 2} \ms2
       \Bigl| \frac 1 {\Gamma(1 + \sigma + \ii \tau)} \Bigr|
 \ms{12} \hbox{when} \ms{ 8}
 k = 0, 1,
 \label{GammaNeg0}
\end{equation}
and when $j \le -2$, we have 
$
     ( (j + \sigma)^2 + \tau^2)^{1/2}
 \le (|j| - \sigma) \ms2 (1 + \tau^2)^{1/2}
$.
It follows for $z = - k + \sigma + \ii \tau$, $k \ge 2$, that
\begin{equation}
   \Bigl| \frac 1 {\Gamma(z)} \Bigr|
 \le ( k - \sigma) ( k - 1 - \sigma) \ldots
      (2 - \sigma) \ms2 \bigl( 1 + \tau^2 \bigr)^{ (k+1) / 2} \ms2
       \Bigl| \frac 1 {\Gamma(1 + \sigma + \ii \tau)} \Bigr|.
 \label{GammaNeg}
\end{equation}
By the functional equation and the convexity of $\ln \Gamma$ on
$(0, +\infty)$, we have
\begin{align*}
  &\ms5    \Gamma(1 + \sigma)^{-1} 
       ( k - \sigma) ( k-1 - \sigma) \ldots (2 - \sigma)
  =  \frac { \Gamma( k + 1 - \sigma) } 
           { \Gamma( 2 - \sigma ) \ms1 \Gamma(1 + \sigma) }
  \\
 \le &\ms5 \frac { \Gamma( k + 1 - \sigma) } 
                 { \Gamma( 3 / 2 )}
  =   \frac 2 {\sqrt \pi} \ms2 \Gamma( k + 1 - \sigma )
  <  2 \ms1 \Gamma( k + 1 - \sigma).
\end{align*}
Coming back to~\eqref{GammaNeg} and using~\eqref{Base}, we conclude 
when $k \ge 2$ that
\begin{equation}
     \Bigl| \frac 1 {\Gamma(-k + \sigma + \ii \tau)} \Bigr|
 \le 2 \ms2 \Gamma( k - \sigma + 1 ) \ms2
      \bigl( \sqrt{ 1 + \tau^2} \bigr)^{1/2 + k - \sigma}
       \ms2 \e^{\pi |\tau| / 2}.
 \label{Nega}
\end{equation}
When $\Re z \ge - 1$, it follows from~\eqref{Posi} and~\eqref{GammaNeg0}
that
\[
     \Bigl| \frac 1 {\Gamma(z) } \Bigr|
 \le 2 \ms2  
      \bigl( \sqrt{1 + (\Im z)^2} \bigr)^{ 1/2 - \Re z}
       \ms2 \e^{\pi |\Im z| / 2},
\]
so, in every half-plane of the form $\Re z \ge a$, one has by~\eqref{Nega}
an upper bound
\begin{subequations}\label{MajoGammG}
\begin{equation}
     \Bigl| \frac 1 {\Gamma(z)} \Bigr|
 \le \beta_a \ms1 \bigl( \sqrt{1 + |\Im z|^2} \bigr)^{1/2 - a}
      \e^{\pi |\Im z| / 2},
 \label{MajoGamm} \tag{\ref{MajoGammG}.$\Gamma$}
\end{equation}
\end{subequations}
with $\beta_a = 2 \ms2 \Gamma(|a| + 1)$ when $a \le - 1$, and $\beta_a = 2$
otherwise.

\begin{rem}
The rather crude estimate~\eqref{MajoGamm} is sufficient for our purposes.
In~\cite{SteinTHA}, Stein refers to Titchmarsh~\cite[p.~259]{Titchmarsh},
for an exact asymptotic estimate. When $\sigma$ is fixed and
$|\tau| \rightarrow +\infty$, one has
\[
      |\Gamma(\sigma + \ii \tau)|
 \simeq \sqrt{2 \pi} \e^{ - \pi |\tau| / 2} |\tau|^{\sigma - 1/2}.
\]
When $\sigma \ge 1$, the preceding proof gives a \emph{lower bound\/}
$2^{-1} \sqrt{2 \pi} \e^{ - \pi |\tau| / 2} |\tau|^{\sigma - 1/2}$ for
every $\tau$. We can see it by replacing the inequality~\eqref{Calculus}
with the evident inequality
$\sinh(\pi x) / (\pi x) \le (2 \pi |x|)^{-1} \e^{\pi |x|}$.  It is not
possible to replace $\sqrt{1 + |\Im z|^2}$ by $|\Im z|$ in~\eqref{MajoGamm}
when $\Re z \le - 1$, because the zeroes $-1, -2, \ldots$ of $1 / \Gamma$
are simple. For more results on the Gamma function, we refer to
Andrews--Askey--Roy~\cite{A-Askey-R}.
\end{rem}

\subsection{The interpolation scheme%
\label{InterpoSchem}}

\noindent
We begin with the classical \emph{three lines lemma}, an easier version of
which is the \emph{Hadamard three-circle theorem}. After this, we shall
turn to interpolation of holomorphic families of linear operators.

\subsubsection{The three lines lemma\label{TheTLL}}

\begin{lem}%
\label{TLL}
Let $S$ denote the open strip\/ $\{z : 0 < \Re z < 1\}$ in the complex
plane. Let $f$ be a function holomorphic in $S$ and continuous on the
closure of~$S$. Assume that $f$ is bounded in $S$ and that
\[
 |f(0 + \ii \tau)| \le C_0,
 \ms{12}
 |f(1 + \ii \tau)| \le C_1
\]
for all $\tau \in \R$. Then, for every $\theta \in (0, 1)$, one has that\/
$
 |f(\theta)| \le C_0^{1 - \theta} C_1^\theta
$.
\end{lem}

\begin{remN}%
\label{PhragLin}
Of course $f(\theta + \ii \tau)$ admits the same bound for every
$\tau \in \R$, by translating $f$ vertically. The somewhat strange
assumption that $f$ must be \emph{bounded} on the whole strip by a value
which does not appear in the final result is not the finest assumption that
makes the conclusion valid, see a better criterion below. However, when
Lemma~\ref{TLL} applies, the function $f$ \emph{is bounded} at last. It is
well known that some restriction on the size of $f$ inside the strip is
needed for the lemma to hold true. Indeed, in the strip 
$S_\pi = \{ z : |\Re z| \le \pi / 2\}$, the function $f(z) = \e^{\cos z}$
has modulus one on the two lines $\Re z = \pm \pi / 2$, but it is \og very
big\fge when $\Re z = 0$, since $|f(\ii \tau)| = \e^{\cosh(\tau)}$. For a
function $f$ holomorphic in an open vertical strip $S$, continuous on the
closure and bounded by $1$ on the two boundary lines, either~$f$ is bounded
by $1$ on~$S$, or else $\sup_{ |\Im z| = |\tau| } |f(z)|$ must
become extremely large when $|\tau|$ tends to infinity. This is the typical
situation with the theorems of Phragm\'en--Lindel\"of type, 
see~\cite[Chap~12, 12.7]{RudinRCA} for example.
\end{remN}

\begin{smal}
\noindent
Here is a sufficient criterion ensuring that $|f|$ is bounded by its
supremum on the boundary $\partial S_w$ of a vertical strip $S_w$ of
width $w$. If $f$ is holomorphic on $S_w$, continuous on the closure with
$|f| \le 1$ on $\partial S_w$, and if for some $a < \pi / w$ one has
\[
 |f(z)| = O \bigl( \exp( \e^{a \ms1 |\Im z|} ) \bigr)
\]
when $z$ tends to infinity in~$S_w$, then $|f|$ is bounded by $1$ on the
strip. Let us prove it assuming
$\ln |f(z)| \le \kappa \e^{a \ms1 |\Im z|}$ in 
$S_\pi = \{ |\Re z| \le \pi / 2 \}$, for an $a < 1 = \pi / w$. Set 
$g_\varepsilon(z) = \e^{- \varepsilon \cos(b z)}$, with
$\varepsilon > 0$ and $a < b < 1$. If $z = \sigma + \ii \tau$ and 
$|\sigma| \le \pi / 2$, we have 
\begin{align*}
     |g_\varepsilon(z)| 
  &=  \exp \bigl( -\varepsilon \Re \cos(b z) \bigr)
   =  \exp \bigl( -\varepsilon \cos(b \sigma) \cosh(b \tau) \bigr)
 \\
 &\le \exp \bigl( -\varepsilon \cos(b \pi / 2) \cosh(b \tau) \bigr)
 \le \exp \bigl( -B_\varepsilon \e^{b |\tau|} \bigr)
 \le 1,
 \quad B_\varepsilon > 0,
\end{align*}
hence $|f(z) \ms1 g_\varepsilon(z)| \le 1$ on $\partial S_\pi$, and
if $|\tau| = |\Im z| \ge (b - a)^{-1} \ln (\kappa / B_\varepsilon)$ we get
\begin{subequations}\label{Moderation}
\SmallDisplay{\eqref{Moderation}}%
\[
     \ln |f(z) \ms1 g_\varepsilon(z)| 
 \le \kappa \e^{a |\tau|} - B_\varepsilon \e^{b |\tau|}
 \le 0.
\]
\end{subequations}
\noindent
Given any $z_0 \in S_\pi$, we can find a rectangle 
$R_\varepsilon
 = \{ |\Re z| \le \pi / 2, \ms4 |\Im z| \le \tau_0(\varepsilon) \}$
containing $z_0$ such that $|f(z) \ms1 g_\varepsilon(z)| \le 1$
on~$\partial R_\varepsilon$. We then have
$|f(z_0) \ms1 g_\varepsilon(z_0)| \le 1$ by the maximum principle,
$
 |f(z_0)| \le | \e^{\varepsilon \cos(b z_0)} |
$
for every $\varepsilon > 0$, thus $|f(z_0)| \le 1$.

\end{smal}%

\noindent
Several times later on, we encounter situations where the function $f$ is
not bounded on the two lines limiting a vertical strip $S$, but has instead
a growth exponential in~$|\tau| = |\Im z|$. The next lemma generalizes the
preceding. Our proof and estimate are not the \og correct\fge ones, as we
shall explain below after Corollary~\ref{InStrip}, but they give a
reasonable explicit bound. In these Notes, we shall say that a function~$f$
defined on a vertical strip $S$ has
an \emph{admissible growth}\label{AdmiGro} 
in the strip if for some $\kappa > 0$, the function $f$ admits in~$S$ a
bound of the form $|f(z)| \le \kappa \e^{ \kappa |\Im z|}$.

\begin{lem}%
\label{TLLb}
Let $f$ be a function holomorphic in the strip $S = \{z : 0 < \Re z < 1\}$,
with admissible growth in $S$ and continuous on the closure of $S$. Assume
that there exist real numbers $a_0, a_1 \ge 0$ and $b_0, b_1$ such that for
every $\tau \in \R$, one has
\[
 |f(0 + \ii \tau)| \le \e^{a_0 |\tau| + b_0},
 \ms{12}
 |f(1 + \ii \tau)| \le \e^{a_1 |\tau| + b_1}.
\]
For every $\theta \in (0, 1)$, it follows that
\[
     |f(\theta)|
 \le \exp \Bigl( \sqrt{ \theta(1 - \theta) \vphantom{a_0^2} } \ms3
                  \sqrt{ (1 - \theta) a_0^2 + \theta a_1^2}
                 + (1 - \theta) b_0 + \theta b_1
          \Bigr).
\]
\end{lem}

\begin{proof}
We introduce the holomorphic function
$
 g(z) := \e^{ c z^2 / 2 + d z }
$,
with $c > 0$ and $d$ real. If $z = \sigma + \ii \tau$, we see that
$
 |g(z)| = \e^{ c (\sigma^2 - \tau^2) / 2 + d \ms1 \sigma}
$.
On the vertical side $\Re z = 0$ of~$S$, we have that
\[
     |f(\ii \tau) g(\ii \tau)| 
 \le \e^{ a_0 |\tau| + b_0 - c \tau^2 / 2 }
 \le \e^{ a_0^2 / (2 c) + b_0}
  =: E_0
\]
and when $\Re z = 1$, we get the upper bound
\[
     |f(1 + \ii \tau) g(1 + \ii \tau)| 
 \le \e^{ a_1 |\tau| + b_1  - c \tau^2/2 + c/2 + d }
 \le \e^{ a_1^2 / (2 c) + b_1 + c/2 + d}
  =: E_1.
\]
We choose $d$ so that $E_0 = E_1$, and we need not mention the value
of~$d$.
\dumou

 It follows from the assumption $|f(z)| \le \kappa \e^{\kappa |\Im z|}
 = \kappa \ms1 \e^{\kappa |\tau|}$ that $f(z) g(z)$ tends to zero at
infinity in $S$. Let us fix $\theta \in (0, 1)$. If $f(\theta) \ne 0$,
there exists $\tau_0 > 0$ such that 
$|f(z) g(z)| < |f(\theta) g(\theta)|$ when $|\Im z| \ge \tau_0$. By the
maximum principle for the compact rectangle 
$R = \{ 0 \le \Re z \le 1, \ms4 |\Im z| \le \tau_0 \}$, we know that the
maximum of $|f(z) g(z)|$ is reached at the boundary of $R$, but it cannot
be on the horizontal sides $|\Im z| = \tau_0$. Hence
$|f(\theta) g(\theta)| \le E_0 = E_1 = E_0^{1-\theta} E_1^\theta$,
we get therefore
\[
     |f(\theta)| 
 \le \e^{ -c \theta^2 / 2 - d \theta} E_0^{1-\theta} E_1^\theta
   = \exp \Bigl( \frac {(1 - \theta) a_0^2 + \theta a_1^2} { 2 c}
             + (1-\theta) b_0 + \theta b_1 + c \theta (1 - \theta) / 2
           \Bigr)
\]
and after optimizing in $c > 0$, we conclude that
\[
     |f(\theta)| 
 \le \exp \Bigl( 
            \sqrt{(1 - \theta) a_0^2 + \theta a_1^2} \ms2
             \sqrt{ \theta (1 - \theta) \vphantom{a_0^2} } 
              + (1 - \theta) b_0 + \theta b_1 
           \Bigr).
 \qedhere
\]
\end{proof}

\begin{cor}%
\label{InStrip}
Let $f$ be a function holomorphic in 
$S = \{z : \alpha_0 < \Re z < \alpha_1\}$, with admissible growth in the
strip $S$ and continuous on the closure of~$S$. Assume  that there exist
real numbers $u_0, u_1 \ge 0$ and $v_0, v_1$ such that 
\[
 |f(\alpha_0 + \ii \tau)| \le \e^{u_0 |\tau| + v_0},
 \ms{12}
 |f(\alpha_1 + \ii \tau)| \le \e^{u_1 |\tau| + v_1}
\]
for every $\tau \in \R$. Let $\theta \in [0, 1]$, set
$\alpha_\theta = (1 - \theta) \alpha_0 + \theta \alpha_1$,
$u_\theta = (1 - \theta) u_0 + \theta u_1$ and
$v_\theta = (1 - \theta) v_0 + \theta v_1$. For every $\tau \in \R$, one has
\[
     |f(\alpha_\theta + \ii \tau)|
 \le E_{w, \theta} (u_0, u_1) \ms3
     \e^{ u_\theta \ms1 |\tau| + v_\theta},
\]
where $w = \alpha_1 - \alpha_0$ denotes the width of the strip $S$ and
where
\[
    E_{w, \theta} (u_0, u_1)
 := \exp \Bigl( w \ms2
          \sqrt{ \theta (1 \ns3-\ns3 \theta) \vphantom{u_0^2} } 
           \ms2
          \sqrt{(1 \ns3-\ns3 \theta) u_0^2 + \theta u_1^2} \ms1
         \Bigr).
\] 
\end{cor}

 Notice that $\sqrt { \theta (1 - \theta) } \le 1/2$ for every
$\theta \in [0, 1]$. When $0 \le u_0, u_1 \le u$, one can always employ the
simpler bound
$
     E_{w, \theta} (u, u)
 \le \e^{w u / 2}
$.

\begin{proof}
We begin with $S_1 := \{ 0 < \Re z < 1 \}$. We bound the modulus of
$f(\theta + \ii \tau_0)$ for $\tau_0$ in $\R$ by performing a
vertical translation of~$f$, then invoking Lemma~\ref{TLLb}. The function
$F(z) = f(z + \ii \tau_0)$ satisfies
$
     |F(j + \ii \tau)| 
 \le \e^{u_j |\tau| + ( u_j |\tau_0| + v_j )}$,
$j = 0, 1$,
and the bound for $|F(\theta)|$ given at Lemma~\ref{TLLb} implies that
\begin{equation}
     |f(\theta + \ii \tau_0)|
 \le E_{1, \theta} (u_0, u_1) \ms3
     \e^{ u_\theta \ms1 |\tau_0| + v_\theta},
 \quad
 \tau_0 \in \R.
 \label{FTheta}
\end{equation}
It is easy to pass to $S =  \{ \alpha_0 < \Re z < \alpha_1 \}$ with
the transform that replaces $f(z)$, defined for $z \in S$, by 
$F(Z) = f(\alpha_0 + Z w)$ for $Z \in S_1$. If 
$|f(\alpha_j + \ii \tau)| \le \e^{u_j |\tau| + v_j}$, $j = 0, 1$, then
$|F(j + \ii \tau)| \le \e^{w u_j |\tau| + v_j}$ and by~\eqref{FTheta} we
have that
\begin{align*}
 &    \ms{24} 
       |f(\alpha_\theta + \ii \tau_0)|
   =  |F(\theta + \ii \tau_0 / w)|
 \\
 &\le E_{1, \theta} (w u_0, w u_1) \ms3
       \e^{ (w u_\theta) \ms1 |\tau_0 / w| + v_\theta}
   =  E_{w, \theta} (u_0, u_1) \ms3
      \e^{ u_\theta \ms1 |\tau_0| + v_\theta}.
 \qedhere
\end{align*}
\end{proof}

 Applying Corollary~\ref{InStrip} in the case where 
$u_0 = u_1 = u \ge 0$ and $v_j = 0$, one sees that when $f$ has an
admissible growth in $S$, the hypothesis 
$|f(\alpha_j + \ii \tau)| \le \e^{ u |\tau|}$ for all $\tau \in \R$ and 
$j = 0, 1$, implies 
$|f(\alpha + \ii \tau)| \le \e^{w \ms1 u/2} \ms1 \e^{ u |\tau|}$ in the
strip. It is not possible to replace the \og bounding factor\fge 
$\e^{w \ms1 u/2}$ by~$1$, as we shall understand below.

\begin{smal}
\noindent
The \og correct\fge proof of Lemma~\ref{TLLb} uses a lemma given by
Hirschman~\cite[Lemma~1]{Hir}, cited by Stein~\cite{SteinILO}. In our case,
we consider the function $U$, harmonic in the open strip 
$S_1 = \{ 0 < \Re z < 1 \}$ and continuous on the closed strip, equal
to $a_j |\tau| + b_j$ at each boundary point $j + \ii \tau$, with 
$a_j \ge 0$, $\tau \in \R$ and $j = 0, 1$. Let $V$ be the \emph{harmonic
conjugate}\label{HarmonConj} 
of $U$ in $S_1$, defined up to an additive constant by the fact that
$\nabla V(z)$, for $z \in S_1$, is equal to $R \ms1 \nabla U(z)$ where $R$
is the rotation of angle $+\pi / 2$ in $\R^2 \simeq \C$. Let us set 
$V(1/2) = 0$ in order to fix $V$ entirely. Since $U$ is harmonic, the
$1$-form $-U_y \ms2 dx + U_x \ms2 dy$ is closed and
$
   V(z)
 = \int_0^1 R \ms1 \nabla U(\gamma(s)) \ps \gamma'(s) \, \d s
$
for any $C^1$ path $\gamma$ in $S_1$ such that $\gamma(0) = 1/2$ and
$\gamma(1) = z$. Then $U + \ii V$ is holomorphic, by the Cauchy--Riemann
equations. Consider the holomorphic \emph{outer function}
\[
 g(z) = \exp \bigl( -U(z) - \ii V(z) \bigr),
 \quad
 z \in S_1,
\]
for which $|g(z)| = \exp (-U(z))$
and $|g(z)| \le \e^{ - (b_0 \wedge b_1)}$ in~$S_1$. If $f$ is as in
Lemma~\ref{TLLb}, then $|f g| \le 1$ at the boundary of $S_1$ and $f g$ has
an admissible growth. It follows from an easy variation of Lemma~\ref{TLL}
that
$
 |(f g)(\theta)| \le 1
$ 
thus
$
 |f(\theta)| \le \e^{U(\theta)}
$, and it remains to express $U(\theta)$, with the help of the harmonic
measure at $\theta$ for~$S_1$. 
\end{smal}

\noindent
We shall obtain the harmonic measures for
$S_\pi = \{ z \in \C : |\Re z | < \pi / 2 \}$ from the case of the open
unit disk~$D$, by a conformal mapping (see also~\cite[proof of
Lemma~1.3.8]{GrafakosCFA}). Let $\sigma$ belong to 
$I_\pi = (- \pi / 2, \pi / 2) = S_\pi \ms1 \cap \ms2 \R$. The Poisson
probability measure $\mu_\sigma$ at~$\sigma$ relative to~$S_\pi$ can be
written as $\mu_\sigma = \mu_{\sigma, 0} + \mu_{\sigma, 1}$,
where $\mu_{\sigma, 0}$ is supported on $B_0 = -\pi/2 + \ii \R$ and 
$\mu_{\sigma, 1}$ on $B_1 = \pi/2 + \ii \R$. If $h$ is real, harmonic
in~$S_\pi$, bounded and continuous on the closure of $S_\pi$, the value of
$h$ at $\sigma$ is equal to
\begin{subequations}\label{Expression}
\SmallDisplay{\eqref{Expression}}%
\[
   h(\sigma)
 = \int_{\partial S_\pi} h \, \d \mu_\sigma
 = \int_{B_0} h \, \d \mu_{\sigma, 0}
    + \int_{B_1} h \, \d \mu_{\sigma, 1}.
\] 
\end{subequations}
\noindent
The Poisson probability measure~$\nu_r$ for~$D$ at $r \in (-1, 1)$ 
has density  $g_r(\e^{\ii \beta})
 = ( 1 - r^2 ) / ( 1 - 2 r \cos \beta + r^2 )$
with respect to the invariant probability measure on the unit circle~$\T$.
Let $\Phi$ be the holomorphic bijection from $S_\pi$ onto $D$ given by
$\Phi(z) = \tan(z / 2)$ when $z \in S_\pi$, extended to
$|\ns1 \Re z| = \pi / 2$ by the same formula. Then $\partial S_\pi$ is sent
to~$\T \setminus \{\ii, - \ii \}$ and if 
$\Phi(\pi / 2 + \ii \tau) = \e^{\ii \beta}$, we have 
$\beta \in (-\pi/2, \pi/2)$ and $\tanh(\tau / 2) = \tan(\beta / 2)$. For 
$r = \tan(\sigma / 2)$ we see that $\nu_r = \Phi_{\#} \ms2 \mu_\sigma$ and
\[
   \int_{B_1} h \, \d \mu_{\sigma, 1}
 = \int_\R h( \pi / 2 + \ii \tau) \ms1 
            f_\sigma (\tau) \, \d \tau
 \ms{14} \hbox{with} \ms6
   f_{\sigma}(\tau)
 = \frac { \cos \sigma }
         { 2 \pi (\cosh \tau - \sin \sigma) } \up,
\]
\noindent
while $\int_{B_0} h \, \d \mu_{\sigma, 0}
 = \int_\R h( - \pi / 2 + \ii \tau) \ms1 
            f_{-\sigma} (\tau) \, \d \tau$.
One finds $\|\mu_{\sigma, 1}\|_1 = \theta  := \sigma / \pi + 1 / 2$ and
$\|\mu_{\sigma, 0}\|_1 = 1 - \theta$ by harmonicity of $h(z) = \Re z$. 
When $\sigma$ tends to~$\pi / 2$, the density $f_\sigma$ resembles the
Cauchy kernel $P_\varepsilon^{(1)}$ in~\eqref{CauchyK}
with $\varepsilon = \pi / 2 - \sigma$, since
\[
      f_\sigma(\tau)
  =   \frac 1 {2 \pi} 
       \frac { \sin \varepsilon }
             { \cosh \tau - \cos \varepsilon}
 \simeq \frac \varepsilon
              { \pi (\tau^2 + \varepsilon^2) } \up.
\]
One can also comprehend $f_\sigma$ as sum of the alternate series of Cauchy
kernels
\[
   f_\sigma
 = P_{\pi / 2 {-} \sigma}^{(1 \ns1)} 
   {-} P_{\pi {+} \pi / 2 {+} \sigma}^{(1 \ns1)}
   {+} P_{2 \pi {+} \pi / 2 {-} \sigma}^{(1 \ns1)} 
   {-} P_{2 \pi {+} \pi {+} \pi / 2 {+} \sigma}^{(1 \ns1)}
   {+} P_{4 \pi {+} \pi / 2 {-} \sigma}^{(1 \ns1)} 
   {-} P_{4 \pi {+} \pi {+} \pi / 2 {+} \sigma}^{(1 \ns1)}
   {+} \cdots \ms2,
\]
indeed, if $\varphi_\sigma$ denotes the sum of the series above and if
$g$ belongs to $\ca K(\R)$, then 
$G(\sigma + \ii \tau) = (\varphi_\sigma * g)(\tau)$ is harmonic in $S_\pi$,
tends to $g(\tau)$ when $\sigma \rightarrow \pi / 2$ and to $0$ when
$\sigma \rightarrow - \pi / 2$, the same properties as for 
$(f_\sigma * g)(\tau)$.
\dumou

 Let $h_*$ be a continuous function on $\partial S_\pi$, and suppose that
the two functions $t \mapsto \e^{-|t|} h_*(\pm \pi/2 {+} \ii t)$ are
Lebes\-gue-integrable on the real line. Then, writing
\begin{equation}
   \widetilde h_*(z)
 = \int_\R 
    \bigl( 
      h_*( \pi/2 {+} \ii (\tau {-} t)) \ms1 f_\sigma(t)
      + h_*( {-} \pi/2 {+} \ii (\tau {-} t)) \ms1 f_{- \sigma} (t)
    \bigr) \, \d t
 \label{Extends}
\end{equation}
\noindent 
for every $z = \sigma + \ii \tau \in S_\pi$,
one defines an harmonic function $\widetilde h_*$ in $S_\pi$, continuous
on the closure if one sets $\widetilde h_*(z_*) = h_*(z_*)$ for
$z_* \in \partial S_\pi$. Let $\ca H_c(S_\pi)$ denote the
class of functions harmonic in~$S_\pi$ and continuous on the closure. Not
every $h \in \ca H_c(S_\pi)$ can be expressed by~\eqref{Extends} from its
restriction $h_* = h\barre_{\partial S_\pi}$. First, $h_*$ must be
$\mu_\sigma$-integrable, but even then, 
$h(z) = \Re \cos(z) = \cos(\sigma) \cosh(\tau)$, for which $h_* = 0$, is a
counterexample.
\dumou

 Let us say here that $g$ defined on $S_\pi$, resp. $\partial S_\pi$, is
\emph{moderate} if there is $a < 1$ such that $g(z) = O(\e^{a |\Im z|})$
for $z \in S_\pi$, resp. $\partial S_\pi$. If $h_*$ is moderate and
continuous on~$\partial S_\pi$, the extension $\widetilde h_*$
in~\eqref{Extends} is in~$\ca H_c(S_\pi)$, and it is moderate because
\[
     |\widetilde h_*(\sigma + \ii \tau)|
 \le \kappa 
      \int_\R \e^{ a |\tau - t| } (f_\sigma + f_{-\sigma})(t)
               \, \d t
 \le \kappa 
       \Bigl( \e^a +
        \int_1^{+\infty} \ns5
         \frac { \e^{ a |t| } } {\cosh t - 1 } \,  \d t
       \Bigr) \ms2 \e^{ a |\tau| }.
\]

\begin{lem}[after~\cite{Hir}]\label{Hirsch}
If $h \in \ca H_c(S_\pi)$ is moderate and 
$h_* = h\barre_{\partial S_\pi}$, then $h = \widetilde h_*$.
\end{lem}

 If one replaces $S_\pi$ by a strip $S_w$ of width $w$, then clearly the
moderation condition in $S_w$ must be formulated for $z \in S_w$ as 
$g(z) = O(\e^{a |\Im z|})$ with $a < \pi / w$.

\begin{proof} 
We have that $h_*$ is moderate on $\partial S_\pi$, hence
$U = h - \widetilde h_*$ is moderate on~$S_\pi$ and vanishes 
on~$\partial S_\pi$. Given $z_0 \in S_\pi$, $a < 1$ such that
$U = O(\e^{a |\tau|})$, $\varepsilon > 0$ and $b \in (a, 1)$, we see as
in~\eqref{Moderation} that $U - \varepsilon \Re \cos(b z)$ is $\le 0$ on
the boundary of a rectangle containing $z_0$, hence 
$U(z_0) \le \varepsilon \Re \cos(z_0)$ by the maximum principle. Doing it
also with $-U$ and letting $\varepsilon \rightarrow 0$ we conclude that 
$h - \widetilde h_* = 0$.
\end{proof}

\begin{smal}
\noindent
We now study the function $h_1$ defined by 
$h_1(\pi/2 + \ii \tau) = |\tau|$, $h_1( - \pi/2 + \ii \tau) = 0$ for every
$\tau \in \R$ and its (moderate) harmonic extension given at 
$\sigma \in I_\pi$ by
\[
   h_1(\sigma)
 = \int_\R |\tau| f_\sigma(\tau) \, \d \tau
 = \frac 2 \pi \int_0^{+\infty} \ns3
    \arctan \Bigl( 
             \frac { \cos \sigma } 
                   { \e^\tau - \sin \sigma }
            \Bigr) \, \d \tau.
\]
Recall that $\|f_\sigma\|_{L^1(\R)} = \theta = \sigma / \pi + 1 / 2$.
When $\sigma = 0$, we have the easy bound
\[
   h_1(0)
 = \frac 2 \pi \int_0^{+\infty} \arctan(\e^{-\tau}) \, \d \tau
 < \frac 2 \pi \int_0^{+\infty} \e^{-\tau} \, \d \tau
 = \frac 2 \pi \up.
\]
One can find $h_1(0)$ by writing the power series expansion of $\arctan(x)$,
letting then $x = \e^{-\tau}$ and integrating in $\tau \in (0, +\infty)$.
One gets $h_1(0) = 2 \ms1 G / \pi < 0.584$, where
$G = \sum_{k=0}^{+\infty} (-1)^k (2 k + 1)^{-2}$ is the 
\emph{Catalan constant}, $0.915 < G < 0.916$. 
One has
\[
    h'_1(\sigma)
 = \frac 2 \pi \int_0^{+\infty} 
    \frac { \e^{-2 \tau} - \e^{-\tau} \sin \sigma }
          {\e^{-2 \tau} - 2 \e^{-\tau} \sin \sigma + 1 } \, \d \tau
 = \frac 1 \pi \ln \bigl( 2 - 2 \sin \sigma \bigr),
\]
thus $h_1$ is concave on $I_\pi$ and maximal when $\sigma = \pi / 6$. 
One can find numerically that $0.646 < h_1(\pi/6) < 0.647$. 
By concavity, we obtain for each $\sigma \in I_\pi$ that
\begin{subequations}\label{h1Concave}
\SmallDisplay{\eqref{h1Concave}}%
\[
     h_1(\sigma)
  =  h_1(\sigma) - h_1(-\pi/2)
 \le h'_1(-\pi/2) (\sigma + \pi / 2)
  =  \theta \ms1 \ln 4.
\]
\end{subequations}
\noindent
One has $h_1(\pi / 2) = 0$, the behavior of~$h_1(\sigma)$ when 
$\varepsilon = \pi / 2 - \sigma \rightarrow 0$ is given by
\begin{subequations}\label{BoundaryBehav}
\SmallDisplay{\eqref{BoundaryBehav}}%
\[
             h_1(\sigma)
  \ns2=\ns2  - \frac 1 \pi \int_0^\varepsilon \ln(2 - 2 \cos s) \, \d s
 \simeq      - \frac 1 \pi \int_0^\varepsilon \ln(s^2) \, \d s
  \ns2=\ns2  \frac 2 \pi 
              \bigl(
               \varepsilon \ln(1 / \varepsilon) + \varepsilon 
              \bigr).
\]
\end{subequations}
\noindent
Since $h_1( \cdot + \ii \tau) - h_1( \cdot )$ is bounded by $|\tau|$ on
$B_1$ and vanishes on $B_0$, we have
\begin{subequations}\label{h1bound}
\SmallDisplay{\eqref{h1bound}}%
\[
     0
  <  h_1(\sigma + \ii \tau) 
 \le h_1(\sigma) + \theta \ms1 |\tau|
 \le h_1(\pi / 6) + \theta \ms1 |\tau|,
 \quad
 \sigma \in I_\pi, \ms7 \tau \in \R.
\]
\end{subequations}
\noindent
If $S_w = \{ z \in \C : \alpha_0 < \Re z < \alpha_1 \}$ has width 
$w = \alpha_1 - \alpha_0$ and if $\lambda = w / \pi$, we may associate to
$h$, harmonic on $S_w$, the harmonic function 
$H(Z) = h(\alpha_{1/2} + \lambda Z)$ for $Z \in S_\pi$, where
$\alpha_t = (1 - t) \alpha_0 + t \alpha_1$ when $t \in [0, 1]$. If we set 
$h_{1, w}(\alpha_1 + \ii \tau) = |\tau|$ and $h_{1, w} = 0$ on 
$\alpha_0 + \ii \R$, then $H_{1, w} = \lambda h_1$, and
we get from~\eqref{h1bound} that
\[
     h_{1, w}(\alpha_\theta + \ii \tau) 
  =  h_{1, w}(\alpha_{1/2} + \lambda \sigma + \ii \tau) 
  =  \lambda h_1 (\sigma + \ii \tau / \lambda) 
 \le w h_1(\pi / 6) / \pi + \theta \ms1 |\tau|.
\]
\dumou

 We now comment on Corollary~\ref{InStrip}. If~$f$ is holomorphic in $S_w$
with admissible growth, satisfies 
$|f(\alpha_j + \ii \tau)| \le \e^{u_j |\tau|}$ on $\partial S_w$, 
$u_j \ge 0$, $j = 0, 1$, the \og correct\fge bound at $z \in S_w$ for~$f$
is $\e^{ U_{\gr u, w}(z) }$ where
$U_{\gr u, w} = u_0 h_{0, w} + u_1 h_{1, w}$, with 
$h_{0, w}(\alpha_{1/2} + \zeta) = h_{1, w}(\alpha_{1/2} - \zeta)$. One gets
in particular $U_{\gr u, w} (\alpha_{1/2})
 = 2 \ms1 \lambda (u_0 + u_1) \ms1 G / \pi$. When $u_0 = u_1 = 1$, this
finer method gives at $\alpha_{1/2}$ a bounding factor 
$\e^{(4 G / \pi) (w / \pi)}$ instead of $E_{w, 1/2} (1, 1) = \e^{w / 2}$,
and
$
   4 G / \pi^2
 < 0.3713
 < 1/2
$.

 Let $V_{\gr u, w}$ be the harmonic conjugate of $U_{\gr u, w}$.
Our first method in Corollary~\ref{InStrip} applied to
$f_0(z) = \e^{U_{\gr u, w} (z) + \ii V_{\gr u, w} (z)}$ yields
\begin{subequations}\label{BoundingU}
\SmallDisplay{\eqref{BoundingU}}%
\[
     U_{\gr u, w}(\alpha_\theta + \ii \tau)
 \le \ln E_{w, \theta} (u_0, u_1)
      + u_\theta \ms1 |\tau|.
\]
\end{subequations}
\noindent
If $u_0 = u_1 = u > 0$, we get 
$    U_{\gr u, w} (\alpha_\theta) 
  =  u (h_{0, w} + h_{1, w})(\alpha_\theta)
 \le w \sqrt{ \theta (1 - \theta) } \ms2 u$.
This estimate~\eqref{BoundingU} has the right order of magnitude in $w$
and $u$, but not in $\theta$ when $\theta$ tends to~$0$ or~$1$. The correct
order when $\theta \rightarrow 0$ is $\kappa \ms2 \theta \log (1/\theta)$,
according to~\eqref{BoundaryBehav}.

\end{smal}

\begin{remN}%
\label{PolyFacto}
We shall have to deal with cases where the bounds on the lines limiting
the strip $S_w = \{z \in \C : \alpha_0 < \Re z < \alpha_1\}$,
$w = \alpha_1 - \alpha_0$, have the form
\[
     |f(\alpha_j + \ii \tau)| 
 \le (1 + \tau^2)^{c_j} \ms2 \e^{u_j |\tau| + v_j},
 \ms{14}
 c_j, u_j \ge 0,
 \ms9
 j = 0, 1.
\]
It is obviously possible to \og absorb\fge the polynomial
factor by replacing $u_j$ in the exponential with $u_j + \varepsilon$,
$\varepsilon > 0$ arbitrary, and modifying $v_j$ accordingly, but one can
work a little more carefully as follows.
\dumou

\begin{smal}
\noindent
Let $\ell_{1, w}$ be the moderate harmonic function on $S_w$ such that
$\ell_{1, w}(\alpha_1 + \ii \tau) =  \ln(1 + \tau^2)$ for $\tau \in \R$ and
$\ell_{1, w} = 0$ on $\alpha_0 + \ii \R$. Let
$\alpha_\theta = (1 - \theta) \alpha_0 + \theta \alpha_1$, 
$\lambda = w / \pi$, $\sigma = \pi \theta - \pi / 2$ and
$L_{1, w}(Z) = \ell_{1, w}(\alpha_{1/2} {+} \lambda Z)$. By
Lemma~\ref{Hirsch} and~\eqref{Extends}, we get
\def\eqphantom{\mathrel{\hphantom =}}
\begin{align*}
 &   \ms6 \eqphantom
     \ell_{1, w}(\alpha_\theta + \ii \tau)
  =  \ell_{1, w}(\alpha_{1/2} + \lambda \sigma + \ii \tau)
  =  L_{1, w}(\sigma + \ii \tau / \lambda)
 \\
 &=  \int_\R L_{1, w}(\pi / 2 \ms1 {+} \ii \tau / \lambda {-} \ii t) 
      f_\sigma(t) \, \d t
 \le \int_\R \ln \bigl( 1 {+} ( \lambda |t| {+} |\tau| )^2 \bigr) 
       f_\sigma(t) \, \d t.
\end{align*}
Applying Jensen's inequality to the probability density 
$\widetilde f_\sigma = \theta^{-1} f_\sigma$, one sees that 
\begin{align*}
 &   \ms6 \eqphantom
      \exp \Bigl( 
            \int_\R 
             \ln \bigl(
                  [ 1 + ( \lambda |t| + |\tau| )^2 ]^{1/2}
                 \bigr)
              \widetilde f_\sigma(t) \, \d t
           \Bigr)
 \\
 &\le \int_\R [ 1 {+} ( \lambda |t| {+} |\tau| )^2 ]^{1/2}
       \widetilde f_\sigma(t) \, \d t
  \le (1 + \tau^2)^{1/2} + \lambda \int_\R |t| \ms1
       \widetilde f_\sigma(t) \, \d t,
\end{align*}
bounded by $(1 + \tau^2)^{1/2} + \lambda \ms1 \ln 4$ by~\eqref{h1Concave}.
For every $\tau \in \R$, one has therefore
\begin{subequations}\label{BoundingL}
\SmallDisplay{\eqref{BoundingL}}%
\[
   0
 < \ell_{1, w}(\alpha_\theta + \ii \tau)
 < 2 \theta \ms2 
    \ln \bigl( (1 + \tau^2)^{1/2} + \lambda \ms1 \ln 4 \bigr).
\]
\end{subequations}
\noindent
Define an harmonic function $U$ in $S_w$, continuous on the closure, by
$U = c_0 \ms1 \ell_{0, w} + c_1 \ms1 \ell_{1, w}$, where
$\ell_{0, w}(z) = \ell_{1, w}(2 \alpha_{1/2} - z)$, so that 
$U(\alpha_j + \ii \tau) = c_j \ln ( 1 + \tau^2)$. Let $V$ be conjugate to
$U$ in $S_w$. Then $g = \e^{-U - \ii V}$ is holomorphic in $S_w$ and 
$|(f g)(\alpha_j + \ii \tau)| \le \e^{u_j |\tau| + v_j}$ on
$\partial S_w$. By~\eqref{BoundingL}, we can bound $|f(z)|$ at 
$z = \alpha_\theta + \ii \tau$ by multiplying the inside bound of
Corollary~\ref{InStrip} for $f g$ with the additional factor
\[
     \e^{U(\alpha_\theta + \ii \tau)}
 \le \bigl( (1 + \tau^2)^{1/2} + \ln(4) \ms1 w / \pi \bigr)^{2 c_\theta}
 \le ( 1 + \ln(4) \ms1 w / \pi )^{2 c_\theta} \ms1 
      ( 1 + \tau^2 )^{c_\theta},
\]
where $c_\theta = (1 - \theta) c_0 + \theta \ms1 c_1$. 
Since $\ln(4) / \pi < 1/2$, we may remember that
\begin{subequations}\label{PolyBd}
\SmallDisplay{\eqref{PolyBd}}%
\[
     |f(\alpha_\theta + \ii \tau)|
 \le ( 1 + w / 2 )^{2 c_\theta} \ms2 E_{w, \theta} (u_0, u_1) \ms3
     (1 + \tau^2)^{c_\theta} \ms2 \e^{ u_\theta \ms1 |\tau| + v_\theta}.
\]
\end{subequations}

\end{smal}
\end{remN}

\subsubsection{Interpolation of holomorphic families of linear operators%
\label{Ihf}}

\noindent
We now recall the classical complex interpolation method for bounding in 
the norm of $L^p(X, \Sigma, \mu)$, when $1 < p < +\infty$, a linear operator
$T_\alpha$ that is a member of a holomorphic family of operators $(T_z)$,
for $z$ in a vertical strip $S$ containing $\alpha$. We consider a linear
space~$\ca E$ which is a common subspace of all $L^r(X, \Sigma, \mu)$, 
$1 \le r \le +\infty$, and which is dense in $L^r(X, \Sigma, \mu)$ when 
$1 \le r < +\infty$. This space $\ca E$ can be the space of simple
$\Sigma$-measurable and $\mu$-integrable functions, or for the specific
spaces $L^r(\R^n)$, it can be $\ca S(\R^n)$ or the space~$\ca K(\R^n)$. We
consider a closed strip $\alpha_0 \le \Re z \le \alpha_1$ in $\C$, with 
$\alpha_0 < \alpha < \alpha_1$. We assume that each $T_z$, for~$z$ in this
closed strip, is defined on~$\ca E$ and linear with values in 
$L^1(X, \Sigma, \mu) + L^\infty(X, \Sigma, \mu)$. The holomorphy assumption
means that for $f, g \in \ca E$, the function
$
 z \mapsto \langle T_z f, g \rangle
$
is holomorphic in the open strip $\alpha_0 < \Re z < \alpha_1$, but one also
assumes that it extends as a continuous function on the closed strip. The
above bracket is bilinear, given by $\int_X (T_z f) \ms1 g \, \d \mu$.
Later in these Notes, we shall abuse slightly and speak about holomorphic
family of linear operators in the \emph{closed} strip 
$\alpha_0 \le \Re z \le \alpha_1$.
\dumou

 We consider $1 \le p_0, p_1 \le +\infty$ and $p$ between $p_0$ and $p_1$,
so that $1 < p < +\infty$. We assume that when $\Re z = \alpha_j$, 
$j = 0, 1$, the $T_z\ms1$s are uniformly bounded from~$\ca E$, equipped
with the $L^{p_j}$ norm, to $L^{p_j}(X, \Sigma, \mu)$, and we assume that
for a certain $\theta \in (0, 1)$, we have both
\[
 \frac 1 p = \frac {1 - \theta} {p_0} + \frac {\theta} {p_1}
 \ms{18} \hbox{and} \ms{18}
 \alpha = (1 - \theta) \alpha_0 + \theta \alpha_1.
\]
We want to show that $T_\alpha$ is bounded from $\ca E$, equipped with the
$L^p$ norm, to $L^p(X, \Sigma, \mu)$. Then, by the density of $\ca E$, we
will be able to extend to $L^p(X, \Sigma, \mu)$ the bound obtained for the
functions in~$\ca E$.
\dumou
 
 We must of course bound
$
 \sca {T_\alpha f} g
$,
uniformly for $f$ in the intersection of $\ca E$ with the unit ball of
$L^p(X, \Sigma, \mu)$ and for~$g$ in the unit ball of the dual 
$L^q(X, \Sigma, \mu)$, $1 / p + 1 / q = 1$. Denote by $q_0$ the conjugate
of $p_0$ and by $q_1$ that of $p_1$. Observe that we have also
$
 1 / q = (1 - \theta) / q_0 + \theta / q_1
$.
We write $f(x) = u(x) |f(x)|$, $g(x) = v(x) |g(x)|$ for every $x \in X$,
with $|u(x)| = |v(x)| = 1$. Next, for each $z \in \C$, we set 
\begin{equation}
 f_z(x) = u(x) |f(x)|^{p(s z + t)},
 \ms{16}
 g_z(x) = v(x) |g(x)|^{q(1 - s z - t)},
 \ms{16}
 x \in X,
 \label{FzGzDef}
\end{equation}
where $s, t$ real are chosen such that $s \alpha_0 + t = 1 / p_0$ and
$s \alpha_1 + t = 1 / p_1$. This yields $s \alpha + t = 1 / p$. We see
that $f_\alpha = f$, $g_\alpha = g$ and we also see that the exponents have
been chosen so that the assumptions $\|f\|_p \le 1$ and $\|g\|_q \le 1$
imply
\[
 \forall \tau \in \R,
 \ms{12}
 \|f_{\alpha_0 + \ii \tau}\|_{p_0} \le 1,
 \ms{ 8}
 \|f_{\alpha_1 + \ii \tau}\|_{p_1} \le 1,
 \ms{12}
 \|g_{\alpha_0 + \ii \tau}\|_{q_0} \le 1,
 \ms{ 8}
 \|g_{\alpha_1 + \ii \tau}\|_{q_1} \le 1.
\]
We notice for future reference that if $f$ and $g$ are bounded by $M$
on~$X$, then
\begin{equation}
 |f_z| \le \max(M^{p/p_0}, M^{p/p_1}),
 \ms{12}
 |g_z| \le \max(M^{q/q_0}, M^{q/q_1})
 \label{FzGzBound}
\end{equation}
when $\alpha_0 \le \Re z \le \alpha_1$, because $\Re ( sz + t )$ stays
between $1 / p_0$ and $1 / p_1$ and $\Re ( 1 - sz - t )$ between $1 / q_0$
and $1 / q_1$ when $z \in S$. We now apply the three lines Lemma~\ref{TLL}
for bounding the value
$
   H(\alpha)
 = \sca {T_\alpha f} g
$
of the holomorphic function 
\begin{equation}
 H : z \mapsto \sca {T_z f_z} {g_z},
 \quad z \in S,
 \label{UsualH}
\end{equation}
from the bounds on the lines $\Re z = \alpha_0$ and $\Re z = \alpha_1$.
When $\Re z = \alpha_j$, we get
\[
     |H(z)|
  =  |\sca {T_z f_z} {g_z}|
 \le \|T_z\|_{p_j \rightarrow p_j}
      \ms2 \|f_z\|_{p_j} \ms1 \|g_z\|_{q_j}
 \le \|T_z\|_{p_j \rightarrow p_j},
\]
for $j = 0, 1$. In addition, the holomorphic function $H$ must be bounded
on the strip, see Remark~\ref{PhragLin} above. If true, we know by
Lemma~\ref{TLL} that
\[
     |H(\alpha)|
  =  \bigl| \sca {T_\alpha f} g \bigr|
 \le \Bigl( \ms1
      \sup_{\tau \in \R} 
       \|T_{\alpha_0 + \ii \tau}\|_{p_0 \rightarrow p_0}
     \Bigr)^{1 - \theta}
     \ms2
     \Bigl( \ms1
      \sup_{\tau \in \R} 
       \|T_{\alpha_1 + \ii \tau}\|_{p_1 \rightarrow p_1} 
     \Bigr)^{\theta},
\]
and by taking the supremum over $f$ and~$g$, we obtain
\begin{equation}
     \|T_\alpha\|_{p \rightarrow p}
 \le \Bigl( \ms1
      \sup_{\tau \in \R} 
       \|T_{\alpha_0 + \ii \tau}\|_{p_0 \rightarrow p_0}
     \Bigr)^{1 - \theta}
     \ms2
     \Bigl( \ms1
      \sup_{\tau \in \R} 
       \|T_{\alpha_1 + \ii \tau}\|_{p_1 \rightarrow p_1} 
     \Bigr)^{\theta}.
 \label{HalphaBound}
\end{equation}
Finally, we can extend $T_\alpha$ from the dense subspace $\ca E$ 
to~$L^p(X, \Sigma, \mu)$. Sometimes, rather than looking for extension, one
obtains in this way a sharper estimate for the norm of an operator
$T_\alpha$ already known to be bounded on $L^p(X, \Sigma, \mu)$.
\dumou

 This complex method, introduced for $L^p$ spaces by 
Thorin~\cite{ThorinA, Thorin}\label{ByThor}
for \emph{one} linear operator, extended by Stein~\cite{SteinILO} to
families, can also be extended (see~\cite{BenePanz}) to spaces of the form
$L^p(L^r)$ and more generally, by the abstract complex interpolation method
due to Calder\'on~\cite{Calderon}, to a pair of the form 
$(L^{p_0}(A_0), L^{p_1}(A_1))$. One then obtains estimates in
$L^p(A_\theta)$, where $A_\theta$ is the space associated to the pair
$(A_0, A_1)$ by Calder\'on's method with parameter $\theta \in (0, 1)$.
\dumou

 In many cases later on, the norms of the operators~$(T_z)_{z \in S}$ are
not \emph{uniformly} bounded on the boundary lines, but obey for some 
$\lambda > 0$ estimates of the form
\[
     \|T_{\alpha_0 + \ii \tau}\|_{p_0 \rightarrow p_0}
 \le C_0 \e^{\lambda |\tau|},
 \ms{16}
     \|T_{\alpha_1 + \ii \tau}\|_{p_1 \rightarrow p_1} 
 \le C_1 \e^{\lambda |\tau|},
 \quad
 \tau \in \R.
\]
Using Corollary~\ref{InStrip}, we can handle this situation. We must simply
check that the above function $H(z) = H_{f, g}(z)$ in~\eqref{UsualH} has an
admissible growth in the strip. We have to find an {\it ad hoc} argument
giving such a growth for each choice of $f$ and $g$ in suitable dense
subsets, growth depending on $f, g$. Indeed, in general, we do not know yet
bounds on the norm $\|T_z\|_{p_z \rightarrow p_z}$ for $z \in S$, where
$w / p_z = (\alpha_1 - \Re z) / p_0 + (\Re z - \alpha_0) / p_1$
and where $w = \alpha_1 - \alpha_0$ is the width of $S$. If each
function~$H_{f, g}$ has an admissible growth in $S$, we obtain here at
last that
\[
     \|T_\alpha\|_{p \rightarrow p} 
 \le C_0^{1 - \theta} C_1^\theta 
      \e^{\lambda \ms1 w \ms1 \sqrt{ \theta (1 - \theta) }}.
\]
If an additional polynomial factor is present in the bound of
$\|T_{\alpha_j + \ii \tau}\|_{p_j \rightarrow p_j}$, $j = 0, 1$, then we
make use of Remark~\ref{PolyFacto} and of the estimate~\eqref{PolyBd}.

\subsection{On the definition of maximal functions%
\label{DefiMaxiFunc}}

\noindent
Let us consider a family $(K_t)_{t > 0}$ of integrable functions on $\R^n$
and define a related maximal function by the formula
\begin{equation}
 \mu(f) = \sup_{t > 0} \ms1 |K_t * f|
 \label{DefiMaxi}
\end{equation}
for $f \in L^p(\R^n)$, $1 \le p < +\infty$. We are faced with a standard
difficulty of processes with continuous time parameter. In this generality,
the convolution $K_t * f$ is only defined almost everywhere, for each 
$t > 0$, and the preceding supremum is not a well defined equivalence class
of measurable functions. However, if $D$ is a countable subset of 
$(0, +\infty)$, there is no problem in considering
\[
 \mu_D(f) = \sup_{t \in D} \ms1 |K_t * f|,
\]
and a classical workaround for defining $\mu(f)$ consists in introducing the
\emph{essential supremum}:\label{EssenSup} 
there is a countable subset $D_0 \subset (0, +\infty)$ such that
$\mu_D(f) = \mu_{D_0}(f)$ almost everywhere, whenever $D \supset D_0$. In
other words, for every $t > 0$, we then have $|K_t * f| \le \mu_{D_0}(f)$
almost everywhere. The essential supremum is defined to be the equivalence
class of $\mu_{D_0}(f)$. It is also the \emph{least upper bound} of the
family $(|K_t * f|)_{t > 0}$ in the Banach lattice $L^p(\R^n)$.

 Most often, we shall have the specific problem where one considers an
integrable kernel~$K$ on $\R^n$ and defines a maximal function using the
dilates of~$K$, by
\[
 \mu(f) = \sup_{t > 0} \ms1 |{\Di K t} * f|.
\]
If $f \in L^p(\R^n)$ and if $K$ belongs to $L^q(\R^n)$, with $q < +\infty$
and $1 / q + 1 / p = 1$, then ${\Di K t} * f$ is defined pointwise and 
$t \mapsto {\Di K t}$ is continuous from $(0, +\infty)$ to $L^q$. It
follows that $t \mapsto ({\Di K t} * f)(x)$ is continuous for every 
$f \in L^p(\R^n)$, $x \in \R^n$, and the aforementioned problem disappears.
If $K \in L^1(\R^n)$ and $f \in L^p(\R^n)$ are nonnegative, then
$({\Di K t} * f)(x)$ is a definite value in $[0, +\infty]$ for every 
$x \in \R^n$, but it is not immediately clear that a direct application
of~\eqref{DefiMaxi} gives what we want. However, we can find an increasing
sequence $(f_k)_{k \ge 0}$ of \emph{bounded} nonnegative Borel functions
tending almost everywhere to~$f$. Then for every $x \in \R^n$ and 
$k \ge 0$, the map $t \mapsto ({\Di K t} * f_k)(x)$ is continuous from 
$(0, +\infty)$ to $[0, +\infty)$, because $t \mapsto {\Di K t}$ is
continuous from $(0, +\infty)$ to $L^1(\R^n)$. It follows that 
$t \mapsto ({\Di K t} * f)(x)$ is lower semi-continuous, since it is an
increasing limit of continuous functions. For every countable dense
set $D$ one has thus
\[
   \mu_D(f)(x)
 = \sup_{s \in D} \ms1 ({\Di K s} * f)(x)
 = \sup_{t > 0} \ms1 ({\Di K t} * f)(x).
\]
This argument does not apply to kernels that can also assume negative
values, and it is precisely the case that will appear later.
\dumou

 We will have to investigate maximal functions such as
$\mu(f) = \sup_{t > 0} \ms1 |{\Di K t} * f|$, usually when 
$K \in L^1(\R^n)$, but also more generally when $K$ is a bounded measure on
$\R^n$. It will be often convenient to start the study with nice functions,
for example functions $\varphi$ belonging to the Schwartz class 
$\ca S(\R^n)$, for which $\mu(\varphi)$ is clearly defined. If a function 
$f \in L^p(\R^n)$ is given and since $\ca S(\R^n)$ is dense in $L^p(\R^n)$,
we may find for every $\varepsilon > 0$ a sequence $(\varphi_k)_{k \ge 0}$
in $\ca S(\R^n)$ such that
\[
     f = \sum_{k=0}^{+\infty} \varphi_k
 \ms{12} \hbox{in} \ms8 L^p(\R^n),
 \ms{16} \hbox{and} \ms{11}
 \sum_{k=0}^{+\infty} \|\varphi_k\|_p
 < \|f\|_p + \varepsilon.
\]
Since the convolution with ${\Di K t}$ is linear and continuous on
$L^p(\R^n)$, we have
\[
     {\Di K t} * f = \sum_{k=0}^{+\infty} {\Di K t} * \varphi_k
 \ms{12} \hbox{in} \ms8 L^p(\R^n),
 \ms{16} \hbox{and} \ms{11}
 \sum_{k=0}^{+\infty} \|{\Di K t} * \varphi_k\|_p < +\infty,
\]
so the series $\sum_{k=0}^{+\infty} {\Di K t} * \varphi_k$ converges also
almost everywhere to ${\Di K t} * f$, and we have almost everywhere
\[
     |{\Di K t} * f|
 \le \sum_{k=0}^{+\infty} |{\Di K t} * \varphi_k|
 \le \sum_{k=0}^{+\infty} \mu(\varphi_k).
\]
For any countable subset $D \subset (0, +\infty)$ we get
$\mu_D(f) \le \sum_{k=0}^{+\infty} \mu(\varphi_k)$, implying that~$\mu(f)$,
defined as essential supremum, is bounded by 
$\sum_{k=0}^{+\infty} \mu(\varphi_k)$. If we know that there 
exists~$\kappa$ such that 
$\|\mu(\varphi)\|_p \le \kappa \ms1 \|\varphi\|_p$ when 
$\varphi \in \ca S(\R^n)$, it follows that
\[
     \| \mu(f) \|_p
 \le \sum_{k=0}^{+\infty} \|\mu(\varphi_k)\|_p
 \le \kappa \ms1 \sum_{k=0}^{+\infty} \|\varphi_k\|_p
 \le \kappa \ms1 (\|f\|_p + \varepsilon),
\]
for every $\varepsilon > 0$. In order to bound $\mu(f)$ in $L^p(\R^n)$, it
is therefore enough to obtain a uniform bound for Schwartz functions.
Clearly, any dense linear subspace of~$L^p(\R^n)$ consisting of nice
functions can be used instead of~$\ca S(\R^n)$.
\dumou

 The classical maximal function $\M f$, as well as $\M_C f$
in~\eqref{OpMaxi}, is actually defined by means of 
$\sup_{t > 0} {\Di K t} * |f|$. This makes sense whenever the kernel $K$ is
nonnegative, but not for a general $K$. We shall 
distinguish\label{MKetMK}
\[
 \M_K f := \sup_{t > 0} {\Di K t} * |f|
 \ms{18} \hbox{and} \ms{12}
 \Mg_K f := \sup_{t > 0} \ms1 \bigl| {\Di K t} * f \bigr|
\]
by the tiny notational difference between the slanted or unslanted letter M.
When the kernel $K$ is nonnegative, we have obviously
$ \Mg_K f \le \M_K f = \Mg_K (|f|)$.

\section{The results of Stein for Euclidean balls%
\label{SteinsResults}}

\noindent
We prove here the remarkable fact due to Stein~\cite{SteinSF} that for 
$p > 1$, the maximal operator associated to Euclidean balls, \textit{i.e.},
the classical Hardy--Littlewood maximal operator $\M$ defined
in~\eqref{ClassicalM}, may be bounded in $L^p(\R^n)$ independently of the
dimension~$n$. Full details appeared in~\cite{SteStr}. Other proofs have
appeared since then, let us mention Auscher and Carro~\cite{AuCa} who found
the simple explicit bound $2 + \sqrt 2$ in $L^2(\R^n)$, extended by
interpolation as $(2 + \sqrt 2)^{2 / p}$ for $p \ge 2$. It is not known
whether or not the weak $(1, 1)$ norm of the maximal operator $\M$ is also
bounded independently of the dimension. Even if we shall not develop this
weak type aspect mentioned in our introduction, let us recall that the best
upper estimate that is known for the weak $(1, 1)$ norm of $\M$ is the
Stein--Str\"omberg $\mr O(n)$ bound~\cite{SteStr}.

\begin{thm}[Stein~\cite{SteinSF}]%
\label{indepdim}
Let\/ $1 < p \le +\infty$. For every integer $n \ge 1$ and all functions 
$f \in L^p(\R^n)$, one has that
\[
      \|\M f\|_{L^p(\R^n)}
 \le C(p) \|f\|_{L^p(\R^n)},
\]
where $C(p)$ is a constant independent of the dimension~$n$.
\end{thm}

\subsection{Proof of Theorem \ref{indepdim}%
\label{SphericOp}}
The main tool in the proof is the spherical maximal operator $\ca M$ defined
by
\[
 \label{RadiMaxi}
   (\ca M f)(x)
 = (\Mg_\sigma f)(x)
 = \sup_{r > 0} \ms2 
    \Bigl| 
     \int_{S^{n-1}} f(x - r \theta) \, \d \sigma(\theta) 
    \Bigr|,
 \quad \ x \in \R^n,
\]
where $\sigma$ is the normalized 
Haar measure\label{HaarMeas} 
on the unit sphere $S^{n-1}$. It is clear that $\ca M f$ is well defined
when $f$ is regular, but not when $f \in L^1_{\rm loc}(\R^n)$.
Theorem~\ref{sph} below means in particular that for suitable $p$ and $n$, 
$\ca M f$ can be defined when $f \in L^p(\R^n)$, for example by the method
described at the end of Section~\ref{DefiMaxiFunc}. The maximal function 
$\ca M (|f|)$ controls $\M f$ pointwise, as one sees easily by using polar
coordinates. The maximal operator $\ca M$ is bounded in $L^p(\R^N)$ for
some $p$ and $N$, with a bound depending on the dimension~$N$, according to
the following theorem also due to Stein. An extension by Bourgain of this
result can be found in~\cite{BourgAve}.

\begin{thm}[Stein~\cite{SteinMF}]%
\label{sph}
Let $N \ge 3$ and assume that $N / (N - 1) < p \le +\infty$. There exists a
constant $C(N, p)$ such that for every function $f \in L^p(\R^N)$, one has
\[
      \|\ca M f\|_{L^p(\R^N)}
 \le C(N, p) \ms2 \|f\|_{L^p(\R^N)}.
\]
\end{thm}

 The condition $p > N / (N - 1)$ can be easily seen necessary, and the case 
$p = +\infty$ is obvious, with $C(N, \infty) = 1$. We postpone the proof of
this theorem to the next section. It requires a number of harmonic analysis
methods, including square function, multipliers and Littlewood--Paley
decomposition.
\dumou

 In order to prove Theorem~\ref{indepdim}, we first introduce the following
weighted maximal operator, depending on a parameter $k \in \N$. For 
$f \in \ca S(\R^n)$, let
\[
   (\M_{n, k} f) (x)
 = \sup_{r > 0}
    \frac {\int_{|y| \le r} |f(x - y)| \, |y|^k \, \d y}
          {\int_{|y| \le r} |y|^k \, \d y} \up, 
 \quad x \in \R^n,
\]
where $|y|$ denotes the Euclidean norm of $y \in \R^n$. Taking polar
coordinates gives us the pointwise inequality 
\[
 (\M_{n, k} f) (x) \le (\ca M |f|)(x),
  \quad x \in \R^n,
\]
from which we can deduce by applying Theorem~\ref{sph} that for every
integer $N \ge 3$, for $p$ such that $N / (N-1) < p \le +\infty$ and for
every $f$ in $L^p(\R^N)$, we have
\begin{equation}
      \|\M_{N, k} f\|_{L^p(\R^N)}
 \le C(N, p) \ms2 \|f\|_{L^p(\R^N)}, 
 \label{pass} 
\end{equation}
where $C(N, p)$ is the constant in Theorem~\ref{sph}. We shall obtain 
Theorem~\ref{indepdim} by lifting to $\R^n$ the inequality~\eqref{pass}
obtained in a lower dimension $N = n - k$. This is done by integrating over
the Grassmannian of $(n-k)$-planes in $\R^n$. This method of descent is in the
spirit of the Calder\'on--Zygmund 
\emph{method of rotations}.\label{RotaMeth}

 We write $\R^n = \R^{n-k} \times \R^k$ and $x = (x_1 , x_2)$ accordingly,
for every $x \in \R^n$, with $x_1 \in \R^{n-k}$ and $x_2 \in \R^k$. For
each $U$ in the orthogonal group $\ca O(n)$, we introduce the auxiliary
maximal operator 
\[
   (\M^U_k f)(x)
 = \sup_{r > 0} 
    \frac { \int_{\R^{n-k}} \gr 1_{ \{|y_1| \le r\} } \ms2
             |f (x - U(y_1 , 0) )| \,|y_1|^{k} \, \d y_1 
          }
          { \int_{\R^{n-k}} \gr 1_{ \{|y_1| \le r\} } \ms2 |y_1|^{k}
             \, \d y_1 
          } \up,
 \quad x \in \R^n.
\] 
We need two lemmas.

\begin{lem}
\label{L1}
Let $n \ge k + 3$ and $p > (n-k) / (n-k-1)$. Then for all 
$f \in L^p(\R^n)$ and $U \in \ca O(n)$, we have
\[
      \|\M^U_{k} f\|_{L^p(\R^n)}
 \le C(n - k, p) \, \|f\|_{L^p(\R^n)},
\]
where $C(n-k, p)$ is the constant appearing in Theorem~\ref{sph}.
\end{lem}

\begin{proof}
Let us set ${ \di f U }(x) = f(U x)$, for every $x \in \R^n$. Since 
$U \in \ca O(n)$, the mapping $S_U : f \mapsto { \di f U }$ is an isometry of
$L^p(\R^n)$. Observe that
\[
   \int_{|y_1| \le r}
             \bigl| f \bigl(U x - U(y_1 , 0) \bigr) \bigr|
              \,|y_1|^{k} \, \d y_1
 = \int_{|y_1| \le r}
             \bigl| { \di f U } (x - (y_1 , 0) ) \bigr|
              \,|y_1|^{k} \, \d y_1,
\]
hence we have that
$  (\M^U_{k} f)(U x) 
 = (\M^{\Id}_{k} { \di f U })(x)$, for every $x \in \R^n$. This means that 
$S_U \M^U_k = \M^{\Id}_k S_U$. It follows that we need only consider
$\M^{\Id}_k$. Now, for every $x = (x_1, x_2) \in \R^n$, we have
\[
   (\M^{\Id}_k f) (x_1, x_2)
 = \sup_{r > 0} 
    \frac { \int_{\R^{n-k}} \gr 1_{ \{|y_1| \le r\} } \ms2
             \bigl| f (x_1 - y_1, x_2) \bigr|
              \,|y_1|^k \, \d y_1 }
          { \int_{\R^{n-k}} \gr 1_{ \{|y_1| \le r\} } \ms2 |y_1|^k
            \, \d y_1 }
 = \bigl( \M_{n-k, k} \ms2 f_{x_2} \bigr) (x_1 )
\]
with
$
   f_{x_2}(x_1)
 = f(x_1, x_2)
$.
Applying~\eqref{pass} to $\M_{n-k, k}$ for each $x_2 \in \R^k$ gives
\[
     \int_{\R^{n-k}} \bigl| (\M^{\Id}_k f) (x_1, x_2) \bigr|^p \, \d x_1
 \le C(n-k, p)^p \int_{\R^{n-k}} 
      \bigl| f_{x_2} (x_1) \bigr|^p \, \d x_1,
\]
therefore
\begin{align*}
      \| \M_k^{\Id} f\|^p_{L^p(\R^n)}
 &\le C(n-k, p)^p \int_{\R^k}
       \biggl( \int_{\R^{n-k}}
        \bigl| f_{x_2} (x_1) \bigr|^p \, \d x_1         
       \biggr) \d x_2
 \\
  &=  C(n-k, p)^p \ms3 \| f \|^p_{L^p(\R^n)}.
\end{align*}
\end{proof}

\begin{lem}%
\label{L2}
For every locally integrable function $f$ on\/ $\R^n$ and\/ $1 \le k \le n$,
one has the pointwise inequality
\[
      (\M f) (x)
 \le \int_{\ca O(n)}
       (\M^U_k f) (x) \, \d \mu_n(U), 
 \quad x \in \R^n,
\]
where $\mu_n$ denotes the normalized Haar measure on $\ca O(n)$.
\end{lem}

\begin{proof}
The desired pointwise inequality follows from the next equality, true for
every nonnegative Borel function $g$ on $\R^n$, stating that
\begin{equation}
   \frac {\int_{|y| \le r} g(y) \, \d y}
         {\int_{|y| \le r} \, \d y}
 = \frac {\int_{\ca O(n)} \int_{\R^{n-k}} 
           \gr 1_{ \{|y_1| \le r\} } \ms2
            g \bigl( U(y_1 ,0) \bigr) 
             \, |y_1|^{k}
              \, \d y_1  \d \mu_n(U)}
         {\int_{\R^{n-k}} \gr 1_{ \{|y_1| \le r\} } \ms2
          |y_1|^{k} \, \d y_1 } \up.
 \label{intorth}
\end{equation}
Indeed, for each $r > 0$ and $x \in \R^n$, the previous equality allows us
to write
\begin{align*}
    \frac 1 {|B_r|} \int_{B_r} |f(x - y)| \, \d y
 &= \frac {\int_{\ca O(n)} \int_{|y_1| \le r} \bigl
           |f \bigl( (x - U(y_1 , 0) \bigr) \bigr| \, |y_1|^{k}
            \, \d y_1  \d \mu_n(U)}
          {\int_{|y_1| \le r} |y_1|^{k} \, \d y_1 } 
 \\ 
 & \le \int_{\ca O(n)} (\M^U_k f) (x) \, \d \mu_n(U),
\end{align*}
and we conclude by taking the supremum over all $r > 0$. 

 It remains to check~\eqref{intorth}. By standard measure-theoretic
arguments about classes of functions generating the Borel $\sigma$-algebra
of~$\R^n$, we can suppose that $g$ has the form $g(x) = g_0(|x|) g_1(x')$,
with $x = |x| \ms1 x'$ and $x' \in S^{n-1}$. By taking polar coordinates,
we see that the left-hand side of~\eqref{intorth} is equal to
\[
 \frac n {r^n} \Bigl( \int_0^r g_0(t) t^{n-1} \, \d t \Bigr)
  \Bigl( \int_{S^{n-1}} g_1(y') \, \d \sigma_{n-1}(y') \Bigr),
\]  
where $\sigma_{n-1}$\label{SigmSubN} 
is the invariant probability measure on $S^{n-1}$. The right-hand
side is
\[
 \frac n {r^n} \Bigl( \int_0^r g_0(t) \ms1 t^{n-1} \, \d t \Bigr)
  \Bigl( \int_{\ca  O(n)}
   \int_{S^{n-k-1}} g_1 \bigl( U(y'_1 , 0) \bigr)
    \, \d \sigma_{n-k-1} (y'_1 ) \ms2 \d \mu_n(U) \Bigr).
\]
Observe that for every $\theta_0 \in S^{n-1}$, we have
\[
   \int_{\ca O(n)} g_1(U \theta_0) \, \d \mu_n(U)
 = \int_{S^{n-1}} g_1(\theta) \, \d \sigma_{n-1} (\theta),
\]
since the left-hand side of this equality defines a probability measure on
$S^{n-1}$, namely $\ca B_{S^{n-1}} \ni A \mapsto
 \int_{\ca O(n)} \gr 1_A(U \theta_0) \, \d \mu_n(U)$, which is invariant
under the left-action of $\ca O(n)$, hence equal to~$\sigma_{n-1}$. We have
therefore
\begin{align*}
   & \ms5 \int_{\ca  O(n)}
    \int_{S^{n-k-1}} g_1 \bigl( U(y'_1 , 0) \bigr)
     \, \d \sigma_{n-k-1} (y'_1 ) \ms1 \d \mu_n(U)
 \\
 =& \ms4 \int_{S^{n-k-1}}
    \Bigl( \int_{\ca  O(n)}
     g_1 \bigl( U(y'_1 , 0) \bigr) \, \d \mu_n(U) \Bigr)
      \, \d \sigma_{n-k-1} (y'_1 )
 \\
 =& \ms4 
   \int_{S^{n-k-1}} 
    \Bigl(
     \int_{S^{n-1}} g_1 (\theta) \, \d \sigma_{n-1} (\theta) 
    \Bigr)
     \, \d \sigma_{n-k-1} (y'_1 )
 = \int_{S^{n-1}} g_1(y') \, \d \sigma_{n-1} (y'),
\end{align*}
completing the proof.
\end{proof}

\begin{proof}[Proof of Theorem~\ref{indepdim}]
Let $1 < p \le +\infty$. There is obviously nothing to do if $n \le 2$.
When $n \le p / (p - 1)$, the \og bad\fge Vitali-bound $C(n) = 3^n$ in
the classical maximal inequality~\eqref{strong-type} is less than a
function of $p$ alone, namely $3^{p / (p-1)}$. We can therefore assume that
both inequalities $n > p / (p-1)$ and $n \ge 3$ hold. We then write 
$n = (n-k) + k$ with 
$n - k
 = \bigl \lfloor \max \bigl( p / (p-1), 2 \bigr) \bigr \rfloor + 1$,
and the result follows from Lemma~\ref{L1} and Lemma~\ref{L2} since with
this choice, the bound $C(n - k, p)$ in Lemma~\ref{L1} is now a function of
$p$ alone.
\end{proof}

\subsection{Boundedness of the spherical maximal operator%
\label{Bsmo}}

In this section, we prove Theorem~\ref{sph} following the approach of Rubio
de Francia~\cite{rubio}, see also Grafakos~\cite{GrafakosCFA}. Let $n \ge 2$.
The spherical maximal operator is expressed by
\[
 \label{MsubS}
   (\ca M f)(x)
 = \sup_{r > 0} \ms1
    \bigl|
     \bigl[
      m_\sigma(r \, \cdot) \ms2 \widehat f ( \cdot )
     \bigr]^\vee(x)
    \bigr|
 = \sup_{r > 0} \ms1 \bigl| \bigl( {\Di \sigma r} * f \bigr)(x) \bigr|,
 \quad
 x \in \R^n,
\]
where $h^\vee(x) = \widehat h (-x)$ denotes the inverse Fourier transform
of a function $h$, $m_\sigma$ is the Fourier transform of the uniform
probability measure $\sigma$ on the unit sphere~$S^{n-1}$, and 
$\Di \sigma r$ is the dilated probability measure defined
in~\eqref{MuLambda}. It is known that
\begin{equation}
   m_\sigma(\xi)
 = \widehat \sigma (\xi)
 = (2 \pi |\xi|)^{- (n - 2)/2}
    \mr J_{ (n - 2) / 2 } (2 \pi |\xi|),
 \quad \xi \in \R^n,
 \label{Bess}
\end{equation}
with $\mr J_\nu$ the Bessel function of order $\nu$. This equality follows
from the fact that the two functions
$t \mapsto t^{- (n - 2)/2} \mr J_{(n-2)/2}(t)$ and
\[
   t \mapsto F(t) 
 = \int_{S^{n-1}} \e^{\ii t x_1} \, \d \sigma(x)
 = \frac {2 s_{n-2}} {s_{n-1}}
    \int_0^1 (1 - s^2)^{ (n-3)/2 } \cos(s \ms2 t) \, \d s
\]
are entire functions $g$ satisfying $g(0) = 1$ and
$t^2 (g''(t) + g(t)) = - (n-1) t \ms1 g'(t)$.
\dumou

 We shall rely on the Littlewood--Paley theory, decomposing multipliers
into dyadic pieces with localized frequencies. More precisely, we shall
dominate $\ca M$ by a series of maximal operators 
$\sum_{\ell = 0}^{+\infty} \Mg_{K_\ell}$, where each kernel $K_\ell$ is
radial with a well localized Fourier transform $m_\ell$. We establish that
$\Mg_{K_\ell}$ is of strong type when $p = 2$ and of weak type $(1, 1)$.
Then, we get an $L^p$ bound for $\Mg_{K_\ell}$ by interpolation, and the
range of $p$ in Theorem~\ref{sph} is chosen for making the series of bounds
convergent. For the case $p = 2$, we mainly use for $m_\ell$ both the decay
at infinity and a support property, together with a precise upper bound for
the $L^2(\R^n)$ norm of a related square function. When $p = 1$, we invoke
the usual Hardy--Littlewood theorem. Before giving the proof of
Theorem~\ref{sph}, we introduce the dyadic decomposition\label{LiPaDec}
of~$m_\sigma = \widehat \sigma$. 
\dumou

 Let $\varphi_0$ be a smooth radial function on $\R^n$ satisfying for every
$\xi \in \R^n$ that
\[
   \varphi_{0}(\xi)
 = \begin{cases} 
     \ms3 1 \mr{\ \, if \ \,} |\xi| \le 1 
     \\ 
     \ms3 0 \mr{\ \, if \ \,} |\xi| \ge 2.
   \end{cases}
\]
Let $\psi(\xi) = \varphi_0(\xi) - \varphi_0(2 \xi)$ for $\xi$ in $\R^n$.
This function $\psi$ is supported in the annulus 
$\{ 1/2 \le |\xi| \le 2 \}$. For every integer $\ell \ge 1$ we define 
\[
   \varphi_\ell(\xi) 
 = \varphi_0(2^{-\ell}\xi) - \varphi_0(2^{1 - \ell}\xi)
 = \di \psi{2^{-l}} (\xi),
 \quad
 \xi \in \R^n,
\]
and for every $\ell \ge 0$, we consider the dyadic radial piece
$
 m_\ell = \varphi_\ell \ms2 m_\sigma
$
associated to the multiplier~$m_\sigma$. We can check that 
$\sum_{\ell = 0}^{+\infty} \varphi_\ell = 1$, thus
$
   m_\sigma
 = \sum_{\ell = 0}^{+\infty} m_\ell
$.
For every $\ell \ge 0$, we introduce the integrable
kernel $K_\ell = m_\ell^\vee = \varphi_\ell^\vee * \sigma$ and we set
\[
   (\Mg_{K_\ell} f) (x)
 = \sup_{r > 0} \ms2 
     \bigl| 
      \bigl[ 
       m_\ell(r \, \cdot) \widehat f ( \cdot ) 
      \bigr]^\vee(x) 
     \bigr|
 = \sup_{r > 0} \ms2 
     \bigl| 
      \bigl[ 
       \Di {(\varphi_\ell^\vee)} r * {\Di \sigma r} * f
      \bigr] (x) 
     \bigr|,
 \quad
 x \in \R^n,
\]
when $f \in \ca S(\R^n)$. In particular, we have 
$
   \Mg_{K_0} f
 = \sup_{r > 0} 
    \bigl| \Di {(\varphi_0^\vee)} r * {\Di \sigma r} * f \bigr|$
and
\[
   \Mg_{K_\ell} f
 = \sup_{r > 0} \ms2
    \bigl|
     \Di {(\psi^\vee)} {2^{-\ell} r} * {\Di \sigma r} * f
    \bigr|,
  \quad
 \ell \ge 1.
\]
For every $x \in \R^n$ and $r > 0$, we see that
$  ({\Di \sigma r} * f) (x)
 = \sum_{\ell = 0}^{+\infty} \ms1 ( {\Di{(K_\ell)} r} * f) (x)$ and we get
the pointwise inequality 
\begin{equation}
      (\ca M f) (x)
 \le \sum_{\ell = 0}^{+\infty} \ms1 (\Mg_{K_\ell} f) (x).
 \label{ptw}
\end{equation}
In a first subsection, we present some useful results on this type of
maximal operators and associated square functions. Then, we shall
prove that each~$\Mg_{K_\ell}$, for $\ell \ge 0$, is of strong type when 
$p = 2$ and of weak type when $p = 1$, and we give the proof of
Theorem~\ref{sph} in a third subsection.

\subsubsection{Maximal operator and square function%
\label{Mmoasf}}

\noindent
Let $m(\xi)$ be a multiplier that is a bounded continuous function on
$\R^n$, vanishing at $0$, with $|m(\xi)| = O(|\xi|)$ in a neighborhood of
$0$. For $f$ in the Schwartz class~$\ca S(\R^n)$ and for $x \in \R^n$, set
\begin{align*}
   (g_m f)(x)
 &= \Bigl( \int_0^{+\infty} 
    |(T_{ \di m u } f)(x)|^2 \, \frac {\d u} u \Bigr)^{1/2}
 \\
 &= \Bigl( \int_0^{+\infty} \Bigl|
    \int_{\R^n} m(u \xi) \widehat f(\xi) \e^{2 \ii \pi x \ps \xi}
     \, \d \xi \Bigr|^2 \, \frac {\d u} u \Bigr)^{1/2}.
\end{align*}
We obtain the Littlewood--Paley function $g_1(f)$ of~\eqref{L-PfunctionG1}
when $m(\xi) = 2 \pi |\xi| \ms1 \e^{ - 2 \pi |\xi|}$.

\begin{lem}%
\label{ClefA}
Assume that the multiplier $m(\xi)$ is a bounded function of 
$\xi \in \R^n$, supported in an annulus of the form\/ 
$\{ a \le |\xi| \le r \ms1 a \}$, $a > 0$ and $r > 1$. For every
function $f \in \ca S(\R^n)$, one has that
\[
     \|g_m f\|_{L^2(\R^n)} 
 \le \sqrt{\ln r} \ms4 \|m\|_{L^\infty(\R^n)} \ms1 \|f\|_{L^2(\R^n)}.
\]
\end{lem}

\begin{proof}
According to the Fubini theorem, followed by Parseval, Fubini again and
setting finally $v = u |\xi|$, we have
\begin{align*}
     \int_{\R^n} |(g_m f)(x)|^2 \, \d x
 &= \int_0^{+\infty} 
    \| T_{ \di m u } f \|_2^2 \, \frac {\d u} u
  =  \int_0^{+\infty} 
      \int_{\R^n} |m(u \xi)|^2 |\widehat f(\xi)|^2 
       \, \frac {\d u} u \ms2 \d \xi
 \\
 &\le \|m\|_\infty^2 \int_{\R^n} 
      \Bigl( \int_a^{r \ms1 a}  \frac {\d v} v \Bigr)
       |\widehat f(\xi)|^2 \, \d \xi    
  =  \|m\|_\infty^2 \ms1 \ln(r) \ms1 \|f\|_2^2.
 \qedhere
\end{align*}
\end{proof}

\begin{lem}%
\label{ClefB}
Assume that $m(\xi)$ is of class $C^1$ on\/ $\R^n$ and vanishes outside a
compact subset of\/ $\R^n \setminus \{0\}$. For every $t > 0$ and
$f \in \ca S(\R^n)$, we have
\[
     \Bigl|
      \int_{\R^n} 
       m(t \xi) \widehat f(\xi) \e^{2 \ii \pi x \ps \xi} \, \d \xi
     \Bigr|^2
 \le 2 \ms2 (g_m f)(x) \ms3 (g_{m^*} f)(x),
 \quad
 x \in \R^n,
\]
where we have set\label{mStar}
$
 m^*(\xi) = \xi \ps \nabla m(\xi)
$ for every $\xi \in \R^n$.
\end{lem}

\begin{proof}
For each $s \ge 0$ let us set 
\[
   (g_{m, s} f)(x)
 = (T_{ \di m s } f)(x)
 = \int_{\R^n} 
    m(s \xi) \widehat f(\xi) \e^{2 \ii \pi x \ps \xi} \, \d \xi,
 \quad
 x \in \R^n.
\]
We note that
\[
   s \ms2 \frac {\d} {\d s \ns2} \ms2 (g_{m, s} f)(x)
 = \int_{\R^n} 
    s \xi \ps \nabla m(s \xi) \widehat f(\xi) 
     \e^{2 \ii \pi x \ps \xi} \, \d \xi
 = \int_{\R^n} 
    m^*(s \xi) \widehat f(\xi) 
     \e^{2 \ii \pi x \ps \xi} \, \d \xi,
\]
which allows us to see this quantity as $(g_{m^*\ns3, s} f)(x)$. Since
$m$ vanishes in a neighborhood of $0$, one has $(g_{m, 0} f)(x) = 0$, thus
\begin{align*}
   &\ms5 | (g_{m, t} f)(x) |^2
 = \int_0^t \frac {\d} {\d s \ns2} \ms2 |(g_{m, s} f)(x)|^2 \, \d s
 \\
 = &\ms5 2 \Re \int_0^t  \overline{ (g_{m, s} f)(x) } \ms4
    s \ms2 \frac {\d} {\d s \ns2} \ms2 (g_{m, s} f)(x) 
     \, \frac {\d s} s
 = 2 \Re \int_0^t  \overline{ (g_{m, s} f)(x) } \ms2
    (g_{m^*\ns2, s} f)(x) 
     \, \frac {\d s} s \ms1 \up.
\end{align*}
By Cauchy--Schwarz, and bounding the integral on $[0, t]$ by the integral
on $[0, +\infty)$, we obtain that
\begin{align*}
      | (g_{m, t} f)(x) |^2
 &\le 2 \ms2 \Bigl( \int_0^{+\infty}  \bigl| (g_{m, s} f)(x) \bigr|^2
          \, \frac {\d s} s \Bigr)^{1/2} 
        \Bigl( \int_0^{+\infty}  \bigl| (g_{m^*\ns2, s} f)(x) \bigr|^2
          \, \frac {\d s} s \Bigr)^{1/2}
 \\
 & =  2 \ms2 (g_m f)(x) \ms3 (g_{m^*} f)(x).
 \qedhere
\end{align*}
\end{proof}

\begin{lem}\label{MaxiLoca}
Let $K$ be an integrable kernel on\/ $\R^n$. Suppose that $m$, the Fourier
transform of $K$, is of class $C^1$ on\/~$\R^n$ and supported in an annulus
of the form\/ $\{ a \le |\xi| \le r \ms1 a \}$, $a > 0$ and $r > 1$. For
every function $f \in \ca S(\R^n)$, one has that
\[
     \| \Mg_K f \|_{L^2(\R^n)}^2
  =  \bigl\| \sup_{t > 0} \ms1 |\Di K t * f | \ms1 \bigr\|_{L^2(\R^n)}^2
 \le 2 \ms1 \ln (r) \ms2 \|m\|_{L^\infty(\R^n)} \ms1 
      \|m^*\|_{L^\infty(\R^n)} \ms1 \|f\|_{L^2(\R^n)}^2,
\]
where
$
 m^*(\xi) = \xi \ps \nabla m(\xi)
$ for $\xi \in \R^n$.
\end{lem}

\begin{proof}
By Lemma~\ref{ClefB}, we have for every $x \in \R^n$ and $t > 0$ that
\[
     \bigl| ({\Di K t} * f \ms1 ) (x) \bigr|^2 
  =  \Bigl| 
      \int_{\R^n} m(t \xi) 
       \widehat f(\xi) \e^{2 \ii \pi x \ps \xi} \, \d \xi
     \Bigr|^2
 \le 2 \ms2 (g_m f)(x) \ms2 (g_{m^*} f)(x).
\]
This upper bound is independent of~$t$, thus
$
     \bigl( (\Mg_K f)(x) \bigr)^2
 \le 2 \ms2 (g_m f)(x) \ms2 (g_{m^*} f)(x)
$, and
by Cauchy--Schwarz we get
\[
     \bigl\| \Mg_K f \bigr\|_{L^2(\R^n)}^2
 \le 2 \ms2 \| g_m f \|_{L^2(\R^n)}
      \ms1 \| g_{m^*} f \|_{L^2(\R^n)}.
\]
According to Lemma~\ref{ClefA}, we conclude that
\[
     \bigl\| \Mg_K f \bigr\|_{L^2(\R^n)}^2
 \le 2 \ms2 \ln (r) \ms2 
      \|m\|_{L^\infty(\R^n)} \ms1 \|m^*\|_{L^\infty(\R^n)} \ms1 
       \|f\|_{L^2(\R^n)}^2.
 \qedhere
\]
\end{proof}

 The following proposition is nearly obvious. 

\begin{prp}%
\label{map} 
Let $K \in \ca S(\R^n)$ be a radial kernel. For every $p$ in\/ 
$(1, +\infty]$, the maximal operator\/ $\Mg_K$ is bounded on~$L^p(\R^n)$.
\end{prp}

 One also gets the weak type $(1, 1)$ for $\Mg_K$, but we shall
not use it.

\begin{proof}
Since $K$ is a Schwartz radial function, we can find an
integrable function~$\Omega$, radial and radially decreasing, such that 
$|K| \le \Omega$. It implies that
\[
     \sup_{r > 0} \ms2 
      \bigl| \bigl[ {\Di K r } * f \bigr](x) \bigr|
 \le \sup_{r > 0} \ms2 
      \bigl( \Di \Omega r * |f| \bigr) (x),
 \quad x \in \R^n,
\]
and $\Omega$ being radial and radially decreasing, we classically have
\begin{equation}
     \sup_{r > 0} \ms2 \bigl( \Di \Omega r * |f| \bigr) (x)
 \le \|\Omega\|_{L^1(\R^n)} (\M f) (x), 
 \quad x \in \R^n.
 \label{MajOmega1}
\end{equation}
By Theorem~\ref{TheoHL}, the usual maximal theorem for $\M$, we get the
conclusion.

\begin{smal}%
\noindent
For proving~\eqref{MajOmega1}, it suffices to show that
\begin{subequations}\label{MajOmega}
\SmallDisplay{\eqref{MajOmega}}%
\[
     \bigr| (\Omega * f ) (x) \bigr|
 \le \|\Omega\|_{L^1(\R^n)} (\M f) (x), 
 \quad x \in \R^n.
\]
\end{subequations}
\noindent
Suppose that $\Omega \le 1$ for simplicity, and consider for each integer
$k \ge 1$ the set
\[
 A_k = \{ x \in \R^n : \Omega(x) > 2^{-k} \}.
\]
This set $A_k$ is a Euclidean ball, and if we define
$
 g = \sum_{k \ge 1} 2^{-k} \gr 1_{A_k}
$,
we can check that
$
 g / 2 \le \Omega \le g
$.
We rewrite $g$ as
\[
 g = \sum_{k \ge 1} a_k \ms1 \frac {\gr 1_{A_k}} {|A_k|} \up,
\]
with $a_k > 0$ for every $k \ge 1$. Since $\Omega$ is integrable, $g$ is
also integrable and
\[
     \sum_{k \ge 0} a_k
  =  \int_{\R^n} g(x) \, \d x 
 \le 2 \int_{\R^n} \Omega(x) \, \d x
  =  2 \ms1 \|\Omega\|_{L^1(\R^n)}.
\]
We have for every $x \in \R^n$ that
\begin{align*}
      \bigl| (\Omega * f) (x) \bigr|
 &\le ( g * |f| )(x)
  =  \sum_{k \ge 1} \frac {a_k} {|A_k|}
      \int_{x + A_k} |f(y)| \, \d y
 \\
 &\le \Bigl( \sum_{k \ge 0} a_k \Bigr) (\M f)(x)
  \le 2 \ms1 \|\Omega\|_{L^1(\R^n)} \ms1 (\M f)(x).
\end{align*}
The inequality with constant $1$ can be reached by refining the partition,
replacing the values $2^{-k}$ by $(1 + \varepsilon)^{-k}$, with
$\varepsilon > 0$ tending to~$0$. One can also give a direct proof
involving integration by parts, or the Fubini theorem and level sets
of~$\Omega$.
\qedhere

\end{smal}
\end{proof}

\subsubsection{Strong and weak type results for $\Mg_{K_\ell}$, 
 $\ell \ge 1$%
\label{Sawtr}}

\noindent
We begin with the strong type result, when $p = 2$.

\begin{prp}%
\label{propL2}
For every integer $\ell \ge 1$ and every $f \in L^2(\R^n)$ one has that
\[
      \bigl\| \Mg_{K_\ell} f \bigr\|_{L^2(\R^n)}
 \le C(n) \ms1 2^{ - \ell(n-2) / 2 } \|f\|_{L^2(\R^n)},
\]
where $C(n)$ is a constant independent of $\ell$.
\end{prp}

\begin{proof}
For each $\ell \ge 1$, the multiplier $m_\ell = \widehat {K_\ell}$ is
$C^1$, supported in the annulus
\[
   I_\ell
 = \{\xi \in \R^n: 
            2^{\ell - 1} \le |\xi| \le 2^{\ell + 1}\}.
\]
Applying Lemma~\ref{MaxiLoca} to $K_\ell$, with 
$m^*_\ell(\xi) = \xi \ps \nabla m_\ell(\xi)$ and $r = 4$, we obtain
\[
     \bigl\| \Mg_{K_\ell} f \bigr\|^2_{L^2(\R^n)}
 \le 2 \ms1 \ln (4) \ms2 \|m_\ell\|_{L^\infty(\R^n)} \ms1 
      \|m^*_\ell\|_{L^\infty(\R^n)} \ms1 \|f\|_{L^2(\R^n)}^2.
\]
The desired result will be consequence of the inequalities 
\begin{equation}
      \|m_\ell\|_{L^\infty(\R^n)}
 \le C_1(n) \ms1 2^{ - { {\ell (n-1) / 2} }}, \quad
       \|m^*_\ell\|_{L^\infty(\R^n)}
 \le C_2(n) \ms1 2^{ - { {\ell (n-3) / 2} }}
 \label{bess2}
\end{equation}
that we establish now, with $C_1(n)$ and $C_2(n)$ independent of $\ell$.
Thanks to well-known properties of Bessel functions (see for
instance~\cite[p.~238]{A-Askey-R}), we have
\begin{equation}
 \sup_{t \ge 1} \, t^{1/2} |\mr J_\alpha(t)| < +\infty,
 \ms{12} \hbox{and} \ms{12}
   \frac {\d} {\d t} \ms1 \mr J_\alpha(t)
 = \frac 1 2 \ms1
    \bigl( \mr J_{\alpha - 1}(t) - \mr J_{\alpha + 1}(t) \bigr).
 \label{ClassiBess}
\end{equation}

\begin{smal}

\noindent
The first property follows from the fact that 
$u_\alpha(t) = \sqrt t \ms1 \mr J_\alpha(t)$ verifies a differential
equation $u''_\alpha(t) + (1 + \kappa_\alpha t^{-2}) \ms1 u_\alpha(t) = 0$
for $t > 0$, hence $v_\alpha(t) := ( u_\alpha (t)^2 + u'_\alpha (t)^2 ) / 2$
satisfies $v'_\alpha(t) = - \kappa_\alpha t^{-2} u_\alpha(t) u'_\alpha(t)
 \le |\kappa_\alpha| \ms1 t^{-2} v_\alpha(t)$, yielding
$v_\alpha(t) \le \e^{|\kappa_\alpha|} v_\alpha(1)$ for every $t \ge 1$. The
second property can be checked on the coefficients of the power series
$\sum_{m \ge 0} (-1)^m
 \bigl( m! \ms1 \Gamma(m + \alpha + 1) \bigr)^{-1} (t/2)^{2m}$
of $t^{-\alpha} \mr J_\alpha(t)$, and when $\alpha = n \in \N$, it is even
simpler to see it on the integral expression $2 \pi \mr J_n(t)
 = \int_0^{2 \pi} \e^{ \ii ( t \sin s - n s)} \, \d s$.
 
\end{smal}
 
\noindent
Since $m_\ell$ and $m^*_\ell$ are supported in the annulus $I_\ell$, we
need only bound $m_\ell(\xi)$ and $m^*_\ell(\xi)$ when 
$1 \le 2^{\ell - 1} \le |\xi| \le 2^{\ell + 1}$ (we have $\ell \ge 1$). We
then obtain~\eqref{bess2} by recalling~\eqref{Bess} and by
applying~\eqref{ClassiBess} to $t = 2 \pi |\xi| > 1$, which give that
\[
      |m_\ell(\xi)|
 \le c_1(n) |\xi|^{ - n / 2 + 1 / 2}
 \ms{16} \hbox{and} \ms{8}
      |m^*_\ell(\xi)|
 \le c_2(n) |\xi|^{ - n / 2 + 3 / 2}.
 \qedhere
\]
\end{proof}

 We state in the next proposition a crucial weak type estimate for 
$\Mg_{K_\ell}$.

\begin{prp}%
\label{propL1}
Let $\ell \ge 1$. For all $f \in L^1(\R^n)$ and every $\lambda > 0$, one
has that
\[
     \bigl| 
      \bigl\{
       x \in \R^n : (\Mg_{K_\ell} f)(x) > \lambda 
      \bigr\} 
     \bigr|
 \le C(n) \, \frac {2^\ell} \lambda \|f\|_{L^1(\R^n)},
\]
where $C(n)$ is a constant independent of $\ell$ and $\lambda$.
\end{prp}

\begin{proof} 
We claim that it is enough to prove that for each $\ell \ge 1$, we have
\begin{equation}
      \bigl| K_\ell(x) \bigr|
 \le C(n) \, 
      \frac {2^\ell} {\bigl( 1 + |x| \bigr)^{n+1} \ns{16}} \ms{16} \up,
 \quad
 x \in \R^n.
 \label{schw}
\end{equation}
Indeed, since $(1 + |x|)^{-n-1}$ is radial, radially decreasing and
integrable, we will have for all $x \in \R^n$, as in~\eqref{MajOmega}, that
\[
     \sup_{r > 0} 
      \bigl|
       \bigl[ \Di{(K_\ell)}r * f \bigr](x) 
      \bigr|
 \le \widetilde{C}(n) \, 2^\ell \ms1 (\M f) (x).
\]
The result of Proposition~\ref{propL1} follows then from the weak estimate
in Theorem~\ref{TheoHL}, the standard maximal theorem. We now turn to the
proof of~\eqref{schw}. We want a bound for
$K_\ell = \varphi_\ell^\vee * \sigma$, for $\ell \ge 1$, where
$\sigma$ is the uniform probability measure on $S^{n-1}$ and
$
 \varphi_\ell^\vee = \Di {(\psi^\vee)} {2^{-\ell}}
$.
Since $\psi^\vee$ belongs to the Schwartz class, we can bound $|\psi^\vee|$
by a multiple $c_n \ms1 g$ of the radial and radially decreasing integrable
function $g(x) = (1 + |x|)^{-n-1}$. In order to bound $K_\ell$, we
shall prove that
\[
     c_n^{-1} |K_\ell(x)|
 \le ({\Di g {2^{-\ell}}} * \sigma)(x)
  =  \int_{S^{n-1}} {\Di g {2^{-\ell}}}(x - z) \, \d \sigma(z)
 \le C(n) 2^\ell (1 + |x|)^{-n-1}.
\]
This is easy when $|x| > 2$, because for each~$z$ in $S^{n-1}$, we have
then $|x - z| \ge |x| - 1 \ge |x| / 2$ and $1 + |x| \le 2 \ms1 |x|$.
Recalling ${\Di g {2^{-\ell}}} (y)  = 2^{n \ell} g(2^\ell y)$, we get
\begin{align*}
     G_\ell(x)
  &:= ( \Di g {2^{-\ell}}  * \sigma)(x) 
 \le \max_{z \in S^{n-1}} \Di g {2^{-\ell}}  (x - z)
 \le 2^{n \ell} \ms2 (1 + 2^\ell \ms1 |x| / 2)^{-n-1}
 \\
 &\le 2^{n \ell} 2^{- (\ell - 1)(n+1)} \ms2 |x|^{-n-1}
  =  2^{n + 1 - \ell} \ms2 |x|^{-n-1}
 \le 2^{2n+1} \ms2 (1 + |x|)^{-n-1},
\end{align*}
even better than required. Suppose now that $|x| \le 2$. It is enough to
prove that $G_\ell(x) \le C(n) \ms1 2^\ell$, since we have $1 + |x| \le 3$
in this second case, hence it will follow that
$C(n) \ms2 2^\ell \le [C(n) \ms1 3^{n+1}] \ms1 2^\ell (1 + |x|)^{-n-1}$.
For $y \in \R^n$, we write $y = (v, t)$ with $v \in \R^{n-1}$ and $t$
real. By the rotational invariance, we may restrict the study to 
$x = (0, s)$, $s \ge 0$. We write each $z \in S^{n-1}$ as $z = (v, t)$, and
thus $x - z = (-v, s - t)$. Let $\pi_0$ be the orthogonal projection of
$\R^n$ onto the hyperplane of vectors $(w, 0)$, $w \in \R^{n-1}$. Since
$\Di g {2^{-\ell}}$ is radial and radially decreasing, we see that
$\Di g {2^{-\ell}} (x - z) \le \Di g {2^{-\ell}} (\pi_0(x - z))
 = \Di g {2^{-\ell}} (-v, 0)$. This yields
\begin{align*}
     G_\ell(x)
   = G_\ell(0, s)
  &=  \int_{S^{n-1}} \Di g {2^{-\ell}} (x - z)
       \, \d \sigma(z)
 \le  \int_{S^{n-1}} \Di g {2^{-\ell}} (\pi_0(x - z))
       \, \d \sigma(z)
 \\
  &=  \int_{ \R^{n-1} } \Di g {2^{-\ell}}  (-v, 0)
      \, \d \nu(v),
\end{align*}
where $\nu$ is the projection on $\R^{n-1}$ of the probability measure
$\sigma$. We have that
\[
   \d \nu(v) 
 = \frac 2 {s_{n-1}} \ms2
    \frac { \gr 1_{ \{|v| < 1\} } } {\sqrt {1 - |v|^2} } \, \d v
 = C(n) \frac { \gr 1_{ \{|v| < 1\} } } {\sqrt {1 - |v|^2} } \, \d v,
\]
where $s_{n-1}$ is the measure of $S^{n-1}$ recalled in~\eqref{OmegaN}. We
cut the integral with respect to $\nu$ into two parts, according to 
$|v| < 1/2$ or not. In the part~$E_1$ corresponding to $|v| < 1/2$, we have
$1 - |v|^2 \ge 3/4$, hence
\[
     E_1
 \le \sqrt{ \frac 4 3} \ms2 C(n) 
      \int_{|v| < 1/2} {\Di g {2^{-\ell}}}(v, 0) \, \d v
 \le 2 \ms1 C(n) \int_{\R^{n-1}} \Di g {2^{-\ell}}  (v, 0) \, \d v.
\]
We are integrating on $\R^{n-1}$ the function $\Di g {2^{-\ell}}$ that is
normalized for a change of variable in dimension~$n$. This implies that
\[
     E_1
 \le 2 \ms1 C(n) 2^\ell 2^{(n-1)\ell}
      \int_{\R^{n-1}} g(2^\ell v, 0) \, \d v
  =  2 \ms1 C(n) 2^\ell \int_{\R^{n-1}} g(u, 0) \, \d u,
\]
a bound of the expected form. In the second case, we have $|v| > 1/2$ and
\[
     {\Di g {2^{-\ell}}}(v, 0)
  =  2^{n \ell} (1 + 2^\ell |v|)^{-n-1}
 \le 2^{n \ell} 2^{ -(\ell - 1)(n + 1)}
 \le 2^n.
\]
It follows that the integral $E_2$ limited to $|v| > 1/2$, with respect to
the probability measure $\nu$, is bounded by a function of $n$.
\end{proof}

\subsubsection{Conclusion%
\label{PoTsph}}%

\begin{proof}[Proof of Theorem \ref{sph}]
Thanks to the results of the previous subsection, the proof is easy. Using
the Marcinkiewicz theorem (see Zygmund~\cite[Chap.~XII]{ZygmundTS}, 
or~\cite[Theorem~5.60]{PisierMart}), we shall interpolate between the weak
type $(1, 1)$ and the strong type $(2, 2)$. We apply
Proposition~\ref{propL2}, Proposition~\ref{propL1} in $\R^N$ and
interpolation with parameter $\theta = 2 - 2/p$, where $1 < p \le 2$. For
all $\ell \ge 1$ and all $f \in L^p(\R^N)$, since the chosen interpolation
parameter $\theta$ verifies $(1 - \theta) / 1 + \theta / 2 = 1 / p$, we
have
\[
     \bigl\| \Mg_{K_\ell} f \bigr\|_{L^p(\R^N)}
 \le \kappa(1, 2, p) \ms1 C(N) \ms1 \bigl( 2^{\ell} \bigr)^{ -1 + 2 / p }
       {\bigl( 2^{ - \ell (N - 2) / 2}} \bigr)^{ 2 - 2 / p }
        \|f\|_{L^p(\R^N)},
\]
where $\kappa(1, 2, p)$ is independent of $N$ and $\ell$. We have thus
obtained that
\[
     \bigl\| \Mg_{K_\ell} f \bigr\|_{L^p(\R^N)}
 \le C'(N, p) \ms2 
      2^{ \ell \ms2 [ N / p - (N - 1) ] } \ms1 \|f\|_{L^p(\R^N)}.
\]
For $p > N / (N - 1)$, the series 
$
 \sum_{\ell \ge 1} 2^{ \ell \ms3 [ N / p - (N-1) ] }
$ 
converges. Moreover, we know by Proposition~\ref{map} that $\Mg_{K_0}$
maps $L^p(\R^N)$ to itself for all $1 < p < +\infty$. Therefore, in view 
of~\eqref{ptw}, we obtain that $\ca M$ is bounded on $L^p(\R^N)$ for every 
real number~$p$ such that $N / (N-1) < p \le 2$. For $p > 2$, we proceed by
interpolation between the $L^2(\R^N)$ case and the trivial $L^\infty(\R^N)$
case.
\end{proof}

\section{The $L^2$ result of Bourgain%
\label{ArtiBour}}

\noindent
In an article published in 1986, Bourgain has generalized the $L^2$ case of
the Stein result presented in Section~\ref{SteinsResults}. This $L^2$ case
for Euclidean balls only required Proposition~\ref{propL2} and the 
\og method of rotations\fg. The maximal operator~$\M_C$ associated to a
symmetric convex body $C$ was defined in~\eqref{OpMaxi}.

\begin{thm}[Bourgain~\cite{BourgainL2}]%
\label{TheoBour}
There exists a universal constant $\kappa_2$ such that for every integer 
$n \ge 1$ and every symmetric convex body $C \subset \R^n$, one has
\[
 \forall f \in L^2(\R^n),
 \ms{18}
     \| \M_C f \|_{L^2(\R^n)}
 \le \kappa_2 \ms2 \| f \|_{L^2(\R^n)}.
\]
\end{thm}
 
 The rest of this section is devoted to the proof of this maximal theorem,
together with the description of the general framework concerning maximal
functions associated to convex sets. We shall in particular establish some
geometric inequalities for log-concave distributions that will be applied in
the subsequent sections.

\subsection{The general setting%
\label{TheSetting}}

\noindent
Let $C$ be a symmetric convex body in~$\R^n$. Throughout these Notes, we
let $K_C$\label{KdeC} 
be the density of the uniform probability measure 
$\mu_C$\label{MusubC} 
on $C$, and $m_C$\label{MsubC} 
denotes the Fourier transform of $K_C$ or of $\mu_C$. Hence, we have
\def\vmu{\vphantom{\hbox{\ninerm b}}}
\[
 K_C(x) = \frac 1 {|C|} \ms2 \gr 1_C(x),
 \ms{16}
 \d \mu_C(x) = K_C(x) \, \d x,
 \ms{16}
   m_C(\xi) = \widehat{K_C}(\xi) 
 = \widehat{ \vmu \mu_C}(\xi),
\]
for all $x, \xi \in \R^n$. Notice that 
$K_{\lambda C} = \Di {(K_C)} \lambda$ and
$m_{\lambda C}(\xi) = m_C(\lambda \xi)$ for each $\lambda > 0$ and
$\xi \in \R^n$. We already know that the maximal operator $\M_C$ acts
boundedly on $L^p(\R^n)$, $1 < p \le +\infty$, but the bounds we have so
far depend on~$n$. 

\begin{smal}
\noindent
This $L^p$ result comes from the weak type estimate~\eqref{VitaliEstim}
given by the Vitali covering lemma. Except for the value of the constant,
it is clear that this weak type $(1, 1)$ result for $\M_C$ is optimal, as
we can see by taking for $f$ the indicator $\gr 1_C$ of the symmetric
convex body $C \subset \R^n$. Let $C$ have volume~$1$, so that 
$\|f\|_1 = 1$. For any given $r > 0$ and $x \in r C$, we see that 
$x + (r+1) C$ contains $C$, therefore
\[
     (\M_C f)(x)
 \ge |(r+1) C|^{-1} \int_{x + (r+1) C} \gr 1_C(y) \, \d y
  =  |(r + 1) C|^{-1}
  =  (r + 1)^{-n}
\]
and
$
 \{ \M_C f \ge (r + 1)^{-n} \} \supset r C
$.
Every value $c$ in the interval $(0, 2^{-n}]$ can be written as 
$c = (r + 1)^{-n}$ for some $r \ge 1$, hence
\[
 \forall c \in (0, 2^{-n}],
 \ms{18}
     \bigl| \{ \M_C f \ge c \} \bigr|
 \ge |r C|
  =  \frac {(r + 1)^{-n} \ns{14}} c \ms{12} r^n
 \ge \frac {2^{-n} \ns{14}} c \ms2 \up.
\]
The maximal function $\M_C \gr 1_C$ is not integrable. It belongs to the
space $L^{1, \infty}(\R^n)$, the so-called weak-$L^1$ space, and nothing
better: any bounded radial and radially decreasing function belonging to
$L^{1, \infty}(\R^n)$ is smaller than a multiple of $\M_C \gr 1_C$.
\end{smal}

\noindent 
The maximal function $\M_C f$ is given by
$
 \M_C f = \sup_{t > 0} \ms3 {\Di {(K_C)} t}  * |f|
$,
where ${\Di {(K_C)} t}$ is the dilate from~\eqref{Dilata}. More generally,
let $K$ be a probability density on $\R^n$, resp. an integrable kernel $K$.
We define the maximal function $\M_K$ or $\Mg_K$ by
\[
 \M_K f = \sup_{t > 0} \ms3 {\Di K t}  * |f|,
 \ms{18} \hbox{resp.} \ms{10}
 \Mg_K f = \sup_{t > 0} \ms3 \bigl| {\Di K t}  * f \bigr|.
\]
If $A$ is linear and bijective on~$\R^n$, we can see that the maximal
operators $\M_C$ and $\M_{A C}$ have the same norm on $L^p(\R^n)$. For a
function~$f$ on~$\R^n$ we define $\Di f A $ by
\[
 \forall x \in \R^n,
 \ms{16}
 \Di f A (x) = |\det A|^{-1} f(A^{-1} x).
\]
We have $\Di {|f|} A = \bigl| \Di f A \bigr|$, 
$\Di {(\sup_i f_i)} A = \sup_i \Di {(f_i)} A$, and 
$\Di {(f * g)} A = \Di f A * \Di g A$ since
\[
   \int_{\R^n} |\det A|^{-2} f(A^{-1} (x - y)) \ms1
      g (A^{-1} y) \, \d y
 = |\det A|^{-1} \ms2 
    \int_{\R^n} f(A^{-1} x - z) \ms1 g(z) \, \d z.
\]
It is clear that $\Di {(\Di f A)} t = \Di f {t A}
 = \Di {(\Di f t)} A$. If $S_A$ is the mapping $f \mapsto \Di f A $, then 
$S_{A, p} := |\det A|^{1/q} S_A$, with $q$ conjugate to~$p$, is an onto
isometry of $L^p(\R^n)$. 
\dumou

 The density $K_{A C}$ is equal to $\Di {(K_C)} A$. For every integrable
kernel~$K$ on $\R^n$,\label{MKetMKbis} 
we see now that $K$ and~$\Di K A$ produce maximal functions that are
conjugate by the isometry $S_{A, p}$ of $L^p(\R^n)$, and have therefore the
same norm on $L^p(\R^n)$. We have
\begin{align*}
    \Mg_{\Di K A} \Di f A
 &= \sup_{t > 0} \ms1 \bigl| {\Di {(\Di K A)} t} * \Di f A \bigr|
  = \sup_{t > 0} \ms1 \bigl| {\Di {(\Di K t)} A} * \Di f A \bigr|
 \\
 &= \sup_{t > 0} \ms1 \bigl| \Di {(\Di K t * f)} A \bigr|
  = {\Di {(\Mg_K f)} A}.
\end{align*}
It follows that
$
   \Mg_{\Di K A} \circ S_{A, p}
 = S_{A, p} \circ \Mg_K
$.
This remark allows us to assume that $C$ is in \emph{isotropic
position\/}: one says that a symmetric convex body~$C$ is in isotropic
position\label{IsoPosit} 
if the quadratic form
\[
 Q_C : \xi \mapsto
 Q_C(\xi) = \int_C (\xi \ps x)^2 \, \d x,
 \quad
 \xi \in \R^n,
\]
is a multiple of the square $\xi \mapsto |\xi|^2$ of the Euclidean norm
on~$\R^n$. Since $Q_C$ is positive definite for every symmetric convex body
$C$, we can bring it to the form 
$\xi \mapsto \lambda |\xi|^2$, $\lambda > 0$, by a suitable linear change
of coordinates. For an isotropic symmetric convex set $C_0$ of volume~$1$,
one defines the \emph{isotropy constant} $L(C_0)$ by
\[
 \label{IsotCons}
   L(C_0)^2
 = \int_{C_0} (\gr e_1 \ps x)^2 \, \d x,
 \ms{16} \hbox{and one has then} \ms{ 8}
   \int_{C_0} (\xi \ps x)^2 \, \d x
 = L(C_0)^2 \ms1 |\xi|^2
\]
for every $\xi \in \R^n$. For $C_*$ isotropic of the form $C_* = r C_0$, 
$r > 0$, we get $|C_*| = r^n$ and for every $\xi \in \R^n$, we have
\begin{align}
   \int_{\R^n} (\xi \ps x)^2 K_{C_*}(x) \, \d x
 &= \frac 1 {|C_*|} \int_{C_*} (\xi \ps x)^2 \, \d x
 = r^{-n} \ms1 \int_{C_0} (\xi \ps r u)^2  r^n \, \d u
 \label{Isotro}
 \\
 &= r^2 L(C_0)^2 \ms1 |\xi|^2
 = |C_*|^{2/n} L(C_0)^2 \ms1 |\xi|^2.
 \notag
\end{align}
Let $A$ linear and invertible put $C_*$ in another isotropic position 
$A C_*$, so that $Q_{A C_*}(\xi) = \lambda |\xi|^2$ for some 
$\lambda > 0$ and all $\xi \in \R^n$. Letting 
$\nu = \lambda |A C_*|^{-1}$ we get
\[
   \nu |\xi|^2
 = \int_{\R^n} (\xi \ps y)^2 K_{A C_*}(y) \, \d y 
 = \int_{\R^n} (\xi \ps A x)^2 K_{C_*}(x) \, \d x
 = |C_*|^{2/n} L(C_0)^2 \ms1 |A^T \xi|^2,
\]
hence $A$ is a multiple $\rho \ms1 U$ of an isometry $U$, 
$|\det A| = \rho^n$ and 
$\nu = |C_*|^{2/n} L(C_0)^2 \rho^2
 = |A C_*|^{2/n} L(C_0)^2$, 
thus 
$
 |A C_*|^{-2/n} \int_{\R^n} (\theta \ps y)^2 K_{A C_*}(y) \, \d y
 = L(C_0)^2
$ for every $\theta \in S^{n-1}$. 
\dumou

 When $C$ is isotropic, it follows that $L(C):= L(C_0)$ is well defined by 
\begin{equation}
    L(C)^2
 = |C|^{-2/n} \ns5
    \int_{\R^n} (\theta \ps x)^2 K_{C}(x) \, \d x
 = |C|^{-1 - 2/n} \ns5 
    \int_C (\theta \ps x)^2 \, \d x,
 \quad \theta \in S^{n-1}.
 \label{IsotroDef}
\end{equation}
A well-known open question (see~\cite{MiPaj}) is to decide whether the
isotropy constant is bounded above by a universal constant valid for all
symmetric convex bodies and every $n$. The best upper bound that is known
so far, due to Klartag~\cite{Klar} improving Bourgain~\cite{BoIsotro}, is 
$L(C) \le \kappa \ms1 n^{1/4}$ in dimension~$n$. It is known that
$L(C)$ is bounded below by a universal constant. However, neither this
known fact nor the unsolved problem will interfere with the treatment of
the maximal function problem.
\dumou

 Clearly, $K_C$ and $\Di {(K_C)} \lambda$ have the same maximal function
for every $\lambda > 0$, so we can choose any multiple among isotropic
positions of~$C$. Here, we do not follow Bourgain~\cite{BourgainL2} who
chooses the isotropic position of volume~$1$, we prefer the isotropic
position such that $\mu_C$ has covariance matrix $\I_n$. We thus assume
that 
\begin{equation}
 \forall \theta \in S^{n-1},
 \ms{16}
   \int_{\R^n} (\theta \ps x)^2 \, \d \mu_C(x)
 = \frac 1 {|C|} \int_C (\theta \ps x)^2 \, \d x
 = 1.
 \label{NormaVari}
\end{equation}
This means that the one-dimensional marginals of $\mu_C$, images of $\mu_C$
by $x \mapsto \theta \ps x$ for $\theta \in S^{n-1}$, have all
variance~$1$. We shall say in this case that $C$ is isotropic and
\emph{normalized by variance}.\label{NormVari} 
We have then in addition that
\[
 \int_C |x|^2 \, \d x = n \ms1 |C|
 \ms{18} \hbox{and} \ms{12}
 |C| = L(C)^{-n}.
\]
If we look for a (centrally symmetric) Euclidean ball in $\R^n$ normalized
by variance, its radius $r = r_{n, V}$ must therefore satisfy 
$\int_0^r t^{n+1} s_{n-1} \, \d t  = n \int_0^r t^{n-1} s_{n-1} \, \d t$,
giving
\begin{equation}
 r_{n, V} = \sqrt { n + 2}.
 \label{RnV}
\end{equation}
\dumou

 In the same way, we can bring to isotropy a symmetric probability density
$K$ on~$\R^n$, \textit{i.e}, such that $K(-x) = K(x)$ for $x \in \R^n$, by
a linear change to $\Di K A$ for some $A$ linear and invertible. When $K$
is isotropic, there exists $\sigma > 0$ such that
\[
   \int_{\R^n} (\xi \ps x)^2 \ms1 K(x) \, \d x 
 = \sigma^2 |\xi|^2,
 \quad
 \xi \in \R^n,
\]
which means that all one-dimensional marginals of $K$ have the same 
variance~$\sigma^2$. We shall then say for brevity that $K$ is
\emph{isotropic with variance $\sigma^2$}. The dilated density 
${\Di K {1 / \sigma}} : x \mapsto \sigma^n K(\sigma x)$ is normalized by
variance. For example, the standard Gaussian $\gamma_n$
in~\eqref{LoiNZeroId} is normalized by variance. For the study of maximal
functions, we can always assume that~$K$ is normalized by variance.

\subsection{On the volume of sections%
\label{VoluSections}}

\noindent
We have seen in~\eqref{VarphiTheta} that the Fourier transform $m$ of a
kernel $K \in L^1(\R^n)$ can be expressed as
\[
   m(u \xi)
 = \int_\R \varphi_{\theta, K} (s) \e^{ - 2 \ii \pi s u \ms1 |\xi|} \, \d s,
 \quad u \in \R, \ms5 \xi \in \R^n \setminus \{0\},
\]
where one has set $\theta = |\xi|^{-1} \xi$ and
$  \varphi_\theta(s)
 = \varphi_{\theta, K} (s)
 = \int_{\theta^\perp} K(y + s \ms1 \theta) \, \d^{n-1} y
$ for every $s \in \R$.
When $K$ is the kernel $K_C$ corresponding to a symmetric convex body~$C$,
the function $\varphi_\theta$ is the \og normalized\fge function of
$(n - 1)$-dimensional volumes of hyperplane sections parallel to
$\theta^\perp$, defined by
\[
   \varphi_{\theta, C}(s)
 = \int_{\theta^\perp} K_C(y + s \ms1 \theta) \, \d^{n-1} y
 = \frac { \bigl| C \cap (\theta^\perp + s \theta) \bigr|_{n-1}} 
         {|C|_n} \up.
\]
We know by the Brunn--Minkowski inequality~\cite[Theorem~4.1]{GardnerBM}
that $\varphi_{\theta, C}$ is log-concave on $\R$. Indeed, a form of this
inequality states that
\[
 \label{BrunnMin}
     |(1 - \lambda) A + \lambda B|
 \ge |A|^{1 - \lambda} |B|^\lambda
\]
whenever $A, B$ are compact subsets of $\R^n$ and $\lambda \in [0, 1]$.
Recall that a function $K \ge 0$ on $\R^n$ is \emph{log-concave} 
when\label{LogConca}
\[
     K \bigl( (1 - \alpha) x_0 + \alpha x_1 \bigr)
 \ge K(x_0)^{1 - \alpha} K(x_1)^\alpha,
 \quad
 x_0, x_1 \in \R^n, \ms2 \alpha \in [0, 1],
\]
in other words, when $\log K$ is concave on the convex set $\{K > 0\}$.
\dumou

 More generally than Brunn--Minkowski, the Pr\'ekopa--Leindler
inequality~\cite[Theorem~7.1]{GardnerBM} implies that the function
$\varphi_{\theta, K}$ defined in~\eqref{VarphiTheta} is a log-concave
probability density on the real line if~$K$ is a log-concave probability
density on~$\R^n$. The statement of Pr\'ekopa--Leindler is as follows: if
$\alpha$ is in $(0, 1)$, if $f_0, f_1, f_\alpha$ nonnegative and integrable
Borel functions on~$\R^n$ are such that\label{PrekoLei}
\[
     f_\alpha \bigl( (1 - \alpha) x_0 + \alpha x_1 \bigr)
 \ge f_0(x_0)^{1 - \alpha} \ms2 f_1(x_1)^{\alpha}
\]
for all $x_0, x_1 \in \R^n$, then
\[
 \int_{\R^n} f_\alpha(x) \, \d x
 \ge \Bigl(\int_{\R^n} f_0(x) \, \d x \Bigr)^{1 - \alpha}
     \Bigl(\int_{\R^n} f_1(x) \, \d x \Bigr)^{\alpha}.
\]
Given $\theta \in S^{n-1}$, $s_0, s_1$ real and letting
$f_j(y) = K(y + s_j \theta)$ for $y \in \theta^\perp$ and $j = 0, 1$, 
$s_\alpha = (1 - \alpha) s_0 + \alpha s_1$ and
$f_\alpha(y) = K(y + s_\alpha \theta)$, we obtain that 
$\varphi_{\theta, K}$ is log-concave by applying Pr\'ekopa--Leindler on
$\theta^\perp \simeq \R^{n-1}$ to these functions $f_0$, $f_1$ 
and~$f_\alpha$. Similarly, one shows that convolutions of log-concave
densities are log-concave. Without more effort, Bourgain's proof also gives
the following theorem.
\dumou

\begin{thm}\label{LogConcBourg}
There exists a constant $\kappa_2 < 140$ such that for every integer 
$n \ge 1$ and every symmetric log-concave probability density $K$ 
on\/~$\R^n$, one has
\[
 \forall f \in L^2(\R^n),
 \ms{18}
     \| \M_K f \|_{L^2(\R^n)}
 \le \kappa_2 \ms2 \| f \|_{L^2(\R^n)}.
\]
\end{thm}
\dumou

 We turn to the proof of the main inequalities about log-concave functions,
which will be used throughout our Notes. We introduce the \emph{right
maximal function} $f^*_r$\label{FStar} 
of a locally integrable function $f$ on an interval $[\tau, +\infty)$ of
the line by setting
\begin{equation}
   f^*_r(x) 
 = \sup_{t > 0} \ms1 \frac 1 t 
    \int_x^{x+t} |f(s)| \, \d s,
 \quad
 x \ge \tau.
 \label{RightMax}
\end{equation}
One sees that $f^*_r \le f^* \le 2 \ms1 \M f$, where $f^*$ is the
uncentered maximal function from~\eqref{UncentOp}, and 
$f^*_r(x) \ge |f(x)|$ at each 
\emph{Lebesgue point} $x$ of $f$,\label{LebesPoint} 
hence almost everywhere. When $\psi$ is nonnegative, integrable and
decreasing on $[x, + \infty)$, then
\begin{equation}
     \int_x^{+\infty} |f(s)| \ms2 \psi(s) \, \d s 
 \le \Bigl( \int_x^{+\infty} \psi(s) \, \d s \Bigr) f^*_r(x).
 \label{RightMaxPsi}
\end{equation}
One can get~\eqref{RightMaxPsi} as in~\eqref{MajOmega}, by approximating
$\psi$ by a combination of functions $t_k^{-1} \gr 1_{ [x, x + t_k] }$. We
can also define in a similar way a \emph{left} maximal function $f^*_\ell$.

\begin{lem}%
\label{EstimaLogConcNew}
Let $\varphi$ be an integrable log-concave function on an 
interval\/~$[\tau, +\infty)$, let $p$ belong to\/ $(0, +\infty)$ and let
\[
 S_0(\tau) = \int_\tau^{+\infty} \varphi(s) \, \d s,
 \ms{18}
 S_p(\tau) = \int_\tau^{+\infty} (s - \tau)^p \ms1 \varphi(s) \, \d s.
\]
Then $S_p(\tau)$ is finite. Furthermore, assuming $S_p(\tau) > 0$, we have
\begin{equation}
     \varphi(\tau)^p
 \le \frac {\Gamma(p+1) \ms1 S_0(\tau)^{p+1} \ns5} {S_p(\tau)} \ms1 \up,
 \ms{20}
     \max_{s \ge \tau} \varphi(s)^p
 \ge \varphi^*_r(\tau)^p
 \ge \frac {S_0(\tau)^{p+1} \ns5}
           {(p + 1) \ms1 S_p(\tau)} \ms1 \up.
 \label{PropLogConcNew}
\end{equation}
\end{lem}

\begin{proof}
We have $\varphi \ge 0$ by definition of log-concavity. We assume
$S_p(\tau) > 0$, hence $S_0(\tau) > 0$. We may suppose $\tau = 0$
by translating and $S_0 := S_0(0) = 1$ by homogeneity. We
begin with the left-hand inequality in~\eqref{PropLogConcNew}, assuming 
$a := \varphi(0) > 0$. Consider the log-affine probability density 
$\psi(s) = a \e^{- a s}$ on~$[0, +\infty)$, chosen so that 
$\psi(0) = \varphi(0)$. By log-concavity, the set 
$I = \{\varphi \ge \psi\}$ is an interval, 
such that $0 \in I \subset [0, +\infty)$.
Since~$\varphi$ and~$\psi$ both have integral~$1$ on 
$[0, +\infty)$, the interval $I$ is not reduced to $\{0\}$. If
$I = [0, +\infty)$, the densities are equal and
\[
    S_p
 := \int_0^{+\infty} s^p \varphi(s) \, \d s 
  = \int_0^{+\infty} s^p \psi(s) \, \d s 
  = \frac 1 {a^p \ns5} \ms3 
     \int_0^{+\infty} (a s)^p \e^{- a s} \ms1 a \, \d s
  = \frac {\Gamma(p+1)} {a^p} \up.
\] 
Otherwise, the interval $I$ is bounded, let $s_0 := \sup I > 0$. We have
$\psi \le \varphi$ on~$[0, s_0)$ and $\varphi(s) < \psi(s)$ when $s > s_0$,
implying that $S_p(0)$ is finite. The antiderivative~$F$ of~$\varphi - \psi$
vanishing at $0$ is first increasing, then decreasing on $[0, +\infty)$,
and tends to~$0$ at infinity because $\varphi$ and $\psi$ have equal
integrals. It follows that $F$ is nonnegative on $[0, +\infty)$. Recalling
that $0 \le \varphi(s) < \psi(s)$ at infinity, we know that $|F(s)|$ is
exponentially small at infinity, and integrating by parts we obtain
\[
   \int_0^{+\infty} s^p \bigl( \varphi(s) - \psi(s) \bigr) \, \d s
 = - p \int_0^{+\infty} s^{p-1} \ms1 F(s) \, \d s \le 0.
\]
One concludes the first part by writing
\[   
     S_p
  =  \int_0^{+\infty} s^p \varphi(s) \, \d s 
 \le \int_0^{+\infty} s^p \psi(s) \, \d s 
  = \frac {\Gamma(p+1)} {a^p} \up.
\]

 For the right-hand inequality in~\eqref{PropLogConcNew}, we let 
$b = \varphi^*_r(0) > 0$ and consider the probability density 
$\psi(s) = b \ms1 \gr 1_{ [0, 1 / b] } (s)$ on $[0, +\infty)$. Let $F$ be
the antiderivative of~$\varphi - \psi$ vanishing at~$0$. When 
$0 < x \le 1/b$ we have by definition of $\varphi^*_r(0)$ that
\[
     \frac {F(x)} x
  =  \frac 1 x \int_0^x (\varphi(s) - \psi(s)) \, \d s
  =  \Bigl( \frac 1 x \int_0^x \varphi(s) \, \d s \Bigr) - b
 \le 0.
\]
We see that $\psi(x) = 0 \le \varphi(x)$ when $x \ge 1 / b$. It follows
that the function $F$ is $\le 0$ on $[0, 1 / b]$, then increasing on
$[1 / b, +\infty)$, tends to~$0$ at infinity, thus $F$ is~$\le 0$ on 
the half-line $[0, +\infty)$. Arguing as before, we have consequently
\[
     S_p
  =  \int_0^{+\infty} s^p \varphi(s) \, \d s 
 \ge b \int_0^{1 / b} s^p \, \d s
  =  \frac 1 {(p+1) \ms2 b^p} \up.
\]
\end{proof}

 For every $\theta \in S^{n-1}$, the function $\varphi_{\theta, C}$
associated to a symmetric convex set $C$ is even, log-concave and has
integral~$1$ by definition. We shall thus be in a position to apply to it
the following Corollary~\ref{EstimaLogConcC}.

\begin{cor}%
\label{EstimaLogConcC}
Suppose that $\varphi$ is a symmetric log-concave probability density
on\/~$\R$ and let~$
 \sigma^2 := \int_\R s^2 \varphi(s) \, \d s
$.
One has that
\[
     \frac 1 {12 \ms2 \sigma^2 \ns5} \ms3
 \le \varphi(0)^2 = \max_{s \in \R} \varphi(s)^2
 \le \frac 1 {2 \ms2 \sigma^2 \ns5} \ms2 \up.
\]
\end{cor}

\begin{proof}
Since $\varphi$ is even and log-concave, we have
$\varphi(0) = \max_{s \in \R} \varphi(s)$. We apply
Lemma~\ref{EstimaLogConcNew} with $p = 2$, $\tau = 0$, and observe that
$S_0(0) = 1/2$, $S_2(0) = \sigma^2 / 2$.
\end{proof}

 The preceding result is sharp, as one sees with the two examples
\begin{equation}
   \varphi_0(s)
 = \frac 1 {\sqrt 2} \ms3 \e^{ - \sqrt 2 |s|},
 \ms{16}
   \varphi_1(s)
 = \frac 1 {2 \sqrt 3} \ms4 \gr 1_{ [- \sqrt 3, \sqrt 3]} (s),
 \quad s \in \R.
 \label{Varphis}
\end{equation}
The next corollary is not very sharp, but easy to deduce from
Lemma~\ref{EstimaLogConcNew}. When the function~$\varphi \ge 0$ is defined
on the line and $p \in [0, +\infty)$, we set
\[
 S_p^+(\tau) = \int_\tau^{+\infty} (s - \tau)^p \ms1 \varphi(s) \, \d s,
 \ms{18}
 S_p^-(\tau) = \int_{-\infty}^\tau |s - \tau|^p \ms1 \varphi(s) \, \d s.
\]
 
\begin{cor}%
\label{EstimaLogConcCB}
Let $\varphi$ be a \emph{centered} log-concave probability density
on\/~$\R$ and let~$
 \sigma^2 := \int_\R s^2 \varphi(s) \, \d s
$.
We have that
\[
     \frac 1 {24 \ms2 \sigma^2 \ns5} \ms3
 \le \frac {\varphi^*_\ell(0)^2 + \varphi^*_r(0)^2 \ns4 } 2 \ms3
 \le \max_{s \in \R} \varphi(s)^2
 \le \frac 4 {\ms2 \sigma^2 \ns5} \ms3 \up.
\]
\end{cor}

\begin{proof}
We begin with the rightmost inequality. Let us fix $\tau$ real. Since
$\varphi$ is a centered probability density, one has that
\[   
     S_2^+(\tau) + S_2^-(\tau)
  =  \int_\R (s - \tau)^2 \varphi(s) \, \d s
  =  \sigma^2 + \tau^2
 \ge \sigma^2,
 \quad \tau \in \R.
\]
Up to a symmetry around $\tau$, possibly replacing the function~$\varphi$
by $s \mapsto \varphi(2 \tau - s)$, we may assume that
$    S_2^+(\tau)
 \ge \sigma^2 / 2
$.
We have $S_0^+(\tau) = \int_\tau^{+\infty} \varphi(s) \, \d s \le 1$ 
since $\varphi$ is a probability density on $\R$, thus by
Lemma~\ref{EstimaLogConcNew} with $p = 2$ we get
\[
     \varphi(\tau)^2
 \le \frac {2 \ms2 S_0^+(\tau)^3 \ns4} {S_2^+(\tau)}
 \le \frac 4 {\sigma^2 \ns7 } \ms5 \up. 
\]
Since $\tau$ is arbitrary, we obtain the right-hand inequality. Let us pass
to the other inequality. By Lemma~\ref{EstimaLogConcNew} with $p = 2$ on
the intervals $(0, +\infty)$ and $(-\infty, 0)$, we conclude using
$S_2^{\pm}(0) \le \sigma^2$ and $S_0^+(0) + S_0^-(0) = 1$ that
\[
     \varphi^*_r(0)^2 + \varphi^*_\ell(0)^2
 \ge \frac {S_0^+(0)^{3} \ns5} {3 \ms1 S_2^+(0)}
      + \frac {S_0^-(0)^{3} \ns5} {3 \ms1 S_2^-(0)}
 \ge \frac {S_0^+(0)^{3} + S_0^-(0)^{3} \ns5}
           {3 \ms1 \sigma^2}
 \ge \frac 1 {12 \ms1 \sigma^2 \ns7} \ms5 \up.
\]
\end{proof}

\begin{lem}%
\label{LExpoDecay}
Let $\varphi$ be a symmetric log-concave probability density on\/~$\R$, with
variance $\sigma^2$. The function $\varphi$ decays exponentially at
infinity, with a rate depending on its variance and satisfying
\[
 \forall s \in \R,
 \ms{14}
       \sigma \varphi(\sigma s) 
 \le \min( 2 \ms1 \e^{- |s| / 2}, 11 \ms2 \e^{- |s|} ).
\]
\end{lem}

\begin{proof}
Without loss of generality, we may assume that $\sigma = 1$. It follows
then from Corollary~\ref{EstimaLogConcC} that 
$1 / (2 \sqrt 3) \le a := \varphi(0) \le 1 / \sqrt 2$. Consider the
log-affine function $\psi_\beta(s) = a \e^{ - \beta s}$ on $[0, +\infty)$,
with $\beta > 0$, satisfying $\psi_\beta(0) = \varphi(0)$. If we have 
$\varphi(\tau_0) \le \psi_\beta(\tau_0)$ for some $\tau_0 > 0$, it implies
by log-concavity that 
$\varphi(s) \le \psi_\beta(s) \le \e^{ - \beta s} / \sqrt 2$ for 
$s \ge \tau_0$, and in order to obtain a bound for $\varphi$ everywhere, we
can apply for the values $0 \le s \le \tau_0$ the obvious inequalities
\[ 
     \varphi(s) 
 \le \varphi(0)
  =  a 
 \le a \e^{\beta (\tau_0 - s)}
 \le (\e^{\beta \tau_0} / \sqrt 2) \e^{- \beta s}.
\]
\dumou\noindent
For any $\tau_0 > 0$, we obtain since $\varphi$ is even that
\begin{equation}
 \varphi(\tau_0) \le \psi_\beta(\tau_0)
 \ms5 \Rightarrow \ms5
     \varphi(s) 
  =  \varphi(|s|)
 \le \e^{\beta \tau_0 - \ln \sqrt 2} \ms1 \e^{- \beta |s|},
 \quad
 s \in \R.
 \label{Tau0}
\end{equation}

 On the other hand, if $\varphi(s) > \psi_\beta(s)$ for every
$s \in (0, \tau]$, then
\begin{align*}
     1/2
  &=  \int_0^{+\infty} s^2 \varphi(s) \, \d s
  >  \int_0^\tau s^2 \psi_\beta(s) \, \d s
  =  \frac a {\beta^3} \int_0^{\beta \tau} u^2 \e^{-u} \, \d u
 \\
  &=  \frac a {\beta^3} 
       \Bigl[ - \e^{-u} (u^2 + 2 u + 2) \Bigr]_{u=0}^{\beta \tau}
 \ge \frac 1 {2 \sqrt 3 \ms1 \beta^3}
       \bigl( 2 
         - \e^{- \beta \tau} (\beta^2 \tau^2 + 2 \beta \tau + 2) 
       \bigr).
\end{align*}
\dumou\noindent
Equivalently, when $\varphi(s) > \psi_\beta(s)$ for every 
$s \in (0, \tau]$, we get that
\begin{equation}
     \e^{- \beta \tau} 
          \bigl( \beta^2 \tau^2 + 2 \beta \tau + 2 \bigr)     
  >  2 - \sqrt 3 \ms2 \beta^3.
 \label{Tau}
\end{equation}
\dumou

 Suppose that $\beta^3 < 2 / \sqrt 3$. Then~\eqref{Tau} cannot be true if
$\tau$ is large. For every such~$\beta$, there exists $\tau_0 > 0$ such
that $\varphi(\tau_0) \le \psi_\beta(\tau_0)$ and by~\eqref{Tau0}, there is
a constant $c(\beta)$ such that $\varphi(s) \le c(\beta) \e^{- \beta |s|}$
on the line. For numerical purposes, it is more convenient to express this
as follows. If $0 < \sqrt 3 \ms2 \beta^3 < 2$ and if
\begin{equation}
     \e^{- x} (x^2 + 2 x + 2) 
 \le 2 - \sqrt 3 \ms2 \beta^3,
 \label{Impossible}
\end{equation}
\dumou\noindent
then $x > 0$, and letting $x_0(\beta) = x$, we know that
$\varphi(s) \le \psi_\beta(|s|) \le \e^{ - \beta \ms1 |s|} / \sqrt 2$ when
$|s| \ge \tau_0(\beta) := x_0(\beta) / \beta$, and
$
     \varphi(s) 
 \le c(\beta) \e^{ - \beta \ms1 |s|}
$
for every $s \in \R$ by~\eqref{Tau0}, with
\begin{equation}
   c(\beta) 
 = \e^{\beta \tau_0(\beta) - \ln \sqrt 2}
 = \e^{x_0(\beta) - \ln \sqrt 2}. 
 \label{ValC0}
\end{equation}
\dumou
\noindent
An almost optimal $x$ satisfying~\eqref{Impossible} can be found
numerically. We have for example that 
$\varphi(s) \le 2.218 \ms2 \e^{-|s|/2}$ for all $s$ when $\beta = 1/2$,
with $x_0(0.5) = 1.143$. We also find $c(1) < 94.295$ with a choice 
$x_0(1) = 4.893$. We can then improve the first estimate given
by~\eqref{ValC0} for $\beta = 1$. When 
$|s| \le x_0(1) = \tau_0(1)$, we write
\[
     \varphi(s) 
 \le 2.218 \ms2 \e^{ - |s| / 2} 
  =  2.218 \ms2 \e^{|s| / 2} \e^{ - |s|} 
 \le 2.218 \ms2 \e^{\tau_0(1) / 2} \e^{ - |s|} 
  <  26 \e^{ - |s|}.
\]
\dumou
\noindent
More generally, if we know a modified bound $c_m(\beta_1)$ such that 
$\varphi(s) \le c_m(\beta_1) \e^{- \beta_1 |s|}$ for every $s$ and if
$\varphi(s) \le \e^{- \beta_2 |s|} / \sqrt 2$ when 
$|s| \ge \tau_0(\beta_2)$, with $\beta_1 < \beta_2$, then for
$|s| \le \tau_0(\beta_2)$ we can write
\[
     \varphi(s)
 \le c_m(\beta_1) \e^{ - \beta_1 |s|}
  =  c_m(\beta_1) \e^{(\beta_2 - \beta_1)|s|} \e^{ - \beta_2 |s|}
 \le c_m(\beta_1) 
      \e^{(\beta_2 - \beta_1) \tau_0(\beta_2)} \e^{ - \beta_2 |s|},
\]
\dumou
\noindent
so that 
\begin{equation}
     c_m(\beta_2)
 \le \max \bigl( \e^{(\beta_2 - \beta_1) \tau_0(\beta_2)} c_m(\beta_1),
                 1 / \sqrt 2 \bigr). 
 \label{ValC1}
\end{equation}
\dumou
\noindent
The following table displays admissible values for $x_0(\beta)$,
$\tau_0(\beta)$, then the corresponding rough bound $c(\beta)$
from~\eqref{ValC0}, and the modified bounds $c_m(\beta)$ obtained step by
step applying~\eqref{ValC1}, by dividing the interval $[0, 1]$ in ten equal
segments, beginning with
$c(0) = c_m(0) = \varphi(0) \le 1 / \sqrt 2 < 0.708$.
We have replaced each higher precision value of $x$ by the upper bound
$x_0(\beta) = \lceil 1000 \ms1.\ms1 x \rceil / 1000$, and used this
replacement consistently in the further calculations of $\tau_0(\beta)$,
$c(\beta)$ and $c_m(\beta)$.
\dumou
\def\mK{\kern10pt\relax }
\def\mE{\mK &\mK }
\def\noa{\noalign{\vskip 0.190mm\hrule\vskip 0.950mm\relax }}
\def\hp{\hphantom{0}\relax }
\[
\begin{matrix}
  \beta  & x_0(\beta) & \tau_0(\beta) & \ms5 c(\beta) & \ms7 c_m(\beta)  
 \\ 
 \noa
 \mK 0.0 \mE  0.000  \mE    0.000    \mE \hp 0.708 \mE \hp 0.708 \mK
 \\ 
 \mK 0.1 \mE  0.182  \mE    1.820    \mE \hp 0.849 \mE \hp 0.850 \mK
 \\ 
 \mK 0.2 \mE  0.381  \mE    1.906    \mE \hp 1.036 \mE \hp 1.029 \mK
 \\ 
 \mK 0.3 \mE  0.603  \mE    2.010    \mE \hp 1.293 \mE \hp 1.259 \mK
 \\ 
 \mK 0.4 \mE  0.854  \mE    2.135    \mE \hp 1.662 \mE \hp 1.559 \mK
 \\ 
 \mK 0.5 \mE  1.143  \mE    2.287    \mE \hp 2.218 \mE \hp 1.960 \mK
 \\ 
 \mK 0.6 \mE  1.484  \mE    2.474    \mE \hp 3.119 \mE \hp 2.511 \mK
 \\ 
 \mK 0.7 \mE  1.903  \mE    2.719    \mE \hp 4.742 \mE \hp 3.296 \mK
 \\ 
 \mK 0.8 \mE  2.451  \mE    3.064    \mE \hp 8.203 \mE \hp 4.478 \mK
 \\ 
 \mK 0.9 \mE  3.255  \mE    3.617    \mE    18.328 \mE \hp 6.430 \mK
 \\ 
 \mK 1.0 \mE  4.893  \mE    4.893    \mE    94.295 \mE    10.489 \mK
 \\ 
 \noa
\end{matrix}
\]
\dumou
\noindent
We obtain the announced bounds when $\beta = 1/2$ and $\beta = 1$. One can
obviously refine the previous argument and show that
\[
     \varphi(s) 
 \le c(0) \ms1 
      \exp \Bigl( \int_0^1 \tau_0(\beta) \, \d \beta \Bigr) \e^{ - |s|}
 \le \frac 1 {\sqrt 2} \ms1 
      \exp \Bigl( \int_0^1 \tau_0(\beta) \, \d \beta \Bigr) \e^{ - |s|}.
\]
We may get in this way that $\varphi(s) < 9 \e^{-|s|}$. An exact estimate
could perhaps be obtained by an extreme point argument, as
in~\cite{FraGue}. Some numerical experiments suggest that for every 
$\beta \ge 0$, the maximum on $\R$ of 
$s \mapsto \e^{\beta |s|} \varphi(s)$, for $\varphi$ symmetric log-concave
probability density with variance $1$, occurs for one of the two examples
$\varphi_0$, $\varphi_1$ mentioned in~\eqref{Varphis}. The example
$\varphi_0(s)$ shows that $\e^{\beta |s|} \varphi(s)$ is unbounded when
$\beta > \sqrt 2$ and~$\sigma = 1$.
\end{proof}

 Our next estimate is so poor that it does not deserve to be given
explicitly.

\begin{cor}\label{NotDeserve} 
There exists a numerical value $\kappa > 0$ such that for every centered
log-concave probability density $\varphi$ on\/~$\R$ with variance
$\sigma^2 = 1$, one has
\[
 \forall s \in \R,
 \ms{14}
 \varphi(s) \le \kappa \ms1 \e^{- |s| / \kappa}.
\]
\end{cor}

\begin{proof}
Since $\varphi$ is centered, we know that
$
    \int_0^{+\infty} s \varphi(s) \, \d s
  = \int_{-\infty}^0 |s| \varphi(s) \, \d s
$, and we can thus set $S_1 := S_1^+(0) = S_1^-(0)$. For $p \ne 1$, let us
write $S^{\pm}_p$ instead of $S^{\pm}_p(0)$. We have that 
$S_2^+, S_2^- \le \sigma^2 = 1$. By Corollary~\ref{EstimaLogConcCB} and
Lemma~\ref{EstimaLogConcNew} with $p = 1$, applied on the intervals 
$[0, +\infty)$ and $(-\infty, 0]$, we get
\[
     2
 \ge \max_{s \ge 0} \varphi(s)
 \ge \frac {(S_0^+)^2} {2 \ms1 S_1} \up,
 \ms{18}
     2
 \ge \max_{s \le 0} \varphi(s)
 \ge \frac {(S_0^-)^2} {2 \ms1 S_1} \up.
\]
It follows that
$
     8 \ms1 S_1
 \ge (S_0^+)^2 + (S_0^-)^2
 \ge 1 / 2
$
so $S_1 \ge 1 / 16$. We also need a lower bound for $S_0^\pm$. Let
$\kappa_1 = 16$. By Cauchy--Schwarz we have
\[
     \kappa_1^{-2}
 \le S_1^2 
 \le S_0^- \ms1 S_2^-
 \le S_0^-,
 \ms{16}
     \kappa_1^{-2}
 \le S_1^2 
 \le S_0^+ \ms1 S_2^+
 \le S_0^+,
\]
hence $S_0^-, S_0^+ \ge \kappa_1^{-2}$. Suppose that the maximum of
$\varphi$ is reached at $s_0 \ge 0$. Then~$\varphi$ is non-decreasing on
$(-\infty, s_0]$ and by Lemma~\ref{EstimaLogConcNew} with $p = 2$ we get
\begin{equation}
     4
 \ge \varphi(0)^2
  =  \max_{s \le 0} \varphi(s)^2
 \ge \frac {(S_0^-)^3} {3 S_2^-}
 \ge \frac {\kappa_1^{-6} \ns7} 3 \ms2
  =: \kappa_2^{-2}.
 \label{EstimKappa}
\end{equation}
The symmetric probability density 
$\varphi_1(s) = (2 S_0^-)^{-1} \varphi(-|s|)$ on $\R$ is log-concave, has
variance $\sigma_1^2 = S_2^- / S_0^- \le \kappa_1^2$.
By~\eqref{EstimKappa}, we have $(S_0^-)^3 / S_2^- \le 12$. By
Lemma~\ref{LExpoDecay}, we know that
\[
     \varphi_1(s) 
 \le \frac {11} {\sigma_1} \e^{ - |s| / \sigma_1},
 \ms{18} \hbox{and} \ms{18}
     \varphi(s) 
 \le 22 \ms3 
      \Bigl( \frac {(S_0^-)^3} {S_2^-} \Bigr)^{1/2}
       \e^{ - |s| / \sigma_1}
 \le 77 \ms2 \e^{ - |s| / \kappa_1}
\]
for $s \le 0$. Let us pass to the positive side. We let 
$\widetilde \varphi$ be equal to $\varphi(s_0)$ on~$[0, s_0]$ and to
$\varphi$ on $[s_0, +\infty)$. Then 
$\widetilde S_0^+ \ge S_0^+ \ge \kappa_1^{-2}$ and since 
$\kappa_2^{-1} \le \varphi(0) \le \varphi(x) \le \varphi(s_0) \le 2$ when
$0 \le x \le s_0$, we have
$\widetilde \varphi \le 2 \ms1 \kappa_2 \ms1 \varphi$ on $[0, +\infty)$. The
symmetrized function $\varphi_1$ corresponding to $\widetilde \varphi$
satisfies $\sigma_1^2 = \widetilde S_2^+ / \widetilde S_0^+
 \le 2 \ms1 \kappa_1^2 \kappa_2$. Also, we know that
$(\widetilde S_0^+)^3 / \widetilde S_2^+
 \le 3 \ms2 \max \widetilde \varphi(s)^2 \le 12$. The rest is identical to
the negative case.
\end{proof}

 The next lemma is easy and classical. The \emph{(total) mass} of a real
valued\label{TotaMass} 
(thus bounded) measure $\mu$ on $(\Omega, \ca F)$ is defined by setting
$\|\mu\|_1 = \mu^+(\Omega) + \mu^-(\Omega) = |\mu| (\Omega)$, where
$\mu = \mu^+ - \mu^-$ is the Hahn\label{MuPlusMoins} 
decomposition of $\mu$ as difference of
two nonnegative measures, and $|\mu| = \mu^+ + \mu^-$. On the line or on
$\R^n$ we have
\[
   \|\mu\|_1
 = \sup \Bigl\{ 
         \Bigl| \int_{\R^n} \psi \, \d \mu \Bigr| : 
           \ms2 \psi \in \ca K(\R^n),
            \ms6 \|\psi\|_\infty \le 1
        \Bigr\},
\]
and when $\mu$ has a density $f$, one has that
$\|f(x) \, \d x\|_1 = \|f\|_{L^1(\R^n)}$.

\begin{lem}%
\label{PreBourg}
\begin{subequations}%
Let $\mu$ be a real valued measure on\/ $\R$ and let 
$  m(t) 
 = \widehat \mu(t)$ be its Fourier transform. For every $t \in \R$ we have
\begin{equation}
     |m(t)|
 \le \| \mu \|_1.
 \label{PreBourgA}
\end{equation}
If\/ $\d \mu(s) = \psi(s) \, \d s$ with $\psi$ integrable, then
$m = \widehat \mu = \widehat \psi$ and\/ $|m(t)| \le \|\psi\|_{L^1(\R^n)}$.

 Let us further assume that $\int_\R (1 + |s|) \, \d |\mu| (s) < +\infty$.
Then $m$ is $C^1$ on\/ $\R$ and
\[
   \ii \ms1 m'(t) 
 = 2 \pi \int_\R s \ms1 \e^{- 2\ii \pi s t} \, \d \mu(s),
\]
so $\ii \ms1 m'$ is the Fourier transform of the real valued measure\/ 
$2 \pi s \, \d \mu(s)$.

 Let $\nu$ be a real valued measure on\/ $\R$ and let
$\psi(s) = \nu \bigl( (-\infty, s] \bigr)$, for every $s \in \R$. The
measure $\nu$ is the derivative of $\psi$ in the sense of distributions and
assuming $\psi$ integrable, we have
\begin{equation}
   2 \ii \pi t \ms1 \widehat \psi(t)
 = 2 \ii \pi t \int_\R \psi(s) \e^{- 2\ii \pi s t} \, \d s
 = \int_\R \e^{- 2\ii \pi s t} \, \d \nu(s),
 \label{DeriFour}
\end{equation}
so\/ $2 \ii \pi t \ms1 \widehat \psi(t)$ is the Fourier transform of the
derivative $\nu$ of $\psi$.

 Let $j, k$ be nonnegative integers. Suppose that $\psi$ is of class
$C^{k-1}$ on the line, with a $k\ms1$th derivative $\psi^{(k)}$ in the
sense of distributions that is a bounded measure~$\nu_k$ on\/~$\R$, and
that\/ $\lim_{ |s| \rightarrow +\infty} \psi(s) = 0$,
$\int_\R |s|^j \, \d |\nu_k|(s) < +\infty$. Then $m$ is $C^j$ and
\begin{equation}
     (2 \pi |t|)^k \ms2 |m^{(j)}(t)| 
 \le (2 \pi)^j \ms3
      \bigl\| \bigl( s^j \psi(s) \bigr)^{(k)} \bigr\|_1.
\end{equation}
Consequently, for $t \ne 0$, we have that
\begin{align}
     |m^{(j)}(t)| 
 \le & \ms4 \frac { (2 \pi)^{j - k} \ns{18} } { |t|^k } \ms{12}
      \sum_{i = (k - j)^+}^{k-1}  \binom k i
       \frac {j \ms1 !} { (j + i - k) \ms1 ! }
        \int_\R |s|^{i + j - k} \ms1 |\psi^{(i)}(s)| \, \d s
 \label{EstimaDeriv}
 \\ & \ms8 + \frac { (2 \pi)^{j - k} \ns{18} } { |t|^k } \ms{12}
         \int_\R |s|^j \, \d |\nu_k|(s).
 \notag
\end{align}
\end{subequations}
\end{lem}

 In the line above, one can replace $\int_\R |s|^j \, \d |\nu_k|(s)$ with
$\int_\R |s|^j \ms1 |\psi^{(k)}(s)| \, \d s$, when~$\psi$ admits a
derivative $\psi^{(k)}$ and $\d \nu_k(x) = \psi^{(k)}(x) \, \d x$.

\begin{proof}
The first inequality~\eqref{PreBourgA} is obvious. Assuming that 
$\int_\R |s| \ms1 \, \d |\mu|(s)$ is finite, we write
\[
 m(t) = \int_\R \e^{- 2 \ii \pi s t} \, \d \mu(s)
     := \int_\R \e^{- 2 \ii \pi s t} \, \d \mu^+(s)
         - \int_\R \e^{- 2 \ii \pi s t} \, \d \mu^-(s),
\]
and we obtain by the dominated convergence theorem that
\[
   m' (t) 
 = - 2 \ii \pi \int_\R s \ms1 \e^{- 2 \ii \pi s t} \, \d \mu(s).
\]
If $\nu$ in~\eqref{DeriFour} has the form $\d \nu(x) = \psi'(x) \, \d x$
with $\psi'$ a true derivative, we use integration by parts, otherwise we
use Fubini's theorem for $\nu^+$ and $\nu^-$. We get
\[
   2 \ii \pi t \int_\R \psi(t) \e^{- 2\ii \pi s t} \, \d s
 = \int_\R \e^{- 2\ii \pi s t} \, \d \nu(s).
\]
The verification of~\eqref{EstimaDeriv} is left to the reader. Notice that
by~\eqref{L5.9}, the hypotheses imply that
$ \int_\R |s|^{i + j - k} \ms1 |\psi^{(i)}(s)| \, \d s < +\infty$ when 
$(k - j)^+ \le i < k$. Indeed, if $g^{(\ell+1)}$ is integrable 
on~$[0, +\infty)$, then $g^{(\ell)}$ tends to a limit $L$ at infinity
and if $g$ tends to $0$ at infinity, it follows that $L = 0$, for example
by the Taylor formula.
\end{proof}

 The next lemma is straightforward.

\begin{lem}%
\label{Repeti}
Let $\nu$ be a nonnegative measure on\/ $(0, +\infty)$ and $\alpha > 0$.
One has
\[
   \alpha \int_0^{+\infty} s^{\alpha - 1}
    \nu \bigl( [s, +\infty) \bigr) \, \d s
 = \int_0^{+\infty} s^\alpha \, \d \nu(s).
\]
Let $F$ be a function on\/ $(0, +\infty)$ such that\/
$|F(s)| \le \int_s^{+\infty} \, \d \nu(s)$ for $s > 0$. One has
\[
     \alpha \int_0^{+\infty} s^{\alpha - 1} |F(s)| \, \d s
 \le \int_0^{+\infty} s^\alpha \, \d \nu(s).
\]
Suppose that the function $g$ is differentiable on\/ $\R$, with\/
$\lim_{s \rightarrow \pm \infty} g(s) = 0$ and $g'$ integrable on the line.
It follows that
\begin{equation}
     \alpha \int_\R |s|^{\alpha - 1} \ms1 |g(s)| \, \d s
 \le \int_\R |s|^\alpha \ms1 |g'(s)| \, \d s.
 \label{L5.9}
\end{equation}
If in addition $g$ is even and non-increasing on\/ $[0, +\infty)$, one has
\[
   \int_\R |s|^\alpha \ms1 |g'(s)| \, \d s
 = \alpha \int_\R |s|^{\alpha - 1} \ms1 g(s) \, \d s,
 \ms{10} \hbox{and} \ms{10}
   \int_\R |g'(s)| \, \d s
 = 2 \ms1 g(0).
\]
\end{lem}

\begin{proof}
The first assertion is an immediate consequence of Fubini, because
\[
   \alpha \int_0^{+\infty} s^{\alpha - 1}
    \nu \bigl( [s, +\infty) \bigr) \, \d s
 = \alpha \dint \gr 1_{ \{0 < s < t\} } \ms2 s^{\alpha - 1} 
    \, \d \nu(t) \ms{1.5} \d s
 = \int_0^{+\infty} t^\alpha \, \d \nu(t),
\]
with integrals finite or not. The remaining facts are left to the reader.
For~\eqref{L5.9}, use $\d \nu(s) = |g'(s)| \, \d s$.
\end{proof}

 We arrive to the main result of this section.
 
\def\CiteEstimaFouriC{\cite[\S$\ms2$4]{BourgainL2}}
\begin{prp}[Bourgain~\CiteEstimaFouriC]%
\label{EstimaFouriC}
Let $\Klc$ be a symmetric log-concave probability density on\/ $\R^n$,
isotropic with variance $\sigma^2$. Let $\mlc$ be the Fourier transform
of~$\Klc$. For every $\xi \in \R^n$ one has that
\begin{subequations}\label{EstimaBourG}%
\begin{equation}
     \pi \sqrt 2 \ms2 \sigma \ms1 |\xi| \ms2 |\mlc(\xi)| 
 \le 1,
 \ms{16}
     |1 - \mlc(\xi)| 
 \le 2 \pi \ms1 \sigma \ms1 |\xi|,
 \ms{16}
     |\xi \ps \nabla \mlc(\xi)| 
 \le 2.
 \label{EstimaBour} \tag{\ref{EstimaBourG}.$\gr B$}
\end{equation}
\end{subequations}
The middle inequality follows from the fact that for every 
$\theta \in S^{n-1}$, one has
\[
     |\theta \ps \nabla \mlc(t \theta)| 
 \le 2 \ms1 \pi \ms1 \sigma,
 \quad t \in \R.
\]
\end{prp}

\begin{rem} 
These inequalities are valid for $m_C$, when $C$ is a symmetric convex body,
isotropic and normalized by variance. The case of convex bodies is the one
given by Bourgain, but the proof is the same in the log-concave case. 
\end{rem}

\begin{proof}
We have seen in~\eqref{VarphiTheta} that for $\theta \in S^{n-1}$ and $t$
real, one can write
\[
   \mlc(t \theta)
 = \int_{\R} \varphi_\theta(s) \e^{ - 2 \ii \pi s t} \, \d s,
\]
where $\varphi_\theta$ is obtained by integrating $\Klc$ on affine
hyperplanes parallel to~$\theta^\perp$. It is enough to prove the case
$\sigma = 1$. We know that $\varphi_\theta$ is log-concave according to
Pr\'ekopa--Leindler, it is even, has integral~$1$ and variance~$1$ by
hypothesis. By Lemma~\ref{LExpoDecay}, one has that
$
     \varphi_\theta(s) 
 \le 2 \ms1 \e^{- |s| / 2}$
for every $s \in \R$, but the desired estimates do not depend on this
exponential decay, which ensures however absolute convergence for the
integrals that follow. For every~$t$, by~\eqref{EstimaDeriv} with $j=0$, 
$k = 1$ and since $\varphi_\theta$ is even and decreasing on 
$(0, +\infty)$, we have using Lemma~\ref{Repeti} that
\[
     |\mlc(t \theta)|
  =  \Bigl| \int_\R \varphi_\theta(s) \e^{- 2 \ii \pi s t} \, \d s \Bigr|
 \le \frac 1 {2 \pi |t|} 
       \int_\R |\varphi'_\theta(s)| \, \d s 
  =  \frac {\varphi_\theta(0)} {\pi |t|} \up.
\]
The function $\varphi_\theta$ has variance $1$ by our normalization
assumption, and according to Corollary~\ref{EstimaLogConcC} we have the
upper bound
$
     \varphi_\theta(0)
 \le 1 / \sqrt 2
$.
Writing $\xi = |\xi| \theta$, it follows that
$\pi \sqrt 2 \ms2 |\xi| \ms2 |\mlc(\xi)| \le 1$ for every $\xi$ in~$\R^n$.
\dumou

 Notice that our writing is not correct, because $\varphi_\theta$ might be
discontinuous at the ends of its support, so that $\varphi'_\theta$ is a
measure in that case, with two Dirac masses at the end points of the
support. This happens for example with $\varphi_{\theta, C}$ when $C$ is
polyhedral and $\theta$ orthogonal to a facet. We leave the easy changes
to the reader.
\dumou

 Given $\theta \in S^{n-1}$, the derivative of
$t \mapsto \mlc(t \theta)$ is expressed by
\[
   \theta \ps \nabla \mlc(t \theta)
 = \int_\R (- 2 \ii \pi s) \varphi_\theta(s) \e^{ - 2 \ii \pi s t} \, \d s,
\]
and
\[
     |\theta \ps \nabla \mlc(t \theta)|
 \le 2 \pi \int_\R |s| \ms1 \varphi_\theta(s) \, \d s 
 \le 2 \pi \Bigl( \int_\R s^2 \ms1 \varphi_\theta(s) \, \d s \Bigr)^{1/2}
  =  2 \pi,
\]
hence
$
     |1 - \mlc(\xi)| 
  =  |\mlc(0) - \mlc(|\xi| \ms1 \theta)| 
 \le {2 \pi} \ms1 |\xi|
$. 
We see also that
\[
   t \ms1 \theta \ps \nabla \mlc(t \theta)
 = \int_\R (- 2 \ii \pi t) s \varphi_\theta(s) \e^{ - 2 \ii \pi s t} \, \d s
 = - \int_\R 
      \bigl( s \varphi_\theta(s) \bigr)' \e^{ - 2 \ii \pi s t} \, \d s.
\]
We estimate the two parts coming from
$\bigl( s \varphi_\theta(s) \bigr)'$, first
\[
     \Bigl| \int_\R 
      \varphi_\theta(s) \e^{ - 2 \ii \pi s t} \, \d s \Bigr|
 \le \int_\R \varphi_\theta(s) \, \d s
  =  1,
\]
and as $\varphi_\theta$ is even and non-increasing on $[0, +\infty)$, we
have by Lemma~\ref{Repeti} that
\[
     \Bigl| \int_\R 
      s \varphi'_\theta(s) \e^{ - 2 \ii \pi s t} \, \d s \Bigr|
 \le \int_\R |s \varphi'_\theta(s)| \, \d s
  = \int_\R \varphi_\theta(s) \, \d s
  = 1.
\]
We conclude that $|t \ms1 \theta \ps \nabla \mlc(t \ms1 \theta)| \le 2$ and
get~$|\xi \ps \nabla \mlc(\xi)| \le 2$ for every $\xi$.
\end{proof}

\begin{lem}%
\label{LEstimatesForC}
Let $\Klc$ be an even log-concave probability density on\/ $\R^n$,
normalized by variance, and $\mlc$ its Fourier transform. For every 
$\theta \in S^{n-1}$ one has
\[
     \Bigl| \frac {\d^j} {\d t^j \ns2} \ms2 
       \mlc(t \theta) \Bigr|
 \le \delta_{j, c} \ms2 
      \frac 1 
            {1 + 2 \pi \ms1 |t|} \ms1 \up,
 \quad j \ge 0, \ms7 t \in \R,
\]
where $\delta_{j, c}$ is a universal constant, estimated
at\/~\eqref{deltas}.
\end{lem}

\begin{proof}
We know that
$
   \mlc(t \theta)
 = \widehat{ \varphi_\theta } (t)
$.
From Lemma~\ref{PreBourg}, \eqref{EstimaDeriv} with $k = 0$, it follows that
\[
     \Bigl| \frac {\d^j} {\d t^j \ns2} \ms2 \mlc(t \theta) \Bigr|
  =  \bigl| \widehat{ \varphi_\theta }^{(j)} (t) \bigr|
 \le (2 \pi)^j  
      \int_\R |s|^j \varphi_\theta (s) \, \d s,
\]
and with $k = 1$,
\[
     \Bigl| \frac {\d^j} {\d t^j \ns2} \ms2 \mlc(t \theta) \Bigr|
 \le \frac {(2 \pi)^{j-1} \ns{22}} {|t|} \ms{17}
      \Bigl(
       j \int_\R |s|^{j-1} \varphi_\theta (s) \, \d s
        + \int_\R |s|^j \ms1 |\varphi_\theta' (s)| \, \d s
      \Bigr).
\]
The function $\varphi_\theta$ is a symmetric log-concave probability
density on~$\R$, with variance~$1$. By Corollary~\ref{EstimaLogConcC}, we
have for $j = 0$ that
\[
     \bigl( 1 + 2 \pi |t| \bigr) \ms2
      | \mlc(t \theta) |
 \le \int_\R (\varphi_\theta(u) + |\varphi_\theta'(u)|) \, \d u
 \le 1 + 2 \ms2 \varphi_\theta(0)
 \le 1 + \sqrt 2.
\]
For $j \ge 1$, we have
$
   \int_\R |u|^j |{\varphi_\theta}' (u)| \, \d u
 = j \int_\R |u|^{j-1} \varphi_\theta(u) \, \d u
$ by Lemma~\ref{Repeti}, and
\[
     \bigl( 1 + 2 \pi |t| \bigr) \ms2
      \Bigl| \frac {\d^j} {\d t^j \ns2} \ms2 \mlc(t \theta) \Bigr|
 \le (2 \pi )^j
      \int_\R \bigl( |u|^j + 2 j |u|^{j-1} \bigr)
       \varphi_\theta (u) \, \d u.
\]
The function $\varphi_\theta$ satisfies 
$\int_\R s^2 \varphi_\theta(s) \, \d s = 1$, implying that
\begin{subequations}\label{deltas}
\begin{equation}
     \delta_{0, c}
 \le 1 + \sqrt 2 < 3 \ms2 ; 
 \ms{12}
     \delta_{1, c}
 \le 6 \ms1 \pi \ms2 ; 
 \ms{12}
     \delta_{2, c}
 \le 20 \ms1 \pi^2.
\end{equation}
We know by Lemma~\ref{LExpoDecay} that 
$\varphi_\theta(s) \le 11 \e^{-|s|}$. This implies for $j > 2$ that
\begin{equation}
     \delta_{j, c}
 \le 22 \ms1 (2 \pi)^j 
      \int_0^{+\infty} ( s^j + 2 j s^{j-1} )
       \e^{-s} \, \d s 
  =  66 \ms2 (2 \pi)^j \ms1 \Gamma(j+1).
\end{equation}
\end{subequations}
\end{proof}

\begin{rems}
One gets
$\int_\R |s|^j \varphi_\theta(s) \, \d s \le 3^{j/2} \Gamma(j+1)$ by
applying Lemma~\ref{EstimaLogConcNew} and Corollary~\ref{EstimaLogConcC};
Lemma~\ref{LExpoDecay} yields the bound
$22 \ms1 \Gamma(j+1)$, better when $j$ is large.

 If the log-concave probability density $K$ on $\R^n$ is normalized by
variance but is simply \emph{centered}, then $\varphi_{\theta, K}$ is
log-concave and centered for each $\theta$, and satisfies the exponential
decay of Corollary~\ref{NotDeserve}. If $\varphi_{\theta, K}$ reaches its
maximum at $s_0$, then
\[
     \int_\R |s|^j |\varphi'_{\theta, K}(s)| \, \d s
 \le 2 \ms1 |s_0|^j \varphi_{\theta, K}(s_0)
      + j \int_\R |s|^{j-1} \varphi_{\theta, K}(s) \, \d s
\]
admits a universal bound~$\kappa_j$. Lemma~\ref{LEstimatesForC} remains
valid in this extended case, with other constants $(\delta_j)_{j \ge 0}$
for which we do not have satisfactory explicit expressions.
Fradelizi~\cite[Theorem~5]{FradelHSCB} extended the $L^p(\R^n)$ result of
Theorem~\ref{TheoMaxi} (Bourgain, Carbery) to \emph{centered} bodies $C$
in~$\R^n$, not necessarily symmetric (unluckily, the word 
\og centered\fge was forgotten in the statement given in~\cite{FradelHSCB}).

\begin{smal}
\noindent
If $C$ is an arbitrary convex body, then $\M_C$ is bounded on $L^p(\R^n)$,
$p \in (1, +\infty]$, but for each fixed $n \ge 1$ and $p < +\infty$, there
is no uniform bound for the family of arbitrary convex bodies in~$\R^n$ (if 
$n = 1$, examine $\M_C f$ when $C = [1, 1 + \varepsilon]$, $f = \gr1_C$ and
$\varepsilon \to 0$). In a somewhat related direction, it is known that the
$L^p(\R^n)$ norm of the uncentered operator in~\eqref{UncentOp} is 
$\ge C_p^n$ for some $C_p > 1$, when $1 < p < +\infty$~\cite{GMS}.
\end{smal}

\noindent
\end{rems}

\begin{cor}%
\label{EstimatesForC}
Let $\Klc$ be a symmetric log-concave probability density on\/ $\R^n$,
iso\-tropic with variance $\sigma^2$, and let $\mlc$ be its Fourier
transform. For every $\xi \in \R^n$ and $j \ge 0$ one has that
\begin{equation}
     \Bigl| \frac {\d^j} {\d t^j \ns2} \ms2 \mlc (t \xi) \Bigr|
 \le \delta_{j, c} \ms2 
      \frac {|\sigma \ms1 \xi|^j} 
            {1 + 2 \pi \ms1 |t \ms1 \sigma \ms1 \xi|} \ms1 \up,
 \quad t \in \R,
 \label{NoticeFirst}
\end{equation}
where $\delta_{j, c}$ is the universal constant of
Lemma~\ref{LEstimatesForC}.
\end{cor}

\begin{proof}
The result is obvious when $\xi = 0$, otherwise we apply
Lemma~\ref{LEstimatesForC} with $\theta = |\xi|^{-1} \xi$ to the normalized
Fourier transform $N(\xi) = \mlc(\xi / \sigma)$, obtaining thus
\[
     \frac {\d^j} {\d t^j \ns2} \ms2 \mlc(t \xi)
  =  \frac {\d^j} {\d t^j \ns2} \ms2 N(t |\sigma \ms1 \xi| \theta)
  =  |\sigma \ms1 \xi|^j 
      \frac {\d^j} {\d u^j \ns2} \ms2 N(u \theta) 
       \barre_{u = t |\sigma \ms1 \xi|}
 \le \delta_{j, c} 
      \ms2 \frac {|\sigma \ms1 \xi|^j} 
                 {1 + 2 \pi \ms1 |t \ms1 \sigma \ms1 \xi|} \ms1 \up.
\]
\end{proof}

\subsection{Fourier analysis in $L^2(\R^n)$%
\label{AnaFouri}}

\begin{lem}[Bourgain~\cite{BourgainL2}]%
\label{Clef} 
Let $K$ be a kernel in $L^1(\R^n)$ and assume that its Fourier
transform~$m$ is $C^1$ outside the origin. For every $j \in \Z$, define
\[
   \alpha_j(m)
 = \sup_{2^{j-1} \le |\xi| \le 2^{j + 1} } | m(\xi)|
 \ms{20} \hbox{and} \ms{20}
   \beta_j(m)
 = \sup_{2^{j-1} \le |\xi| \le 2^{j + 1} } | \xi \ps \nabla m(\xi)|.
\]
If\/
$\label{GammSubB}
    \Gamma_B(K) 
 := \sum_{j \in \Z} \sqrt {\alpha_j(m)} \ms2
      \sqrt{\alpha_j(m) + \beta_j(m)} 
  < +\infty
$,
then the maximal operator~$\Mg_K$ associated to $K$ is bounded on
$L^2(\R^n)$. More precisely, one has that
\[
     \bigl\| \Mg_K \ms1 f \bigr\|_{L^2(\R^n)}
  =  \bigl\| \ms1 \sup_{t > 0} |{\Di K t} * f| \ms1 \bigr\|_{L^2(\R^n)}
 \le 2 \ms2 \Gamma_B(K) \ms2 \|f\|_{L^2(\R^n)},
 \quad f \in L^2(\R^n).
\]
\end{lem}

 We shall simply write $\alpha_j = \alpha_j(m)$ and $\beta_j = \beta_j(m)$
in the rest of the section.

\begin{rem} 
Clearly, we have that
\[
     \sum_{j \in \Z} \sqrt { \vphantom{\beta} \alpha_j} \ms2
      \sqrt{\alpha_j + \beta_j} 
 \le \sum_{j \in \Z} {\alpha_j} 
     + \sum_{j \in \Z} \sqrt{\alpha_j \ms2 \beta_j},
\] 
and each of the two terms in the right-hand side is less than the left-hand
side. Bourgain employs both expressions as definitions of $\Gamma_B(K)$, one
in~\cite{BourgainL2} and the other in~\cite{BourgainLp} or
in~\cite{BourgainCube}. The convergence of the series of $\alpha_j\ms1$s
when $j$ tends to $-\infty$ implies that $m(\xi)$ tends to~$0$ when $\xi$
tends to~$0$, thus $m(0) = 0$, which means that the integral of~$K$ on
$\R^n$ is equal to $0$. This lemma will not be applied to $K_C$ or $\Klc$,
but typically, to the difference of two kernels with equal integrals.
\end{rem}

\begin{proof}
We shall give a proof less rough than Bourgain's, relying on the tools
introduced in Section~\ref{SteinsResults}. We consider a $C^\infty$
function $\eta$ on~$\R$ such that
\[
 \eta(t) = 1
 \ms9 \hbox{if} \ms9
 t \le 1,
 \ms{12}
 \eta(t) = 0
 \ms9 \hbox{if} \ms9
 t \ge 2,
 \ms{12} \hbox{and} \ms9
 0 \le \eta \le 1.
\]
Next, we set $\rho(t) = \eta(t) - \eta(2 t)$ for $t \in \R$. We see
that~$\rho$ vanishes outside $[1/2, 2]$. Also, $\rho(t) = 1 - \eta(2 t)$ on
$[1/2, 1]$ and $\rho(t) = \eta(t)$ on $[1, 2]$, so that 
$0 \le \rho(t) \le 1$ and
\[
    d_0
 := \sup_{t \in \R} |t \ms1 \rho'(t)|
  = \sup_{t \in \R} |t \ms1 \eta'(t)|
  = \sup_{t \in [1, 2]} t \ms1 |\eta'(t)|.
\]
Let $\varepsilon > 0$ be given. One can make sure that
$d_0 < (1 + \varepsilon) / \ln 2$, choosing for $\eta$ a $C^\infty$
approximation of the function $\eta_0$ defined on $[0, 2]$ by
$\eta_0(t) = \min(1, 1 - \log_2 t)$, for which 
$t |\eta'_0(t)| = 1 / \ln(2)$ when $t \in [1, 2]$. 

 For every $j \in \Z$ and $\xi \in \R^n$, let
$
 \varphi_j(\xi) = \rho(2^{-j} |\xi|) 
$
and consider the annulus
\[
 C_j = \{ \xi \in \R^n : 2^{j - 1} \le |\xi| \le 2^{j + 1} \} 
 \subset \R^n.
\]
From the properties of $\rho$, we have that $0 \le \varphi_j \le 1$,
$\varphi_j$ vanishes outside~$C_j$, and
\[
   \sum_{j \in \Z} \varphi_j(\xi) 
 = \sum_{j \in \Z} 
    \bigl( \eta(2^{-j} |\xi|) - \eta(2^{-j+1} |\xi|) \bigr)
 = 1
\]
for every $\xi \ne 0$, because $\eta(2^{-j} |\xi|) = 0$ when 
$j \le \log_2(|\xi|) - 1$ and $\eta(2^{-j} |\xi|) = 1$ when 
$j \ge \log_2(|\xi|)$. We introduce for every $j \in \Z$ a multiplier
$m_j$ defined by
\[
 m_j(\xi) = \varphi_j(\xi) \ms1 m(\xi),
 \quad
 \xi \in \R^n,
\]
and we let $K_j = m_j^\vee = \varphi_j^\vee * K$. One has 
$\sum_{j \in \Z} K_j = K$, which allows us to write for $f \in \ca S(\R^n)$
and every $x \in \R^n$ the upper bound
\[
     (\Mg_K f)(x)
  =  \sup_{t > 0} |({\Di K t} * f)(x)|
 \le \sup_{t > 0} \ms1 
      \sum_{j \in \Z} \bigl| \ms1 [ {\Di {(K_j)} t} * f \ms1 ] (x) \bigr|
 \le \sum_{j \in \Z} \ms2 (\Mg_{K_j} f)(x).
\]
By Lemma~\ref{MaxiLoca} with $r = 4$, one has
\begin{equation}
     \bigl\| \Mg_{K_j} f \bigr\|_{L^2(\R^n)}^2
 \le 2 \ms2 \ln 4 \ms2 
      \|m_j\|_{L^\infty(\R^n)} \ms1 \|m^*_j\|_{L^\infty(\R^n)} \ms1 
       \|f\|_{L^2(\R^n)}^2.
 \label{truc}
\end{equation}
We see that $\|m_j\|_\infty \le \alpha_j$, since $|m_j| \le |m|$ and since
$m_j$ is supported in the annulus~$C_j$. On the other hand, 
$m_j^*(\xi) = \xi \ps \nabla m_j(\xi)$ and we have
\[
   \nabla m_j(\xi) 
 = \varphi_j(\xi) \nabla m(\xi) + m(\xi) \nabla \varphi_j(\xi).
\]
As $\varphi_j$ is supported in $C_j$, we get
$
      |\varphi_j(\xi) \ms2 \xi \ps \nabla m(\xi) |
  \le \beta_j 
   <  (1 + \varepsilon) \beta_j \ms1 / \ln 2
$,
and
\[
     |m(\xi) \ms2 \xi \ps \nabla \varphi_j(\xi)| 
 \le \alpha_j \ms2
      \Bigl|
       \xi \ps 2^{-j} \rho' (2^{-j} |\xi|) \ms2 \frac \xi {|\xi|} 
      \Bigr|
 \le \alpha_j \ms1 d_0
  <  (1 + \varepsilon) \alpha_j \ms1 / \ln 2.
\]
It follows that
$
 \|m^*_j\|_\infty \le (1 + \varepsilon) \ms1 (\alpha_j + \beta_j) / \ln 2
$.
By~\eqref{truc} we get 
\[
     \bigl\| \Mg_{K_j} f \bigr\|_{L^2(\R^n)}
 \le 2 \ms1 \sqrt{1 + \varepsilon \vphantom{\beta_j}} \ms2 
      \sqrt{\alpha_j \vphantom{\beta}} \ms2
       \sqrt{\alpha_j + \beta_j} \ms4 \|f\|_{L^2(\R^n)}.
\]
After summation in $j \in \Z$ and letting $\varepsilon \rightarrow 0$, we
conclude that
\[
     \bigl\| \Mg_K f \bigr\|_{L^2(\R^n)}
 \le 2 \ms4 \Gamma_B(K) \ms4 \|f\|_{L^2(\R^n)}.
\]
We pass from $f \in \ca S(\R^n)$ to $f \in L^2(\R^n)$ as explained
in Section~\ref{DefiMaxiFunc}.
\end{proof}

\subsubsection{Conclusion of Bourgain's argument%
\label{ConcluBour}}

\begin{proof}[End of the proof of Theorem~\ref{TheoBour}]
We begin with a version of the proof that illustrates well the fact that
Lemma~\ref{Clef} is a comparison lemma: in vague terms, if we know that the
conclusion of Theorem~\ref{TheoBour} is true for \emph{one} family of convex
sets, then it is true for all convex sets.
\dumou

 We rely here on Stein's Theorem~\ref{indepdim} for the Euclidean ball $B$,
asserting that the maximal operator $\M_B$ is bounded on $L^p(\R^n)$ for
every $p$ in~$(1, +\infty]$, with a bound independent of the dimension~$n$.
In this paragraph, we only use the $L^2$~case of this result. Let us call
$B = B_{n, V}$ the Euclidean ball in~$\R^n$, centered at~$0$ and normalized
by variance, which has radius $\sqrt{n+2}$ by~\eqref{RnV}. Let $m_B$ denote
the Fourier transform of~$K_B$. Consider also a symmetric log-concave
probability density $\Klc$ on~$\R^n$, isotropic and normalized by variance.
The two functions $\mlc$ and $m_B$ verify the estimates~\eqref{EstimaBour}
of Proposition~\ref{EstimaFouriC}. We apply Lemma~\ref{Clef} to the
difference kernel $K = \Klc - K_B$. According to~\eqref{EstimaBour}, for
every $\xi \in \R^n$, the Fourier transform $m = \mlc - m_B$ satisfies
\[
     |\xi| \ms2 |m(\xi)| 
 \le \sqrt 2 / \pi,
 \ms{ 9}
     |m(\xi)| 
 \le |1 - \mlc(\xi)| + |1 - m_B(\xi)|
 \le 4 \pi \ms1 |\xi|,
 \ms{ 9}
     |\xi \ps \nabla m(\xi)| 
 \le 4.
\]
We deduce that
$
     \beta_j
   = \sup_{2^{j - 1} \le |\xi| \le 2^{j + 1} } |\xi \ps \nabla m(\xi)|
 \le 4
$
for $j \in \Z$. For $j < 0$ one has
\[
     \alpha_j 
  =  \sup_{2^{j - 1} \le |\xi| \le 2^{j + 1} } | m(\xi)|
 \le 4 \pi \ms1 2^{j + 1}
  =  4 \pi \ms1 2^{-|j| + 1}
 \le 3\ns{0.5}2 \ms2.\ms2 2^{-|j|},
\]
and for $j \ge 0$, we have
$
     \alpha_j 
 \le \sqrt 2 \ms2 \pi^{-1} \ms1 2^{-j + 1}
 \le 2^{-j}
$.
It follows that the two series $\sum_{j \in \Z} \alpha_j$ and 
$\sum_{j \in \Z} \sqrt{\alpha_j  \beta_j}$ converge, and
\[
     \sum_{j \in \Z} \alpha_j 
 \le 32 + 2,
 \ms{16}
     \sum_{j \in \Z} \sqrt{\alpha_j  \beta_j}
 \le 20 + 10 \ms1 \sqrt 2,
\]
thus the maximal operator
$
 f \mapsto \sup_{t > 0} |{\Di K t} * f|
$
is bounded on $L^2(\R^n)$ by a constant independent of the dimension, say,
less than $2 \ms1 \Gamma_B(K) < 2 \ms1 (54 + 10 \sqrt 2) < 137$. Finally,
for $f \ge 0$, we write 
\begin{align*} 
     \M_{\Klc} f
  &=  \sup_{t > 0} |\Di {(\Klc)} t  * f|
 \\
 &\le \sup_{t > 0} | \Di{(K_B)}t * f|
      + \sup_{t > 0} | \Di{(\Klc - K_B)}t * f|
  =  \M_B f + \Mg_K f,
\end{align*} 
and we can estimate $\M_{\Klc}$ by the sum of two operators that are bounded
on~$L^2(\R^n)$ by constants independent of the dimension~$n$.
\end{proof}

 The proof actually given by Bourgain~\cite{BourgainL2} bypasses the $L^2$
result of Stein on Euclidean balls. The kernel $K$ is now given as 
$K = \Klc - P$, where $P$ is the Poisson kernel $P = P_1$
from~\eqref{PoissonDensi} for the value $t = 1$ of the parameter. We know
by~\eqref{MaxiPoiss} that the maximal operator
$f \mapsto \sup_{t > 0} \ms1 |P_t f|$ associated to the Poisson kernel acts
boundedly on $L^p(\R^n)$, $1 < p \le +\infty$, with a bound $\le 2$ when 
$p = 2$, thus independent of the dimension~$n$. Now, everything is said: we
replace the multiplier $m_B$ by $\widehat P$ and it suffices to see that
$\widehat P$ also satisfies good estimates similar to~\eqref{EstimaBour}.
But $\widehat P(\xi) = \e^{- 2 \pi |\xi|}$ clearly verifies the even better
estimates
\begin{subequations}%
\begin{equation}
     |\xi| \ms2 |\widehat P(\xi)| 
   = |\xi| \ms2 \e^{- 2 \pi |\xi|}
 \le (2 \pi \e)^{-1},
 \label{FPoissonBoundsA}
\end{equation}
\begin{equation}
     |1 - \widehat P(\xi)| 
 \le 2 \pi \ms1 |\xi|,
 \ms{10}
     |\xi \ps \nabla \widehat P(\xi)| 
   = 2 \pi |\xi| \ms2 \e^{- 2 \pi |\xi|}
 \le \e^{-1},
 \label{FPoissonBoundsB}
\end{equation}
\end{subequations}
where we made use of the inequality $x \e^{-x} \le \e^{-1}$, true for every
$x \ge 0$. This ends the second proof of Theorem~\ref{TheoBour}, with
different constants whose exact values are rather irrelevant. However, we
found here an explicit bound $\kappa_2 < 2 + 137 < 140$, explicit but
definitely not sharp.

\section{The $L^p$ results of Bourgain and Carbery%
\label{ArtiCarbe}}

\noindent
One gives again a symmetric convex body $C$ in $\R^n$, and $\mu_C$ denotes
the uniform probability measure on~$C$. Beside the maximal function $\M_C f$
from~\eqref{OpMaxi}, for every function $f \in L^1_{\rm loc}(\R^n)$ and
every $x \in \R^n$ we set
\[
 \label{DyadiMax}
   (\M_C^{(d)} \ns1 f)(x) 
 = \sup_{j \in \Z} \frac 1 {|2^j C|} \int_{x + 2^j C} |f(y)| \, \d y
 = \sup_{j \in \Z} \int_{\R^n} |f(x + 2^j v)| \, \d \mu_C(v).
\]
One can call $\M_C^{(d)} \ns1 f$ the \emph{\og dyadic\fge 
maximal function\label{DyMaFu}
associated to the convex set $C$}. Obviously, $\M_C^{(d)} \le \M_C$. More
generally, we associate to every kernel $K$ integrable on $\R^n$ the dyadic
maximal function
\[
   (\Mg_K^{(d)} \ns1 f)(x) 
 = \sup_{j \in \Z} \ms1
    \Bigl| \int_{\R^n} f(x + 2^j v) \ms2 K(v) \, \d v \Bigr|,
 \quad
 x \in \R^n.
\]
In 1986, Bourgain and Carbery have obtained identical results
for~$L^p(\R^n)$. Somewhat surprisingly, the cases $\M_C^{(d)}$ and $\M_C$ are
different, the boundedness of~$\M_C$ on~$L^p(\R^n)$ being obtained only
when $p > 3/2$, as opposed to $p > 1$ for~$\M_C^{(d)}$.

\begin{thm}[Bourgain~\cite{BourgainLp}, Carbery~\cite{CarberyLp}]%
\label{TheoDyad}
For every $p$ in\/ $(1, +\infty]$, there exists a constant $\kappa^{(d)}(p)$
such that for every integer $n \ge 1$ and every symmetric convex body
$C \subset \R^n$, one has
\[
 \forall f \in L^p(\R^n),
 \ms9
     \| \M_C^{(d)} \ns1 f \|_{L^p(\R^n)}
 \le \kappa^{(d)}(p) \ms2 \| f \|_{L^p(\R^n)}.
\]
\end{thm}

\begin{thm}[Bourgain~\cite{BourgainLp}, Carbery~\cite{CarberyLp}]%
\label{TheoMaxi}
For every $p$ in\/ $(3/2, +\infty]$, there exists a constant $\kappa(p)$
such that for every integer $n \ge 1$ and for every symmetric convex 
set~$C \subset \R^n$, one has that
\[
 \forall f \in L^p(\R^n),
 \ms9
     \| \M_C  f \|_{L^p(\R^n)}
 \le \kappa(p) \ms2 \| f \|_{L^p(\R^n)}.
\]
\end{thm}

 We recalled in the Introduction that the maximal theorem of strong type is
not true for $p = 1$, even with a constant depending on $n$, and even for
the smaller function $\M_C^{(d)} \ns1 f$, since 
$\M_C f \le 2^n \M_C^{(d)} \ns1 f$. Note that Theorems~\ref{TheoDyad}
and~\ref{TheoMaxi} are obvious for $L^\infty(\R^n)$, with 
$\kappa^{(d)}(\infty) =\kappa(\infty) = 1$. By Bourgain~\cite{BourgainL2}, we
have the result in $L^2(\R^n)$, so we obtain it for $p \in [2, +\infty]$ by
interpolation. Consequently, our work will be limited to values of~$p$
in the interval $(1, 2]$. We shall follow Carbery's approach to
both theorems. This approach has been applied later in the Detlef M\"uller
article~\cite{MullerQC} (see Section~\ref{AMuller}), on which relies
Bourgain's recent article~\cite{BourgainCube} devoted to the maximal
function associated to high dimensional cubes (see Section~\ref{LeCube}).
\dumou

 The proof will use the inequalities~\eqref{EstimaBour}
and~\eqref{NoticeFirst}, which are also true for log-concave densities, and
by simply following the proofs of Bourgain or Carbery, we can extend the
results to the log-concave setting. As suggested in~\cite{BourgainLp},
one can actually take one more step, forget convexity and exploit only the
inequalities on the Fourier transform given by Lemma~\ref{LEstimatesForC}.
In this more general framework, we consider a probability density $\Kg$ on
$\R^n$,\label{KernelKg} 
or merely a kernel $\Kg$ integrable on~$\R^n$ and having a Fourier
transform $\mg$ which satisfies the following: there exist 
$\delta_{0, g}, \delta_{1, g} > 0$ such that for every 
$\theta \in S^{n-1}$, we have
\begin{subequations}\label{EstimaGeneG}%
\begin{equation}
     \Bigl| 
       \mg(t \theta) \Bigr|
 \le \frac {\delta_{0, g}} 
           {1 +  |t|} \ms1 \up,
     \ms{16}
     \Bigl| \frac {\d} {\d t \ns2} \ms2 
       \mg(t \theta) \Bigr|
  =  \bigl| \theta \ps \nabla \mg (t \theta) \bigr|
 \le \frac {\delta_{1, g}} 
           {1 + |t|} \ms1 \up,
 \quad t \in \R.
 \label{EstimaGene} \tag{\ref{EstimaGeneG}.$\gr H$}
\end{equation}
\end{subequations}
The form of the $\delta_{0, g}$-bound of $\mg$ has been chosen for the sake
of uniformity, but when $\Kg$ is a probability density, we know of course
that $\|\mg\|_{L^\infty(\R^n)} = 1$ and in particular we have 
$\delta_{0, g} \ge 1$ in that case.

\begin{prp}%
\label{MoreGeneral}
Theorems~\ref{TheoDyad} and~\ref{TheoMaxi} are also valid for any symmetric
log-concave probability density $\Klc$ on\/ $\R^n$, namely
\begin{align*}
     \| \M_{\Klc}^{(d)} \ns1 f \|_{L^p(\R^n)}
 &\le \kappa^{(d)}(p) \ms2 \| f \|_{L^p(\R^n)},
 \ms{16} 1 < p \le +\infty,
 \\
     \| \M_{\Klc} f \|_{L^p(\R^n)}
 &\le \kappa(p) \ms2 \| f \|_{L^p(\R^n)},
 \ms{16} 3/2 < p \le +\infty.
\end{align*}
If a probability density $\Kg$ satisfies\/~\eqref{EstimaGene}, then for\/
$3/2 < p \le 2$ we have
\[
     \| \M_{\Kg} \|_{p \rightarrow p}
 \le \kappa_p \ms2 (\delta_{0, g} + \delta_{1, g})^{2 - 2 / p},
\]
and this result extends to every $p \in (1, 2]$ in the case of the dyadic
operator\/ $\M^{(d)}_{\Kg}$.
\end{prp}

 All these results are obvious when $p = +\infty$, and easy when
$p > 2$ by interpolation $(L^2, L^\infty)$ after the case $p = 2$ is
obtained. When $p \le 2$, the log-concave statements follow from the 
\og general\fge one. Indeed, for the study of maximal functions, we may
assume that the convex set $C$ or the symmetric log-concave probability
density $\Klc$ is isotropic and normalized by variance. Then,
by~\eqref{EstimaBour} or by Lemma~\ref{LEstimatesForC}, $m_C$ or~$\mlc$
satisfy~\eqref{EstimaGene} with universal constants $\delta_{0, c}$
and~$\delta_{1, c}$.

\subsection{{\it A priori\/} estimate and interpolation%
\label{EstiPrio}}

\noindent
Suppose that a family $(T_j)_{j \in \Z}$ of operators on 
$L^p(X, \Sigma, \mu)$ is given, for a set of values of $p$ and on a certain
measure space $(X, \Sigma, \mu)$ (further down, it will be $\R^n$, equipped
with the Lebesgue measure). These operators can be linear operators, or
nonlinear operators of the form
\[
 \label{TsubJV}
 T_j f = \sup_{v \in V} |T_{j, v} f|,
\] 
where each $T_{j, v}$ is linear and where $v$ runs over a certain set $V$ of
indices. We want to study the maximal function
\[
   T^* f 
 = \sup_{j \in \Z} |T_j f|
 = \sup_{j \in \Z, \ms1 v \in V} |T_{j, v} f|.
\]
\dumou

 We also consider later a kernel $K$ integrable on $\R^n$. In the
application to the geometrical problem, this kernel will be (as in
Section~\ref{ConcluBour}) the difference $K = K_1 - K_2$ of two kernels,
where $K_1$ is the uniform probability density on an isotropic convex set
$C$ or a probability density $\Kg$ satisfying~\eqref{EstimaGene}, and $K_2$
is a kernel for which the dimensionless maximal inequality is already
known. We have to deal with two cases. In the first one, $T_j$ will be the
convolution with the dilate ${\Di K {2^j}}$ from~\eqref{Dilata} of $K$, and
the maximal function $T^* f = \Mg_K^{(d)} \ns1 f$ will then permit us to
relate the dyadic maximal function $\M_C^{(d)} \ns1 f$ to a maximal
function whose bounded character on $L^p(\R^n)$ is already known. In the
second one, the operator $T_{j ,v}$ will be the convolution 
with~$\Di{K}{v 2^{j}}$ with $v \in [1, 2] = V$, in which case
\begin{equation}
 T_j f = \sup_{2^j \le t \le 2^{j+1}} |{\Di K t} * f|,
 \label{DefTj}
\end{equation}
and $T^* f = \Mg_K f$ allows us to study the \og global\fge maximal
function $\M_C f$ or $\M_{\Kg} f$.
\dumou

 We assume that linear operators $(Q_j)_{j \in \Z}$ such that 
$\sum_{j \in \Z} Q_j = \Id$ are given. In the applications to come, these
operators will be those of Equation~\eqref{LesQjs}, in the
Section~\ref{LittlePal} on Littlewood--Paley functions.

\begin{dfn}[Carbery~\cite{CarberyLp}]
Given families $(T_j)_{j \in \Z}$ and $(Q_j)_{j \in \Z}$ as above, we say
that $T^*$ is \emph{weakly bounded} on $L^p(X, \Sigma, \mu)$ if there
exists a cons\-tant~$A$ such that
\begin{equation}
 \forall f \in L^p(X, \Sigma, \mu),
 \ms{12}
 \forall k \in \Z,
 \ms{12}
     \Bigl\| \sup_{j \in \Z} |T_j Q_{j + k} f| \Bigr\|_{L^p(\mu)}
 \le A \ms2 \|f\|_{L^p(\mu)}.
 \label{grWp} \tag{$\gr W_p$}
\end{equation}
We say that $T^*$ is \emph{strongly bounded\/} on $L^p(X, \Sigma, \mu)$ if
there exists a real nonnegative sequence $(a_k)_{k \in \Z}$, verifying 
$\sum_{k \in \Z} a_k^r < +\infty$ for every $r > 0$, and such that
\begin{equation}
 \forall f \in L^p(X, \Sigma, \mu),
 \ms{12}
 \forall k \in \Z,
 \ms{12}
     \Bigl\| \sup_{j \in \Z} |T_j Q_{j + k} f| \Bigr\|_{L^p(\mu)}
 \le a_k \ms2 \|f\|_{L^p(\mu)}.
 \label{grSp} \tag{$\gr S_p$}
\end{equation}
By $T_j Q_{j+k} f$, we mean of course $T_j (Q_{j+k} f)$.
\end{dfn}

\begin{rems} 
In this generality, the supremum for $v \in V$ in
$T_j f = \sup_{v \in V} |T_{j, v} f|$ must be understood as
essential supremum, as explained in Section~\ref{DefiMaxiFunc}. In our cases
of application, the function $v \mapsto T_{j, v}(x)$, $x \in X$, will be a
continuous function on an interval $V$ of the line, in which case the
pointwise supremum coincides with the supremum on any countable dense
subset of $V$.

 It is evident that~\eqref{grSp} implies~\eqref{grWp}, and~\eqref{grSp}
implies that $T^*$ is bounded, because
\[
     |T_{j, v} f| 
  =  \Bigl| \sum_{k \in \Z} T_{j, v} Q_{j+k} f \Bigr|
 \le \sum_{k \in \Z} |T_{j, v} Q_{j+k} f|
 \le \sum_{k \in \Z} |T_j Q_{j+k} f|,
\]
thus
\[
     |T_j f| 
  =  \sup_{v \in V} |T_{j, v} f| 
 \le \sum_{k \in \Z} |T_j Q_{j+k} f|,
 \ms{18} \hbox{then} \ms{18}
     T^* f
 \le \sum_{k \in \Z} \ms2 \sup_{j \in \Z} |T_j Q_{j+k} f|
\]
and
\begin{equation}
     \|T^* f\|_{L^p(\mu)}
 \le \sum_{k \in \Z} \ms2
      \Bigl\| \sup_{j \in \Z} |T_j Q_{j + k} f| \Bigr\|_{L^p(\mu)}
 \le \Bigl( \sum_{k \in \Z} a_k \Bigr) \ms2 \|f\|_{L^p(\mu)}.
 \label{SimpliesBdd}
\end{equation}

 If one has $(\gr W_{p_0})$ and $(\gr S_{p_1})$ and if 
$1 / p = (1 - \theta) / p_0 +  \theta / p_1$, with $0 < \theta \le 1$,
then as in~\eqref{HalphaBound} we obtain by interpolation
\[
 \forall f \in L^p(\mu),
 \ms6
 \forall k \in \Z,
 \ms{12}
     \Bigl\| \sup_{j \in \Z} |T_j Q_{j + k} f| \Bigr\|_{L^p(\mu)}
 \le A^{1 - \theta} a_k^\theta \ms2 \|f\|_{L^p(\mu)},
\]
and $\sum_{k \in \Z} A^{(1 - \theta) r} a_k^{\theta r} < +\infty$ for every
$r > 0$, so $(\gr S_p)$ is satisfied. In order to obtain this, we apply
the complex interpolation of linear operators between spaces
$L^p(\ell^q)$~\cite[Chap.~5, Th.~5.1.2]{BerLof}. Here, the range space is
of the form $L^p(\mu, \ell^\infty(\Z))$, a case covered by complex
interpolation. Indeed, in the simpler case where the $T_j \ms1$s are
linear, we obtain the result by considering for each $k \in \Z$ the linear
operator
\[
 f 
 \mapsto (T_j Q_{j+k} f)_{j \in \Z} 
  \in L^p(X, \Sigma, \mu, \ell^\infty(\Z)),
 \quad
 f \in L^p(\mu).
\]
If $V$ has more than one element, the range space will be 
$L^p(\mu, \ell^\infty(\Z \times V))$. The nonlinear operator
$f \mapsto \sup_{j \in \Z} |T_j Q_{j + k} f|$ belongs to the class of
\emph{linearizable operators} considered in~\cite{GarciRubio}.
\end{rems}

 We now describe the assumptions that will be made in the main result of
this section. First of all, we assume that there exist constants $C_p$, 
$1 < p \le 2$, such that
\begin{equation}
 \forall p \in (1, 2],
 \ms8
 \forall f \in L^p(\mu),
 \ms{12} 
     \Bigl\| \bigl( \sum_{j \in \Z} |Q_j f|^2 \bigr)^{1/2} \Bigr\|_{L^p(\mu)}
 \le C_p \ms1 \|f\|_{L^p(\mu)}.
 \label{A0} \tag{$A_{\gr 0}$}
\end{equation}
If the $(Q_j)_{j \in \Z}$ are those of~\eqref{LiPaIneq}, then we can take
$C_p = \textrm{q}_p$ which behaves as $1 / (p-1)$ when $p \to 1$,
according to~\eqref{EstiQp}.
\dumou

 We assume that $T_{j, v} = U_{j, v} - S_{j, v}$, where $U_{j, v}$ and
$S_{j, v}$ are positive linear operators, and we assume for $S^*$, defined
by $S^* f = \sup_{j \in \Z, \ms2 v \in V} |S_{j, v} f|$, that there exist
$p_{\rm min}$ in the open interval $(1, 2)$ and constants $C'_p$, 
$p_{\rm min} < p \le 2$, such that
\begin{equation}
 \forall p \in (p_{\rm min}, 2],
 \ms{10} 
 \|S^*\|_p \le C'_p,
 \label{A1} \tag{$A_{\gr 1}$}
\end{equation}
where $\|R\|_p$ is a shorter notation for the norm 
$\|R\|_{p \rightarrow p}$ of an operator $R$. The condition \og $U_{j, v}$
\emph{positive}\fge will be the only reason for requiring that the
kernel~$\Kg$ in Proposition~\ref{MoreGeneral} be a probability density rather
than an arbitrary integrable kernel. The $U_{j, v}\ms1$s will correspond
to the kernel $\Kg$ under study, while the $S_{j, v}\ms1$s will often refer
to Poisson kernels for which the maximal function estimates in $L^p(\R^n)$
are already known by~\eqref{MaxiPoiss}.
\dumou

 We assume that for every $p \in (p_{\rm min}, 2]$, there exists a constant
$C''_p$ such that
\begin{equation}
 \forall j \in \Z,
 \ms{10} 
 \|T_j\|_p \le C''_p.
 \label{A2} \tag{$A_{\gr 2}$}
\end{equation}
\dumou

 We shall assume that $T^*$ verifies $(\gr S_2)$, hence we have that
\begin{equation}
 \forall f \in L^2(\mu),
 \ms4
 \forall k \in \Z,
 \ms9
     \Bigl\| \sup_{j \in \Z} |T_j Q_{j+k} f| \Bigr\|_{L^2(\mu)}
 \le a_k \ms1 \|f\|_{L^2(\mu)},
 \label{A3} \tag{$A_{\gr 3}$}
\end{equation}
where $\sum_{k \in \Z} a_k^r < +\infty$ for every $r > 0$.

\begin{prp}[Carbery~\cite{CarberyLp}]%
\label{PropoPrio}
Under the assumptions\/~\eqref{A0}, \eqref{A1}, \eqref{A2} and\/ \eqref{A3},
the maximal operator~$T^*$ is bounded on $L^p(X, \Sigma, \mu)$ for every
real number $p$ such that $p_{\rm min} < p \le 2$. For every $p_0$ such
that $p_{\rm min} < p_0 < p \le 2$, we have
\begin{equation}
     \|T^*\|_p
 \le (C_{r_0})^{ 2 \ms1 \gamma / p_0 }
       \ms2 (C''_{p_0})^\gamma
        \Bigl( \sum_{k \in \Z} a_k^{ (1 - \gamma)p / 2} \Bigr)^{ 2 / p }
         + 2 \ms2 C'_p,
 \label{ExplicitBound}
\end{equation}
with $r_0 = 2 p / (p + 2 - p_0) \in (p_0, p)$ and 
$\gamma = [1/p - 1/2] / [1/p_0 - 1/2]$.
\end{prp}

 Our main interest in applications will be the maximal operator $U^*$, which
is also bounded on $L^p(X, \Sigma, \mu)$ since $S^*$ is bounded on 
$L^p(X, \Sigma, \mu)$ according to~\eqref{A1}.

\begin{proof}
Under the assumption~\eqref{A3}, one already knows by~\eqref{SimpliesBdd}
that $T^*$ is bounded on~$L^2(X, \Sigma, \mu)$. We fix $p_1 = p$ such that 
$p_{\rm min} < p_1 < 2$ and we try to prove that~$T^*$ is bounded
on~$L^{p_1}(X, \Sigma, \mu)$. For doing this, it is enough to show that for
every \emph{finite} subfamily $(T_j)_{j \in J}$ of $(T_j)_{j \in \Z}$, the
corresponding maximal operator
\[
 f \mapsto \max_{j \in J} |T_j f|
\]
is $L^{p_1}$-bounded  by a constant independent of the chosen finite subset
$J \subset \Z$.

 We thus consider a family $(T_j)$ that has only a finite number of nonzero
terms, implying that $\|T^*\|_{p_1} < +\infty$ by Property~\eqref{A2}. We
choose $p_0$ arbitrary such that $p_{\rm min} < p_0 < p_1$, and we
introduce $r_0$ such that 
$
 p_{\rm min} < p_0 < r_0 < p_1 < r_1 := 2
$,
defined in this way: if $\theta \in (0, 1)$ is such that
\begin{subequations}%
\begin{equation}
   \frac 1 2 
 = \frac {1 - \theta} {p_0} + \frac \theta \infty \ms{0.5} \up,
 \label{numA}
\end{equation}
that is to say, if $\theta = 1 - p_0 / 2$, then we set
\begin{equation}
 \ms{24}
 \frac 1 {r_0} = \frac {1 - \theta} {p_0} + \frac {\theta} {p_1}
 \ms{20}
 \Bigl(
  = \frac 1 2 + \frac 1 {p_1} - \frac {p_0} {2 p_1} \ms1 \up,
  \ms{10} r_0 = \frac{ 2 \ms1 p_1 } { p_1 + 2 - p_0 }
 \Bigr).
 \label{numB}
\end{equation}
\end{subequations}

 Here is the plan: by a first interpolation between $p_0$ and $p_1$, we
will show that~$T^*$ verifies~$(\gr W_{r_0})$ with a constant bounded by a
function of $\|T^*\|_{p_1}$. Next, we will interpolate between 
$(\gr W_{r_0})$ and $(\gr S_{r_1}) = (\gr S_2)$ and obtain $(\gr S_{p_1})$,
giving a new bound for the norm~$\|T^*\|_{p_1}$, whose particular form 
\[
 \|T^*\|_{p_1} \le A (\|T^*\|_{p_1} + B)^\beta,
 \ms{18} \hbox{for some} \ms{12}
 \beta \in (0, 1),
\] 
implies that~$\|T^*\|_{p_1}$ is bounded by a constant independent of the
chosen finite subfamily. This will complete the proof.
\dumou

 For $1 \le r, s \le +\infty$, let $\kappa(r, s)$ be the smallest constant
such that
\[
     \Bigl\|
      \bigl( \sum_{j \in \Z} |T_j g_j|^s \bigr)^{1/s} 
     \Bigr\|_{L^r}
 \le \kappa(r, s) \ms2
      \Bigl\|
       \bigl( \sum_{j \in \Z} |g_j|^s \bigr)^{1/s} 
      \Bigr\|_{L^r}
\]
for every sequence $(g_j)_{j \in \Z}$ in $L^r(X, \Sigma, \mu)$.
\dumou

 ---$\ms4$One sees that $\kappa(p_0, p_0) \le C''_{p_0}$, by~\eqref{A2} and
the simple sum-integral inversion 
\[
     \Bigl\| 
      \bigl( \sum_{j \in \Z} |T_j g_j|^{p_0} \bigr)^{1/p_0} 
     \Bigr\|^{p_0}_{L^{p_0}}
  =  \sum_{j \in \Z} \|T_j g_j \|^{p_0}_{L^{p_0}}
 \le (C''_{p_0})^{p_0} \ms2
      \Bigl\| 
       \bigl( \sum_{j \in \Z} |g_j|^{p_0} \bigr)^{1/{p_0}} 
      \Bigr\|^{p_0}_{L^{p_0}}.
\]

 ---$\ms4$One has also 
$\kappa(p_1, +\infty) \le \|T^*\|_{p_1} + 2 \ms1 C'_{p_1}$. Indeed, when
$(W_j)_{j \in \Z}$ is a family of positive operators and
$g = \sup_{j \in \Z} |g_j|$, one has
\[
 |W_j \ms1 g_j| \le W_j |g_j| \le W_j \ms1 g,
 \ms{18}
     \sup_{j \in \Z} |W_j \ms1 g_j| 
 \le \sup_{j \in \Z} W_j \ms1 g.
\]
Because $S_{j, v}$ is positive, we have 
$\sup_{j \in \Z} |S_{j, v} \ms1 g_j| 
 \le \sup_{j \in \Z} S_{j, v} g$ for every $v \in V$, and
letting $S_j g_j = \sup_{v \in V} |S_{j, v} g_j|$ we see according
to~\eqref{A1} that
\[
 \sup_{j \in \Z} S_j \ms1 g_j \le S^* g,
 \ms{12}
     \bigl\| \sup_{j \in \Z} S_j \ms1 g_j \bigr\|_{L^{p_1}}
 \le \| S^* g\|_{L^{p_1}}
 \le C'_{p_1} \bigl\| \sup_{j \in \Z} |g_j| \bigr\|_{L^{p_1}}.
\]
Since $U_{j, v} = T_{j, v} + S_{j, v}$ is positive, we obtain also 
for $U_j f = \sup_{v \in V} |U_{j, v} f|$ that
\[
     \bigl\| \sup_{j \in \Z} |U_j \ms1 g_j| \bigr\|_{L^{p_1}}
 \le \|U^* g\|_{L^{p_1}}
 \le \|T^* g\|_{L^{p_1}} + \|S^* g\|_{L^{p_1}}
 \le (\|T^*\|_{p_1}
      + C'_{p_1} ) \ms2 \| g \|_{L^{p_1}},
\]
and finally
$
     \bigl\| \sup_{j \in \Z} |T_j \ms1 g_j| \bigr\|_{L^{p_1}}
 \le (\|T^*\|_{p_1} + 2 \ms1 C'_{p_1} ) \ms2 
      \bigl\| \sup_{j \in \Z} |g_j| \bigr\|_{L^{p_1}}
$,
which proves the inequality
$\kappa(p_1, +\infty) \le \|T^*\|_{p_1} + 2 \ms1 C'_{p_1}$.
\dumou

 We apply complex interpolation between spaces 
$L^p(\ell^q)$~\cite[Chap.~5, Th.~5.1.2]{BerLof}, namely between the spaces
$L^{p_0}(\ell^{p_0})$ and $L^{p_1}(\ell^{\infty})$, which gives the space
$L^{r_0}(\ell^2)$ for the chosen value~$\theta$ of the interpolation
parameter, by~\eqref{numA} and~\eqref{numB}. We already explained that the
case where $T_j$ is not linear can also be covered by complex
interpolation. It follows from~\eqref{HalphaBound} that
\[
     \kappa(r_0, 2) 
 \le \kappa(p_0, p_0)^{1 - \theta} \kappa(p_1, +\infty)^\theta
 \le (C''_{p_0})^{1 - \theta} (\|T^*\|_{p_1} + 2 \ms1 C'_{p_1})^{\theta}.
\]

\begin{smal}
\noindent\textbf{Remark} (in passing).
It is exactly in this manner that
Stein~\cite[Chap.~VI, Theorem~8, p.~103]{SteinTHA} shows the
inequality~\eqref{SteinIneq} on the square function 
$(\sum_n |E_n f_n|^2)^{1/2}$ of a sequence $(E_n)$ of conditional
expectations with respect to an increasing sequence of $\sigma$-fields,
stating that
\begin{subequations}\label{SteinIneq}
\SmallDisplay{\eqref{SteinIneq}}
\[
     \Bigl\| \bigl(\sum_n |E_n f_n|^2\bigr)^{1/2} \Bigr\|_q
 \le \kappa_q \ms2 
      \Bigl\| \bigl( \sum_n |f_n|^2 \bigr)^{1/2} \Bigr\|_q,
 \quad
 1 < q < +\infty.
\]
\end{subequations}
\noindent
When $1 < q < 2$, the proof applies inversion for a pair $(q_0, q_0)$, and
Doob's maximal theorem for a pair $(q_1, +\infty)$ with $q_0 < q < q_1$ and
$q (q_1 - q_0) = 2 (q_1 - q)$.

\end{smal}

\noindent
Thus, with $g_j = Q_{j + k} f$ for a fixed $k \in \Z$, one has 
\begin{align*}
     \Bigl\| \sup_{j \in \Z} |T_j Q_{j+k} f| \Bigr\|_{L^{r_0}}
 &\le \Bigl\| 
      \bigl( \sum_{j \in \Z} |T_j Q_{j+k} f|^2 \bigr)^{1/2} 
     \Bigr\|_{L^{r_0}}
 \le \kappa(r_0, 2) \ms2
      \Bigl\| \bigl( \sum_{j \in \Z} |Q_{j+k} f|^2 \bigr)^{1/2} 
      \Bigr\|_{L^{r_0}}
 \\
  &=  \kappa(r_0, 2) \ms2
      \Bigl\| 
       \bigl( \sum_{j \in \Z} |Q_j f|^2 \bigr)^{1/2} 
      \Bigr\|_{L^{r_0}}
 \le C_{r_0} \ms2 \kappa(r_0, 2) \ms2 \|f\|_{L^{r_0}}.
\end{align*}
We have proved the property $(\gr W_{r_0})$, since we got that
\[
 \forall f \in L^{r_0},
 \ms5
 \forall k \in \Z,
 \ms{16}
     \Bigl\| \sup_{j \in \Z} |T_j Q_{j+k} f| \Bigr\|_{L^{r_0}}
 \le C_{r_0} \ms2 \kappa(r_0, 2) \ms2 \|f\|_{L^{r_0}}.
\]
If for a certain $\rho \in (0, 1)$, we write
\[
   \frac 1 {p_1} 
 = \frac {1 - \rho} {r_0} + \frac \rho {r_1}
 = \frac {1 - \rho} {r_0} + \frac \rho 2
 \ms{28}
 \Bigl( \rho = \frac {p_1 - p_0} {2 - p_0} \Bigr),
\]
we get $(\gr S_{p_1})$ by interpolating between $(\gr W_{r_0})$ and 
$(\gr S_2) = (\gr S_{r_1})$, obtaining thus
\[
 \forall k \in \Z,
 \ms{12}
     \Bigl\| \sup_{j \in \Z} |T_j Q_{j+k} f| \Bigr\|_{L^{p_1}}
 \le (C_{r_0} \ms2 \kappa(r_0, 2))^{1 - \rho} 
      a_k^\rho \ms2 \|f\|_{L^{p_1}}.
\]
By~\eqref{SimpliesBdd}, it follows that 
\[
     \Bigl\| \sup_{j \in \Z} |T_j f| \Bigr\|_{L^{p_1}}
 \le (C_{r_0} \ms2 \kappa(r_0, 2))^{1 - \rho} 
      \Bigl( \sum_{k \in \Z} a_k^\rho \Bigr) \ms2 \|f\|_{L^{p_1}}.
\]
One has finally an implicit inequality about $\|T^*\|_{p_1}$, namely
\begin{align*}
     \|T^*\|_{p_1}
 &\le \bigl[ C_{r_0} \ms2 \kappa(r_0, 2) \bigr]^{1 - \rho} 
      \Bigl( \sum_{k \in \Z} a_k^\rho \Bigr)
 \le \bigl[ C_{r_0} \ms2 (C''_{p_0})^{1 - \theta} 
     (\|T^*\|_{p_1} + 2 \ms1 C'_{p_1})^{\theta} \bigr]^{1 - \rho} 
      \Bigl( \sum_{k \in \Z} a_k^\rho \Bigr)
 \\
  &=  \bigl( C_{r_0} \ms2 (C''_{p_0})^{1 - \theta} \bigr)^{1 - \rho}
      \Bigl( \sum_{k \in \Z} a_k^\rho \Bigr)
       (\|T^*\|_{p_1} + 2 \ms1 C'_{p_1})^{\theta (1 - \rho)}, 
\end{align*}
implying that $\|T^*\|_{p_1}$ is bounded by a constant depending only upon
$C_{r_0}$, $C'_{p_1}$, $C''_{p_0}$ and the $a_k\ms1$s. Indeed, suppose that
$C \ge 0$ verifies $C \le A (C + B)^\beta$, where $A, B > 0$ and 
$0 < \beta < 1$. We write
\[
     C 
 \le \bigl( A^{1 / (1 - \beta)} \bigr)^{1 - \beta} (C + B)^\beta
 \le (1 - \beta) A^{1 / (1 - \beta)} + \beta (C + B),
\]
yielding
\[
 C \le A^{1 / (1 - \beta)} + \frac \beta {1 - \beta} \ms2 B.
\]
This bound is essentially correct when $B$ is small, and we shall use it
below with $A
 = \bigl( C_{r_0} \ms2 (C''_{p_0})^{1 - \theta} \bigr)^{1 - \rho}
    \Bigl( \sum_{k \in \Z} a_k^\rho \Bigr)$,
$B = 2 \ms1 C'_{p_1}$ and $\beta = \theta(1 - \rho)$.

\begin{smal}%
\noindent
However, when $B \ge A^{1 / (1-\beta)}$, a better bound
$(1 - \beta)^{-1}  A B^\beta$ is available. In this case, 
$A \le B^{1-\beta}$, thus $C \le B^{1-\beta} (C + B)^\beta \le B + \beta C$,
hence
\[
     C 
 \le A \Bigl( \frac B { 1 - \beta } + B \Bigr)^\beta
  =  \Bigl( \frac {2 - \beta} {1 - \beta} \Bigr)^\beta A B^\beta
 \le \frac 1 {1 - \beta} \ms2 A B^\beta,
\]
because $(2 - \beta)^\beta (1 - \beta)^{1 - \beta}
 \le \beta (2 - \beta) + (1 - \beta)^2 = 1$.

\end{smal}

\noindent
Recall that $\rho = (p_1 - p_0) / (2 - p_0)$, so
$\beta = \theta (1 - \rho) = 1 - p_1 / 2 < 1$. We find an explicit bound for
$\|T^*\|_{p_1}$, independent of the finite subfamily $(T_j)_{j \in J}$ of
$(T_j)\ms1$s that was chosen at the beginning, of the form
\begin{align*}
     \|T^*\|_{p_1}
 &\le \bigl( C_{r_0} \ms2 (C''_{p_0})^{1 - \theta} 
      \bigr)^{ 2 (1 - \rho) / p_1 }
        \Bigl( \sum_{k \in \Z} a_k^\rho 
        \Bigr)^{ 2 / p_1 }
         + \frac {2 - p_1} {p_1} \ms2 2 \ms2 C'_{p_1}       
 \\
 &\le (C_{r_0})^{ 2 \ms1 \gamma / p_0 }
       \ms2 (C''_{p_0})^\gamma
        \Bigl( \sum_{k \in \Z} a_k^\rho \Bigr)^{ 2 / p_1 }
         + 2 \ms2 C'_{p_1},
\end{align*}
with $\gamma = [1/p_1 - 1/2] / [1/p_0 - 1/2]$. Observe that
$\rho = [p_1 / (2 p_0) - 1/2] / [1/p_0 - 1/2]
 = (1 - \gamma) p_1 / 2$.
We get in particular a bound of $C''_{p_1}$ by a power $\gamma < 1$ of
$C''_{p_0}$. There is no miracle: this power $\gamma$ is the one
corresponding to interpolation between~$C''_{p_0}$ and the value $C''_2$
hidden in the assumption~\eqref{A3}.
\end{proof}

\subsection{Fractional derivatives%
\label{FractiDeri}}

\noindent 
If a function $h$ is given in the Schwartz space $\ca S(\R)$, one can
express it as Fourier transform of another function $k \in \ca S(\R)$ and
write
\[
 \forall t \in \R,
 \ms{12}
 h(t) = \int_\R k(s) \e^{ - 2 \ii \pi s t } \, \d s.
\]
One has then an expression for the derivatives of $h$ by means of
(unbounded) multipliers. For every integer $j \ge 1$ and every $t \in \R$,
one sees that
\[
   (-1)^j h^{(j)}(t) 
 = \int_\R ( 2 \ii \pi s)^j \ms1 k(s) \e^{ - 2 \ii \pi s t } \, \d s.
\]
It is tempting to extend the notion of derivative, from the integer case
$j \in \N$ to every complex value $z$ such that $\Re z > - 1$, by setting
\begin{equation}
 \forall t \in \R,
 \ms{12}
   (D^z h)(t) 
 = \int_\R ( 2 \ii \pi s)^z k(s) \e^{ - 2 \ii \pi s t } \, \d s.
 \label{DalphaDef}
\end{equation}
Note that $D^1 h = - h'$ with this definition. We define complex powers by
\[
   ( 2 \ii \pi s)^z 
 = \e^{ z \ln(2 \ii \pi s)}
 = \e^{ z ( \ln(2 \pi |s|) + \ii \Arg(2 \ii \pi s) ) }
 = |2 \pi s|^z \e^{ \ii \pi z \sign(s) / 2},
\]
and we have that $(\lambda \ii s)^z = \lambda^z (\ii s)^z$ when
$\lambda > 0$. If we dilate the function $h$ to~$\di h \lambda$, with
$\lambda > 0$ as in~\eqref{Dilata}, we know that
$
   \di h \lambda
 = \ca F( \Di k \lambda)
$,
therefore
\[
   (D^z {\di h \lambda} )(t)
 = \int_\R ( 2 \ii \pi s)^z \lambda^{-1} k(\lambda^{-1} s) 
    \e^{ - 2 \ii \pi s t} \, \d s
 = \lambda^z \int_\R ( 2 \ii \pi u)^z k(u) 
    \e^{ - 2 \ii \pi u \lambda t} \, \d u.
\]
This means that
\begin{equation}
   D^z (\di h \lambda)
 = \lambda^z {\di {\bigl( D^z h \bigr)} \lambda},
 \ms{18} \hbox{or} \ms{18}
   D^z_t h(\lambda t) 
 = \lambda^z (D^z h)(\lambda t),
 \label{Dilate}
\end{equation}
where we use the notation $D^z_t h(\lambda t)$ when the function of $t$
does not have an explicit name, as in $t \mapsto h(\lambda t)$. For a
specific value, we shall write for example\label{SpeciVal} 
$D^z_t h(\lambda t) \barre_{t = 1}$.

\begin{smal}
\def\sea{\raise 1.3pt \hbox{$\scriptstyle \searrow$} }
\noindent
If we would like to extend $D^z$ to $h = \gr 1$, we might consider the
function $\gr 1$ as the limit of $\di h \lambda$ when $h(0) = 1$ and
$\lambda \ms1 \sea \ms4 0$. Then~\eqref{Dilate} suggests that 
$D^z \gr 1 = 0$ when $\Re z > 0$, and that $D^z \gr 1$ is undefined if 
$\Re z < 0$.

 When $z$ is not a nonnegative integer, the operator $D^z$ is not
local. We will see later however that $(D^z h)(t_0)$ depends only on
the values of $h$ on $[t_0, +\infty)$. This could be checked right now by
arguments involving holomorphic functions.

\end{smal}

\noindent
When $-1 < \Re z < 0$, the differentiation $D^z$ is in fact a 
\emph{fractional integration}. We shall see below that 
$(D^z h)(t) = (I^{-z} h)(t)$, where $I^w$ is given for $\Re w > 0$
by\label{FractiIntA}
\begin{equation}
   (I^w h)(t)
 = \frac 1 {\Gamma(w)} 
    \int_t^{+\infty} (u - t)^{w - 1} h(u) \, \d u.
 \label{IntegraZero}
\end{equation}
The next lemma provides the tool that relates the
definitions~\eqref{DalphaDef} and~\eqref{IntegraZero}.

\begin{lem}%
\label{GammaSG}
Let $\zeta$ be a complex number such that\/ $\Re \zeta < 0$ and let 
$\varepsilon > 0$. The inverse Fourier transform of the function
$ 
 t \mapsto
 \Gamma(-\zeta)^{-1} 
  \gr 1_{(-\infty, 0)}(t) \ms2 (-t)^{-\zeta-1} \e^{\varepsilon t}
$
is equal to $s \mapsto (\varepsilon + 2 \ii \pi s)^\zeta$, namely
\[
   \frac 1 {\Gamma(-\zeta)} 
    \int_\R  \gr 1_{(-\infty, 0)}(t) \ms1 (-t)^{-\zeta-1} 
     \e^{\varepsilon t} \e^{2 \ii \pi s t} \, \d t
 = (\varepsilon + 2 \ii \pi s)^\zeta,
 \quad s \in \R.
\]
\end{lem}

\begin{proof}
By a contour integral of $(-z)^{-\zeta - 1} \e^{z}$, running along the
negative real half-line and along the half-line 
$H_s = \{ (\varepsilon + 2 \ii \pi s) \ms1 t \in \C : t < 0 \}$, we obtain
\[
   \Gamma(-\zeta)
 = \int_{-\infty}^0 (-t)^{-\zeta - 1} \e^t \, \d t
 = (\varepsilon + 2 \ii \pi s)^{-\zeta}
    \int_{-\infty}^0 (-t)^{-\zeta - 1} \e^{(\varepsilon + 2 \ii \pi s) t} 
     \, \d t,
\]
giving the announced result. 
\end{proof}

 Integrating~\eqref{IntegraZero} by parts, we see that
\[
   (I^w h)(t) 
 = - \frac 1 {\Gamma(w + 1)}
      \int_t^{+\infty} (u - t)^{w} h'(u) \, \d u.
\]
This new formula makes sense for $\Re w > -1$ and could be used for
defining the fractional derivative $D^z$ if $z = - w$ and 
$\Re w \in (-1, 0)$, by setting for $t$ real
\begin{equation}
   (D^z h)(t)
 = - \frac 1 {\Gamma(1 -z)}
      \int_t^{+\infty} (u - t)^{- z} h'(u) \, \d u.
 \label{IntegraUn}
\end{equation}
This is proved in Lemma~\ref{GeneralD}. It is coherent with
the fact that $D^\alpha$, for $0 < \alpha < 1$, can be considered as the
antiderivative of order $1 - \alpha$ of the derivative $D^1 h = -h'$,
\[
   D^\alpha h 
 = D^{\alpha - 1} D^1 h 
 = - D^{\alpha - 1} h'
 = - I^{1 - \alpha} h'.
\]

\begin{smal}
\noindent
The operation $D^z$ is not symmetric on~$\R$; this is obvious from the
formulas for~$I^w$. The choice that was done of $(2 \ii \pi s)^z$ instead
of $( - 2 \ii \pi s)^z$ in~\eqref{DalphaDef} induces the direction in which
the fractional antiderivative is computed. This direction, to~$+\infty$, is
well adapted to the \og radial\fge Carbery's method introduced
in~\cite{Carbery}.

\end{smal}

\begin{lem}%
\label{GeneralD}
Let $\alpha \in (0, 1)$, $t_0 \in \R$ be given and let $k$ be a function
on\/ $\R$ such that\/ $(1 + |s|^\alpha) k(s)$ is integrable on the real
line. Assume that $h = \widehat k$ is Lipschitz with\/
$|h'(t)| \le \kappa_1 (1 + |t|)^{-1}$ for almost every $t \ge t_0$. Then,
for every $t > t_0$ and $z$ such that\/ $\Re z = \alpha$, we have
\[
   - \frac 1 {\Gamma(1 - z)}
    \int_t^{+\infty} (u - t)^{-z} h'(u) \, \d u
 = \int_\R (2 \ii \pi s)^z k(s) \e^{- 2 \ii \pi s t} \, \d s.
\]
\end{lem}

\begin{proof} 
Let $\eta$ be a nonnegative $C^\infty$ function on $\R$, with integral $1$
and with compact support in $[-1, 1]$. Consider $\varepsilon \in (0, 1)$ and
\[
   k_\varepsilon(s) 
 = k(s) {\di {(\eta^\vee)} \varepsilon} (s)
 = k(s) \ms1 \eta^\vee (\varepsilon s),
 \quad
 s \in \R.
\] 
Then $\eta^\vee \in \ca S(\R)$, $s k_\varepsilon(s)$ is integrable and
$h_\varepsilon := \widehat{k_\varepsilon} = h * {\Di \eta \varepsilon}$ 
is~$C^1$. We can write
\begin{equation}
   - h'_\varepsilon(t)
 = \int_\R 2 \ii \pi s \ms1 k_\varepsilon(s) \e^{- 2 \ii \pi s t} \, \d s,
 \quad
 t \in \R.
 \label{HprimeEpsA}
\end{equation}
Since $h$ is Lipschitz, we also know that 
${h'_\varepsilon} = h' * {\Di \eta \varepsilon}$. Fix 
$t \ge t_0 + \varepsilon$. When $|\tau| \le 1$ and $u \ge t$, we have
$u - \varepsilon \tau \ge t_0$,
$1 + |u| \le 1 + \varepsilon |\tau| + |u - \varepsilon \tau|
 \le 2 + 2 \ms1 |u - \varepsilon \tau|$, so
\begin{equation}
     |h'_\varepsilon(u)|
  =  \Bigl| 
      \int_{-1}^1 h'(u - \varepsilon \tau) \eta(\tau) \, \d \tau 
     \Bigr|
 \le \kappa_1
      \int_{-1}^1 
       \frac {\eta(\tau)} { 1 + |u - \varepsilon \tau| } \, \d \tau
 \le \frac {2 \kappa_1} {1 + |u|} \up.
 \label{HprimeEps}
\end{equation}
Applying~\eqref{HprimeEpsA} and $|(u - t)^{-z}|  = (u - t)^{-\alpha}$,
Fubini's theorem and the inverse Fourier transform of 
$v \mapsto [(-v)_+]^{-z} \e^{\varepsilon v}$ given by
Lemma~\ref{GammaSG} with $\zeta = z - 1$, we get
\begin{align*}
   & \ms2 - \frac 1 {\Gamma(1 - z)}
    \int_t^{+\infty} (u - t)^{-z} \e^{\varepsilon (t - u)}
     h'_\varepsilon(u) \, \d u
 \\
 =& \ms4 \frac 1 {\Gamma(1 - z)}
    \int_t^{+\infty} (u - t)^{-z} \e^{\varepsilon (t - u)}
     \Bigl( \int_\R 2 \ii \pi s \ms1 k_\varepsilon(s) \e^{- 2 \ii \pi s u} 
      \, \d s \Bigr) \, \d u
 \\
 =& \ms4 \frac 1 {\Gamma(1 - z)}
    \dint \gr 1_{ \{t - u < 0\} } 
     (u - t)^{-z} \e^{\varepsilon (t - u)}
      2 \ii \pi s \ms1 k_\varepsilon(s) \e^{2 \ii \pi s (t - u)} 
       \e^{ - 2 \ii \pi s t} \, \d s \ms2 \d u
 \\
 =& \ms2 \int_\R (\varepsilon + 2 \ii \pi s)^{z - 1}
    (2 \ii \pi s) k_\varepsilon(s)
     \e^{ - 2 \ii \pi s t}
      \, \d s.
\end{align*}
Letting $\varepsilon$ tend to $0$, by a double application of Lebesgue's
dominated convergence, using~\eqref{HprimeEps} and since 
$h'_\varepsilon(u) \rightarrow h'(u)$ at every Lebesgue point $u$ of $h'$,
we obtain
\[
   - \frac 1 {\Gamma(1 - z)}
    \int_t^{+\infty} (u - t)^{-z} h'(u) \, \d u
 = \int_\R (2 \ii \pi s)^z k(s)
     \e^{ - 2 \ii \pi s t} \, \d s.
\]
\end{proof}

 It is quite comforting to have two possible ways of defining $D^z h$.
However, we will have to handle cases where the Fourier transform $h(t)$ is
well controlled, but where the estimates on $k(s)$ are not so good. We shall
therefore concentrate on the integral definition~\eqref{IntegraUn} of 
$D^z h$. We have to check that the properties obtained with the first
definition remain true when only the second applies.

\begin{smal}
\noindent
When $\alpha \in (0, 1)$ tends to $1$, one has
$\Gamma(1 - \alpha) \simeq (1 - \alpha)^{-1}$ and for $\varepsilon > 0$ we
get
\[
   - \frac 1 {\Gamma(1 -\alpha)}
      \int_{t+\varepsilon}^{+\infty} (u - t)^{- \alpha} h'(u) \, \d u
 \ms2 \rightarrow \ms2 0,
 \ms{12}
   (1 - \alpha)
    \int_t^{t+\varepsilon} (u - t)^{- \alpha} \, \d u
 = \varepsilon^{1 - \alpha}
 \ms2 \rightarrow \ms2 1.
\]
We recover the fact that $(D^1 h)(t) = - h'(t)$, already known by
Fourier. 

 Let us mention the case of $h(t) = \e^{ - \lambda |t|}$, the Fourier
transform of a Cauchy kernel. When $t \ge 0$ and $0 < \Re z < 1$, we have
\begin{subequations}\label{DerivExpo}
\SmallDisplay{\eqref{DerivExpo}}
\[
   D^z_t \e^{- \lambda |t|} 
 = \frac 1 {\Gamma(1 - z)} 
    \ms1 \e^{-\lambda t}
     \int_t^{+\infty} (u - t)^{-z} \lambda \e^{ - \lambda (u - t)}
      \, \d u
 = \lambda^z \e^{-\lambda |t|}.
\]
\end{subequations}
\noindent
The dilation relation~\eqref{Dilate} follows from a simple change of
variable similar to the one in the line above, and is left to the reader.

\end{smal}

\noindent
We have introduced in~\eqref{RightMax} the right maximal function $h^*_r$
of $h$. Notice that for $h$ Lipschitz on $(t_0, +\infty)$ and for every 
$t \ge t_0$, $\delta > 0$, we have
\begin{equation}
     |h(t + \delta)| 
 \le |h(t)| + \int_t^{t+\delta} |h'(u)| \, \d u
 \le h^*_r(t) + \delta \ms1 (h')^*_r(t).
 \label{Hlocal}
\end{equation}

\begin{lem}%
\label{GeneralDb2}
Let $h$ be Lipschitz on\/ $(t_0, +\infty)$, $\alpha \in (0, 1)$ and 
$h(t) = o(t^\alpha)$ at $+\infty$. Let $h_0 = h^*_r$ be the right maximal
function of $h$ and $h_1 = (h')^*_r$ that of $h'$. Then
\[
     |(D^\alpha h)(t)| 
 \le 6 \ms2 h_0(t)^{1 - \alpha} h_1(t)^\alpha,
 \quad 
 t \ge t_0.
\]
If $w$ is complex and\/ $\Re w = \alpha$, then for every $t \ge t_0$
we have
\[
     |(D^w h)(t)| 
 \le \frac 2 {\alpha (1 - \alpha) } \ms2
      \frac {(1 + |w|)^{1 - \alpha} \ns{15}} {|\Gamma(1 - w)|} \ms{10} 
       \ms2 h_0(t)^{1 - \alpha} h_1(t)^\alpha.
\]
\end{lem}

\begin{proof}
For $t \ge t_0$ and $\delta > 0$, we express
$E_\alpha := - \Gamma(1 - \alpha) \ms2 (D^\alpha h)(t)$ as
\[
 \int_t^{t+\delta} (u - t)^{-\alpha} h'(u) \, \d u
    + \int_{t+\delta}^{+\infty} (u - t)^{-\alpha} h'(u) \, \d u.
\]
Applying~\eqref{RightMaxPsi} and integration by parts, we bound each of the
two pieces
\begin{align*}
      |E_\alpha|
 &\le \frac {\delta^{1 - \alpha}} {1 - \alpha}
       \ms2 h_1(t)
        + \Bigl|
           \Bigl[
            (u - t)^{-\alpha} h(u) 
           \Bigr]_{t+\delta}^{+\infty} 
          \Bigr|
       + \alpha 
          \Bigl|
           \int_{t+\delta}^{+\infty} (u - t)^{-\alpha - 1} h(u) \, \d u 
          \Bigr|
 \\
 &\le \frac {\delta^{1 - \alpha}} {1 - \alpha} \ms2 h_1(t) +
      \delta^{-\alpha} |h(t + \delta)|
 \\
    & \ms{18} + \alpha 
       \Bigl(
        \int_t^{t+\delta} \delta^{-\alpha-1} |h(u)| \, \d u
         + \int_{t+\delta}^{+\infty} (u - t)^{-\alpha - 1} |h(u)| \, \d u 
       \Bigr).
\end{align*}
By~\eqref{Hlocal}, by~\eqref{RightMaxPsi} for the non-decreasing
function $\psi$ defined by $\psi(u) = \delta^{-\alpha-1}$ when
$u \in [t, t + \delta]$ and $\psi(u) = (u - t)^{-\alpha-1}$ for 
$u \ge t + \delta$, we obtain
\begin{align*}
      |E_\alpha|
 &\le  \frac {\delta^{1 - \alpha}} {1 - \alpha} \ms2 h_1(t) +
      \delta^{-\alpha} \bigl( h_0(t) + \delta h_1(t) \bigr) +
      (1 + \alpha) \ms1 \delta^{-\alpha} h_0(t)
 \\
 & =  \frac {2 - \alpha} {1 - \alpha} 
       \ms4 \delta^{1 - \alpha} \ms2 h_1(t)
        + (2 + \alpha) \ms1 \delta^{-\alpha} h_0(t)
 \le  \frac 2 {1 - \alpha} 
       \ms4 \delta^{1 - \alpha} \ms2 h_1(t)
        + 3 \ms1 \delta^{-\alpha} h_0(t).
\end{align*}
We choose $\delta = \delta_0 = h_0(t) / h_1(t)$ and get that
\[
     |E_\alpha|
 \le \Bigl( \frac 2 {1 - \alpha} + 3 \Bigr)
       h_0(t)^{1 - \alpha} h_1(t)^\alpha.
\]
Recalling $\Gamma(1 - \alpha) \ge 1$ and the minimal value 
$\Gamma(\xG) > 0.88$ in~\eqref{MiniGamm} we have
\begin{align*}
      |(D^\alpha h)(t)|
 &\le \Bigl(
       \frac 2 {\Gamma(2 - \alpha)} + \frac 3 {\Gamma(1 - \alpha)} 
      \Bigr)
       h_0(t)^{1 - \alpha} h_1(t)^\alpha
 \\
 &\le \Bigl( \frac 2 {\Gamma(\xG)} + 3 \Bigr)
       h_0(t)^{1 - \alpha} h_1(t)^\alpha
 \le 6 \ms2 h_0(t)^{1 - \alpha} h_1(t)^\alpha.
\end{align*}
\dumou

 When $w$ is complex and $\Re w = \alpha$, we use $|(u - t)^{-w}|
 = (u - t)^{- \alpha}$, the same integration by parts,
$|(u - t)^{-w-1}| = (u - t)^{- \alpha-1}$ and we get
\begin{align*}
      |E_w|
 &\le \frac {\delta^{1 - \alpha}} {1 - \alpha} \ms2 h_1(t)
        + \delta^{-\alpha} \bigl( h_0(t) + \delta h_1(t) \bigr)
         + |w| \Bigl( 1 + \frac 1 \alpha \Bigr) 
            \ms1 \delta^{-\alpha} h_0(t)
 \\
 &\le \frac {2 \ms1 \delta^{1 - \alpha}} {1 - \alpha} \ms2 h_1(t)
        + \frac 2 \alpha \ms2 (1 + |w|) \ms1 \delta^{-\alpha} h_0(t).
\end{align*}
Choosing $\delta =  (1 + |w|) \ms1 h_0(t) / h_1(t)$ we obtain the
announced result.
\end{proof}

 In what follows, we shall consider the following assumptions on a function
$h$:
\begin{equation}
 \left \{
 \begin{matrix}
     \ms6 h & \hbox{is Lipschitz on} \ms6 [t_0, +\infty), \hfill
  \\
     \ms6 |h(t)| 
  &\le \kappa_0 (1 + |t|)^{-1}
  \hbox{ for }
  t \ge t_0, \hfill
  \\
     \ms6 |h'(t)| 
  &\le \kappa_1 (1 + |t|)^{-1} 
  \hbox{ for almost every } 
  t \ge t_0. \hfill
 \end{matrix}
 \right.
 \label{MesHypos}
\end{equation}

\begin{cor}%
\label{GeneralDb3}
Suppose that the function $h$ defined on\/ $(t_0, +\infty)$, $t_0 \ge 0$,
satisfies\/~\eqref{MesHypos}. Then for every $\alpha \in (0, 1)$, we have
\[
     |(D^\alpha h)(t)| 
 \le 6 \ms2 
      \frac {\kappa_0^{1 - \alpha} \kappa_1^\alpha}
            {1 + |t|},
 \quad t \ge t_0.
\]
\end{cor}

\begin{proof}
The two upper bounds in~\eqref{MesHypos} are decreasing functions of
$t \in [t_0, +\infty)$, hence they also bound $h^*_r$ or $(h')^*_r$. We
conclude by applying Lemma~\ref{GeneralDb2}.
\end{proof}

 Assuming that $h$ has enough derivatives and continuing integrations by
parts, starting from~\eqref{IntegraUn}, we get successive formulas 
for~$D^z h$ for each integer $j > 0$, which make sense when $\Re z < j$.
Let $z = j - 1 + w$, with $j \ge 1$ and $\Re w \in (0, 1)$. We obtain that
\[
   (D^z h)(t)
 = (-1)^{j-1} (D^w h^{(j-1)}) (t)
 = \frac {(-1)^j} {\Gamma(1 - w)} 
    \int_t^{+\infty} (u - t)^{- w} h^{(j)}(u) \, \d u,
\]
and for every $z \in \C$ such that $\Re z < j$, we have
\begin{equation}
   (D^z h)(t)
 = \frac {(-1)^j} {\Gamma(j - z)} 
    \int_t^{+\infty} (u - t)^{- z + j - 1} h^{(j)}(u) \, \d u.
 \label{DzOrdrek}
\end{equation}
By gluing the successive definitions, we define entire functions of $z$ for
every fixed~$t$ and $h \in \ca S(\R)$. By the principle of analytic
continuation, we conclude that the integral formula for $D^z h$ coincides
when $\Re z > - 1$ with the one obtained by Fourier transform (a fact that
we have checked in Lemma~\ref{GeneralD} when $0 < \Re z < 1$).

\begin{lem}\label{IaDa}
Let $\alpha$ be in\/ $(0, 1)$. Suppose that the function $h$ satisfies the
assumptions\/~\eqref{MesHypos} on\/ $[t_0, +\infty)$, $t_0 \ge 0$, and define
$D^\alpha h$ by\/~\eqref{IntegraUn}. We have that
\[
 (I^\alpha D^\alpha h)(t) = t,
 \quad t \ge t_0.
\]
\end{lem}

\begin{proof}
We first assume in addition that
\[
 \int_{t_0}^{+\infty} |h'(u)| \, \d u < +\infty,
 \ms{18} \hbox{thus} \ms{18}
 h(t) = - \int_t^{+\infty} h'(u) \, \d u 
\]
for every $t \ge t_0$ since $h$ is Lipschitz. For $u \ge t_0$, accepting
possibly infinite integrals of nonnegative measurable functions, set
\[
 G(u) = \frac 1 {\Gamma(1 - \alpha)} 
         \int_u^{+\infty} (v - u)^{-\alpha} |h'(v)| \, \d v. 
\]
When $h$ is decreasing on $(t_0, +\infty)$, the function $G$ is equal to
$D^\alpha h$, and $|D^\alpha h| \le G$ in general. Then, consider $F$,
equal to $I^\alpha G$ in good cases, defined for $t \ge t_0$ by
\begin{align*}
    F(t)
 &:= \frac 1 {\Gamma(\alpha)} 
     \int_t^{+\infty} (u - t)^{\alpha - 1} G(u) \, \d u
 \\
 &= \frac 1 {\Gamma(\alpha) \Gamma(1 - \alpha)} 
    \int_t^{+\infty} (u - t)^{\alpha - 1} 
     \int_u^{+\infty} (v - u)^{-\alpha} |h'(v)| \, \d v \, \d u
 \\
 &= \frac 1 {\Gamma(\alpha) \Gamma(1 - \alpha)} 
    \int_t^{+\infty} 
     \Bigl( \int \gr 1_{t \le u \le v} \ms2 (u - t)^{\alpha - 1} 
      (v - u)^{-\alpha} \, \d u \Bigr) |h'(v)| \, \d v.
\end{align*}
Setting $u = t + y (v - t)$, one gets with 
$\gamma_\alpha = \Gamma(\alpha) \Gamma(1 - \alpha)$ that
\[
   F(t)
 = \gamma_\alpha^{-1} 
    \Bigl( \int_0^1 y^{\alpha - 1} (1 - y)^{-\alpha} \, \d y \Bigr)
     \ms1 \int_t^{+\infty} |h'(v)| \, \d v
 = \int_t^{+\infty} |h'(v)| \, \d v
 < +\infty.
\]
The last equality can be deduced from~\eqref{DerivExpo} by applying the
preceding computation to $h(v) = \e^{ - |v - t_0|}$, or one can check
directly that
$\gamma_\alpha
 = \int_0^1 y^{\alpha - 1} (1 - y)^{-\alpha} \, \d y
$.
From the Fubini theorem and the same calculation without absolute values,
it follows that if $\int_{t_0}^{+\infty} |h'(u)| \, \d u < +\infty$, then
for every $t \ge t_0$ we have
\[
   (I^\alpha D^\alpha h)(t)
 = - \int_t^{+\infty} h'(u) \, \d u
 = h(t).
\]
Under~\eqref{MesHypos}, we introduce
$
 h_\varepsilon(t) = \e^{- \varepsilon |t - t_0|} h(t)
$
with $\varepsilon > 0$, for which we use the preceding case and convergence
when $\varepsilon \rightarrow 0$. When $\varepsilon \in (0, 1)$ and 
$t > t_0$ we have
\[
     |h_\varepsilon(t)| 
 \le |h(t)| 
 \le \frac {\kappa_0} {1 + |t|} \up,
 \ms{22}
     |h'_\varepsilon(t)| 
 \le (\varepsilon |h(t)| + |h'(t)|)
 \le \frac {\kappa_0 + \kappa_1} {1 + |t|} \up.
\]
By Corollary~\ref{GeneralDb3}, we have 
$|D^\alpha h_\varepsilon| \le \kappa (1 + |t|)^{-1}$, and we can apply twice
dominated convergence when $\varepsilon \rightarrow 0$ in
\[
   \int_t^{+\infty} (u-t)^{\alpha - 1}
    \Bigl( \int_u^{+\infty} (v-u)^{-\alpha} h'_\varepsilon(v) 
     \, \d v \Bigr) \, \d u
 = h_\varepsilon(t).
\]
\end{proof}

 Assuming~\eqref{MesHypos} and $\Re z > 0$, we have
\begin{equation}
   D^z_t ( t h(t) )
 = t (D^z h)(t) - z (D^{z - 1} h)(t).
 \label{DeriveTH}
\end{equation}
This is obtained when $0 < \Re z < 1$ with an integration by parts,
writing
\begin{align*}
   &\ms6 \Gamma(1 - z)
    \bigl( - D^z_t (t h(t)) + t (D^z h)(t) \bigr)
 = \int_t^{+\infty} (u - t)^{- z} 
    \bigl( (u - t) h'(u) + h(u) \bigr) \, \d u
 \\
 = &\ms5 \int_t^{+\infty} (u - t)^{- z + 1} h'(u) \, \d u
   + \int_t^{+\infty} (u - t)^{- z} h(u) \, \d u
 \\
 = &\ms5 z \int_t^{+\infty} (u - t)^{- z} 
    h(u) \, \d u
 = z \Gamma(1 - z) \ms1 (D^{z - 1} h)(t).
\end{align*}

\subsubsection{Multipliers associated to fractional derivatives%
\label{MultipAssoc}}

\noindent
If $K$ is a kernel integrable on $\R^n$, we know by~\eqref{FourierPhi} that
its Fourier transform $m$ is expressed for $\xi \ne 0$ as
\[
   m(u \xi)
 = \int_\R \varphi_\theta(s) \e^{ - 2 \ii \pi s u \ms1 |\xi|} \, \d s
 = \int_\R \frac 1 {|\xi|} \varphi_\theta \Bigl( \frac v {|\xi|} \Bigr) 
    \e^{ - 2 \ii \pi v u} \, \d v,
 \quad u \in \R,
\]
where $\theta = |\xi|^{-1} \xi$ and where the function $\varphi_\theta$ is
defined on $\R$ by~\eqref{VarphiTheta}. Letting $\alpha > 0$ and assuming
that $x \mapsto |x|^\alpha K(x)$ is integrable on~$\R^n$, this yields
\begin{align*}
   D_u^\alpha m(u \xi)
 &= \int_\R (2 \ii \pi v)^\alpha \frac 1 {|\xi|}
    \varphi_\theta \Bigl( \frac v {|\xi|} \Bigr) 
     \e^{ - 2 \ii \pi v u} \, \d v
 \\
 &= \int_\R (2 \ii \pi s |\xi|)^\alpha \varphi_\theta ( s ) 
     \e^{ - 2 \ii \pi s |\xi| u} \, \d s
 = \int_{\R^n} (2 \ii \pi x \ps \xi)^\alpha
    K(x) \e^{- 2 \ii \pi u \ms1 x \ps \xi} \, \d x,
\end{align*}
which is naturally extended by $0$ when $\xi = 0$.
We set in what follows
\begin{subequations}\label{OpAlphaG}
\begin{align}
    (\xi \ps \nabla)^\alpha m(\xi)
 := &\ms4 D_u^\alpha m(u \xi) \barre_{u = 1}
  = \int_{\R^n} (2 \ii \pi x \ps \xi)^\alpha \ms2
      K(x) \e^{- 2 \ii \pi x \ps \xi} \, \d x
 \label{OpAlpha} \tag{\ref{OpAlphaG}.$\nabla^\alpha$}
 \\
 = &\ms4 \int_\R (2 \ii \pi s |\xi|)^\alpha \varphi_\theta (s) 
    \e^{ - 2 \ii \pi s |\xi|} \, \d s.
 \notag
\end{align}
\end{subequations}
When $\alpha = 1$ and $\xi \ne 0$, the quantity $(\xi \ps \nabla)^1 m(\xi)$
is equal to $- \xi \ps \nabla m(\xi)$, which is the product by~$- |\xi|$ of
the usual directional derivative of the function $m$ in the direction of the
norm-one vector $\theta = |\xi|^{-1} \xi$. When $0 < \alpha < 1$, under
the assumptions~\eqref{MesHypos}, we can give according
to~Lemma~\ref{GeneralD} the integral formula
\begin{equation}
    (\xi \ps \nabla)^\alpha m(\xi)
  = - \frac 1 {\Gamma(1 - \alpha)}
     \int_1^{+\infty} (u - 1)^{- \alpha} 
      \ms1 \frac {\d} {\d u \ns2} \ms2 m(u \xi) \, \d u.
 \label{NablaAlpha}
\end{equation}
We shall use the integral formula~\eqref{NablaAlpha} when $m(\xi)$ is
Lipschitz outside the origin and when for every $u_0 > 0$ and $u \ge u_0$,
we have for every $\theta \in S^{n-1}$ that
\[
     |m(u \theta)|
      + \bigl| \frac {\d} {\d u \ns2} \ms2 m(u \theta) \bigr|
 \le \frac {\kappa(\theta, u_0)} {1 + |u|} \up.
\]
If $K$ is an isotropic log-concave probability density with variance
$\sigma^2$, we know by Corollary~\ref{EstimatesForC}
that $|(\d / \d u) \ms2 m(u \xi)|
 \le \delta_{1, c} \ms1 |\sigma \xi| / (1 + 2 \pi |u \ms1 \sigma \xi|)
 \le \delta_{1, c}  / (2 \pi |u|)$, thus
\begin{equation}
     |(\xi \ps \nabla)^\alpha m(\xi)|
 \le  \frac {\delta_{1, c}} {2 \pi \ms1 |\Gamma(1 - \alpha)|}
     \int_1^{+\infty} (u - 1)^{- \alpha} u^{-1} \, \d u
  =  \kappa_\alpha \ms1 \delta_{1, c},
 \label{NablaAlphaBounded}
\end{equation}
and the bounded function $\xi \mapsto (\xi \ps \nabla)^\alpha m(\xi)$
defines an $L^2$ multiplier. We reach of course the same conclusion
under~\eqref{EstimaGene} for a \og general\fge kernel $\Kg$.
\dumou

 We have seen in~\eqref{InvariMul} that the multiplier norm of $m(\xi)$ on
$L^p(\R^n)$ is the same as that of the dilate $m(\lambda \xi)$, for every
$\lambda > 0$. It is thus natural to look for a norm invariant by dilation,
if we want a norm capable to control the action on~$L^p$ of a multiplier.
Since we shall work radially with Carbery's approach, we begin with a
smooth function~$h$ compactly supported in $(0, +\infty)$, and when 
$\alpha \in (0, 1)$ we set with Carbery~\cite{CarberyLp}
\begin{equation}
    \| h \|_{ {\hbox{\sevenit L}}^2_\alpha}
 := \Bigl( \int_0^{+\infty} 
     \Bigl|
      t^{\alpha + 1} D^\alpha_t \Bigl( \frac {h(t)} t \Bigr) 
     \Bigr|^2
      \, \frac {\d t} t 
    \Bigr)^{1/2}.
 \label{LdeuxAlpha}
\end{equation}
One verifies that this norm is invariant by dilation. By~\eqref{Dilate}, we
have
\begin{equation}
   t^{\alpha + 1} D^\alpha_t \Bigl( \frac { {\di h \lambda} (t)} t \Bigr)
 = t^{\alpha + 1} \lambda D^\alpha_t 
    \Bigl( \frac { h (\lambda t) } {\lambda t } \Bigr)
 = (\lambda t)^{\alpha + 1} D^\alpha_v 
    \Bigl( \frac {h (v)} v \Bigr) \barre_{v = \lambda t},
 \label{NormInvar}
\end{equation}
and the change of variable $u = \lambda t$ in~\eqref{LdeuxAlpha} completes the
proof. Let $h$ be Lip\-schitz on $(t_0, +\infty)$ for all $t_0 > 0$.
Applying~\eqref{DeriveTH} to $\widetilde h (t) = h(t) / t$, we get
for all $t > 0$
\begin{equation}
   D^\alpha_t \Bigl( \frac {h(t)} t \Bigr)
 = \frac {\alpha} t \ms2 D^{\alpha-1}_t \Bigl( \frac {h(t)} t \Bigr)
    + \frac 1 t \ms2 (D^\alpha h)(t)
 = \frac 1 t \ms2
    D^{\alpha-1}_t \Bigl( \frac {\alpha h(t)} t - h'(t) \Bigr).
 \label{FormuDeri}
\end{equation}

\begin{remN}\label{basis}
When $1/2 < \alpha < 1$, the $\mathit{L}^2_\alpha$ norm dominates the
$L^\infty(0, +\infty)$ norm of the function $h$. For a justification, let
us assume in addition that $h$ is bounded and Lipschitz on each interval
$(t, +\infty)$ with $t > 0$. Then $H : u \mapsto h(u) / u$
satisfies~\eqref{MesHypos} on $(t, +\infty)$ and we can apply
Lemma~\ref{IaDa}, giving $I^\alpha D^\alpha H = H$, thus
\begin{align*}
   \frac {h(t)} t
 &= \frac 1 {\Gamma(\alpha)} \int_t^{+\infty} (u - t)^{\alpha - 1} 
    D^\alpha_u \Bigl( \frac {h(u)} u \Bigr) \, \d u
 \\
 &= \frac 1 {t \ms2 \Gamma(\alpha)}
     \int_t^{+\infty} (t/u) (1 - t/u)^{\alpha - 1} 
      \Bigl[ u^{\alpha+1} D^\alpha_u \Bigl( \frac {h(u)} u \Bigr) \Bigr]
       \, \frac {\d u} u \up.
\end{align*}
Applying Cauchy--Schwarz, $\Gamma(\alpha) > 1$ for $\alpha \in (0, 1)$ and
letting $y = t / u$, we get
\begin{align*}
   h(t)^2
 &\le \Bigl( \int_t^{+\infty} (t/u)^2 (1 - t/u)^{2 \alpha - 2} 
       \frac {\d u} u \Bigr)
      \Bigl( \int_t^{+\infty}
       \Bigl[ u^{\alpha+1} D^\alpha_u \Bigl( \frac {h(u)} u \Bigr) \Bigr]^2
       \, \frac {\d u} u \Bigr)
 \\
 &\le \Bigl( \int_0^1 y (1 - y)^{2 \alpha - 2} \, \d y \Bigr)
       \ms3 \| h \|_{ {\hbox{\sevenit L}}^2_\alpha}^2
  \le  \frac 1 {2 \alpha - 1} \ms3 \| h \|_{ {\hbox{\sevenit L}}^2_\alpha}^2.
\end{align*}
The latter calculation is the basis for the $L^2$ part of Carbery's
Proposition~\ref{FouCarbe}.
\end{remN}

\begin{remN}
Using the second expression in~\eqref{FormuDeri}, we see that
$\| h \|^2_{ {\hbox{\sevenit L}}^2_\alpha}$ is the integral on 
$(0, +\infty)$, and with respect to $(\d t) / t$, of the square of the
modulus of
\begin{align*}
 &  \ms{ 8}
     t^\alpha D^{\alpha-1}_t \Bigl( \frac {\alpha h(t)} t - h'(t) \Bigr)
 =  \frac 1 { \Gamma(1-\alpha) }
     \int_t^{+\infty} (u/t - 1)^{- \alpha}
      \bigl( \alpha h(u) - u h'(u) \bigr) \, \, \frac {\d u} u
 \\
 = & \ms4 \frac 1 { \Gamma(1-\alpha) }
     \int_1^{+\infty} (v - 1)^{- \alpha}
      \bigl( \alpha h(t v) - t v h'(t v) \bigr) \, \frac {\d v} v \ms1 \up.
\end{align*}
In most cases, this expression tends to $\kappa h(0)$ when 
$t \rightarrow 0$, with $\kappa > 0$, and then we have that
$\| h \|_{ {\hbox{\sevenit L}}^2_\alpha}$ is finite only if $h(0) = 0$,
as for Bourgain's criterion $\Gamma_B(K)$.
\dumou

 We do not see an easy way to compare the $\hbox{\textit L}^2_\alpha$
norm and the quantity appearing in the $\Gamma_B$ criterion. However, in
the very special case where $H(t) = h(t) / t$ is $\ge 0$, convex and
decreasing on $(0, +\infty)$, the function $|H'| = - H'$ is decreasing and
it follows from Lemma~\ref{GeneralDb2} that $(D^{1/2} H)(t)$ is bounded by 
$\kappa \ms1 \sqrt{ |H(t) H'(t)| }$, hence
\begin{align*}
     \| h \|^2_{ {\hbox{\sevenit L}}^2_{1/2}}
 &\le \kappa \int_0^{+\infty} t^3
      \frac {|h(t)|} t
       \Bigl( \frac {|h'(t)|} t + \frac {|h(t)|} {t^2} \Bigr)
        \, \frac {\d t} t
 \\
 &\le \kappa \int_0^{+\infty}
       \Bigl( |h(t)| \ms2 |t \ms1 h'(t)| + |h(t)|^2 \Bigr)
        \, \frac {\d t} t
  \le \kappa' \sum_{j \in \Z} 
       \Bigl( \alpha_j(h) \beta_j(h) + \alpha_j(h)^2 \Bigr).
\end{align*}
We obtain then (in this very special situation) that 
$\| h \|_{ {\hbox{\sevenit L}}^2_{1/2}}
 \le \kappa \ms1 \Gamma_B(h^\vee)$.
\end{remN}

\subsection{Fourier criteria for bounding the maximal function%
\label{CritFou}}

\noindent
In the next proposition due to Carbery, we impose conditions that fit into
our presentation but are certainly too restrictive.

\begin{prp}[Carbery~\cite{CarberyLp}]%
\label{FouCarbe}
Let $K$ be a kernel integrable on\/~$\R^n$ with integral equal to\/~$0$,
let $m$ be the Fourier transform of~$K$. Assume that 
$m_\theta := u \mapsto m(u \theta)$ is differentiable on\/~$(0, +\infty)$
for every $\theta \in S^{n-1}$, and that $m'_\theta(u)$, $u > 0$, is
bounded by a constant independent of~$\theta$.
\dumou

 $(i)$--- If there exists $\alpha \in (1/2, 1)$ such that
\begin{equation}
 C_\alpha(m) := \sup_{\theta \in S^{n-1}} 
          \bigl\| 
           t \mapsto m (t \theta) 
          \bigr\|_{ {\hbox{\sevenit L}}^2_\alpha} 
       < +\infty,
  \label{Calpha}
\end{equation}
then for every function $f \in L^2(\R^n)$ one has
\[
     \| \Mg_K f \|_{L^2(\R^n)}
  =  \bigl\| \ms1 \sup_{t > 0} | {\Di K t} * f | \ms1 \bigr\|_{L^2(\R^n)} 
 \le \frac 1 {\sqrt{ 2 \alpha - 1}}
      \ms3 C_\alpha(m) \ms1 \| f \|_{L^2(\R^n)}.
\]
\dumou

 $(ii)$--- Suppose that $p < +\infty$ and\/ $1 / p < \alpha < 1$. 
If the multiplier\/ $(\xi \ps \nabla)^\alpha m(\xi)$ from\/~\eqref{OpAlpha}
is bounded on $L^p(\R^n)$, then for every $f$ in~$L^p(\R^n)$ one has that
\begin{equation}
     \bigl\| \sup_{1 \le t \le 2} |{\Di K t} * f| \ms1 \bigr\|_{L^p(\R^n)} 
 \le \kappa_{\alpha, p} \ms1
      \bigl( 
       2 \ms1 \|m\|_{p \rightarrow p} 
        + \| (\xi \ps \nabla)^\alpha m(\xi)\|_{p \rightarrow p}
      \bigr) \ms1 \| f \|_{L^p(\R^n)},
 \label{KalphaP}
\end{equation}
with $\kappa_{\alpha, p}
 \le (2 \alpha)^{1 - 1 / p}
      (p - 1)^{1 - 2/p} ( \alpha - 1/p )^{1/p - 1} $.
\end{prp}

 When\/ $1 < p \le 2$, one has the simpler larger bound
$\kappa_{\alpha, p}
 \le \sqrt 2 \ms1 \bigl( \alpha - 1/p \bigr)^{-1/p} $. Indeed,
for $0 < \alpha < 1$, we have that
$2^{1/2 - 1/p} \alpha^{1 - 1 / p}
      (p \ns1-\ns1 1)^{1 - 2/p} ( \alpha \ns1-\ns1 1/p )^{2/p - 1}$
is less than
$  \bigl( [ \alpha \ns1-\ns1 1/p ] \ms2/\ms2 [\sqrt {2 \alpha} \ms2 (p-1)]
   \bigr)^{2/p - 1}$.
When $1 < p \le 2$, this expression increases with $\alpha \in (1 / p , 1]$,
and for $\alpha = 1$, one has
$( 1 \ns1-\ns1 1/p ) / (\sqrt 2 \ms2 (p - 1)) 
 = 1 / ( \sqrt 2 \ms2 p) \le 1$.
\dumou
 
 Observe that if we set $\xi = |\xi| \ms1 \theta$ for some nonzero vector
$\xi \in \R^n$, we have
\[
   \bigl\| 
    t \mapsto m (t \xi) 
   \bigr\|_{ {\hbox{\sevenit L}}^2_\alpha} 
 = \bigl\| 
    t \mapsto m (t \theta) 
   \bigr\|_{ {\hbox{\sevenit L}}^2_\alpha} 
\]
according to the invariance by dilation~\eqref{NormInvar} of the norm
$\mathit{L}^2_\alpha$. So the supremum in~$(i)$ is also the supremum on
$\xi \in \R^n$. We shall need the following Lemma, slightly more general
than the conclusion~$(i)$ in Proposition~\ref{FouCarbe}.

\begin{lem}%
\label{LdeuxCarbeCrit}
Let\/ $(K_t)_{t > 0}$ be a family of integrable kernels on\/~$\R^n$, and
denote by $\xi \mapsto m(\xi, t)$ the Fourier transform of~$K_t$. Assume
that for every $u_0 > 0$, there exist $N$ and $\kappa(u_0)$ satisfying the
following: for every $\xi$ in\/ $\R^n$, the function 
$g_\xi : u \mapsto m(\xi, u) / u$, for $u \in [u_0, +\infty)$, is Lipschitz
and
\begin{equation}
     |g_\xi(u)| + |g'_\xi(u)|
 \le \kappa(u_0) \ms2 \frac { (1 + |\xi|)^N } {1 + |u|} \up,
 \quad \xi \in \R^n, \ms6 u \ge u_0.
 \label{PolyBound}
\end{equation}
If there is $\alpha \in (1/2, 1)$ such that
$
 c_\alpha := \sup_{\xi \in \R^n} 
          \bigl\| 
           t \mapsto m (\xi, t) 
          \bigr\|_{ {\hbox{\sevenit L}}^2_\alpha} 
       < +\infty
$,
then
\[
 \forall f \in \ca S(\R^n),
 \ms{18}
     \bigl\| \ms1 \sup_{t > 0} |K_t * f| \ms1 \bigr\|_{L^2(\R^n)} 
 \le \frac 1 {\sqrt{ 2 \alpha - 1}}
      \ms4 c_\alpha \ms1 \| f \|_{L^2(\R^n)}.
\]
\end{lem}

\begin{proof}
By the assumptions, the function $g_\xi$ satisfies~\eqref{MesHypos}. As in
Remark~\ref{basis}, we obtain by Lemma~\ref{IaDa} for all 
$\xi \in \R^n$ and $t > 0$ that
\[
   \frac {m(\xi, t)} t
 = \frac 1 {\Gamma(\alpha)} \int_t^{+\infty} (u - t)^{\alpha - 1} 
    D^\alpha_u \Bigl( \frac {m(\xi, u)} u \Bigr) \, \d u.
\]
For $f \in \ca S(\R^n)$, according to~\eqref{PolyBound} and
Corollary~\ref{GeneralDb3}, we can use Fubini and get
\begin{align*}
   &\ms4 (K_t * f)(x)
 = \int_{\R^n} m(\xi, t) \widehat f(\xi) \e^{ 2 \ii \pi x \ps \xi} \, \d \xi
 \\
 = &\ms4 \frac 1 {\Gamma(\alpha)} \int_t^{+\infty} t (u - t)^{\alpha - 1}
    \int_{\R^n} D^\alpha_u \Bigl( \frac {m(\xi, u)} u \Bigr) \widehat f(\xi) 
     \e^{ 2 \ii \pi x \ps \xi} \, \d \xi \ms2 \d u
 \\
 = &\ms4 \frac 1 {\Gamma(\alpha)} 
     \int_t^{+\infty} (t/u) (1 - t/u)^{\alpha - 1}
      \Bigl( \int_{\R^n} u^{\alpha + 1} 
       D^\alpha_u \Bigl( \frac {m(\xi, u)} u \Bigr) 
        \widehat f(\xi) \e^{ 2 \ii \pi x \ps \xi} \, \d \xi \Bigr)
         \ms2 \frac {\d u} u \ms1 \up.
\end{align*}
For $u > 0$ and $x \in \R^n$, let us set
\[
   (P^\alpha_u f)(x)
 = \int_{\R^n} u^{\alpha + 1} 
     D^\alpha_u \Bigl( \frac {m(\xi, u)} u \Bigr) 
      \widehat f(\xi) \e^{ 2 \ii \pi x \ps \xi} \, \d \xi.
\]
This operator $P^\alpha_u$ is associated to the multiplier
\[
   p^\alpha_u (\xi)
 = u^{\alpha + 1} 
    D^\alpha_v \Bigl( \frac {m(\xi, v)} v \Bigr) \barre_{v = u},
 \quad
 \xi \in \R^n.
\]
One can rewrite
\begin{equation}
   (K_t * f)(x)
 = \frac 1 {\Gamma(\alpha)} 
    \int_t^{+\infty} (t/u) (1 - t/u)^{\alpha - 1}
     (P^\alpha_u f)(x) \ms2 \frac {\d u} u \ms1 \up.
 \label{palphaA}
\end{equation}
By Cauchy--Schwarz and since $\Gamma(\alpha) > 1$ when $\alpha \in (0, 1)$,
we get
\[
     |(K_t * f)(x)|^2
 \le \Bigl( \int_t^{+\infty} (t/u)^2 (1 - t/u)^{2 (\alpha - 1)}
       \ms2 \frac {\d u} u 
     \Bigr)
     \Bigl( \int_0^{+\infty} \bigl| (P^\alpha_u f)(x) \bigr|^2 
       \ms2 \frac {\d u} u 
     \Bigr).
\]
For $\alpha > 1/2$, one has $2(\alpha - 1) > -1$ and letting $y = t / u$,
one sees that
\[
   \int_t^{+\infty} (t/u)^2 (1 - t/u)^{2 (\alpha - 1)}
       \ms2 \frac {\d u} u 
 = \int_0^1 y (1 - y)^{2 (\alpha - 1)} \, \d y
 < \frac 1 {2 \alpha - 1} \ms1 \up.
\]
We have obtained for $|(K_t * f)(x)|^2$ a bound independent of $t$, hence
\[
     \sup_{t > 0} \ms2 |(K_t * f)(x)|^2
 \le \kappa_\alpha^2 
      \Bigl( \int_0^{+\infty} \bigl| (P^\alpha_u f)(x) \bigr|^2 
       \ms2 \frac {\d u} {u} 
      \Bigr),
\]
with $\kappa_\alpha^{-2} = 2 \alpha - 1$. By Fubini and Parseval, we have
\begin{align*}
     &\ms4 \Bigl\| \sup_{t > 0} \ms1 |K_t * f| \ms1 \Bigr\|^2_{L^2(\R^n)}
 \le \kappa^2_\alpha \int_{\R^n}
      \Bigl( \int_0^{+\infty} \bigl|(P^\alpha_u f)(x) \bigr|^2 
       \ms2 \frac {\d u} {u} \Bigr) \, \d x
 \\
  =  &\ms4 \kappa^2_\alpha 
      \int_0^{+\infty} \|P^\alpha_u f \|_{L^2(\R^n)}^2 \ms2 \frac {\d u} u 
  =  \kappa^2_\alpha \int_{\R^n}
      \int_0^{+\infty} \Bigl| u^{\alpha + 1} 
       D^\alpha_u \Bigl( \frac {m(\xi, u)} u \Bigr) 
        \widehat f(\xi) \Bigr|^2 \ms2 \frac {\d u} u \, \d \xi
 \\
 \le &\ms4 \kappa^2_\alpha \int_{\R^n}
       c_\alpha^2 \ms1 | \widehat f(\xi) |^2 \, \d \xi
  =  \kappa^2_\alpha \ms1 c_\alpha^2 
      \ms1 \|f\|^2_{L^2(\R^n)}.
\end{align*}
\end{proof}

\dumou

\begin{remN}\label{Suffit}
If $|a(t)| \le c(t_0)$ when $t \ge t_0 > 0$ and if $b(t) = a(t) / t$, then
we have
$(1 + t) |b(t)| = (t^{-1} + 1) |a(t)| \le c(t_0) (1 + t_0^{-1})$ when
$t \ge t_0$. If we add that $|a'(t)| \le c(t_0)$ for $t \ge t_0$,
we have also $(1 + t) |a'(t) / t| \le c(t_0) (1 + t_0^{-1})$,
$t \ge t_0$, and
\[
     |b'(t)|
 \le \Bigl( \frac { |a'(t)| } t + \frac { |b(t)| } t \Bigr)
 \le \frac{ c(t_0) (1 + t_0^{-1})^2 \ns8 } { 1 + t } \ms1 \up,
 \quad t \ge t_0 > 0.
\]
If we know that for every $u_0 > 0$, there is $c(u_0)$ such that
\[
     |m(\xi, u)| + \Bigl| \frac { \d } { \d u \ns3 } \ms3 m(\xi, u) \Bigr|
 \le c(u_0) \ms2 (1 + |\xi|)^N,
 \quad \xi \in \R^n, \ms6 u \ge u_0,
\]
it follows that~\eqref{PolyBound} is true, with
$\kappa(u_0) \le 2 \ms1 c(u_0) (1 + t_0^{-1})^2$.
\end{remN}

\dumou

\begin{proof}[Proof of Proposition~\ref{FouCarbe}]
We apply Lemma~\ref{LdeuxCarbeCrit} to the family $K_t = {\Di K t}$ of
dilates of $K$, $t > 0$. Under the assumptions of
Proposition~\ref{FouCarbe}, we first have that
$|m(t \xi)| + | (\d / \d t) m(t \xi)|
 \le \kappa (1 + |\xi|)$. Remark~\ref{Suffit} implies then that the
family of functions $g_\xi : t \mapsto m(t \xi) / t$
satisfies~\eqref{PolyBound}. We thus obtain by Lemma~\ref{LdeuxCarbeCrit}
the $L^2$-maximal inequality when $f \in \ca S(\R^n)$, and we may extend it
to all functions in~$L^2(\R^n)$ by the density of~$\ca S(\R^n)$
in~$L^2(\R^n)$, as explained in Section~\ref{DefiMaxiFunc}. 
\dumou
\dumou

 Let us pass to the proof of $(ii)$, the $L^p$ case. We use the notation of
the proof of Lemma~\ref{LdeuxCarbeCrit}, adapted to $m(\xi, t) = m(t \xi)$.
Denote by $q$ the conjugate exponent of~$p$, and observe that $q - 2 > - 1$
because $p < +\infty$. When $\alpha \in (1 / p, 1)$ and $t \ge 1$, by
applying H\"older to~\eqref{palphaA} and since $\alpha - 1 > - 1 / q$, 
$\Gamma(\alpha) > 1$, we obtain
\begin{align*}
     &\ms5 |({\Di K t} * f)(x)|
 \\
 \le &\ms5 \Gamma(\alpha)^{-1} \Bigl( \int_t^{+\infty} (t/u)^q 
          (1 - t/u)^{q(\alpha - 1)} \, \d u \Bigr)^{1/q}
       \Bigl( \int_t^{+\infty} |(P^\alpha_u f)(x)|^p
       \ms2 \frac {\d u} {u^p} \Bigr)^{1 / p}
 \\
 \le &\ms5 t^{1/q} \Bigl( \int_t^{+\infty} (t/u)^q 
          (1 - t/u)^{q(\alpha - 1)} \, \frac {\d u} t \Bigr)^{1/q}
       \Bigl( \int_t^{+\infty} |(P^\alpha_u f)(x)|^p
       \ms2 \frac {\d u} {u^p} \Bigr)^{1 / p}
 \\
 \le &\ms5 t^{1/q} \Bigl( \int_1^{+\infty} v^{-q \alpha}
          (v - 1)^{q(\alpha - 1)} \, \d v \Bigr)^{1/q}
       \Bigl( \int_1^{+\infty} |(P^\alpha_u f)(x)|^p
       \ms2 \frac {\d u} {u^p} \Bigr)^{1 / p}
 \\
 \le &\ms5 
       \Bigl(
        \int_1^2 (v-1)^{q\alpha - q} \, \d v
         \ns2 + \ns3 \int_2^{+\infty} (v-1)^{ - q} \, \d v 
       \Bigr)^{1/q}
       \ms2 t^{1 / q}
       \Bigl(
        \int_1^{+\infty} |(P^\alpha_u f)(x)|^p \ms2 \frac {\d u} {u^p} 
       \Bigr)^{1 / p}.
\end{align*}
With $c_{\alpha, p}^q
 =  1 / (q \alpha - q + 1) + 1 / (q - 1)
 = \alpha (p - 1) / (\alpha - 1 / p)$, it follows that
\begin{align*}
     \Bigl\| \sup_{1 \le t \le 2} |{\Di K t} * f| \ms1 \Bigr\|^p_{L^p(\R^n)}
 &\le c^p_{\alpha, p} \ms1 2^{p/q} 
      \int_1^{+\infty}
       \Bigl( \int_{\R^n} |(P^\alpha_u f)(x)|^p \, \d x \Bigr) 
        \frac {\d u} {u^p}
 \\
  &=  c^p_{\alpha, p} \ms1 2^{p/q} 
      \int_1^{+\infty} \|P^\alpha_u f\|_p^p \, \frac {\d u} {u^p}
 \le  \frac {c^p_{\alpha, p} \ms1 2^{p/q} \ns8 } {p - 1} \ms3
       \sup_{u \ge 1} \|P^\alpha_u f\|_p^p,
\end{align*}
and we shall see that
$
     \|P^\alpha_u \|_{p \rightarrow p} 
 \le 2 \ms2 \|m\|_{p \rightarrow p} +
         \| (\xi \ps \nabla)^\alpha m(\xi)\|_{p \rightarrow p}
$.
The multipliers~$p^\alpha_u$ are dilates of one another, indeed, for every
$\lambda > 0$, we have by~\eqref{NormInvar} that
\[
   p^\alpha_u( \lambda \xi)
 = u^{\alpha + 1} 
    D^\alpha_v \Bigl( \frac {m(v \lambda \xi)} v \Bigr) \barre_{v = u}
 = u^{\alpha + 1} \lambda^{\alpha + 1}
    D^\alpha_v \Bigl( \frac {m(v \xi)} v 
               \Bigr) \barre_{v = \lambda u}
 = p^\alpha_{\lambda u} (\xi).
\]
It suffices therefore to consider $p^\alpha_1$. According 
to~\eqref{FormuDeri}, one has 
\[
   p^\alpha_1(\xi)
 = D^\alpha_t \Bigl( \frac {m(t \xi)} t \Bigr) \barre_{t = 1}
 = \alpha \ms2 
    D^{\alpha-1}_t \Bigl( \frac {m(t \xi)} t \Bigr) \barre_{t = 1}
     + D^\alpha_t m(t \xi) \barre_{t = 1}.
\]
The multiplier $D^\alpha_t m(t \xi) \barre_{t = 1}$ is precisely equal to 
$(\xi \ps \nabla)^{\alpha} m(\xi)$. The other term, since
$\alpha - 1 < 0$, can be written by~\eqref{IntegraZero} as
\[
    U(\xi)
  = \alpha \ms2 
     D^{\alpha-1}_t \Bigl( \frac {m(t \xi)} t \Bigr) \barre_{t = 1}
  = \frac \alpha {\Gamma( 1 - \alpha)} \ms1 
     \int_1^{+\infty} (v - 1)^{- \alpha}
             \Bigl( \frac {m(v \xi)} v \Bigr) \, \d v.
\]
By Lemma~\ref{EasyIntegral}, we have
$\|U\|_{p \rightarrow p} \le 2 \ms2 \|m\|_{p \rightarrow p}$ because
\[
     \frac \alpha {\Gamma( 1 - \alpha)} \ms1 
      \int_1^{+\infty} (v - 1)^{- \alpha} \, \frac {\d v} v
 \le \frac \alpha {\Gamma( 1 - \alpha)} \ms1 
      \Bigl( \frac 1 {1 - \alpha} + \frac 1 \alpha \Bigr)
  =  \frac 1 {\Gamma(2 - \alpha)}
 \le 2,
\]
cutting $\int_1^{+\infty}$ at $v = 2$, and using~\eqref{MiniGamm}.
\end{proof}

\subsection{Proofs of Theorems~\ref{TheoDyad} and~\ref{TheoMaxi},
 and Proposition~\ref{MoreGeneral}%
\label{PreuveTheos}}

\noindent
We need only show Proposition~\ref{MoreGeneral}, and we can limit
ourselves to $1 < p \le 2$. As in Bourgain's proof of the $L^2$
theorem for $\M_C$ at the end of Section~\ref{ConcluBour}, the kernel~$K$
to which we shall apply Proposition~\ref{PropoPrio} is given by 
$K = \Kg - P$, where~$P$ is the Poisson kernel~$P_1$
from~\eqref{PoissonDensi}, and $\Kg$ is a probability density on $\R^n$
satisfying~\eqref{EstimaGene} with two constants $\delta_{0, g} \ge 1$ and
$\delta_{1, g}$ controlling the Fourier transform $\mg$ and its gradient.
We know by~\eqref{MaxiPoiss} that the maximal operator associated to the
Poisson kernel acts on $L^r(\R^n)$, $1 < r \le +\infty$, with constants
independent of the dimension~$n$. Letting~$B$ denote the Euclidean ball
normalized by variance in $\R^n$, we could replace~$P$ by~$K_B$ and invoke
Stein's Theorem~\ref{indepdim} instead. 
\dumou

 We shall apply Proposition~\ref{PropoPrio} in the two cases corresponding
to Theorems~\ref{TheoDyad} and~\ref{TheoMaxi}, in order to show that the
maximal function (or the dyadic maximal function) associated to the kernel
$K$ is bounded on~$L^p$ for $p > 3/2$ (or for $p > 1$). We shall get by
difference that the maximal function for $\Kg$ (or $\Klc$, $K_C$) is
bounded as well. In the \og dyadic\fge case of Theorem~\ref{TheoDyad}, the
operator $T_j$, for $j \in \Z$, is the convolution with the 
dilate~${\Di K {2^j}}$ of $K$. For Theorem~\ref{TheoMaxi},
$T_{j, v}$ is the convolution with~${\Di K {v 2^j}}$, $1 \le v \le 2$, and
$T_j$ is given by
\[ 
   T_j f 
 = \sup_{1 \le v \le 2} \ms1 |T_{j, v} f|
 = \sup_{2^j \le t \le 2^{j+1}} |{\Di K t} * f|.
\]
\dumou
 
 One has to check that the assumptions of Proposition~\ref{PropoPrio},
namely, \eqref{A0}, \eqref{A1}, \eqref{A2} and~\eqref{A3} p.~\pageref{A3},
are satisfied in these two cases. If the $(Q_j)$ are those of
Littlewood--Paley, defined by
\[
   \widehat Q_j(\xi)
 = \e^{ - 2 \pi 2^j |\xi|} - \e^{ - 2 \pi 2^{j+1} |\xi|},
 \quad
 \xi \in \R^n,
\]
then the assumption~\eqref{A0} is satisfied according to~\eqref{LiPaIneq},
with $C_p = \textrm{q}_p$.
\dumou

 For~\eqref{A1}, we write $T_{j, v} = U_{j, v} - S_{j, v}$, where the
$U_{j, v} = \Di {(\Kg)} {v 2^j}$ are related to~$\Kg$ and the
$S_{j, v} = \Di P {v 2^j}$ to the Poisson kernel. The operators $U_{j, v}$
and $S_{j, v}$ are positive, as convolutions with probability densities. As
mentioned before, this is the only place where we need $\Kg$ to be a
probability density rather than a general integrable kernel. We
know by~\eqref{MaxiPoiss} that the maximal operator~$S^*$ associated to the
Poisson kernel is bounded on $L^p(\R^n)$, $1 < p < +\infty$, by a
constant~$C'_p$ independent of the dimension~$n$. Consequently, the
property~\eqref{A1} is satisfied.
\dumou

 Let us consider~\eqref{A2}. The first case is when $T_j = {\Di K {2^j}}$
and in this case, according to~\eqref{LOneMultiplier}, the operator $T_j$ is
bounded on all the spaces $L^p(\R^n)$, $1 \le p \le +\infty$, by the
$L^1$~norm of $K$ and we get that
\begin{equation}
     \|T_j\|_{p \rightarrow p}
 \le \|K\|_{L^1(\R^n)}
 \le 2.
 \label{GotA2Dyad}
\end{equation}
In the second case, we have to use the part~$(ii)$ of
Proposition~\ref{FouCarbe}. This will be discussed below.
\dumou

 Finally, we must show~\eqref{A3}, \textit{i.e.}, prove that $T^*$ verifies
the property $(\gr S_2)$ from~p.~\pageref{grWp}. For $k$ fixed in $\Z$, we
shall bound the maximal operator of the kernel $N_k := K * Q_k$ using the
conclusion~$(i)$ of Proposition~\ref{FouCarbe}. We show in
Section~\ref{AutrePreuve} that for every value $\alpha \in (1/2, 1)$, the
norm $b_k := C_\alpha(\widehat{N_k})$ decays exponentially with~$|k|$, with
constants depending on $\alpha$ and (linearly) on 
$\delta_{0, g} + \delta_{1, g}$. In the \og dyadic case\fg, the bound
obtained in this way by $(i)$ for the maximal function of~$N_k$ implies that
\begin{align*}
     \bigl\| \ms1 \sup_{j \in \Z} \ms1 |T_j Q_{j+k} f| \ms1 \bigr\|_2
  &=  \bigl\| \ms1 \sup_{j \in \Z} 
              \ms1 |{\Di K {2^j}} * (\Di P {2^{j+k}} - 
               \Di P {2^{j+1+k}}) * f| \ms1 \bigr\|_2
 \\
  &=  \bigl\|
       \ms1 \sup_{j \in \Z} \ms1 | \Di{(N_k)}{2^j} * f| \ms1 
      \bigr\|_2
 \le \bigl\| \ms1 \sup_{t > 0} | \Di{(N_k)}t * f| \ms1 \bigr\|_2
 \le \kappa_\alpha \ms1 b_k,
\end{align*}
which proves $(\gr S_2)$ in this case. The case of the global maximal
function requires a small adaptation, Carbery says: \og This is not exactly
what being strongly bounded on $L^2$ means, but a slight modification of
this argument will give precisely what we require\fg. Indeed, there is now a
gap between what we get from Proposition~\ref{FouCarbe} and the assumption
we need for applying Proposition~\ref{PropoPrio}. We shall discuss it in the
subsection~\ref{GapQuestion} and resolve this \og gap question\fge in the
subsection~\ref{SGQ}. We obtain at last by Lemma~\ref{CloseGap}
and by Lemma~\ref{LdeuxCarbeCrit} that there exist universal 
coefficients~$(a_k)_{k \in \Z}$ such that 
$\sum_{k \in \Z} a_k^s < +\infty$ for every $s > 0$, and such that
\begin{equation}
     \bigl\| \ms1 \sup_{j \in \Z} \ms1 |T_j Q_{j+k} f| \ms1 \bigr\|_2
 \le (\delta_{0, g} + \delta_{1, g}) \ms1 a_k, 
 \quad k \in \Z.
 \label{GotA3}
\end{equation}
\dumou

 For~\eqref{A2} in the \og global\fge case, we study the operators
$(W_t)_{t > 0}$ defined by
\[
 W_t f = \sup_{t \le u \le 2 t} | \Di K u * f|,
 \ms{16}
 t > 0,
\]
and we want to prove~\eqref{A2} for the family of $T_j = W_{2^j}$
from~\eqref{DefTj}, with $j \in \Z$. Using the invariance by
dilation~\eqref{InvariMulB} of multiplier norms, we see that the
operators $W_t$ have the same norm when~$t$ varies, hence we need to find a
bound for $T_0 = W_1$ only. For this, we want to apply the
conclusion~$(ii)$ of Proposition~\ref{FouCarbe}, so we must show that the
multipliers $m$ and $(\xi \ps \nabla)^\alpha m(\xi)$ are bounded on
$L^p(\R^n)$ for some $\alpha \in (1/p, 1)$. For~$m$ it is clear by the
elementary fact~\eqref{LOneMultiplier}. 
\dumou

 For $(\xi \ps \nabla)^\alpha m(\xi)$ we shall use complex interpolation
between $(\xi \ps \nabla)^0 m(\xi) = m(\xi)$ that acts on $L^1(\R^n)$, and
$(\xi \ps \nabla) \ms1 m(\xi)$ that acts on $L^2(\R^n)$ since it is a
bounded function on~$\R^n$ by~\eqref{NablaAlphaBounded}
and~\eqref{EstimaGene}. We get by interpolation that the multiplier 
$(\xi \ps \nabla)^\alpha m(\xi)$ is bounded on $L^p(\R^n)$, with $p$ given
by
\[
   \frac 1 p
 = \frac {1 - \alpha} 1 + \frac \alpha 2
 = 1 - \frac \alpha 2 \ms1 \up,
\]
and we need $1 - \alpha / 2 = 1/p < \alpha$ for applying $(ii)$, thus 
$1 < 3 \alpha / 2 = 3 - 3 / p$. We must therefore have $p > 3 / 2$ in order
to conclude. We see that the reason for the restriction on the values of $p$
in Theorem~\ref{TheoMaxi} is to be found precisely here. 
\dumou

 This sketch is not fully accurate. For being able to interpolate, one must
control in~$L^2$ the values $\alpha = 1 + \ii \tau$, for every $\tau$ real,
which causes no difficulty, but also the values $\alpha = 0 + \ii \tau$ in
$L^1$, and this is more technical. The precise work, involving a slight
modification of the strategy described here, is done in
Section~\ref{InterpoCarbe} when we are well embedded
by~M\"uller~\cite{MullerQC} in the mood for interpolation. For every 
$p \in (3/2, 2]$, we shall then obtain for some $\alpha > 1/p$, function
of~$p$, a bound of the 
form~$
     \| (\xi \ps \nabla)^\alpha m(\xi) \|_{p \rightarrow p}
 \le \kappa_p (\delta_{0, g} + \delta_{1, g})^{2 - 2 / p}
$.
By Proposition~\ref{FouCarbe}, we deduce
\[
     \bigl\| T_0 f \bigr\|_{L^p(\R^n)} 
  =  \bigl\| \sup_{1 \le t \le 2} |{\Di K t} * f| \ms1 \bigr\|_{L^p(\R^n)} 
 \le \kappa'_p \ms1 (\delta_{0, g} + \delta_{1, g})^{2 - 2 / p}
      \| f \|_{L^p(\R^n)}
\]
for every function $f \in L^p(\R^n)$. We get~\eqref{A2} with 
$p_{\rm min} = 3/2$, since
\begin{equation}
     \|T_j\|_{p \rightarrow p}
  =  \|T_0\|_{p \rightarrow p}
 \le \kappa'_p \ms1 (\delta_{0, g} + \delta_{1, g})^{2 - 2 / p},
 \quad j \in \Z, \ms8
 3/2 < p \le 2.
 \label{GotA2}
\end{equation}
\dumou

 Applying Proposition~\ref{PropoPrio}, we finish the proof of
Proposition~\ref{MoreGeneral}. For $p \in (3/2, 2]$, we shall bound 
$T^* = \Mg_K$ in $L^{p}(\R^n)$, thus also $\M_{\Kg}$. We choose a value
$p_0$, function of $p$, such that $3 / 2 < p_0 < p$, and we let 
$\delta = \delta_{0, g} + \delta_{1, g}$. We have by~\eqref{GotA2} 
that~$C''_{p_0}
 \le \kappa''_{p} \ms2 \delta^{2 - 2 / p_0}$. Then,
applying~\eqref{ExplicitBound}, \eqref{GotA3}, \eqref{GotA2} and
$\delta_{0, g} \ge 1$, we obtain
\[
     \| \M_{\Kg} \|_{p \rightarrow p}
 \ns1\le\ns1 \| T^* \|_{p \rightarrow p} \ns1+\ns1 \kappa_{p, 0} 
 \ns1\le\ns1 \kappa_{p, 1} 
       (C''_{p_0})^\gamma
        \Bigl(
         \sum_{k \in \Z} (\delta a_k)^{ (1 - \gamma)p / 2} 
        \Bigr)^{ 2 / p }
         \ns7+\ns1 \kappa_{p, 2}
 \ns1\le\ns1 \kappa_p \delta^{2 - 2 / p}
\]
as announced, observing that 
$1 - \gamma = [1/p_0 - 1/p] / [1/p_0 - 1/2]$ is the interpolation parameter
for $L^p$ and the pair $(L^{p_0}, L^2)$, and that the powers of $\delta$
under the exponents $\gamma$ and $1 - \gamma$ are of the form $2 - 2/r$, 
$r = p_0$ or $2$. In the dyadic case, we may replace~\eqref{GotA2}
by~\eqref{GotA2Dyad} and obtain the result for $\M^{(d)}_{\Kg}$ when 
$p \in (1, 2]$.

\begin{smal}
\noindent\addtocounter{thm}{1}%
\textit{Remark \thethm}.  
Bringing back the question to the Poisson kernel leads to some
complications, because the function $\varphi_\theta(s)$ associated to the
Poisson kernel, \textit{i.e.}, the Cauchy kernel~\eqref{CauchyK}, does not
have decay properties as good as that of the function $\varphi_{\theta, C}$
of a convex set. This approach however does not depend on the $L^p$ result
of Stein for the Euclidean ball. 

 Why not employ the Gaussian semi-group instead? In some non Euclidean
situations, like Heisenberg groups or Grushin operators for instance, and
especially for the weak type $(1, 1)$ property of associated maximal
functions, the Poisson kernel is preferable. Indeed, some asymptotic
estimates, uniform in the dimension, are required on the kernel and are
easier to obtain for the Poisson kernel. But in the Euclidean case, we
cannot see a compelling obstacle to the use of the Gaussian kernel. We
would get an excellent decay, both in the space variable and in the Fourier
variable. We have chosen to stick to the original proofs, but we urge the
reader to rewrite them with Gaussian kernels instead. We shall see in
Section~\ref{LeCube} that Bourgain uses Gaussian kernels.

\end{smal}

\subsubsection{Where is the gap?%
\label{GapQuestion}}
As was said above, we will arrive for $N_k = K * Q_k$ at
\[
    C_\alpha(N_k) 
 := \sup_{\theta \in S^{n-1}} 
     \bigl\| 
       t \mapsto N_k (t \theta) 
     \bigr\|_{ {\hbox{\sevenit L}}^2_\alpha} 
  \le \kappa_\alpha \ms1 2^{ - \gamma |k|},
 \quad k \in \Z,
\]
for some $\gamma > 0$. This implies by Proposition~\ref{FouCarbe},~$(i)$
that
\[
     \bigl\| \ms1 \sup_{t > 0} | \Di {(N_k)}t * f| \ms1 \bigr\|_2
 \le \kappa_\alpha \ms1 2^{ - \gamma |k|}.
\]
Translating the definition of $N_k$ gives
\[
     \bigl\|
      \ms1 \sup_{t > 0} | {\Di {(K * Q_k)} t} * f| \ms1 
     \bigr\|_2
 \le \kappa_\alpha \ms1 2^{ - \gamma |k|}
\]
where $K = \Kg - P$, or
\[
     \bigl\|
      \ms1 \sup_{v \in [1, 2]} \sup_{j \in \Z}
       |( \Di K {v 2^j} * 
         ( \Di P {v 2^{j+k}} - \Di P {v 2^{j+k+1}} ) * f| \ms1 
     \bigr\|_2
 \le \kappa_\alpha \ms1 2^{ - \gamma |k|}.
\]
This must be compared to bounding the expression
\[
     \bigl\|
      \ms1 \sup_{v \in [1, 2]} \sup_{j \in \Z}
       |( \Di K {v 2^j} * 
         ( \Di P {2^{j+k}} - \Di P {2^{j+k+1}} ) * f| \ms1 
     \bigr\|_2,
\]
which is what we are waiting for, in the definition p.~\pageref{grSp} 
of Property~$(\gr S_2)$ for the family of operators~$(T_{j, v})$, 
$j \in \Z$, $v \in [1, 2]$.

\subsection{A proof for the property $(\gr S_2)$%
\label{AutrePreuve}}

\noindent
In what follows, $m = \mg - \widehat P$ is the Fourier transform of the
kernel $K = \Kg - P$ that appears in the proof of
Proposition~\ref{MoreGeneral}, where $\Kg$ is a probability density on
$\R^n$ satisfying~\eqref{EstimaGene}. We have
\[
 \widehat P(\xi) = \e^{ - 2 \pi |\xi|}
 \ms9 \hbox{and we let} \ms9
 \rho(\xi) = \widehat P(\xi) - \widehat P(2 \xi),
 \quad \xi \in \R^n.
\]
For every $k \in \Z$, every $\xi \in \R^n$ and $u > 0$, we set 
\[
   m_k(\xi) 
 = \widehat{N_k}(\xi)
 = m(\xi) \bigl( \e^{- 2^{k+1} \pi |\xi|} - \e^{- 2^{k+2} \pi |\xi|} \bigr)
 = m(\xi) \rho(2^k \xi),
 \ms{16}
   h_k^\xi(u) 
 = \frac {m_k(u \xi)} u \ms1 \up.
\]
One must show that for any given $\alpha \in (1/2, 1)$, the quantity
\[
   C_\alpha(m_k)^2
 = \sup_{\theta \in S^{n-1}} \ms2 
    \bigl\|
     u \mapsto m_k (u \theta) 
    \bigr\|^2_{ {\hbox{\sevenit L}}^2_\alpha} 
 = \sup_{\theta \in S^{n-1}} \ms2
    \int_0^{+\infty} 
     \bigl( u^{\alpha + 1} (D^\alpha h_k^\theta)(u) \bigr)^2
      \, \frac {\d u} u
\]
introduced in~\eqref{Calpha} decays exponentially to $0$ when $|k|$ tends to
infinity. We fix therefore $\theta \in S^{n-1}$ and for $u \in \R$, we set
\[
 \phi(u) = m(u \theta),
 \ms{16}
   \chi(u) 
 = \e^{- 2 \pi |u|} - \e^{- 4 \pi |u|}
 = \widehat P(u \theta) - \widehat P(2 u \theta)
 = \rho(u \theta). 
\]
Let $\delta = \delta_{0, g} + \delta_{1, g} \ge 1$, where 
$\delta_{0, g}, \delta_{1, g}$ are the constants in~\eqref{EstimaGene}. We
know that
\[
     |u| \ms2 |\mg(u \ms1 \theta)| 
 \le \delta_{0, g} \le \delta,
 \ms{12}
     |\theta \ps \nabla \mg(u \ms1 \theta)| 
 \le \delta_{1, g} \le \delta,
 \ms{12}
     |u \ms1 \theta \ps \nabla \mg(u \ms1 \theta)| 
 \le \delta,
 \ms9  u \in \R.
\]
On the other hand, the derivative with respect to $u > 0$ of
$\widehat P(u \theta) = \e^{- 2\pi |u|}$ is bounded by $2 \pi$, and
according to~\eqref{FPoissonBoundsA}, \eqref{FPoissonBoundsB}, we have
\[
     |u \ms1 \widehat P(u \theta)| 
 \le (2 \pi \e)^{-1}
  <  1
 \le \delta,
 \ms{20}
     \Bigl| u \ms1 \frac {\d} {\d u \ns2} \widehat P(u \theta) \Bigr|
 \le \e^{-1}
  <  \delta.
\]
For $\phi(u) = m(u \theta) = \mg(u \theta) - \widehat P(u \theta)$ we get 
$|\phi'(u)| \le \delta + 2 \pi$. Using again $\delta > 1$, we
simplify this bound as $|\phi'(u)| < 8 \delta$. It follows first that
$|\phi(u)| \le 8 \delta |u|$, and 
\begin{subequations}
\begin{equation}
 |\phi(u)| \le 8 \ms1 \delta \bigl( |u| \wedge |u|^{-1} \bigr),
 \quad
 |\phi'(u)| \le 8 \ms1 \delta \bigl( 1 \wedge |u|^{-1} \bigr).
 \label{PreciseBdsA}
\end{equation}
For $\chi(u)$, we see when $u > 0$ that $0 \le \chi(u) \le \e^{- 2 \pi u}$
and
\[
     - 2 \pi \e^{-2\pi u} 
 \le \chi'(u)
  =  - 2 \pi \e^{-2\pi u} + 4 \pi \e^{-4\pi u}
 \le 2\pi \e^{-2\pi u},
\]
implying that $|\chi'(u)| \le 2 \pi$ for $u \ne 0$ and
\begin{equation}
 |\chi(u)| \le (2 \pi |u|) \wedge |2 \pi \e \ms1 u|^{-1},
 \quad
 |\chi'(u)| \le (2 \pi) \wedge |\e u|^{-1}.
 \label{PreciseBdsB}
\end{equation}
\end{subequations}
We obtain a symmetric treatment of the two functions $\chi$ and
$\phi_\delta := \delta^{-1} \phi$ since, up to some \emph{universal}
multiple $\kappa$ (we express this by the sign~$\simleq$), we have
\begin{equation}
 |\phi_\delta(u)|, \ms2 |\chi(u)| \bsle2 |u| \wedge |u|^{-1}, 
 \ms{22}
 |\phi'_\delta(u)|, \ms2 |\chi'(u)| \Dsle2 1 \wedge |u|^{-1}.
 \label{PhiEtPsi}
\end{equation}
We set
$
   p_k(u) 
 = m_k(u \theta) 
 = \phi(u) \chi(2^k u)
$,
$
   h_k(u ) 
 = p_k(u) / u
$
and we want to estimate 
$
 \|p_k\|_{ {\hbox{\sevenit L}}^2_\alpha}
$
for every $k \in \Z$. Notice that
\[
   p_{-k}(2^{k} v)
 = \chi(v) \phi(2^{k} v).
\]
The $\mathit{L}^2_\alpha$ norm is invariant by dilation and the assumptions
on $\phi_\delta$ and~$\chi$ are identical, we may therefore restrict the
verification to the case $k \ge 0$. Let us fix an integer $k \ge 0$. We
have the following table, divided into the three regions where the chosen
bounds~\eqref{PhiEtPsi} for the functions~$h_k$ and~$h'_k$ keep the same
analytical expression, namely, the intervals $(0, 2^{-k})$, $(2^{-k}, 1)$
and $(1, +\infty)$. We consider that $h'_k$ is the derivative of the
product of $u^{-1} \phi(u)$ and $\chi(2^k u)$, we bound therefore $|h'_k|$ by
the sum of $\bigl| (u^{-1} \phi(u))' \bigr| \ms2 |\chi(2^k u)|$ and
$| u^{-1} \phi(u)| \ms2 2^k |\chi'(2^k u) \bigr|$.
\dumou

{ \smaller
\[
 \begin{matrix}
       u :      &  &   0    &    & 2^{-k} &   & 1 &  \cr
 \noalign{\vskip 2pt \hrule \vskip 5pt}
 u^{-1} |\phi_\delta(u)|  &\simleq \ns8 
                   & \vrule 
                            & 1  & \vrule  
                                          & 1 & \vrule 
                                                  & u^{-2} \cr
 |\chi(2^k u)|  &\simleq \ns8 
                   & \vrule
                            & 2^k u & \vrule 
                                          & 2^{-k} u^{-1} & \vrule 
                                                & 2^{-k} u^{-1} \cr
 \ns9 u^{\ns1-\ns1 1} |\phi'_\delta(u)|
  \ns3+\ns3  u^{\ns1-2\ns1} |\phi_\delta(u)| \ns9 
                &\simleq \ns8 
                   & \vrule 
                            & u^{-1} + u^{-1} & \vrule 
                                          & u^{-1} + u^{-1} & \vrule 
                                                & u^{-2} + u^{-3}
                                                  \simleq u^{-2} \cr
 2^k |\chi'(2^k u)|  
                &\simleq \ns8 
                   & \vrule 
                            &    2^k   & \vrule 
                                          & u^{-1} & \vrule 
                                                & u^{-1} \cr
 \noalign{\vskip 2pt \hrule \vskip 5pt}
 \delta^{-1} |h_k(u)|  
                &\simleq \ns8 
                   & \vrule 
                            & \ns{14} 2^k u \le 2^{-k} u^{-1} \ns{14} & \vrule 
                                          & 2^{-k} u^{-1} & \vrule 
                                                & 2^{-k} u^{-3} \cr
 \delta^{-1} |h'_k(u)|  
                &\simleq \ns8 
                   & \vrule 
                            & \ns{14} 2^k + 2^k \simleq u^{-1} \ns{14} & \vrule 
                                          & \ns{14} 2^{-k} u^{-2} + u^{-1}
                                            \simleq u^{-1} \ns{14} & \vrule 
                                                & \ns{14} 2^{-k} u^{-3} + u^{-3}
                                                   \simleq u^{-3} \ns{14} \cr
 \noalign{\vskip 2pt \hrule \vskip 5pt}
 \end{matrix}
\]

}

\dumou
\noindent
We see that $\delta^{-1} |h'_k(u)|
 \simleq H_1(u) := u^{-1} \wedge u^{-3}$. This function $H_1$ is
non-increasing on $(0, + \infty)$ and independent of $k$, and
$\delta^{-1} |h_k(u)| \simleq H_{0, k} (u) = 2^{-k} H_1(u)$. It follows
from Lemma~\ref{GeneralDb2} that for $t > 0$, we have
\[
    \delta^{-1} |(D^\alpha h_k)(t)| 
 \simleq H_{0, k} (t)^{1 - \alpha} \ms2 H_1(t)^\alpha
 \simleq 2^{ -(1 - \alpha) \ms1 k } H_1(t),
\]
and the conclusion is reached since we obtain then
\dumou
\noindent
\begin{align*}
    \|\phi \ms2 {\di \chi {2^k}} \|^2_{ {\hbox{\sevenit L}}^2_\alpha}
 &= \| p_k \|^2_{ {\hbox{\sevenit L}}^2_\alpha}
  = \int_0^{+\infty} \bigl| t^{\alpha + 1} 
      (D^\alpha h_k)(t) \bigr|^2 \, \frac {\d t} t
 \\
 &\simleq \delta^2 \ms1 2^{ -2(1 - \alpha) \ms1 k}  
       \Bigl( \int_0^1 (t^{\alpha + 1} t^{-1})^2 \, \frac {\d t} t
       + \int_1^\infty (t^{\alpha + 1} t^{-3})^2 \, \frac {\d t} t
       \Bigr)
\end{align*}
\dumou
\noindent
and
\[
   \int_0^1 t^{2 \alpha - 1} \, \d t
    + \int_1^\infty t^{2 \alpha - 5} \, \d t
 = \frac 1 {2 \alpha} + \frac 1 {4 - 2 \alpha}
 = \frac 1 {\alpha (2 - \alpha)}
 < \frac 1 \alpha
 < + \infty,
\]
thus $\| p_k \|_{ {\hbox{\sevenit L}}^2_\alpha}
 \simleq \delta \alpha^{-1/2} \ms1 2^{-(1 - \alpha) k}$ when $k \ge 0$,
and $\| p_k \|_{ {\hbox{\sevenit L}}^2_\alpha}
 \le \kappa \ms1 \alpha^{-1/2} \ms1 \delta \ms2 2^{-(1 - \alpha) |k|}$ when
$k \in \Z$. This implies by Proposition~\ref{FouCarbe},~$(i)$ that
\begin{equation}
     \bigl\| \sup_{r > 0} \ms1 
         \bigl( [ {\di {(m_k)} r} \widehat f \ms3]^\vee \bigr)
     \bigr\|_{L^2(\R^n)}
 \le \kappa \ms1 \delta \ms2 2^{-(1 - \alpha) |k|}
      \ms1 \|f\|_{L^2(\R^n)}
 \label{AkAlpha}
\end{equation}
for every $\alpha \in (1/2, 1)$, giving the property $(\gr S_2)$ 
(see p.~\pageref{grSp}) in the dyadic case.
\dumou
\dumou

 It would be just as simple to work with the $\Gamma_B(K)$ criterion of
Bourgain given in Section~\ref{AnaFouri}. We prove a general Lemma that 
will be invoked again in Section~\ref{LeCube} for the cube problem. 
\dumou

\begin{lem}\label{UsingGammaB}
Suppose that two integrable kernels $K_1$ and $K_2$ on\/ $\R^n$ satisfy,
for a certain $\kappa$ and every $\theta \in S^{n-1}$, that
\[
     |\widehat K_j(u \theta)| 
 \le \kappa (|u| \wedge |u|^{-1}),
 \quad
     |\theta \ps \nabla \widehat K_j(u \theta)| 
 \le \kappa (1 \wedge |u|^{-1}),
 \ms{16}
 j = 1, 2,
 \ms9 u \in \R.
\]
It follows that\/ $\Gamma_B \bigl( K_1  * {\Di {(K_2)} {2^k}} \bigr)
 \le C(\kappa) \ms1 2^{- |k| / 2}$ for $k \in \Z$.
\end{lem}

\begin{proof}
We fix $\theta \in S^{n-1}$, and in order to remind us about the preceding
case, we let $m$ be the Fourier transform of $K_1$ and $\rho$ that of
$K_2$. We will modify the table above, in order to emphasize now 
$\phi(u) := m(u \theta)$ and 
$u \theta \ps \nabla m(u \theta) = u \phi'(u)$ that appear in the
components $\alpha_j(m)$ and $\beta_j(m)$ of $\Gamma_B(K)$, and we
proceed similarly
for $\chi(u):= \rho(u \theta)$. 
\dumou
\dumou

 Let $m_k$ be the Fourier transform of the kernel 
$K_1 * {\Di {(K_2)} {2^k}}$. We have that 
$m_k(u \theta) = m(u \theta) \rho(2^k u \theta)$ and we may again restrict
ourselves to $k \ge 0$, since a dilation by~$2^i$ on a multiplier $g(\xi)$
produces a shift of $i$ places on the indices~$j$ of the sequences
$(\alpha_j(g))_{j \in \Z}$, $(\beta_j(g))_{j \in \Z}$, leaving 
$\sum_{j \in \Z}$ unchanged. The bounds below do not depend on 
$\theta \in S^{n-1}$, so we will be able to estimate
\dumou
\[
    A_k(u)
 := \sup_{\theta \in S^{n-1}} |m_k(u \theta)|
 \ms{16} \hbox{and} \ms{16}
    B_k(u)
 := \sup_{\theta \in S^{n-1}} |u \theta \ps \nabla m_k(u \theta)|.
\]
Note that $B_k(u)$ is controlled by $\phi(u) 2^k u \chi'(2^k u)$ and 
$u \phi'(u) \chi(2^k u)$. We have $\alpha_j(m_k) \sim A_k(2^j)$,
$\beta_j(m_k) \sim B_k(2^j)$, for every $j \in \Z$. The new table is
divided into the same three regions as before.
\dumou

{ \smaller
\[
 \begin{matrix}
       u :      &  &   0    &    & 2^{-k} &   & 1 &  \cr
 \noalign{\vskip 2pt \hrule \vskip 5pt}
        |\phi(u)|  &\simleq \ns8 
                   & \vrule 
                            & u     & \vrule  
                                          & u & \vrule 
                                                  & u^{-1}             \cr
 |\chi(2^k u)|  &\simleq \ns8 
                   & \vrule
                            & 2^k u & \vrule 
                                          & 2^{-k} u^{-1} & \vrule 
                                                & 2^{-k} u^{-1}        \cr
 \ns9  u |\phi'(u)| \ns9 
                &\simleq \ns8 
                   & \vrule 
                            & u     & \vrule 
                                          &  u            & \vrule 
                                                &   1                  \cr
 2^k u |\chi'(2^k u)|  
                &\simleq \ns8 
                   & \vrule 
                            &  2^k u   & \vrule 
                                          &  1     & \vrule 
                                                &  1                   \cr
 \noalign{\vskip 2pt \hrule \vskip 5pt}
 A_k(u)   
                &\simleq \ns8 
                   & \vrule 
                            & 2^k u^2   & \vrule 
                                          & 2^{-k}        & \vrule 
                                                & 2^{-k} u^{-2}          \cr
 B_k(u)  
                &\simleq \ns8 
                   & \vrule 
                            & 2^k u^2 + 2^k u^2   & \vrule 
                                          &  u +    2^{-k}                
                                            \simleq u      \ns{14} & \vrule 
                                                & \ns{14} u^{-1} + 2^{-k} u^{-1}
                                                   \simleq u^{-1} \ns{14} \cr
 \sqrt{ A_k(u) \ms1 B_k(u)}
                &\simleq \ns8 
                   & \vrule 
                            & 2^k u^2    & \vrule 
                                          &  2^{-k/2} u^{1/2} & \vrule 
                                                & 2^{-k/2} u^{-3/2}       \cr
 \noalign{\vskip 2pt \hrule \vskip 5pt}
 \end{matrix}
\]

}

\noindent
It follows that for every $j \in \Z$, we have
\[
 \alpha_j(m_k) \simleq \ms4
 \begin{cases}
   \ms3 2^{k + 2 j} &\hbox{if} \ms6  j \le - k, 
   \\
   \ms3 2^{-k}      &\hbox{if} \ms6  -k \le j \le 0, 
   \\
   \ms3 2^{-k - 2j} &\hbox{if} \ms6  0 \le j,
 \end{cases}
 \ms{10} \hbox{so} \ms{10}
 \sum_{j \in \Z} \alpha_j(m_k) \simleq (k+1) 2^{-k},
\]
and
\[
 \sqrt{\alpha_j(m_k) \beta_j(m_k)} \simleq 
 \begin{cases}
   \ms3 2^{k + 2 j}       \ns6 &\hbox{if} \ms6  j \le - k,
   \\
   \ms3 2^{- k/2 + j/2}   \ns6 &\hbox{if} \ms6  -k \le j \le 0,
   \\
   \ms3 2^{- k/2 - 3 j/2} \ns6 &\hbox{if} \ms6  0 \le j,
 \end{cases}
 \hbox{ so }
 \sum_{j \in \Z} \sqrt{\alpha_j(m_k) \beta_j(m_k)}
 \simleq 2^{-k/2}.
\]
Taking the supremum, we obtain
$\Gamma_B \bigl( K_1  * {\Di {(K_2)} {2^k}} \bigr)
 \le C(\kappa) \ms1 2^{-|k| / 2}$, for $k \in \Z$.
\end{proof}

 Coming back to Carbery's situation, we obtain in this way by
Lemma~\ref{Clef} that
\[
 \|m_k\|_{2 \rightarrow 2} \le \kappa \ms1 \delta \ms2 2^{- |k| / 2},
 \quad k \in \Z,
\]
slightly better than what we got with $C_\alpha(m_k)$. Indeed, we must
choose $\alpha > 1/2$ with Carbery, and we have obtained for
$C_\alpha(m_k)$ a bound of order $\delta \ms1 2^{-(1-\alpha) \ms1 |k|}$.

\subsubsection{A solution to the gap question%
\label{SGQ}}

\noindent
The gap question has been exposed in Section~\ref{GapQuestion}. Instead of
the function studied precedently, equal to
\[
  \widehat{N_k} (\xi) : t \mapsto
   {\di m t} (\xi) 
     ( \di {\widehat P} {t \ms1 2^k}
       - \di {\widehat P} {t \ms1 2^{k+1} } )(\xi),
 \quad t > 0, \ms5 \xi \in \R^n,
\]
we need to study the family of multipliers defined by
\[
     \widehat n_k(\xi, t)
   = {\di m t} (\xi) 
     ( \di {\widehat P} {2^{j+k}}
        - \di {\widehat P} {2^{j + k+1} } )(\xi),
 \quad j \in \Z 
 \ms{10} \hbox{and} \ms{10}
 2^j \le t \le 2^{j+1},
\]
which are the Fourier transforms of the kernels ${\Di K t}
 * ( \Di P {2^{j(t) + k} } - \Di P {2^{j(t) + k+1} })$ with 
$j(t) = \lfloor \log_2 t \rfloor$. They do not fit into the
setting of Proposition~\ref{FouCarbe}, but can be treated using
Lemma~\ref{LdeuxCarbeCrit}. We do the following: for every $j \in \Z$, let
$x_j = 2^j + 2^{j-1}$ be the midpoint of the interval 
$I_j = [2^j, 2^{j+1}]$. Let the \og new\fge function be
\[
  t \mapsto
   \di m {2^j + 2(t - 2^j)}(\xi) 
     ( \di {\widehat P} {2^{j+k}}
        - \di {\widehat P} { 2^{j + k+1} } ) (\xi) 
\]
for $t$ in the first half $[2^j, x_j]$ of the interval $I_j$, and
\[
  t \mapsto
   \di m {2^{j + 1}} (\xi) 
     ( \di {\widehat P} { 2^k (2^j + 2(t - x_j)) }
       - \di {\widehat P} { 2^{k+1} (2^j + 2(t - x_j)) } ) (\xi) 
\]
in the second half. The first half \og contains\fge the family 
$\widehat n_k(\xi, t)$ that we have to study, and adjoining the second half
will allow us to exploit easily what has been done in
Section~\ref{AutrePreuve} for the regular setting. We can describe more
compactly the new setting if we define two motions going along 
$(0, +\infty)$ according to
\[
    X(t) =
 \begin{cases}
  \ms3 2^j + 2(t - 2^j),
   &  \quad  2^j \le t \le x_j,
 \\
  \ms3 2^{j + 1}, 
    & \quad  x_j \le t \le 2^{j+1},
 \end{cases}
\]
and
\[
    Y(t) =
\begin{cases}
    \ms3 2^j,
    &\quad  2^j \le t \le x_j,
 \\
    \ms3 2^j + 2(t - x_j), 
    &\quad  x_j \le t \le 2^{j+1}.
\end{cases}
\]
Then, the new function can be written as
\begin{equation}
    \widetilde m_k(\xi, t)
 := \di m {X(t)} (\xi) \ms1
     ( \di {\widehat P} {2^{k} Y(t)}
       - \di {\widehat P} {2^{k+1} Y(t)} ) (\xi),
 \label{NewFunction}
\end{equation}
corresponding to the family of kernels $K_t = {\Di K {X(t)}}
 * ( \Di P {2^k Y(t)} - \Di P {2^{k+1} Y(t)})$. The two functions $X, Y$
are non-decreasing, continuous, piecewise linear, and we have 
$X(2^j) = Y(2^j) = 2^j$ for $j$ in $\Z$. Notice that $X(2 t) = 2 X(t)$ and
$Y(2 t) = 2 Y(t)$ (make use of $2 x_j = x_{j+1}$). Also, 
$0 \le X'(t), Y'(t) \le 2$. Applying Remark~\ref{Suffit}, one sees easily
that the functions $g_\xi(t) = \widetilde m_k(\xi, t) / t$ 
satisfy~\eqref{PolyBound}.
\dumou

 In the \og dilation case\fge where $m_0(\xi, t) = m(t \xi)$, we have that
$m_0(s \xi, t) = m_0(\xi, s t)$ for every $s > 0$, and it allowed us to
restrict the study of the functions $t \mapsto m_0(\xi, t)$, $\xi \in \R^n$,
to the case $|\xi| = 1$. This is not true anymore, but we still have that 
$m(2 \xi, t) = m(\xi, 2 t)$ for the two components $\Phi$ and $\Psi$ of
$\widetilde m_k(\xi, t)$, defined by
\[
   \Phi(\xi, t)
 = \di m {X(t)} (\xi),
 \ms{16}
   \Psi(\xi, t)
 = \bigl(
    \di {\widehat P} { Y(t) } - \di {\widehat P} { Y(2t) } 
   \bigr) (\xi),
\]
and this permits us to restrict to the case $1 \le |\xi| < 2$. Indeed,
\[
   \Phi(2 \ms1 \xi, t)
 = \di m {X(t)} (2 \xi) 
 = m \bigl( 2 X(t) \ms1 \xi \bigr)
 = m \bigl( X(2 t) \xi \bigr)
 = \Phi(\xi, 2 t).
\]
The same property holds true for $\Psi(\xi, t)$, with $Y$ replacing $X$.
\dumou

 Let us fix $\xi$ such that $1 \le |\xi| < 2$, and consider now
\[
   \phi_1(u) 
 = \Phi(\xi, u)
 = m(X(u) \xi),
 \quad
   \chi_1(u) 
 = \Psi(\xi, u)
 = \e^{ - 2 \pi Y(u) |\xi|} - \e^{ - 4 \pi Y(u) |\xi|}.
\]
Letting $\xi = |\xi| \ms1  \theta$, we compare $\phi(u) = m(u \theta)$ with
$\phi_1(u) = \phi(X(u) \ms1 |\xi|)$. For every $u > 0$, we have 
$u \le X(u) \le 2 u$ and $u/2 \le Y(u) \le u$. We have therefore that 
$u \le X(u) |\xi| \le 4 u$ and $u/2 \le Y(u) |\xi| \le 2 u$. Recall that
$m$, difference of $\mg$ and~$\widehat P$, satisfies~\eqref{PreciseBdsA}.
It follows that
\begin{align*}
     \delta^{-1} |\phi_1(u)| 
  &=  |\phi_\delta(X(u) \ms1 |\xi|)|
 \le 8 \ms1 
      \bigl[
       (X(u) \ms1 |\xi|) \wedge (X(u)^{-1} \ms1 |\xi|^{-1})
      \bigr]
 \\
 &\le 32 \ms1 \bigl( |u| \wedge |u|^{-1} \bigr)
 \simleq |u| \wedge |u|^{-1}.
\end{align*}
We also have $\phi'_1(u) = X'(u) \ms2 \phi'(X(u) \ms1 |\xi|)$, and
since $X'(u) \le 2$, 
\[
     \delta^{-1} |\phi'_1(u)| 
 \le 2 \ms1 |\phi'_\delta(X(u) \ms1 |\xi|)|
 \le 16 \bigl[ 1 \wedge (X(u)^{-1} \ms1 |\xi|^{-1}) \bigr]
 \le 16 \bigl( 1 \wedge |u|^{-1} \bigr),
\]
which can be written as $\delta^{-1} |\phi'_1(u)| 
 \simleq 1 \wedge |u|^{-1}$. Using~\eqref{PreciseBdsB}, we have the same
kind of inequalities for $\chi_1$. The proof in Section~\ref{AutrePreuve}
depended only on these two bounds, so the result in~\eqref{AkAlpha} is also
valid in the modified setting and gives the following lemma.

\begin{lem}\label{CloseGap}
Suppose that $\Kg$ is a probability density on\/ $\R^n$
satisfying~\eqref{EstimaGene}, that $m = \mg - \widehat P$ and that
$\widetilde m_k$ is defined by\/~\eqref{NewFunction}. For 
$\alpha \in (1/2, 1)$, one has
\[
     \sup_{\xi \in \R^n} 
          \bigl\| 
           t \mapsto \widetilde m_k (\xi, t) 
          \bigr\|_{ {\hbox{\sevenit L}}^2_\alpha} 
 \le \kappa (\delta_{0, g} + \delta_{1, g}) 
      2^{-(1 - \alpha) |k|},
 \quad k \in \Z.
\]
\end{lem}

\subsection{Annex: proof of Bourgain's $L^2$ theorem by Carbery's criterion%
\label{PreuvBour}}

\begin{proof}
This section is intended to illustrate the Fourier
definition~\eqref{DalphaDef} of~$D^\alpha$, and we shall have to perform
some contortions in order to enter into the suitable setting. The kernel
$K$ on $\R^n$ to which we want to apply the conclusion~$(i)$ of Carbery's
Proposition~\ref{FouCarbe} is again $K = \Klc - P$, as in
Section~\ref{PreuveTheos}, where $\Klc$ is a symmetric log-concave
probability density on $\R^n$ normalized by variance. Let us fix a norm one
vector $\theta \in \R^n$; here, the function 
$  \varphi_\theta(s) 
 = \int_{\theta^\perp} K(y + s \ms1 \theta) \, \d^{n-1} y$,
for $s \in \R$, is the difference of two symmetric probability densities
$\phi_j$, associated respectively to $\Klc$ and to the Poisson kernel~$P$.
The function $\phi_1$ of integrals of $\Klc$ on affine hyperplanes parallel
to $\theta^\perp$ satisfies, according to Lemma~\ref{LExpoDecay}, an
estimate of exponential decay $\phi_1(s) \le \kappa \e^{- |s| / \kappa}$,
for $s \in \R$ and for a certain $\kappa > 0$ universal. On the other hand,
$\phi_2(s)$ is the Cauchy kernel~\eqref{CauchyK} equal 
to~$\pi^{-1} (1 + s^2)^{-1}$, for which one has only
$
 \phi_2(s) \le 1 \wedge s^{-2}
$,
where $a \wedge b$ denotes the minimum of two real numbers $a$ and $b$.
This estimate is valid also for~$\phi_1$, up to some universal factor
$\kappa$, and we shall remember for the absolute value of $\varphi_\theta$
that
\begin{equation}
 \forall s \in \R,
 \ms{16}
     |\varphi_\theta(s)|
 \le \kappa \Bigl( 1 \wedge \frac 1 {s^2} \Bigr).
 \label{estimephi}
\end{equation}
The Fourier transform $m$ of $K$ is given by
\[
   m(t \ms1 \theta) 
 = \int_\R \varphi_\theta(s) \e^{- 2 \ii \pi s t} \, \d s.
\]
Denote by $\Phi$ the antiderivative of $\varphi_\theta$ vanishing at $0$.
The function~$\Phi$ is odd, it vanishes also at infinity because
$\varphi_\theta$ is even with integral zero. We deduce
from~\eqref{estimephi}, for some $\kappa > 0$ and every $s \in \R$, that
\begin{equation}
     |\Phi(s)| 
 \le \kappa \ms1 ( |s| \wedge |s|^{-1} ).
 \label{estimePhi}
\end{equation}
For $t \ne 0$, we could, performing an integration by parts, express 
$m(t \ms1 \theta)$ by a simply converging integral 
\[
   m(t \ms1 \theta) 
 = 2 \ii \pi t 
    \int_{- \infty}^{+\infty} \Phi(s) \e^{- 2 \ii \pi s t} \, \d s,
\]
but we prefer to work with absolutely converging integrals, for example in
this way: let us denote by $\widetilde P_0$ the $L^1$-normalized truncation 
$
 \widetilde P_0 = \| \gr 1_B P \|^{-1}_{L^1(\R^n)} \ms2 \gr 1_B P
$
of the Poisson kernel $P$ at a sufficiently large Euclidean ball $B$ in
$\R^n$, so that $\| \gr 1_B P \|_1 > 1 / 2$. We can see according 
to~\eqref{IntegraPois} that the radius of $B$ must be at least of order
$\kappa \ms1 \sqrt n$. Another possibility is to introduce a modified
Poisson kernel
\[
   \widetilde P (x) 
 = 2 \ms2 P(x) \e^{ - \varepsilon_0 |x|^2 / 2},
\]
where $\varepsilon_0 > 0$ is chosen so that the integral of $\widetilde P$
is equal to $1$. With both choices, one has 
$\widetilde P_0, \widetilde P \le 2 \ms1 P$, and the estimates of the
maximal function for the kernel $P$ are thus clearly true for 
$\widetilde P$, with a bound simply doubled. For the same fixed $\theta$ of
norm one, the modified function $\phi_2$ defined by
\[
     \phi_2(s)
  =  2 \ms2 
      \int_{\theta^\perp} P(y + s \theta) 
           \e^{ - \varepsilon_0 (|y|^2 + s^2) / 2} \, \d^{n-1} y
 \le C(n) \ms2 \e^{-\varepsilon_0 s^2 / 2}
\]
decays exponentially at infinity, and since
$\phi_2(s) \le 2 \ms1 \pi^{-1} (1 + s^2)^{-1}$, the modified function
$\phi_2$ satisfies~\eqref{estimephi} and~\eqref{estimePhi}. The modified
antiderivative $\Phi$ inherits now at infinity of the exponential decay of
$\phi_1$ and of $\phi_2$, and this makes the integrals that follow absolutely
convergent. However, the \og universal\fge estimates remain given
by~\eqref{estimephi} and~\eqref{estimePhi}.
\dumou

\begin{smal}
\noindent
The situation would be simpler
using a Gaussian kernel, letting
\[
 K(x) = K_C(x) - G(x),
 \quad
 x \in \R^n,
\]
with $G$ being the $N(0, \I_n)$ density~\eqref{LoiNZeroId} on $\R^n$.
\end{smal}

\noindent
We apply here the Fourier definition~\eqref{DalphaDef} for $D^\alpha$. For
every $t > 0$ we write 
\[
   \frac {m(t \ms1 \theta)} t
 = 2 \ii \pi \int_\R \Phi(s) \e^{- 2 \ii \pi s t} \, \d s,
\]
where $|\Phi|$ decays exponentially at infinity. This ensures that
$t \mapsto m( t \theta ) / t$ is $C^\infty$ on the line, with bounded
derivatives. By~\eqref{DalphaDef}, we can express the fractional derivative
appearing in Carbery's criterion as
\[
   D^\alpha_t \Bigl( \frac {m(t \ms1 \theta)} t \Bigr)
 = 2 \ii \pi (2 \pi)^\alpha
    \int_\R (\ii s)^\alpha \Phi(s) \e^{- 2 \ii \pi s t} \, \d s.
\]
For $0 < \alpha < 1$, we write 
\[
   \int_0^\infty s^\alpha \Phi(s) \e^{- 2 \ii \pi s t} \, \d s
 = \frac 1 {2 \ii \pi t}
     \int_0^\infty 
      \bigl( s^\alpha \Phi(s) \bigr)' \e^{- 2 \ii \pi s t} \, \d s,
\]
and because $\bigl( s^\alpha \Phi(s) \bigr)'$ vanishes at~$0$, we see that
\[
   \int_0^\infty s^\alpha \Phi(s) \e^{- 2 \ii \pi s t} \, \d s
 = - \frac 1 {4 \pi t^2 \ns5} \ms3
     \int_0^\infty 
      \bigl( s^\alpha \Phi(s) \bigr)'' \e^{- 2 \ii \pi s t} \, \d s.
\]
The integrals on the side of negative $s$ ask for an analogous treatment,
essentially already seen in Section~\ref{VoluSections},
Lemma~\ref{PreBourg}. We estimate the various parts (five parts) issued
from the differentiations of $s^\alpha \Phi(s)$ to the first and se\-cond
order, by applying the upper bounds~\eqref{estimephi} and~\eqref{estimePhi}
and the fact that $0 < \alpha < 1$. Notice that
\[
    \int_0^\infty (s^{\alpha-1} + s^{\alpha - 2}) ( s \wedge s^{-1} ) \, \d s
 =  \frac 1 {1 + \alpha} + \frac 1 \alpha + \frac 1 {1 - \alpha}
     + \frac 1 {2 - \alpha}
 =: \kappa_\alpha. 
\]
Grouping two of the terms issued from $(s^\alpha \Phi)'$, 
$(s^\alpha \Phi)''$ and using~\eqref{estimePhi}, we have 
\[
     \Bigl|
      \int_0^\infty s^{\alpha-1} \Phi(s) \e^{- 2 \ii \pi s t} \, \d s
     \Bigr|
      + 
       \Bigl|
        \int_0^\infty s^{\alpha-2} \Phi(s) \e^{- 2 \ii \pi s t} \, \d s
       \Bigr|
 \le \kappa \ms2 \kappa_\alpha,
\]
we also have
$
     \int_0^\infty (s^{\alpha}
        + s^{\alpha - 1}) |\varphi_\theta(s)| \, \d s
 \le \kappa \ms2 \kappa_\alpha
$ for two other terms by~\eqref{estimephi}, and finally for each 
$\phi = \phi_j$, $j = 1, 2$, decreasing on the positive side of the real
line, we know by Lemma~\ref{Repeti} that
\[
    \int_0^\infty s^{\alpha} |\phi'(s)| \, \d s
  = \alpha \int_0^{+\infty} s^{\alpha - 1} \phi(s) \, \d s
  < +\infty,
\]
which permits us to close this list of estimates for
$\varphi_\theta = \phi_1 - \phi_2$. It follows that for every $t > 0$, we
have
\[
     \Bigl| D^\alpha_t \Bigl( \frac {m(t \ms1 \theta)} t \Bigr) \Bigr|
 \le \kappa'_\alpha \ms1 (t^{-1} \ns2 \wedge \ms1 t^{-2}),
\]
with $\kappa'_\alpha \le \kappa' (2 \pi)^\alpha \kappa_\alpha$ independent
of the direction $\theta$. Recalling the definition~\eqref{Calpha} and
since $0 < \alpha < 1$, we get
\begin{align*}
   C_\alpha(m)^2
 &= \sup_{\theta \in S^{n-1}} 
          \bigl\| 
           t \mapsto m (t \theta) 
          \bigr\|_{ {\hbox{\sevenit L}}^2_\alpha}^2
 = \sup_{\theta \in S^{n-1}} 
    \int_0^{+\infty} 
     \Bigl| t^{\alpha + 1} 
      D^\alpha_t \Bigl( \frac {m (t \theta)} t \Bigr) \Bigr|^2
       \, \frac {\d t} t
 \\
 &\le (\kappa'_\alpha)^2
       \Bigl( \int_0^1 (t^{\alpha + 1} t^{-1})^2 \, \frac {\d t} t
         + \int_1^{+\infty} (t^{\alpha + 1} t^{-2})^2 \, \frac {\d t} t
       \Bigr)
 \\
  &=  (\kappa'_\alpha)^2
       \Bigl( \int_0^1 t^{2 \alpha - 1} \, \d t
         + \int_1^{+\infty} t^{2 \alpha - 3} \, \d t
       \Bigr)
   =  (\kappa'_\alpha)^2
       \Bigl( \frac 1 {2 \alpha} + \frac 1 {2 - 2 \alpha} \Bigr)
  <  +\infty.          
\end{align*}
One thus chooses $\alpha \in (1/2, 1)$ arbitrary and applies Carbery's
Proposition~\ref{FouCarbe}~$(i)$, which gives the boundedness on $L^2(\R^n)$
of the maximal operator associated to the difference kernel 
$K = \Klc - \widetilde P$. We get in this way that the maximal operator
$\M_{\Klc}$ is bounded on~$L^2(\R^n)$ by a constant independent of the
dimension~$n$.
\end{proof}

\section{The Detlef M\"uller article%
\label{AMuller}}

\noindent
M\"uller~\cite{MullerQC} introduces a geometrical parameter 
$Q(C)$\label{QuC} 
associated to every symmetric convex body~$C$ in~$\R^n$. When $C$ is
isotropic of volume $1$, this parameter $Q(C)$ is equal to the maximum of
$(n-1)$-dimensional volu\-mes of hyperplane projections of~$C$. M\"uller
shows that in the class $\ca C(\lambda)$ consisting of $C$s for which
$Q(C)$ and the isotropy constant $L(C)$ are bounded by a given $\lambda$,
the existence for the maximal operator $\M_C$ associated to $C$ of an
$L^p(\R^n)$~bound, uniform in $n$, can be pushed to every value~$p > 1$ with
a constant $\kappa(p, \lambda)$ depending on $p$ and $\lambda$ only. This
removes ---$\ms1$in a way$\ms1$--- the restriction~$p > 3/2$ imposed by
Bourgain and Carbery.
\dumou 

 We have seen in~\eqref{Isotro} and~\eqref{NormaVari} that when $C_0$ is
isotropic of volume~$1$ in $\R^n$, then the dilate $C_1 = r_0 C_0$ with 
$r_0 = L(C_0)^{-1}$ is isotropic and normalized by variance. The proof of
M\"uller will actually make use of a parameter~$q(C_1)$ equal to the
supremum in $\theta \in S^{n-1}$ of the masses of the signed measures
$\theta \ps \nabla K_{C_1}$. We shall see that for $\theta$ of norm one,
the mass of the measure $\theta \ps \nabla K_{C_1}$, the directional
derivative in the sense of distributions of the probability measure
$\mu_{C_1}$, is given by
\[
     \frac {2 \ms1 |P_\theta \ms1 C_1|_{n-1} \ns{20} } {|C_1|_n} \ms{18}
  =  2 \ms1 r_0^{-n} \ms1 r_0^{n-1} \ms1 |P_\theta \ms1 C_0|_{n-1}
 \le \frac 2 {\ms2 r_0} \ms2 Q(C_0)
  =  2 \ms1 L(C_0) \ms2 Q(C_0),
\]
where $P_\theta$ is the orthogonal projection onto the hyperplane
$\theta^\perp$. For every symmetric convex set~$C$, we let $C_0$ be an
isotropic position of volume~1 for $C$ and we set 
\begin{equation}
 q(C) = 2 L(C_0) \ms1 Q(C_0).
 \label{qC}
\end{equation}
M\"uller~\cite[Section~3]{MullerQC} proves that $q(C)$ is uniformly bounded
for the family of unit balls $B^q_n$ of $\ell^q_n$, $1 \le q < +\infty$
fixed and $n \in \N^*$. This is easy when $q = 2$. By~\eqref{RnV}, we know
that the Euclidean ball $B_{n, V}$ in~$\R^n$ normalized by variance has a
radius $r_{n, V}$ equal to~$\sqrt{n+2}$, hence by the log-convexity of the
Gamma function we get
\begin{align*}
     q(B^2_n) 
 &=  \sup_{\theta \in S^{n-1}}
      \frac {2 \ms1 |P_\theta \ms1 B_{n, V}|_{n-1} } {|B_{n, V}|_n}
  =  \frac {2 \ms1 \omega_{n-1}} { r_{n, V} \ms1 \omega_n }
  =  \frac {2 \ms1 \Gamma(n/2 + 1)} 
           {\sqrt {\pi (n+2)} \ms3 \Gamma(n/2 + 1/2)}
 \\
 &\le \frac {2 \Gamma(n/2 + 1/2)^{1/2} \ms4 \Gamma(n/2 + 3/2)^{1/2}} 
            {\sqrt {\pi (n+2)} \ms3 \Gamma(n/2 + 1/2)} 
  =   2 \ms2 \sqrt{ \frac {n+1} { 2 \pi (n+2)} } 
  <   \sqrt{ \frac 2 \pi} \ms1 \up.
\end{align*}
\dumou

 Given a kernel $K$ integrable on $\R^n$ and having partial derivatives
$\partial_j \ms1 K$ in the sense of distributions that are (signed)
measures $\mu_j$, for $j = 1, \ldots, n$, we define the 
\emph{directional variation} $V(K)$ of $K$ by
\begin{equation}
 V(K)
 := \sup_{\theta \in S^{n-1}} 
     \bigl\| \theta \ps \nabla K \bigr\|_1
  = \sup_{\theta \in S^{n-1}} 
     \bigl\| \sum_{j=1}^n \theta_j \ms1 \mu_j \bigr\|_1.
 \label{V(K)}
\end{equation}
We will show at Lemma~\ref{q=V} that $V(K_C) = q(C)$ when $C$ is an
isotropic symmetric convex body normalized by variance. For the 
$N(0, \I_n)$ Gaussian density $\gamma_n$, we see that
$  V(\gamma_n) 
 = \int_{\R^n} |x \ps \gr e_1| \, \d \gamma_n(x)
 = \int_\R |u| \, \d \gamma_1(u)
 = \sqrt{ 2 / \pi}$. Notice that
\begin{equation}
 V( {\Di K t} ) = t^{-1} V(K),
 \quad t > 0,
 \ms{24} \hbox{and} \ms{12}
 V(K * \mu) \le V(K)
 \label{VKt}
\end{equation}
\def\Pn{P_{\ms2 1}^{(\ns{0.7} n \ns1)}}%
for any probability measure $\mu$ on~$\R^n$. Since $V(\gamma_n)$ is
independent of~$n$, it follows from the subordination formula~\eqref{Subor}
that the same is true for the Poisson kernel $\Pn$ on $\R^n$ expressed
in~\eqref{PoissonDensi}. Precisely, because $G_s$ in~\eqref{Subor} is a
$N(0, s \I_n)$ Gaussian measure, we have
$V(G_s) = s^{-1/2} \ms1 V(\gamma_n)$ by~\eqref{VKt} and we first get
\begin{equation}
     V(\Pn)
 \le \int_0^{+\infty} V(G_s)
        \frac {s^{-3/2}} {\sqrt{2 \pi}} \e^{ - 1 / (2 s) } \, \d s
  =  \int_0^{+\infty} 
        \frac { \e^{ - 1 / (2 s) } } \pi  
         \, \frac { \d s } { s^2 } 
  =  \frac 2 \pi \up,
 \label{VPoiss}
\end{equation}
but actually $V(\Pn) = 2 / \pi$ since for each $x \in \R^n$, all gradients
$\nabla G_s(x)$, $s > 0$, are nonnegative multiples of the same vector 
$- x$. This equality is of course also easy to derive by a direct
calculation on the Poisson density. 
\dumou

 Besides the appearance of the parameter $q(C)$, M\"uller's proof draws on
estimates such as~\eqref{EstimaGene}, but extended to more derivatives of
the Fourier transform $m_C$ of~$K_C$. That bounding more derivatives leads
to improved results was already seen in Bourgain~\cite{BoGAFA}, who
obtained a dimension free bound in $L^p(\R^n)$ for all $p > 1$ in the case
of the maximal operator $\M_C$ of $\ell^q_n$~balls when $q$ is an even
integer. We shall consider a probability density $\Kg$ on $\R^n$ or
more generally an integrable kernel $\Kg$, with a Fourier transform $\mg$
satisfying that for every integer $j \ge 0$, there exists a constant
$\delta_{j, g}$ such that
\begin{subequations}\label{EstimaGeneGN}%
\begin{equation}
     \Bigl| \frac {\d^j} {\d t^j \ns2} \ms2 
       \mg(t \theta) \Bigr|
 \le \frac {\delta_{j, g}} 
           {1 + t} \ms1 \up,
 \quad \theta \in S^{n-1}, \ms9 t > 0.
 \label{EstimaGeneN} \tag{\ref{EstimaGeneGN}.$\gr H_\infty$}
\end{equation}
\end{subequations}
Actually, for each specific value $p \in (1, 3/2]$, bounding $\M_C$ in
$L^p(\R^n)$, knowing that $q(C) \le \lambda$, requires a certain finite
number of estimates from the infinite list~\eqref{EstimaGeneN}, and this
number increases to infinity when $p$ tends to $1$. We let
\begin{equation}
 \Delta_k = \sum_{j=0}^k \delta_{j, g}.
 \label{DefiDel}
\end{equation}
The \og radial\fge estimate~\eqref{EstimaGeneN} implies
$ |\d^j / (\d t^j) \mg(t \xi) |
 \le \delta_{j, g} |\xi|^j / (1 + |t \xi|)$ for $\xi \ne 0$. It is natural to
disregard $\xi = 0$ in a radial method, but when $j > 0$, we can extend
continuously $\xi \mapsto \d^j / (\d t^j) \mg(t \xi)$ by giving the value
$0$ at $\xi = 0$.

\begin{thm}[M\"uller~\cite{MullerQC}]\label{TheoMull}
For every $p \in (1, +\infty]$ and $\lambda > 0$, there exists a constant
$\kappa(p, \lambda)$ independent of $n$ such that
\[
     \| \M_{\Klc} f \|_{L^p(\R^n)}
 \le \kappa(p, \lambda) \ms2 \| f \|_{L^p(\R^n)}
\]
if $\Klc$ is an isotropic symmetric log-concave probability density
on\/~$\R^n$, normalized by variance and with $V(\Klc) \le \lambda$. In
particular, for every symmetric convex body~$C$ in\/ $\R^n$ such that 
$q(C) \le \lambda$, one has\/
$
     \| \M_C f \|_{L^p(\R^n)}
 \le \kappa(p, \lambda) \ms2 \| f \|_{L^p(\R^n)}
$. When $p \in (1, 2]$, we can write more precisely
\[
     \| \M_{\Klc} f \|_{L^p(\R^n)}
 \le \kappa(p) (1 + \lambda^{2 / p - 1}).
\]
If a probability density $\Kg$ satisfies\/~\eqref{EstimaGeneN} and if
$p \in (1, 2]$, then we have
\[
     \| \M_{\Kg} f \|_{L^p(\R^n)}
 \le \kappa_p \ms2 \Delta_{k_0(p)}^{1 - 1/p} \Delta_1^{1 - 1/p} 
      (1 + V(\Kg)^{2 / p - 1}),
 \ms{10} \hbox{with} \ms{10}
 k_0(p) < p / (p - 1).
\]
\end{thm}

 The subsequent proof furnishes for the constant $\kappa_p$ in the line
above an order exponential in $q = p / (p-1)$ that is certainly not right,
see Remarks~\ref{ButPoly} and~\ref{ButPolyB}. The case $p > 3/2$ is already
known, with $\kappa(p, \lambda)$ independent of $\lambda$, see
Theorem~\ref{TheoMaxi} and Proposition~\ref{MoreGeneral}. We know by
Lemma~\ref{LEstimatesForC} that isotropic symmetric log-concave probability
densities verify~\eqref{EstimaGeneN} with absolute constants 
$(\delta_{j, c})_{j=0}^\infty$. We shall thus concentrate on the $\Kg$ case
and on values $p \in (1, 3/2]$. Taking Carbery's results into account, the
following proposition will be (essentially) enough for proving M\"uller's
theorem.

\def\MullerPropOne{\cite[Proposition 1]{MullerQC}}%
\begin{prp}[after M\"uller~\MullerPropOne]%
\label{PropoMull}
Let $\Kg$ be an integrable kernel on\/~$\R^n$ 
satisfying\/~\eqref{EstimaGeneN} and let $\mg$ be its Fourier transform.
For every $\alpha \in (0, 1)$ and every
$p \in (1, +\infty)$, the multiplier\/ $(\xi \ps \nabla)^\alpha \mg(\xi)$
in\/~\eqref{OpAlpha} admits on $L^p(\R^n)$ a bound that depends upon~$p$,
$\alpha$, $\gr d = (\delta_{j, g})_{j=0}^\infty$ and~$V(\Kg)$, but not on
the dimension~$n$. When $p \in (1, 2]$ and if\/ 
$\|\Kg\|_{L^1(\R^n)} \le 1$, we can write
\[
     \| (\xi \ps \nabla)^\alpha \mg (\xi) \|_{p \rightarrow p}
 \le 1 + \kappa(\alpha, p) \ms1 \Delta_{k(p)}^{ (4/3)(1 - 1/p )} 
      \bigl(
       1 + \delta_{0, g}^{ (2/3)(1 - 1/p) } V(\Kg)^{2 / p - 1}
      \bigr),
\]
with $k(p) = \lceil 3 \ms1 p / (4 p - 4) \rceil$.
\end{prp}

 The case $p = 2$ follows easily from Parseval~\eqref{EasyBound}
by~\eqref{EstimaGene} and \eqref{NablaAlphaBounded}. The result for 
$p \ge 2$ can be obtained by duality from the case $1 < p \le 2$.
\dumou

\begin{proof}[Proof of Theorem~\ref{TheoMull}]
Let $p \in (1, 2)$ be given. We then choose $p_0 \in (1, p)$ and 
$\alpha \in (1/p_0, 1)$ as being functions of $p$, for example
$p_0 = (2 p + 2 ) / (5 - p)$ and $\alpha = (p + 7) / (4 p + 4)$. We apply in
$L^{p_0}(\R^n)$ the part $(ii)$ of Proposition~\ref{FouCarbe} to the kernel
$K = \Kg - P$. We know by Proposition~\ref{PropoMull} that 
$(\xi \ps \nabla)^\alpha \mg(\xi)$ is bounded on $L^{p_0}(\R^n)$ by a
function of~$V(\Kg)$ and we will check in Section~\ref{Model} that 
$(\xi \ps \nabla)^\alpha \widehat P(\xi)$ is also bounded
on~$L^{p_0}(\R^n)$ by some $\pi_{\alpha, p_0} = \pi(p)$. It follows for 
$m = \mg - \widehat P$ that
\[
     \| (\xi \ps \nabla)^\alpha m (\xi) \|_{p_0 \rightarrow p_0}
 \le \kappa_0(p, \gr d) (1 + V(\Kg)^{2 / p_0 - 1})
 \le \kappa_0(p, \gr d) (1 + \lambda^{2 / p_0 - 1}),
\]
with $ \kappa_0(p, \gr d)
 \le \kappa(p) \Delta_{k(p_0)}^{ (4/3) ( 1 - 1/p_0 ) }
  \delta_{0, g}^{ (2/3) (1 - 1/p_0) }$, where 
$\Delta_j \ge \delta_{0, g} \ge 1$ because $\Kg$ here is a probability
density. We obtain in this way that
\[
 f \mapsto W_1 f:= \sup_{1 \le u \le 2} | {\Di K u} * f|
\]
is bounded on $L^{p_0}(\R^n)$. This was the only missing information for
deducing from Proposition~\ref{PropoPrio} that $\Mg_K$ is bounded on
$L^p(\R^n)$ when $1 < p \le 3/2$. Indeed, with the notation of
Section~\ref{EstiPrio}, let $T_{j, v}$ be the convolution with 
${\Di K {2^j v}}$, $v \in [1, 2]$ and let $T_j$ be as in~\eqref{DefTj}. By
Proposition~\ref{FouCarbe}~$(ii)$, we have for every $j \in \Z$ that
\begin{equation}
     \|T_j\|_{p_0 \rightarrow p_0}
  =  \|T_0\|_{p_0 \rightarrow p_0}
  =  \|W_1\|_{p_0 \rightarrow p_0}
 \le \kappa_{\alpha, p_0}
      \bigl(
       2 + \|(\xi \ps \nabla)^\alpha m(\xi)\|_{p_0 \rightarrow p_0}
      \bigr),
 \label{MajoTj}
\end{equation}
with $\kappa_{\alpha, p_0}$ from~\eqref{KalphaP}. We bound it by 
$C''_{p_0}(\lambda) := \kappa_{\alpha, p_0} 
      \bigl(
       2 + \kappa_0(p, \gr d) (1 + \lambda^{2 / p_0 - 1})
      \bigr)$.
By~\eqref{ExplicitBound}, with $p_0$ already set and 
$r_0 = 2 p / (p + 2 - p_0)$ function of $p$ and $p_0$, we get
\begin{equation}
     \|M_K\|_{p \rightarrow p} 
 \le (C_{r_0})^{ 2 \ms1 \gamma / p_0 } \ms2 
      {C''_{p_0}(\lambda)}^\gamma
       \Bigl( \sum_{k \in \Z} a_k^{ (1 - \gamma)p / 2} \Bigr)^{ 2 / p }
        + 2 \ms1 C'_p,
 \label{ProofTheoMull}
\end{equation}
where $\gamma = [1/p - 1/2] / [1/p_0 - 1/2] = (p+1) / (2p)$. The constants
$C_{r_0}$ in~\eqref{A0}, $C'_p$ in~\eqref{A1} p.~\pageref{A1} depend only
on $p$, $p_0$ and $r_0$, hence on~$p$ alone, and they exist regardless of
$p > 3/2$ or not. By Section~\ref{AutrePreuve}, we know that
under~\eqref{EstimaGene}, the $(a_k)_{k \in \Z}$ in~\eqref{A3} satisfy 
$a_k \le (\delta_{0, g} + \delta_{1, g}) \ms1 a_{\alpha, k}$ with 
$(a_{\alpha, k})_{k \in \Z}$ universal. We obtain
\[
     \| \M_{\Kg} \|_{p \rightarrow p} 
 \le \| \Mg_K\|_{p \rightarrow p} + \kappa_p
 \le \kappa(p, \gr d)
      (1 + \lambda^{2 (1/ p_0 - 1/2) \gamma})
  =  \kappa(p, \gr d)
      (1 + \lambda^{2 / p - 1}),
\]
with $1 - 1/p_0 = (3p - 3) / (2 p + 2)$,
$k(p_0) = \lceil (p + 1) / (2p - 2) \rceil < p / (p - 1)$, and
\[
     \kappa(p, \gr d)
 \le \kappa(p) 
      \bigl(
       \Delta_{k(p_0)}^{ (4/3) (1 - 1/p_0) }
        \delta_{0, g}^{ (2/3) (1 - 1/p_0) }
      \bigr)^\gamma \Delta_1^{1-\gamma}
 \le \kappa(p) \Delta_{k_0(p)}^{ 1 - 1/p }
      \Delta_1^{ 1 - 1/p }.
 \qedhere
\]
\end{proof}

\subsection{The M\"uller strategy%
\label{StrateMull}}

\noindent
M\"uller prefers to work with another version $i^w$ of the fractional
integral $I^w$ from~\eqref{IntegraZero}. This version is defined when
$\Re w > 0$, beginning this time with $f \in C^\infty(\R)$, by the formula
\[
   (i^w f)(t) 
 = \frac 1 {\Gamma(w)}
    \int_t^2 (u - t)^{w - 1} f(u) \, \d u,
 \quad
 t \le 2.
 \label{FractiInt}
\]
The chosen limit $2$ is rather arbitrary, but will be quite convenient for
the computations that follow, in particular because $(2 - 1)^w = 1$ for
every $w$. Integrating by parts as we did for $I^w$ in
Section~\ref{FractiDeri}, we get 
\[
   (i^w f)(t) 
 = \frac {(2 - t)^w f(2)} {\Gamma(w + 1)}
   - \frac 1 {\Gamma(w + 1)}
      \int_t^2 (u - t)^{w} f'(u) \, \d u.
\]
This new formula makes sense for $\Re w > -1$ and defines a fractional
derivative~$d^z$ if $z = - w$ and $\Re z < 1$, by setting 
\begin{equation}
   (d^z f)(t)
 = \frac {(2 - t)^{-z} f(2)} {\Gamma(1 - z)}
   - \frac 1 {\Gamma(1 - z)}
      \int_t^2 (u - t)^{- z} f'(u) \, \d u,
 \quad
 t \le 2.
 \label{dalpha}
\end{equation}
Notice that $(d^0 f)(t) = f(2) - \int_t^2 f'(u) \, \d u = f(t)$. Continuing
integration by parts as in Section~\ref{FractiDeri}, we get successive
formulas defining $d^z f$, for each integer~$k$, which make sense for 
$\Re z < k$ and extend each other. Gluing them together, we can define
entire functions of $z$ for every $t$ fixed and every given function 
$f \in C^\infty(\R)$, for example $(d^z \gr 1)(1) = 1 / \Gamma(1 - z)$ if 
$f = \gr 1$. Suppose that $\Re z < 0$. From
\[
   (d^z f)(t) 
 = \frac 1 {\Gamma(- z)}
    \int_t^2 (u - t)^{- z - 1} f(u) \, \d u,
\]
we get for every integer $k \ge 1$ that
\begin{equation}
   (d^z f)(t)
 = E_k(z, t) 
    + (-1)^k {\frac 1 {\Gamma(k - z)}}
       \int_t^2 (u - t)^{-z + k - 1} \ms2 f^{(k)}(u) \, \d u,
 \label{dzOrdrek}
\end{equation}
a formula to be compared with~\eqref{DzOrdrek}, and where $E_k(z, t)$ is
equal to
\[
   E_k(z, t) 
 = \sum_{j=0}^{k-1} (-1)^j \ms2
    \frac { (2 - t)^{-z + j} \ms2 f^{(j)}(2) }     
          { \Gamma(j + 1 - z)} \ms1 \up.
\]
If $z$ is in $\C$, $t \le 2$ and $\Re z < k$, we can take~\eqref{dzOrdrek}
as definition for $(d^z f)(t)$.
\dumou

 When $-1 < \Re z < 0$, $f \in \ca S(\R^n)$ and $t < 2$, we see that
\begin{equation}
   (D^z f)(t) \ns1-\ns1 (d^z f)(t) 
 = ([I^{-z}  \ns3-\ns2 i^{-z}] f)(t)
 = \frac 1 {\Gamma(- z)}
    \int_2^{+\infty} (u - t)^{- z - 1} f(u) \, \d u.
 \label{Dmoinsd}
\end{equation}
This equality can be extended by analytic continuation to every $z \in \C$
with $\Re z > - 1$, or it can be proved by successive integrations by parts.
In particular, one has $(d^N f)(t)  = (D^N f)(t) = (-1)^N f^{(N)}(t)$ for
every integer $N \ge 0$ because $\Gamma(-N)^{-1} = 0$. As we did for
$D^\alpha$, when the function of $t$ does not have an explicit name, we use
the notation $d^\alpha_t f(2 t)$, and $d^\alpha_t f(2 t) \barre_{t = 1}$
for the value at $t = 1$.
\dumou

\begin{lem}[M\"uller~\cite{MullerQC}]%
\label{CPareil}
Let $m$ denote the Fourier transform of a kernel $K$ integrable on\/
$\R^n$. For every $\alpha \in (0, 1)$, the difference
\[
   (\xi \ps \nabla)^\alpha m(\xi)
    - d_t^\alpha m(t \xi) \barre_{t = 1} \ms2,
 \quad
 \xi \in \R^n,
\]
is a multiplier on $L^p(\R^n)$, $1 \le p \le +\infty$, with a norm bounded
by\/ $\|K\|_{L^1(\R^n)}$.
\end{lem}

\begin{proof}
By~\eqref{OpAlpha} we have
$
   (\xi \ps \nabla)^\alpha m(\xi)
 = D_t^\alpha m(t \xi) \barre_{t = 1}
$.
From~\eqref{Dmoinsd}, we get 
\[
   (\xi \ps \nabla)^\alpha m(\xi)
    - d_t^\alpha m(t \xi) \barre_{t = 1} \ms2 
 = \frac 1 {\Gamma(-\alpha)}
    \int_2^{+\infty} (u - 1)^{ - \alpha - 1} m(u \xi) \, \d u.
\]
The result follows by Lemma~\ref{EasyIntegral}, since
\[
   \frac 1 { |\Gamma(-\alpha)| } 
    \int_2^{+\infty} (u - 1)^{ - \alpha - 1} \, \d u
 = \frac 1 { | \ns3-\ns3 \alpha \ms1 \Gamma(-\alpha)| } 
 = \frac 1 { \Gamma(1-\alpha) } 
 < 1.
\] 
\end{proof}

 Thanks to the reduction from $(\xi \ps \nabla)^\alpha m(\xi)$ to
$d^\alpha_t m(t \xi) \barre_{t = 1}$ given by
Lemma~\ref{CPareil}, one can transform the condition~($ii$) of
Proposition~\ref{FouCarbe}. The objective now is to control the action on
$L^p(\R^n)$ of the multiplier $d^\alpha_t \mg(t \xi) \barre_{t = 1}$, for
some fixed $\alpha \in (1/p, 1)$ denoted by $\alpha = 1 - \varepsilon$,
where $\varepsilon > 0$ gets arbitrarily small when~$p$ tends to~$1$.
M\"uller embeds the \og objective\fge into the holomorphic family of
multipliers\label{HoloMull}
\begin{equation}
   m_z^\varepsilon (\xi)
 = (1 + |\xi|)^{1 - \varepsilon - z} \ms4 
     d^z_t \mg(t \xi) \barre_{t = 1},
 \ms{22} \Re z > -1,
 \label{mzFunc}
\end{equation}
and applies the complex interpolation scheme described in
Section~\ref{InterpoSchem}. For the value $z = \alpha = 1 - \varepsilon$,
one has
\[
   m_\alpha^\varepsilon (\xi)
 = m_{1 - \varepsilon}^\varepsilon (\xi)
 = d^{1 - \varepsilon}_t \mg(t \xi) \barre_{t = 1}
 = d^\alpha_t \mg(t \xi) \barre_{t = 1},
\]
which is the objective to be controlled. M\"uller studies this holomorphic
family for~$z \in \C$ varying in a strip of the form
$- \varepsilon \le \Re z \le \nu$, with $\nu > 0$ real. He shows by rather
long and delicate calculations that the multipliers $m_z^\varepsilon(\xi)$
are bounded functions of~$\xi \in \R^n$, for all~$z$ in this strip, not
uniformly in $z$, but with a $L^\infty(\R^n)$ norm of order
$\Gamma(z)^{-1}$. This allows him to control the action on $L^2(\R^n)$,
which is used for one end of the interpolation scale, the one
corresponding to $\Re z = \nu$.
\dumou

 The other end of the scale is $\Re z = - \varepsilon$, where the operator
associated to
\[
   m^\varepsilon_{-\varepsilon + \ii \tau}
 = (1 + |\xi|)^{1 - \ii \tau} \ms4 
    d^{-\varepsilon + \ii \tau}_t \mg(t \xi) \barre_{t = 1}
 = (1 + |\xi|)^{- \ii \tau} (1 + |\xi|) \ms4 
    d^{-\varepsilon + \ii \tau}_t \mg(t \xi) \barre_{t = 1}
\]
involves a \og small\fge fractional integration 
$d^{-\varepsilon + \ii \tau}$ of order $\varepsilon$, and a multiplication
on the Fourier side by $1 + |\xi|$. We will show that these multipliers 
$m^\varepsilon_{-\varepsilon + \ii \tau}$ are bounded on all the spaces
$L^r(\R^n)$, $1 < r < +\infty$. In order to do it, we shall have to work
mainly on the multiplier $|\xi| \ms1 \mg(\xi)$. The parameter $V(\Kg)$
appears when bounding the action of this multiplier on~$L^r(\R^n)$, and
the proof will use the dimensionless estimates for the Riesz transforms
given in~\eqref{TransfosRiesz}. Next, given $p$ in $(1, 2]$, we choose 
$p_0 \in (1, p)$, $\alpha \in (1/p_0, 1)$, and $\nu > \alpha$ which is a
function of $p, p_0, \alpha$. By interpolation between~$L^2(\R^n)$ (when 
$\Re z = \nu$) and~$L^{p_0}(\R^n)$ (when $\Re z = - \varepsilon$), we shall
obtain for the value $\alpha = 1 - \varepsilon$ the boundedness
on~$L^p(\R^n)$ of the multiplier~$m_\alpha^\varepsilon (\xi)$ that is our
\og objective\fg, thus proving Proposition~\ref{PropoMull}.
\dumou

 Let us comment on the formulas for the M\"uller multipliers. We know 
by~\eqref{NoticeFirst} in Corollary~\ref{EstimatesForC} that
differentiating $N$~times the function $t \mapsto \mg(t \xi)$ introduces a
factor of order $(1 + |\xi|)^{N - 1}$, which must be compensated for being
in a position to apply Parseval for the $L^2$ bound,
using~\eqref{EasyBound} as usual. This is done by multiplying by 
$(1 + |\xi|)^{1 - \varepsilon - \nu}$ when $z = \nu$. On the other hand, we
do not want a compensating factor when $z = \alpha$, where we want to
precisely recover our objective. The compensation will thus be of the form
$(1 + |\xi|)^{a z + b}$, with $a \nu + b = 1 - \varepsilon - \nu$ and 
$a \alpha + b = 0$. We then get a \og compensating factor\fge with a
positive power of $|\xi|$ for $\Re z < \alpha$, which becomes an additional
problem and requires more work.

\begin{smal}
\noindent
The interpolation strip technique has been often employed by Stein. For
example, in~\cite[Chapter~III, \S3]{SteinTHA}, for studying the maximal
function $\sup_{t > 0} |P_t f|$ of general semi-groups, Stein works
on a strip $S$ of the form $-1 \le \Re z \le N$. If $z = -1$, he
considers that the maximal inequality of Hopf concerns the derivative of
order~$-1$ of the semi-group, that is to say, its antiderivative
(multiplied by $t^z = t^{-1}$)
\[
   t^{-1} \ms1 D^{-1}_t (P_t f) 
 = \frac 1 t \ms1 \int_0^t (P_s f) \, \d s.
\]
By Hopf, this operator is known to be $L^p$ bounded, $1 < p < +\infty$.
Stein must check in addition that the extension to complex values in
the vertical line $z = - 1 + \ii \tau$ also gives bounded operators on
$L^p(\R^n)$. 

 Stein's objective is to study the maximal function of the semi-group
itself, which corresponds to the derivative of order $z = 0$. In order to do
this, he interpolates between Hopf in $L^{p_0}$, $p_0 < p < 2$, for 
$\Re z = -1$, and an $L^2$ estimate of derivatives of the semi-group, for
$\Re z = N$. For each integer $k$, the quantity
$
 t^k D^k_t (P_t f)
$
appears in the Littlewood--Paley function $g_k(f)$, so one can control
in~$L^2$ its maximal function, see Section~\ref{MaxiLittlePal}. The
holomorphic family is then defined by
$z \mapsto t^z D^z_t (P_t f)$, $z \in S$,
for a suitable version $D^z$ of fractional differentiation.

 The general strategy above was already applied in~\cite{SteinMET}
to the discrete case.

\end{smal}

\subsection{Model of proof: the Poisson case%
\label{Model}}

\noindent
For proving Theorem~\ref{TheoMull}, we have to apply Carbery's
Proposition~\ref{FouCarbe},~$(ii)$ to the difference $K = \Kg - P$.
M\"uller shows that $(\xi \ps \nabla)^\alpha \mg(\xi)$ acts on $L^p(\R^n)$
when $0 < \alpha < 1$ and $1 < p < +\infty$, and we need to verify that the
corresponding multiplier $(\xi \ps \nabla)^\alpha \widehat P(\xi)$ for the
Poisson kernel $P$ also acts on $L^p(\R^n)$, $1 < p < +\infty$, with
bounds independent of the dimension~$n$. This could be covered by
Proposition~\ref{PropoMull}, by observing that the Poisson kernel
$\Pn$ in~\eqref{PoissonDensi}, with Fourier transform $\e^{- 2 \pi |\xi|}$,
clearly satisfies~\eqref{EstimaGeneN} and has $V(\Pn)$ bounded
independently of~$n$ according to~\eqref{VPoiss}. We actually prefer to
take an opportunity to examine the structure of M\"uller's proof in a
simple case. When $\alpha \in (0, 1)$, we could find a shorter specific
proof, but the longer one that is given below provides a better
introduction to what follows in this Section~\ref{AMuller}. 
\dumou

 One sees that
$
   (\xi \ps \nabla)^\alpha \widehat P(\xi)
 = (2 \pi |\xi|)^\alpha \e^{- 2 \pi |\xi|}
$,
either by applying~\eqref{DerivExpo} that gives
$D^\alpha_t \e^{-\lambda |t|} = \lambda^\alpha \e^{-\lambda |t|}$ for
$\lambda, t > 0$, or by making use of the residue theorem.

\begin{smal}
\noindent
Indeed, according to~\eqref{OpAlpha} with $\xi = |\xi| \ms1 \theta$, one has
\[
   (\xi \ps \nabla)^\alpha \widehat P(\xi)
 = \int_\R (2 \ii \pi s |\xi|)^\alpha \varphi_\theta ( s ) 
    \e^{ - 2 \ii \pi s |\xi|} \, \d s
 = \int_\R (2 \ii \pi s |\xi|)^\alpha
    \frac {\e^{ - 2 \ii \pi s |\xi|}} {\pi (1 + s^2)} \, \d s,
\]
that can be computed using a contour formed of $[-R, R]$ with $R > 1$, and
of a half-circle of radius $R$ centered at $0$, located in the lower
complex half-plane.

\end{smal}

\noindent
We are going to bound the action on $L^p(\R^n)$ of the multiplier 
$|\xi|^\alpha \e^{-|\xi|}$ by the interpolation scheme of
Section~\ref{InterpoSchem}. Consider the holomorphic family of multipliers
\[
 \ca P_z(\xi) = |\xi|^z \e^{- |\xi|},
 \ms{16}
 \Re z \ge 0, 
 \ms8
 \xi \in \R^n.
\]
We will interpolate between $L^2(\R^n)$ and $L^{p_0}(\R^n)$, $p_0 > 1$
close to $1$. For proving the boundedness on~$L^2(\R^n)$, it is enough
by~\eqref{EasyBound} to see that the function 
$\xi \mapsto |\xi|^z \e^{- |\xi|}$ is bounded when $\xi$ varies in $\R^n$,
and since this function is radial, its supremum is independent of~$n$. If
we write $z = a + \ii b$, $a \ge 0$, we have
\begin{equation}
   \sup_{\xi \in \R^n} \ms2
    \sup_{\Re z = a} |\ca P_z(\xi)|
 = \sup_{\xi \in \R^n \ns1, \ms1 b \in \R} \ms2 
    \bigl\{ \bigl| |\xi|^{a + \ii b} \bigr| \e^{- |\xi|} \bigr\}
 = \sup_{r \ge 0} \ms2 \{ r^a \e^{- r} \}
 = a^a \e^{-a}.
 \label{Radiale}
\end{equation}
\dumou

 We work on a line $\Re z = \nu$, with $\nu$ \og large\fg, for dealing
with the $L^2$ boundedness, and the other line is $\Re z = 0$. For the
values $z = 0 + \ii b$, $b$ real, we know by~\eqref{ImagiPow} when 
$1 < r < +\infty$ that the norm on $L^r(\R^n)$ of the multiplier
$|\xi|^{\ii b}$ is bounded by $\lambda_r \ms1 \e^{\pi |b| / 2}$,
with~$\lambda_r$ independent of the dimension~$n$. The multiplier 
$\e^{- |\xi|}$ corresponds to the convolution with a Poisson probability
measure, so it is bounded by $1$ on $L^r(\R^n)$ when $1 \le r \le +\infty$
by~\eqref{LOneMultiplier}.
\dumou

 Let $\alpha \in (0, 1)$ be given. Consider $p \in (1, 2)$, introduce 
$p_0 = 2 p / (p+1) \in (1, p)$, making $1/p_0$ the midpoint between
$1$ and $1/p$. Then with $\theta = p - 1 \in (0, 1)$ we can check that
$
 1 / p = (1 - \theta) / p_0 + \theta / 2
$,
and we define $\nu$ by the condition
$
 \alpha = (1 - \theta) \ms1.\ms1 0 + \theta \nu
$, namely, we set $\nu = \alpha / (p - 1)$. Let $T_z$ be the operator
associated to the multiplier~$\ca P_z$. We have to estimate the norm of
$T_\alpha$ on $L^p$ by bounding
$
 \sca {T_{\alpha} f} g
$
uniformly for $f$ in the unit ball of $L^p(\R^n)$ and $g$ in the unit ball
of the dual $L^q(\R^n)$, where $1/q + 1/p = 1$. Consider the holomorphic
function 
\[
 H : z \mapsto \sca {T_z f_z} {g_z}
\]
where $f_z, g_z$ are as in~\eqref{FzGzDef}. The bounds obtained for the
family~$T_z$ do not allow us to apply directly the three lines
Lemma~\ref{TLL}, but Corollary~\ref{InStrip} will do the job. We got at the
boundary of the strip, for the norms~$\|T_{z}\|_{p_0 \rightarrow p_0}$ when
$\Re z = 0$, a bound of the form $O(\e^{ \kappa \ms1 |\Im z|})$. For every
real number $\tau$, the function $H$ satisfies 
\[
 |H(0 + \ii \tau)| \le \lambda_{p_0} \e^{\pi |\tau| / 2}
 \ms{12} \hbox{and also} \ms{12} 
 |H(\nu + \ii \tau)| \le \nu^\nu \e^{-\nu}.
\]
By Corollary~\ref{InStrip}, the value $H(\alpha)$ is bounded uniformly by a
quantity~$\eta$ depending on $p_0$, $\theta$ and on the width $w = \nu$ of
the strip, hence on $\alpha, p$ only. As explained in~\eqref{HalphaBound},
this gives then for the action of $T_\alpha$ on $L^p(\R^n)$ a bound
$\|T_\alpha\|_{p \rightarrow p} \le \eta$.
\dumou

 For applying Corollary~\ref{InStrip}, it remains to check that
$H$ has an admissible growth in~$S = \{ z : 0 < \Re z < \nu \}$. We may
actually reduce the discussion to a function $H$ bounded in the strip (but
without universal estimate). Indeed, one can observe that all operators
$T_z$, $z \in S$, are uniformly bounded on $L^2(\R^n)$, since 
$|\ca P_z(\xi)|$ is bounded by $\nu^\nu$ for all $\xi \in \R^n$ and $z$
in~$S$ by~\eqref{Radiale}. We may limit
ourselves to $f$, $g$ continuous with compact support, so that $f_z$,
$g_z$, $z \in S$, stay in a bounded subset of $L^2$, according
to~\eqref{FzGzBound}, implying that $H = H_{f, g}$ is bounded in
the strip.

\subsection{The interpolation part of Carbery's proof for 
 Theorem~\ref{TheoMaxi}%
\label{InterpoCarbe}}

\begin{proof}
In order to complete the proof of Theorem~\ref{TheoMaxi} and
Proposition~\ref{MoreGeneral}, it remains to show that the multiplier 
$(\xi \ps \nabla)^\alpha m(\xi)$, where $m$ is the Fourier transform of 
$K = \Kg - P$, is bounded on~$L^p(\R^n)$ for at least one value 
$\alpha > 1/p$ when $p > 3/2$. We have seen in the preceding
Section~\ref{Model} that $(\xi \ps \nabla)^\alpha \widehat P(\xi)$ is
bounded on~$L^p(\R^n)$, we need only consider now 
$(\xi \ps \nabla)^\alpha \mg(\xi)$. We will obtain the result by
interpolating between the boundedness on $L^1(\R^n)$, for 
$\alpha_0 = - \varepsilon$, and the boundedness on~$L^2(\R^n)$, for 
$\alpha_1 = 1 - \varepsilon$, of a certain holomorphic family $N_z(\xi)$
such that $N_\alpha(\xi)$ controls 
$(\xi \ps \nabla)^\alpha \mg (\xi)$. If $p > 3/2$ is fixed, its
conjugate~$q$ is $< 3$. We write 
\[
   \frac 2 3 
 > \frac 1 p 
 = \frac {1 - \theta} 1 + \frac \theta 2
 = 1 - \frac \theta 2 \ms1 \up,
 \ms{20}\hbox{thus}\ms{16}
   \frac \theta 2 
 = 1 - \frac 1 p 
 = \frac 1 q
\]
and $\theta = 2 / q > 2 / 3 > 1 - \theta / 2$. One can then find
$\varepsilon \in (0, 1)$ small enough, and independent of the
dimension~$n$, so that
\[
    \alpha 
 := (1 - \theta) (-\varepsilon) + \theta (1 - \varepsilon)
  = \theta - \varepsilon
  > 1 - \frac \theta 2
  = \frac 1 p \ms1 \up.
\]
We need $0 < \varepsilon \ms4{<}\ms5 3 \theta / 2 - 1$, we can set for example
$\varepsilon = 3 \theta / 4 - 1/2  = (p - 3/2) / p$. By
Lemma~\ref{CPareil}, it is enough to show that
$d^\alpha_t \mg(t \xi) \barre_{t=1}$ is bounded on $L^p$. Consider the
holomorphic family of multipliers $(N_z)$,
simpler than that of M\"uller, namely,
$
 N_z(\xi) := d^z_t \mg(t \xi) \barre_{t=1}
$
in the strip $- \varepsilon \le \Re z \le 1 - \varepsilon$. When 
$\Re z < 0$, we have
\begin{equation}
   N_z (\xi)
 = \frac 1 {\Gamma( - z) }
    \int_1^2 (u - 1)^{ - z - 1} \mg(u \xi) \, \d u,
 \label{NotDecor}
\end{equation}
and in particular
\[
   N_{-\varepsilon + \ii \tau}(\xi)
 = \frac 1 {\Gamma(\varepsilon - \ii \tau) }
    \int_1^2 (u - 1)^{\varepsilon - \ii \tau - 1} \mg(u \xi) \, \d u.
\]
We see that
\[
   \int_1^2 \bigl| (u - 1)^{\varepsilon - \ii \tau - 1} \bigr| \, \d u
 = \int_1^2 (u - 1)^{\varepsilon - 1} \, \d u
 = \varepsilon^{-1}
 < + \infty,
\]
thus $N_{-\varepsilon + \ii \tau}$ acts on $L^1$, with norm 
$\le 2 \ms1 \varepsilon^{-1}
 (1 + \tau^2)^{1/4 - \varepsilon/2} \e^{\pi |\tau| / 2}$, according to 
Lemma~\ref{EasyIntegral}, to the inequality~\eqref{MajoGamm} for the Gamma
function and since the $L^1$ norm of the kernel~$\Kg$ is equal to~$1$. When
$\Re z = 1 - \varepsilon$, we have by~\eqref{dalpha} that
\[
   N_{1 -\varepsilon + \ii \tau}(\xi)
 = \frac {\ms2 \mg (2 \xi)} {\Gamma(\varepsilon  - \ii \tau) } -
   \frac 1 {\Gamma(\varepsilon - \ii \tau) }
    \int_1^2 (u - 1)^{\varepsilon - \ii \tau - 1} 
     \xi \ps \nabla \mg(u \xi) \, \d u.
\]
The kernel $N_{1 -\varepsilon + \ii \tau}$ is a bounded function of $\xi$,
because we have $|\mg(2 \xi)| \le \delta_{0, g}$
and $ \bigl| u \xi \ps \nabla \mg(u \xi) \bigr| \le \delta_{1, g}$
by~\eqref{EstimaGene}. Using~\eqref{MajoGamm} we obtain
\begin{align*}
     \bigl| N_{1 -\varepsilon + \ii \tau}(\xi) \bigr|
 &\le \frac 1 {|\Gamma(\varepsilon - \ii \tau)| } 
      \Bigl( \delta_{0, g} + 
       \int_1^2 \bigl| (u - 1)^{\varepsilon - \ii \tau - 1} \bigr|
        \ms3 \frac {\delta_{1, g}} u \, \d u \Bigr)
 \\
 &\le 2 (\delta_{0, g} + \delta_{1, g}) \ms2
       \varepsilon^{-1} \ms1 (1 + \tau^2)^{1/4 - \varepsilon/2} 
        \ms1 \e^{\pi |\tau| / 2}.
\end{align*}
This shows that the operator associated to 
$N_{1 -\varepsilon + \ii \tau}(\xi)$ is bounded on $L^2(\R^n)$ with a
bound $O(\e^{\kappa |\tau|})$. We deal with this estimate as in the
preceding Section~\ref{Model}, and we obtain by interpolation that
$N_\alpha(\xi)$ is a $L^p(\R^n)$-multiplier. Remark~\ref{PolyFacto}
takes care of the polynomial factor 
$(1 + \tau^2)^{1/4 - \varepsilon/2} \le (1 + \tau^2)^{1/4}$. 
By Lemma~\ref{CPareil} and~\eqref{PolyBd} with $w = 1$, $c_j = 1/4$,
$u_j = \pi / 2$, and since $\delta_{0, g} \ge 1$ we get
\[
     \| (\xi \ps \nabla)^\alpha \mg(\xi) \|_{p \rightarrow p}
 \le 1 
     + \frac {2 \ms1 p} {p - 3/2} \ms2 
      \Bigl( \frac 3 2 \Bigr)^{1/2} 
      \e^{\pi / 4} 
       \ms1 (\delta_{0, g} + \delta_{1, g})^{2 - 2 / p}
 \le \kappa_p \ms1 \Delta_1^{2 - 2 / p}.
\]
\dumou

 We now check that the function 
$H(z) = \sca {N_z f_z} {g_z}$ of~\eqref{UsualH} has an admissible growth 
in~$S = \{ - \varepsilon \le \Re z \le 1 - \varepsilon\}$. We
may again observe that all kernels $N_z(\xi)$ are bounded functions of $\xi$.
Indeed, $N_z(\xi)$ can be expressed in the whole strip by
\[
   N_z(\xi)
 = \frac {\mg (2 \xi)} {\Gamma(- z + 1) } -
   \frac 1 {\Gamma(- z + 1) }
    \int_1^2 (u - 1)^{- z} 
     \xi \ps \nabla \mg (u \xi) \, \d u,
 \quad
 \xi \in \R^n,
\]
so that 
$|N_z(\xi)| \le \kappa_{\varepsilon, \delta} (1 + \tau^2)^{1/4} 
   \ms1 \e^{\pi |\tau| / 2}$. Next, we can assume that the two
functions~$f, g$ appearing in the definition of $H$ are bounded with
bounded support, and argue with~\eqref{FzGzBound} as at the end of
Section~\ref{Model}, obtaining that
$|H(z)| \le \kappa \|N_z\|_{2 \rightarrow 2}
 \le \kappa'_\varepsilon (1 + \tau^2)^{1/4} \ms1 \e^{\pi |\tau| / 2}$, a
growth admissible for applying Corollary~\ref{InStrip}.
\end{proof}

 We see pretty well why M\"uller finds a better result than the one given
by the preceding argu\-ment, which suffices for Carbery's theorem. It is
because M\"uller is able to make use of multipliers more difficult to
handle, which contain an extra factor $|\xi|$ on the line 
$\Re z = - \varepsilon$, for example
$
   m^\varepsilon_{-\varepsilon} (\xi)
 = (1 + |\xi|) N_{-\varepsilon} (\xi)
$
when $z = - \varepsilon$. This factor~$|\xi|$ is precisely the one that
will be treated by the geometrical parameter~$q(C)$. On the other hand, 
M\"uller's approach is not better when $p > 3/2$, since the result is known
in this case without assumption on~$q(C)$.

\begin{rem} The factor $1 / \Gamma(-z)$ in~\eqref{NotDecor} is not purely
decorative. Without it, $N_z(\xi)$ would have a \og pole\fge at $z = 0$,
which is compensated by the zero of $1 / \Gamma(z)$ at $0$. One could
perhaps get away here with a less sophisticated factor such as 
$z / (a - z)$, with $a$ real and $> 1$. See also Remark~\ref{ButPoly}.
\end{rem}

\subsection{Upper bounds for the functions  
 $\xi \mapsto m^\varepsilon_z (\xi)$%
\label{LesMajos}}

\noindent
We present a version of M\"uller's upper bounds for the functions
$m^\varepsilon_z$ defined in~\eqref{mzFunc}. M\"uller's bounds
in~\cite{MullerQC} are not fully explicit since they use asymptotic
estimates, but they do not contain the annoying factor $\varepsilon^{-1}$
that our somewhat shorter proof introduces below.

\def\MajoMulAref{after M\"uller~\cite[Lemma~2]{MullerQC}}%
\begin{lem}[\MajoMulAref]%
\label{MajoMullA}
Assume that the kernel $\Kg$ integrable on\/~$\R^n$
satisfies\/~\eqref{EstimaGeneN}, let $\varepsilon \in (0, 1)$, let
$\nu \ge 1 - \varepsilon$ and set $\ell = \lceil \nu + \varepsilon \rceil$.
For every $z \in \C$ such 
that~$-\varepsilon \le \Re z \le \nu$, one has that
\[
 \forall \xi \in \R^n,
 \ms{9}
     | m_z^\varepsilon(\xi) |
 \le \kappa_\nu \ms2 \varepsilon^{-1} 
      \Delta_\ell  \ms2 
       (1 + (\Im z)^2)^{\nu/2 - 1/4}
        \ms1 \e^{\pi |\Im z| / 2},
\]
where $\kappa_\nu = 4 \ms1 \Gamma \bigl( \max(\nu, 2) \bigr)
 \ms2 (3/2)^{\nu - 1/2} \e^{\pi/4}$ and where $\Delta_\ell$ is defined
at\/~\eqref{DefiDel}.
\end{lem}

 One of the difficulties in M\"uller's article is the following: with the
operator $D^\alpha$, we have been able to compute certain integrals by the
residue theorem, on entire half-lines. The corresponding values
for~$d^\alpha$ are less pleasant, because they involve bounded segments,
and quarters of circle at finite distance whose contribution is not zero.
Let us mention another difficulty, somewhat related to the latter. If we
know that $|D^z_t m(t \theta)| \le \kappa (1 + |t|)^{-1}$ for every $t$
real and $\theta \in S^{n-1}$, then by the homogeneity
relation~\eqref{Dilate} we get
$|D^z_t m(t \xi)| \le \kappa |\xi|^{\Re z} (1 + |t \xi|)^{-1}$ for 
$\xi \in \R^n$, but this kind of behavior is not clear for $d^z$. The more
delicate analysis of~\cite{MullerQC} will not be given here, but some special
cases are rather easy. Indeed, the computation is not difficult when 
$\Re z = k - \varepsilon$, for every integer $k \ge 0$. We will however be
able to deduce Lemma~\ref{MajoMullA} from the easy cases that are treated
in the next lemma.

\begin{lem}%
\label{MajoMull}
Assume that $\Kg$ is integrable on\/~$\R^n$ and
satisfies\/~\eqref{EstimaGeneN}. For every $\varepsilon \in (0, 1)$, every
integer $k \ge 0$ and $z \in \C$ such that\/ $\Re z = k - \varepsilon$, one
has
\[
 \forall \xi \in \R^n,
 \ms{9}
     | m_z^\varepsilon(\xi) |
 \le 2 \ms1 \kappa'_k \ms2 \varepsilon^{-1} \Delta_k \ms1
      (1 + (\Im z)^2)^{ k^*/2 - \varepsilon/2 - 1/4} 
       \ms1 \e^{\pi |\Im z| / 2},
\]
where $k^* = \max(k, 1)$, $\kappa'_k \le \Gamma(k - \varepsilon)$ for 
$k \ge 3$ and $\kappa'_0, \kappa'_1, \kappa'_2 \le 1$. 
\end{lem}

\begin{proof}
We first give the proof for $k = 0$, when $z = -\varepsilon + \ii \tau$.
We have
\[
    m^\varepsilon_{-\varepsilon + \ii \tau}(\xi)
 = (1 + |\xi|)^{1 - \varepsilon - (-\varepsilon + \ii \tau)} \ms2
    \frac 1 {\Gamma(\varepsilon - \ii \tau)}
     \int_1^2 (u - 1)^{\varepsilon - \ii \tau - 1}
      \mg(u \xi) \, \d u
\]
and it follows that
\[
     |m^\varepsilon_{-\varepsilon + \ii \tau}(\xi)|
 \le \Bigl| \frac 1 {\Gamma(\varepsilon - \ii \tau)} \Bigr|
      \int_1^2 (u - 1)^{\varepsilon - 1} 
       (1 + |\xi|) \ms3 |\mg(u \xi)| \, \d u.
\]
By~\eqref{EstimaGeneN}, we know that
$|\mg(u \xi)| \le \delta_{0, g} (1 + |\xi|)^{-1}$ when $u \ge 1$, thus
\[
     |\Gamma(\varepsilon - \ii \tau) \ms2 
       m^\varepsilon_{-\varepsilon + \ii \tau}(\xi)|
 \le \delta_{0, g} \int_1^2 (u - 1)^{\varepsilon - 1} \, \d u
  =  \frac {\delta_{0, g}} \varepsilon \ms1 \up.
\]
Using also~\eqref{MajoGamm}, this simplest case reads as
\[
     \|m^\varepsilon_{-\varepsilon + \ii \tau} \|_\infty
 \le 2 \ms1 \delta_{0, g} \ms1 \varepsilon^{-1} \ms1 
      (1 + \tau^2)^{1/4 - \varepsilon/2} \e^{\pi |\tau| / 2},
 \quad
 \tau \in \R,
\]
and $1/4 = k^*/2 {-} 1/4$ here. For $k > 0$, we have by~\eqref{dzOrdrek}
with $z = k - \varepsilon + \ii \tau$ that
\[
   d^{k - \varepsilon + \ii \tau}_t \mg(t \xi) \barre_{t=1}
 = E_k + (-1)^k \frac 1 {\Gamma(\varepsilon - \ii \tau)}  
        \int_1^2 (u - 1)^{\varepsilon - \ii \tau - 1}
         \frac {\d^k} {\d u^k \ns2} \ms2 \mg(u \xi) \, \d u,
\]
where
\[
   E_k 
 = \sum_{j=0}^{k-1} (-1)^j \ms2
    \frac {\frac {\d^j} {\d u^j \ns2} \ms2 \mg(u \xi) \barre_{u = 2}}
          { \Gamma(j + 1 - k + \varepsilon - \ii \tau)} \ms1 \up.
\]
By our assumption~\eqref{EstimaGeneN}, the function $u \mapsto \mg(u \xi)$
satisfies
\begin{equation}
 \forall u \ge 1,
 \ms{16}
     \Bigl| \frac {\d^j} {\d u^j \ns2} \ms2 \mg(u \xi) \Bigr|
  =  \Bigl|
      \frac {\d^j} {\d u^j \ns2} \ms2 
       \mg \bigl( u \ms1 |\xi| \ms1 \theta \bigr) 
     \Bigr|
 \le \delta_{j, g} \ms2 \frac {|\xi|^j} {1 + |\xi|}
 \label{Notice}
\end{equation}
for each integer $j \ge 0$, if $\xi \ne 0$ and
$\theta = |\xi|^{-1} \xi$. This yields
\[
     \Bigl| \int_1^2 (u - 1)^{\varepsilon - \ii \tau - 1}
        \frac {\d^k} {\d u^k \ns2} \ms2 \mg(u \xi) \, \d u \Bigr|
 \le \delta_{k, g} \int_1^2 (u - 1)^{\varepsilon - 1}
        \frac {|\xi|^k} {1 + |\xi|} \, \d u
  =  \frac {\delta_{k, g}} \varepsilon
        \frac {|\xi|^k} {1 + |\xi|} \ms1 \up.
\]
For the terms in the expression $E_k$, we have by~\eqref{Notice} that
\[
     \Bigl|
      \frac {\d^j} {\d u^j \ns2} \ms2 \mg(u \xi) \barre_{u = 2}
     \Bigr|
 \le \delta_{j, g} \ms2 \frac {|\xi|^j} {1 + |\xi|}
 \le \delta_{j, g} \ms2 \frac {(1 + |\xi|)^j} {1 + |\xi|}
 \le \delta_{j, g} \ms2 (1 + |\xi|)^{k-1},
 \quad
 j = 0, \ldots, k-1.
\]
Recalling $\Delta_k = \sum_{j=0}^k \delta_{j, g}$ and~\eqref{MajoGamm}
with $a = -k + 1 + \varepsilon$, we get 
\begin{align}
  &\ms4   |m^\varepsilon_{k - \varepsilon + \ii \tau} (\xi)|
  = (1 + |\xi|)^{1 - \varepsilon - (k - \varepsilon)} \ms2
     \bigl|
      d^{k - \varepsilon + \ii \tau}_t \mg(t \xi) \barre_{t=1} 
     \bigr|
 \notag
 \\
 \le&\ms4 \Delta_k \ms2 \varepsilon^{-1} \ms1 
      (1 + |\xi|)^{1 - k}
       \ms3 (1 + |\xi|)^{k - 1} \ms2 
        \max \{
              |\Gamma(\varepsilon - \ii \tau - j_1)|^{-1} : 0 \le j_1 \le k-1
             \}
 \label{MajoK}
 \\
 \le&\ms4 \beta_a 
           \ms2 \Delta_k \ms2 \varepsilon^{-1} \ms1 
            (1 + \tau^2)^{1/4 + (k-1-\varepsilon)/2} \e^{\pi |\tau| / 2}.
 \notag
\end{align}
We may take $\beta_a = 2$ when $k \le 2$ and 
$\beta_a = 2 \ms1 \Gamma(k - \varepsilon)$ otherwise. 
\end{proof}

\begin{rem}
One could not make the same simple computation for $k - \varepsilon'$ when
$\varepsilon' > \varepsilon$. Indeed, we have then
\[
   m^\varepsilon_{k - \varepsilon'}(\xi)
 = (1 + |\xi|)^{1 -(k - \varepsilon') - \varepsilon}
    d^{k - \varepsilon' + \ii \tau}_t \mg(t \xi) \barre_{t=1},
\]
so $|m^\varepsilon_{k - \varepsilon'}(\xi)|$ contains the factor
$(1 + |\xi|)^{1 - k + \varepsilon' - \varepsilon}$, that is not controllable
by the preceding proof when $\varepsilon' > \varepsilon$. With one
more integration by parts in the log-concave case we obtain
\[
   \frac {\d^j} {\d u^j \ns2} \ms2 \mlc(u \xi)
 = \frac {(-2 \ii \pi |\xi|)^{j-2} \ns{14}} {u^2} \ms6
    \int_\R \bigl( s^j \varphi_\theta(s) \bigr)''
     \e^{ - 2 \ii \pi s u \ms1 |\xi|} \, \d s
\]
that seems to give an additional improvement, able to swallow the bad factor
$|\xi|^{\varepsilon' - \varepsilon}$ above. However, we would need now for
$
 \int_\R |s|^j |\varphi''_\theta(s)| \, \d s
$
a universal bound that does not exist, and actually, this integral does
not make sense in general.
\dumou
\dumou

 When $k \ge 1$, the kernel 
$\widetilde m^\varepsilon_{k-\varepsilon + \ii \tau}
 := (1 + |\xi|)^{1 - k - \ii \tau}
     D^{k - \varepsilon + \ii \tau}_u \mg(u \xi) \barre_{u=1}$ is even
easier to bound since we can write directly
\[
     \Bigl| \int_1^{+\infty} (u - 1)^{\varepsilon - \ii \tau - 1}
            \frac {\d^k} {\d u^k \ns2} \ms2 \mg(u \xi) \, \d u 
     \Bigr|
 \le \delta_{k, g} 
      \Bigl( 
       \int_1^{+\infty} (u - 1)^{\varepsilon - 1} \frac {\d u } u 
      \Bigr) \ms2 |\xi|^{k - 1},
\]
but $D^{- \varepsilon + \ii \tau}_u \mg(u \xi) \barre_{u=1}$ is not a
bounded function of $\xi$ in the neighborhood of~$\xi = 0$. For example, we
have $D^{-\varepsilon}_u \e^{-u|\xi|} \barre_{u=1}
 = |\xi|^{-\varepsilon} \e^{- |\xi|}$. Thus
$\widetilde m^\varepsilon_{-\varepsilon}$ is not an $L^2$ multiplier,
nor an $L^p$ multiplier for any $p \ne 2$, and this justifies working with
$d^z_t$ instead.
\end{rem}

\begin{proof}[Proof of Lemma~\ref{MajoMullA}]
Let $\nu > 1 - \varepsilon$ and 
$\ell = \lceil \nu + \varepsilon \rceil \ge 2$, so that 
$\nu \le \ell - \varepsilon$. If $\Re z < \ell$, we have
by~\eqref{dzOrdrek} that
\[
   d^z_t \mg(t \xi) \barre_{t=1}
 = E_\ell(z) 
    + (-1)^\ell \frac 1 {\Gamma(\ell - z)}  
       \int_1^2 (u-1)^{-z + \ell - 1}
        \frac {\d^\ell} {\d u^\ell \ns2} \ms2 \mg(u \xi) \, \d u,
\]
with
$
   E_\ell(z) 
 = \sum_{i=0}^{\ell-1} (-1)^i \ms2 \Gamma(i + 1 - z)^{-1}
    \frac {\d^i} {\d u^i \ns2} \ms2 \mg(u \xi) \barre_{u = 2}
$.
We fix $\xi \in \R^n$ and consider the holomorphic function
\[
 H_\xi : z \mapsto  m_z^\varepsilon (\xi)
          = (1 + |\xi|)^{1 - \varepsilon - z}
             d^z_t \mg(t \xi) \barre_{t=1}
\]
in the strip $-\varepsilon < \Re z < \ell - \varepsilon$. We have
$
      \bigl|  ( \d^i / \d u^i ) \ms1 \mg(u \xi) \bigr|
  \le \delta_{i, g} |\xi|^i  
$ 
by~\eqref{EstimaGeneN}, and it follows from~\eqref{MajoGamm} that 
$|H_\xi(z)| \le \kappa \e^{\kappa |\tau|}$ in the strip, with $\kappa$
depending on~$|\xi|$. Consider an arbitrary $z_0$ such that 
$1 - \varepsilon \ms4{<}\ms5 \nu_0 := \Re z_0 \le \nu$. Let $k$ integer be
such that $k - \varepsilon < \nu_0 \le k + 1 - \varepsilon$, thus 
$1 \le k < \ell$. By Lemma~\ref{MajoMull}, when $\Re z = k - \varepsilon$
or $\Re z = k + 1 - \varepsilon$, we have for~$H_\xi(z)$ the good bound 
\begin{equation}
     |H_\xi(z)| 
 \le 2 \ms1 \kappa'_{\Re z + \varepsilon} \Delta_{\Re z + \varepsilon} 
      \ms2 \varepsilon^{-1}
       (1 + (\Im z)^2)^{\Re z / 2 - 1/4} \e^{\pi |\Im z| / 2}.
 \label{HdeXideZ}
\end{equation}
We write $\Delta_k < \Delta_{k + 1} \le \Delta_\ell$ and 
$\nu_0 + \varepsilon = (1 - \theta) k + \theta (k+1)$. When $k \ge 3$, we
have
\begin{align*}
      {\kappa'}_k^{1 - \theta} {\kappa'}_{k+1}^\theta
 &\le \Gamma(k {-} \varepsilon)^{1 - \theta} 
       \Gamma(k {+} 1 {-} \varepsilon)^\theta
  =   \frac { \Gamma(k {+} 1 {-} \varepsilon) } 
           { (k - \varepsilon)^{1 - \theta} }
 \\
 &\le \frac { \Gamma(\nu_0)^\theta \ms1 \Gamma(\nu_0 + 1)^{1-\theta} \ns{16} } 
            { (k - \varepsilon)^{1 - \theta} } \ms8
  =   \frac { \nu_0^{1 - \theta} } { (k - \varepsilon)^{1 - \theta} \ns{20} }
       \ms{22} \Gamma(\nu_0)
  <   2 \ms2 \Gamma(\nu),
\end{align*}
and ${\kappa'}_1^{1 - \theta} {\kappa'}_2^\theta \le 1$,
${\kappa'}_2^{1 - \theta} {\kappa'}_3^\theta \le 2$.
By Corollary~\ref{InStrip} and Remark~\ref{PolyFacto},
\eqref{PolyBd} with $w = 1$ and $c_j = (k + j - \varepsilon) / 2 - 1/4$, 
$j = 0, 1$, we get for $|H_\xi(z_0)|$ a bound
\[
 4 \ms1 \Gamma(\max(\nu, 2)) \ms1 (3 / 2)^{\Re z_0 - 1/2} \e^{ \pi / 4}
  \varepsilon^{-1} \Delta_\ell \ms1
  (1 + (\Im z_0)^2)^{\Re z_0/2 - 1/4} \e^{\pi |\Im z_0| / 2}.
\]
This proves Lemma~\ref{MajoMullA} when
$1 - \varepsilon < \Re z_0 \le \nu$. The case
$-\varepsilon \le \Re z_0 \le \nu = 1 - \varepsilon$ is left to the reader,
one has $k = 0$ and the polynomial component of the bound is then 
$(1 + \tau^2)^{\nu / 2 - 1 / 4}$ on both sides of the strip
$-\varepsilon < \Re z < 1-\varepsilon$.
\end{proof}

\begin{smal}
\noindent
An alternative proof could go like this: divide the integral $\int_1^2$
in the definition of $d^z_t \mg(t \xi) \barre_{t = 1}$ into 
$\int_1^{1 + a}$ and $\int_{1+a}^2$, for some suitable $a \in [0, 1]$. For
the first integral $\int_1^{1 + a}$, we modify~\eqref{dzOrdrek} and get when
$-1 < \Re z < 0$ that
\[
   d_{z, 1}(a)
 := {\frac 1 {\Gamma(- z)}}
      \int_1^{1 + a} (u - 1)^{-z -1} \ms1 
       \mg(u \xi) \, \d u
\]
\[
 = E_{k+2, a}(z) 
    \ns1+\ns1 (-1)^{k+2} {\frac 1 {\Gamma(k + 2 - z)}}
       \int_1^{1 + a} \ns4 (u - 1)^{-z + k + 1} \ms1 
        \frac {\d^{k+2}} {\d u^{k+2} \ns2} \ms2 \mg(u \xi) \, \d u
\]
for every integer $k \ge -1$, where $E_{k+2, a}(z)$ is equal to
\[
   E_{k+2, a}(z) 
 = \sum_{j=0}^{k+1} (-1)^j \ms4
    \frac { a^{-z + j } \ms2 
            \frac {\d^j} { \d u^j \ns2} \ms3 
             \mg \bigl( u \ms1 \xi \bigr) \barre_{u = 1 + a} }     
          { \Gamma(j + 1 - z) } \ms1 \up.
\]
Let now $-1 < \Re z \le \nu$ and write $z = k + \sigma + \ii \tau$ with $k$
integer and $0 < \sigma \le 1$. Applying the preceding formulas it follows 
by~\eqref{EstimaGeneN} that
\[
   |d_{z, 1}(a)|
 \le \sum_{j=0}^{k+1}          
         \frac { a^{-k -\sigma + j} \ms1 \delta_{j, g} \ms1 |\xi|^j }
               { |\Gamma(j + 1 - z)| \ms2 (1 + |\xi|) }             
         +        
         \frac { (2 - \sigma)^{-1} a^{2 - \sigma} \ms1 
                  \delta_{k+2, g} \ms1 |\xi|^{k+2} }
               { |\Gamma(k + 2 - z)| \ms2 (1 + |\xi|) } \ms1 \up.
\]
When $|\xi| \le 1$, we choose $a = 1$ and obtain
$|d_{z, 1}(a)| \le C_k(z) \ms1 (1 + |\xi|)^{-1}$ where
\[
   C_k(z)
 = \Delta_{k+2} \ms2
    \max \bigl\{ |\Gamma(i - z)|^{-1} : 0 \le i \le k+2 \bigr\}.
\]
When $|\xi| > 1$, we let $a = |\xi|^{-1}$ and get
$
     |d_{z, 1}(a)|
 \le C_k(z) \ms1 |\xi|^{k + \sigma} \ms1 (1 + |\xi|)^{-1}
$.
The other term $d_{z, 2}(a)$, corresponding to $\int_{1+a}^2$, is zero
when $|\xi| \le 1$ since $a = 1$ in this case. Otherwise, we have
$a = |\xi|^{-1}$ and assuming $k + \sigma \ne 0$, we get
\[
     |d_{z, 2}(a)|
  =  \Bigl| \frac 1 {\Gamma(- z) }
      \int_{1+|\xi|^{-1}}^2 (u-1)^{-z-1} \ms2 \mg(u \xi) \, \d u \Bigr|
 \le \frac 1 { |\Gamma(- z)| } 
      \frac { |\xi|^{k + \sigma} \ns1-\ns1 1 } {k + \sigma}
       \frac {\delta_{0, g}} {1 + |\xi| } \ms1 \up.
\]
There is no problem as long as $\Re z = k + \sigma$ is not close to $0$.
Otherwise, we can apply
$
     \bigl| \ms1 |\xi|^{k + \sigma} \ns1-\ns1 1 \bigr|
 \le |k + \sigma| \ms3 |\xi|^{ (k+\sigma)^+ } \ms0 \ln |\xi|
$, where $t^+ = \max(t, 0)$. Summing up and letting 
$L(\xi) = 1 + (\ln |\xi|)^+$, we have when $-1 < \Re z =: s < \nu$ that
\[
     \bigl| d^z_t \mg(t \xi) \barre_{t = 1} \bigr|
 \le C_\nu(z)
      \bigl(
       [1 + |s|^{-1}] \wedge L(\xi) 
      \bigr)
       (1 + |\xi|)^{ s^+ - 1},
\] 
giving bounds multiple of $(1 + |\xi|)^{-1}$, 
$(1 + |\xi|)^{-1} L(\xi)$ for $s$ in $[-1, -\varepsilon / 2]$ and
$[-\varepsilon / 2, 0]$ respectively, $(1 + |\xi|)^{s - 1} L(\xi)$ and 
$(1 + |\xi|)^{s - 1}$ in $[0, \varepsilon / 2]$ and 
$[\varepsilon / 2, \nu]$ respectively. For $m^\varepsilon_z(\xi)$, we get
bounds multiple of $1$, $(1 + |\xi|)^{-\varepsilon / 2} L(\xi)$ and
$(1 + |\xi|)^{-\varepsilon}$ for $s$ in 
$[- \varepsilon, - \varepsilon / 2]$, $[-\varepsilon / 2, \varepsilon/2]$
and $[\varepsilon/2, \nu]$ respectively. This shows that
$m^\varepsilon_z(\xi)$ is a bounded function of $\xi \in \R^n$.
\end{smal}

\subsection{Lemma 4 of M\"uller's article%
\label{LemmeQuatre}}

\noindent
We must control the action on $L^p(\R^n)$, $p > 1$ close to~$1$, of
multipliers $m_z^\varepsilon$ when $\Re z = - \varepsilon$. If 
$z = - \varepsilon + \ii \tau$, we have
\[
     \Gamma(\varepsilon - \ii \tau) \ms2
      m_{- \varepsilon + \ii \tau}^\varepsilon (\xi)
 = (1 + |\xi|)^{1 - \ii \tau} \ms2
     \int_1^2 (s - 1)^{\varepsilon - \ii \tau - 1} \mg(s \xi) \, \d s.
\]
Since
$
   \int_1^2 \bigl| (s - 1)^{\varepsilon - \ii \tau - 1} 
    \bigr| \, \d s 
 = \varepsilon^{-1}
$,
it is enough to bound uniformly in $s \in [1, 2]$ the norm of
$
    n_s(\xi)
 := (1 + |\xi|)^{1 - \ii \tau} \mg(s \xi)
$.
This multiplier can be decomposed into several parts: first 
$(1 + |\xi|)^{- \ii \tau}$, which is taken care of by
Proposition~\ref{LaplaceMultip} on multipliers of Laplace type. Indeed,
replacing $\lambda$ by $1 + \lambda$ in~\eqref{LaplaPow} and integrating by
parts, one finds that
$
   (1 + \lambda)^{- \ii \tau}
 = \lambda \int_0^{+\infty} \e^{- \lambda t} a_\tau(t) \, \d t
$,
with
\begin{equation}
   a_\tau(t) 
 = \frac 1 {\Gamma(1 + \ii \tau)}
        \Bigl(
         t^{\ii \tau} \e^{- t} + \int_0^t s^{\ii \tau} \e^{-s} \, \d s
        \Bigr)
 \label{LaplacEstim}
\end{equation}
and
$    |a_\tau(t)|
 \le \bigl| \Gamma(1 + \ii \tau) \bigr|^{-1}
 \le (1 + \tau^2)^{-1/4} \e^{\pi |\tau| / 2}
$
according to~\eqref{IneGamma}. Next, in $n_s(\xi)$, we have 
$(1 + |\xi|) \ms1 \mg(s \xi)$, which is formed of $\mg(s \xi)$, multiplier
bounded by~$\|\Kg\|_{L^1(\R^n)}$ on all $L^p(\R^n)$ spaces, and of 
$s^{-1} \ms2 |s \xi| \ms1 \mg(s \xi)$, $s > 1$, with a multiplier norm less
than that of $|\xi| \ms1 \mg(\xi)$, according to~\eqref{InvariMul}. 
\dumou

 Given an integrable kernel $K$ on $\R^n$ and its Fourier transform $m$,
the question boils down to handling the crucial multiplier 
\begin{equation}
 m^{\#}(\xi) := |\xi| \ms1 m(\xi).
 \label{MZero}
\end{equation}
We summarize the latter discussion in the lemma that follows, where we
include the bound $2 \ms1 (1 + \tau^2)^{1/4} \e^{\pi |\tau| / 2}$
from~\eqref{MajoGamm} for the factor 
$|\Gamma(\varepsilon - \ii \tau)|^{-1}$ that was left apart above. So far,
the kernel $K$ can be arbitrary in~$L^1(\R^n)$.

\begin{lem}%
\label{LquatreA}
Let $p$ belong to\/ $(1, 2]$. One has that
\[
     \| m^\varepsilon_{-\varepsilon + \ii \tau} \|_{p \rightarrow p}
 \le 2 \ms2 \varepsilon^{-1} \ms1 \lambda_p \ms2
      \e^{\pi |\tau| } \ms2 
       (\|\Kg\|_{L^1(\R^n)} + \| m^{\#} \|_{p \rightarrow p}),
 \quad \tau \in \R,
\]
where $\lambda_p$ is the constant appearing in
Proposition~\ref{LaplaceMultip}.
\end{lem}

 The serious work will be done in the proof of the following essential
lemma.

\begin{lem}%
\label{LquatreB}
Let $\Kg$ be a kernel integrable on\/~$\R^n$
satisfying\/~\eqref{EstimaGene}, and $\mg$ its Fourier transform. Let
$\mg^{\#}$ be defined by\/~\eqref{MZero} and $p \in (1, 2]$. One has that
\[
     \| \mg^{\#} \|_{p \rightarrow p}
 \le (2 \pi)^{1 - 2/p} \ms1 \rho_p \ms2 \delta_{0, g}^{2 - 2 / p} \ms1 
      V(\Kg)^{ - 1 + 2/p},
\]
where $\rho_p$ is the constant from\/~\eqref{TransfosRiesz} and where
$V(\Kg)$ is defined at\/~\eqref{V(K)}.
\end{lem}
\dumou

 The proof of Lemma~\ref{LquatreB} will be broken into several easy
statements. Some of them are used again in Section~\ref{LeCube}. To begin
with, we merely assume that $K$ is an integrable kernel on $\R^n$ having
partial derivatives $\partial_j \ms1 K$ in the sense of distributions that
are (signed) measures $\mu_j$, and we let $m = \widehat K$. We can express
$m^{\#}(\xi)$ with the help of the Riesz transforms $(R_j)_{j=1}^n$
introduced in Section~\ref{TransfoRiesz}, writing
\[
   2 \pi \ms2 m^{\#}(\xi) 
 = \sum_{j=1}^n \frac {- \ii \xi_j} {|\xi|} \ms2 
    (2 \ii \pi \xi_j) \ms1 m(\xi).
\]
The functions $(2 \ii \pi \xi_j) \ms1 m(\xi)$, $j = 1, \ldots, n$, are the
Fourier transforms of the measures $\mu_j = \partial_j \ms1 K$. When $K$ is
the uniform probability density $K_C$ on a symmetric convex set~$C$, the
$\mu_j \ms1$s are supported on the boundary of $C$, and we shall see below
that $V(K_C) = q(C)$ if $C$ is isotropic and normalized by variance.
\dumou

 The convolution operator $T_{m^{\#}}$ can thus be written under the form
\[
 T_{m^{\#}} : f \mapsto
 T_{m^{\#}} f = (2 \pi)^{-1} \sum_{j=1}^n R_j \mu_j * f.
\]
Riesz transforms commute with convolutions. If $g$ is in the dual $L^q$
of~$L^p$, we have
\begin{align*}
   2 \pi \ms1 |\sca {T_{m^{\#}} f} g |
 &= \Bigl| \sum_{j=1}^n \sca {R_j \mu_j * f} g \Bigr|
 = \Bigl| \sum_{j=1}^n \sca {(R_j f) * \mu_j} g \Bigr|
 \\
  &=  \Bigl| \sum_{j=1}^n \sca {R_j f} {\widetilde \mu_j * g} \Bigr|
 \le \int_{\R^n} \Bigl(\sum_{j=1}^n |R_j f|^2 \Bigr)^{1/2}
      \Bigl(\sum_{j=1}^n |\widetilde \mu_j * g|^2 \Bigr)^{1/2},
\end{align*}
where $\widetilde \mu_j$ denotes the image of $\mu_j$ under the symmetry
$x \mapsto - x$ of $\R^n$. By~\eqref{TransfosRiesz}, the Riesz transforms
are \og collectively bounded\fge in $L^p(\R^n)$ by a constant $\rho_p$
independent of the dimension~$n$, and we obtain therefore that
\[
     2 \pi \ms1 |\sca {T_{m^{\#}} f} g |
 \le \rho_p \ms1 \|f\|_p \ms2
      \Bigl\|
       \Bigl( \sum_{j=1}^n |\widetilde \mu_j * g|^2 \Bigr)^{1/2}
      \Bigr\|_q.
\]
Noticing that $\widetilde \mu_j * g
 = (\mu_j * \widetilde g) \ms5 \widetilde{} \ms6$ and
$\bigl( \sum_{j=1}^n |\mu_j * g|^2 \bigr)^{1/2}
 = |\nabla K * g|$, we are led to study the operator
\begin{equation}
 U_{K} : g \in L^q(\R^n) \mapsto  \nabla K * g
  \in L^q(\R^n, \R^n)
 \label{UK}
\end{equation}
given by the vector-valued convolution with $\nabla K$. Let us state what
we have got.

\begin{lem}%
\label{LquatreC}
Let $K$ be an integrable kernel on\/ $\R^n$, $m$ its Fourier transform and
let $m^{\#}$ be defined by\/~\eqref{MZero}. For every $p \in (1, 2]$ and
$q = p / (p - 1)$, one has
\[
     \|T_{m^{\#}}\|_{p \rightarrow p}
 \le (2 \pi)^{-1} \rho_p \sup_{\|g\|_q \le 1}
      \| \nabla K * g\|_{L^q(\R^n)}
  =  (2 \pi)^{-1} \rho_p \ms2 \| U_K \|_{q \rightarrow q}.
\]
\end{lem}
\dumou

 When $K = \Kg$ satisfies~\eqref{EstimaGene}, we shall estimate 
$\| U_{\Kg} \|_{q \rightarrow q}$ by interpolation between $L^2$
and~$L^\infty$. Contrary to the $L^2$~estimate which will make use
of~\eqref{EstimaGene}, the $L^\infty$~estimate is a straightforward
observation following from the definition of~$V(K)$. In the special case
$\Kg = K_C$ of a convex body~$C$, this $L^\infty$ case will bring~in the
geometrical parameter $q(C) = 2 \ms1 Q(C_0) L(C_0)$, equal to~$V(K_C)$.

\begin{lem}\label{VerySimple}
Let $K$ be an integrable kernel on\/ $\R^n$ having a finite directional
variation~$V(K)$, and let $U_K$ be defined by\/~\eqref{UK}. One has that
\begin{equation}
     \| U_K \|_{\infty \rightarrow \infty}
 \le V(K).
 \label{Vbound}
\end{equation}
\end{lem}

\begin{proof}
For each $x \in \R^n$, the Euclidean norm of the vector 
$(\nabla K * g)(x) \in \R^n$ is given by the supremum over 
$\theta \in S^{n-1}$ of
\begin{align*}
   \Bigl|
    \theta \ps \Bigl( \int_{\R^n} g(x - y) \, \d (\nabla K)(y) \Bigr)
   \Bigr|
 &= \Bigl| 
     \int_{\R^n} g(x - y) \, \d (\theta \ps \nabla K)(y)
    \Bigr|
 \\
 &\le \|g\|_\infty \ms3 \| \theta \ps \nabla K \|_1
  \le V(K) \ms2 \|g\|_\infty.
\end{align*}
\end{proof}

\begin{lem}\label{q=V}
For every symmetric convex body $C$, isotropic and normalized by variance,
one has that
$
 V(K_C) = q(C)
$.
\end{lem}

\begin{proof}
Let $\theta$ belong to $S^{n-1}$ and let $y$ be in $\theta^\perp$. For each
line $y + \R \ms1 \theta$ that meets the set~$C$, the jumps of the density
$K_C = |C|^{-1} \gr 1_C$ of $\mu_C$, when traveling on the line in the
direction of increasing real numbers, are equal to $|C|^{-1}$ when we 
enter~$C$, and to $-|C|^{-1}$ when leaving~$C$, implying that the mass of
the directional derivative is equal to $2 / |C|$ times the measure of the
projection of~$C$ onto~$\theta^\perp$. More precisely, suppose without loss
of generality that~$\theta$ is the first basis vector $\gr e_1$ of $\R^n$
and let~$\pi_1$ be the orthogonal projection onto~$\gr e_1^\perp$. Let 
$\psi \in \ca S(\R^n)$ be given, and write each $x \in \R^n$ as 
$x = (s, y)$ with $s \in \R$ and $y \in \R^{n-1}$. Using Fubini, we get
\begin{align*}
    - \sca { \gr e_1 \ps \nabla \mu_C} \psi
  &= - \sca {\frac {\partial \mu_C} {\partial x_1}} \psi
  = \sca {\mu_C} {\frac {\partial \psi} {\partial x_1}}
 \\
  &= \int_{\R^{n-1}} \Bigl(
     \int_\R |C|^{-1} \gr 1_C (s, y)
       \frac {\partial \psi} {\partial x_1}(s, y) \, \d s \Bigr)
        \, \d y.
\end{align*}
The inside integral is $0$ if $L_y = y + \R \ms1 \gr e_1$ does not meet the
convex set $C$. Otherwise, the line $L_y$ cuts $C$ along a segment 
$[y + s_1(y) \ms1 \gr e_1, y + s_2(y) \ms1 \gr e_1]$,
$s_1(y) \le s_2(y)$, and
\[
     - \sca { \gr e_1 \ps \nabla \mu_C} \psi
  =  \frac 1 {|C|_n \ns8 } \ms7
       \int_{\pi_1 C} \bigl( \psi(s_2(y), y) - \psi(s_1(y), y) \bigr)
        \, \d y 
 \le \frac {2 \ms1 |\pi_1 C|_{n-1} \ns{20}} {|C|_n} \ms{20} 
      \|\psi\|_\infty.
\]
Going back to a general $\theta \in S^{n-1}$ and according to~\eqref{qC},
we conclude that
\[
     \| \theta \ps \nabla \mu_C \|_1
 \le \frac 2 {\ms1 |C|_n \ns{10}} \ms{10} | P_\theta C |_{n-1}
 \le 2 \ms1 Q(C_0) L(C_0)
  =  q(C).
\]
We get $V(K_C) \le q(C)$, which suffices for our purpose.
M\"uller~\cite[Lemma~3]{MullerQC} shows that this inequality is actually an
equality. 
\end{proof}

 When $\Kg = K_C$, we have 
$\|U_{K_C}\|_{\infty \rightarrow \infty} \le q(C)$, specifying the
estimate~\eqref{Vbound} obtained in the general case. We complete now the
interpolation for $U_{\Kg}$. We formulate the next Lemma so that
we can apply it again in Section~\ref{LeCube}.

\begin{lem}\label{LemmeUK} 
Let $K$ be an isotropic log-concave probability density on\/~$\R^n$ with
variance $\sigma^2$. For\/ $2 \le q \le +\infty$, one has that
\[
     \| U_K f \|_q
  =  \| \nabla K * f \|_q
 \le 2^{1 / q} \ms2 \sigma^{-2 / q} \ms2 V(K)^{ 1 - 2/q} \|f\|_q,
 \quad f \in L^q(\R^n).
\]
If $\Kg$ is an integrable kernel on\/~$\R^n$
satisfying\/~\eqref{EstimaGene}, then\/
\[
     \| U_{\Kg} \|_{q \rightarrow q}
 \le (2 \pi \delta_{0, g})^{2 / q} \ms2 V(\Kg)^{ 1 - 2/q}.
\]
\end{lem}

\begin{proof}
Let $m = \widehat K$ and consider first $q = 2$. By
Parseval~\eqref{EasyBound} we have
\[
   \| \nabla K * f \|_2^2
 = \Bigl\|
    \Bigl(\sum_{j=1}^n |\partial_j K * f|^2 \Bigr)^{1/2} 
   \Bigr\|_2^2
 = \int_{\R^n} 
    \sum_{j=1}^n 4 \ms1 \pi^2 \ms1 \xi_j^2 \ms1
       |m (\xi)|^2 \ms1  |\widehat f (\xi)|^2 \, \d \xi,
\]
and
$
     \sum_{j=1}^n  4 \ms1 \pi^2 \ms1 \xi_j^2 \ms1
       |m (\xi)|^2 
  =  4 \ms1 \pi^2 \ms1 |\xi|^2 \ms1 |m (\xi)|^2
 \le 2 \ms1 \sigma^{-2}
$
by~\eqref{EstimaBour} (or by~\eqref{EstimaGene}, it is 
$\le 4 \pi^2 \delta_{0, g}^2$), hence
$
     \| U_K \|_{2 \rightarrow 2}
 \le \sqrt 2 \ms2 \sigma^{-1}
$ (or $\le 2 \pi \delta_{0, g}$). If $q \in (2, +\infty)$, we write 
$1 / q = (1 - \theta) / 2 + \theta / \infty$ with $\theta = 1 - 2/q$. We
get that
$
     \| U_K \|_{q \rightarrow q}
 \le (\sqrt 2 \ms2 \sigma^{-1})^{2 / q} \ms2 V(K)^{ 1 - 2/q}
$
(or we get $
 \le (2 \pi \delta_{0, g})^{2 / q} \ms2 V(K)^{ 1 - 2/q}
$)
by Lemma~\ref{VerySimple} and interpolation $(L^2, L^\infty)$.
\end{proof}

\begin{proof}[End of the proof of Lemma~\ref{LquatreB}]
We use Lemma~\ref{LquatreC}, then apply Lemma~\ref{LemmeUK} to~$\Kg$ with 
$1 / q = 1 - 1/p$ and obtain that
\[
     \|T_{\mg^{\#}}\|_{p \rightarrow p}
 \le (2 \pi)^{-1} \rho_p \ms1 \|U_{\Kg}\|_{q \rightarrow q}
 \le (2 \pi)^{1 - 2/p} \ms1 \rho_p \ms2 \delta_{0, g}^{2 - 2 / p} \ms1 
      V(\Kg)^{ - 1 + 2/p}.
\]
\end{proof}

\subsubsection{Conclusion%
\label{ConclusionMull}}

\noindent
We finish the proof of Proposition~\ref{PropoMull}. We first run over half
of the way in the following lemma, which we shall refer to again in
Section~\ref{LeCube}.

\begin{lem}%
\label{Prop2.0}
Let $\Kg$ be an integrable kernel on\/ $\R^n$
satisfying\/~\eqref{EstimaGeneN}, and let $\mg^{\#}$ be defined
by\/~\eqref{MZero}. Let $\alpha \in (0, 1)$ and suppose that\/ 
$1 < p_0 < p < 2$. There exists a constant $\kappa(p, p_0)$, independent of
$n$, such that
\[
     \| (\xi \ps \nabla)^\alpha \mg (\xi) \|_{p \rightarrow p}
 \le \|\Kg\|_{L^1(\R^n)}
      \ns3+\ns2 \frac {\kappa(p, p_0) } { 1 - \alpha } \ms1 
         \bigl( \|\Kg\|_{L^1(\R^n)}
           \ns2+\ns1 \| \mg^{\#} \|_{p_0 \rightarrow p_0} 
         \bigr)^{1 - \theta} 
          \Delta_{k(\theta)}^\theta,
\]
where $\theta \in (0, 1)$ is defined by\/
$
   1 / p
 = (1 - \theta) / p_0 + \theta / 2
$ and $k(\theta) = \lceil 1 / \theta \rceil$.
\end{lem}

\begin{proof}
Lemma~\ref{CPareil} gives
\[
     \| (\xi \ps \nabla)^\alpha \mg (\xi) \|_{p \rightarrow p}
 \le \|\Kg\|_{L^1(\R^n)} + 
      \| d_t^\alpha \mg (t \xi) \barre_{t = 1} \|_{p \rightarrow p}.
\]
Let $\varepsilon = 1 - \alpha > 0$. We apply complex interpolation to the
M\"uller family $(m^\varepsilon_z)$ in the strip
$S = \{ z \in \C : -\varepsilon \le z \le \nu \}$ of width
$w := \nu + \varepsilon$. We bound
$m^\varepsilon_\alpha(\xi) = d^\alpha_t \mg(t \xi) \barre_{t = 1}$ on
$L^p(\R^n)$, using $L^{p_0}$~estimates of~$m^{\varepsilon}_z$ for 
$\Re z = - \varepsilon$ and $L^2$ estimates when $\Re z = \nu$. The
value~$\nu$ must satisfy
$
 \alpha = (1 - \theta)(-\varepsilon) + \theta \ms1 \nu
$,
hence $\nu = 1 / \theta - \varepsilon > 1 - \varepsilon$ and 
$w = 1 / \theta$. It follows from Lemma~\ref{LquatreA} that
\[
     \|m^\varepsilon_{-\varepsilon + \ii \tau} \|_{p_0 \rightarrow p_0}
 \le 2 \ms2 \varepsilon^{-1} \ms1 \lambda_{p_0} \ms2
      \e^{\pi |\tau| } \ms2 
       \bigl( \|\Kg\|_{L^1(\R^n)}
              + \| \mg^{\#} \|_{p_0 \rightarrow p_0} 
       \bigr).
\]
By Lemma~\ref{MajoMullA}, each operator $m^\varepsilon_{\nu + \ii \tau}$ is
bounded by $\kappa_\nu \ms1 \varepsilon^{-1} \Delta_\ell \ms1
 (1 + \tau^2)^{\nu/2 - 1/4} \e^{\pi |\tau| / 2}$ on $L^2(\R^n)$, with
$\kappa_\nu \le \kappa_w
 =  4 \ms1 \Gamma(\max(w, 2)) (3/2)^{w - 1/2} \e^{\pi / 4}
 =: c_\theta$, a function of $\theta$ alone, and
$\ell = \lceil \nu + \varepsilon \rceil = \lceil w \rceil$. If we check the
admissible growth condition in~$S$, we can rely on
Corollary~\ref{InStrip}, Remark~\ref{PolyFacto} and~\eqref{PolyBd} in order
to get a bound
\[
     \| m^\varepsilon_\alpha(\xi) \|_{p \rightarrow p}
 \le \kappa(p, p_0) \ms1 \varepsilon^{-1}
       (\|\Kg\|_{L^1(\R^n)}
         +  \| \mg^{\#} \|_{p_0 \rightarrow p_0} )^{1 - \theta} 
        \Delta_{k(\theta)}^\theta,
\]
with $k(\theta) = \ell = \lceil 1 / \theta \rceil$ and
$\kappa(p, p_0)
 \le (1 + w / 2)^{\theta(\nu - 1/2)} \e^{\pi w / 2} \ms1
     (2 \lambda_{p_0})^{1 - \theta} c_\theta^\theta$. Observing
that $\theta \nu < 1$, $w > 1$ and $\lambda_{p_0} \ge 1$, we may simplify 
this bound as
\begin{equation}
     \kappa(p, p_0)
 \le \kappa \ms2 w \ms1 \e^{\pi w / 2} \ms1
      \lambda_{p_0} \ms1 \Gamma( \max(w, 2) )^{1/w} 
 \le \kappa \ms2 w^2 \ms1 \e^{\pi w / 2} \ms1 \lambda_{p_0}.
 \label{KappaBound}
\end{equation} 
We now verify that the holomorphic function 
$H(z) = \sca {m^{\varepsilon}_z f_z} {g_z}$ has an admissible growth
in~$S$. Since the kernels are bounded functions of $\xi$ by
Lemma~\ref{MajoMullA}, all multipliers $m^\varepsilon_z$, $z \in S$, are
$L^2$-bounded with a bound of the form $\kappa \e^{\kappa |\Im z|}$. If we
restrict to functions $f$ and~$g$ bounded with bounded support, we have
by~\eqref{FzGzBound} that $f_z$, $g_z$ are uniformly bounded
in~$L^2(\R^n)$, and we can conclude as in Section~\ref{InterpoCarbe}.
\end{proof}

\begin{proof}[End of the proof of Proposition~\ref{PropoMull}]
Given $p \in (1, 2)$ and $\alpha = 1 - \varepsilon \in (0, 1)$, we select 
$p_0 \in (1, p)$ and let $\theta \in (0, 1)$ satisfy
$
   1 / p
 = (1 - \theta) / p_0 + \theta / 2
$.
Since $1 < p_0 < p < 2$, we have that $0 < \theta < 2(1 - 1/p) < 1$.
It follows from Lemma~\ref{LquatreB} that
\[
     \| \mg^{\#} \|_{p_0 \rightarrow p_0}
 \le (2 \pi)^{1 - 2 / p_0} 
      \rho_{p_0} \ms2 \delta_{0, g}^{2 - 2 / p_0} \ms1
        V(\Kg)^{ 2/p_0 - 1}.
\]
By Lemma~\ref{Prop2.0}, and because $\|\Kg\|_{L^1(\R^n)} \le 1$,
$\rho_{p_0} \ge 1$ (see Remark~\ref{RemIwaMar}), we get
\[
     \| (\xi \ps \nabla)^\alpha \mg (\xi) \|_{p \rightarrow p}
 \le 1 + \kappa(p, p_0) \ms1 \varepsilon^{-1} \ms1 \rho_{p_0} \ms1
       \Delta_{k(\theta)}^\theta \ms1    
        \bigl( 
         1 + \delta_{0, g}^{1 - \theta - (2/p - 1)} \ms1 V(\Kg)^{2 / p - 1}
        \bigr).
\]

 We still have a choice of $\theta \in (0, 2 - 2/p)$. If $\theta$ gets
small, then the power of $\Delta_{k(\theta)}$ gets small, but the number
$k(\theta)$ of constants $\delta_{j, g}$ involved increases to infinity. In
the log-concave case, the estimate~\eqref{deltas} indicates a growth of
order $\Delta_{k, c} \sim k!$ yielding 
$\Delta_{k(\theta)}^\theta \sim 1 / \theta$. Furthermore, the width 
$w = 1 / \theta$ of the strip and the associated interpolation constants
also tend to $+\infty$ in this case, and we should thus keep $\theta$
away from $0$, as much as possible. If $\theta$ approaches its upper limit
$2 (1 - 1/p)$, then $p_0$ tends to~$1$ and the constants such as
$\lambda_{p_0}$, $\rho_{p_0}$ tend to infinity. Choosing 
$\theta = (4 / 3) (1 - 1 / p)$ has the merit to provide the relatively
simple bound
\begin{equation}
     \| (\xi \ps \nabla)^\alpha \mg (\xi) \|_{p \rightarrow p}
 \le 1 + \kappa(\alpha, p) \Delta_{k(p)}^{ (4 / 3)(1 - 1 / p) } \ms1      
       \bigl(
        1 + \delta_{0, g}^{ (2 /3)(1 - 1 / p) } \ms1 V(\Kg)^{2 / p - 1}
       \bigr),
 \label{KAlphaP}
\end{equation}
with 
$
     \kappa(\alpha, p)
  =  \kappa(p, p_0) \ms1 \varepsilon^{-1} \rho_{p_0}
 \le \kappa \ms2 \varepsilon^{-1} w^2 \ms1 \e^{\pi w / 2} \ms1  
      \lambda_{p_0} \ms1 \rho_{p_0}
$ by~\eqref{KappaBound},
with $p_0 - 1 = (p - 1) / (5 - 2 p)$ and
$k(p) = \lceil 1 / \theta \rceil = \lceil 3 \ms1 p / (4 p - 4) \rceil$. 
\end{proof}

\begin{remN}\label{ButPoly}
It is usual to have a factor $1 / \Gamma(z)$ in fractional derivatives,
which led us to seeing $\e^{\pi |\tau| / 2}$ in many places, ending with
$\e^{\pi w/2}$ in our estimate~\eqref{KappaBound} of $\kappa(p, p_0)$, with
$w = 1 / \theta \sim q := p / (p-1)$ after the final choice of
$\theta = 4 / (3 q)$ above. We could avoid this exponential though.
Consider the modified M\"uller family
\[
   \widetilde m^\varepsilon_z (\xi)
 = \Gamma(1 + \varepsilon + z) \ms2
    \frac { \Gamma(2 \ms1 \varepsilon + z) }
          { \Gamma(1 + \varepsilon) }
     \ms2 m^\varepsilon_z (\xi),
 \quad
 \Re z \ge - \varepsilon, \ms8
 \xi \in \R^n,
\]
which coincides with the former at $z = \alpha$ since
$\varepsilon + \alpha = 1$ and $\Gamma(2) = 1$. For the $L^{p_0}$~bound
when $z = -\varepsilon + \ii \tau$, we decompose
$\widetilde m_{- \varepsilon + \ii \tau}^\varepsilon (\xi)$ as
\[
   \frac 1 { \Gamma(1 + \varepsilon) } \ms3
   \bigl[ \Gamma(1 + \ii \tau) (1 + |\xi|)^{ - \ii \tau} \bigr]
   \Bigl[
     \frac { \Gamma(\varepsilon + \ii \tau) }
           { \Gamma(\varepsilon - \ii \tau) } \ms2
   (1 + |\xi|) \ms2
     \int_1^2 (s - 1)^{\varepsilon - \ii \tau - 1} \mg(s \xi) \, \d s
   \Bigr].
\]
Introducing $\Gamma(1 + \ii \tau)$ in the Laplace type 
multiplier~\eqref{LaplacEstim}, we obtain a new function
$\widetilde a_\tau(t)$ bounded by~$1$, and $\Gamma(2 \ms1 \varepsilon + z)$
is used for the bound of 
$d^{-\varepsilon + \ii \tau}_t \mg(t \xi) \barre_{t = 1}$ because
$|\Gamma(\varepsilon + \ii \tau)| = |\Gamma(\varepsilon - \ii \tau)|$.
We get in this way for $\widetilde m^\varepsilon_z (\xi)$ a bound
\[
     \| \widetilde m^\varepsilon_z (\xi) \|_{p_0 \rightarrow p_0}
 \le 2 \ms{1.5} \varepsilon^{-1} \lambda_{p_0} \ms1 
      (\|\Kg\|_{L^1(\R^n)} + \| m^{\#} \|_{p_0 \rightarrow p_0})
\]
that replaces Lemma~\ref{LquatreA} (we use again
$1 / \Gamma(1 + \varepsilon) \le 2$). The $L^\infty$~bounds obtained
in~\eqref{MajoK} when $z = k - \varepsilon + \ii \tau$, $k > 0$, have now a
largest factor of $\Delta_k \ms1 \varepsilon^{-1}$ equal to
$
     \Gamma(k + 1 + \ii \tau) \ms2 \Gamma(k + \varepsilon + \ii \tau)
      / [ \Gamma(1 + \varepsilon) \Gamma(-k + 1 + \varepsilon - \ii \tau) ]
$ (when the index $j_1$ in~\eqref{MajoK} is equal to $k-1$). The modulus
of this factor is the same as that of
\[
   \frac { \Gamma(k + 1 + \ii \tau) \ms2 \Gamma(k + \varepsilon + \ii \tau) }
         { \Gamma(1 + \varepsilon) \Gamma(-k + 1 + \varepsilon + \ii \tau) }
 = \frac { \Gamma(1 + \ii \tau) } { \Gamma(1 + \varepsilon) } \ms1
    \Bigl( \prod_{j=1}^k (j + \ii \tau) \Bigr)
     \Bigl( \prod_{j=-k+1}^{k-1} (j + \varepsilon + \ii \tau) \Bigr).
\]
This is a bounded function of $\tau$ according to~\eqref{RefFor}, with a
rough bound given by
$2 \sqrt {2 \pi} \ms2  (k + |\tau|)^{3k} \e^{- \pi |\tau| / 2}
 \le 6 \ms1.\ms1 2^k k^{3k}$ (use $x / \sinh x \le (1 + 2 x) \e^{-x}$ for
$x \ge 0$). One need not be too careful here since this
term will be raised to the power $\theta = 1 / w \simleq 1 / k$. We use it
as in~\eqref{HdeXideZ} for two values $k$, $k+1$ such that 
$k \le \nu_0 + \varepsilon \le k + 1 \le \ell
 = \lceil w \rceil < w + 1$. One has then for
the $L^2$ bound of $\widetilde m_{\nu + \ii \tau}^\varepsilon (\xi)$ an
estimate by $\kappa \ms1 2^w w^{3(w+1)} \Delta_\ell \ms2 \varepsilon^{-1}$.
By interpolation we have
\[
     \| m^\varepsilon_\alpha(\xi) \|_{p \rightarrow p}
 \le \kappa \ms1 \varepsilon^{-1}
     \lambda_{p_0}^{1 - \theta} (2^w w^{3(w+1)})^\theta
       (\|\Kg\|_{L^1(\R^n)}
         +  \| \mg^{\#} \|_{p_0 \rightarrow p_0} )^{1 - \theta} 
        \Delta_\ell^\theta.
\]
We thus get for $\kappa(p, p_0)$ in~\eqref{KappaBound} a new estimate
$\kappa'(p, p_0) \le \kappa w^3 \lambda_{p_0}$, leading in~\eqref{KAlphaP}
to $\kappa'(\alpha, p)
 \le \kappa \ms1 \varepsilon^{-1} q^3 \lambda_{p_0} \rho_{p_0}$.
The final choice in the proof of Proposition~\ref{PropoMull} gives
$p_0 - 1 = (p - 1) / (5 - 2p)$ of order $p - 1$ as $p \to 1$, and since
$\lambda_{p_0}$, $\rho_{p_0}$ are $O \bigl( (p_0 - 1)^{-1} \bigr)$ as 
$p_0 \to 1$ (see~\eqref{ConstaLambda} and~\eqref{EstiRho}), we end up with
$\kappa'(\alpha, p)
 \le \kappa \ms1 \varepsilon^{-1} q^5$, a bound which is polynomial but has
no reason to be accurate. After these modifications, we have for 
Proposition~\ref{PropoMull} when $1 < p \le 2$ a new form
\begin{equation}
     \| (\xi \ps \nabla)^\alpha \mg (\xi) \|_{p \rightarrow p}
 \le 1 + \kappa \ms1 \varepsilon^{-1} q^5 \ms1 \Delta_{k(p)}^{ (4/3)(1 - 1/p )} 
      \bigl(
       1 + \delta_{0, g}^{ (2/3)(1 - 1/p) } V(\Kg)^{2 / p - 1}
      \bigr),
 \label{NewForm}
\end{equation}
with $k(p) = \lceil 3 \ms1 p / (4 p - 4) \rceil$ and $q = p / (p-1)$.
\end{remN}
\dumou 
 
\begin{remN}\label{ButPolyB}
With the new information above, we go back to the proof of
Theorem~\ref{TheoMull}. One has chosen there
$\varepsilon = 1 - \alpha = 3 (p - 1) / (4 p + 4)$, and $p_0 \in (1, p)$
such that $p_0 - 1 = 3 (p - 1) / (5 - p)$. Both $\varepsilon$ and $p_0 - 1$
behave as multiples of $p - 1$ when $p \to 1$. If we consider the Poisson
kernel $P$ as another $\Kg$ satisfying~\eqref{EstimaGeneN}, and with 
$V(P) \le 2 / \pi$ by~\eqref{VPoiss}, we can apply to it~\eqref{NewForm}
for the value $p_0$ and obtain that
$\| (\xi \ps \nabla)^\alpha \widehat P (\xi) \|_{p_0 \rightarrow p_0}
 \le \kappa \ms1 \varepsilon^{-1} q^5$. Applying also~\eqref{NewForm} to
$\mg$ and $p_0$, we get for $m = \mg - \widehat P$ that
\[
     \| (\xi \ps \nabla)^\alpha m (\xi) \|_{p_0 \rightarrow p_0}
 \le \kappa \ms2 q^6 \ms1 \Delta_{k(p_0)}^{ (4/3)(1 - 1/p_0 )} 
       \delta_{0, g}^{ (2/3)(1 - 1/p_0) } (1 + \lambda^{2 / p_0 - 1})
  =: B.
\]
We have $\alpha - 1 / p_0 = 3 (p-1) / (4p+4)$ which again is of order 
$p - 1$. The constant $\kappa_{\alpha, p_0}$ from~\eqref{KalphaP}, seen
in~\eqref{MajoTj}, behaves thus as $(p - 1)^{-1/p_0} \simeq q$, so
$C''_{p_0}(\lambda)$ is bounded by $\kappa \ms1 q \ms1 (2 + B)$. Also 
$r_0 - 1 \simeq (p + p_0 -2) / 2$ in~\eqref{ProofTheoMull} is of order 
$p - 1 \simeq 1 / q$. In~\eqref{ProofTheoMull}, the constants $C_{r_0}$ and
$C'_p$ are of order~$q$. Indeed, we can take 
$C_{r_0} = \textrm{q}_{\ms1 r_0}$ from~\eqref{LiPaIneq}, that was estimated
by $r_0 / (r_0 - 1)$ in~\eqref{EstiQp}, and $C'_p$ can be the bound for the
maximal function of the Poisson kernel, see~\eqref{MaxiPoiss}. Also, 
$1 - \gamma = (p-1) / (2 p)$, and with Lemma~\ref{CloseGap} we know that
\[
     \sum_{k \in \Z} a_{\alpha, k}^{ (1 - \gamma) p / 2}
 \le \kappa \ms1
      \sum_{k \in \Z} 2^{ - (1 - \alpha)(1 - \gamma) p |k| / 2}
 \le \frac { \kappa' }
           { (1 - \alpha)(1 - \gamma) }
 \le \kappa'' \ms1 q^2.
\]
Finally, we obtain for Theorem~\ref{TheoMull} another bizarre polynomial
estimate
\begin{align*}
     \| \M_{\Kg} \|_{p \rightarrow p} 
 \le \| \Mg_K\|_{p \rightarrow p} + O(q)
 \le C_{r_0}^2 C''_{p_0}(\lambda) (\kappa'' q^2)^2 + O(q)
 \\
 \le \kappa \ms2 q^{13} \ms1 \Delta_{\lceil q \rceil}^{ 1 - 1/p }
      \Delta_1^{ 1 - 1/p }(1 + \lambda^{2 / p - 1}),
 \quad
 1 < p \le 2.
\end{align*}

\end{remN}

\section{Bourgain's article on cubes%
\label{LeCube}}

\noindent
In this section, $Q$ is a cube in dimension $n$, more precisely, the
symmetric cube
\[
 \label{QuN}
 Q = Q_n = \biggl[ - \frac 1 2 \ms1 \up, \ms1 \frac 1 2 \biggr]^n
\]  
of volume~$1$ in~$\R^n$. It is isotropic, but if we look for a multiple 
$b \ms1 Q$ normalized by variance, we would need that the half-side 
$a = b/2$ of $b \ms1 Q$ satisfy $\sigma^2_{b \ms1 Q} = 1$, where
\[
   \sigma_{b \ms1 Q}^2
 = \frac 1 {|b \ms1 Q|} \int_{b \ms1 Q} x_1^2 \, \d x
 = \frac 1 {2 a} \int_{-a}^a s^2 \, \d s
 = \frac 1 a \int_0^a s^2 \, \d s
 = \frac {a^2 \ns8} {\ms1 3} \ms2 \up,
\]
and where $x = (x_1, \ldots, x_n) \in \R^n$. This gives $a = \sqrt 3$ in
every dimension~$n$, but the cube~$[-\sqrt 3, \sqrt 3]^n$ is not very
pleasant to manipulate, and we shall rather follow
Bourgain~\cite{BourgainCube} and keep the volume~1 cube~$Q$. With 
$a = 1/2$, the covariance for~$Q$ is given by $(12)^{-1} \I_n$. Since the
variance $\sigma_Q^2 = 1 / 12$ is independent of the dimension, we shall
have no problem with the estimates~\eqref{EstimaBour}
or~\eqref{NoticeFirst}. The Fourier transform of the probability measure
$\mu_Q$ is given by
\[
   m_Q(\xi)
 = \widehat{ \vmu \mu_Q }(\xi)
 = \widehat{K_Q}(\xi)
 = \prod_{j=1}^n \frac {\sin (\pi \xi_j)} {\pi \xi_j} \ms1 \up,
 \qquad
 \xi = (\xi_1, \ldots, \xi_n) \in \R^n.
\]
Bourgain observes that a decay better than the usual~\eqref{EstimaBour} for
a Fourier transform $m_C$ would allow to relax the limitation $p > 3/2$
of Theorem~\ref{TheoMaxi}, and that this better decay is achieved by $m_Q$
is most directions. He says that his proof proceeds therefore to diverse
localizations in Fourier space.

\begin{thm}[Bourgain~\cite{BourgainCube}]\label{TheoCube}
For every $p$ in\/ $(1, +\infty]$, there exists a constant $\kappa_p$ such
that\/ $\| \M_{Q_n} \|_{p \rightarrow p} \le \kappa_p$ for every integer
$n \ge 1$.
\end{thm}

 We shall approach the maximal function problem for the cube by summing
expressions such as $K^R - K^{2R}$, with
\[
 \label{KaR}
 K^R = K_Q * \Di G {1/R},
\]
where $G$ is a Gaussian probability kernel, $\Di G {1/R}$ its
dilate~\eqref{Dilata}, and where $R$ takes the values 
$1, 2, \ldots, 2^j, \ldots\ms4$ with $j$ being any integer~$\ge 0$. This is
a Littlewood--Paley-type decomposition, similar to what we have seen
before. By Pr\'ekopa--Leindler, $K^R$ is a log-concave probability density.
We shall set
$
 m^R = \widehat{K^R}
$ in what follows.
\dumou

 We will call the Carbery--M\"uller artillery and obtain when $1 < r < 2$,
for every $\delta > 0$ and $R = 2^j$ with $j \ge 0$, bounds of the form
\[
     \bigl\| \sup_{1 \le t \le 2} | \DiE K t R * f| \ms1 \bigr\|_r
 \le \kappa_{\delta, r} \ms1 R^\delta \ms1 \|f\|_r,
 \ms{20} \hbox{where} \ms{9}
 {\DiE K t R} := {\Di {(K^R)} t}.
\]
Why this may be a decisive step will be explained below. According to
Carbery's Proposition~\ref{FouCarbe}~$(ii)$, this bound will be consequence
of the $L^r(\R^n)$-boundedness of the multiplier
$
 (\xi \ps \nabla)^\alpha m^R(\xi)
$
for a value of $\alpha \in (1/r, 1)$. Next, following M\"uller, it will be
enough to estimate in~$L^s(\R^n)$, with $1 < s < r$, the \og crucial\fge
multiplier
$
 |\xi| m^R(\xi)
$.
This is what Bourgain does along several pages, in a series of reductions
bringing in many tools that are specific to the product structure of the
cube.

\subsection{Holding on M\"uller and Carbery%
\label{Raccord}}

\noindent
Let us specify the preceding rough outline. The final objective is to bound
in $L^p(\R^n)$ the maximal operator $\M_Q$ for~$p$ below the limit $3/2$
that is known so far, proving that
\[
     \Bigl\| \ms1 \sup_{t > 0} |\Di {(K_Q)} t  * f| \ms1 \Bigr\|_p
 \le \kappa_p \ms2 \|f\|_p,
 \ms{16}
 1 < p < 2, \ms8 f \in L^p(\R^n).
\]
We fix a value $p \in (1, 2)$ in all that follows. In order to obtain the
property~\eqref{A2} from p.~\pageref{A2}, needed for applying Carbery's
Proposition~\ref{PropoPrio}, we must show that
\begin{equation}
     \Bigl\| \ms1 \sup_{1 \le t \le 2} |{\Di K t} * f| \ms1 \Bigr\|_{p_2}
 \le \kappa \ms1 \|f\|_{p_2},
 \ms{13}
 1 < p_2 < p < 2,
 \label{grA} \tag{$\gr A$}
\end{equation}
where $K = K_Q - P$ and $P$ is the Poisson kernel~\eqref{PoissonDensi}.
This is the only missing fact for lowering the limitation $p > 3/2$ down to
$p > 1$, as explained in the proof of M\"uller's Theorem~\ref{TheoMull}.
For the Poisson side it is fine, it remains to work on~$K_Q$. We introduce
the Gaussian kernel $G = \Di {(\gamma_n)} { \sqrt 2 / \sqrt \pi }$
on~$\R^n$. The variance of $G$ is equal to $2 / \pi$, thus independent of
$n$, and
$
 \widehat G(\xi) = \e^{- 4 \pi |\xi|^2}\label{GKernel}
$
for every $\xi \in \R^n$. With this normalization for $G$, we have
by~\eqref{VKt} that
\begin{equation}
 V(G) = \sqrt { \pi / 2 } \ms4 V(\gamma_n) = 1.
 \label{VdeG}
\end{equation}
We decompose the Dirac probability measure $\delta_0$ at the origin, in the
sense of distributions, by means of the simple telescopic series
\[
   \delta_0
 = \Di G 1 + (\Di G {1/2} - \Di G 1)
    + \cdots + (\Di G {2^{-k-1}} - \Di G {2^{-k}}) + \cdots
\]
and we decompose $K_Q$ accordingly, using the approximations
$K^R = K_Q * \Di G {1/R}$, for $R = 2^j \ge 1$ and $j$ nonnegative
integer, under the form
\[
 K_Q = K^1 + (K^2 - K^1) + \cdots + (K^{2^{j+1}} - K^{2^j}) + \cdots.
\]
By Pr\'ekopa--Leindler, each $K^R$ is a log-concave symmetric probability
density on~$\R^n$. It is isotropic, with a variance $\sigma_R^2$ satisfying
\begin{equation}
 12^{-1} < \sigma_R^2 = 12^{-1} + 2 \ms1 \pi^{-1} R^{-2} < 1,
 \quad R \ge 1.
 \label{SigmaR}
\end{equation}
We set
$ \d \mu^R(x) = K^R(x) \, \d x$,
$m^R = \widehat{K^R} = \widehat{\mu^R}$.\label{MuR}
It follows from~\eqref{NoticeFirst} that $m^R$
satisfies~\eqref{EstimaGeneN} with constants independent of $n$. We get
\begin{equation}
     \Bigl| \frac {\d^j} {\d t^j \ns2} \ms4 m^R (t \theta) \Bigr|
 \le \frac { \delta_{j, c} }
           { 1 + \pi \ms1 |t| / \sqrt 3 }
 \le \frac { \delta_{j, c} }
           { 1 + |t| } \ms1 \up,
 \quad \theta \in S^{n-1}, \ms9 t \in \R, \ms9 j \ge 0.
 \label{mRuniform}
\end{equation}
We shall obtain the desired estimate~\eqref{grA} for $p_2$ by interpolating
between $p_1$ and~$2$, where $1 < p_1 < p_2 < p < 2$. As we have
said previously, we will show that for every $\delta > 0$, we have for all
$R= 2^j \ge 1$ that
\begin{equation}
     \bigl\| \sup_{1 \le t \le 2} |{\DiE K t R} * f| \ms1 
     \bigr\|_{p_1} 
 \le \kappa_\delta \ms1 R^\delta \ms2 \|f\|_{p_1},
 \quad
 f \in L^{p_1}(\R^n),
 \label{grB} \tag{$\gr B$}
\end{equation}
and on the other hand, we prove that for every $f \in L^2(\R^n)$ we have
\begin{equation}
     \bigl\| \sup_{t > 0} | \Di{(K^R - K^{2R})}t * f| \ms1 \bigr\|_{2} 
 \le \kappa \ms2 R^{-1/2} \ms2 \|f\|_2,
 \label{grC} \tag{$\gr C$}
 \ms{16}
     \bigl\| \sup_{t > 0} |{\DiE K t 1} * f| \ms1 \bigr\|_{2} 
 \le \kappa \ms2 \|f\|_2.
\end{equation}
The second inequality in~\eqref{grC} is the log-concave version of
Bourgain's $L^2$~theorem, Theorem~\ref{LogConcBourg}. One can obtain
the first part of~\eqref{grC} by the $\Gamma_B(K)$ criterion,
Lemma~\ref{Clef}. We have indeed, uniformly in
$\theta \in S^{n-1}$ and in the dimension $n$ (observe that
$\widehat G$ has a radial expression independent of $n$), that
\begin{align*}
     |\widehat G(u \theta) - \widehat G(2 u \theta)|
 & \simleq u^2 \wedge \e^{ - 4 \pi u^2}
 \le |u| \wedge |u|^{-1}
 \ms{16} \hbox{and}
 \\
     |\theta \ps \nabla \widehat G(u \theta)
       - \theta \ps \nabla \widehat G(2 u \theta)|
 & \simleq |u| \ms1 (1 \wedge \e^{ - 4 \pi u^2})
 \le 1 \wedge |u|^{-1},
 \qquad u \in \R.
\end{align*}
We apply Lemma~\ref{UsingGammaB} with $K_1 = K_Q$, $K_2 = G$, replacing
$2^{|k|}$ with $R$ and obtaining
\[
    \sum_{j \in \Z} 
     \Bigl(
      \alpha_j (K^R) + \sqrt{\alpha_j (K^R) \beta_j (K^R)} \ms3
     \Bigr)
 \simleq R^{-1 / 2}.
\]
\dumou

 If $\delta > 0$ is sufficiently small we deduce by interpolation
between~\eqref{grB} and~\eqref{grC} that there exists $\delta_1 > 0$ such
that
\[
     \bigl\| \sup_{1 \le t \le 2} |( {\DiE K t R} - {\DiE K t {2R}}) * f| \ms1 
     \bigr\|_{p_2} 
 \le \kappa_\delta \ms1 R^{ - \delta_1} \ms2 \|f\|_{p_2}
\]
and we get Property~\eqref{grA} by summing on the values $R = 2^j$ for all
integers $j \ge 0$. We fix thus a value 
$\delta_* = \delta_*(p, p_2, p_1) > 0$ of $\delta$, sufficiently small for
implying that $\delta_1 > 0$ whenever $0 < \delta \le \delta_*$. Precisely,
if $\lambda \in (0, 1)$ is such that
\[
 \frac 1 {p_2} = \frac {1 - \lambda} {p_1} + \frac {\lambda} 2 \ms1 \up,
\]
we need to choose $\delta_* > 0$ so that
$
 - \delta_1 = (1 - \lambda) \delta_* - \lambda / 2 < 0
$, \textit{i.e.}, we select a value $\delta_* = \delta_*(p, p_2, p_1)$ such
that $0 < \delta_*(p, p_2, p_1) < (p_2 - p_1) / (2 p_1 - p_2 p_1)$.
\dumou

 For obtaining~\eqref{grB} we shall use the conclusion~$(ii)$ of Carbery's
Proposition~\ref{FouCarbe} and also apply M\"uller's analysis. 
We need to show that for some 
$\alpha \in (1 / p_1, 1)$ and $0 < \delta \le \delta_*$, we have
\begin{equation}
     2 \ms1 \|m^R\|_{p_1 \rightarrow p_1}
      + \| (\xi \ps \nabla)^\alpha m^R(\xi)\|_{p_1 \rightarrow p_1}
 \le \kappa \ms1 R^\delta
 \label{WhatWeWant}
\end{equation}
for all $R = 2^j$, $j \in \N$. There is no problem for $m^R$, which
corresponds to convolution with a probability density, and for the other
term we shall apply Lemma~\ref{Prop2.0} with $1 < p_0 < p_1 < 2$. For
technical reasons, the value $p_0$, close to $1$, is chosen in a way that
its conjugate $q_0$ is an integer of the form $2^\nu$, with $\nu$
integer~$> 0$. If we can prove that for a fixed $\delta > 0$ and for every
$R = 2^j$, we have
\begin{equation}
     \bigl\| \ms1 |\xi| \ms2 m^R(\xi) \bigr\|_{p_0 \rightarrow p_0}
 \le \kappa_\delta \ms1 R^\delta,
 \label{Wanted}
\end{equation}
it follows from Lemma~\ref{Prop2.0} that
$
     \| (\xi \ps \nabla)^\alpha m^R(\xi)\|_{p_1 \rightarrow p_1}
 \le \kappa'_\delta \ms1 (1 + R^{\delta \beta})
 \le \kappa''_\delta \ms1 R^{\delta}
$
for some $\beta \in (0, 1)$, uniformly in the dimension $n$ according
to~\eqref{mRuniform}. The conclusion~\eqref{WhatWeWant} is then obtained. 
\dumou

 By exploiting the inequality~\eqref{TransfosRiesz} on Riesz transforms,
M\"uller's plan went on with a reduction to estimating the expression
$
    \bigl\| \nabla \mu^R * g \bigr\|_{q_0}
$ when $g \in L^{q_0}(\R^n)$ and $1 / q_0 + 1 / p_0 = 1$. We must show that
for every $R = 2^j$ we have
\[
     \bigl\| \nabla \mu^R * g \bigr\|_{q_0}
 \le \kappa_{p_0, \delta} R^\delta \|g\|_{q_0},
\]
yielding~\eqref{Wanted} by Lemma~\ref{LquatreC}. We use~\eqref{VdeG}
and~\eqref{VKt}, which give
\begin{equation}
     V(K^R)
  =  V(\mu_Q * \Di G {1/R})
 \le V(\Di G {1/R})
  =  R.
 \label{EasyMuller}
\end{equation}
By Lemma~\ref{VerySimple}, this bound for the mass of 
$\theta \ps \nabla \mu^R$ when $\theta \in S^{n-1}$ implies that
$
     \| \nabla \mu^R * g \|_{L^\infty(\R^n)}
 \le R \ms2 \| g \|_{L^\infty(\R^n)}
$.
Then, by interpolation with the $L^2$ case given by~\eqref{grC}, we can find
when $2 < q < +\infty$ a bound in~$L^q(\R^n)$ of the form
\[
     \| \nabla (\mu^R - \mu^{2R}) * g \|_{L^q(\R^n)}
 \le \kappa \ms1 (R^{ - 1 / 2 })^{2 / q} R^{1 - 2 / q} \ms1 
      \| g \|_{L^q (\R^n)}
  =  \kappa \ms1 R^{1 - 3 / q} \ms1 \| g \|_{L^q (\R^n)}.
\]
This interpolation $(L^\infty, L^2)$ does not give the desired bound
$R^\delta$ in $L^{q_0}(\R^n)$, with~$\delta$ small, when $q_0 \ge 3$.
However, it does give the right ingredient for the
Bourgain--Carbery Theorem~\ref{TheoMaxi} when $3/2 < p \le 2$, since 
$1 - 3 / q < 0$ in this case.
\dumou

 For going farther than M\"uller, one has to prove inequalities that allow
one to work in $L^r(\R^n)$, $2 < r < +\infty$, instead of~$L^\infty(\R^n)$.
This is done with the help of certain analytic semi-groups
(Section~\ref{FirstReduc}), as well as a {\it ad hoc\/} method a la
Bourgain, which he says inspired from martingale techniques
(Section~\ref{SecondReduc}). Theorem~\ref{TheoCube} will be obtained once
we have the following proposition, which we can apply with a value 
$\delta \le \delta_*(p, p_2, p_1)$. We then conclude by the preceding
discussion. 
\dumou

\begin{prp}\label{MainObj}
For every $\delta > 0$ and $q_0 = 2^\nu$, with $\nu$ an integer $\ge 1$,
there exists a constant $\kappa(q_0, \delta)$ such that for every $n \ge 1$
and $R = 2^k$, $k = 0, 1, \ldots,$ one has
\[
     \| \nabla \mu^R * g \|_{L^{q_0}(\R^n)}
 \le \kappa(q_0, \delta) \ms1 R^\delta \ms1 \| g \|_{L^{q_0}(\R^n)},
 \quad 
 g \in L^{q_0}(\R^n).
\]
\end{prp}
\dumou

 We shall keep $\delta > 0$, $p_0 = q_0 / (q_0 - 1)$ and $R = 2^{k_0}$
fixed in the rest of Section~\ref{LeCube}.
\dumou

\subsubsection{\textit{A priori} estimate%
 \label{APriori}}
 
\noindent
The proof will play with an {\it a priori} estimate
\begin{equation}
     \bigl\| \nabla \mu^R * g \bigr\|_{L^{q_0}(\R^n)}
 \le B(q_0, R, n) \ms1 \|g\|_{L^{q_0}(\R^n)},
 \quad
 g \in L^{q_0}(\R^n),
 \label{Objective}
\end{equation}
and will aim to find a relation of the form
$B(q_0, R, n)
 \le c(q_0, \delta) \ms1 R^\delta + \varepsilon B(q_0, R, n)$ for some
$\varepsilon < 1$ and for $R$ larger than some $R_1$, for example with
$\varepsilon = 1/2$, thus reaching the conclusion that
$B(q_0, R, n) \le 2 \ms1 c(q_0, \delta) \ms1 R^\delta$ when $R > R_1$. We
know that $B(q_0, R, n)$ is finite for every dimension $n$, for instance as
a consequence of the trivial bound 
$\| \nabla \mu^R \|_1 \le \| \nabla G^R \|_1
 \le \kappa \ms1 \sqrt n \ms2 R$.
\dumou

 We must notice that the \textit{a priori} estimate in $\R^n$ yields the
same estimate for the dimensions $\ell \le n$, with a smaller or equal
constant, precisely, we must know that $B(q_0, R, \ell) \le B(q_0, R, n)$
when $1 \le \ell \le n$. Indeed, the forthcoming proof in dimension~$n$
will bring the question down to dimensions $\ell \le n$, where we shall use
the \textit{a priori} bound by $B(q_0, R, \ell)$. For justifying the
validity of the same bound when $\ell \le n$, apply the case~$n$ to a
function $g$ of the form $g_1 \otimes \varphi$, namely
\[
 g(x_1, x_2) = g_1(x_1) \varphi(x_2),
\]
where $x_1$ is in $\R^\ell$, $g_1 \in L^{q_0}(\R^\ell)$, 
$x_2 \in \R^{n - \ell}$ and where $\varphi$ is a fixed $C^\infty$ function
with compact support in $\R^{n - \ell}$, not identically zero. The
indicator of the cube and the Gaussian density have a product structure,
which allows us to write
\[
 K^R(x_1, x_2) = K^R_1(x_1) \ms1 \psi(x_2),
 \quad
 \d \mu^R(x_1, x_2) = \d \mu^R_1(x_1) \otimes (\psi(x_2) \, \d x_2),
\]
where $K_1$, $K_1^R$ and $\d \mu^R_1(x_1) = K^R_1(x_1) \, \d x_1$
correspond to the cube in $\R^\ell$, and $\psi$ is a probability density on
$\R^{n - \ell}$ corresponding to the cube in $\R^{n-\ell}$. We also have
\[
 \mu^R * g = (\mu^R_1 * g_1) \otimes (\psi * \varphi).
\]
The gradient of $\mu^R * g$ contains 
$
 (\nabla \mu^R_1 * g_1) \otimes (\psi * \varphi)
$
in its first $\ell$ coordinates, thus
\begin{align*}
   &\ms4 \bigl\| \nabla \mu^R_1 * g_1 \bigr\|_{L^{q_0}(\R^\ell)}
    \ms4 \bigl\| \psi * \varphi \bigr\|_{L^{q_0}(\R^{n - \ell})}
  = \ms4 \bigl\| 
    (\nabla \mu^R_1 * g_1) \otimes (\psi * \varphi) 
   \bigr\|_{L^{q_0}(\R^n)}
 \\
 \le &\ms4 \bigl\| \nabla \mu^R * g \bigr\|_{L^{q_0}(\R^n)}
 \le B(q_0, R, n) \ms2 \|g\|_{L^{q_0}(\R^n)}
 \\
  =  &\ms4 B(q_0, R, n) \ms2 \|g_1\|_{L^{q_0}(\R^\ell)} 
      \|\varphi\|_{L^{q_0}(\R^{n - \ell})}.
\end{align*}
This yields
$
     B(q_0, R, \ell)
 \le B(q_0, R, n) \ms2 
      \|\varphi\|_{L^{q_0}(\R^{n - \ell})} \ms1\big/\ms1
       \bigl\| \psi * \varphi \bigr\|_{L^{q_0}(\R^{n - \ell})}
$
and by spreading $\varphi$, replacing it with
$\varphi_k : x \mapsto \varphi(x / k)$, $k \rightarrow +\infty$, one makes
the quotient of norms tend to $1$, thus proving that
$
     B(q_0, R, \ell)
 \le B(q_0, R, n)
$.

\subsection{First reduction%
\label{FirstReduc}}

\noindent
One applies a result of Pisier~\cite{PisierHSG} about holomorphic semi-groups.
If $\gr T = (T_j)_{j=1}^n$ is a family of bounded linear operators 
on~$L^q(X, \Sigma, \mu)$, $1 \le q \le +\infty$, we introduce for every
subset $J \subset N = \{1, \ldots, n\}$ the operators
\[
 \label{TupJ}
 \gr T^J = \prod_{j \in J} T_j,
 \ms{16}
   \gr T^{\sim J} 
 = \gr T^{N \setminus J} 
 = \prod_{j \notin J} T_j,
\]
and $\gr T^{\sim j}$ will be a short form for $\gr T^{\sim \{j\} }$,
$1 \le j \le n$. We found the notation~$\gr T^{\sim J}$ convenient, but it
might be ambiguous, since it depends on the ambient set~$N$.
\dumou

 Given commuting projectors $(E_j)_{j=1}^n$, one can consider the
semi-group 
\[
   T_t
 = \prod_{j=1}^n \ms1 \bigl( E_j + \e^{-t} (I - E_j) \bigr),
 \ms{12}
 t \ge 0,
\]
where $I$ denotes the identity operator.\label{IdOp}
If we set $z = \e^{-t}$ and expand the product, we can arrange it according
to powers of~$z$, displaying in this way 
homogeneous parts $z^k H_k$\label{HomogPart}
of degree~$k$. We see that
\[
   T_t
 = \sum_{k=0}^n z^k 
    \Bigl( \sum_{|J| = k} \ms1 \gr E^{\sim J} (\gr I - \gr E)^J \Bigr)
 = \sum_{k=0}^n z^k H_k
 = \sum_{k=0}^n \e^{- k t} \ms1 H_k.
\]
Letting $\Sigma_k$\label{SigmaK} 
denote the family of subsets $J \subset N$ of cardinality $k$, we have
\begin{equation}
   H_k
 = \sum_{J \in \Sigma_k} \ms1 \gr E^{\sim J} (\gr I - \gr E)^J,
 \quad k = 0, \ldots, n,
 \ms{18} \hbox{and} \ms{12}
 \sum_{k=0}^n H_k = T_0 = I.
 \label{IciChk}
\end{equation}

\begin{prp}[after Pisier~\cite{PisierHSG}]%
\label{PisierSG}
Let\/ $(E_j)_{j=1}^n$ be a family of commuting conditional expectation
projectors on $L^q(X, \Sigma, \mu)$, $1 < q < +\infty$, and consider the
semi-group 
\[
   P_t
 = \prod_{j=1}^n \ms2 \bigl( \e^{-t} I + (1 - \e^{-t}) E_j \bigr)
 = \prod_{j=1}^n \ms2 \bigl( E_j + \e^{-t} (I - E_j) \bigr),
 \quad t \ge 0.
\]
This semi-group is analytic on $L^q(X, \Sigma, \mu)$, $1 < q < +\infty$,
with an extension\/ $(P_z)_{z \in \Omega_{\varphi_q}}$ to a sector\/ 
$\Omega_{\varphi_q} = \{ z = r \e^{\ii \theta} :
 r > 0, \ms6 |\theta| < \varphi_q \}$ in\/~$\C$, where $\varphi_q > 0$
depends on $q$ only. The extension is bounded uniformly in $q$ on every
compact subset of\/~$\Omega_{\varphi_q}$. There exists~$h_q \ge 1$
independent of~$n$ such that whenever\/ $0 \le k \le n$, the homogeneous
part $H_k$ in\/~\eqref{IciChk} is bounded on $L^q(X, \Sigma, \mu)$ by\/
$(h_q)^k$.
\end{prp}

 That $h_q \ge 1$ can be seen on any example
$P_t \ms1 f = E_1 f + \e^{-t} ( f - E_1 f )$ with $n = 1$
and $E_1 \ne I$. Then
$H_1$ is the projector $I - E_1 \ne 0$, hence 
$h_q \ge \|H_1\|_{q \rightarrow q} \ge 1$.
If $(E_{j, s})_{j=1}^n$, $s \in [0, 1]$, is a family of such conditional
expectations, where $E_{j, s}$ and $E_{k, t}$ commute for all $j \ne k$ and
all $s, t \in [0, 1]$, and if we set for example
\[
 U_j = \int_0^1 E_{j, s} \, \d s,
 \quad j = 1, \ldots, n,
\]
then we see that
\[
   Q_t
 = \prod_{j=1}^n \bigl( \e^{-t} I + (1 - \e^{-t}) U_j \bigr)
 = \int_{[0, 1]^n} P_{t, s_1, \ldots, s_n} \, \d s_1 \ldots \d s_n,
\]
where each
$
   P_{t, s_1, \ldots, s_n}
 = \prod_{j=1}^n \bigl( \e^{-t} I + (1 - \e^{-t}) E_{j, s_j} \bigr)
$
is of \og Pisier type\fg. Also, the corresponding homogeneous parts are of
the form
\[
   \widetilde H_k
 = \sum_{J \in \Sigma_k} \gr U^{\sim J} (\gr I - \gr U)^J
 = \int_{ [0, 1]^n } 
    \sum_{J \in \Sigma_k} \Bigl( \prod_{i \notin J} E_{i, s_i} \Bigr)
     \Bigl( \prod_{j \in J} (I - E_{j, s_j}) \Bigr)
      \, \d s_1 \ms1 \d s_2 \ldots \d s_n 
\]
that are averages of terms $H_k(s_1, \ldots, s_n)$ bounded by $h_q^k$
according to Proposition~\ref{PisierSG}. The result of
Proposition~\ref{PisierSG} generalizes thus to families such as
$(U_j)_{j=1}^n$.
\dumou

 We shall apply Proposition~\ref{PisierSG} to operators $(E_j)_{j=1}^n$ of
conditional expectation on $L^q(\R^n)$, where each $E_j$ is acting in the
$x_j$~variable and $1 \le j \le n$. For one variable and $s_0 \in \R$
fixed, we associate to a locally integrable function $f$ on~$\R$ its
averages on length one intervals $I_r = [s_0 + r, s_0 + r + 1)$, 
$r \in \Z$, defining $E_{s_0}$ by
\[
   (E_{s_0} f)(v)
 = \sum_{r \in \Z} \Bigl( \int_{I_r} f(s) \, \d s \Bigr)
    \ms2 \gr 1_{I_r}(v),
 \quad
 v \in \R.
\]
This operator is a conditional expectation, as considered in
Remark~\ref{InfiniProba}. We define operators $E_{j, s_0}$, 
$j = 1, \ldots, n$, on $L^1_{\mathrm{loc}}(\R^n)$ by the analogous formula,
acting on the~$x_j$ variable. When $j = 1$ for example, we let
\[
   (E_{1, s_0} f)(x_1, x_2, \ldots, x_n)
 = \sum_{r \in \Z} 
    \Bigl( \int_{I_r} f(s, x_2, \ldots, x_n) \, \d s \Bigr)
     \ms2 \gr 1_{I_r}(x_1).
\]
Averaging on values of $s_0$, one can replace the $E_j\ms1$s by
convolution operators with probability densities $\chi$ on $\R$ of the form
\begin{equation}
 \chi(x) = \int_\R \gr 1_{[s, s+1]} (x) \, \d \nu(s),
 \quad
 x \in \R,
 \label{ChiFunction}
\end{equation}
where $\nu$ is a probability measure on the line. We see that 
$\chi(x) = F(x) - F(x - 1)$, with $F(x) = \nu[ (-\infty, x) ]$
non-decreasing, $F(-\infty) = 0$ and $F(+\infty) = 1$. One can also proceed
to changes of scale. Summarizing, we have the lemma that follows.
\dumou

\begin{lem}[Bourgain~\cite{BourgainCube}, Lemma~5]%
\label{CubeLemme5}
Let $\chi$ be a compactly supported probability density on\/~$\R$ of the 
form\/~\eqref{ChiFunction}. Denote by $T_j$ the convolution operator with
$\Di \chi {t_j}$ in the $x_j$ variable, $t_j > 0$, $j = 1, \ldots, n$.
For\/ $0 \le k \le n$, the norm of the operator
\[
    H_k
 := \sum_{S \in \Sigma_k} \gr T^{\sim S} (\gr I - \gr T)^S
\]
on $L^q(\R^n)$ is bounded by $h_q^k$, with\/ $1 < q < +\infty$ and $h_q$
from Proposition~\ref{PisierSG}.
\end{lem}

 In what follows, we denote by~$T_j$, $j = 1, \ldots, n$, the convolution in
the $x_j$ variable on $L^q(\R^n)$ by $\Di \eta {w_0} (x_j)$, where 
$w_0 = R^{-\delta / 2}$\label{Wsub0} 
will stay fixed and where
\[
   \eta(x) = (1 - |x|)_+
 = \int_{-1/2}^{1/2} \gr 1_{[-1/2 + s, 1/2+s]}(x) \, \d s
 = (\gr 1_{[-1/2, 1/2]} * \gr 1_{[-1/2, 1/2]})(x).
\]
Since $\eta$ is a convolution square, $\widehat \eta$ is real and
nonnegative. We have
\[
   \widehat \eta(t)
 = \Bigl( \frac {\sin(\pi t)} {\pi t} \Bigr)^2,
 \ms{20} \hbox{and} \ms{20}
   \widehat \eta {\ms3}''(t)
 = - 4 \pi^2 \int_\R s^2 (1 - |s|)_+ \cos(2 \pi s t) \, \d s
\]
for every $t \in \R$, thus
\[
     |\widehat \eta {\ms3}''(t)|
 \le 8 \pi^2 \int_0^1 s^2 (1 - s) \, \d s
  =  \frac {8 \pi^2 \ns7} {12} \ms3
  <  8.
\]
By the Taylor formula we get
\begin{equation}
 0 \le 1 - \widehat \eta(t) \le (4 t^2) \wedge 1.
 \label{Eta}
\end{equation}
For every subset $S \subset N := \{1, \ldots, n\}$ let us set 
\begin{equation}
 \Gamma^S = \gr T^{\sim S} \ms1 (\gr I - \gr T)^S.
 \label{DefGammaS}
\end{equation}
The homogeneous parts $(H_k)$ in 
$Q_t = \prod_{j=1}^n \bigl( \e^{-t} I + (1 - \e^{-t}) T_j \bigr)$ have
the form
\[
 H_k = \sum_{S \in \Sigma_k} \Gamma^S,
 \quad
 0 \le k \le n,
 \ms{18} \hbox{and} \ms{12}
 \sum_{k=0}^n H_k = I.
\]
In particular,
$
 H_0 = \Gamma^{\emptyset} = \gr T^N = \prod_{j=1}^n T_j
$
has norm $\le 1$ on every space $L^q(\R^n)$, for $1 \le q \le +\infty$,
since $H_0$ is the convolution with the product probability density
$\prod_{j=1}^n \Di \eta {w_0} (x_j)$. When $1 < q < +\infty$ and 
$1 \le k \le n$, we have
$
 \|H_k\|_{q \rightarrow q} \le h_q^k
$
by Proposition~\ref{PisierSG}. It is convenient to
set $H_k = 0$ below when $k > n$.
\dumou

 To every given function~$g$ in $L^q(\R^n)$, we shall apply a decomposition
of the form
$
 g = H_0 \ms1 g + \cdots + H_{M-1} \ms1 g + h
$,
and consider the corresponding expression
\begin{equation}
   \nabla \mu^R * g 
 = \nabla \mu^R * H_0 \ms1 g + \cdots + \nabla \mu^R * H_{M-1} \ms1 g
    + \nabla \mu^R * h,
 \label{decomp}
\end{equation}
where $M \ge 1$ will be chosen as a function of the already fixed $p_0$ and
$\delta > 0$. We have to estimate in $L^{q_0}(\R^n)$ the successive terms
in~\eqref{decomp}. The function~$h$ is considered as a small rest, the
mapping $g \mapsto \nabla \mu^R * h$ will be handled in $L^2(\R^n)$ by a
Fourier estimate, and in some $L^{q_1}(\R^n)$, $q_1 > q_0$, as a
consequence of Proposition~\ref{PisierSG}. We choose~$M$ large enough for
deducing from
$
     \| \nabla \mu^R * h \|_2 
 \le \kappa \ms1 R^{1 - \delta M / 2} \ms2 \|g\|_2
$
and
$
     \| \nabla \mu^R * h \|_{q_1}
 \le \kappa \ms1 R \ms2 \|g\|_{q_1}
$
that one has by interpolation
\begin{equation}
     \| \nabla \mu^R * h \|_{q_0}
 \le \kappa(q_0, \delta) \ms1 \|g\|_{q_0},
 \label{RestH}
\end{equation}
which is just perfect in the direction of~\eqref{Objective}. Recall that
$\mu^R_j$ denotes the $j$th partial derivative 
$\partial_j \mu^R = (\partial_j \mu_Q) * G^R$ of $\mu^R$, so that
$
   | \nabla \mu^R * h |^2
 = \sum_{j=1}^n |\mu_j^R * h|^2
$.
\def\vI{\vphantom{\vrule height 8.50pt depth 0pt width 0pt}\ns{2.1} }%
\dumou

 We factor the mapping $g \mapsto \nabla \mu^R * h$ into
$U_{\vI K^R} : h \mapsto \nabla \mu^R * h$ and $A : g \mapsto h$,
\textit{i.e.}, $A = I - H_0 - \cdots - H_{M-1} = \sum_{k \ge M} H_k$. We
look for estimates in $L^2$ and~$L^q$, $q_0 < q < +\infty$. For 
$U_{\vI K^R}$ we use Lemma~\ref{LemmeUK} and get by~\eqref{SigmaR}
and~\eqref{EasyMuller} that
\[
     \| U_{\vI K^R} \|_{q \rightarrow q}
 \le 2^{1 / q} \ms2 \sigma_R^{-2 / q} \ms2 V(K^R)^{1 - 2/q}
 \le (24)^{1/q} \ms1 R^{1 - 2/q}
  <  5 \ms2 R
\]
since $q \ge 2$.
On the other hand, by Lemma~\ref{CubeLemme5}, the mapping $A : g \mapsto h$
is bounded in $L^q(\R^n)$ by 
$1 + \sum_{k=0}^{M-1} h_q^k \le (M + 1) h_q^{M-1}$.
It follows that
\begin{equation}
     \| \nabla \mu^R * h \|_q
  =  \| U_{\vI K^R} h \|_q
 \le 5 \ms1 R \ms2 \|h\|_q
 \le 5 R \ms2 
      (M + 1) h_q^{M-1} \ms1 \|g\|_q.
 \label{hDansLq}
\end{equation}

 This is also valid when $q = 2$, but the point is that we will then get a
much better bound by factoring now $g \mapsto \nabla \mu^R * h$ as
$U_{\vI G^R} \circ B$, with $U_{\vI G^R} : f \mapsto \nabla G^R * f$ and 
$B : g \mapsto \mu_Q * A g = \mu_Q * h$. We begin by estimating
\[
   \| \mu_Q * h \|_2
 = \Bigl\|
      \mu_Q * \Bigl( \sum_{k \ge M} H_k \Bigr) g 
   \Bigr\|_2
 = \Bigl\|
      \mu_Q * \Bigl( \sum_{|S| \ge M} \Gamma^S \Bigr) g 
   \Bigr\|_2.
\]
One needs to control the $L^\infty(\R^n)$~norm of the function 
$\xi \mapsto L(\xi)$, where $L$ is the multiplier associated to the mapping
$B$. It is the aim of the next lemma. One sees that
\[
    L(\xi)
 := \Bigl( \prod_{j=1}^n \frac {\sin(\pi \xi_j)} {\pi \xi_j} \Bigr)
      \Bigl( \sum_{|S| \ge M}
       \prod_{j \notin S} \widehat \eta(w_0 \xi_j)
        \prod_{j \in S} \bigl( 1 - \widehat \eta(w_0 \xi_j) \bigr) 
     \Bigr).
\]

\def\CiteA{\cite[Equations (2.9), (2.11)]{BourgainCube}}
\begin{lem}[after Bourgain~\CiteA]%
\label{GrosCalcul}
For\/ $0 \le u \le 1 / 4$ and every $\xi \in \R^n$, one has that
\[
     \Bigl|
      \Bigl(
       \prod_{j=1}^n \frac {\sin(\pi \xi_j)} {\pi \xi_j} 
      \Bigr)
       \Bigl( \sum_{|S| \ge M}
        \prod_{j \notin S} \widehat \eta(u \xi_j)
        \prod_{j \in S} \bigl( 1 - \widehat \eta(u \xi_j) \bigr) 
       \Bigr)
     \Bigr|
 \le u^M.
\]
\end{lem}

\begin{proof}
We know from~\eqref{Eta} that $0 \le \widehat \eta(t) \le 1$ and
$1 - \widehat \eta(t) \le (4 t^2) \wedge 1$. We introduce 
$v = 1 / u \ge 4$ and begin by checking that for every~$t \ge 0$, we
have
\[
     X(t)
  := \Bigl| \frac {\sin(\pi t)} {\pi t} \Bigr|
      \bigl( 1 + v \ms1 [ (4 u^2 t^2) \wedge 1 ] \bigr)
 \le 1.
\]
Consider first the case $0 \le t \le 1 / (2u)$. One has then
$4 u^2 t^2 \le 1$ and it follows that 
$1 + v [ (4 u^2 t^2) \wedge 1 ] = 1 + 4 u t^2$. If in addition
$0 \le t \le 1$, then, for example by the Euler product
formula~\eqref{EulerFormula}, we have 
$
     \bigl| {\sin(\pi t)} \ms1\big/\ms1 {\pi t} \bigr|
 \le 1 - t^2,
$
and since $4 \ms1 u \le 1$ by assumption, we get
\[
     X(t)
 \le (1 - t^2)(1 + 4 u t^2)
 \le (1 - t^2)(1 + t^2)
 \le 1.
\]
When $1 < t \le 1 / (2u)$, we have 
\[
     \Bigl| \frac {\sin(\pi t)} {\pi t} \Bigr|
      \bigl( 1 + 4 u t^2 \bigr)
 \le \frac {1 + 4 u t^2} {\pi t}
  =  \frac 1 \pi \bigl( 1/t + 4 u t \bigr)
 \le \frac 3 \pi 
  <  1.
\]
In the second case, when $2 u t > 1$, we can write
\[
     X(t)
 \le \frac {1 + v} {\pi t}
 \le \frac { 2 u (1 + v)} {\pi}
 \le \frac { 1/2 + 2} {\pi}
  <  1.
\]
Expanding the product $\prod_{j=1}^n X(\xi_j)$ and since $X$ is even, one
sees that
\begin{align*}
     1
 &\ge \prod_{j=1}^n X(\xi_j)
  =  \prod_{j=1}^n \ms2 \Bigl| \frac {\sin (\pi \xi_j)} {\pi \xi_j} \Bigr|
      \bigl( 1 + v [ (4 u^2 \xi_j^2) \wedge 1 ] \bigr)
 \\
 &\ge \prod_{j=1}^n \ms2 \Bigl| \frac {\sin (\pi \xi_j)} {\pi \xi_j} \Bigr|
      \Bigl( \widehat \eta(u \xi_j)
         + v \bigl( 1 - \widehat \eta(u \xi_j) \bigr) \Bigr)
 \\
 &\ge v^M \ms2
      \Bigl|
       \Bigl( \prod_{j=1}^n \frac {\sin (\pi \xi_j)} {\pi \xi_j} \Bigr)
       \Bigl(
        \sum_{|S| \ge M}
         \prod_{j \notin S} \widehat \eta(u \xi_j)
          \prod_{j \in S} \bigl( 1 - \widehat \eta(u \xi_j) \bigr)
       \Bigr)
      \Bigr|.
\end{align*}
\end{proof}

 By Lemma~\ref{LemmeUK}, we have that 
$\|U_{\vI G^R}\|_{2 \rightarrow 2} \le \sqrt \pi \ms1 R < 2 \ms1 R$,
because the variance of $G^R$ is $2 \pi^{-1} R^{-2}$. Let us define 
$R_0$\label{Rzero}
by $R_0^{\delta/2} = 4$. If $R \ge R_0$, then $w_0 = R^{-\delta/2} \le 1/4$,
we obtain from Lemma~\ref{GrosCalcul} with $u = w_0$ the final control
\[
     \| \nabla \mu^R * h \|_2
  =  \| U_{\vI G^R} (\mu_Q * h) \|_2
 \le 2 \ms1 R \ms2 \|\mu_Q * h\|_2
 \le 2 \ms1 R \ms2 R^{-\delta M/2} \ms1 \|g\|_2.
\]
We use now~\eqref{hDansLq} with for example $q = q_1 = 2 q_0 / p_0 > q_0$.
Letting $\theta = 1 / p_0$, we have
$(1 - \theta) / 2 + \theta / q_1 = 1 / q_0$ and we see by interpolation for
$g \mapsto \nabla \mu^R * h$ that
\[
     \| \nabla \mu^R * h \|_{q_0} 
 \le \bigl(
      2 \ms1 R^{ - \delta M / 2 } 
     \bigr)^{1 / q_0} \ms2 
      R \ms2
       \bigl( 5 \ms1 (M + 1) h_{q_1}^{M-1} \bigr)^{1/p_0} \ms2
        \|g\|_{q_0}.
\]
We select $M = M(\delta) = \lceil 2 q_0 / \delta \rceil$, so that
$\delta M / (2 q_0) \ge 1 $. When $R \ge R_0$ we get
\begin{equation}
     \|\nabla \mu^R * h \|_{q_0} 
 \le \kappa_{q_0, \delta} \|g\|_{q_0}
 \ms{14} \hbox{with} \ms{10}
     \kappa_{q_0, \delta} 
 \le 5 \ms1 (2 + 2 q_0 / \delta)^{1/p_0}
      h_{2 q_0 / p_0}^{ 2 q_0 / (\delta p_0) }.
 \label{EstimeH}
\end{equation}
In what follows we assume that $R \ge R_0$, hence $R^\delta \ge 16$. In the
conclusion section, we shall need the following bound for a Fourier
transform.

\begin{lem}%
\label{FourBound2}
For every $r \in \R$, $\ell \ge 1$ and all 
$\xi = (\xi_1, \ldots, \xi_\ell) \in \R^\ell$, one has that
\[
     (1 - \e^{ - r^2 |\xi|^2}) \ms2
      \prod_{j = 1}^\ell \widehat \eta(\xi_j)
 \le r^2.
\] 
\end{lem}

\begin{proof}
We observe first that
\[
     \widehat \eta(t) 
  =  \Bigl( \frac {\sin(\pi t)}{\pi t} \Bigr)^2
 \le \frac 1 {1 + t^2} \up.
\]
This is clear when $|t| \ge 1$ because $\widehat \eta(t) \le (\pi t)^{-2}$
and $1 + t^2 < \pi^2 t^2$ in this case. When $|t| \le 1$ we
have $\widehat \eta(t) \le |\sin (\pi t)| / |\pi t|
 \le (1 - t^2) \le (1 + t^2)^{-1}$ by~\eqref{EulerFormula}. It suffices
thus to bound for $x \in \R^\ell$ the expression
\[
     F(x)
  =  (1 - \e^{ - r^2 |x|^2}) \ms2
      \prod_{j = 1}^\ell \frac 1 {1 + x_j^2} 
 \ge 0.
\]
\def\ovx{\overline{x \ns{1.3}
 \raise0.7pt\hbox to0pt{\vphantom {$x$}}}\ms{1.3}}%
\def\ovxs{\overline{x \ns{1.2}
 \raise0.5pt\hbox to0pt{\vphantom 
 {$\scriptstyle x$}}}\ms{1.2} }%
The function $F$ tends to $0$ at infinity, we have at any maximum
$\ovx \ne 0$ that
\[
   \frac {2 \ms1 r^2 \ms1 \ovx_j \e^{ - r^2 |\ovxs|^2}} 
         {1 - \e^{ - r^2 |\ovxs|^2}}
 = \frac {2 \ms1 \ovx_j} {1 + {\ovx \ms{0.4}}_j^2} \up,
 \quad
 j = 1, \ldots, \ell.
\]
The nonzero coordinates of $\ovx$ have the same
square ${\ovx \ms{0.4}}_j^2 =: y > 0$, and if $k$ denotes their cardinality,
we have $0 < k \le \ell$ and $|\ovx|^2 = k \ms1 y$. It follows that
\[
     k \ms1 r^2 y
 \le \e^{ k \ms1 r^2 y} - 1 
  =  r^2 (1 + y)
 \le r^2 (1 + y)^k.
\]
Finally, we have $F(\ovx) = (1 - \e^{ - k \ms1 r^2 y})(1 + y)^{-k}
 \le k \ms1 r^2 y (1 + y)^{-k} \le r^2$.
\end{proof}

\subsubsection{Decoupling%
\label{Decoupling}}

\noindent
We have to analyze each of the expressions $\nabla \mu^R * H_k g$
in~\eqref{decomp}, for $0 \le k < M$. When $1 \le k < M$, we handle this by
a decoupling argument that will allow us to essentially reduce to the cases
where $k = 0, 1$, but in a dimension $\ell \le n$. Before proceeding by a
Bourgainian technique of \og selectors\fg, we split 
\[
    \bigl| \nabla \mu^R * H_k g \bigr|
  = \Bigl( \sum_{j=1}^n |\mu_j^R * H_k g|^2 \Bigr)^{1/2}
  = \Bigl( 
     \sum_{j=1}^n \ms8
      \Bigl|
       \mu_j^R * \Bigl( \sum_{S \in \Sigma_k} \Gamma^S g \Bigr) 
      \Bigr|^2 
    \Bigr)^{1/2}
\]
into two. For each $j$ in $\{1, \ldots, n\}$, let $\Sigma_k^j$ and
$\Sigma_k^{\sim j}$ denote respectively the family of subsets $S$ of
$\{1, \ldots, n\}$ with cardinality $|S| = k$ containing $j$, resp. such
that $j \notin S$. Then $\bigl| \nabla \mu^R * H_k g \bigr|$ is bounded by
the sum of the two expressions
\begin{subequations}\label{E0split}
\begin{equation}
    \gr E_k(R, n, g)
 := \Bigl( 
     \sum_{j=1}^n \ms8
      \Bigl|
       \mu_j^R * \Bigl( \sum_{S \in \Sigma_k^{\sim \ns1 j}} \Gamma^S g \Bigr) 
      \Bigr|^2 
    \Bigr)^{1/2}
\end{equation}
and
\begin{equation}
    \gr F_k(R, n, g)
 := \Bigl( 
     \sum_{j=1}^n \ms8 
      \Bigl|
       \mu_j^R * \Bigl( \sum_{S \in \Sigma_k^j} \Gamma^S g \Bigr)
      \Bigr|^2 
    \Bigr)^{1/2}.
\end{equation}
\end{subequations}
\dumou

 Assume that $1 \le k < M = M(\delta)$. 
Let $(\gamma_i)_{1 \le i \le n}$ be\label{Selektor}
independent $\{0, 1\}$-valued random variables with mean $1 / (k+1)$ on some
probability space $(\Omega, \ca F, P)$. For each~$j$ in~$\{1, \ldots, n\}$
and $S \in \Sigma_k^{\sim j}$, let
$
 \sigma_{S, j} = \gamma_j \prod_{i \in S} (1 - \gamma_i)
$.
We have that
\[
   \E \sigma_{S, j} 
 = \frac 1 {k+1} \Bigl( 1 - \frac 1 {k+1} \Bigr)^k
 = \frac {k^k} { (k+1)^{k+1} \ns{13}} \ms8
 =: e_k,
 \quad j = 1, \ldots, n,
\]
and $e_k^{-1} \le \e \ms1 (k+1) \le \e \ms1 M$ because
$\e^{1/k} > 1 + 1 / k$. By convexity, we see that
\begin{align*}
     e_k \ms1 \gr E_k (R, n, g)
  &=   \Bigl( \sum_{j=1}^n \ms1
        \Bigl|
         \mu_j^R * 
          \Bigl(
           \sum_{S \in \Sigma_k^{\sim \ns1 j}} 
            \bigl[ \E_\omega \sigma_{S, j}(\omega) \bigr]
             \ms2 \Gamma^S g
          \Bigr)
        \Bigr|^2 
       \Bigr)^{1/2}
 \\
 &\le \E_\omega 
       \Bigl[ \Bigl( \sum_{j=1}^n \ms1
        \Bigl|
         \mu_j^R * 
          \Bigl(
           \sum_{S \in \Sigma_k^{\sim \ns1 j}} 
            \sigma_{S, j}(\omega) \ms2 \Gamma^S g
          \Bigr)
        \Bigr|^2 
       \Bigr)^{1/2} \Bigr].
\end{align*}
Let $q \ge 1$ be given. It follows that for some $\omega_0 \in \Omega$, we
have
\begin{equation}
     \bigl\| \gr E_k (R, n, g) \bigr\|_{L^q(\R^n)}
 \ns2\le\ns1 \e M \ms2
      \Bigl\| 
       \Bigl( \sum_{j=1}^n \ms1
        \Bigl|
         \mu_j^R \ns1*\ns1 
          \Bigl(
           \sum_{S \in \Sigma_k^{\sim j}}
            \sigma_{S, j}(\omega_0) \ms2 \Gamma^S g
          \Bigr) 
        \Bigr|^2 
       \Bigr)^{1/2}
      \Bigr\|_{L^q(\R^n)}.
 \label{2.17}
\end{equation}
Let $J_0 = \{ j : \gamma_j(\omega_0) = 1 \}$. Then
$\sigma_{S, j}(\omega_0) = 0$ whenever $S$ meets $J_0$ or $j \notin J_0$.
The $L^q(\R^n)$~norm at the right-hand side of~\eqref{2.17} is therefore
the norm of
\[
    E(J_0, g)
 := \Bigl( \sum_{j \in J_0} \ms1
     \Bigl| 
      \mu_j^R * \Bigl( \sum_{S \in \Sigma_k^{\sim J_0}} \Gamma^S g \Bigr)      
     \Bigr|^2 
    \Bigr)^{1/2},
\]
where $\Sigma_k^{\sim J_0}$ denotes the family of subsets $S$ of 
$\{1, \ldots, n\}$ such that $|S| = k$ and that are disjoint from $J_0$.
Let us introduce the operator
\[
   \gr U 
 = \sum_{S \in \Sigma_k^{\sim J_0}}
    \gr T^{\sim (J_0 \cup S)} \ms2 (\gr I - \gr T)^S
 \ms{18} \hbox{and the function} \ms{18}
 \Psi = \gr U \ms2 g
\]
on $\R^n$. We see that
$\gr T^{J_0} \gr U = \sum_{S \in \Sigma_k^{\sim J_0}} \Gamma^S$, and the
operator $\gr U$ acts on the variables not in $J_0$ as does the~$k\ms1$th
homogeneous part~$H_k$ relative to $\R^{ \{1, \ldots, n\} \setminus J_0}$. 
Consequently, applying Proposition~\ref{PisierSG} in the variables
$\gr x^{\sim J_0} = (x_i)_{i \notin J_0}$, we get
\def\down#1{\raise -1.95pt \hbox{${}_{#1}$}}
\begin{equation}
     \| \Psi_{\down{\gr x^{J_0}}} \|_{\down{L^q(\R^{\sim J_0})}} 
 \le h_q^k \ms2 \|g_{\down{\gr x^{J_0}}} \|_{\down{L^q(\R^{\sim J_0})}}
 \label{2.20}
\end{equation}
for every fixed $\gr x^{J_0} = (x_i)_{i \in J_0}$,
where $f_{\down{\gr x^J}} (\gr x^{\sim J})
 := f(\gr x^J, \gr x^{\sim J}) = f(x)$, and we see that
$
   E(J_0, g)
 = \bigl( \sum_{ j \in J_0}
    \bigl| \mu_j^R * \gr T^{J_0} \Psi \bigr|^2
   \bigr)^{1/2}
$.
Assume that there exists $b_0(q_0, R, n)$ such that for every subset $J$
of $\{1, \ldots, n\}$ and $f \in L^{q_0}(\R^J)$ we have
\begin{equation}
     \Bigl\|
      \Bigl(
       \sum_{j \in J} \ms1 
        \Bigl|
         (\mu_{Q^J})_j^R * \Bigl( \prod_{i \in J} T_i \Bigr) f 
        \Bigr|^2 
      \Bigr)^{1/2} 
     \Bigr\|_{L^{q_0}(\R^J)}
 \le b_0(q_0, R, n) \ms2 \|f\|_{L^{q_0}(\R^J)},
 \label{2.22}
\end{equation}
with $\mu_{Q^J}$ uniform on $Q^J := [-1/2, 1/2]^J$ in $\R^J$.
It follows from~\eqref{2.20}, by integrating in the $J_0$~variables, that
$\|E(J_0, g)\|_{L^{q_0}(\R^n)}
 \le b_0(q_0, R, n) \ms2 h_q^k \ms1 \|g\|_{L^{q_0}(\R^n)}$. 
\dumou

 For $\gr F_k (R, n, g)$ we proceed similarly, writing each 
$S \in \Sigma_k^j$ as $S = \{j\} \cup S_1$, with $|S_1| = k - 1$, and
using now
$  \sigma_{S_1, j}
 = \gamma_j \prod_{i \in S_1} (1 - \gamma_i)
$ for which we have $\E \sigma_{S_1, j}
 = k^{k-1} (k+1)^{-k} \ge 1 / (\e k) > 1 / (\e M)$.
We obtain for some $\omega_0 \in \Omega$ that
\[
     \bigl\| \gr F_k (R, n, g) \bigr\|_{L^q(\R^n)}
 \ns2\le\ns1 \e \ms1 M \ms1
      \Bigl\| 
       \Bigl( \sum_{j=1}^n \ms1
        \Bigl|
         \mu_j^R \ns1*\ns1 
          \Bigl(
           \sum_{S \in \Sigma_k^j}
            \sigma_{S_1, j}(\omega_0) \ms2 \Gamma^S g
          \Bigr) 
        \Bigr|^2 
       \Bigr)^{1/2}
      \Bigr\|_{L^q(\R^n)}.
\] 
Considering again $J_0 = \{ j : \gamma_j(\omega_0) = 1 \}$, we get instead
of $E(J_0, g)$ the expression
\[
   F(J_0, g)
 = \Bigl( \sum_{j \in J_0} 
    \Bigl|
     \mu_j^R *
      \Bigl(
       \sum_{S_1 \in \Sigma_{k-1}^{\sim J_0}} \Gamma^{ \{j\} \cup S_1} g 
      \Bigr) 
    \Bigr|^2 
   \Bigr)^{1/2}.
\]
When $k = 1$, we have $S_1 = \emptyset$, $S = \{j\}$ and
$\sigma_{S_1, j} = \gamma_j$, the argument remains correct but becomes
\og inactive\fg. Let now
$
 \Psi = \sum_{S_1 \in \Sigma_{k-1}^{\sim J_0}}
         \gr T^{\sim (J_0 \cup S_1)} (\gr I - \gr T)^{S_1} g
$,
satisfying by Proposition~\ref{PisierSG} applied to $L^q(\R^{\sim J_0})$
the inequality
\[
     \| \Psi_{\down{\gr x^{J_0}}} \|_{\down{L^q(\R^{\sim J_0})}} 
 \le h_q^{k-1} \ms2 
      \| g_{\down{\gr x^{J_0}}} \|_{\down{L^q(\R^{\sim J_0})}}.
\]
For each $j \in J_0$, let
$
 B_j = (I - T_j) \ms1 \gr T^{J_0 \setminus \{j\} }
$.
Then
$
   F(J_0, g) 
 = \bigl(
    \sum_{j \in J_0} |\mu_j^R * B_j \Psi|^2 
   \bigr)^{1/2}
$.
If there exists $b_1(q_0, R, n)$ such that for every subset $J$ of 
$\{1, \ldots, n\}$ and every function $f \in L^{q_0}(\R^J)$ we have an
inequality
\begin{equation}
     \Bigl\|
      \Bigl(
       \sum_{j \in J} 
        |\mu_j^R * (I - T_j) 
          \Bigl( \prod_{i \in J, \ms2 i \ne j} T_i \Bigr) f|^2 
      \Bigr)^{1/2}
     \Bigr\|_{L^{q_0}(\R^J)}
 \le b_1(q_0, R, n) \|f\|_{L^{q_0}(\R^J)},
 \label{2.27}
\end{equation}
it implies that $F(J_0, g)$ may be bounded in $L^{q_0}(\R^n)$ by 
$b_1(q_0, R, n) \ms2 h_q^{k-1} \ms1 \|g\|_{L^{q_0}(\R^n)}$.

 In view of~\eqref{2.22} and~\eqref{2.27}, all we need to do in order to
control in $L^{q_0}(\R^n)$ the expressions $\nabla \mu^R * H_k g$, when 
$1 \le k < M$, is to establish in all lower dimensions $\ell \le n$ and for
every function $f \in L^{q_0}(\R^\ell)$ the inequalities
\begin{equation}
     \bigl\| 
      \nabla \mu^R \ns2*\ns2 H_0 f
     \bigr\|_{L^{q_0}(\R^\ell)}
  =  \Bigl\|
      \Bigl( \sum_{j = 1}^\ell |\mu_j^R \ns1*\ns1 H_0 f|^2 \Bigr)^{1/2} 
     \Bigr\|_{L^{q_0}(\R^\ell)}
 \ns4\le\ns1 b_0(q_0, R, n) \ms1 \|f\|_{L^{q_0}(\R^\ell)}
 \label{2.28}
\end{equation}
and
\begin{equation}
     \bigl\| \gr F(R, \ell, f) \bigr\|_{L^{q_0}(\R^\ell)}
 :=  \Bigl\|
      \Bigl(
       \sum_{j = 1}^\ell |\mu_j^R \ns1*\ns1 \Gamma^j f|^2 
      \Bigr)^{1/2}
     \Bigr\|_{L^{q_0}(\R^\ell)}
 \ns4\le\ns1 b_1(q_0, R, n) \ms1 \|f\|_{L^{q_0}(\R^\ell)}
 \label{2.29}
\end{equation}
for suitable $b_0(q_0, R, n)$ and $b_1(q_0, R, n)$, with 
$\Gamma^j := \Gamma^{ \{j\} }
 = (I - T_j) \ms2 \gr T^{ \{1, \ldots, \ell\} \setminus \{j\} }$. Note 
that~\eqref{2.28} controls the so far neglected term $k = 0$
in~\eqref{decomp}. From~\eqref{RestH} and the preceding, this will permit
us to estimate
\[
     \bigl\|
      \nabla \mu^R * g
     \bigr\|_{L^{q_0}(\R^n)}
  =  \Bigl\|
      \Bigl( \sum_{j=1}^n |\mu_j^R * g|^2 \Bigr)^{1/2} 
     \Bigr\|_{L^{q_0}(\R^n)}
 \le C(q_0, R, n) \ms2 \|g\|_{L^{q_0}(\R^n)}.
\]
Recalling~\eqref{EstimeH}, \eqref{2.17} and that $M = M(\delta)$ depends
on the fixed value $\delta > 0$, we have when $R \ge R_0$ that
\begin{equation}
     C(q_0, R, n)
 \le \kappa_{q_0, \delta} 
      + \e M(\delta)^2 \ms2 h_{q_0}^{M(\delta)} \ms1 
         \bigl( b_0(q_0, R, n) + b_1(q_0, R, n) \bigr), 
 \label{BoundC}
\end{equation}
where the three terms correspond to the decompositions~\eqref{decomp}
and~\eqref{E0split}. By definition, it will follow that the 
\textit{a priori} bound $B(q_0, R, n)$ is less than $C(q_0, R, n)$. Bounds
on $b_0(q_0, R, n)$ and $b_1(q_0, R, n)$ will be obtained below, and will
use the other quantities $B(q_0, R, \ell) \le B(q_0, R, n)$, with 
$\ell \le n$. We shall get a relation
\[
     B(q_0, R, n) 
 \le c(q_0, \delta) \ms1 R^{4 \ms1 \delta} + B(q_0, R, n) / 2,
 \quad n \ge 1,
\]
for $R$ larger than some $R_1 \ge R_0$, and we shall be able to
conclude.

\subsection{Second reduction%
\label{SecondReduc}}

\noindent
Let $\tau > 0$ be given. We say that a nonnegative function~$f$ defined
on~$\R$ is \emph{$\tau$-stable\label{StabiTransla} 
with constant $C$} if whenever $|t| \le \tau$, we have
\[
 f(s + t) \le C \ms1 f(s),
 \quad s \in \R.
\]
One sees that $C \ge 1$. Evident properties are to be observed about
products, integrals, translations, convolutions\dots\ For example, if 
$f_1, \ldots, f_k$ are $\tau$-stable with respective constants $C_i$, then
clearly the product $f_1 \ldots f_k$ is $\tau$-stable with constant 
$C_1 \ldots C_k$. If $f$ is $\tau$-stable with constant~$C$ and if 
$g \ge 0$, then for $|t| \le \tau$ we have
\begin{align}
     (f * g)(s + t)
   &= \int_\R f(s + t - v) g(v) \, \d v
 \label{ConvoluStable}
 \\
 &\le C \int_\R f(s - v) g(v) \, \d v
  =  C \ms1 (f * g)(s),
 \notag
\end{align}
hence $f * g$ is also $\tau$-stable with constant $C$. Suppose that 
$f, g, h$ are nonnegative on $\R$, and that $f$ is $\tau$-stable with
constant $C$. If $|t| \le \tau$ then
\begin{align*}
     \int_\R f(s) g(s - t) h(t) \, \d s
 &\ge C^{-1} \int_\R f(s - t) g(s - t) h(t) \, \d s
 \\
  &=  C^{-1} h(t) \Bigl( \int_\R f(v) g(v) \, \d v \Bigr),
\end{align*}
therefore
\begin{equation}
     \int_\R f(s) (g * h)(s) \, \d s
 \ge C^{-1} \Bigl( \int_{|t| \le \tau} h(t) \, \d t \Bigr)
      \Bigl( \int_\R f(v) g(v) \, \d v \Bigr).
 \label{Ustable}
\end{equation}
\dumou

 We shall now move to $\R^\ell$ with $\ell \ge 1$. Let $\Phi$ be a
probability density on~$\R$ that is $\tau$-stable with constant $C$, for
some $\tau > 0$. This implies that $\Phi(s) > 0$ for every $s \in \R$. Let
us define $\beta \ge 1$ by
\begin{equation}
 \beta^{-1} = \int_{|t| \le \tau}  \Phi (t) \, \d t \in (0, 1).
 \label{DefBeta}
\end{equation}
We denote by~$\Phi_j$\label{PhiSubJ} 
the operator on $L^q(\R^\ell)$ of convolution with $\Phi$ in the variable
$x_j$, for each $j \in L = \{1, \ldots, \ell\}$. For instance, when $j = 1$
we let
\[
   (\Phi_1 f)(x_1, x_2, \ldots, x_\ell)
 = \int_\R f(x_1 - s, x_2, \ldots, x_\ell) \ms1 \Phi (s) \, \d s.
\]
For $j = 2, \ldots, \ell$ we let the transposition $\tau_j = (1 \ms2 j)$
act on $x = (x_1, \ldots, x_\ell)$ in $\R^\ell$ by
$\tau_j \ms2 x = (x_{\tau_j(i)})_{i=1}^\ell$ and on functions
by $\tau_j (g) = g \circ \tau_j$. Letting $\tau_1 = I$, we
have
\begin{equation}
 \Phi_j f = \tau_j \bigl( \Phi_1 (f \circ \tau_j) \bigr),
 \quad j = 1, \ldots, n.
 \label{LesPhij}
\end{equation}
For every subset $J \subset L$ we set
$
 \Phi^J = \prod_{k \in J} \Phi_k$, and
$   \Phi^{\sim j}
 = \Phi^{L \setminus \{j\}}
 = \prod_{k \ne j} \Phi_k
$.
We understand that $\Phi^\emptyset = I$. Each $\Phi^J$ is an
operator acting on $L^q(\R^\ell)$ with norm equal to $1$, when 
$1 \le q \le +\infty$. The next Bourgain's lemma is not too difficult, but
the details are long and painful to write down precisely. We have chosen to
break it into two parts, the first one containing the serious work.
\dumou

\def\CiteLsept{\cite[Lemma~7]{BourgainCube}}
\begin{lem}[a first part of Bourgain's~\CiteLsept]%
\label{ProduitDeLsA}
Let\/ $\Phi$ be a probability density on\/~$\R$ that is $\tau$-stable with
constant $C$, let $\beta \ge 1$ be defined by\/~\eqref{DefBeta}. Let~$\ell$
be an integer $\ge 1$, $L = \{1, \ldots, \ell\}$ and define\/ $\Phi_j$
by\/~\eqref{LesPhij}, for $j = 1, \ldots, \ell$. For all integers $q \ge 1$,
for all nonnegative integrable functions\/ $(f_j)_{j=1}^\ell$ on\/
$\R^\ell$, one has
\[
     \bigl\| \sum_{j \in L} \Phi^{\sim j} f_j \bigr\|_q
 \le \beta \ms1 C^{q-1} \ms1
      \bigl\| \sum_{j \in L} \Phi^L f_j \bigr\|_q
       + \sqrt{q - 1} \ms3 
          \bigl\|
           \sum_{j \in L} \Phi^{\sim j} f_j^2 
          \bigr\|_{q/2}^{1/2}.
\]
\end{lem}

\begin{proof}
The fundamental remark compares
\[
 \int_\R (\Phi_1 g_1)(s) \ms1 (\Phi_1 g_2)(s) \ms1 
    \ldots \ms1 (\Phi_1 g_{k-1}) (s) \ms1 g_k(s) \, \d s
\]
and
\[
 \int_\R (\Phi_1 g_1)(s) \ms1 (\Phi_1 g_2)(s) \ms1 
    \ldots \ms1 (\Phi_1 g_{k-1}) (s) \ms1 (\Phi_1 g_k)(s) \, \d s,
\]
when $k \ge 2$ and when the functions $g_j \ms1$s are nonnegative on~$\R$.
We know by~\eqref{ConvoluStable} that $\Phi_1 g = \Phi * g$ is
$\tau$-stable with constant~$C$ for every~$g$ nonnegative, so the product
$f = (\Phi_1 g_1)(\Phi_1 g_2) \ldots (\Phi_1 g_{k-1})$ is $\tau$-stable
with constant $C^{k-1}$. Applying~\eqref{Ustable} and the definition of
$\beta$ with $f$, $g = g_k$, $h = \Phi$ and $g * h = \Phi_1 g_k$, we~get
\begin{align}
     &\ms5 \int_\R (\Phi_1 g_1)(s) \ms1 (\Phi_1 g_2)(s) \ms1 
       \ldots \ms1 (\Phi_1 g_{k-1})(s) \ms1 g_k(s) \, \d s
 \label{Fundam}
 \\
 \le &\ms5 C^{k-1} \ms1 \beta \ms1
       \int_\R (\Phi_1 g_1)(s) \ms1 (\Phi_1 g_2)(s) \ms1 
        \ldots \ms1 (\Phi_1 g_{k-1})(s) \ms1 (\Phi_1 g_k)(s) \, \d s.
 \notag
\end{align}

 The case $q = 1$ of the lemma follows from $\beta \ge 1$ and
$\int_{\R^\ell} g * f = \int_{\R^\ell} f$ for every probability density $g$.
For the simplest non-trivial case, when $q = 2$, we write
\[
   \Bigl( \sum_{j \in L} \Phi^{\sim j} f_j \Bigr)^2
 = \sum_{i \ne j} (\Phi^{\sim i} f_i) (\Phi^{\sim j} f_j)
    + \sum_{j \in L} (\Phi^{\sim j} f_j)^2.
\]
When $j \ne i$, the function $\Phi^{\sim i} f_i
 = \Phi^{L \setminus \{ i\} } f_i
 = \Phi_j \Phi^{L \setminus \{ i, j\} } f_i$ 
is of the form $\Phi_j g_1$, and letting $g_2 = \Phi^{\sim j} f_j$ we get 
by~\eqref{Fundam} for the $x_j$ variable that
\begin{align*}
     & \ms6 \int_\R (\Phi^{\sim i} f_i) (\Phi^{\sim j} f_j) \, \d x_j
  =  \int_\R (\Phi_j g_1) \ms1 g_2 \, \d x_j
 \\
 \le & \ms6 C \ms1 \beta \int_\R (\Phi_j g_1) (\Phi_j g_2 ) \, \d x_j
  =   C \ms1 \beta \int_\R (\Phi^{\sim i} f_i) (\Phi^L f_j) \, \d x_j
\end{align*}
because $\Phi_j \Phi^{\sim j} = \Phi^L$. Integrating in the remaining
variables, and since the functions are nonnegative, we obtain
\begin{align*}
     & \ms6 \int_{\R^\ell}
      \sum_{i \ne j} (\Phi^{\sim i} f_i) (\Phi^{\sim j} f_j) \, \d x
 \le C \ms1 \beta \int_{\R^\ell} 
      \sum_{i \ne j} (\Phi^{\sim i} f_i) (\Phi^L f_j) \, \d x
 \\
 \le & \ms6 C \ms1 \beta \int_{\R^\ell}
       \Bigl(\sum_{i \in L} \Phi^{\sim i} f_i \Bigr)
        \Bigl(\sum_{j \in L} \Phi^L f_j \Bigr) \, \d x
 \le C \ms1 \beta \ms2
      \Bigl\| \sum_{j \in L} \Phi^{\sim j} f_j \Bigr\|_2
        \Bigl\| \sum_{j \in L} \Phi^L f_j \Bigr\|_2.
\end{align*}
When $j = i$, we use $(\Phi^{\sim j} * g)^2
 \le \Phi^{\sim j} * g^2$ and get
\[
     \int_{\R^\ell} \sum_{j \in L} (\Phi^{\sim j} f_j)^2 \, \d x
 \le \int_{\R^\ell} \sum_{j \in L} \Phi^{\sim j} f_j^2 \, \d x
  =: B.
\]
It follows that $E:= \bigl\| \sum_{j \in L} \Phi^{\sim j} f_j \bigr\|_2$
satisfies an inequality $E^2 \le A E + B$, where we let
$A := C \ms1 \beta \ms1 \bigl\| \sum_{j \in L} \Phi^L f_j \bigr\|_2$. This
yields $E \le A + B^{1/2}$ and we have
\[
     \Bigl\| \sum_{j \in L} \Phi^{\sim j} f_j \Bigr\|_2
 \le C \ms1 \beta \Bigl\| \sum_{j \in L} \Phi^L f_j \Bigr\|_2
      + \Bigl\| \sum_{j \in L} \Phi^{\sim j} f_j^2 \Bigr\|_1^{1/2}.
\]
This is Lemma~\ref{ProduitDeLsA} when $q = 2$. In general, when $q \ge 3$,
we expand
\begin{equation}
   \int_{\R^\ell} \bigl( \sum_{j \in L} \Phi^{\sim j} f_j \bigr)^q
    \, \d x
 = \sum_{j_1, j_2, \ldots, j_q \in L} 
    \int_{\R^\ell} 
     (\Phi^{\sim j_1} f_{j_1}) \ldots (\Phi^{\sim j_q} f_{j_q})
      \, \d x.
 \label{Expd}
\end{equation}
Consider a multi-index $(j_1, j_2, \ldots, j_q)
 \in \{ 1, \ldots, \ell \}^q = L^q$ and suppose that
$j_q$ is not equal to any of $j_1, \ldots, j_{q-1}$. Then we can write
$\Phi^{\sim j_k} f_{j_k} = \Phi_{j_q} g_k$ for each $k < q$, so as before,
by~\eqref{Fundam} applied in the $x_{j_q}$ variable, we get that
\[
     \int_{\R^\ell} (\Phi^{\sim j_1} f_{j_1}) \ldots 
                    (\Phi^{\sim j_q} f_{j_q}) \, \d x
 \le C^{q-1} \beta    
      \int_{\R^\ell} (\Phi^{\sim j_1} f_{j_1}) \ldots 
            (\Phi^{\sim j_{q-1}} f_{j_{q-1}}) (\Phi^L f_{j_q}) \, \d x.
\]
Let us denote by $\sum_1$ the part of the summation at the right-hand side
of~\eqref{Expd} that is extended to all $j_1, \ldots, j_q$ such that 
$j_q \notin \{j_1, \ldots, j_{q-1}\}$. We obtain that
\begin{align*}
      {\sum}_1
       \int_{\R^\ell} (\Phi^{\sim j_1} f_{j_1}) \ldots 
                    (\Phi^{\sim j_q} f_{j_q}) \, \d x
 &\le C^{q-1} \beta    
      \int_{\R^\ell} \Bigl( \sum_{j \in L} \Phi^{\sim j} f_j \Bigr)^{q-1}
            \Bigl( \sum_{j \in L} \Phi^L f_j \Bigr) \, \d x
 \\
 &\le C^{q-1} \beta \ms1
       \Bigl\| \sum_{j \in L} \Phi^{\sim j} f_j \Bigr\|_q^{q - 1}
       \Bigl\| \sum_{j \in L} \Phi^L f_j \Bigr\|_q.
\end{align*}
The remaining sum $\sum_2$ is less than the sum of $q - 1$ terms
corresponding to which index $j_k$, $k = 1, \ldots, q-1$ is equal to $j_q$.
Each of these $q - 1$ terms is similar to
\[
   \sum_{j_1, j_2, \ldots, j_{q-1} \in L} 
    \int_{\R^\ell} (\Phi^{\sim j_1} f_{j_1}) \ldots 
         (\Phi^{\sim j_{q-2}} f_{j_{q-2}})
         (\Phi^{\sim j_{q-1}} f_{j_{q-1}})^2 \, \d x,
\]
which is bounded by
\[
     \int_{\R^\ell} 
      \Bigl( \sum_{j \in L} \Phi^{\sim j} f_j \Bigl)^{q - 2}
      \Bigl( \sum_{j \in L} \Phi^{\sim j} f_j^2 \Bigr)
 \le \Bigl\| \sum_{j \in L} \Phi^{\sim j} f_j \Bigr\|_q^{q - 2}
     \Bigl\| \sum_{j \in L} \Phi^{\sim j} f_j^2 \Bigr\|_{q/2}.
\]
We obtain for $E_q = \Bigl\| \sum_{j \in L} \Phi^{\sim j} f_j \Bigr\|_q^q$
a bound by $\Sigma_1 + \Sigma_2$ of the form
\[
     E_q 
 \le C^{q-1} \beta \ms1
       \Bigl\| \sum_{j \in L} \Phi^L f_j \Bigr\|_q E_q^{1 - 1 / q}
      + (q - 1) 
         \Bigl\| \sum_{j \in L} \Phi^{\sim j} f_j^2 \Bigr\|_{q/2}
          E_q^{1 - 2 / q},        
\]
which can be written also as
\[
     E_q^{2/q} 
 \le C^{q-1} \beta \ms1
       \Bigl\| \sum_{j \in L} \Phi^L f_j \Bigr\|_q E_q^{1 / q}
      + (q - 1) \ms1
         \Bigl\| \sum_{j \in L} \Phi^{\sim j} f_j^2 \Bigr\|_{q/2}.
\]
This implies as before that
\[
     \Bigl\| \sum_{j \in L} \Phi^{\sim j} f_j \Bigr\|_q 
  =  E_q^{1/q}
 \le C^{q-1} \beta \ms1
      \Bigl\| \sum_{j \in L} \Phi^L f_j \Bigr\|_q
      + \sqrt{q - 1} \ms3
         \Bigl\| \sum_{j \in L} \Phi^{\sim j} f_j^2 \Bigr\|_{q/2}^{1/2}.
 \qedhere
\]
\end{proof}

\begin{lem}[Bourgain~\CiteLsept]%
\label{ProduitDeLs}
Let\/ $\Phi$ be a probability density on\/~$\R$ that is $\tau$-stable with
constant $C$, and let $\beta \ge 1$ be defined by\/~\eqref{DefBeta}. Let
$\ell \ge 1$ be an integer, $L = \{1, \ldots, \ell\}$ and define\/
$\Phi_j$ by\/~\eqref{LesPhij}, for $j = 1, \ldots, \ell$. For every integer
$\nu \ge 1$ and for all nonnegative integrable functions\/
$(f_j)_{j=1}^\ell$ on\/ $\R^\ell$, one has 
\begin{equation}
      \kappa_\nu^{-1} \ms2 
       \Bigl\| \sum_{j \in L} \Phi^{\sim j} f_j \Bigr\|_{2^\nu}
 \le \sum_{k=0}^\nu \ms2
      \bigl\|
       \sum_{j \in L} \Phi^L f_j^{2^k} 
      \bigr\|_{2^{\nu-k}}^{2^{-k}}
 \le (\nu + 1) \ms2 \Bigl\| \sum_{j \in L} f_j \Bigr\|_{2^\nu},
 \label{KappaNu}
\end{equation}
with $\kappa_\nu \le \max(2^\nu, \beta \ms1 C^{2^\nu})$. Each term\/
$\bigl\| \sum_{j \in L} \Phi^L f_j^{2^k} \bigr\|_{2^{\nu-k}}^{2^{-k}}$,
for\/ $0 \le k \le \nu$, satisfies
\[
     \bigl\| \sum_{j \in L} \Phi^L f_j^{2^k} \bigr\|_{2^{\nu-k}}^{2^{-k}}
  =  \Bigl\|
      \Bigl( \sum_{j \in L} \Phi^L f_j^{2^k} \Bigr)^{2^{-k}}
     \Bigr\|_{2^\nu}
 \le \Bigl\| \sum_{j \in L} f_j \Bigr\|_{2^\nu}.
\]
\end{lem}

\begin{proof}
We begin with the easy last sentence. For $r = 2^k$ and 
$k = 0, \ldots, \nu$, we have
\[
     \Bigl\| \sum_{j \in L} \Phi^L f_j^r \Bigr\|_{2^{\nu-k}}^{1/r}
 \le \Bigl\|
      \Phi^L \Bigl( \sum_{j \in L} f_j \Bigr)^r
     \Bigr\|_{2^{\nu-k}}^{1/r}
 \le \Bigl\|
      \Bigl( \sum_{j \in L} f_j \Bigr)^r
     \Bigr\|_{2^{\nu-k}}^{1/r}
  =  \Bigl\| \sum_{j \in L} f_j \Bigr\|_{2^\nu}.
\]
The constant in the right-hand inequality of~\eqref{KappaNu} is therefore
bounded by $\nu + 1$.

 We pass to the left-hand inequality. Let $q = 2^\nu$. By
Lemma~\ref{ProduitDeLsA}, we can reduce the case $q = 2^\nu$ to the case 
$q / 2$. We proceed by induction, with a number of steps bounded by~$\nu$.
Using $(a + b)^\alpha \le a^\alpha + b^\alpha$ when $a, b \ge 0$ and
$\alpha \in (0, 1]$, we obtain
\begin{align*}
     \bigl\| \sum_{j \in L} \Phi^{\sim j} f_j \bigr\|_q
 &\le \beta \ms1 C^{q - 1} \ms1
      \bigl\| \sum_{j \in L} \Phi^L f_j \bigr\|_{q}
       + \sqrt{q - 1} \ms3 
        \bigl\| \sum_{j \in L} \Phi^{\sim j} f_j^2 \bigr\|_{q / 2}^{1/2}
 \\
 &\le \beta \ms1 C^{2^\nu} \ms1
      \bigl\| \sum_{j \in L} \Phi^L f_j \bigr\|_{2^\nu}
       + 2^{\nu/2} \ms1 \beta^{1/2} \ms2 C^{2^{\nu-2}} \ms1\ms1 
        \bigl\| \sum_{j \in L} \Phi^L f_j^2 \bigr\|_{2^{\nu-1}}^{1/2}
 \\ & \ms{18}
       + 2^{\nu/2} \ms1 2^{(\nu-1)/4} \ms1 
        \bigl\| \sum_{j \in L} \Phi^{\sim j} f_j^4 \bigr\|_{2^{\nu-2}}^{1/4}
   \le \dots
\end{align*}
and the successive factors in front of
$\bigl\| \sum_{j \in L} \Phi^L f_j^{2^k} \bigr\|_{2^{\nu-k}}^{2^{-k}}$, for
$0 \le k \le \nu$, have the form 
$q \ms1 (\beta / q)^{2^{-k}} (C^q)^{4^{-k}}
 \le q (\beta C^q / q)^{2^{-k}}$,
leading to $\kappa_\nu \le \max(q, \beta \ms1 C^q)$.
\end{proof}

 We can try to optimize the constant $\kappa_\nu$ in the following way.
Suppose that the function $\ln \Phi$ is Lipschitz on $\R$ with constant
$\lambda$. Then we see that $\Phi$ is $\tau$-stable with constant 
$C_\tau = \e^{\lambda \tau}$ for every $\tau > 0$, and
\[
     1
 \ge \beta_\tau^{-1}
 :=  \int_{|t| \le \tau} \Phi(t) \, \d t
 \ge 2 \ms1 \Phi(0) \int_0^\tau \e^{ - \lambda t} \, \d t
  =  2 \ms1 \Phi(0) \frac {1 - \e^{- \lambda \tau}} \lambda \up.
\]
Let $q = 2^\nu$ and select $\tau = 1 / (\lambda q)$. Then 
$C_\tau \le \e^{1 / q}$ and 
\[
     \beta_\tau C_\tau^q
 \le \frac {\e \lambda} {2 \Phi(0) (1 - \e^{- 1 / q})} 
 \le \frac {\e^2 \lambda} {2 \ms1 \Phi(0)}
      \ms3 q
 \le \frac {4 \ms1 \lambda} {\Phi(0)}
      \ms3 q.
\]
Coming back to Lemma~\ref{ProduitDeLs} and noticing that
$\lambda \ge 2 \Phi(0)$, we obtain
\begin{equation}
     \kappa_\nu 
 \le \frac {4 \ms1 \lambda} {\Phi(0)}
      \ms3 2^\nu.
 \label{Improved}
\end{equation}
\dumou

 We now introduce Bourgain's specific example $\varphi$ of a 
function~$\Phi$, defined by
\[
 \forall s \in \R,
 \ms{16}
 \varphi(s) = \frac c {1 + s^4} \ms1 \up,
\]
where $c = \sqrt 2 / \pi$ is chosen so that $\varphi$ is a probability
density. This value $c$ is obtained by the residue theorem, which also gives
the Fourier transform
\[
   \widehat \varphi(t) 
 = \bigl( \cos(\pi \sqrt 2 \ms1 |t|) + \sin(\pi \sqrt 2 \ms1 |t|) \bigr)
    \e^{-\pi \sqrt 2 \ms1 |t|},
 \quad t \in \R.
\]
Notice that
$  (\cos u + \sin u ) \e^{- u}
 = \sqrt 2 \ms1 \cos (u - \pi / 4 ) \e^{- u} \ge \e^{- u^2}
$
when $0 \le u \le \pi / 2$, because 
$h(u) = \ln \bigl( \sqrt 2 \ms1 \cos (u - \pi / 4 ) \bigr)
         - u + u^2 \ge 0$ on this interval. Indeed, we have 
$h(0) = h'(0) = 0$ and $h''(u) = 1 - \tan(u - \pi / 4)^2 \ge 0$ on
$[0, \pi / 2]$. It follows that
\[
     \widehat \varphi(t)
 \ge \e^{ - 2 \pi^2 t^2}
 \ms{15} \hbox{when} \ms{15}
 \pi \ms1 \sqrt 2 \ms1 |t| \le \frac \pi 2 \up, 
\]
in particular when $\pi \ms1 |t| \le 1$. We shall need later the estimate
given by Lemma~\ref{EtaPhi}.

\begin{lem}\label{EtaPhi}
For all $s \in \R$, $\ell$ integer $\ge 1$ and 
$\xi = (\xi_1, \ldots, \xi_\ell) \in \R^\ell$, one has that
\[
     \Bigl(
      1 - \prod_{j=1}^\ell \widehat \varphi(s \ms1 \xi_j) 
     \Bigr) \ms2
      \prod_{j=1}^\ell \widehat \eta(\xi_j)
 \le 2 \ms1 \pi^2 s^2.
\]
\end{lem}

\begin{proof}
Suppose that $\pi \ms2  |s \ms1 \xi| \le 1$. Then
$\pi \ms2 |s \ms1 \xi_j| \le 1$
and $\widehat \varphi(s \ms1 \xi_j)
 \ge \e^{ - 2 \pi^2 s^2 \xi_j^2}$ for each index $j = 1, \ldots, \ell$,
thus by Lemma~\ref{FourBound2} we have 
\[
     \Bigl(
      1 - \prod_{j=1}^\ell \widehat \varphi(s \ms1 \xi_j) 
     \Bigr) \ms2
      \prod_{j=1}^\ell \widehat \eta(\xi_j)
 \le \bigl( 1 - \e^{ - 2 \pi^2 s^2 |\xi|^2} \bigr) \ms2
      \prod_{j=1}^\ell \widehat \eta(\xi_j)
 \le 2 \pi^2 s^2.
\]
Otherwise, we have $\pi \ms2 |s \ms1 \xi| \ge 1$ and applying
Lemma~\ref{FourBound2} with $r \rightarrow 0$ we get
\[
     \Bigl(
      1 - \prod_{j=1}^\ell \widehat \varphi(s \ms1 \xi_j) 
     \Bigr) \ms2
      \prod_{j=1}^\ell \widehat \eta(\xi_j)
 \le 2 \ms1 \prod_{j=1}^\ell \widehat \eta(\xi_j)
 \le 2 \ms1 |\xi|^{-2}
 \le 2 \pi^2 s^2.
\]
\end{proof}

 We know that $\varphi$ is $1$-stable, because $F(x) = \ln \varphi(x)$ is
Lipschitz. Indeed, its derivative
$
 F'(x) = - 4 x^3 / (1 + x^4)
$
is bounded on the real line. To be precise, the second derivative $F''$
vanishes when $x^4 = 3$, which implies that $|F'(x)| \le 3^{3/4}$ for
every $x$. When $|t| \le 1$, we have thus
\[
     \varphi(s + t) 
 \le \e^{3^{3/4}} \ms1 \varphi(s),
 \qquad s \in \R,
\]
with
$
 \e^{3^{3/4}} < 9,772 < 10
$.
This shows that $\varphi$ is $1$-stable with constant $\le 10$.
We shall need more than the $1$-stability of the function $\varphi$,
namely, we shall use the polynomial character of $1 / \varphi$. When 
$|t| \ge 1$ and $u \in \R$, we have
\begin{equation}
     1 + (u - t)^4
 \le 1 + 8 \ms1 (u^4 + t^4)
 \le 8 \ms1 (1 + t^4) \ms1 (1 + u^4)
 \le 16 \ms2 t^4 \ms1 (1 + u^4),
 \label{H4}
\end{equation}
implying in this case, and with $u = s + t$, that
$
 \varphi(s + t) \le 16 \ms2 t^4 \ms1 \varphi(s)
$.
\dumou

 We introduce $w_1 = w_0^2 = R^{-\delta} < w_0$.\label{Wsub1} 
The dilate $\Di \varphi {w_1}$ of $\varphi$ is $w_1$-stable with constant
$10$ and we shall consider from now on that $\Phi = \Di \varphi {w_1}$. We
denote as before by~$\Phi_j$ the operator on $L^q(\R^\ell)$ of convolution
with $\Di \varphi {w_1}$ in the variable $x_j$, where
$j \in L = \{1, \ldots, \ell\}$. For every subset $J \subset L$ we define
$\Phi^J$ as before, as well as $\Phi^{\sim j} = \Phi^{L \setminus \{j\}}$.
For $|t| \ge w_1$ we have by~\eqref{H4} the inequality 
\begin{equation}
     \Di \varphi {w_1}(s + t)
 \le 16 \ms1 (t / w_1)^4 \Di \varphi {w_1}(s)
  =  16 \ms1 R^{4 \delta} t^4 {\Di \varphi {w_1}} (s),
 \quad
 s \in \R.
 \label{H4b}
\end{equation}
\dumou

 Here is perhaps the crux of the matter. The boundary measures $\mu_j$,
partial derivatives of $\mu_Q$, will be swallowed and disappear as if by
magic. The cube $Q$ here is the cube $Q_\ell$ in~$\R^\ell$.

\def\CiteLhuit{\cite[Lemma~8]{BourgainCube}}
\begin{lem}[Bourgain~\CiteLhuit]\label{LHuit}
Let $\nu$ be an integer $\ge 1$, let $q = 2^\nu$, and let 
$f_1, \ldots, f_\ell$ be functions in~$L^q(\R^\ell)$. Let $\mu_j$ denote
the partial derivative $\partial_j \mu_Q$ of the probability measure
$\mu_Q$, for $j = 1, \ldots, \ell$. With\/ $\Phi_j$ defined as 
in~\eqref{LesPhij} from\/ $\Phi f = { \Di \varphi {w_1} } * f$ 
when $f \in L^q(\R)$, one has that
\[
     \Bigl\| 
      \Bigl(
       \sum_{j=1}^\ell \ms4 \bigl| \mu_j * \Phi^{\sim j} f_j \bigr|^2
      \Bigr)^{1/2}
     \Bigr\|_{L^q(\R^\ell)}
 \le \kappa \ms1 \sqrt{q \ln q} \ms2 R^{4 \delta} 
       \Bigl\|
        \Bigl( \sum_{j=1}^\ell |f_j|^2 \Bigr)^{1/2} 
       \Bigr\|_{L^q(\R^\ell)}.
\]
\end{lem}

\begin{proof}
Let us write $L = \{1, \ldots, \ell\}$ and $\R^L$ for $\R^\ell$. For each
$j \in L$, let $Q^{\sim j}$ denote the cube $\prod_{i \ne j} [ -1/2, 1/2]$
in $\R^{L \setminus \{j\}}$, let $\d \gr x^{\sim j}$ be the Lebesgue
measure on~$\R^{L \setminus \{j\}}$ and consider the probability measures
$\tau_j$, $K^{\sim j}$, $K^{ \{j\} }$ on $\R^\ell$ defined by
\begin{align*}
   \tau_j
 &= \frac 1 2 \bigl( \delta_{1/2}(x_j) + \delta_{-1/2}(x_j) \bigr)
         \otimes \Bigl( \otimes_{i \ne j} \ms2 \delta_0(x_i) \Bigr),
 \\
   K^{\sim j}
 &= \delta_0(x_j)
         \otimes \Bigl(
                  \otimes_{i \ne j} \gr 1_{[-1/2, 1/2]}(x_i) \, \d x_i
                 \Bigr)
 = \delta_0(x_j)
         \otimes \bigl(
                  \gr 1_{Q^{\sim j}} \, \d \gr x^{\sim j}
                 \bigr),
 \\
   K^{ \{j\} }
 &= \Bigl( \gr 1_{[-1/2, 1/2]}(x_j) \, \d x_j \Bigr)
         \otimes \Bigl(
                  \otimes_{i \ne j} \delta_0(x_i)
                 \Bigr).
\end{align*}
When convenient, we shall identify a kernel $K$ and the convolution
operator with that kernel. Note that the signed measure 
$\mu_j = \partial_j \gr 1_Q$ satisfies
\[
     |\partial_j \gr 1_Q|
  =  \bigl( \delta_{1/2}(x_j) + \delta_{-1/2}(x_j) \bigr) 
      \otimes \bigl( \gr 1_{Q^{\sim j}} \, \d \gr x^{\sim j} \bigr)
  =  2 \ms1 \tau_j * K^{\sim j}.
\]
Using $|\mu * f|^p \le \mu * |f|^p$ when $\mu$ is a
probability measure and $p \ge 1$, we have
\[
     \sum_{j=1}^\ell |\partial_j \gr 1_Q * \Phi^{\sim j} f_j|^2
 \le 4 \ms2 \sum_{j=1}^\ell 
       \Phi^{\sim j} (\tau_j * K^{\sim j} * |f_j|^2).
\]
We evaluate the $L^q$ norm applying Lemma~\ref{ProduitDeLs}, obtaining that
\[
     \Bigl\| 
      \Bigl(
       \sum_{j=1}^\ell \ms4 \bigl| \mu_j * \Phi^{\sim j} f_j \bigr|^2
      \Bigr)^{1/2}
     \Bigr\|_q^2
 \le 4 \ms1
      \Bigl\| 
       \sum_{j=1}^\ell \ms1 \Phi^{\sim j} 
        (\tau_j * K^{\sim j} *  |f_j|^2)
     \Bigr\|_{q/2}
 \le 4 \ms1 \kappa_{\nu-1} \sum_{k=0}^{\nu - 1} (E_k)^{2^{-k}},
\]
where the expressions $E_k$ are given by
\[
    E_k
 := \Bigl\| 
     \sum_{j=1}^\ell \Phi^L \ms1
      \bigl[
       \tau_j * K^{\sim j} * |f_j|^2
      \bigr]^{2^k}
    \Bigr\|_{q / 2^{k+1}},
 \quad
 1 \le 2^k \le q / 2 = 2^{\nu - 1}. 
\]
Using again $|\mu * f|^p \le \mu * |f|^p$ for $p \ge 1$, we get
\[
     E_k
 \le F_k
  := \bigl\| 
      \sum_{j=1}^\ell \Phi^L \ms1
       ( \tau_j * K^{\sim j} * |f_j|^{2^{k+1}} )
     \bigr\|_{q / 2^{k+1}}.
\]
Next, observe that
${\Di \varphi {w_1}} (s + t) \le w_1^{-4} {\Di \varphi {w_1}} (s)
 = R^{4 \delta} {\Di \varphi {w_1}} (s)$ for $|t| \le 1/2$. Indeed, when 
$w_1 \le |t| \le 1/2$ we have ${\Di \varphi {w_1}} (s + t)
 \le 16 \ms1 (t / w_1)^4 {\Di \varphi {w_1}} (s)
 \le w_1^{-4} {\Di \varphi {w_1}} (s)$ by~\eqref{H4b}, and
${\Di \varphi {w_1}} (s + t) \le 10 \ms2 {\Di \varphi {w_1}} (s)
 \le R^{4 \delta} {\Di \varphi {w_1}} (s)$ when $|t| \le w_1$, because we
assumed that $R^\delta \ge R_0^\delta = 16$. When $\mu$ is a probability
measure supported on $[-1/2, 1/2]$, it follows that
$\mu * { \Di \varphi {w_1} } \le R^{4 \delta} { \Di \varphi {w_1} }$ and 
${ \Di \varphi {w_1} } \le R^{4 \delta} \mu * { \Di \varphi {w_1} }$. We
have therefore $\tau_j \Phi_j \le R^{4 \delta} \Phi_j$ and
$\Phi_j \le R^{4 \delta} \Phi_j K^{ \{j\} }$. For $g$ nonnegative
we obtain
\[
     \Phi_j \tau_j \ms1 g
 \le R^{4 \delta} \Phi_j \ms1 g
 \le R^{8 \delta} \Phi_j K^{ \{j\} } \ms1 g.
\]
Consequently, observing that $K^{ \{j\} } * K^{\sim j}
 = \gr 1_Q(x) \, \d x$, we have
\[
     \Phi^L (\tau_j K^{\sim j} g)
  =  \Phi^{\sim j} \Phi_j \tau_j K^{\sim j} g
 \le R^{8 \delta} \Phi^{\sim j} \Phi_j K^{ \{j\} } K^{\sim j} g
  =  R^{8 \delta} \Phi^L * \gr 1_Q *g,
\]
and by the last assertion of Lemma~\ref{ProduitDeLs}, we obtain 
for $k = 0, \ldots, \nu - 1$ that
\begin{align*}
     F_k
 &\le R^{8 \delta}
      \Bigl\| 
       \Bigl( \sum_{j=1}^\ell \Phi^L |f_j|^{2^{k+1}} \Bigr)
        * \gr 1_Q
      \Bigr\|_{q / 2^{k+1}}
 \\
 &\le R^{8 \delta}
      \Bigl\| 
        \sum_{j=1}^\ell \Phi^L |f_j|^{2^{k+1}} 
      \Bigr\|_{q / 2^{k+1}}
 \le R^{8 \delta} 
      \Bigl\|
       \bigl( \sum_{j=1}^\ell |f_j|^2 \bigr)^{1/2} 
      \Bigr\|_q^{2^{k+1}}.
\end{align*}
Finally, assuming 
$\bigl\|
  \bigl( \sum_{j=1}^\ell |f_j|^2 \bigr)^{1/2} 
 \bigr\|_q \le 1$ we get
\[
     \Bigl\| 
      \Bigl(
       \sum_{j=1}^\ell \ms4 \bigl| \mu_j * \Phi^{\sim j} f_j \bigr|^2
      \Bigr)^{1/2}
     \Bigr\|_q^2
 \le 4 \ms1
      \kappa_{\nu-1} \sum_{k=0}^{\nu - 1} (R^{8 \delta})^{2^{-k}}
 \le 4 \ms1 \nu \kappa_{\nu - 1} R^{8 \delta}.
\]
Since $\ln \varphi$ is Lipschitz on $\R$, we can estimate 
$\kappa_\nu$ by~\eqref{Improved} and conclude.
\end{proof}

 Recalling that $G^R$ is a probability density and $\mu_j^R = \mu_j * G^R$,
we immediately deduce the result that we really need.

\def\CiteLneuf{\cite[Lemma~9]{BourgainCube}}
\begin{lem}[Bourgain~\CiteLneuf]%
\label{LNeuf}
Assume that $q = 2^\nu$, with $\nu \ge 1$ an integer. Let 
$f_1, \ldots, f_\ell$ be elements of~$L^q(\R^\ell)$. With\/ $\Phi_j$ as
in Lemma~\ref{LHuit}, we have
\[
     \Bigl\| 
      \Bigl(
       \sum_{j=1}^\ell \ms4 \bigl| \mu_j^R * \Phi^{\sim j} f_j \bigr|^2
      \Bigr)^{1/2}
     \Bigr\|_{L^q(\R^\ell)}
 \le \kappa_q \ms1 R^{4 \delta} \ms2
      \Bigl\|
       \Bigl( \sum_{j=1}^\ell |f_j|^2 \Bigr)^{1/2} 
      \Bigr\|_{L^q(\R^\ell)}.
\]
\end{lem}

\subsection{Conclusion%
\label{ConcluBourCube}}

\noindent
It remains to estimate the two terms
$\gr E (R, \ell, f) := | \nabla \mu^R * H_0 f |$ and~$\gr F (R, \ell, f)$
defined in~\eqref{2.29}, for $f \in L^{q_0}(\R^\ell)$, $q_0 = 2^\nu$ and
$\ell \le n$. Each one will be cut into two pieces, one of order a power of
$R^\delta$ and the second bounded by a \og small\fge multiple 
of~$B(q_0, R, n)$. Let us start with $\gr E (R, \ell, f)$, and cut 
$| \nabla \mu^R * H_0 f |$ into
\[
    \gr E'(R, \ell, f)
 := \ms2
     \bigr|
      \nabla \mu^R * {\Di G {w_1}} * H_0 f
     \bigr|, \ms8
    \gr E''(R, \ell, f)
 := \bigl|
     \nabla \mu^R * (\delta_0 - {\Di G {w_1}}) * H_0 f
    \bigr|.
\]

 We begin with $\gr E'(R, \ell, f)$. The mapping 
$f \mapsto \nabla \mu^R * {\Di G {w_1}} * H_0 f$, equal to
$ U_{ \mu^R * {\Di G {w_1}} } \circ H_0$ is studied by applying
Lemma~\ref{LemmeUK} to the log-concave probability density
$\mu^R * {\Di G {w_1}}$. Using~\eqref{VKt} and~\eqref{VdeG}, we see that
$
     V( \mu^R * {\Di G {w_1}} )
 \le V( {\Di G {w_1}} )
  =  w_1^{-1}
  =  R^\delta
$.
The variance of
$
   \mu^R * {\Di G {w_1}}
 = \mu_Q * {\Di G {1/R}} * {\Di G {w_1}}
$ is larger than that of $\mu_Q$, which is equal to~$(12)^{-1}$. By
Lemma~\ref{LemmeUK} and $q_0 \ge 2$, we get that
\[
     \bigl\| \gr E'(R, \ell, f) \bigr\|_{q_0}
 \le 24^{1/q_0}
      (R^\delta)^{1 - 2 / q_0} \ms1 \|H_0 f\|_{q_0}
 \le 5 \ms1 R^\delta \ms1 \|f\|_{q_0}.
\]
\dumou

 We study now $\gr E''(R, \ell, f)$ with the \textit{a priori\/} estimate
that involves the constant $B(q_0, R, \ell)$. By the
definition~\eqref{Objective}, one writes
\begin{align*}
     \bigl\| \gr E''(R, \ell, f) \bigr\|_{q_0}
  & = \bigl\|
       \nabla \mu^R * (\delta_0 - {\Di G {w_1}}) * H_0 f
      \bigr\|_{q_0}
 \\
 & \le B(q_0, R, \ell) \ms3 \|(\delta_0 - {\Di G {w_1}}) * H_0 f \|_{q_0}.
\end{align*}
We continue by interpolation $(L^\infty, \ms1 L^2)$ for
$f \mapsto (\delta_0 - {\Di G {w_1}}) * H_0 f$. In~$L^\infty(\R^\ell)$ one
has simply
$
     \| (\delta_0 - {\Di G {w_1}}) * H_0 \|_{\infty \rightarrow \infty} 
 \le 2
$
by using the $L^1$~norm of the convolution kernel. Lemma~\ref{FourBound2}
with $r = 2 \sqrt \pi \ms1 w_1 / w_0$ gives for the Fourier transform
a bound
\[
     (1 - \e^{ - 4 \pi w_1^2 |\xi|^2}) \ms2
      \prod_{j = 1}^\ell \widehat \eta(w_0 \xi_j)
 \le 4 \pi (w_1 / w_0)^2
  =  4 \pi w_0^2
  =  4 \pi R^{-\delta},
 \qquad \xi \in \R^\ell,
\] 
implying
$
     \| (\delta - {\Di G {w_1}}) * H_0 \|_{2 \rightarrow 2}
 \le 4 \pi R^{-\delta}
$.
We get in this way that
\[
     \| (\delta_0 - {\Di G {w_1}}) * H_0 \|_{q_0 \rightarrow q_0} 
 \le 2^{1 - 2/q_0} \ms1 (4 \pi R^{- \delta})^{ 2 / q_0 }
 \le 4 \pi \ms1 R^{- 2 \delta / q_0 },
\]
thus
$
     \bigl\| \gr E'' (R, \ell, f) \bigr\|_{q_0}
 \le \kappa B(q_0, R, \ell) R^{- 2 \delta / q_0} \ms2 \|f\|_{q_0}
$
and we obtain
\[
     \bigl\| \gr E (R, \ell, f) \bigr\|_{q_0}
 \le \kappa \bigl( R^\delta + B(q_0, R, \ell) R^{- 2 \delta / q_0}
            \bigr) \|f\|_{q_0}.
\]
\dumou

 Now we consider $\gr F (R, \ell, f)$ and we cut it into
\[
    \gr F' (R, \ell, f)
 := \Bigr(
      \sum_{j=1}^\ell \ms2 
       \bigl|
         \mu_j^R * \Gamma^j \Phi^{\sim j} f
       \bigr|^2 
    \Bigr)^{1/2} \ns9,
 \ms{10}
    \gr F'' (R, \ell, f)
 := \Bigr(
      \sum_{j=1}^\ell \ms2 
       \bigl|
         \mu_j^R * \Gamma^j (I - \Phi^{\sim j}) f 
       \bigr|^2 
    \Bigr)^{1/2} \ns6.
\]
By Lemma~\ref{LNeuf}, we have that
\[
     \bigl\| \gr F' (R, \ell, f) \bigr\|_{q_0}
 \le \kappa_{q_0} \ms1 R^{4 \delta} \ms2 
      \Bigl\| 
       \Bigl( \sum_{j=1}^\ell \ms2 
        \bigl| \Gamma^j f \bigr|^2 \Bigr)^{1/2}
      \Bigr\|_{q_0}.
\]
Using Khinchin's~\eqref{Hincin} and~\eqref{KinLp} we reduce to
$
 \bigl\| \sum_{j=1}^\ell \pm \Gamma^j f \bigr\|_{q_0}
$, and dividing according to the sign~$\pm$, we further reduce to
$
 \bigl\| \sum_{j \in J_1} \Gamma^j f \bigr\|_{q_0}
$
and
$
 \bigl\| \sum_{j \notin J_1} \Gamma^j f \bigr\|_{q_0}
$,
where $J_1 \subset \{1, \ldots, \ell\}$. The first sum corresponds to
the operator~$H_1$ relative to~$J_1$, the second is the one for 
$\sim J_1 := \{1, \ldots, \ell\} \setminus J_1$. By
Proposition~\ref{PisierSG} for the set $J_1$ of variables, writing 
$x = (\gr x^{J_1}, \gr x^{\sim J_1}) \in \R^\ell$, we have for 
$1 < q < +\infty$ that
\[
     \Bigl\| 
      \bigl( \sum_{j \in J_1} \Gamma^j f \bigr)_{\gr x^{\sim J_1}}
     \Bigr\|_{L^{q}(\R^{J_1})}
 \le h_q \ms1 
     \bigl\|
      f_{\down{\gr x^{\sim J_1}}} 
     \bigr\|_{ \down{L^{q}(\R^{J_1})} },
 \quad \gr x^{\sim J_1} \in \R^{\sim J_1}, 
\]
and integrating in the variables in $\sim J_1$ we get with $A_q$
from~\eqref{Hincin} that
\begin{equation}
     \Bigl\| 
      \Bigl( \sum_{j=1}^\ell \ms2 
       \bigl| \Gamma^j f \bigr|^2 \Bigr)^{1/2}
     \Bigr\|_q
 \le 2 \ms1 A_q^{-1} \ms1 h_q \ms1 \|f\|_q
 \label{SumGamJ}
\end{equation}
for $1 < q < +\infty$. It follows that
$
     \bigl\| \gr F' (R, \ell, f) \bigr\|_{q_0}
 \le \kappa'_{q_0} \ms1 R^{4 \delta} \ms1 \|f\|_{q_0}
$.
\dumou

 For the second term $\gr F'' (R, \ell, f)$ we first obtain an $L^2$~bound
for the nonlinear operator $V : f \mapsto
 \bigl(
  \sum_{j=1}^\ell |\Gamma^j (I - \Phi^{\sim j}) f|^2
 \bigl)^{1/2}$, by estimating the Fourier transform
\[
     \chi(\xi)
 :=  \sum_{j=1}^\ell 
      \bigl( 1 - \widehat \eta(w_0 \xi_j) \bigr)^2 \ms1
       \Bigl( \prod_{i \ne j} \widehat \eta(w_0 \xi_i) \Bigr)^2 \ms2
        \bigl( 1 - \prod_{i \ne j} \widehat \varphi(w_1 \xi_i) \bigr)^2
 \le 4 \pi^2  \ms1 R^{- \delta}.
\]
Indeed, we know that $0 \le \widehat \eta(t) \le 1$ and
$-1 \le \widehat \varphi(t) \le 1$, therefore
\begin{align*}
     \ms5 & \sum_{j=1}^\ell 
      \bigl( 1 - \widehat \eta(w_0 \xi_j) \bigr)^2 \ms1
       \Bigl( \prod_{i \ne j} \widehat \eta(w_0 \xi_i) \Bigr) \ms2
        \bigl( 1 - \prod_{i \ne j} \widehat \varphi(w_1 \xi_i) \bigr)
 \\
 \le \ms5 & 2 \ms1 \sum_{j=1}^\ell 
      \bigl( 1 - \widehat \eta(w_0 \xi_j) \bigr) \ms1
       \prod_{i \ne j} \widehat \eta(w_0 \xi_i) \ms2
 \le 2 \prod_{j=1}^\ell
      \bigl(
       ( 1 - \widehat \eta(w_0 \xi_j) )
        + \widehat \eta(w_0 \xi_j) 
      \bigr)
  =  2,
\end{align*}
and by Lemma~\ref{EtaPhi} applied to $\R^{L \setminus \{j\}}$ with 
$s = w_1 / w_0 = w_0 = R^{- \delta / 2}$, it follows that
\[
     \chi(\xi)
 \le 2 \max_{1 \le j \le \ell} 
      \Bigl( \prod_{i \ne j} \widehat \eta(w_0 \xi_i) \Bigr) \ms2
       \Bigl( 1 - \prod_{i \ne j} \widehat \varphi(w_1 \xi_i) \Bigr)
 \le 4 \pi^2 R^{-\delta},
 \qquad \xi \in \R^\ell.
\]
We get $\|V f\|_2^2 \le 4 \ms1 \pi^2 \ms1 R^{-\delta} \|f\|_2^2$ and
$\|V\|_{2 \rightarrow 2} \le 2 \ms1 \pi \ms1 R^{-\delta / 2}$. On the other
hand, given functions $(g_j)_{j=1}^\ell$ and independent Bernoulli random
variables $(\varepsilon_j)_{j=1}^\ell$, we have
\begin{align*}
  &   \ms{22} 2 \ms2
       \Bigl( \sum_{j=1}^\ell |\mu_j^R * g_j|^2 \Bigr)^{1/2}
   =  2 \ms2
       \Bigl( \sum_{j=1}^\ell |\mu_j^R * \varepsilon_j g_j|^2 \Bigr)^{1/2}
 \\
 &\le \Bigl(
       \sum_{j=1}^\ell 
        \bigl| \mu_j^R * 
         \bigl(
          \varepsilon_j g_j + \sum_{i \ne j} \varepsilon_i g_i
         \bigr) 
        \bigr|^2 
      \Bigr)^{1/2}
      +
      \Bigl(
       \sum_{j=1}^\ell 
        \bigl| \mu_j^R * 
         \bigl(
          \varepsilon_j g_j - \sum_{i \ne j} \varepsilon_i g_i
         \bigr) 
        \bigr|^2 
      \Bigr)^{1/2}
\end{align*}
hence with $g_j = \Gamma^j ( I - \Phi^{\sim j} ) f$ and
$F_\varepsilon = \sum_{i=1}^\ell \varepsilon_i \Gamma^i
 (I - \Phi^{\sim i}) f$ we see that
\[
     \bigl\| \gr F'' (R, \ell, f) \bigr\|_{q_0}
 \le \E_\varepsilon 
      \Bigl\| 
       \Bigl(
        \sum_{j=1}^\ell |\mu_j^R * F_\varepsilon|^2 
       \Bigr)^{1/2}
      \Bigr\|_{q_0} 
  =  \E_\varepsilon 
      \bigl\| 
       \nabla \mu^R * F_\varepsilon
      \bigr\|_{q_0} 
  =: D.
\]
With Khinchin~\eqref{KinLp} and the \textit{a priori}
bound~\eqref{Objective} we obtain
\[
      D
  \le B(q_0, R, \ell) \E_\varepsilon \|F_\varepsilon\|_{q_0}
  \le B_{q_0} \ms1 B(q_0, R, \ell) 
       \Bigl\| 
        \Bigl( 
         \sum_{i=1}^\ell
          |\Gamma^i ( I - \Phi^{\sim i} ) f|^2
        \Bigr)^{1/2}
       \Bigr\|_{q_0}.
\]
In $L^{q_1}(\R^\ell)$ with $q_1 = 2 \ms1 q_0 = 2^{\nu+1}$ we have
by~\eqref{SumGamJ} and Lemma~\ref{ProduitDeLs} that
\begin{align*}
    \ms5 & \Bigl\|
     \Bigl(
      \sum_{j=1}^\ell |\Gamma^j (I - \Phi^{\sim j}) f|^2
     \Bigl)^{1/2}
    \Bigl\|_{q_1}
 \\
 \le & \ms4 \Bigl\|
      \Bigl(
       \sum_{j=1}^\ell |\Gamma^j f|^2
      \Bigl)^{1/2}
     \Bigl\|_{q_1}
      +
     \Bigl\|
       \sum_{j=1}^\ell \Phi^{\sim j} \bigl( \Gamma^j f)^2
     \Bigl\|_{q_1/2}^{1/2}
 \le \kappa_{q_0} \|f\|_{q_1}.
\end{align*}
Interpolating with the $L^2$ bound, and with $\kappa_{q_0}$ changing from
line to line, we get
\[
    \Bigl\|
     \Bigl(
      \sum_{j=1}^\ell |\Gamma^j (I - \Phi^{\sim j}) f|^2
     \Bigl)^{1/2}
    \Bigl\|_{q_0}
 \le \kappa_{q_0} \ms1 R^{- \delta / (2q_0-2) } \ms1 \|f\|_{q_0}
 \le \kappa_{q_0} \ms1 R^{- \delta / (2 q_0) } \ms1 \|f\|_{q_0},
\]
therefore
$
     \bigl\| \gr F'' (R, \ell, f) \bigr\|_{q_0}
 \le \kappa_{q_0} \ms1 B(q_0, R, \ell) 
      \ms1 R^{- \delta / (2 q_0) } \ms1 \|f\|_{q_0}
$ and
\[
     \bigl\| \gr F(R, \ell, f)  \bigr\|_{q_0}
 \le \kappa_{q_0} 
      \bigl(
       R^{4 \delta} + B(q_0, R, \ell) R^{- \delta / (2 q_0) } 
      \bigr) \ms1 \|f\|_{q_0}.
\]
\dumou

 The estimates are proved in every dimension $\ell \le n$, we have thus
realized our objectives~\eqref{2.28} and~\eqref{2.29}. Noticing that 
$R \ge 1$, we have consequently
\[
     b_0(q_0, R, n) + b_1(q_0, R, n)
 \le \kappa_{q_0} 
      \bigl(
       R^{4 \delta}
        + B(q_0, R, n) \ms1 R^{- \delta / (2 q_0) } 
      \bigr).
\]
At last, we put all parts of~\eqref{BoundC} together. We may assume that
$4 \delta < 1$. We use again $R \ge 1$ in order to absorb the constant
bound from~\eqref{EstimeH}, thus obtaining
\[
     \bigl\| \nabla \mu^R * g \bigr\|_{q_0}
 \le c(q_0, \delta) \ms2 
      \bigl( R^{4 \delta}
       +  B(q_0, R, n) \ms1 R^{- \delta / (2 q_0) } 
      \bigr) 
       \ms1 \|g\|_{q_0},
\]
for $g \in L^{q_0}(\R^n)$ and $R \ge R_0$. Since $B(q_0, R, n)$ is the
maximum of $\bigl\| \nabla \mu^R * g \bigr\|_{q_0}$ for~$g$ of norm~$\le 1$
in $L^{q_0}(\R^n)$, we deduce that
$
     B(q_0, R, n)
 \le c(q_0, \delta) \ms2 R^{4 \delta} + B(q_0, R, n) / 2
$
for $R \ge R_1$, if $R_1 \ge R_0$ is such that
$c(q_0, \delta) \ms1 R_1^{- \delta / (2 q_0) } \le 1/2$,
thus
$
     B(q_0, R, n)
 \le 2 \ms1 c(q_0, \delta) \ms2 R^{4 \delta}
$ for $R \ge R_1$.
The value of $R_1$ depends on~$\delta$ and $q_0$ that are fixed. For 
$R \le R_1$, we may estimate directly 
$\| \nabla \mu^R * g \|_{q_0} \le \kappa \ms1 R \ms2 \|g\|_{q_0}
 \le \kappa \ms1 R_1^{1 - 4 \delta} R^{4 \delta} \|g\|_{q_0}$ by
Lemma~\ref{LemmeUK}. It follows finally that
$
     B(q_0, R, n)
 \le c'(q_0, \delta) \ms2 R^{4 \delta}
$, and
$\delta$ being arbitrarily small, we have proved Proposition~\ref{MainObj}.

\section{The Aldaz weak type result for cubes, and improvements%
\label{AlAu}}

\noindent
We work again in this section with the symmetric cube $Q_n$ of volume~1 
in~$\R^n$, that is to say, with $Q_1 = [-1/2, 1/2]$ when $n = 1$ and
$Q_n = (Q_1)^n$. We first present, following Aubrun~\cite{Aubrun}, a rather
soft argument proving the result of Aldaz~\cite{AldazWT} that the weak type 
$(1, 1)$ constant $\kappa_{Q, n}$ associated to the cubes $Q_n$ is not
bounded when~$n$ tends to infinity. We shall indicate and comment the
quantitative improvement obtained by Aubrun~\cite{Aubrun}, who gave a lower
bound $\kappa_{Q, n} \ge \kappa_\varepsilon (\log n)^{1 - \varepsilon}$ for
every~$\varepsilon > 0$. We then give a version of the proof of Iakovlev
and Str\"omberg~\cite{IS} who considerably improved this lower bound,
showing that $\kappa_{Q, n} \ge \kappa \ms1 n^{1/4}$. All the arguments
though are based on the same initial principle that we now recall.
\dumou

 We begin with a few simple reflections. If we want to contradict the
uniform boundedness of the weak type $(1, 1)$ constant $\kappa_{Q, n}$
we must, in view of Bourgain's Theorem~\ref{TheoCube}, look for functions
$f_n$ on~$\R^n$ that do not stay bounded, as $n \rightarrow \infty$, in
any~$L^p(\R^n)$ with $p > 1$. Also, we may easily obtain by mollifying
techniques that the weak type inequality for $L^1$~functions, stating that
\begin{equation}
     c \ms2 \bigl| \{x \in \R^n : (\M_Q f)(x) > c \} \bigr|
 \le \kappa_{Q, n} \ms2 \| f \|_{L^1(\R^n)},
 \quad c > 0, \ms6 f \in L^1(\R^n),
 \label{WTforCubes}
\end{equation}
where we let $\M_Q = \M_{Q_n}$, extends to bounded nonnegative
measures $\mu$ on~$\R^n$: if for every $x \in \R^n$ we define 
$(\M_Q \ms2 \mu)(x)$ to be the supremum over $r > 0$ of all quotients
$\mu(x + r Q) / |x + r Q|$, then~\eqref{WTforCubes} extends with the same
constant $\kappa_{Q, n}$ as
\[
     c \ms2 \bigl| \{x \in \R^n : (\M_Q \ms2 \mu)(x) > c \} \bigr|
 \le \kappa_{Q, n} \ms2 \mu(\R^n),
 \quad c > 0.
\] 
These two remarks lead naturally to consider measures on~$\R^n$ that are
sums of Dirac measures, in order to contradict the boundedness of
$\kappa_{Q, n}$ when $n \rightarrow \infty$.
\dumou
 
 Let 
$\mu_N = \sum_{j= 1}^N (\delta_{j-1/2} + \delta_{-j+1/2})$\label{MuN} 
stand for an \og approximation\fge of the Lebesgue measure $\lambda$ on a
large segment $S_N = [-N, N]$.\label{SN} 
The measure $\mu_N$ has a unit mass at the middle of each interval 
$(j, j+1)$, $j$ integer and $-N \le j < N$. Every interval $[u, u + h)$
contained in $S_N$, with length an integer $h > 0$, has the same measure
$h$ for $\mu_N$ or for~$\lambda$. However, if $I$ is a segment of length 
$1 + \alpha$, $0 < \alpha = 1 - \varepsilon < 1$, centered at $s = 0$ or at
any $s = j$, integer with $|j| < N$, then $I$ contains $j \pm 1/2$ and
\[
 \mu_N(I) = 2 
 \ms{16} \hbox{but} \ms{10} 
 \lambda(I) = 1 + \alpha = 2 - \varepsilon < 2,
\]
so that $(\M_Q \ms2 \mu_N)(s) \ge \mu_N(I) / \lambda(I)
 = 2 / (2 - \varepsilon)$.
The same observation is valid if~$s$ is not too far from an integer $j$ in
$(-N, N)$, precisely, if $|s - j| < \alpha / 2$. If we pass to $\R^n$ and
to the tensor product measure $\mu_N^{(n)} := \otimes^n \mu_N$, we obtain a
huge magnification due to dimension, which reads as
\[
     \bigl( \M_Q \ms2 \mu_N^{(n)} \bigr) (x) 
 \ge \Bigl( \frac 2 {2 - \varepsilon} \Bigr)^n
\]
when all coordinates $x_i$, $i = 1, \ldots, n$, of the point 
$x = (x_1, \ldots, x_n) \in \R^n$ belong to the subset $C_\alpha$ of 
$[-N, N]$ defined by
\begin{equation}
 C_\alpha = \bigcup_{-N < j < N} (j - \alpha/2, j + \alpha/2).
 \label{CalphaSet}
\end{equation}
\dumou

 If $\ell = 2 h + 1 \ge 1$ is an odd integer, if 
$J = (-h -1/2 - \alpha / 2,  h + 1/2 + \alpha / 2)
 = (\ell + \alpha) \ms1 Q$ and if $s + J$ is contained in $S_N$, we see in
the same way, when $s \in C_\alpha$, that the segment $s + J$ contains 
$\ell + 1 = 2 h + 2$ of the unit masses forming $\mu_N$. Consequently, we
have
$(\M_Q \ms2 \mu_N)(s)
 \ge (\ell + 1) / |s + J|
  =  (\ell + 1) / (\ell + \alpha)$ and
\[
     (\M_Q \ms2 \mu_N^{(n)})(x) 
 \ge \Bigl( \frac {\ell + 1} {\ell + 1 - \varepsilon} \Bigr)^n
\]
when $x = (x_1, \ldots, x_n)$ has all coordinates $x_i$ in $C_\alpha$ and
$x + J^n \subset S_N^n = 2 N Q_n$.
\dumou

 This case is much too particular, since the set of such points $x$
represents only a tiny proportion $\alpha^n$ of the cube $S_N^n$. One
has actually to consider that \emph{some} coordinates $x_i$ of
$x = (x_1, \ldots, x_n) \in \R^n$ are in $C_\alpha$, say $m \le n$ of
them. For the other coordinates $x_i$, observe that any interval of length
$\ell + \alpha$ contained in $S_N$ contains at least $\ell$ points of the
support of $\mu_N$. Assuming that $x + (\ell + \alpha) Q \subset S_N^n$, we
get for this point~$x$ with $m$ coordinates in~$C_\alpha$ the lower bound
\begin{equation}
     (\M_Q \ms2 \mu_N^{(n)})(x) 
 \ge \frac {\mu_N^{(n)} \bigl( x + (\ell + \alpha) Q \bigr)}
           { \bigl| x + (\ell + \alpha) Q \bigr| }
 \ge \Bigl( \frac {\ell + 1} {\ell + \alpha} \Bigr)^m
     \Bigl( \frac {\ell    } {\ell + \alpha} \Bigr)^{n - m}.
 \label{FirstQuot}
\end{equation}
We want the cardinality $m$ of the \og good\fg, \og centered\fge
coordinates~$x_i$ to be as big as possible. Since they are chosen out of
subsets of length $\alpha$ in unit intervals $(j-1/2, j+1/2)$, it is likely
that the proportion of \og good coordinates\fge among $n$ coordinates be
around $\alpha$, with a plausible deviation of order $\sqrt n$ from the
expected number~$\alpha \ms1 n$. We shall thus think henceforth that 
$m = \alpha \ms1 n + \delta \sqrt n$ for some $\delta > 0$.
\dumou

 We try to make the lower bound~\eqref{FirstQuot} as large as possible, by
a suitable choice of $\ell$. Setting $\beta = 1 - \alpha$, we rewrite the
right-hand side of~\eqref{FirstQuot} under the form
\[
   \Bigl( \frac {\ell + \alpha + \beta }  {\ell + \alpha} \Bigr)^m
   \Bigl( \frac {\ell + \alpha - \alpha } {\ell + \alpha} \Bigr)^{n - m}
 = \Bigl( 1 + \frac \beta {\ell + \alpha} \Bigr)^m
   \Bigl( 1 - \frac \alpha {\ell + \alpha} \Bigr)^{n - m} \ns9.
\]
Considering now $y = (\ell + \alpha)^{-1}$ as a real parameter, we will
study 
\[
 V(y) := (1 + \beta y)^m (1 - \alpha y)^{n - m},
 \quad 
     -1 / \beta 
 \le y 
 \le 1 / \alpha,
\] 
\def\ovy{\overline{y \ns{0.4}
 \raise0.7pt\hbox to0pt{\vphantom {$x$}}}\ms{1.3}}%
\def\ovys{\overline{y \ns{0.2}
 \raise0.45pt\hbox to0pt{\vphantom 
 {$\scriptstyle y$}}}\ms{1.1} }%
and find the maximal value $V(\ovy)$. Equivalently, we let $f$ denote the
fraction $m / n$ of coordinates of $x$ that are in $C_\alpha$, and we
maximize $v_{f, \alpha}(s) = V(s)^{1/n}$ defined by
\[
 v_{f, \alpha}(s) = (1 + \beta s)^f (1 - \alpha s)^{1 - f},
 \quad
 s \in [-1/\beta, 1 / \alpha].
\]
We have to remember though that the lower bound $V(y)$ for 
$\M_Q \mu_N^{(n)}(x)$ given in~\eqref{FirstQuot} is only valid when 
$1 / y - \alpha$ is an odd integer~$\ell$. We shall replace $\ovy$ by a
\def\yN{y_{\ms{0.60} \N}}%
value $y = \yN > 0$ close to $\ovy$, such that $1 / \yN - \alpha$ is an
odd integer, thus obtaining that $\M_Q \mu_N^{(n)}(x) \ge V(\yN)$. We must
ensure that the value of $V(y)$ does not decrease too much when moving from
$\ovy$ to~$\yN$. We would like to have
\begin{equation} 
     V(\yN)
 \ge \e^{-c} V(\ovy)
 \ms{12} \hbox{or} \ms{12}
     v_{f, \alpha}(\yN)
 \ge \e^{-c/n} v_{f, \alpha}(\ovy),
 \ms{14} \hbox{for some} \ms{10}
 c \ge 0.
 \label{DefAllow}
\end{equation} 
The maximal argument $\ovy$ is produced from $f$ and a choice of 
$\alpha < f$. We shall say that the couple $(f, \alpha)$ is 
$c \ms2$-\emph{allowable}\label{AlloCou} 
if the above condition~\eqref{DefAllow} is
satisfied.

\begin{lem}\label{YandAllow}
Let\/ $0 < \alpha < f < 1$, $\sigma_\alpha^2 = \alpha (1 - \alpha)$ and
let us define $\tau > 0$ by writing $f = \alpha + \sigma_\alpha \tau$.
The function $v_{f, \alpha}$ reaches its maximum at
\begin{equation}
   \ovy
 = \ovy_{f, \alpha}
 = \frac \tau { \sigma_\alpha \ns5 } \ms2
 = \frac { f - \alpha } { \sigma_\alpha^2 } 
 > 0.
 \label{FirstEsti}
\end{equation}
If\/ $0 < y, \ovy \le 1/2$ then
\begin{equation}
     \e^{(y - \ovys)^2 / 2} v_{f, \alpha}(y)
 \ge v_{f, \alpha}(\ovy)
  =  \Bigl( \frac f \alpha \Bigr)^f
     \Bigl( \frac { 1 - f } { 1 - \alpha } \Bigr)^{1-f} \ns3.
 \label{FirstAllow}
\end{equation}
If\/ $0 < \ovy \le 1/4$ and $\ovy^4 \le c / n$, then the couple\/
$(f, \alpha)$ is $c$-allowable.
\end{lem}

\begin{proof}
Let $w(s)
 = \ln v_{f, \alpha}(s)
 = f \ln (1 + \beta s) + (1 - f) \ln ( 1 - \alpha s)$. 
We have
\[
   w'(s) 
 = \frac {\beta f} {1 + \beta s} 
    - \frac {\alpha(1 - f)} {1 - \alpha s} \up,
 \ms{16}
   w''(s) 
 = - \frac {\beta^2 f} { (1 + \beta s)^2 } 
    - \frac {\alpha^2(1 - f)} { (1 - \alpha s)^2 } \up.
\]
The maximal argument $\ovy$ is found by solving $w'(\ovy) = 0$, yielding
\[
   \ovy 
 = \frac {f - \alpha} {\sigma_\alpha^2 } \up,
 \ms{12}
   1 + \beta \ovy 
 = \frac f \alpha \up,
 \ms{12}
   1 - \alpha \ovy 
 = \frac { 1 - f } { 1 - \alpha } \up.
\]
This gives us the maximal value $v_{f, \alpha}(\ovy)$ at the right-hand side
of~\eqref{FirstAllow}. Suppose now that we have $0 < y, \ovy \le 1/2$.
Using Taylor--Lagrange at $\ovy$, we get
\[
 w(y) - w(\ovy) = w''(\xi) \frac { (y - \ovy)^2 \ns 7} 2 \ms1 \up,
\]
for some $\xi$ between $y$ and $\ovy$, hence $0 < \xi \le 1/2$.
We have $1 - \alpha \xi \ge 1/2$ and
\[
     - w''(\xi) 
 \le \beta^2 f + 4 \alpha^2 (1 - f)
 \le \beta + 4 \alpha^2 (1 - \alpha)
 \le \beta + \alpha
  =  1,
\]
because $\alpha (1 - \alpha) \le 1/4$. This implies the left-hand side
of~\eqref{FirstAllow}.
\dumou

 Suppose that $0 < \ovy \le 1/4$. Moving around $\ovy$, we can find 
$\yN > 0$ satisfying
\[
     \frac { |\yN - \ovy| } { \yN \ms1 \ovy }
  =  \Bigl| \frac 1 \ovy - \frac 1 { \yN } \Bigr| 
 \le 1
\]
and such that $1 / \yN - \alpha$ is an odd integer. From 
$|\yN - \ovy| \le \yN \ms2 \ovy$ and $\ovy \le 1/4$ follows that
$\yN \le 4 \ms1 \ovy / 3 \le 1/3 < 1/2$. Also,
$|\yN - \ovy| \le 4 \ovy^2 / 3 < \sqrt 2 \ms2 \ovy^2$.
By~\eqref{FirstAllow}, we deduce that
$
     v_{f, \alpha}(\yN) 
 \ge \e^{ - \ovys^4 } v_{f, \alpha}(\ovy)
$
and the conclusion is reached.
\end{proof}

 Given $f$ and $\alpha$ such that $0 < \alpha < f < 1$, let us now examine
the optimal value
\begin{equation}
    E_{f, \alpha} 
 := v_{f, \alpha}(\ovy)
  = (1 + \beta \ms1 \ovy)^f (1 - \alpha \ms1 \ovy)^{1 - f}
  = \Bigl( \frac f \alpha \Bigr)^f
    \Bigl( \frac {1 - f} {1 - \alpha} \Bigr)^{1 - f} \ns{ 8}.
 \label{EAlpha}
\end{equation}
Consider the function $\phi_\alpha$ defined on $(0, 1)$ by
\begin{equation}
   \phi_\alpha(s) 
 = s \ln \Bigl( \frac s \alpha \Bigr)
    + (1 - s) \ln \Bigl( \frac {1 - s} {1 - \alpha} \Bigr),
 \quad s \in (0, 1).
 \label{DefPhiAlpha}
\end{equation}
We see that $\phi'_\alpha(s) = \ln (s / \alpha)
    - \ln \bigl( (1 - s) / (1 - \alpha) \bigr)$,
$\phi''_\alpha(s) = 1 / s + 1 / (1 - s) = 1 / \bigl( s (1-s) \bigr)$, and 
$\phi^{(3)}_\alpha(s) = - s^{-2} + (1 - s)^{-2}$. Note that
$\phi_\alpha(\alpha) = \phi'_\alpha(\alpha) = 0$, and that
$\phi''_\alpha(\alpha) = \sigma_\alpha^{-2}$.

\begin{lem}\label{MaxiVal}
If\/ $0 < \alpha < f = \alpha + \sigma_\alpha \tau < 1$, the maximal value
$v_{f, \alpha}(\ovy)$ satisfies
\begin{equation}
     \ln v_{f, \alpha}(\ovy)
  =  \phi_\alpha (f)
 \ge \frac {\tau^2 \ns 7} 2
      - \frac { 1 - 2 \alpha } { \sigma_\alpha } \ms3
         \frac {\tau^3 \ns 7} 6 \up.
 \label{TayLagB}
\end{equation}
\end{lem}

\begin{proof}
By Taylor--Lagrange for $\phi_\alpha$ at the point $\alpha$, we have
\[
   \phi_\alpha (f)
 = \phi''_\alpha(\alpha) \frac { (f - \alpha)^2 \ns8 } 2 
    + \phi^{(3)}_\alpha(\xi) \frac { (f - \alpha)^3 \ns 8} 6 
 = \frac {\tau^2 \ns 7} 2
    + \phi^{(3)}_\alpha(\xi) \frac { (\sigma_\alpha \tau)^3 \ns 8} 6
\]
for some $\xi \in (\alpha, f)$. Since $\phi_\alpha^{(3)}$ is increasing, we
get that
\[
     \phi_\alpha (f) - \frac {\tau^2 \ns 7} 2
 \ge \phi^{(3)}_\alpha(\alpha) \frac { (\sigma_\alpha \tau)^3 \ns 8} 6
  =  \frac { 2 \alpha  - 1 } 
           { \alpha^2 (1 - \alpha)^2 } \ms2
      \frac { \sigma_\alpha^3 \tau^3 \ns 8} 6
  =  \frac { 2 \alpha - 1 } 
           { \sigma_\alpha } \ms2
      \frac { \tau^3 \ns 8} 6 \up.
\]
\end{proof}

 In all that follows, we see $\Omega_1 := [-N, N]$ as a probability space
\label{OmegaUn}%
equipped with the uniform probability measure, denoted here by $P_1$, and
we shall consider the cube $S_N^n = 2 N Q_n$, equipped with the product
measure $P = P_1^{\otimes n}$, also the uniform probability measure, as
being our main probability space $(\Omega, \ca F, P)$. On this space, the
random variables $(\gr 1_{C_\alpha}(x_i))_{i=1}^n$, where 
$x = (x_1, \ldots, x_n) \in \Omega$,  are independent and equal to~$0$ 
or~$1$ with respective probabilities $1 - \alpha$ and $\alpha$. Their
expectation is $\alpha$ and their variance is equal to 
$\sigma_\alpha^2 = \alpha (1 - \alpha) \le 1/4$. For every 
$\alpha \in (0, 1)$, we introduce the centered and variance~$1$ Bernoulli
variable $X_{1, \alpha}$ defined on $\Omega_1$ by
\begin{equation}
   X_{1, \alpha}
 = \frac { \gr 1_{C_\alpha} - \alpha } {\sigma_\alpha}
 = \sqrt{ \frac { 1 - \alpha } \alpha } \ms2 \gr 1_{C_\alpha}
    - \sqrt{ \frac \alpha { 1 - \alpha } } \ms2 
       \gr 1_{\Omega_1 \setminus C_{\alpha}},
 \label{Xalpha}
\end{equation}
and we let\label{XnAlpha}
\[
   X_{n, \alpha}(x)
 = \frac {\sum_{i=1}^n X_{1, \alpha} (x_i) }
         { \sqrt n }
 = \sum_{i=1}^n \frac {\gr 1_{C_\alpha}(x_i) - \alpha} 
                      {\sigma_\alpha \sqrt n} \up,
 \quad \ms4
 x = (x_1, \ldots, x_n) \in \Omega.
\]
We also let\label{NAlpha} 
$N_{n, \alpha}(x) = \sum_{i=1}^n \gr 1_{C_\alpha}(x_i)$ denote the number
of coordinates of $x$ that are in $C_\alpha$. We are ready for a first
explicit estimate of the maximal function $\M_Q \mu_N^{(n)}$.

\begin{lem}\label{OneSet}
Let\/ $0 < \alpha < 1$ and $\sigma_\alpha^2 = \alpha (1 - \alpha)$.
Let $n \in \N^*$, $t > 0$ and\/ $0 < \theta < 1$ be such that
$\sqrt n \ge 2 \ms1 t \sigma_\alpha^{-2} (1 - \theta)^{-1}$. We have\/
$
 \M_Q \ms2 \mu_N^{(n)} > \e^{\theta t^2 / 2}
$
on the set
\[
   A_{\alpha, t}^{(n)}
 = \Bigl\{ x \in 2 (N - t^{-1} \sqrt n ) \ms1 Q_n :
    N_{n, \alpha} (x) = \sum_{i=1}^n \gr 1_{C_\alpha} (x_i)
      > \alpha n + t \sigma_\alpha \sqrt n 
   \Bigr\},
\]
where $C_\alpha$ is defined at\/~\eqref{CalphaSet}. When the dimension
$n$ is large, and assuming the size $N$ large enough compared to $n$,  it
follows that
\[
     \frac 
      {\bigl|
       \bigl\{ \M_Q \ms2 \mu_N^{(n)} > \e^{\theta t^2 / 2} \bigr\} 
      \bigr|} { |S_N^n| }
 \ge \frac {| A_{\alpha, t}^{(n)} |} { |2 N Q_n| }
  >  \frac 1 2 \ms2 \gamma_1 \bigl( (t, +\infty) \bigr).
\]
\end{lem}

\begin{proof}
By the central limit theorem (see~\cite{DurrettPTE} for instance), we know
that the distribution of $X_{n, \alpha}$ tends to the distribution of a
$N(0, 1)$ Gaussian random variable~$G$ when $n$ tends to infinity. This
yields
\[
        P \bigl(
           N_{n, \alpha} > \alpha \ms1 n + t \sigma_\alpha \sqrt n 
          \bigr)
  =     P \bigl( X_{n, \alpha} > t \bigr)
 \lra_n P(G > t)
  =     \gamma_1 \bigl( (t, +\infty) \bigr).
\] 
Let $A^{(n, 0)}_{\alpha, t}$ be the set of points $x \in \Omega$ where
$N_{n, \alpha}(x) > \alpha \ms1 n + t \sigma_\alpha \sqrt n$. Fix 
$x \in A^{(n, 0)}_{\alpha, t}$ and let $m = N_{n, \alpha}(x)$. 
We shall apply Lemma~\ref{YandAllow} with $f = m/n$ and
$\tau = t / \sqrt n$. By assumption, the optimal argument $\ovy$ satisfies
\[
     \ovy 
  =  \frac t { \sigma_\alpha \sqrt n }
 \le \frac {\sigma_\alpha ( 1 - \theta ) } 2
  <  1 / 4.
\]
At~\eqref{FirstQuot}, we used a cube centered at $x$, with side length
$\ell + \alpha$, $\ell$ an odd integer. We can choose 
$\ell + \alpha < 1 / \ovy + 2 < 2 / \ovy < t^{-1} \sqrt n$. This cube must
be contained in $\Omega = S_N^n$, so we have to give up a small part of
$A^{(n, 0)}_{\alpha, t}$, close to the boundary of~$\Omega$.
We thus introduce the subset $A_{\alpha, t}^{(n)}
 = A^{(n, 0)}_{\alpha, t} \cap 2 (N - t^{-1} \sqrt n ) \ms1 Q_n$.
The difference $A^{(n, 0)}_{\alpha, t} \setminus A_{\alpha, t}^{(n)}$
gets negligible when the side $2N$ of $S_N^n$ tends to infinity since
$
     (1 - t^{-1} \sqrt n / N)^n
 \rightarrow_N 1
$, so the set $A_{\alpha, t}^{(n)}$ has essentially the same probability 
as~$A_{\alpha, t}^{(n, 0)}$ when $N = N(n) > \kappa(t) n^{3/2}$ is large
enough. When $n$ tends to infinity, the probability of 
$A_{\alpha, t}^{(n)}$ is therefore, say, larger than 
$\gamma_1((t, + \infty)) / 2$.
\dumou

 We first show that the couple $(f, \alpha)$ is $c \ms1$-allowable with 
$c = (1 - \theta) t^2 / 4$. We know that $\ovy < 1/4$ and on the other
hand, we have
\[
     \ovy^4
  =  \frac { t^4 } { \sigma_\alpha^4 n^2 }
  =  \frac c n \ms2 
      \frac {4 t^2 } { (1 - \theta) \ms1 \sigma_\alpha^4 n }
  <  \frac c n \ms2 
      \Bigl(
       \frac {2 t } { (1 - \theta) \ms1 \sigma_\alpha^2 \sqrt n } 
      \Bigr)^2
 \le \frac c n \ms2 \up.
\]
It follows from Lemma~\ref{YandAllow} that $\M_Q \mu_N^{(n)} (x)
 \ge \e^{ - (1 - \theta) t^2 / 4 } V(\ovy)$ for every
$x \in A_{\alpha, t}^{(n)}$. It remains to estimate the optimal value
$V(\ovy)$. For this we apply~\eqref{TayLagB}. It implies that
$V(\ovy) \ge \e^{t^2 / 2}$ when $\alpha \ge 1/2$, and when 
$\alpha \le 1/2$, we see that
\[
   \frac { 1 - 2 \alpha } { \sigma_\alpha } \ms3
    \frac {\tau^3 \ns 7} 6
 = \frac { (1 - 2 \alpha) \tau } { 3 \ms1 \sigma_\alpha } \ms3
    \frac {\tau^2 \ns 7} 2
 < \frac t { 3 \ms1 \sigma_\alpha \sqrt n } \ms3
    \frac {\tau^2 \ns 7} 2
 < \frac { \sigma_\alpha (1 - \theta) } 6 \ms3
    \frac {\tau^2 \ns 7} 2
 < \frac { (1 - \theta) \tau^2 \ns 7} 4 \up,
\]
so that $V(\ovy) \ge \e^{t^2 / 2 - (1 - \theta) t^2 / 4}$
and $\M_Q \mu_N^{(n)} (x)
 \ge \e^{t^2 / 2} \e^{ - (1 - \theta) t^2 / 2 }
  =  \e^{ \theta \ms1 t^2 / 2}$.
\end{proof}
\dumou

 Given $\alpha \in (0, 1)$, we have identified a subset
$A_{\alpha, t}^{(n)}$ of $S_N^n$ where $\M_Q \mu_N^{(n)}$ is large. 
We shall have to use several values of $\alpha$, and show that the union
of the corresponding sets provides a fair amount of the total volume
of~$S_N^n$. We thus introduce
$0 < \alpha_0 < \alpha_2 < \ldots < \alpha_K < 1$ and we will prove that the
probability of the union of sets $(A_{\alpha_j, t}^{(n)})_{j=0}^K$ gets
$> 1/4$, say, when~$K$ is large but fixed and when~$n$ tends to infinity.
Rather than relying, as Aubrun does, on the \emph{law of iterated
logarithm}, we apply easy facts behind the proof of that \og law\fg. In a
simple qualitative approach, we shall analyze the Gaussian limit of the
joint distribution of $(X_{n, \alpha_j})_{j=0}^K$, which is the distribution
of a Gaussian vector $(G_j)_{j=0}^K$ whose covariance matrix $C$ is the
same as that of $(X_{n, \alpha_j})_{j=0}^K$. Letting 
$\sigma_j^2 = \alpha_j (1 - \alpha_j)$, the entries of $C$ are
\[
   C_{j, k}
 = \E ( X_{1, \alpha_j} X_{1, \alpha_k} )
 = \sigma_j^{-1} \sigma_k^{-1} 
   ( \alpha_j \wedge \alpha_k - 2 \alpha_j \alpha_k + \alpha_j \alpha_k),
 \quad 0 \le j, k \le K.
\]
Note that $C_{j, j} = 1$. Assuming $\alpha_j \le \alpha_k$, that is to say,
assuming $j \le k$, we get
\[
   C_{j, k} 
 = \sigma_j^{-1} \sigma_k^{-1} \alpha_j (1 - \alpha_k)
 = \sqrt { \frac {\alpha_j} {1 - \alpha_j} }
   \sqrt { \frac {1 - \alpha_k} {\alpha_k} } \up.
\]
We fix $v \in (0, 1)$ and set $w = \sqrt{1 - v^2}$. We define 
$\alpha_j = (1 + v^{2 j})^{-1}$, $j = 0, \ldots, K$, and obtain
$C_{j, k} = v^{ |k - j|}$. We can realize the distribution of
$(G_j)_{j=0}^K$ by considering the larger Gaussian sequence indexed by $\Z$,
which is defined by the sums of the series
$
 G_j = w \ms2 \sum_{i \le j} v^{j - i} \ms2 U_i$, for every
$j \in \Z$, where the $(U_i)_{i \in \Z}$ are independent $N(0, 1)$ Gaussian
variables. Indeed, if $j \le k$ we have that
\[
   \E ( G_j \ms1 G_k )
 = (1 - v^2) \sum_{i \le j} v^{j + k - 2 i}
 = v^{k - j}
 = v^{ |k - j| }.
\]
We see that $G_j - v \ms1 G_{j-1} = w \ms2 U_j$ and it follows that
\begin{equation}
     \max_{1 \le j \le J} |U_j| 
  =  w^{-1} \ms1 \max_{1 \le j \le J} |G_j - v \ms1 G_{j-1}|
 \le w^{-1} \ms1 (1+v) \ms1 \max_{0 \le j \le J} |G_j|.
 \label{UG}
\end{equation}

 We now recall an extremely classical estimate.
 
\begin{lem}\label{SupGaussi}
Let $J \ge 21$ be an integer and set
$s_J := \sqrt {2 \ln J - \ln (16 \pi \ln J)}$. If $U_1, \ldots, U_J$ are
independent $N(0, 1)$ Gaussian variables, one has that
\[
 P \bigl( \max_{1 \le j \le J} U_j > s_J \bigr) > 1/2.
\] 
\end{lem}

\begin{proof}
We have for $s > 0$ that
\begin{equation}
   \int_s^{+\infty} \, \d \gamma_1(s)
 > \frac s {\sqrt{2 \pi} (1 + s^2)} \e^{- s^2 / 2},
 \label{TailG}
\end{equation}
consequence of
\[
   \e^{-s^2 / 2} / s
 = \int_s^{+\infty} (1 + u^{-2}) \e^{-u^2/2} \, \d u
 < (1 + s^{-2}) \int_s^{+\infty} \e^{-u^2/2} \, \d u.
\]
When $J \ge 21$, one has $\e^{-1} J^2 > 16 \pi \ln J > 1$, hence 
$1 < s_J < \sqrt{2 \ln J}$. Therefore, we see by~\eqref{TailG} for each
$j = 1, \ldots, J$ that
\[
     P \bigl( U_j > s_J \bigr)
 \ge \frac {s_J} {\sqrt{2 \pi} (1 + s_J^2)} 
      \frac { \sqrt {16 \pi \ln J} } J
 \ge \frac {2 s_J^2} {(1 + s_J^2) J} 
 \ge \frac 1 J \ms1 \up.
\]
It follows that
\[
     P \bigl( \max_{1 \le j \le J} U_j \le s_J \bigr)
 \le \Bigl(
      1 -  \frac 1 J 
     \Bigr)^J
  <  \e^{ - 1}
  <  \frac 1 2 \ms1 \up.
\]
\end{proof}

\begin{thm}[Aldaz~\cite{AldazWT}]\label{Qualitat}
The weak type\/ $(1, 1)$ constant $\kappa_{Q, n}$ in\/~\eqref{WTforCubes}
does not stay bounded when the dimension $n$ tends to infinity.
\end{thm}

\begin{proof}
Given an arbitrary $t > 1$, we let $t_1 := t \ms1 w^{-1} (1 + v) > t$ and
choose an integer $K \ge 21$ such that $s_K > t_1$. Applying
Lemma~\ref{SupGaussi}, we obtain that the event 
$\bigl\{ \max_{ \ms1 1 \le j \le K} |U_j| > t_1 \bigr\}$ has probability
$>1/2$, and by~\eqref{UG}, it follows that the~event 
$\{ \max_{ \ms1 0 \le j \le K} |G_j| > t \}$ also has probability $> 1/2$.
We see that $\sup |G_j|$ is the maximum of $\sup G_j$ and $\sup (-G_j)$ that
have the same distribution, hence 
$\{ \max_{ \ms1 0 \le j \le K} G_j > t \}$ has probability $> 1/4$.
Consequently, given any $t > 1$, we obtain by the central limit theorem
that the union of sets $A_{\alpha_j, t}^{(n)}$, for $j = 0, \ldots, K$, has
a probability close to that of $\{ \max_{ \ms1 0 \le j \le K} G_j > t \}$,
hence $> 1/4$ when $n$ is large. By Lemma~\ref{OneSet}, given 
$\theta \in (0, 1)$ and if $\sqrt n \ms1 (1 - \theta) \sigma_K^2 > 2 t$,
the maximal function $\M_Q \mu_N^{(n)}$ is larger than 
$\e^{ \theta t^2 / 2}$ on the union 
$\bigcup_{j=0}^K A_{\alpha_j, t}^{(n)}$, \textit{i.e.}, on a subset of
$\Omega = S_N^n$ having probability $> 1/4$, hence
$\kappa_{Q, n} \ge \e^{ \theta t^2 / 2} / 4$ when $n$ is large enough.
\end{proof}

 Aubrun~\cite{Aubrun} gives a lower bound 
$\kappa_{Q, n} \ge \kappa_\varepsilon (\ln n)^{1 - \varepsilon}$ for every
$\varepsilon > 0$ by making quantitative the proof above. He applies to this
end results proved years before (by Bretagnolle--Massart~\cite{BreMass}
in~1989 and previously, by Koml\'os--Major--Tusn\'ady~\cite{KMT} in~1975)
on the approximation of Brownian bridges, when $n \rightarrow +\infty$ and
with explicit bounds, by binomial processes
\[
   Z_t^{(n)}
 = \sum_{i=1}^n \frac {\gr 1_{ \{Y_i \le t\} } - t} {\sqrt n} \up,
 \qquad t \in [0, 1],
\]
where the $(Y_i)_{i=1}^n$ are independent and uniform on $[0, 1]$. One can
see that the distribution of the process $(Z_t^{(n)})_{t \in (0, 1)}$ is
equal to that of $(\sigma_t \ms1 X_{n, t})_{t \in (0, 1)}$.
\dumou

 Iakovlev and Str\"omberg~\cite{IS} begin with the same observations, in
particular introducing the measure $\mu_N^{(n)}$, using the fundamental
estimate~\eqref{FirstQuot} and, in a less apparent manner, the value 
$\e^{\theta t^2 / 2}$ from Lemma~\ref{OneSet}. But instead of working in a
probabilistic setting, they proceed to a finer combinatorial analysis.
Contrary to Aubrun, they do not use values $\alpha$ close to~$1$, nor close
to $0$. In our exposition of their arguments, we shall work towards
simplicity rather than optimality.

\begin{smal}

\noindent
Let us digress a little with some comments on the Gaussian process
viewpoint, and express in terms of stochastic maximal function the lower
bound for $\M_Q \mu_N^{(n)}$ given in~\eqref{FirstQuot}. Let $x \in \Omega$
and $m = N_{n, \alpha}(x)$, $\sigma_\alpha^2 = \alpha (1 - \alpha)$ and
write $m = \alpha n + \sigma_\alpha t \sqrt n$. Notice that
$t = (m - \alpha n) / (\sigma_\alpha \sqrt n) = X_{n, \alpha}(x)$. We let
$f$ be the fraction $m / n$, and rewrite the preceding formula for $m$
as $f = \alpha + \sigma_\alpha \tau$, with $\tau = t / \sqrt n$. We know
the optimal argument $\ovy$ for $V(y)$, given in~\eqref{FirstEsti} by
\[
 \ms{16}
   \ovy 
 = \frac t {\sigma_\alpha \sqrt n}
 = \frac \tau {\sigma_\alpha} \up,
 \ms{22} \hbox{and} \ms{12}
   \frac {\ln V(\ovy) } n
 = f \ln \Bigl( \frac f \alpha \Bigr)
    + (1 - f) \ln \Bigl( \frac { 1 - f } {1 - \alpha } \Bigr) .
\]
By Lemma~\ref{MaxiVal} we have
$
     \ln E_{f, \alpha}
  =  \phi_\alpha (f)
 \ge \tau^2 / 2
  =  t^2 / (2 n)
$ if $\tau > 0$ and $\alpha \ge 1/2$.
Let $1/2 \le \alpha \le 3/4$ and assume that 
$0 < t = X_{n, \alpha}(x) \le n^{1/4} / 2$. We see then that
$\ovy = t / (\sigma_\alpha \sqrt n) < 2 \ms1 n^{-1/4} / \sqrt 3$, thus
$n \ms1 \ovy^4 \le 16/9$, $\ovy \le 1/4$ for $n > 455$ and by
Lemma~\ref{YandAllow} we are then in the allowable case with $c \le 16/9$.
This yields
\begin{subequations}\label{asymptotics}
\SmallDisplay{\eqref{asymptotics}}%
\[
     \M_Q \mu_N^{(n)} (x)
 \ge \kappa^{-1} E_{f, \alpha}^n
 \ge \kappa^{-1} 
      \exp \Bigl(
            \frac { t^2 \ns4} 2 \ms3
           \Bigr),
 \ms{12} \hbox{with} \ms{ 8}
 \kappa < \e^{16/9} \ns3< 6,
 \ms{ 6} n > 455.
\]
\end{subequations}
\noindent
Let us define a maximal function
$
   X^*(x) 
 = \sup_{ \ms1 1/2 \le \alpha \le 3/4} \ms2 X_{n, \alpha}^{(1)} (x)
$,
where $X_{n, \alpha}^{(1)} (x) = X_{n, \alpha}(x)$ when
$0 \le 2 \ms1 X_{n, \alpha} (x) \le n^{1/4}$ and 
$X_{n, \alpha}^{(1)} (x) = 0$ otherwise. We get
\[
     6 \ms2 \M_Q \mu_N^{(n)} (x)
 \ge \exp \Bigl(
           \frac { X^*(x)^2 \ns7} 2 \ms6
          \Bigr)
\]
and the weak type $(1, 1)$ constant $\kappa_{Q, n}$ must verify that
\[
     P \bigl( \{ X^* > s \} \bigr) 
 \le P \bigl( \{ \M_Q \mu_N^{(n)} > \e^{s^2/2} / 6\} \bigr) 
 \le 6 \ms2 \kappa_{Q, n} \e^{-s^2 / 2},
 \quad s > 0.
\]
This explains how delicate the question can be. Indeed, given a subgaussian
process $(Y_t)_{t \in T}$ satisfying tail estimates of the form 
$P( Y > s ) \le \kappa \e^{- s^2 / (2 d^2)}$ for every $s > 0$, for
each difference $Y = Y_{t_2} - Y_{t_1}$ and with
$d = d(t_1, t_2) = \|Y_{t_1} - Y_{t_2}\|_2$, the well known chaining
technique of Dudley~\cite{RMDud} does not allow one to prove for the
maximal process $\sup_{t \in T} Y_t$ a subgaussian inequality with the same
bounding function $\e^{- s^2 / 2}$, but rather with $\e^{- C s^2 / 2}$ for
some $C < 1$, which is inoperative here.

\end{smal}
\dumou

\begin{thm}[Iakovlev and Str\"omberg~\cite{IS}]\label{PropIS} 
One has that
\[
 \kappa_{Q, n} \ge \kappa \ms2 n^{1/4}.
\]
\end{thm}

 Rather than exploiting the exponential asymptotics~\eqref{asymptotics} of
$E_{f, \alpha}^n$, we shall observe some more nice features of the
expression $E_{f, \alpha}$ defined in~\eqref{EAlpha}, where 
$f = m / n = \alpha + t \sigma_\alpha / \sqrt n
 = \alpha + \sigma_\alpha \tau$. We replace the value $\e^{ \theta t^2 / 2}$
seen in Lemma~\ref{OneSet} by\label{ValueV} 
a fixed large value $V > 1$ and we try to keep the (conditional on
allowability) lower bound $E_{f, \alpha}^n$ for $\M_Q \mu_N^{(n)}$
constantly equal to~$V$. Equivalently, we keep
\begin{equation}
   E_{f, \alpha}
 = \e^{\phi_\alpha (f)}
 = V^{1/n}
 > 1
 \label{EAlphaEq}
\end{equation}
for all values of $f$ (or of $m$) that will be handled. The possibility of
finding $\alpha$ verifying~\eqref{EAlphaEq} comes from the fact that
for every given $f \in (0, 1)$, the function
\begin{equation}
 \psi_f : s \mapsto
  \Bigl( \frac f s \Bigr)^f
   \Bigl( \frac {1 - f} {1 - s} \Bigr)^{1 - f}
 = \e^{\phi_s(f)},
 \quad s \in (0, 1),
 \label{PsiF}
\end{equation}
is convex on $(0, 1)$ (actually, log-convex), tends to infinity at $0$ and
at $1$, and assumes its minimal value $\psi_f(f) = 1$ at $s = f$.
Consequently, there are exactly two values $\alpha_0 < f < \alpha_1$ of
$\alpha \in (0, 1)$ solving~\eqref{EAlphaEq}, we shall consider the
smallest one 
and set\label{AlphaF}
$\alpha(f) = \alpha_0$. Notice that $(\ln \psi_f)'(s) 
 = - f / s + (1 - f) / (1 - s)$ vanishes at $s = f$, and
\begin{equation}
   (\ln \psi_f)''(s) 
 = \frac f {s^2} + \frac { 1 - f } { (1 - s)^2 \ns8 } \ms5 
 > f + (1 - f) = 1.
 \label{PsiFb}
\end{equation}
We have therefore for every $s \in (0, 1)$ that
\begin{equation}
 \ln \psi_f(s) \ge (s - f)^2 / 2,
 \ms{14} \hbox{thus} \ms{10}
     (f - \alpha(f))^2 / 2 
 \le \ln \psi_f(\alpha(f))
  =  (\ln V) / n.
 \label{MinoGf}
\end{equation}
\dumou

 From now on, we fix two values $0 < f_* < f^* \le 1/2$,\label{FStars} 
independent of the dimension~$n$. For every integer $m$ in the range 
$[f_* \ms1 n, f^* \ms1 n]$, we shall consider the~set\label{SetFm}
\[
 F_m = \{ x \in \Omega : N_{n, \alpha(f)}(x) = m \},
 \ms{16} \hbox{with} \ms{12}
 f = m / n.
\]
Let us write $\alpha = \alpha(f)$ for brevity. We have that
$E_{f, \alpha} = V^{1/n}$ and if we assume $c \ms1$-allowability for
$(f, \alpha)$ we get
$
 \M_Q \mu_N^{(n)} (x) \ge \e^{-c} V$ for every $x \in F_m$,
by~\eqref{DefAllow}. The probability of $F_m$ is 
$\alpha^m (1 - \alpha)^{n - m} \binom n m$ and we see that
\[
   V P(F_m)
 = \Bigl( \frac f \alpha \Bigr)^m
    \Bigl( \frac {1 - f} {1 - \alpha} \Bigr)^{n - m}
     \ms3 \alpha^m (1 - \alpha)^{n - m}
      \ms3 \binom n m
 = \frac {m^m (n - m)^{n - m} \ns{28}} {n^n} \ms{20}
    \ms3 \binom n m.
\]
Stirling's formula in the form
$\e^{ - 1 / (12 p) } p \ms1 ! \le p^p \e^{-p} \sqrt {2 \pi p}
 \le p \ms1 !$ (see~\cite{Robbins}) gives
\begin{equation}
     \e^{ - 1 / (12 \ms1 n) } \ms1 V P(F_m)
 \le \frac {\sqrt n} {\sqrt{2 \pi m (n - m) }}
 \le \e^{ n / (12 \ms1 m (n-m)) } \ms1 V P(F_m).
 \label{Stirl}
\end{equation}
With\label{SStars} 
$s_* = \sqrt { f_* (1 - f_*) }$ and $s^* = \sqrt { f^* (1 - f^*) }$,
it follows that
\begin{equation}
     V P(F_m)
 \ge \frac { \e^{ - 1 / (12 f(1-f) n) } } 
           { \sqrt { 2 \pi f (1 - f) n } }
 \ge \frac { \e^{ - 1 / (12 s_*^2 n) } } 
           { s^* \sqrt { 2 \pi } } \ms4
     \frac 1 { \sqrt n } \up.
 \label{StirlA}
\end{equation}
If the sets~$F_m$ were disjoint (and the couples $(f, \alpha(f))$
$c \ms1$-allowable) we would get immediately, by summing on $m$ between 
$f_* \ms1 n$ and~$f^* \ms1 n$, a lower bound of
\[
     \kappa_{Q, n}
 \ge \e^{-c} V P \bigl( \{  \M_Q \mu_N^{(n)} \ge \e^{-c} V \} \bigr)
 \ms{12} \hbox{by} \ms{16}
 \kappa \ms3 \bigl[ \e^{-c} (f^* - f_*)
  / ( s^* \sqrt { 2 \pi } ) \bigr] \ms2 \sqrt n,
\]
but this disjointness property is clearly not true. We shall specify a
suitable large~$V$ such that the probability of the intersection of two
events $F_{m_1}$ and $F_{m_2}$ will be small compared to the probability of
$F_{m_1}$, when $m_1 < m_2$ are not too close. We shall find a subset 
$M \subset [f_* \ms1 n, f^* \ms1 n]$, as large as possible, consisting of
\og well spaced\fge values~$m_j$ giving rise to $c \ms1$-allowable couples.
The final estimate has the form
\begin{equation}
     \kappa_{Q, n}
 \ge \e^{-c} V \ms1
      P \bigl( \bigcup_{m \in M} F_m \bigr),
 \label{ReNum}
\end{equation}
where the probability of the union will be larger than half of the sum of
probabilities. The seemingly harmless allowability restriction that 
$y^{-1} - \alpha = \sigma_\alpha / \tau - \alpha$ must be an odd
integer~$\ell$ will actually cause a heavy loss at the end.
\dumou

 We fix $\varepsilon \in (0, f_*]$ and 
introduce\label{EpsilEta} 
$\eta := \sqrt{ 1 - \varepsilon / f_*}$. We define the \og big\fge value
$V$ as $V = \e^{\varepsilon^2 n / 2}$. By~\eqref{MinoGf}, we have that
\begin{equation}
     0
  <  f - \alpha(f) 
 \le \varepsilon.
 \label{FmoinsA}
\end{equation}

\begin{lem}
Suppose that\/ $0 < \alpha < \xi \le f \le \alpha + \varepsilon$ and 
$f_* \le f \le 1/2$. One has
\begin{equation}
     \eta^2
 \le \frac \alpha \xi
  <  \frac { \alpha (1 - \alpha) } { \xi (1 - \xi) }
  <  1,
 \ms{16} \hbox{in particular} \ms{12}
   \eta \ms1 \sigma_f 
 = \eta \ms1 \sqrt { f (1 - f) }
 < \sigma_\alpha.
 \label{AdHoc}
\end{equation}
Assuming $V = \e^{\varepsilon^2 n / 2}$, $\alpha = \alpha(f)$ and writing
$\sigma_\alpha \tau = f - \alpha$, one has that
\begin{equation}
     \eta \ms3 \tau
 \le \varepsilon
 \le \tau.
 \label{NewEncadre}
\end{equation}
\end{lem}

\begin{proof}
We see that $\alpha (1 - \alpha) < \xi (1 - \xi)$ because
$0 < \alpha < \xi \le 1/2$. Next, we get
\[
     \frac { \alpha (1 - \alpha) } { \xi (1 - \xi) }
  >  \frac \alpha \xi
 \ge \frac {f - \varepsilon} f
 \ge 1 - \frac \varepsilon { f_* } 
  =  \eta^2.
\]
By Taylor--Lagrange at~$\alpha$ for the function $\phi_\alpha$ defined
in~\eqref{DefPhiAlpha}, we have
\[
   \phi_\alpha(f)
 = \phi''_\alpha(\xi_0) \frac { (f - \alpha)^2 \ns6 } 2 \ms3
 = \frac { \sigma_\alpha^2 } { \xi_0 (1 - \xi_0) } \ms3
    \frac { \tau^2 \ns6 } 2
 = \frac { \alpha (1 - \alpha) } { \xi_0 (1 - \xi_0) } \ms3
    \frac { \tau^2 \ns6 } 2
\]
for some $\xi_0 \in (\alpha, f)$, and $\phi_\alpha(f)
 = \phi_{\alpha(f)} (f)
 = (\ln V) / n = \varepsilon^2 / 2$ by assumption. The inequalities
in~\eqref{NewEncadre} follow then from~\eqref{FmoinsA}
and~\eqref{AdHoc}.
\end{proof}

 We have to understand how the values $\alpha(f)$ are distributed when $f$
varies in $[f_*, f^*]$. To this end, we estimate the derivative
$\alpha'(f)$.

\begin{lem}\label{AlphaPrime}
Let\/ $0 < \varepsilon \le f_*$ and $V = \e^{\varepsilon^2 n / 2}$.
The mapping $(0, 1) \ni f \mapsto \alpha(f)$ implicitly defined
at\/~\eqref{EAlphaEq} is increasing, and when $f \in [f_*, f^*]$ we have
that
\[
 \eta^2 < \alpha'(f) < 1.
\]
\end{lem}

\begin{proof}
We express the derivative $\alpha'(f)$ by differentiating with respect to
$f$ the equality $\phi_{\alpha(f)} (f) = (\ln V) / n$. Writing 
$\phi_\alpha$ for $\phi_{\alpha(f)}$, we obtain
\[
   \phi'_\alpha(f)
    + \Bigl( \frac \partial {\partial \alpha} \ms2 
       \phi_\alpha (f) \Bigr) \alpha'(f)
 = \phi'_\alpha(f)
    - \frac {f - \alpha} {\alpha (1 - \alpha)} \ms3 \alpha'(f) = 0.
\]
By Taylor--Lagrange at $\alpha$ for $s \mapsto \phi'_\alpha(s)$, there is 
$\xi \in (\alpha, f)$ such that
\[
   \phi''_\alpha(\xi) (f - \alpha)
 = \phi'_\alpha(f)
 = \frac {f - \alpha} {\alpha (1 - \alpha)} \ms3 \alpha'(f),
 \ms{16} \hbox{hence} \ms{12}
   \alpha'(f)
 = \frac { \alpha (1 - \alpha) }
         { \xi (1 - \xi) }
 > 0
\]
because $\phi''_\alpha(\xi) = \sigma_\xi^{-2}$. We have that 
$\alpha < \xi < f \le \alpha + \varepsilon$ by~\eqref{FmoinsA}, and when 
we further assume $f_* \le f \le f^* \le 1/2$ the conclusion follows
by~\eqref{AdHoc}.
\end{proof}
\dumou

 We need to study the intersections $F_{m_1} \cap F_{m_2}$, when
$m_1, m_2 \in [f_* n, f^* n]$.

\begin{lem}\label{Intersect}
Suppose that $f_* \ms1 n \le m_1 < m_2 \le f^* \ms1 n$. One has that
\[
   \e^{ - 1 / (6 \ms1 n) }
    P \bigl( F_{m_1} \cap F_{m_2} \bigr) / P ( F_{m_1} )
 < \lambda \e^{ - \delta^2 \varepsilon^2 (m_2 - m_1) / 2 }
     / \sqrt{ 2 \pi (m_2 - m_1) },
\]
with $\delta = \eta^3 \ms1 s_* / (1 - f_*)$ and 
$\lambda = \sqrt{ 1 - f_* } / \sqrt { 1 - f^* }$.
\end{lem}

\begin{proof}
Let $f_j = m_j / n$, $f_* \le f_j \le f^*$, and $\alpha_j = \alpha(f_j)$,
for $j = 1, 2$. By Lemma~\ref{AlphaPrime}, we have that 
$\alpha_1 < \alpha_2$ since $f_1 < f_2$.
\def\cd{{ \ns2 {}_A}}
Let $J$ be an arbitrary subset of $\{1, \ldots, n\}$ satisfying $|J| = m_1$,
and let $A(J)$ be the subset of~$\Omega = S_N^n$ defined by
\[
   A
 = A(J)
 = \bigl\{
    x = (x_1, \ldots, x_n) \in \Omega :  
     J = \{ i : x_i \in C_{\alpha_1} \}
   \bigr\}.
\]
One has thus $N_{n, \alpha_1}(x) = m_1$ when $x \in A$. The conditional
probability $p_A$ that $N_{n, \alpha_2}(x) = m_2$ knowing that $x \in A$
is equal to the probability that $m_\cd := m_2 - m_1$ of the remaining 
$n_\cd := n - m_1 = (1 - f_1) n \ge n/2$ coordinates of $x$ (those
coordinates that are in $\Omega_1 \setminus C_{\alpha_1}$) fall in 
$C_{\alpha_2} \setminus C_{\alpha_1}$. This is given by the binomial
distribution corresponding to $n_\cd$ and to
$\alpha_\cd = (\alpha_2 - \alpha_1) / (1 - \alpha_1)$, and we know therefore
that
\[
    p_\cd
 := \frac { P \bigl( \{ N_{n, \alpha_2} = m_2 \} \cap A \bigr) }
          { P(A) }
  = P \bigl( \{ N_{n_\cd, \alpha_\cd} = m_\cd \} \bigr)
  =  \alpha_\cd^{m_\cd} (1 - \alpha_\cd)^{n_\cd - m_\cd} 
      \binom {n_\cd} {m_\cd}.
\]
Let $f_\cd = m_\cd / n_\cd = (f_2 - f_1) / (1 - f_1)$.
Since $\alpha'(f) < 1$ on $[f_*, f^*]$, we get
\begin{align*}
    f_\cd 
 &= \frac {f_2 - f_1} {1 - f_1}
  = \frac {f_2 - f_1} {1 - \alpha_1} \ms2 
     \Bigl( 1 + \frac {f_1 - \alpha_1} {1 - f_1} \Bigr)
 \\
 &> \frac {\alpha_2 - \alpha_1} {1 - \alpha_1}
     + \frac { (f_1 - \alpha_1) (f_2 - f_1) } 
               { (1 - \alpha_1) (1 - f_1) }
  = \alpha_\cd
     + \frac { f_1 - \alpha_1 } 
             { (1 - \alpha_1) (1 - f_1) }
        \ms2 (f_2 - f_1).
\end{align*}
Let $f_1 - \alpha_1 = \sigma_{\alpha_1} \tau_1$. We have
$\tau_1 \ge \varepsilon$ by~\eqref{NewEncadre}, 
$\sigma_{f_1}
 > \sigma_{\alpha_1} > \eta \ms1 \sigma_{f_1} \ge \eta \ms1 s_*$
by~\eqref{FmoinsA} and~\eqref{AdHoc}, and $f_* \le f_1 \le 1/2$. By the
leftmost inequality in~\eqref{AdHoc}, we obtain
\[
     \frac 1 { 1- \alpha_1 }
   > \frac { \alpha_1 } { f_1 (1 - f_1) }
 \ge \frac { \eta^2 } { 1 - f_1 }
 \ge \frac { \eta^2 } { 1 - f_* } \up,
\] 
therefore
\begin{equation}
   f_\cd - \alpha_\cd
 > \frac { \eta^3 s_* \varepsilon } 
         { (1 - f_*)(1 - f_1) } \ms2 (f_2 - f_1)
 = \frac { \delta \ms1 \varepsilon } 
         { 1 - f_1 } \ms2 (f_2 - f_1).
 \label{StepB}
\end{equation}
Recalling the function $\psi_f$ from~\eqref{PsiF}, we see that
\[
   p_\cd
 = \psi_{f_\cd}(\alpha_\cd)^{- n_\cd}
      f_\cd^{m_\cd} (1 - f_\cd)^{n_\cd - m_\cd} 
       \binom {n_\cd} {m_\cd}.
\]
Applying Stirling as before in~\eqref{Stirl}, and
because we have that 
$n_\cd / (n_\cd - m_\cd) = (1 - f_1) / (1 - f_2)
 \le (1 - f_*) / (1 - f^*)$, we obtain
\[
     \e^{ - 1 / (12 \ms1 n_\cd) } \ms2 p_\cd
  <  \psi_{f_\cd}(\alpha_\cd)^{- n_\cd}
       \sqrt{ \frac {n_\cd} { 2 \pi m\cd (n_\cd - m_\cd) } }
 \le \psi_{f_\cd}(\alpha_\cd)^{- n_\cd}
        \sqrt{ \frac { 1 - f_* } { 2 \pi (1 - f^*) m\cd } } \up.
\]
For some $\xi \in (\alpha_\cd, f_\cd)$, and
since $(\ln \psi_{f_\cd})'' (\xi) > f_\cd / \xi^2 > 1 / f_\cd$
by~\eqref{PsiFb}, we get
\[
   \ln \psi_{f_\cd}(\alpha_\cd)
 = (\ln \psi_{f_\cd})'' (\xi) \ms2
    \frac { (f_\cd - \alpha_\cd)^2 \ns7 } 2
 > \frac { (f_A - \alpha_A)^2 \ns7 } { 2 \ms1 f_A } \ms{0.5} \up.
\]
Consequently, we can write
\[
     p_\cd
  <  \e^{ 1 / (12 \ms1 n_\cd) } \ms3
      \exp \Bigl( - n_\cd \frac {(f_\cd - \alpha_\cd)^2} {2 f_\cd} \Bigr)
       \frac \lambda { \sqrt{ 2 \pi m\cd } } \up,
 \ms{16} \hbox{with} \ms{12}
   \lambda 
 = \frac { \sqrt{ 1 - f_* } }
         { \sqrt { 1 - f^* } } \up.
\]
We see that $n_\cd / f_\cd = n_\cd^2 / (m_2 - m_1)$.
By~\eqref{StepB} we have
\[
     \frac {n_\cd} {f_\cd} \ms4 (f_\cd - \alpha_\cd)^2
 \ge \frac { n^2 (1 - f_1)^2  }
           { m_2 - m_1 } \ms4
      \frac { \delta^2 \ms1
              \varepsilon^2 (f_2 - f_1)^2 }
            { (1 - f_1)^2 }
  =  \delta^2 \ms1 \varepsilon^2 (m_2 - m_1).
\]
Using also $n < 2 n_\cd$ and the definition of $p_\cd$, we obtain for 
$A = A(J)$ that
\[
   P \bigl( A(J) \cap \{ N_{n, \alpha_2} = m_2 \} \bigr)
 < \Bigl(
    \e^{ 1 / (6 \ms1 n) } \ns1 
     \lambda \e^{ - \delta^2 \varepsilon^2 (m_2 - m_1) / 2}
       / \sqrt{ 2 \pi (m_2 - m_1) } 
   \Bigr) \ms2 
     P( A(J) ).
\] 
Summing on all subsets $J$ of $\{1, \ldots, n\}$ with $|J| = m_1$, 
and because $\bigcup_{|J| = m_1} A(J)$ is equal to
$\{ N_{n, \alpha_1} = m_1 \} = F_{m_1}$, we get
\[
      P \bigl( F_{m_1} \cap F_{m_2} \bigr)
  <  \Bigl( \e^{ 1 / (6 \ms1 n) } \ns1
      \lambda \e^{ - \delta^2 \varepsilon^2 (m_2 - m_1) / 2 }
       / \sqrt{ 2 \pi (m_2 - m_1) } \Bigr) \ms2 P ( F_{m_1} ).
 \qedhere
\]
\end{proof}

\begin{proof}[End of proof of Theorem~\ref{PropIS}]
Let $H$ be a sufficiently large integer, and let us now define
$M = \{ j H : j \in \N \} \cap [f_* \ms1 n, f^* \ms1 n]$ to be the set of
multiples of $H$ located in the segment $[f_* \ms1 n, f^* \ms1 n]$. We fix
$m_1 \in M$ and let $m_2 > m_1$ be any other element of~$M$. Then 
$m_2 - m_1 = k H$ with $k$ integer $\ge 1$. Summing on $k \ge 1$ we see that
\[
   \sum_{k=1}^{+\infty}
     \frac { \e^{ - \delta^2 \varepsilon^2 k H / 2 } }
           { \sqrt{ k H } }
 < \int_0^{+\infty} \e^{ - \delta^2 \varepsilon^2 H s / 2 } 
    \frac {\d s} {\sqrt { H s } }
 = \frac { \sqrt 2 \ms2 \Gamma(1/2) } 
         { \varepsilon H \delta } 
 = \frac { \sqrt { 2 \pi } } { \varepsilon H \delta } \up.
\]
By Lemma~\ref{Intersect},
we get
$
   \sum_{m_2 \in M, \ms2 m_2 > m_1} 
    P \bigl( F_{m_1} \cap F_{m_2} \bigr)
 < P(F_{m_1}) / 2
$ when $\varepsilon H$ is larger than 
$2 \ms1 \lambda \e^{1 / (6n) } / \delta$. It follows then that at least one
half of the set $F_{m_1}$ is not covered by the other sets $F_{m_2}$ for
$m_2 > m_1$ and $m_2 \in M$, therefore
$P(\bigcup_{m \in M, \ms2 m \ge m_1} F_m)
 \ge P(\bigcup_{m \in M, \ms2 m > m_1} F_m) + P(F_{m_1}) / 2$ for
$m_1 \in M$. The probability of $\bigcup_{m \in M} F_m$ is thus at least
equal to half of the sum of probabilities. By~\eqref{StirlA} 
and~\eqref{ReNum} we get
\begin{equation}
     \kappa_{Q, n}
 \ge \e^{-c} V \ms1
      P \bigl( \bigcup_{m \in M} F_m \bigr)
 \ge \frac { \e^{-c} \ns7 } 2 \ms4
      \sum_{m \in M} V \ms1 P (  F_m )
 \ge  \frac { \e^{-c} \ns7 } 2 \ms7
       \frac { \e^{ - 1 / (12 s_*^2 n) } } { s^* \sqrt { 2 \pi } } \ms5
        \frac {|M|} {\sqrt n} \up.
 \label{ConcluB}
\end{equation}
\dumou

 So far we could hope for a lower bound of order $\sqrt n$ for the weak type
constant. But we have to comply with the allowability restriction, and we
must estimate the number of couples $(f, \alpha(f))$ that are 
$c \ms1$-allowable. We let
\[
   \varepsilon 
 = \frac { s_* n^{-1/4}} {1 +  s_* n^{-1/4} / f_*} \up,
 \ms{16} \hbox{so that} \ms{12}
   \frac \varepsilon { \eta^2 \ns4 } \ms2 
 = \frac \varepsilon { 1 - \varepsilon / f_*} 
 = s_* n^{-1/4}
\]
and $\varepsilon < f_*$. We choose a spacing $H \sim n^{1/4}$. For every 
$m \in M$, for $f = m/n$, $\alpha = \alpha(f)$ and 
$f = \alpha + \sigma_\alpha \tau$ we have by~\eqref{FirstEsti},
\eqref{AdHoc} and~\eqref{NewEncadre} that
\[
     \ovy
  =  \frac \tau {\sigma_\alpha}
 \le \frac \varepsilon \eta \ms3
      \frac 1 { \eta \ms1 \sigma_f }
 \le \frac \varepsilon { \eta^2 s_* }
  =  n^{-1/4}.
\]
For $n > 256$ we see that $\ovy < 1/4$ and $\ovy^4 < 1 / n$, thus
$(f, \alpha(f))$ is allowable with constant $c = 1$ according to
Lemma~\ref{YandAllow}. We choose the spacing integer $H$ such that
$H > 2 \ms1 \lambda \e^{1 / (6n) } / (\delta \varepsilon)$. Since
$\varepsilon = \eta^2 s_* n^{-1/4}$, we arrive to the condition
\[
   H 
 > (2 \ms1 \lambda / \delta \eta^2 s_*) \ms3 
    \e^{1 / (6n) } \ms1 n^{1/4}.
\]
We obtain a set $M \subset [f_* \ms1 n, f^* \ms1 n]$ of multiples of
$H$ with cardinality at least equal to
$ \lfloor (f^* \ms1 n - f_* \ms1 n) / H \rfloor
 > \bigl[ \eta^2 \delta s_* (f^* - f_*)
 / (2 \lambda) \bigr] \ms2 \e^{ - 1 / (6n) }n^{3/4} - 1$ where each 
element~$m$ produces a $1$-allowable couple $(f, \alpha(f))$.
By~\eqref{ConcluB}, we get that
\[
     \kappa_{Q, n}
 \ge \frac 1 { 2 \ms1 \e } \ms 2 
      \frac { \e^{ - 1 / (12 \ms1 s_*^2 n) } } 
            { s^* \sqrt { 2 \pi } } \ms4
       \frac { |M| } { \sqrt n }
 \ge \frac { \eta^2 \delta s_* (f^* - f_*) } 
           { 4 \ms1 \e \sqrt { 2 \pi } \lambda s^* }
       \ms4 n^{1/4}
        - O(n^{-1/2}).
\] 
Our version of the Iakovlev--Str\"omberg proof is not optimal, we shall
however try to figure out a numerical value for the constant that we get in
front of $n^{1/4}$. We have for $n$ large that $\varepsilon = o(1)$, thus
$\eta \simeq 1$. Let us introduce
\[
    z
 := \frac { \delta s_* (f^* - f_*) } 
          { \eta^3 \ms1 \lambda s^* }
  = \frac { s_*^2 } { 1 - f_* } \ms3 
    \frac { f^* - f_* } 
          { \sqrt { f^* (1 - f_*) } }
  = \frac { f_* (f^* - f_*) } 
          { \sqrt { f^* (1 - f_*) } } \up.
\]
This expression increases with $f^*$, so we set $f^* = 1/2$, the maximal
possibility. Then, the resulting value of $z$ is maximal for
$f_* = 3/4 - \sqrt { 11 / 48 } \sim 0.271$, yielding $z > 0.102$.
When $n$ is large, we have
\[
   \kappa_{Q, n} 
 > \frac z { 4 \e \sqrt {2 \pi} } \ms4 n^{1/4} - o(n^{1/4})
 > 0.0037 \ms3 n^{1/4}
 > \frac { n^{1/4} \ns{10} } { 271 } \ms1 \up.
\]
Notice that we have set the constant value $V$ as
$V = V_n \sim \e^{\kappa \sqrt n}$ in dimension~$n$. The corresponding
sequence of values $t_n = \sqrt { 2 \ln V_n} \sim n^{1/4}$ for the 
\og test sets\fge $\{ X_{n, \alpha} > t_n \}$ is \og invisible\fge to the
Gaussian limit argument of Theorem~\ref{Qualitat}.
\dumou

\end{proof}

\newbox\yearbox
\newbox\pagebox
\newbox\volubox
\newbox\virgule
\def\VPY#1#2#3#4{%
\setbox\volubox=\hbox{#1, }%
\setbox\pagebox=\hbox{p.~#2--#3 }%
\setbox\yearbox=\hbox{(#4)}%
\setbox\virgule=\hbox{,}
}
\def\Year#1{\setbox\yearbox=\hbox{ (#1)}}
\def\Pages#1#2{\setbox\pagebox=\hbox{p.~#1--#2}}
\def\Volu#1{\setbox\volubox=\hbox{#1\box\virgule}}
\def\showyear{\box\yearbox}
\def\showpage{\box\pagebox}
\def\showvolu{\box\volubox}
\def\RefSA#1#2#3#4{
{\sc #2~(#1)}.~--- #3. #4\showvolu\showpage\showyear.}
\def\EndRef#1#2#3 \fin{%
.~--- #1. #2\showvolu\showpage\showyear.}
%
\def\Authors#1#2#3 \fin{%
{\sc #2~(#1), }\RefMA#3 \fin}
\def\PreLastAuthor#1#2#3 \fin{%
{\sc #2~(#1) }\RefMA#3 \fin}
\def\LastAuthor#1#2#3 \fin{%
{\sc and #2~(#1)}\EndRef#3 \fin}
\def\RefMA#1#2 \fin{
\ifnum#1=1\LastAuthor#2 \fin\else
\ifnum#1=2\PreLastAuthor#2 \fin\else
 \Authors#2 \fin\fi\fi}

\def\Indx#1#2
{\hbox to \hsize
 {\noindent\hbox{#1} \leaders \hbox to 0.8em{\hss.\hss}\hfill #2
 }
\vskip 0.2pt plus 0.2pt minus 0.2pt}
\def\IndxB#1
{\hbox to \hsize
 {\noindent\hbox{#1}\hfill
 }
\vskip 0.2pt plus 0.2pt minus 0.2pt}
\bigskip
\goodbreak
\def\InSect{}
\def\AtEqua{eq.~}
\def\InTheo{Th.~}
\def\InProp{Prop.~}
\def\InLemm{Lem.~}
\def\InCoro{Cor.~}
\def\InRema{Rem.~}
\def\InIntr{Intro.}
\def\after{af.~}
\def\before{bf.~}
\def\AtEquaN#1{\AtEqua \eqref{#1}, p.~\pageref{#1}}
\def\InTheoN#1{\InTheo \ref{#1}, p.~\pageref{#1}}
\def\InPropN#1{\InProp \ref{#1}, p.~\pageref{#1}}
\def\InLemmN#1{\InLemm \ref{#1}, p.~\pageref{#1}}
\def\InRemaN#1{\InRema \ref{#1}, p.~\pageref{#1}}

\noindent{\bf Index}\label{Index}
\smallskip
 
{
\smaller
\smaller
\smaller

\noindent 
In the index and notation list, 1, 1.2, 1.2.3 refer respectively to
Section~1, subsection~1.1 or subsubsection~1.2.3. We also localize by
equation number, or statement number, as in eq.~(1.1), Prop.~1.2, or
proximity (before, after) to one of those, as in 
{\before \InLemm \ref{TLLb}} for example.
\medskip

\Indx{admissible growth}
     {\before \InLemm \ref{TLLb}, p.~\pageref{AdmiGro}}
\Indx{allowable couple}
     {\before \InLemm \ref{YandAllow}, p.~\pageref{AlloCou}}
\Indx{analytic semi-group}
     {\InPropN{PisierSG}}
\Indx{atom (of a $\sigma$-field)}
     {\InSect \ref{Gdfi}, p.~\pageref{Atom}}
\Indx{barycenter (of a probability measure on $\R^n$)}
     {\InSect \ref{GaussiDist}, p.~\pageref{Baryce}}
\Indx{Bernoulli martingale}
     {\InSect \ref{BGIneqs}, p.~\pageref{BernouMart}}
\Indx{Bernoulli random variables}
     {\before \AtEqua \InSect \eqref{GSG}, p.~\pageref{BernouVaris}}
\Indx{Bernoulli random walk}
     {\InSect \ref{RotaRem}, p.~\pageref{BernouWal}}
\Indx{Bessel function}
     {\AtEquaN{Bess}}
\Indx{Brownian motion}
     {\InSect \ref{GaussiDist}, p.~\pageref{BrowniMot}}
\Indx{Brownian martingale}
     {p.~\pageref{BrowniMarting}}
\Indx{Brunn--Minkowski inequality}
     {\InSect \ref{VoluSections}, p.~\pageref{BrunnMin}}
\Indx{Burkholder--Gundy inequalities}
     {\InSect \ref{BGIneqs}, p.~\pageref{BGIneqs}}
\Indx{Cauchy kernel}
     {\AtEquaN{CauchyK}}
\Indx{centered distribution}
     {\InSect \ref{GaussiDist}, p.~\pageref{CenteDist}}
\Indx{complex interpolation method}
     {\InSect \ref{Ihf}, p.~\pageref{Ihf}}
\Indx{conditional expectation}
     {\InSect \ref{Gdfi}, p.~\pageref{CondExpe}}
\IndxB{conjugate exponent (of $p \ge 1$), number $q \in [1, +\infty]$ such
       that $1/p + 1/q = 1$}
\Indx{convex body}
     {\InSect \ref{GaussiDist}, p.~\pageref{ConveBod}}
\Indx{covariance quadratic form, covariance matrix}
     {\InSect \ref{GaussiDist}, p.~\pageref{CovariMa}}
\Indx{decoupling}
     {\InSect \ref{Decoupling}, p.~\pageref{Decoupling}}
\Indx{dilates of a function}
     {\InSect \ref{FourierMult}, p.~\pageref{DilatOper}}
\Indx{directional variation $V(K)$ (of a kernel $K$ on $\R^n$)}
     {\AtEquaN{V(K)}}
\Indx{distribution of a random variable}
     {\InSect \ref{Gdfi}, p.~\pageref{Distribu}}
\Indx{Doob's inequality}
     {\InSect \ref{MaxiDoob}, p.~\pageref{MaxiDoob}}
\Indx{dyadic filtration}
     {\InSect \ref{BGIneqs}, p.~\pageref{DyadiFiltr}}
\Indx{dyadic maximal function}
     {\InSect \ref{ArtiCarbe}, p.~\pageref{DyMaFu}}
\Indx{essential supremum}
     {\InSect \ref{DefiMaxiFunc}, p.~\pageref{EssenSup}}
\Indx{Euler's formulas}
     {\InSect \ref{FoGamm}, p.~\pageref{EulerForm}}
\Indx{exit time}
     {\InSect \ref{PrincipeReflexion}, p.~\pageref{ExitTime}}
\Indx{expectation of a random variable}
     {\InSect \ref{Gdfi}, p.~\pageref{Expecte}}
\Indx{exponential decay (of log-concave probability densities)}
     {\InLemmN{LExpoDecay}}
\Indx{filtration}
     {\InSect \ref{MaxiDoob}, p.~\pageref{Filtra}}
\Indx{Fourier criteria for bounding the maximal function}
     {\InSect \ref{CritFou}, p.~\pageref{CritFou}}
\Indx{Fourier multipliers}
     {\InSect \ref{FourierMult}, p.~\pageref{FourierMult}}
\Indx{Fourier transform}
     {\InSect \ref{GdfiTwo}, p.~\pageref{FouTran}}
\Indx{fractional derivative}
     {\InSect \ref{FractiDeri}, p.~\pageref{FractiDeri}}
\Indx{fractional integration}
     {\InSect \ref{FractiDeri}, p.~\pageref{FractiIntA}}
\Indx{Gamma function}
     {\InSect \ref{FoGamm}, p.~\pageref{FoGamm}}
\Indx{Gaussian distributions}
     {\InSect \ref{GaussiDist}, p.~\pageref{Gaussiennes}}
\Indx{Gaussian semi-group}{\AtEquaN{maxiSemiGrGa}}
\Indx{Haar measure on $S^{n-1}$}
     {\InSect \ref{SphericOp}, p.~\pageref{HaarMeas}}
\Indx{harmonic conjugate}
     {\after \InCoro \ref{InStrip}, p.~\pageref{HarmonConj}}
\Indx{harmonic extension}
     {\InSect \ref{PoissonSG}, p.~\pageref{HarmonExt}}
\Indx{Hilbert transform}
     {\InSect \ref{TransfoRiesz}, p.~\pageref{HilbTrans}}
\Indx{holomorphic families of multipliers (M\"uller's ---)}
     {\InSect \ref{StrateMull}, p.~\pageref{HoloMull}}
\Indx{holomorphic families of operators}
     {\InSect \ref{Ihf}, p.~\pageref{Ihf}}
\Indx{homogeneous parts $H_k$ (in Pisier's analytic semi-group theorem)}
     {\InSect \ref{FirstReduc}, p.~\pageref{HomogPart}}
\Indx{Hopf maximal inequality (for semi-groups)}
     {\InSect \ref{MaxiHopf}, p.~\pageref{MaxiHopf}}
\Indx{independent random variables}
     {\InSect \ref{Gdfi}, p.~\pageref{IndRaVar}}
\Indx{indicator function}
     {\InSect \ref{Gdfi}, p.~\pageref{IndiFunc}}
\Indx{interpolation of linear operators}
     {\InSect \ref{Ihf}, p.~\pageref{Ihf}}
\Indx{invariance by dilation}
     {\AtEquaN{InvariMul}}
\Indx{invariant measure (for a Markov chain)}
     {\InSect \ref{RotaRem}, p.~\pageref{InvariMe}}
\Indx{inverse Fourier transform}
     {\InSect \ref{GdfiTwo}, p.~\pageref{InveFour}}
\Indx{isotropic position}
     {\InSect \ref{TheSetting}, p.~\pageref{IsoPosit}}
\Indx{isotropy constant}
     {\InSect \ref{TheSetting}, p.~\pageref{IsotCons}}
\Indx{Khinchin inequalities}
     {\AtEquaN{Hincin}}
\Indx{Laplace-type multipliers}
     {\InSect \ref{LaplaceMulTyp}, p.~\pageref{LaplaceMulTyp}}
\Indx{Lebesgue point $t_0$ of $f \in L^1_{\rm loc}(\R)$,
      where $t \rightarrow \int_0^t f(s) \, \d s$ is differentiable}
     {p.~\pageref{LebesPoint}}
\Indx{Littlewood--Paley decomposition}
     {\InSect \ref{Bsmo}, p.~\pageref{LiPaDec}}
\Indx{Littlewood--Paley functions}
     {\InSect \ref{LittlePal}, p.~\pageref{LittlePal}}
\Indx{log-concave function}
     {\InSect \ref{VoluSections}, p.~\pageref{LogConca}}
\Indx{marginal of a distribution}
     {\InSect \ref{Gdfi}, p.~\pageref{Margi}}
\Indx{Markov chain, Markov property}
     {\InSect \ref{RotaRem}, p.~\pageref{RotaRem}}
\Indx{martingale}
     {\InSect \ref{MaxiDoob}, p.~\pageref{MaxiDoob}}
\Indx{martingale difference sequence}
     {\InSect \ref{BGIneqs}, p.~\pageref{BGIneqs}}
\Indx{martingale transform}
     {\InSect \ref{BGIneqs}, p.~\pageref{MartiTrans}}
\Indx{mass (total mass) of a measure}
     {\InSect \ref{VoluSections}, p.~\pageref{TotaMass}}
\Indx{maximal function (classical Hardy--Littlewood ---)}
     {\AtEquaN{ClassicalM}}
\Indx{maximal function associated to a convex set}
     {\AtEquaN{OpMaxi}}
\Indx{maximal process (of a martingale)}
     {\InSect \ref{MaxiDoob}, p.~\pageref{MaxiProc}}
\Indx{method of rotations}
     {\InSect \ref{SphericOp}, p.~\pageref{RotaMeth}}
\Indx{multipliers (Fourier ---)}
     {\InSect \ref{FourierMult}, p.~\pageref{FourierMult}}
\Indx{multipliers associated to fractional derivatives}
     {\InSect \ref{MultipAssoc}, p.~\pageref{MultipAssoc}}
\Indx{normalization by variance}
     {\InSect \ref{TheSetting}, p.~\pageref{NormVari}}
\Indx{parameter (of a Poisson distribution)}
     {\InSect \ref{PoissonSG}, p.~\pageref{ParaPois}}
\Indx{Plancherel--Parseval equality}
     {\InSect \ref{GdfiTwo}, p.~\pageref{PlanPars}}
\Indx{Poisson kernel}
     {\AtEquaN{PoissonDensi}}
\Indx{Poisson semi-group}
     {\InSect \ref{PoissonSG}, p.~\pageref{PoissonSG}}
\Indx{positive operator}
     {\InSect \ref{Gdfi}, p.~\pageref{Positi}}
\Indx{predictable (sequence of random variables)}
     {\AtEquaN{Predict}}
\Indx{Pr\'ekopa--Leindler inequality}
     {\InSect \ref{VoluSections}, p.~\pageref{PrekoLei}}
\Indx{probability space, probability measure}
     {\InSect \ref{Gdfi}, p.~\pageref{ProbaSpa}}
\Indx{random variable}
     {\InSect \ref{Gdfi}, p.~\pageref{RandoVar}}
\Indx{reflection principle (for the Brownian motion)}
     {\InSect \ref{PrincipeReflexion}, p.~\pageref{PrincipeReflexion}}
\Indx{Riesz transforms}
     {\InSect \ref{TransfoRiesz}, p.~\pageref{TransfoRiesz}}
\Indx{right maximal function (of a function on $\R$)}
     {\AtEqua \eqref{RightMax}, p.~\pageref{FStar}}
\Indx{Schwartz class $\ca S(\R^n)$}
     {\InSect \ref{PoissonSG}, p.~\pageref{SchwaSpa}}
\Indx{selectors}
     {\InSect \ref{Decoupling}, p.~\pageref{Selektor}}
\Indx{semi-group}
     {\InSect \ref{MaxiHopf}, p.~\pageref{SemiGrou}}
\Indx{spherical maximal operator}
     {\InSect \ref{SphericOp}, p.~\pageref{SphericOp}}
\Indx{square function (martingale)}
     {\InSect \ref{BGIneqs}, p.~\pageref{SquaFun}}
\Indx{stability under translation (of a function on $\R$)}
     {\InSect \ref{SecondReduc}, p.~\pageref{StabiTransla}}
\Indx{stochastic integral}
     {\InSect \ref{BGIneqs}, p.~\pageref{StochInteg}}
\Indx{stopping time}
     {\InSect \ref{PrincipeReflexion}, p.~\pageref{StoppingT}}
\Indx{strong type inequality}
     {\InIntr, p.~\pageref{StrongType}}
\Indx{subordination principle}
     {\before \AtEqua \eqref{Subor}, p.~\pageref{SuboPrin}}
\Indx{Thorin's method (for interpolation)}
     {\InSect \ref{Ihf}, p.~\pageref{ByThor}}
\Indx{three lines lemma}
     {\InLemmN{TLL}}
\Indx{transition matrix (of a Markov chain)}
     {\InSect \ref{RotaRem}, p.~\pageref{RotaRem}}
\Indx{uncentered maximal function}
     {\InIntr, p.~\pageref{Uncent}}
\Indx{unconditionality (of martingale differences)}
     {\InRemaN{Incondi}}
\Indx{variance}
     {\InSect \ref{GaussiDist}, p.~\pageref{VarianceDef}}
\Indx{Vitali covering lemma}
     {\InIntr, p.~\pageref{ViCovLem}}
\Indx{weak type inequality}
     {\InIntr, p.~\pageref{StrongType}}
\Indx{$\sigma$-field}
     {\InSect \ref{Gdfi}, p.~\pageref{SigmaFie}}
\Indx{$\sigma$-finite measure}
     {\InRemaN{InfiniProba}}
\Indx{$\tau$-stable (--- function on $\R$)}
     {\InSect \ref{SecondReduc}, p.~\pageref{StabiTransla}}

}
\dumou
\medskip
\goodbreak

\def\Notat#1#2#3
{\hbox to \hsize
 {\noindent\hbox to 2cm{#1 \hfill}
  \hbox%
  {#2}  \leaders \hbox to 0.8em{\hss.\hss}\hfill #3
 }
\vskip 0.11pt plus 0.06pt minus 0.08pt}
\def\NotatB#1#2
{\hbox to \hsize
 {\noindent\hbox to 2cm{#1 \hfill}
  \hbox%
  {#2} \hfill
 }
\vskip 0.11pt plus 0.06pt minus 0.08pt}
\smallskip

\noindent{\bf Notation}\label{Notation}
\medskip

{
\smaller
\smaller
\smaller

\Notat{$\gr 1_A$}
 {indicator function of the set $A$}
 {\InSect \ref{Gdfi}, p.~\pageref{IndiFunc}}
\Notat{$A_p$, $B_p$}
 {constants in Khinchin's inequalities}
 {\AtEquaN{Hincin}}
\Notat{\eqref{A0}$\ldots$\eqref{A3}}
 {assumptions for Carbery's Proposition~\ref{PropoPrio}}
 {\before \InProp~\ref{PropoPrio}, p.~\pageref{A0}}
\Notat{$A_{\alpha, t}^{(n)}$}
 {set where the maximal function $\M_Q \mu_N^{(n)}$ is large}
 {\InLemmN{OneSet}}
\NotatB{$a \wedge b$, $a \vee b$}
 {minimum, maximum of two real numbers $a$ and $b$}
\NotatB{$\ca B_Y$}
 {Borel $\sigma$-field of a topological space $Y$}
\Notat{$(B_t)_{t \ge 0}$}
 {Brownian motion in $\R^n$}
 {\InSect \ref{GaussiDist}, p.~\pageref{BrowniMot}}
\Notat{$B_\tau$}
 {Brownian value $\omega \mapsto B_{\tau(\omega)}(\omega)$ 
  at a stopping time $\tau$}
 {\InSect \ref{PrincipeReflexion}, p.~\pageref{BTau}}
\Notat{$B(q_0, R, n)$}
 {\textit{a priori} bound in Bourgain's cube proof}
 {\InSect \ref{APriori}, p.~\pageref{Objective}}
\Notat{$\bar \mu$}
 {barycenter of a probability measure $\mu$ on $\R^n$}
 {\InSect \ref{GaussiDist}, p.~\pageref{Baryce}}
\Notat{$C_p, C'_p, C''_p$}
 {constants for Carbery's Proposition~\ref{PropoPrio}}
 {\AtEqua \eqref{A0}$\ldots$\eqref{A2}, p.~\pageref{A0}}
\Notat{$C_\alpha(m)$}
 {Carbery's constant for a Fourier multiplier $m$}
 {\InPropN{FouCarbe}}
\Notat{$C_\alpha$}
 {a subset of $[-N, N]$ in proof of Aldaz--Aubrun weak type theorem}
 {\AtEquaN{CalphaSet}}
\Notat{$c_p$}
 {Burkholder--Gundy constant, $1 < p < +\infty$}
 {\InTheoN{BurGun}}
\Notat{$D^z h$, $D^\alpha h$}
 {fractional derivative of $h$}
 {\AtEquaN{DalphaDef} 
  \& \eqref{IntegraUn}, p.~\pageref{IntegraUn}}
\Notat{$D^\alpha_t h(\lambda t)$}
 {fractional derivative of $t \mapsto h(\lambda t)$}
 {\AtEquaN{DalphaDef} 
  \& \eqref{IntegraUn}, p.~\pageref{IntegraUn}}
\Notat{$D^\alpha_t h(\lambda t) \barre_{t = t_0} \ns8$ }
 {fractional derivative evaluated at $t_0$}
 {\after \AtEqua \eqref{Dilate}, p.~\pageref{SpeciVal}}
\Notat{$(d_k)_{k=0}^N$}
 {martingale difference sequence}
 {\InSect \ref{BGIneqs}, p.~\pageref{BGIneqs}}
\Notat{$d^z h$, $d^\alpha h$}
 {another fractional derivative of $h$}
 {\AtEquaN{dalpha}}
\Notat{$\E f$}
 {expectation of the random variable $f$}
 {\InSect \ref{Gdfi}, p.~\pageref{Expecte}}
\Notat{$\E (f \ms1 | \ms1 \ca G)$}
 {conditional expectation of $f$ on the $\sigma$-field $\ca G$}
 {\InSect \ref{Gdfi}, p.~\pageref{CondExpe}}
\NotatB{$(\gr e_j)_{j=1}^n$}
 {standard unit vector basis of $\R^n$}
\Notat{$\ca F f$, $\widehat f$, $\widehat \mu$}
 {Fourier transform of a function $f$, of a measure $\mu$}
 {\InSect \ref{GdfiTwo}, p.~\pageref{FouTran}}
\Notat{$F_m$}
 {set of $x \in \R^n$ with $N_{n, \alpha(f)}(x) = m$ in Section~\ref{AlAu}}
 {\after \AtEqua \eqref{MinoGf}, p.~\pageref{SetFm}}
\Notat{$\|f\|_p, \|f\|_{L^p}$}
 {norm of a function $f$ in $L^p$}
 {\InTheo \ref{TheoHL}, p.~\pageref{TheoHL}}
\Notat{$f^*$}
 {uncentered classical maximal function of $f$}
 {\InIntr, p.~\pageref{Uncent}}
\Notat{$f_\# \mu$}
 {pushforward image of the finite measure $\mu$ by the mapping $f$}
 {p.~\pageref{PushForward}}
\Notat{$f^*_r$, $f^*_\ell$}
 {right, left maximal function of a function $f$ on $\R$}
 {\AtEquaN{RightMax}}
\Notat{$f_*, f^*$}
 {lower, upper bound for $f = m / n$, proof of Iakovlev--Str\"omberg}
 {\after \AtEqua \eqref{MinoGf}, p.~\pageref{FStars}}
\Notat{$G_s f$}
 {Gaussian semi-group acting on $f$}
 {\AtEquaN{GSG}}
\Notat{$G$}
 {Gaussian kernel in Bourgain's cube proof,
  $\widehat G(\xi) = \e^{ - 4 \pi |\xi|^2}$}
 {\after \AtEqua \eqref{grA}, p.~\pageref{GKernel}}
\Notat{$g_p$}
 {absolute $p\ms1$-th moment of the Gaussian measure $\gamma_1$ on $\R$}
 {\AtEquaN{AsymptoGauss}}
\Notat{$g^\vee$}
 {inverse Fourier transform of a function $g$ on $\R^n$}
 {\InSect \ref{GdfiTwo}, p.~\pageref{InveFour}}
\Notat{$g(f)$, $g_k(f)$}
 {Littlewood--Paley functions for $f$ on $\R^n$, $k \ge 1$}
 {\InSect \ref{LittlePal}, p.~\pageref{GdeF}}
\Notat{$\Di g \lambda$, $\di g \lambda$}
 {dilates of a function $g$ on $\R^n$}
 {\AtEquaN{Dilata}}
\Notat{$H f$, $H_\T f$}
 {Hilbert transform of $f$ on $\R$ or $\T$}
 {\InSect \ref{TransfoRiesz}, p.~\pageref{HilbTrans}}
\Notat{$H_k$}
 {homogeneous parts in Pisier's semi-group theorem}
 {\before \AtEqua \eqref{IciChk}, p.~\pageref{HomogPart}}
\Notat{$\| h \|_{ {\hbox{\sevenit L}}^2_\alpha}$}
 {Carbery's multiplier norm of a function $h$ on $(0, +\infty)$}
 {\AtEquaN{LdeuxAlpha}}
\Notat{$h_q$}
 {$L^q$ bound for $H_1$ in Pisier's semi-group theorem}
 {\InPropN{PisierSG}}
\Notat{$I$}
 {identity operator on $L^q(\R^n)$}
 {\InSect \ref{FirstReduc}, p.~\pageref{IdOp}}
\NotatB{$\I_n$}
 {identity matrix of size $n \times n$}
\Notat{$I^w f$, $I^\beta f$}
 {fractional integration}
 {\AtEquaN{IntegraZero}}
\Notat{$i^w f$, $i^\beta f$}
 {another fractional integration}
 {\InSect \ref{StrateMull}, p.~\pageref{FractiInt}}
\NotatB{$\mathrm{J}_\nu$}
 {Bessel function of order $\nu$}
\NotatB{$\ca K(\R^n)$}
 {space of compactly supported continuous functions on $\R^n$}
\Notat{$K_C$}
 {uniform probability density on a convex set $C$}
 {\InSect \ref{TheSetting}, p.~\pageref{KdeC}}
\Notat{$\Klc$}
 {a symmetric log-concave probability density on $\R^n$}
 {\InPropN{EstimaFouriC}}
\Notat{$\Kg$}
 {a probability density or kernel on $\R^n$ satisfying~\eqref{EstimaGene}}
 {\before \AtEqua \eqref{EstimaGene}, p.~\pageref{KernelKg}}
\Notat{$K^R$}
 {Bourgain's kernels for the cube}
 {\InSect \ref{LeCube}, p.~\pageref{KaR}}
\NotatB{$L^p(\R^n)$}
 {Lebesgue spaces, $1 \le p \le +\infty$}
\Notat{$L(C)$}
 {isotropy constant of the convex set $C$}
 {\AtEquaN{IsotroDef}}
\Notat{$\M f$}
 {classical Hardy--Littlewood maximal function of $f$}
 {\AtEquaN{ClassicalM}}
\Notat{$\M_C f$}
 {maximal function of $f$ associated to $C$}
 {\AtEquaN{OpMaxi}}
\Notat{$M^*_N$}
 {maximal function of a martingale $(M_k)_{k=0}^N$}
 {\InSect \ref{MaxiDoob}, p.~\pageref{MaxiProc}}
\Notat{$\ca M f$}
 {radial maximal function of $f$}
 {\InSect \ref{SphericOp}, p.~\pageref{RadiMaxi}}
\Notat{$\M_K f$, $\Mg_K f$}
 {maximal function of $f$ associated to a kernel $K$}
 {\InSect \ref{DefiMaxiFunc}, p.~\pageref{MKetMK}
  \& \InSect \ref{TheSetting}, p.~\pageref{MKetMKbis}}
\Notat{$\M_C^{(d)} f$}
 {dyadic maximal function of $f$ associated to $C$}
 {\InSect \ref{ArtiCarbe}, p.~\pageref{DyadiMax}}
\Notat{$\|m\|_{p \rightarrow p}$}
 {norm on $L^p$ of a multiplier $m$}
 {\InSect \ref{FourierMult}, p.~\pageref{MPdansP}}
\Notat{$m_\sigma$}
 {Fourier transform of the uniform probability measure $\sigma$ on
  $S^{n-1}$}
 {\InSect \ref{Bsmo}, p.~\pageref{MsubS}}
\Notat{$m^*$}
 {$m^*(\xi) = \xi \ps \nabla m(\xi)$, for a multiplier $m$ on $\R^n$}
 {p.~\pageref{mStar}}
\Notat{$m_C$, $m_C(\xi)$}
 {Fourier transform of $K_C$}
 {\InSect \ref{TheSetting}, p.~\pageref{MsubC}}
\NotatB{$\mlc$, $\mg$}
 {Fourier transform of $\Klc$, of $\Kg$}
\Notat{$m^\varepsilon_z$}
 {M\"uller's holomorphic family of multipliers}
 {\AtEquaN{mzFunc}}
\Notat{$m^{\#}$}
 {M\"uller's \og crucial\fge multiplier $m^{\#}(\xi) = |\xi| m(\xi)$}
 {\AtEquaN{MZero}}
\Notat{$m^R$}
 {Bourgain's cube multiplier, Fourier transform of $K^R$}
 {\InSect \ref{LeCube}, p.~\pageref{KaR}}
\NotatB{$\N^*$}
 {set of integers $n > 0$}
\Notat{$N(0, \I_n)$}
 {centered Gaussian distribution with covariance matrix $\I_n$}
 {\InSect \ref{GaussiDist}, p.~\pageref{NZero}}
\Notat{$N_{n, \alpha} (x)$}
 {number of coordinates of $x \in \R^n$ that are in $C_\alpha$}
 {\before \InLemm \ref{OneSet}, p.~\pageref{NAlpha}}
\NotatB{$\ca O(n)$}
 {orthogonal group}
\Notat{$P_t$, $P_t f$}
 {Poisson measure, Poisson semi-group acting on $f$}
 {\InSect \ref{PoissonSG}, p.~\pageref{PoissoSG}}
\Notat{$P_{\ms2 t}^{(\ns{0.7} n \ns1)}$}
 {Poisson kernel on $\R^n$}
 {\AtEquaN{PoissonDensi}}
\Notat{$p^*$}
 {$p^* := \max(p, p / (p - 1))$, in Burkholder's 
  unconditional constant $p^* - 1$}
 {p.~\pageref{PStar}}
\Notat{$Q_\mu$}
 {covariance quadratic form of a measure $\mu$}
 {\InSect \ref{GaussiDist}, p.~\pageref{CovariMa}}
\Notat{$Q(C)$}
 {M\"uller's constant for a convex set $C$}
 {\InSect \ref{AMuller}, p.~\pageref{QuC}}
\Notat{$Q$, $Q_n$}
 {symmetric cube of volume one in $\R^n$}
 {\InSect \ref{LeCube}, p.~\pageref{QuN}}
\Notat{$\textrm{q}_{\ms1 p}$}
 {Littlewood--Paley constant, for $1 < p < +\infty$}
 {\AtEquaN{LiPaIneq}}
\Notat{$q(C)$}
 {modified M\"uller's constant for a symmetric convex set $C$}
 {\AtEquaN{qC}}
\Notat{$R_j f$}
 {Riesz transforms of $f$ in $\R^n$, $1 \le j \le n$}
 {\before \AtEqua \eqref{Collective2}, p.~\pageref{RieszTra}}
\Notat{$\ca R f$}
 {vector form of the Riesz transform in $\R^n$}
 {\InSect \ref{TransfoRiesz}, p.~\pageref{VRieszTr}}
\Notat{$R_0$}
 {$R_0^{\delta/2} = 4$, in Bourgain's cube proof}
 {p.~\pageref{Rzero}}
\NotatB{$|S|$}
 {measure of a set $S$}
\NotatB{$|S|_n$}
 {Lebesgue's $n$-dimensional measure of a set $S$ in $\R^n$}
\Notat{$S_N$}
 {square function of a martingale}
 {\InSect \ref{BGIneqs}, p.~\pageref{SquaFun}}
\Notat{$\ca S(\R^n)$}
 {Schwartz function space on $\R^n$}
 {\InSect \ref{PoissonSG}, p.~\pageref{SchwaSpa}}
\NotatB{$S^{n-1}$}
 {unit sphere in $\R^n$}
\Notat{$S_p(\tau)$}
 {one-side moments of a log-concave probability density
  on $[\tau, +\infty)$}
 {\InLemmN{EstimaLogConcNew}}
\Notat{$S_N$, $S_N^n$}
 {sets $[-N, N]$, $[-N, N]^n$ in Section~\ref{AlAu}}
 {\after \AtEqua \eqref{WTforCubes}, p.~\pageref{SN}}
\Notat{$s_{n-1}$}
 {Lebesgue measure of the unit sphere in $\R^n$}
 {\AtEquaN{OmegaN}}
\NotatB{$s^+$}
 {positive part of a real number $s$, $s^+ = \max(s, 0)$}
\NotatB{$\lfloor s \rfloor$, $\lceil s \rceil$}
 {floor, ceiling of $s$, integers such that
  $s-1 \le \lfloor s \rfloor \le s \le \lceil s \rceil < s+1$}
\Notat{$s_*, s^*$}
 {$s_* = \sqrt {f_* (1 - f_*)}$, 
  $s^* = \sqrt {f^* (1 - f^*)}$, in Iakovlev--Str\"omberg}
 {\after \AtEqua \eqref{Stirl}, p.~\pageref{SStars}}
\NotatB{$\T$}
 {unit circle in $\R^2$ or $\C$}
\NotatB{$\|T\|_{p \rightarrow p}$}
 {norm of an operator $T : L^p \rightarrow L^p$}
\Notat{$T_m$}
 {linear operator associated to the multiplier $m$ on $\R^n$}
 {\before \AtEqua \eqref{EasyBoundA}, p.~\pageref{TsubM}}
\Notat{$T_{j, v}$, $T_j$, $T^*$}
 {operators for Carbery's maximal theorem}
 {\InSect \ref{EstiPrio}, p.~\pageref{TsubJV}}
\Notat{$\gr T^J$}
 {product $\prod_{j \in J} T_j$ of linear operators $(T_j)_{j \in J}$}
 {\InSect \ref{FirstReduc}, p.~\pageref{TupJ}}
\NotatB{$t \ms1 C$}
 {dilate by $t > 0$ of the convex set $C$}
\Notat{$U_K$}
 {operator $f \mapsto \nabla K * f$}
 {\AtEquaN{UK}}
\Notat{$u(x, t)$}
 {harmonic extension of $f(x)$, $x \in \R^n$,
  to the upper half-space in $\R^{n+1}$}
 {\InSect \ref{PoissonSG}, p.~\pageref{UdeXT}}
\Notat{$V(K)$}
 {directional variation of a kernel $K$}
 {\AtEquaN{V(K)}}
\Notat{$V$}
 {fixed large value $V$ in Iakovlev--Str\"omberg}
 {\after \InTheo \ref{PropIS}, p.~\pageref{ValueV}}
\Notat{$w_0$}
 {$w_0 = R^{\delta/2}$ in Bourgain's cube proof}
 {\after \InLemm \ref{CubeLemme5}, p.~\pageref{Wsub0}}
\Notat{$w_1$}
 {$w_1 = w_0^2 = R^{\delta}$ in Bourgain's cube proof}
 {\after \AtEqua \eqref{H4}, p.~\pageref{Wsub1}}
\Notat{$X_{1, \alpha}$, $X_{n, \alpha}$}
 {Bernoulli, binomial variable in Aldaz--Aubrun}
 {\AtEqua \eqref{Xalpha} 
  \& \before \InLemm \ref{OneSet}, p.~\pageref{XnAlpha}}
\NotatB{$|x|$}
 {norm of a vector $x$, usually Euclidean norm on $\R^n$}
\Notat{$\|x\|_C$}
 {norm of a vector $x$ relative to a symmetric convex set $C$}
 {\InSect \ref{PrincipeReflexion}, p.~\pageref{CNorm}}
\Notat{$\xG$}
 {point where $\Gamma$ reaches its minimum on $(0, \infty)$}
 {\AtEquaN{MiniGamm}}
\Notat{$\ovy$}
 {maximal argument $\ovy = \tau / {\sigma_\alpha}$}
 {\InLemmN{YandAllow}}
\dumou 
\Notat{$\alpha_j(m)$}
 {constituent of Bourgain's constant $\Gamma_B(K)$ for a kernel $K$}
 {\InLemmN{Clef}}
\Notat{$\alpha(f)$}
 {value associated to $f = m/n$ in Iakovlev--Str\"omberg}
 {\after \AtEqua \eqref{PsiF}, p.~\pageref{AlphaF}}
\Notat{$\beta_a$}
 {in a bound for $|\Gamma(z)|^{-1}$}
 {\AtEquaN{MajoGamm}}
\Notat{$\beta_j(m)$}
 {constituent of Bourgain's constant $\Gamma_B(K)$ for a kernel $K$}
 {\InLemmN{Clef}}
\Notat{$\Gamma_B(K)$}
 {Bourgain's constant for a kernel $K$}
 {\InLemm \ref{Clef}, p.~\pageref{GammSubB}}
\Notat{$\Gamma^S$}
 {operator $\gr T^{\sim S} \ms1 (\gr I - \gr T)^S$,
  in Bourgain's cube proof}
 {\AtEquaN{DefGammaS}}
\Notat{$\gamma_n$, $\gamma_F$}
 {Gaussian probability measure on $\R^n$, on a Euclidean space $F$}
 {\AtEquaN{LoiNZeroId}}
\Notat{$\Delta_k$, $\Delta_{k, c}$}
 {sum of bounds $\sum_{j=0}^k \delta_{j, g}$, $\sum_{j=0}^k \delta_{j, c}$}
 {\AtEquaN{DefiDel}}
\Notat{$\delta_x$}
 {Dirac probability measure at the point $x$}
 {p.~\pageref{Dirac}}
\Notat{$\delta_{j, c}$}
 {bounds for $m_C$, $\mlc$ and their derivatives}
 {\InLemmN{LEstimatesForC}}
\Notat{$\delta_{j, g}$}
 {bounds for $\mg$ and its derivatives}
 {\AtEquaN{EstimaGene} \& \AtEquaN{EstimaGeneN}}
\Notat{$\delta$}
 {$\delta = \eta^3 \ms1 s_* / (1 - f_*)$, in proof of
  Theorem~\ref{PropIS}}
 {\InLemmN{Intersect}}
\NotatB{$\partial S$}
 {boundary of a set $S \subset \R^n$}
\NotatB{$\partial_i f$}
 {$i\ms1$th partial derivative of a function $f$ on $\R^n$,
  $i = 1, \ldots, n$}
\NotatB{$\nabla f$}
 {gradient of the function $f$ on $\R^n$}
\Notat{$(\varepsilon_j)_{j=1}^{+\infty}$}
 {independent Bernoulli random variables}
 {\before \AtEqua \InSect \eqref{GSG}, p.~\pageref{BernouVaris}}
\Notat{$\varepsilon$}
 {value $\varepsilon \in (0, f_*]$, in the proof of Theorem~\ref{PropIS}}
 {\before \AtEqua \eqref{FmoinsA}, p.~\pageref{EpsilEta}}
\Notat{$\eta$}
 {$\eta = \sqrt{ 1 - \varepsilon / f_*}$, in the proof of Theorem~\ref{PropIS}}
 {\before \AtEqua \eqref{FmoinsA}, p.~\pageref{EpsilEta}}
\NotatB{$\theta^\perp$}
 {hyperplane orthogonal to $\theta \in S^{n-1}$}
\Notat{$\kappa_{Q, n}$}
 {weak type $(1, 1)$ constant for the cube in $\R^n$}
 {\InIntr, p.~\pageref{KappaQ} \& \InSect \ref{AlAu}, p.~\pageref{AlAu}}
\Notat{$\lambda_p$}
 {bound on $L^p$ for Laplace-type multipliers}
 {\InPropN{LaplaceMultip}}
\Notat{$\widehat \mu$}
 {Fourier transform of the measure $\mu$}
 {\InSect \ref{GdfiTwo}, p.~\pageref{FouriMu}}
\Notat{$\mu_C$}
 {uniform probability measure on a convex set $C$}
 {\InSect \ref{TheSetting}, p.~\pageref{MusubC}}
\Notat{$\mu^+$, $\mu^-$, $|\mu|$}
 {positive part, negative part, absolute value of a measure}
 {\before~\InLemm \ref{PreBourg}, p.~\pageref{MuPlusMoins}}
\Notat{$\|\mu\|_1$}
 {mass of a real-valued measure $\mu$}
 {\before~\InLemm \ref{PreBourg}, p.~\pageref{MuPlusMoins}}
\Notat{$\mu^R$}
 {Bourgain's cube measure}
 {\after \AtEqua \eqref{SigmaR}, p.~\pageref{MuR}}
\Notat{$\mu_N$, $\mu_N^{(n)}$}
 {discrete measure in the proof of Aldaz--Aubrun weak type theorem}
 {\after \AtEqua \eqref{WTforCubes}, p.~\pageref{MuN}}
\Notat{$ (\xi \ps \nabla)^\alpha $}
 {Carbery's fractional operator}
 {\AtEquaN{OpAlpha}}
\Notat{$\rho_p$}
 {\og collective\fge norm on $L^p$ for Riesz transforms}
 {\AtEquaN{TransfosRiesz}}
\Notat{$\Sigma_k$}
 {subsets of $\{1, \ldots, n\}$ having cardinality $k$}
 {\before \AtEqua \eqref{IciChk}, p.~\pageref{SigmaK}}
\Notat{$\sigma^2$}
 {variance of a probability measure or density}
 {\InSect \ref{GaussiDist}, p.~\pageref{VarianceDef}}
\Notat{$\sigma$, $\sigma_{n-1}$}
 {invariant probability measure on the unit sphere $S^{n-1}$}
 {\InSect \ref{SphericOp}, p.~\pageref{HaarMeas} \& p.~\pageref{SigmSubN}}
\Notat{$\sigma_{S, j}$}
 {Bourgain's selector}
 {\InSect \ref{Decoupling}, p.~\pageref{Selektor}}
\Notat{$\sigma_\alpha^2$}
 {variance $\sigma_\alpha^2 = \alpha (1 - \alpha)$
  of a Bernoulli random variable}
 {\InLemmN{YandAllow}}
\Notat{$\Phi_j$}
 {convolution operator in Bourgain's cube proof}
 {\before \AtEqua \eqref{LesPhij}, p.~\pageref{PhiSubJ}}
\Notat{$\varphi_\theta$, $\varphi_{\theta, K}$}
 {marginal of a kernel $K$ on $\R^n$, image by $\theta \in S^{n-1}$}
 {\AtEquaN{VarphiTheta}}
\Notat{$\varphi_q$}
 {lower bound of angle for Pisier's analytic semi-group theorem}
 {\InPropN{PisierSG}}
\Notat{$\phi_\alpha$}
 {a function on $(0, 1)$ in Section~\ref{AlAu}}
 {\AtEquaN{DefPhiAlpha}}
\Notat{$\Omega_1$, $\Omega$}
 {sets $[-N, N]$, $[-N, N]^n$ in Section~\ref{AlAu}}
 {\before \AtEqua \eqref{Xalpha}, p.~\pageref{OmegaUn}}
\Notat{$\omega_n$}
 {Lebesgue volume of the unit ball in $\R^n$}
 {\AtEquaN{OmegaN}}
 
}
\vfill

\begin{thebibliography}{99}\label{References}

\smaller
 
\bibitem{AldazWT}
\RefSA
 {J.~M.} {Aldaz}
 {The weak type $(1, 1)$ bounds for the maximal function
  associated to cubes grow to infinity with the dimension}
 {Ann. of Math.\VPY {173} {1013}{1023} {2011}}
\dumou

\bibitem{A-Askey-R}
\RefMA
 3 {G.~E.} {Andrews}
 2 {R.} {Askey}
 1 {R.} {Roy}
 {Special Functions}
 {Volume~71 of Encyclopedia of Mathematics and its Applications.
  Cambridge University Press, Cambridge, 1999}
\fin
\dumou

\bibitem{Aubrun}
\RefSA
 {G.} {Aubrun}
 {Maximal inequality for high-dimensional cubes}
 {Confluentes Math. {1}, p.~169--179 (2009)}
\dumou

\bibitem{AuCa}
\RefMA
 2 {P.} {Auscher}
 1 {M. J.} {Carro}
 {Transference for radial multipliers and dimension free estimates}
 {Trans. Amer. Math. Soc. {342}, p.~575--593 (1994)}
\fin
\dumou

\bibitem{Ball}
\RefSA
 {K.} {Ball}
 {Logarithmically concave functions and sections of convex sets in $\R^n$}
 {Studia Math. {88}, p.~69--84 (1988)}
\dumou

\bibitem{BenePanz}
\RefMA
 2 {A.} {Benedek}
 1 {R.} {Panzone}
 {The space $L^p$, with mixed norm}
 {Duke Math. J.\VPY {28} {301}{324} {1961}}
\fin
\dumou

\bibitem{BerLof}
\RefMA
 2 {J.} {Bergh}
 1 {J.} {L\"ofstr\"om}
 {Interpolation spaces. An introduction}
 {Grundlehren der Mathematischen Wissenschaften, No. 223. Springer-Verlag,
  Berlin-New York, 1976}
\fin
\dumou

\bibitem{BourgainL2}
\RefSA
 {J.} {Bourgain}
 {On high dimensional maximal functions associated to convex bodies}
 {Amer. J. Math. {108}, p.~1467--1476 (1986)}
\dumou

\bibitem{BourgAve}
\RefSA
 {J.} {Bourgain}
 {Averages in the plane over convex curves and maximal operators}
 {J. Analyse Math. {47}, p.~69--85 (1986)}
\dumou

\bibitem{BourgainLp}
\RefSA
 {J.} {Bourgain}
 {On the $L^p$-bounds for maximal functions associated to convex bodies
  in $\R^n$}
 {Israel J. Math. {54}, p.~257--265 (1986)}
\dumou

\bibitem{BoGAFA}
\RefSA
 {J.} {Bourgain}
 {On dimension free maximal inequalities for convex symmetric bodies in
  $\R^n$}
 {Geometrical aspects of functional analysis, p.~168--176 (1985/86),
  Lecture Notes in Math. 1267, Springer, Berlin, 1987}
\dumou

\bibitem{BoIsotro}
\RefSA
 {J.} {Bourgain}
 {On the distribution of polynomials on high-dimensional convex sets}
 {Geometric aspects of functional analysis, p.~127--137 (1989--90),
  Lecture Notes in Math., 1469, Springer, Berlin, 1991}
\dumou

\bibitem{BourgainCube}
\RefSA
 {J.} {Bourgain}
 {On the Hardy--Littlewood maximal function for the cube}
 {Israel J. of Math. {203}, p.~275--293 (2014)}
\dumou

\bibitem{BreMass}
\RefMA
 2 {J.} {Bretagnolle} 
 1 {P.} {Massart}
 {Hungarian constructions from the nonasymptotic viewpoint}
 {Ann. Probab. {17}, p.~239--256 (1989)}
\fin
\dumou

\bibitem{BurkhoMT}
\RefSA
 {D.~L.} {Burkholder}
 {Martingale transforms}
 {Ann. Math. Statist. {37}, p.~1494--1504 (1966)}
\dumou

\bibitem{Burkho}
\RefSA
 {D.~L.} {Burkholder}
 {Martingales and singular integrals in Banach spaces}
 {Handbook of the geometry of Banach spaces, Vol. I, 233--269, 
  North-Holland, Amsterdam, 2001}
\dumou

\bibitem{BurkholderGundy}
\RefMA
 2 {D.~L.} {Burkholder}
 1 {R.~F.} {Gundy}
 {Extrapolation and interpolation of quasi-linear operators on
  martingales}
 {Acta Math. {124}, p.~249--304 (1970)}
\fin
\dumou

\bibitem{BGS}
\RefMA
 3 {D.~L.} {Burkholder}
 2 {R.~F.} {Gundy}
 1 {M.~L.} {Silverstein}
 {A maximal function characterization of the class $H^p$}
 {Trans. Amer. Math. Soc. {157}, p.~137--153 (1971)}
\fin
\dumou

\bibitem{Calderon}
\RefSA
 {A.~P.} {Calder\'on}
 {Intermediate spaces and interpolation, the complex method}
 {Studia Math. {24}, p.~113--190 (1964)}
\dumou

\bibitem{Carbery}
\RefSA
 {A.} {Carbery}
 {Radial Fourier multipliers and associated maximal functions}
 {Recent Pro\-gress in Fourier Analysis, 
  North-Holland Math. Studies Vol. 111, 1985, p.~49--56}
\dumou

\bibitem{CarberyLp}
\RefSA
 {A.} {Carbery}
 {An almost-orthogonality principle with applications to maximal
  functions associated to convex bodies}
 {Bull. Amer. Math. Soc. {14}, p.~269--273 (1986)}
\dumou

\bibitem{ClSt}
\RefMA
 2 {J.~L.} {Clerc} 
 1 {E.~M.} {Stein}
 {$L^p$-multipliers for noncompact symmetric spaces}
 {Proc. Nat. Acad. Sci. U.S.A. {71}, p.~3911--3912 (1974)}
\fin
\dumou

\bibitem{CW}
\RefMA
 2 {R.~R.} {Coifman} 
 1 {G.} {Weiss}
 {Analyse harmonique non-commutative sur certains espaces homog\`enes}
 {Lecture Notes in Mathematics Vol.~242.
  Springer-Verlag, Berlin-New York, 1971}
\fin
\dumou

\bibitem{CGGM}
\RefMA
 4 {M.} {Cowling} 
 3 {G.} {Gaudry}
 2 {S.} {Giulini}
 1 {G.} {Mauceri}
 {Weak type $(1, 1)$ estimates for heat kernel maximal
  functions on Lie groups}
 {Trans. Amer. Math. Soc. {323}, p.~637--649 (1991)}
\fin
\dumou

\bibitem{CrSo}
\RefMA
 2 {A.} {Criado} 
 1 {F.} {Soria}
 {Localization and dimension free estimates for maximal functions}
 {J. Funct. Anal. {265}, p.~2553--2583 (2013)}
\fin
\dumou

\bibitem{BurgessDavis}
\RefSA
 {B.} {Davis}
 {On the integrability of the martingale square function} 
 {Israel J. Math. 8, p.~187--190 (1970)}
\dumou

\bibitem{Del}
\RefSA
 {L.} {Deleaval}
 {Two results on the Dunkl maximal operator}
 {Studia Math. {203}, p.~47--68 (2011)}
\dumou

\bibitem{RMDud}
\RefSA
 {R.~M.} {Dudley}
 {The sizes of compact subsets of Hilbert space and continuity of Gaussian
  processes}
 {J. Functional Analysis {1}, p.~290--330 (1967)}
\dumou

\bibitem{DunfordSchwartz}
\RefMA
 2 {N.} {Dunford} 
 1 {J.~T.} {Schwartz}
 {Linear Operators, Part I}
 {New York, 1958}
\fin
\dumou

\bibitem{DuoRubio}
\RefMA
 2 {J.} {Duoandikoetxea} 
 1 {J.~L.} {Rubio de Francia}
 {Estimations ind\'ependantes de la dimension pour les transform\'ees 
  de Riesz}
 {C. R. Acad. Sc. Paris {300}, p.~193--196 (1985)}
\fin
\dumou

\bibitem{Durrett}
\RefSA
 {R.} {Durrett}
 {Brownian motion and martingales in analysis}
 {Wadsworth Mathematics Series. 
  Wadsworth International Group, Belmont, CA, 1984}
\dumou

\bibitem{DurrettPTE}
\RefSA
 {R.} {Durrett}
 {Probability: theory and examples}
 {Fourth edition. 
  Cambridge Series in Statistical and Probabilistic Mathematics, 31. 
  Cambridge U. Press, Cambridge, 2010}
\dumou

\bibitem{FS}
\RefMA
 2 {C.} {Fefferman} 
 1 {E.~M.} {Stein}
 {Some maximal inequalities}
 {Amer. J. Math. {93}, p.~107--115 (1971)}
\fin
\dumou

\bibitem{FradelHSCB}
\RefSA
 {M.} {Fradelizi}
 {Hyperplane sections of convex bodies in isotropic position}
 {Beitr\"age Algebra Geom. {40}, p.~163--183 (1999)}
\dumou

\bibitem{FraGue}
\RefMA
 2 {M.} {Fradelizi} 
 1 {O.} {Gu\'edon}
 {The extreme points of subsets of $s$-concave probabilities and a geometric
  localization theorem}
 {Discrete Comput. Geom. 31, p.~327--335 (2004)}
\fin
\dumou

\bibitem{GarciRubio}
\RefMA
 2 {J.} {Garc\'{i}a-Cuerva}
 1 {J. L.} {Rubio de Francia}
 {Weighted norm inequalities and related topics}
 {North-Holland Mathematics Studies, 116. Notas de Matem\'{a}tica 104. 
  North-Holland Publishing Co., Amsterdam, 1985}
\fin
\dumou

\bibitem{GardnerBM}
\RefSA
 {R.~J.} {Gardner}
 {The Brunn--Minkowski inequality}
 {Bull. Amer. Math. Soc. {39}, no. 3, p.~355--405 (2002)}
\dumou

\bibitem{Garsia}
\RefSA
 {A.} {Garsia}
 {A simple proof of E. Hopf's maximal ergodic theorem} 
 {J. Math. Mech. {14}, p.~381--382 (1965)}
\dumou

\bibitem{GrafakosCFA}
\RefSA
 {L.} {Grafakos} 
 {Classical Fourier analysis. Third edition} 
 {Graduate Texts in Mathematics, 249. Springer, New York, 2014}
\dumou

\bibitem{GMS}
\RefMA
 2 {L.} {Grafakos} 
 1 {S.} {Montgomery-Smith}
 {Best constants for uncentred maximal functions}
 {Bull. London Math. Soc. {29}, p.~60--64 (1997)}
\fin
\dumou

\bibitem{GMSM}
\RefMA
  3 {L.} {Grafakos} 
  2 {S.} {Montgomery-Smith} 
  1 {O.} {Motrunich}
 {A sharp estimate for the Hardy--Littlewood maximal function}
 {Studia Math. {134}, p.~57--67 (1999)}
\fin
\dumou

\bibitem{GuMi}
\RefMA
 2 {O.} {Gu\'edon} 
 1 {E.} {Milman}
 {Interpolating thin-shell and sharp large-deviation estimates for
  isotropic log-concave measures}
 {Geom. Funct. Anal. {21}, p.~1043--1068 (2011)}
\fin
\dumou

\bibitem{Haa}
\RefSA
 {U.} {Haagerup}
 {The best constants in the Khintchine inequality}
 {Studia Math. {70} (1981), p.~231--283 (1982)}
\dumou

\bibitem{HaLi}
\RefMA
 2 {G.~H.} {Hardy} 
 1 {J.~E.} {Littlewood}
 {A maximal theorem with function-theoretic applications}
 {Acta Math. {54}, p.~81--116 (1930)}
\fin
\dumou

\bibitem{Hir}
\RefSA
 {I.~I.} {Hirschman}
 {A convexity theorem for certain groups of transformations}
 {Journal d'Analyse Math. {2}, p.~209--218 (1952)}
\dumou

\bibitem{IS}
\RefMA
 2 {A.~S.} {Iakovlev}
 1 {J.-O.} {Str\"omberg}
 {Lower bounds for the weak type $(1,1)$ estimate for the maximal function
  associated to cubes in high dimensions}
 {Math. Res. Lett. {20}, p.~907--918 (2013)}
\fin
\dumou

\bibitem{IwaMar}
\RefMA
 2 {T.} {Iwaniec}
 1 {G.} {Martin}
 {Riesz transforms and related singular integrals}
 {J. Reine Angew. Math. 473, p.~25--57 (1996)}
\fin
\dumou

\bibitem{IwaSbo}
\RefMA
 2 {T.} {Iwaniec}
 1 {C.} {Sbordone}
 {Riesz transforms and elliptic PDEs with VMO coefficients}
 {J. Anal. Math. {74}, p.~183--212 (1998)}
\fin
\dumou

\bibitem{Klar}
\RefSA
 {B.} {Klartag}
 {On convex perturbations with a bounded isotropic constant}
 {Geom. Funct. Anal. {16}, p.~1274--1290 (2006)}
\dumou

\bibitem{Klar2}
\RefSA
 {B.} {Klartag}
 {A central limit theorem for convex sets}
 {Invent. Math. {168}, p.~91--131 (2007)}
\dumou

\bibitem{KMT}
\RefMA
 3 {J.} {Koml\'os}
 2 {P.} {Major} 
 1 {G.} {Tusn\'ady}
 {An approximation of partial sums of independent RV's 
  and the sample DF. I}
 {Z. Wahrscheinlichkeitstheorie und Verw. Gebiete {32}, 
  p.~111--131 (1975)}
\fin
\dumou

\bibitem{Li1}
\RefSA
 {H.-Q.} {Li}
 {Fonctions maximales centr\'ees de Hardy--Littlewood sur les
  groupes de Heisenberg}
 {Studia Math. {191}, p.~89--100 (2009)}
\dumou

\bibitem{Li3}
\RefSA
 {H.-Q.} {Li}
 {Fonctions maximales centr\'ees de Hardy--Littlewood pour les
  op\'erateurs de Grushin}
 {preprint arXiv:1207.3128, 2012}
\dumou

\bibitem{Li2}
\RefSA
 {H.-Q.} {Li}
 {Centered Hardy--Littlewood maximal function on hyperbolic spaces,
  $p > 1$}
 {preprint arXiv:1304.3261, 2013}
\dumou

\bibitem{LiLo}
\RefMA
 2 {H.-Q.} {Li} 
 1 {N.} {Lohou\'e}
 {Fonction maximale centr\'ee de Hardy--Littlewood sur les
  espaces hyperboliques}
 {Ark. Mat. {50}, p.~359--378 (2012)}
\fin
\dumou

\bibitem{Mau}
\RefSA
 {B.} {Maurey}
 {Le syst\`eme de Haar}
 {S\'eminaire Maurey--Schwartz 1974--1975. Exp.~I et~II.
  Centre Math., \'Ecole Polytechnique, Paris, 1975}
\dumou

\bibitem{Melas}
\RefSA
 {A.~D.} {Melas}
 {The best constant for the centered Hardy--Littlewood
  maximal inequality}
 {Ann. of Math. {157}, p.~647--688 (2003)}
\dumou

\bibitem{MiPaj}
\RefMA
 2 {V.~D.} {Milman} 
 1 {A.} {Pajor}
 {Isotropic position and inertia ellipsoids and zonoids of the unit ball of
  a normed $n$-dimensional space}
 {Geometric aspects of functional analysis, p.~64--104 (1987--88),
  Lecture Notes in Math., 1376, Springer, Berlin, 1989}
\fin
\dumou

\bibitem{MullerQC}
\RefSA
 {D.} {M\"uller}
 {A geometric bound for maximal functions associated to convex bodies}
 {Pacific J. Math. {142}, p.~297--312 (1990)}
\dumou

\bibitem{NaorTao}
\RefMA
 2 {A.} {Naor} 
 1 {T.} {Tao}
 {Random martingales and localization of maximal inequalities}
 {J. Funct. Anal. {259}, p.~731--779 (2010)}
\fin
\dumou

\bibitem{Pich}
\RefSA
 {S. K.} {Pichorides}
 {On the best values of the constants in the theorems of M.~Riesz,
  Zygmund and Kolmogorov}
 {Studia Math. 44, p.~165--179 (1972)}
\dumou

\bibitem{PisierHSG}
\RefSA
 {G.} {Pisier}
 {Holomorphic semi-groups and the geometry of Banach spaces}
 {Ann. of Math. {115}, no. 2, p.~375--392 (1982)}
\dumou

\bibitem{PisierRT}
\RefSA
 {G.} {Pisier}
 {Riesz transforms: a simpler analytic proof of P.-A. Meyer's inequality}
 {S\'eminaire de Probabilit\'es, XXII, p.~485--501,
  Lecture Notes in Math., 1321, Springer, Berlin, 1988}
\dumou

\bibitem{PisierMart}
\RefSA
 {G.} {Pisier}
 {Martingales in Banach spaces}
 {Cambridge Studies in Advanced Mathematics~155.
  Cambridge Univ. Press, 2016}
\dumou

\bibitem{RieszHT}
\RefSA
 {M.} {Riesz}
 {Sur les fonctions conjugu\'ees}
 {Mathematische Zeitschrift {27}, p.~218--244 (1927)}
\dumou

\bibitem{Robbins}
\RefSA
 {H.} {Robbins}
 {A remark on Stirling's formula}
 {Amer. Math. Monthly {62}, p.~26--29 (1955)}
\dumou

\bibitem{Rota}
\RefSA
 {G.-C.} {Rota}
 {An \og Alternierende Verfahren\fge for General Positive Operators}
 {Bull. Amer. Math. Soc. {68}, p.~95--102 (1962)}
\dumou

\bibitem{rubio}
\RefSA
 {J.-L.} {Rubio~de Francia}
 {Maximal functions and Fourier transforms}
 {Duke Math. J. {53}, p.~395--404 (1986)}
\dumou

\bibitem{RudinRCA}
\RefSA
 {W.} {Rudin}
 {Real and complex analysis}
 {Third edition. McGraw-Hill Book Co., New York, 1987}
\dumou

\bibitem{SteinILO}
\RefSA
 {E.~M.} {Stein}
 {Interpolation of linear operators}
 {Trans. Amer. Math. Soc. {83}, p.~482--492 (1956)}
\dumou

\bibitem{SteinMET}
\RefSA
 {E.~M.} {Stein}
 {On the maximal ergodic theorem}
 {Proc. Nat. Acad. Sci. U.S.A. {47}, p.~1894--1897 (1961)}
\dumou

\bibitem{SteinLlogL}
\RefSA
 {E.~M.} {Stein}
 {Note on the class $L$ ${\rm log}$ $L$}
 {Studia Math. {32}, p.~305--310 (1969)}
\dumou

\bibitem{SteinTHA}
\RefSA
 {E.~M.} {Stein}
 {Topics in harmonic analysis related to the Littlewood--Paley theory}
 {Annals of Mathematics Studies 63, Princeton University Press (1970)}
\dumou

\bibitem{SteinMF}
\RefSA
 {E.~M.} {Stein}
 {Maximal functions. {I}. Spherical means}
 {Proc. Nat. Acad. Sci. U.S.A. {73}, p.~2174--2175 (1976)}
\dumou

\bibitem{SteinSF}
\RefSA
 {E.~M.} {Stein}
 {The development of square functions in the work of A.~Zygmund}
 {Bull. Amer. Math. Soc. {7}, p.~359--376 (1982)}
\dumou

\bibitem{SteinTV}
\RefSA
 {E.~M.} {Stein}
 {Three variations on the theme of maximal functions, 
  Recent Progress in Fourier Analysis}
 {North-Holland Math. Studies 111, p.~229--244 (1985)}
\dumou

\bibitem{SteStr}
\RefMA
 2 {E.~M.} {Stein} 
 1 {J.-O.} {Str\"omberg}
 {Behavior of maximal functions in ${\bf R}^{n}$ for large~$n$}
 {Ark. Mat. {21}, p.~259--269 (1983)}
\fin
\dumou

\bibitem{StrA}
\RefSA
 {J.-O.} {Str\"omberg}
 {Weak type $L^{1}$ estimates for maximal functions on
  noncompact symmetric spaces}
 {Ann. of Math. {114}, p.~115--126 (1981)}
\dumou

\bibitem{Sza}
\RefSA
 {S.~J.} {Szarek}
 {On the best constants in the Khinchin inequality}
 {Studia Math. {58}, p.~197--208 (1976)}
\dumou

\bibitem{ThorinA}
\RefSA
 {G.~O.} {Thorin}
 {An extension of a convexity theorem due to M. Riesz}
 {Kungl. Fysiografiska
  Saellskapets i Lund F\"{o}rhandlingar, no. 14, vol. 8 (1939)}
\dumou

\bibitem{Thorin}
\RefSA
 {O.} {Thorin}
 {Convexity theorems generalizing those of M. Riesz and Hadamard with some
  applications}
 {Comm. Sem. Math. Univ. Lund {9}, p.~1--58 (1948)}
\dumou

\bibitem{Titchmarsh}
\RefSA
 {E.~C.} {Titchmarsh}
 {The Theory of Functions}
 {Oxford University Press, 1939}
\dumou

\bibitem{Wiener}
\RefSA
 {N.} {Wiener}
 {The ergodic theorem}
 {Duke Math. J. {5}, p.~1--18 (1939)}
\dumou

\bibitem{Zien}
\RefSA
 {J.} {Zienkiewicz}
 {Estimates for the Hardy--Littlewood maximal function on the
  Heisenberg group}
 {Colloq. Math. {103}, p.~199--205 (2005)}
\dumou

\bibitem{ZygmundTS}
\RefSA
 {A.} {Zygmund}
 {Trigonometric series. Vol. I, II}
 {Third edition. Cambridge Mathematical Library. Cambridge University
  Press, Cambridge, 2002}
\dumou

\end{thebibliography}
\end{document}